\pdfoutput=1
\documentclass[11pt,reqno]{amsart}
\usepackage{amsthm,amsmath,amssymb,booktabs,graphicx,array,xcolor}
\usepackage[margin=0.675in]{geometry}
\usepackage[utf8]{inputenc}
\usepackage[T1]{fontenc}
\PassOptionsToPackage{hyphens}{url}
\usepackage{float,needspace}
\usepackage[hidelinks]{hyperref}
\usepackage{url}
\numberwithin{equation}{section}
\newtheorem{theorem}{Theorem}[section]
\newtheorem{proposition}[theorem]{Proposition}
\newtheorem{lemma}[theorem]{Lemma}
\newtheorem{corollary}[theorem]{Corollary}
\theoremstyle{definition}
\newtheorem{definition}[theorem]{Definition}
\newtheorem{example}[theorem]{Example}
\newtheorem{remark}[theorem]{Remark}
\newcommand{\R}{\mathbb R}
\newcommand{\Z}{\mathbb Z}
\newcommand{\dist}{\operatorname{dist}}
\allowdisplaybreaks
\hypersetup{pdftitle={Infinite Derivatives and Local Extrema of the Delange Function of Every Base},
pdfauthor={Chengyi Zhu},
pdfsubject={MSC2020 26A27; 05C35, 05C40, 11A63, 26A30, 28A80},
pdfkeywords={Edge isoperimetric problem, Hamming graph, extra edge
connectivity, digit sum, Takagi function, Delange function, infinite
derivatives, local extrema, Dini derivatives, level sets, variation,
Assouad dimension}}
\title[Infinite Derivatives and Local Extrema of Delange Functions]{Infinite Derivatives and Local Extrema of the Delange Function of Every Base}
\author[Chengyi Zhu]{Chengyi Zhu}
\address{Department of Mathematics, University of Wisconsin--Madison,
Madison, WI 53706, USA}
\email{czhu284@wisc.edu}
\thanks{\textbf{AI Usage Disclosure.}
ChatGPT was used for language editing and grammatical correction only.
All mathematical content is the author's own work, and the author takes full
responsibility for the entire manuscript.}
\date{}
\begin{document}

\begin{abstract}
For an integer $q\ge2$, the Delange function of base $q$ is
$T_q(x)=\sum_{j=0}^\infty q^{-j}\tau_q(q^jx)$,
where $\tau_q$ has period $1$, is linear on $[k/q,(k+1)/q]$ for integers
$0\le k<q$, and satisfies $\tau_q(k/q)=k(q-k)/q$ for $0\le k\le q$. Thus $T_2$ is
Takagi's function. In the Hamming graph on $\{0,1,\ldots,q-1\}^n$, two words are adjacent
exactly when they differ in one letter. For integers $n\ge1$ and $0\le m\le q^n$,
$q^nT_q(m/q^n)$ is the smallest number of edges leaving a set of $m$ vertices, which is known.
We characterize infinite derivatives by the base $q$ digits
$\xi_1,\xi_2,\ldots$ of $x\in[0,1)$. Outside the $q$-adic rationals
$m/q^n$ with integers $m$ and $n\ge0$, the right derivative is $+\infty$
exactly when $\sum_{j=1}^n(q-1-2\xi_j)\to+\infty$. The condition for the
left derivative to be $+\infty$ also involves the lengths of the maximal
strings of zeros. For $-\infty$, the sides and the digits $0$ and $q-1$
are exchanged. The local minima are exactly the $q$-adic rationals.
The local maxima are exactly the points $(m+u)/q^n+k$ with integers $k$,
$m$, $n$ such that $n\ge0$, $0\le m<q^n$, and the base $q$ digit sum of $m$
is $(q-1)n/2$, where $u$ maximizes $T_q$ on $[0,1]$. For odd $q$, we have
$u=1/2$ and every local maximum is strict. For even $q$, there are
uncountably many such $u$ and no local maximum is strict. We also find the
maximum on each interval $[m/q^n,(m+1)/q^n]$ and all points attaining it.
For an integer $h$ with $1\le h\le q^n/2$, we determine the minimum number
of edges whose deletion disconnects the graph into components with at
least $h$ vertices, including a digit
criterion for equality with the minimum boundary at size $h$. Almost
every level set $\{x\in[0,1]:T_q(x)=y\}$ is finite, whereas the level set
through almost every point is uncountable. Here almost every refers to
Lebesgue measure.
\end{abstract}

\maketitle

\setcounter{tocdepth}{1}
\tableofcontents

\section{Introduction}
\label{intro}

Let $q\ge2$ be an integer, which we call the base. We write $\R$ and $\Z$ for the real numbers and the integers. Let $\tau_q$ be the
continuous function of period $1$ that is linear on each interval
$[k/q,(k+1)/q]$ with $0\le k\le q-1$ and satisfies
$\tau_q(k/q)=k(q-k)/q$ for $0\le k\le q$, and let
\begin{equation}
 T_q(x)=\sum_{j=0}^{\infty}q^{-j}\tau_q(q^jx).
 \label{T0}
\end{equation}
On $[k/q,(k+1)/q]$ the function $\tau_q$ has the slope $q-1-2k$. Galkin
and Galkina~\cite{GalkinGalkina2020} call $T_q$ the Delange function of
base $q$. For $q=2$ we have $\tau_2(x)=\dist(x,\Z)$, the distance from
$x$ to the nearest integer, and $T_2$ is Takagi's
function~\cite{Takagi1903}. For $q=3$ we have
$\tau_3(x)=2\min\{\dist(x,\Z),1/3\}$. In this paper we characterize
completely, in every base, the points at which $T_q$ has an infinite
derivative from one or both sides, and we determine all local maxima and
all local minima of $T_q$. We also determine the maximum of $T_q$ on every
interval $[m/q^n,(m+1)/q^n]$ together with all points at which it is
attained, and we show that almost every level set of $T_q$ is finite,
while the level set through almost every point is uncountable. Through
the values of $T_q$ at the points $m/q^n$, the same methods give, for all
$n$ and all $h$ with $1\le h\le q^n/2$, the smallest number of edges whose
removal splits the Hamming graph of the words of length $n$ over the
letters $0,1,\ldots,q-1$ into parts with at least $h$ vertices each. In
this graph two words are adjacent when they differ in exactly one letter.
These results hold for every base $q\ge2$. For $q=2$ they include the
corresponding known results on Takagi's function and on the hypercube,
the Hamming graph for $q=2$.

The function $T_q$ has appeared under several names.
Delange~\cite[Section~1]{Delange1975} used the function $-T_q/2$ in his
formula for the sum of the digit sums of $0,1,\ldots,m-1$ in base $q$,
and he proved that this function has no finite derivative at any
point~\cite[Section~3]{Delange1975}. Galkin and
Galkina~\cite[Theorems~2 and~9]{GalkinGalkina2020} determined the global
extrema of $T_q$ and its concave hull, the smallest concave function on
$[0,1]$ that is at least $T_q$. Allaart~\cite{Allaart2014} proved
an inequality of approximate convexity for $T_q/2$ in every base. It
bounds the value of $T_q/2$ at the midpoint of two points by the mean of
its values at these points plus a constant times their distance. For
$q=3$, the function $T_3/2$ arose in a model of deterministic diffusion
studied by Gilbert, Dorfman, and
Gaspard~\cite[equation~(6)]{GilbertDorfmanGaspard2000}, and Lev~\cite{Lev2015}
studied it in connection with the edge isoperimetric problem for Cayley
graphs. In a Cayley graph of an abelian group, each element $a$ of the
group is joined to $a+g$ for every $g$ in a fixed generating set. For
$q\ge3$ the function $T_q$ is not a multiple of the Takagi--van~der~Waerden
function $\sum_{j=0}^{\infty}q^{-j}\dist(q^jx,\Z)$. Indeed, by \eqref{T0}
the difference $T_q(x)-T_q(qx)/q$ is $\tau_q(x)$, and the same difference
for the Takagi--van~der~Waerden function is $\dist(x,\Z)$. If $T_q$
were $c$ times this function, then $\tau_q$ would be $c\dist(\cdot,\Z)$,
which is impossible, because $\tau_q$ has the $q$ different slopes
$q-1-2k$ on $[0,1]$ and $c\dist(\cdot,\Z)$ has at most two.

Let $Q_n^q$ be the Hamming graph on $\{0,1,\ldots,q-1\}^n$, in which two
vertices are adjacent when they differ in exactly one coordinate. It is
the Cartesian product of $n$ complete graphs with $q$ vertices each,
where a graph is complete if any two of its vertices are adjacent. In
the Cartesian product of several graphs, the vertices are the sequences
that take one vertex from each graph, and two sequences are adjacent when
they differ in exactly one place, where their entries are adjacent. For
a set $A$ of vertices, the edge boundary $\partial A$ is the set of edges
with exactly one endpoint in $A$. The edge isoperimetric problem asks how
small the boundary of a set of $m$ vertices can be:
\[
 \xi_n(m)=\min\bigl\{|\partial A|:A\subseteq\{0,1,\ldots,q-1\}^n,\ |A|=m\bigr\},
 ~0\le m\le q^n.
\]
We call the numbers $\xi_n(m)$ the boundary counts of $Q_n^q$. In 1964
Lindsey~\cite[pp.~510--515]{Lindsey1964} solved this problem for Cartesian
products of complete graphs: when the factors are ordered by increasing
size, the sets of the first $m$ vertices in the lexicographic order have
the fewest boundary edges. In this order, two vertices are compared at the first coordinate in
which they differ. Let $s_q(k)$ be the sum of the base $q$ digits of $k$,
and let $S_q(m)=\sum_{k=0}^{m-1}s_q(k)$. Then
\begin{equation}
 \xi_n(m)=(q-1)nm-2S_q(m),
 \label{count}
\end{equation}
a formula known in several forms (see Section~\ref{sec:graph});
Table~\ref{tab1} lists these values for $q=3$, $n=3$, and $m\le13$. In
terms of $T_q$,
\begin{equation}
 \xi_n(m)=q^nT_q\!\left(\frac m{q^n}\right),~0\le m\le q^n
 \label{samp0}
\end{equation}
(see Theorem~\ref{sampthm}). We call
\eqref{samp0} the sampling equation. It follows from \eqref{count} and a
formula of Kr\"uppel~\cite[Theorem~5.2]{Kruppel2010b} for the values of
$T_q$ at the points $m/q^n$, and for $q=3$ it is, up to a factor $2$,
Lev's Theorem~1.3 combined with Lindsey's theorem (see
Section~\ref{sec:T}). The partial sums
$F_n(x)=\sum_{j=0}^{n-1}q^{-j}\tau_q(q^jx)$ are the linear interpolants of
these values. For the hypercube $Q_n^2$, the count \eqref{count} was
given by Hart~\cite[Theorem~1.5]{Hart1976}, who referred to earlier work
of Harper and Bernstein in a note added in proof.
Guu~\cite[page~165]{Guu2000} showed that the rescaled counts of the
hypercube converge pointwise to Takagi's function.

Kahane~\cite{Kahane1959} and Kr\"uppel~\cite{Kruppel2007} determined the
local extrema of Takagi's function. The points at which it has an
infinite derivative were characterized by Allaart and
Kawamura~\cite{AllaartKawamura2010} and by Kr\"uppel~\cite{Kruppel2010a}.
For every base $q$ we study how the base $q$ digits of a point
determine the local structure of $T_q$ at that point through the slopes
of the interpolants $F_n$. Since $F_n$ interpolates the values $\xi_n(m)/q^n$, its slope on
$[m/q^n,(m+1)/q^n]$ is $\xi_n(m+1)-\xi_n(m)$, which by \eqref{count} and
$S_q(m+1)-S_q(m)=s_q(m)$ equals $(q-1)n-2s_q(m)$. The rest of the series
is $\sum_{j=n}^{\infty}q^{-j}\tau_q(q^jx)=q^{-n}T_q(q^nx)$, and at
$x=(m+t)/q^n$ it equals $q^{-n}T_q(t)$ because $T_q$ has period $1$.
Hence
\begin{equation}
 T_q\!\left(\frac{m+t}{q^n}\right)
 =T_q\!\left(\frac m{q^n}\right)
 +q^{-n}\bigl(((q-1)n-2s_q(m))\,t+T_q(t)\bigr),
 ~0\le m<q^n,~0\le t\le1
 \label{cell0}
\end{equation}
(see Proposition~\ref{prop-cell}). For $n\ge1$,
Kr\"uppel~\cite[Theorem~5.2]{Kruppel2010b} proved \eqref{cell0} for a
class of Knopp functions that contains $T_q$. Such a function has the form \eqref{T0}, with $\tau_q$ replaced by a
continuous function of period $1$ that is zero at the integers and
linear on every interval $[k/q,(k+1)/q]$ with $k\in\Z$. A $q$-adic rational is
a number $m/q^n$ with
$m\in\Z$ and $n\ge0$. For $0\le m<q^n$ we call $[m/q^n,(m+1)/q^n]$ an
interval of level $n$, and we call the integer $(q-1)n-2s_q(m)$ its
slope (see Definitions~\ref{ternint} and~\ref{intslope}). Most of our
arguments rest on \eqref{cell0} and on the way these slopes change with
$m$ and $n$.

Throughout the paper we use the following notation. We write $\lfloor u\rfloor$ for the largest integer at most $u$ and $\lceil
u\rceil$ for the smallest integer at least $u$. We use the natural logarithm $\log$. For $x>0$ and $b>0$ with $b\ne1$,
we write $\log_bx=\log x/\log b$. The subscript $q$
in $T_q$, $\tau_q$, $s_q$, and $S_q$ shows the base. The partial sums
$F_n$, the boundary counts $\xi_n(m)$, and the other quantities derived
from $T_q$ also depend on $q$, which their notation does not show. When
we write the digits of a number as a string, the
digit of the highest power of the base comes first. A subscript after a
string gives the base, as in $12=(110)_3$. Each point $x\in[0,1)$ has a
base $q$ expansion $x=\sum_{j=1}^{\infty}\xi_jq^{-j}$ with digits
$\xi_j\in\{0,1,\ldots,q-1\}$, and at a $q$-adic rational we use the
terminating one. The digit $\xi_j$ is written without an argument, which
distinguishes it from the boundary count $\xi_n(m)$. We call the numbers
\[
 d_n(x)=\sum_{j=1}^{n}(q-1-2\xi_j),~n\ge0,
\]
the digit slopes of $x$. If $k=\lfloor q^nx\rfloor$, then
$\xi_1,\ldots,\xi_n$ are the base $q$ digits of $k$, written with $n$
digits, so that $d_n(x)=(q-1)n-2s_q(k)$ is the slope of the interval
$[k/q^n,(k+1)/q^n]$ of level $n$, the one with $k/q^n\le x<(k+1)/q^n$
(see \eqref{dnsq}).
A run of zeros of $x$ is a maximal string of consecutive digits $\xi_j$
equal to $0$, and runs of other digits are defined in the same way. A bar
marks the repeating part of an expansion, as in $1/4=(0.\overline{02})_3$.
For $x\ne y$ we write
\[
 \Delta(x,y)=\frac{T_q(y)-T_q(x)}{y-x}=\Delta(y,x)
\]
for the slope of $T_q$ between $x$ and $y$.

The notation $h\downarrow0$
means that $h>0$ tends to $0$, and $h\uparrow0$ means that $h<0$ tends to
$0$. We write $y\downarrow x$ and $y\uparrow x$ in the same sense. For a
function $f$ we denote by $f'_+(x)$ and $f'_-(x)$ the right and left
derivatives, that is, the limits of $(f(x+h)-f(x))/h$ as $h\downarrow0$
and as $h\uparrow0$, taken in the extended real line. If both are equal
to $+\infty$ we write $f'(x)=+\infty$, and similarly for $-\infty$. The
upper and lower limits of the same quotients as $h\downarrow0$ and as
$h\uparrow0$ are the four Dini derivatives of $f$ at $x$ (see
Definition~\ref{dini}). The
phrase ``almost every'' always refers to Lebesgue measure. A function $f\colon\R\to\R$
is H\"older continuous of order $\alpha$, where $0<\alpha\le1$, if there
is a constant $C$ such that $|f(x)-f(y)|\le C|x-y|^\alpha$ for all real $x$
and $y$. If this holds with $\alpha=1$, then $f$ is Lipschitz, and $C$ is
a Lipschitz constant of $f$. A set of real numbers is nowhere dense if
its closure contains no interval of positive length (see
Definition~\ref{baire}). The lower and upper box dimensions of a
nonempty bounded set in $\R$ or $\R^2$ are the lower and upper limits, as
$r\downarrow0$, of the logarithm of the smallest number of sets of
diameter at most $r$ that cover the set, divided by $\log(1/r)$. The
Assouad dimension of such a set is the infimum of the numbers $s\ge0$ for
which there is a constant $C$ such that, for all $0<r<R$, the part of the
set in every closed ball of radius $R$ about one of its points can be
covered by at most $C(R/r)^s$ sets of diameter at most $r$ (see
Definition~\ref{assouad}). The Hausdorff dimension of a set is the
infimum of the numbers $s\ge0$ such that, for every $\varepsilon>0$, the
set can be covered by countably many sets whose diameters, raised to the
power $s$, have a sum less than $\varepsilon$ (see
Definition~\ref{hausdorff}). 

We give another short proof of Delange's theorem that $T_q$ has no
finite derivative at any point (see Theorem~\ref{nodiff}). Our proof uses
\eqref{cell0} and the symmetry $T_q(1-t)=T_q(t)$. The slopes of $T_q$
may still tend to $+\infty$ or $-\infty$, as those of Takagi's function do
at many points. At a $q$-adic rational $a$, Proposition~\ref{cusp} gives
$(T_q)'_+(a)=+\infty$ and $(T_q)'_-(a)=-\infty$. For all other points our
first main result gives a complete criterion, in terms of the digits, for
the right and the left derivative to be infinite (see Figure~\ref{fig6} for $q=3$ and
Figure~\ref{sparsepic} for $q=3$, $4$, and $5$).

\begin{theorem}
\label{thm2}
Let $x\in(0,1)$ be a point that is not a $q$-adic rational, with digit
sequence $\xi_1,\xi_2,\ldots$ and digit slopes
$d_n(x)=\sum_{j=1}^{n}(q-1-2\xi_j)$,
let $n_1<n_2<\cdots$ be the positions of the nonzero digits of $x$, and let
$m_1<m_2<\cdots$ be the positions
of the digits different from $q-1$.
\begin{enumerate}
\item[(a)] $(T_q)'_+(x)=+\infty$ if and only if $d_n(x)\to+\infty$, and
 $(T_q)'_-(x)=-\infty$ if and only if $d_n(x)\to-\infty$.
\item[(b)] $(T_q)'_-(x)=+\infty$ if and only if
 \[
  d_{n_k}(x)-(q-1)(n_{k+1}-n_k)+(q-1)\log_q(n_{k+1}-n_k)\to+\infty
  ~(k\to\infty),
 \]
 and $(T_q)'_+(x)=-\infty$ if and only if
 \[
  d_{m_k}(x)+(q-1)(m_{k+1}-m_k)-(q-1)\log_q(m_{k+1}-m_k)\to-\infty
  ~(k\to\infty).
 \]
\item[(c)] Consequently $T_q'(x)=+\infty$ if and only if the first
 condition in (b) holds, and $T_q'(x)=-\infty$ if and only if the second
 one holds.
\end{enumerate}
\end{theorem}

For $q=2$, Theorem~\ref{thm2} is the characterization of Allaart and
Kawamura~\cite{AllaartKawamura2010} and Kr\"uppel~\cite{Kruppel2010a}
(see Section~\ref{sec:inf}). As for Takagi's function, only the digit
slopes $d_n(x)$ matter from one side. From the other side a long run of
zeros or of digits $q-1$ also matters, because it brings $x$ close to a
$q$-adic rational. Since $n_{k+1}-n_k-1$ is the number of zeros between
the $k$th nonzero digit and the next one, a long run of zeros lowers the
expression in the first condition of (b) by about $q-1$ times its length,
up to a logarithmic correction. In the same way, a long run of digits $q-1$
raises the expression in the second condition. The digits
$1,\ldots,q-2$ enter only through the digit slopes,
and for odd $q$ the middle digit $(q-1)/2$ does not change them, because
$\tau_q$ is constant on $[(q-1)/(2q),(q+1)/(2q)]$. By contrast, for the
Takagi--van~der~Waerden function of an odd base $r$, runs of the middle
digit $(r-1)/2$ also matter, so runs of digits enter the conditions from
both sides~\cite[Theorems~3.11 and~3.13]{FerreraGomezGilLlorente2019}.
Corollary~\ref{rat} records two special cases of Theorem~\ref{thm2}:
points whose runs of zeros, or whose runs of digits $q-1$, have bounded
length, and rational points.

Our second main result determines all local extrema of $T_q$ (see
Figure~\ref{extpic} for $q=3$, $4$, and $5$).

\begin{theorem}
\label{thm1}
The local minima of $T_q$ are exactly the $q$-adic rationals, and each of
them is a strict local minimum. At each $q$-adic rational $a$ we have
$T_q(a+h)-T_q(a)=(q-1)|h|\log_q(1/|h|)+O(|h|)$ as $h\to0$, where the implied
constant depends on $a$. The local maxima of $T_q$ are exactly the points
\[
 \frac{m+u}{q^n}+k,~k\in\Z,~n\ge0,~0\le m<q^n,~
 2s_q(m)=(q-1)n,
\]
where $u$ is a point of $[0,1]$ at which $T_q$ attains its maximum on
$[0,1]$. If $q$ is odd, then $u=1/2$ is the only such point, every local
maximum is strict, and the set of the local maxima is countable. If $q$
is even, then these points $u$ are the numbers
\[
 \sum_{i=0}^{\infty}\Bigl(\frac{\eta_i}{q^{2i+1}}
 +\frac{q-1-\eta_i}{q^{2i+2}}\Bigr)
 ~\text{with}~\eta_i\in\{q/2-1,q/2\}~\text{for every}~i\ge0,
\]
no local maximum is strict, and the set of the local maxima is
uncountable. In every base the set of the local minima is countable, and
the local minima and the local maxima are both dense in $\R$.
\end{theorem}

For $q=2$, Theorem~\ref{thm1} contains the results of
Kahane~\cite[Section~1]{Kahane1959} on Takagi's function: its local
minima are the dyadic rationals $m/2^n$, and its local maxima are the points
whose binary digits satisfy $\xi_1+\cdots+\xi_{2k}=k$ for all large $k$.
None of these maxima is strict, as Kr\"uppel~\cite[Proposition~4.4]{Kruppel2007}
showed. For $q=3$ we have $u=1/2=(0.111\ldots)_3$, so a point $x\in(0,1)$
is a local maximum exactly when $x=(0.\xi_1\ldots\xi_n111\ldots)_3$ with
$\xi_1+\cdots+\xi_n=n$. Odd and even bases differ because the maximum of
$T_q$ on $[0,1]$ is attained at the single point $1/2$ when $q$ is odd and
at uncountably many points when $q$ is even~\cite[Theorem~2 and
Corollary~1]{GalkinGalkina2020}.

The local maxima of $T_q$ are found through the maximum of $T_q$ on the
intervals of level $n$, which also involves the numbers
$\max_{0\le t\le1}(T_q(t)-bt)$ with integers $0\le b\le q-2$. By
Lemma~\ref{lem:tilt} below, such a maximum is $((q-b)^2+q)/(4q)$ minus
$0$, $1/(4q)$, $1/(4(q+1))$, or $1/(4q(q+1))$, depending on whether $q$ is
odd and $b$ is even, $q$ and $b$ are odd, $q$ and $b$ are even, or $q$ is
even and $b$ is odd. Our third main result gives the maximum of $T_q$ on
each interval of level $n$ and all points at which it is attained (see
Figure~\ref{fig4} for $q=3$, $4$, and $5$).

\begin{theorem}
\label{cellmax}
Let $n\ge0$ and $0\le m<q^n$, let $J=[m/q^n,(m+1)/q^n]$, and write the
absolute value of the slope $(q-1)n-2s_q(m)$ of $J$ as $a(q-1)+b$ with
integers $a\ge0$ and $0\le b\le q-2$. Then
\begin{equation}
 \max_{x\in J}T_q(x)
 =\max\bigl\{T_q(m/q^n),T_q((m+1)/q^n)\bigr\}
 +q^{-n-a}\max_{0\le t\le1}\bigl(T_q(t)-bt\bigr),
 \label{cellmax1}
\end{equation}
where the last maximum is given by \eqref{tilt0}. The maximum of $T_q$ on
$J$ is attained exactly at the points of $J$ that are not $q$-adic rationals and whose digit sequence has the following two properties. Its
digits with the indices $n+1,\ldots,n+a$ are all $0$ if the slope of $J$
is negative, and all $q-1$ if it is positive. Its digit slopes satisfy
$|d_k|\le1$ for every $k>n+a$. If $q$ is odd, then $b$ is even, and the
maximum is attained at exactly one point, namely
\begin{equation}
 \frac m{q^n}+\frac{q-b}{2q^{n+a+1}}
 ~\text{if the slope of $J$ is at most $0$,}~
 \frac{m+1}{q^n}-\frac{q-b}{2q^{n+a+1}}
 ~\text{if it is at least $0$.}
 \label{cellmax2}
\end{equation}
The two expressions agree when the slope is $0$. If $q$ is even, then
the maximum is attained at uncountably many points. For $q=3$, let
$d=n-s_3(m)$. Then the slope of $J$ is $2d$, $a=|d|$, and $b=0$, so that
the last term of \eqref{cellmax1} is $3^{-n-|d|}$ and the point
\eqref{cellmax2} is $m/3^n+1/(2\cdot3^{n+|d|})$ if $d\le0$ and
$(m+1)/3^n-1/(2\cdot3^{n+|d|})$ if $d\ge0$.
\end{theorem}

For $n=0$, Theorem~\ref{cellmax} is the theorem of Galkin and Galkina on
the maximum of $T_q$ (see Theorem~\ref{gmax}). For $q=2$,
Kr\"uppel~\cite[Proposition~4.3]{Kruppel2007} found the maximum of
Takagi's function on every dyadic interval $[m/2^n,(m+1)/2^n]$, and its value is the case
$q=2$ of \eqref{cellmax1}. In the same proposition he showed that the maximum of Takagi's
function on a dyadic interval on which the partial sum of Takagi's series has slope zero is
attained on a closed set of Lebesgue measure zero without isolated
points, and that the maximum on every other dyadic interval is attained
on such a set in a smaller dyadic interval of this kind.

Through the sampling equation \eqref{samp0}, results on $T_q$ become
statements about the Hamming graphs. Equation~\eqref{cell0} and
Lemma~\ref{lem:tilt} yield Corollary~\ref{blockmax}, which gives the
maximum of $\xi_n(m)$ over every block of sizes $aq^k\le m\le(a+1)q^k$.
The bound $T_q(x)\ge\lfloor(q+1)/2\rfloor x$ for $0\le x\le1/2$ (see
Proposition~\ref{T2x}) yields Corollary~\ref{xi2m}, which gives
$\xi_n(m)\ge\lfloor(q+1)/2\rfloor m$ for $m\le q^n/2$ and determines all
sizes $m$ for which equality holds. Theorem~\ref{rightmin} describes
the points $x\in[0,1/2]$ at which $T_q$ attains its minimum over
$[x,1/2]$. For $1\le h\le q^n/2$, the $h$-extra edge connectivity
$\lambda_h(Q_n^q)$ of $Q_n^q$ is the smallest number of edges whose
removal disconnects $Q_n^q$ into components that all have at least $h$
vertices. It is the smallest of the boundary counts $\xi_n(m)$ with
$h\le m\le q^n-h$ (see Lemma~\ref{extramin}), so by \eqref{samp0} and
\eqref{cell0} it equals $\xi_n(h)$ exactly when $T_q$ attains its minimum
over $[h/q^n,1/2]$ at $h/q^n$. Thus Theorem~\ref{rightmin} gives our
fourth main result.

\begin{theorem}
\label{lamdig}
Let $n\ge1$ and $1\le h\le q^n/2$, write $h=\sum_{i=0}^{n-1}x_iq^i$ with
digits $x_i\in\{0,1,\ldots,q-1\}$, and for $1\le j\le n$ let
$h_j=\lfloor h/q^{n-j}\rfloor$ be the integer with the digits
$x_{n-1},\ldots,x_{n-j}$. Then $\lambda_h(Q_n^q)=\xi_n(h)$ if and only if
\begin{equation}
 h-q^{n-j}h_j\le\bigl((q-1)j-2s_q(h_j)\bigr)\,q^{n-j-1}
 \label{lamdig1}
\end{equation}
for every $j$ with $1\le j<n$ such that one of the digits
$x_{n-1},\ldots,x_{n-j}$ is not $(q-1)/2$ and one of the digits
$x_{n-j-1},\ldots,x_0$ is not $0$. For $q=3$ this condition holds if and
only if $s_3(h_j)\le j-1$ for every $j$ with $1\le j<n$ such that one of the digits
$x_{n-1},\ldots,x_{n-j}$ is $0$ and one of the digits
$x_{n-j-1},\ldots,x_0$ is not $0$. Moreover,
$\lambda_h(Q_n^q)=\xi_n(h^*)$, where $h^*$ is the smallest integer $m$
with $h\le m\le q^n/2$ that satisfies the condition in place of $h$.
Every $h$ satisfies the condition if $n=1$, and also if $n=2$ and
$q\le3$. In all other cases the smallest $h$ that does not satisfy the
condition is
\[
 \begin{cases}
 \tfrac12(q+3)\,q^{(n-1)/2}-2&\text{if $n$ and $q$ are odd},\\
 \tfrac12(q+2)\,q^{(n-1)/2}-1&\text{if $n$ is odd and $q$ is even},\\
 2q^{n/2}-2&\text{if $n$ is even and $q\ge3$},\\
 2^{n/2+1}-3&\text{if $n$ is even and $q=2$}.
 \end{cases}
\]
For $q=3$ this is $3^{(n+1)/2}-2$ if $n\ge3$ is odd and $2\cdot3^{n/2}-2$
if $n\ge4$ is even.
\end{theorem}

Since $\xi_n(h^*)$ is given by \eqref{count}, Theorem~\ref{lamdig}
determines $\lambda_h(Q_n^q)$ for every $q$, every $n$, and every
$1\le h\le q^n/2$. In every base, Zhang, Tian, and
Zhang~\cite[Corollary~V.5]{ZhangTianZhang2021} and Xu and
Zhou~\cite[Corollary~5.2]{XuZhou2025} showed that
$\lambda_h(Q_n^q)=\xi_n(h)$ for $1\le h\le q^{\lfloor n/2\rfloor}$. Xu and
Zhou derived this from their results on the Cartesian products of $n$
copies of a regular graph, that is, of a graph in which all vertices have
the same number of neighbors. For $q=2$ and $n\ge6$, Li and
Yang~\cite[Theorem~6]{LiYang2013} had shown it earlier. For the hypercube $Q_n^2$, the last statement of
Theorem~\ref{lamdig} gives
the smallest size $h$ with $\lambda_h(Q_n^2)\ne\xi_n(h)$. For $q=3$, Xu,
Zhou, and Hsieh~\cite{Xu2022} showed that $\lambda_h(Q_n^3)=\xi_n(h)$ for
$1\le h\le3^{(n+1)/2}-3$ when $n\ge3$ is odd and for
$1\le h\le2\cdot3^{n/2}-3$ when $n\ge2$ is even, and they gave formulas
that express $\lambda_h(Q_n^3)$ for larger sizes $h$ through smaller
cubes. For $q=4$, Tian, Zhang, and Feng determined $\lambda_h(Q_n^4)$ on
several ranges of
sizes~\cite[Theorems~3.1, 3.2, and~3.3]{TianZhangFeng2022},
\cite[Theorem~6]{TianZhangFeng2024}, and they named its computation for
every $h\le2\cdot4^{n-1}$ as a task for future
work~\cite[Section~7]{TianZhangFeng2024}. For $q=5$, Liang, Zhang, and
Yang~\cite[Theorems~4.1 and~4.2]{LiangZhangYang2022} determined
$\lambda_h(Q_n^5)$ for $\lceil2\cdot5^{n-1}/3\rceil\le h\le5^{n-1}$ and
for $\lceil4\cdot5^{n-1}/3\rceil\le h\le2\cdot5^{n-1}$, and they noted
that the other sizes $h$ with $2\le h\le5^n/2$ remain
open~\cite[page~15529]{LiangZhangYang2022}. Theorem~\ref{lamdig} gives
$\lambda_h(Q_n^4)$ and $\lambda_h(Q_n^5)$ for these sizes as well.
Section~\ref{sec:extra} compares Theorem~\ref{lamdig} with these results
in detail. For $q=3$ and $q=2$, Propositions~\ref{optcount}
and~\ref{optcount2} count the sizes $h$ with $\lambda_h(Q_n^q)=\xi_n(h)$:
their proportion among the sizes $1\le h\le q^n/2$ is asymptotically
$\sqrt{3/(\pi n)}$ for $q=3$ and $\sqrt{2/(\pi n)}$ for $q=2$. For $q\ge4$
we have no formula of this kind.

Our last main result concerns the level sets
$L(y)=\{x\in[0,1]:T_q(x)=y\}$, where $0\le y\le\Lambda(0)$ and
$\Lambda(0)$ denotes the maximum of $T_q$, which is $(q+1)/4$ if $q$ is
odd and $q(q+2)/(4(q+1))$ if $q$ is even (see Theorem~\ref{gmax}).

\begin{theorem}
\label{levelthm}
\begin{enumerate}
\item[(a)] For almost every $y\in[0,\Lambda(0)]$, the level set $L(y)$ is
finite.
\item[(b)] Let $x\in[0,1]$ be a point that is not a $q$-adic rational, and
suppose that $d_n(x)=0$ and $2\xi_{n+1}\ne q-1$ for infinitely many $n$.
Then $L(T_q(x))$ is uncountable. Almost every $x\in[0,1]$ has this
property.
\end{enumerate}
\end{theorem}

If $q$ is even, the condition $2\xi_{n+1}\ne q-1$ in (b) holds for every
digit, and for $q=3$ it reads $\xi_{n+1}\ne1$. By (b), the points
$x\in[0,1]$ for which $L(T_q(x))$ is uncountable form a set of full
measure, and by (a) the heights $y$ for which $L(y)$ is uncountable form
a set of measure zero. Since $T_q$ maps the first set into the second,
it maps a set of full Lebesgue measure into a set of Lebesgue measure
zero. For $q=2$ both parts are known. Part
(a) is due to Buczolich~\cite[Theorem~9]{Buczolich2008}.
Allaart~\cite[Theorem~1.1]{AllaartLevelSets2014} extended Buczolich's
theorem to the Takagi--van~der~Waerden functions
$\sum_{j=0}^{\infty}r^{-j}\dist(r^jx,\Z)$ with every integer base $r\ge2$,
and he showed that for almost every $x$ the level set through $x$ is
uncountable. For $r=2$ this function is $T_2$. In the proof of (b) we
follow Allaart's argument in every base. For $q\le3$ we also follow his
argument for (a), which bounds the measure of the set of heights $y$ at
which $L(y)$ has at least $M$ points (see Corollary~\ref{levelcount}).
This argument uses two facts that hold only for $q\le3$: between two
consecutive indices $n$ with $d_n(x)=0$ the digit slopes keep one sign,
and $T_q$ is
strictly increasing on the set of points whose digit slopes are all
positive (see Remark~\ref{rem:signs} and Lemma~\ref{monopos}). For
$q\ge4$ a single digit can change the sign of the digit slope, and for $q=4$ and $q=5$ the
conclusion of Lemma~\ref{monopos} fails at explicit points. In every
base we prove (a) in another way, with the help of Theorem~\ref{thm2}.

Section~\ref{sec:inf} also treats the Dini derivatives. For
odd $q$, Proposition~\ref{midpt} gives the four Dini derivatives at the
midpoints of the intervals of each level, and Proposition~\ref{quarter}
at the point $(q-1)/(2(q+1))$, which is $1/4$ for $q=3$. For $q=3$ we show that
the two right Dini derivatives differ by at least $1$ wherever both are
finite, with equality exactly at these midpoints, and by at least $9/5$ at
every other such point (see Theorem~\ref{dinigap}). We do not know the
corresponding constants for other bases.
Proposition~\ref{sparse} treats the points
$\sum_{j=1}^{\infty}(q-1)q^{-p_j}$ with integers $1\le p_1<p_2<\cdots$.

Section~\ref{sec:ae} treats the slopes of $T_q$ at almost every point.
For almost every $x$, the slopes $\Delta(x,x+h)$ divided by
$\sqrt{\frac23(q^2-1)\log_q(1/|h|)\log\log\log_q(1/|h|)}$ have the upper
limit $1$ and the lower limit $-1$ as $h\downarrow0$ and as $h\uparrow0$.
This law of the iterated logarithm is Theorem~\ref{liltheorem}. If $x$ is
chosen uniformly in $[0,1]$, then for a fixed small $h$ the slopes
$\Delta(x,x+h)$ are approximately normally distributed with variance
$\frac13(q^2-1)\log_q(1/h)$ (see Theorem~\ref{clttheorem}).
Section~\ref{sec:ae} cites the results of the same kind for Takagi's
function and for the Takagi--van~der~Waerden functions. By
Corollary~\ref{aedini}, the points at which $T_q$ has a derivative from at
least one side, finite or infinite, form a set of Lebesgue measure zero,
while by Proposition~\ref{infdim} the points with an infinite derivative
are dense and form a set of Hausdorff dimension $1$.
Theorem~\ref{slopegrowth} relates the growth of the slopes to the
proportions of the digits, and Corollary~\ref{spectrum} gives the
Hausdorff dimension of the corresponding sets of points by a theorem of
Volkmann~\cite[Satz~3]{Volkmann1958}.

Section~\ref{sec:level} contains further results on the level sets. For
each integer $e$ with $0\le e<(q-1)/2$ there is a closed set of Hausdorff
dimension $\log2/(2\log q)$, described by the digits of its points, on
which $T_q$ takes the constant value $(e+1)(q-e)/(q+1)$ (see
Lemma~\ref{Efacts}). For $q=3$ and $e=0$ this value is $3/4$, and
$L(3/4)$ is the union of this set and an explicit countable set (see
Theorem~\ref{threequarters}). This level set has Assouad dimension
$\log2/\log9$ (see Proposition~\ref{assouadlevel}). For $q\ge3$ every level set has Lebesgue
measure zero and upper box dimension at most $\log(q^2-2)/(2\log q)$ (see
Corollary~\ref{levelbox}). Following Lagarias and
Maddock~\cite[Definition~1.2]{LagariasMaddock2012}, we define the local
level set of a point $x$ that is not a $q$-adic rational. It consists of
the points obtained from $x$ by replacing the digits $\xi_j$ by
$q-1-\xi_j$ for all $j$ in some of the following ranges: the ranges
$n<j\le n'$ between consecutive indices $n<n'$ with $d_n(x)=d_{n'}(x)=0$
and, if there is a largest index $n$ with $d_n(x)=0$, the range $j>n$.
For $q\le3$ we show that the level set at a height chosen uniformly in
the range of $T_q$ contains on average two local level sets for $q=3$
and $3/2$ for $q=2$ (see Theorem~\ref{expected}), and we describe the numbers of
points of the finite level sets (see Theorem~\ref{cardmod}). For $q=2$ the
average $3/2$ is due to Lagarias and
Maddock~\cite[Theorem~1.6]{LagariasMaddock2012}. In every base, the
heights with uncountable level sets form a set of Lebesgue measure zero
and Hausdorff dimension $1$ whose complement in the range of $T_q$ is a
countable union of nowhere dense sets (see Theorem~\ref{heightdim}). For
$q=2$ the statement on the Hausdorff dimension is due to Lagarias and
Maddock~\cite[Theorem~1.5]{LagariasMaddock2011}. Finally, the Lebesgue
measure of the set of the points $x\in[0,1]$ with $T_q(x)\le y$ is a
continuous and strictly increasing function of $y$ on the range of
$T_q$, but a set of heights of Lebesgue measure zero contains $T_q(x)$
for almost every $x$ (see Proposition~\ref{distfn}).

Section~\ref{sec:X} determines the steps $0\le\delta\le1/2$ at which
$T_q(\delta)$ equals the modulus of continuity of $T_q$, that is, the
supremum of $|T_q(x)-T_q(y)|$ over $|x-y|\le\delta$ (see
Theorems~\ref{runmax} and~\ref{lexX}). These steps form a compact set $X$
of Lebesgue measure zero and Hausdorff dimension $1$ (see
Proposition~\ref{modset}). For even $q$ and for $q=3$ the set $X$ has no
isolated points, but for $q=9$ the point $227/486$ is an isolated point of
$X$, and the same happens whenever $q$ is the square of an odd integer
(see Proposition~\ref{Xperfect}). We do not know whether $X$ has isolated
points when $q\ge5$ is odd and is not a square, for instance when $q=5$
or $q=7$.

Section~\ref{sec:convex} proves two inequalities of approximate
convexity. Theorem~\ref{qpoint} bounds $T_q$ at the mean of $q$ real
numbers, and Theorem~\ref{lamq} gives
\[
 T_q\bigl(\lambda x+(1-\lambda)y\bigr)\le\lambda T_q(x)+(1-\lambda)T_q(y)
 +T_q(\lambda)|x-y|
\]
for all real $x$, $y$ and $0\le\lambda\le1$, with the best possible
factor $T_q(\lambda)$ (see Corollary~\ref{modulus}). For $q=2$ the first
inequality is due to Boros~\cite{Boros2008}, for $q=3$ both are due to
Lev~\cite{Lev2015}, and for $\lambda=1/2$ the second is Allaart's
inequality. Proposition~\ref{twobases} compares $T_q$ with Takagi's
function: $T_q\le\lfloor(q+1)/2\rfloor T_2$.

Section~\ref{sec:feq} shows which of the functional equations of $T_q$
determine it. Proposition~\ref{paramderq} expresses $T_q$ through random
base $q$ expansions: if the digits are independent and the digit $i$ has
the probability $1/q+(q-1-2i)s/2$, then $\frac q2T_q(x)$ is the
derivative at $s=0$ of the probability that the random number is at most
$x$. For $q=2$ this is a result of
Hata and Yamaguti, for $q=3$ a result of Okada, Sekiguchi, and Shiota,
and for $q\ge4$ it follows from their results by a short computation.

Sections~\ref{sec:fourier} and~\ref{sec:var} give two results on $T_q$
over a full period, which for $q=2$ are results on Takagi's function
presented in the survey of Lagarias~\cite[Theorems~6.1 and~11.4]{Lagarias2012}.
For an integer $k$, the Fourier coefficient of $T_q$ is
$\widehat T_q(k)=\int_0^1T_q(x)e^{-2\pi ikx}\,dx$, where $i^2=-1$.
By Theorem~\ref{fourier}, we have $\widehat T_q(0)=(q+1)/6$ and, for
$k\ne0$, $\widehat T_q(k)=-q^{1-v}/(2\pi^2m^2)$, where $v\ge0$ and
$m\ge1$ are integers with $|k|=q^vm$ and $q\nmid m$. Let $V_n$ be the variation of $T_q$ over the
partition of $[0,1]$ into the intervals of level $n$, that is, the sum of
$|T_q((m+1)/q^n)-T_q(m/q^n)|$ over $0\le m<q^n$. Theorem~\ref{varthm}
expresses $V_{n+1}-V_n$ through the numbers of intervals of level $n$
with each slope and shows that $V_n/\sqrt n\to\sqrt{2(q^2-1)/(3\pi)}$. The
graph of $T_q$ over $[0,1]$ has Assouad dimension $1$ (see
Proposition~\ref{assouadgraph}).

Besides the five main theorems, the paper studies in every base further
properties of $T_q$ of the kinds that have been studied for Takagi's
function: the Dini derivatives, the slopes at almost every point, the
finer structure of the level sets, the modulus of continuity, approximate
convexity, the functional equations, the Fourier coefficients, and the
variation at each level. These results hold in every base unless we say
otherwise, and the questions that remain open for some bases are
collected in Section~\ref{sec:questions}. For $q=2$ many of them are
known results on Takagi's function, which we cite where we prove the
general statements.

The main new results of this paper are Theorems~\ref{thm2},
\ref{thm1}, \ref{cellmax}, \ref{lamdig}, and~\ref{levelthm} for $q\ge3$.
In Theorem~\ref{lamdig} the new part is the criterion for
$\lambda_h(Q_n^q)=\xi_n(h)$ and the formula
$\lambda_h(Q_n^q)=\xi_n(h^*)$. For $q\ge3$ the following results are
also new: Corollaries~\ref{peakeven} and~\ref{blockmax}, the cases of equality in
Corollary~\ref{xi2m}, and Proposition~\ref{optcount} on the Hamming
graphs, and for $q=2$ Proposition~\ref{optcount2}; Theorem~\ref{rightmin}; Theorems~\ref{runmax} and~\ref{lexX} and
Propositions~\ref{modset} and~\ref{Xperfect} on the set $X$;
Propositions~\ref{midpt}, \ref{quarter}, and~\ref{sparse} and
Theorem~\ref{dinigap} on the Dini derivatives; Theorem~\ref{slopegrowth},
Propositions~\ref{infdim} and~\ref{onesided}, and
Corollary~\ref{spectrum} on the slopes; Corollary~\ref{modulus},
Proposition~\ref{twobases}, and, in Proposition~\ref{convexeq}, the left
inequality in (b), the cases of equality in (b) and (c), and inequality
(c) for points $c$ other than the midpoint of $a$ and $b$, on
approximate convexity; Theorem~\ref{oneeq} and
Propositions~\ref{threeeq}, \ref{replic}, and~\ref{anysol} on the
functional equations; Theorem~\ref{fourier}; the formula~\eqref{Vrec}
for the variation at each level; and, on the level sets,
Lemma~\ref{Efacts}, Theorem~\ref{threequarters},
Proposition~\ref{assouadlevel} for $q=3$, Proposition~\ref{finitelevel},
Corollaries~\ref{levelcount} and~\ref{levelbox},
Proposition~\ref{distfn}(a) and~(b), Theorems~\ref{expected}
and~\ref{cardmod}, Corollary~\ref{twopoints}, and
Theorem~\ref{heightdim}. Theorems~\ref{liltheorem} and~\ref{clttheorem}
and Proposition~\ref{assouadgraph} are new when $q$ is not a power of
$2$, and Theorem~\ref{lamq} and the bound \eqref{paramest} in
Proposition~\ref{paramderq} are new for $q\ge4$. For $q\ge4$,
Theorem~\ref{qpoint} is new for real numbers
$x_1,\ldots,x_q$ that do not all lie in one interval $[k,k+1]$ with an
integer $k$. To the best of our knowledge, these results are new. Some of
them are short consequences of general theorems: Corollary~\ref{spectrum}
combines a theorem of Volkmann with Theorem~\ref{slopegrowth},
Proposition~\ref{assouadlevel} follows from a theorem of K\"aenm\"aki and
Lehrb\"ack, Theorem~\ref{heightdim}(a) follows from a theorem of Garg, and
Proposition~\ref{twobases} combines the midpoint inequality of Allaart
with a theorem of H\'azy and P\'ales on approximately convex functions.

Several of these results have known special cases, which we cite where
we prove the general statements. If $q$ is a power of $2$, then $T_q$ is
a series $\sum_{m=0}^{\infty}2^{-m}\gamma_m\dist(2^mx,\Z)$ with bounded
numbers $\gamma_m$ (see Remark~\ref{pow2}), and Theorem~\ref{liltheorem}
follows from a theorem of
K\^ono~\cite[Theorem~5]{Kono1987} and Allaart~\cite[Theorem~5.2]{Allaart2009}.
Their theorem lets $h$ tend to $0$ from both sides together, and
Allaart's proof~\cite[page~253]{Allaart2009} treats $h\downarrow0$ and
$h\uparrow0$ separately. Theorem~\ref{clttheorem} follows
from a theorem of Nakano~\cite[Corollary~2.2(i)]{Nakano2026}, and
Proposition~\ref{assouadgraph} follows from a theorem of
Jiang~\cite[Theorem~1.1]{Jiang2025}. For $x_1,\ldots,x_q$ in $[0,1]$,
Theorem~\ref{qpoint} is equivalent to an inequality of Mohanty
et~al.~\cite[Theorem~5.1]{Mohanty2023} for the summatory function $S_q$,
and since $T_q$ has period $1$, the same holds for $x_1,\ldots,x_q$ in
every interval $[k,k+1]$. The right inequality in
Proposition~\ref{convexeq}(b) with its best constant, and the case of (c)
in which $c$ is the midpoint of $a$ and $b$, are due to
Allaart~\cite[Theorems~3 and~11]{Allaart2014}. For even $q$ and
$e=q/2-1$, the set in Lemma~\ref{Efacts} is the set of the points at
which $T_q$ attains its maximum, and Galkin and
Galkina~\cite[Corollary~1]{GalkinGalkina2020} computed its Hausdorff
dimension. For $q=2$ several further results are known for Takagi's
function. For instance, Sadowski~\cite[equation~(1)]{Sadowski2017}
considered equation \eqref{replic1} of Proposition~\ref{replic} for
$q=2$. For continuous
functions, a theorem of Barnsley~\cite[Theorem~1]{Barnsley1986}
shows that the $q$ equations
\eqref{threeeqs} together determine $T_q$ (see Section~\ref{sec:feq}).
Proposition~\ref{ratval}, the formula for the values of $T_q$ at rational
points, follows from the functional equation by the argument that
Lagarias gave for Takagi's function.

Results of the same kind are known for other functions. Ferrera, G\'omez
Gil, and Llorente characterized the infinite derivatives of the
Takagi--van~der~Waerden functions $\sum_{j=0}^{\infty}r^{-j}\dist(r^jx,\Z)$
of every integer base $r\ge2$~\cite{FerreraGomezGilLlorente2019}, and for
even $r$ they described the local maxima of these functions in terms of
the base $r$ digits~\cite[Theorems~1.3 and~1.4 and
Corollary~4.6]{FerreraGomezGilLlorente2022}. Dalaklis, Kawamura, Mathis,
and Paizanis~\cite{Dalaklis2023} characterized the infinite derivatives
of the ternary series $\sum_{n=0}^{\infty}3^{-n}g(3^nx)$, where $g$ is the
continuous function of period $1$ with $g(0)=0$ and with slopes $3$,
$-6$, and $3$ on the three thirds of $[0,1]$. For the power Takagi
functions $\sum_{n=0}^{\infty}\bigl(2^{-n}\dist(2^nx,\Z)\bigr)^p$ with
$0<p<1$, Galkin, Galkina, and
Mulyar~\cite[Theorem~4.2]{GalkinGalkinaMulyar2023} showed that the local
minima are exactly the dyadic rationals $m/2^n$ and are strict, and that
the points $k/(3\cdot2^n)$ with $3\nmid k$ are strict local maxima.

By Lemma~\ref{lem:subtract}, the slope of $T_q$ between two points of an
interval of level $n$ is the slope of the interval plus the slope of
$T_q$ between the corresponding points of $[0,1]$. Together with the
bounds of Proposition~\ref{incr} for $T_q(t)/t$, this reduces
Theorem~\ref{thm2} to estimates for the digit slopes, and the cusp of
Proposition~\ref{cusp} at the $q$-adic rationals gives the logarithmic
terms in Theorem~\ref{thm2}(b). By \eqref{cell0}, the maximum of $T_q$
on $[m/q^n,(m+1)/q^n]$ is $T_q(m/q^n)$ plus $q^{-n}$ times the maximum
of $T_q(t)+ct$ over $[0,1]$, where $c$ is the slope of the interval.
Lemma~\ref{lem:tilt} computes this maximum for every integer $c$ from the
formula \eqref{dsq} for $T_q$ in terms of the squares of the digit
slopes, and this gives Theorem~\ref{cellmax}. A local maximum of $T_q$
is a point at which $T_q$ attains its maximum on every short interval of
level $n$ that contains it, and Theorem~\ref{thm1} follows from
Theorem~\ref{cellmax} and Proposition~\ref{cusp}. In the proof of
Theorem~\ref{levelthm}(b), the digits $\xi_j$ of a point $x$ with
$n<j\le n'$, where $d_n(x)=d_{n'}(x)=0$, are replaced by their
complements $q-1-\xi_j$, which does not change the value of $T_q$ (see
Lemma~\ref{flip}). The proof of (a) uses the intervals of slope $0$ of
the levels $n$ with $d_n(x)=0$ that contain $x$, together with
Theorem~\ref{thm2}.

The paper is organized as follows. Section~\ref{sec:graph} recalls
\eqref{count}, and Section~\ref{sec:T} recalls \eqref{samp0} and
\eqref{cell0} and collects the properties of $T_q$ that the later
sections use. Section~\ref{sec:inf} contains the proof of
Theorem~\ref{thm2}, and Section~\ref{sec:max} contains the proofs of
Theorems~\ref{cellmax} and~\ref{thm1}. Section~\ref{sec:ae} treats the
slopes of $T_q$ at almost every point, Section~\ref{sec:level} treats
the level sets and proves Theorem~\ref{levelthm}, and
Section~\ref{sec:extra} applies the results on $T_q$ to the Hamming
graphs and proves Theorem~\ref{lamdig}. Sections~\ref{sec:X}, \ref{sec:convex},
\ref{sec:feq}, \ref{sec:fourier}, and~\ref{sec:var} treat
the set $X$, approximate convexity, the functional equations, the
Fourier coefficients, and the variation, respectively. Section~\ref{sec:questions}
summarizes the results and states open questions.

\section{Boundary counts and digit sums}
\label{sec:graph}

Throughout the paper, the base $q\ge2$ is a fixed integer.
A statement that does not hold for every $q\ge2$ says for which $q$ it
holds. The figures show the case $q=3$ unless their captions say
otherwise.

This section recalls the boundary count \eqref{count} and its first
difference $\xi_n(m+1)-\xi_n(m)$ for an alphabet with $q$ letters.
Throughout this section we assume $n\ge1$.

\begin{definition}
\label{cube}
The Hamming graph $Q_n^q$ is the graph on the vertex set
$\{0,1,\ldots,q-1\}^n$ in which two vertices are adjacent when they
differ in exactly one coordinate. For $q=3$ it is the ternary Hamming
cube $Q_n^3$. For a set $A\subseteq\{0,1,\ldots,q-1\}^n$, let $e(A)$ be
the number of edges with both endpoints in $A$, and let the edge boundary
$\partial A$ be the set of edges with exactly one endpoint in $A$. For
$0\le m\le q^n$, the boundary count $\xi_n(m)$ is the smallest size of the
edge boundary of a set of $m$ vertices:
\[
 \xi_n(m)=\min\bigl\{|\partial A|:A\subseteq\{0,1,\ldots,q-1\}^n,\ |A|=m\bigr\}.
\]
\end{definition}

The graph $Q_n^q$ is the Cartesian product of $n$ complete graphs with
$q$ vertices each, so each of its $q^n$ vertices has degree $(q-1)n$. If
we add up the degrees of the vertices of a set $A$, we count every edge
inside $A$ twice and every boundary edge once, so
\begin{equation}
 |\partial A|=(q-1)n|A|-2e(A).
 \label{edges}
\end{equation}
By \eqref{edges}, a set of size $m$ has the smallest boundary exactly
when it has the largest number of internal edges. The same equation
shows that $\xi_n(m)$ has the same parity as $(q-1)nm$. In particular,
every value $\xi_n(m)$ is even when $q$ is odd. Since a set and its
complement have the same boundary, we also have
$\xi_n(m)=\xi_n(q^n-m)$.

\begin{definition}
\label{labels}
The label of the vertex $(x_0,\ldots,x_{n-1})$ of $Q_n^q$ is the integer
$\sum_{i=0}^{n-1}x_iq^i$, and the word of the label is the string
$x_{n-1}\ldots x_0$, with leading zeros allowed. For
$0\le m\le q^n$, the initial segment $I_n(m)$ is the set of the $m$
vertices with labels $0,\ldots,m-1$.
\end{definition}

Thus the labels are $0,1,\ldots,q^n-1$, and $x_i$ is the base $q$ digit
of the label at the place of $q^i$. In the proof of his Theorem~1,
Lindsey~\cite[pp.~510--515]{Lindsey1964}
showed that these initial segments have the largest possible number of
internal edges, and his argument covers every Cartesian product of
complete graphs. In his terminology, the initial segments have maximal
\emph{connectedness}. Bezrukov~\cite{Bezrukov1999} stated Lindsey's theorem
in modern notation and gave further references.

\begin{lemma}[Lindsey {\cite[pp.~510--515]{Lindsey1964}}]
\label{lindsey}
For $0\le m\le q^n$, the set $I_n(m)$ satisfies
$e(I_n(m))=\max\{e(A):|A|=m\}$.
\end{lemma}

Let $s_q(k)$ be the sum of the base $q$ digits of the nonnegative
integer $k$, and let
\[
 S_q(m)=\sum_{k=0}^{m-1}s_q(k),~ S_q(0)=0.
\]
Mohanty et~al.\ proved that in an initial segment the number of internal
edges is a sum of digit sums. They proved it for the Hamming graph of the
words of a fixed length over any finite alphabet. In this graph two words are
adjacent when they differ in exactly one letter. We state their theorem
for the alphabet $\{0,1,\ldots,q-1\}$, for which this graph is $Q_n^q$.

\begin{theorem}[Mohanty et~al. {\cite[Theorem~3.1]{Mohanty2023}}]
\label{mohanty}
For $0\le m\le q^n$ we have $e(I_n(m))=S_q(m)$.
\end{theorem}

\begin{corollary}
\label{xiformula}
For $0\le m\le q^n$,
\begin{equation}
 \max_{|A|=m}e(A)=S_q(m)
 ~\text{and}~
 \xi_n(m)=(q-1)nm-2S_q(m).
 \label{xi}
\end{equation}
Consequently, for $0\le m<q^n$,
\begin{equation}
 \xi_n(m+1)-\xi_n(m)=(q-1)n-2s_q(m).
 \label{diff}
\end{equation}
\end{corollary}

\begin{proof}
By Lemma~\ref{lindsey} and Theorem~\ref{mohanty},
$\max_{|A|=m}e(A)=e(I_n(m))=S_q(m)$. The formula for $\xi_n(m)$ then
follows from \eqref{edges}, and \eqref{diff} follows from
$S_q(m+1)-S_q(m)=s_q(m)$.
\end{proof}

Zhang, Tian, and
Zhang~\cite[Corollary~III.2]{ZhangTianZhang2021} stated an explicit
formula, in terms of the base $q$ digits of $m$, for the largest sum of
the degrees in the subgraph induced by a set of $m$ vertices of $Q_n^q$,
that is, for $2\max_{|A|=m}e(A)$.

For $q=3$, the same count appears in the literature on the extra edge
connectivity of $Q_n^3$. Xu, Zhou, and Hsieh~\cite[Lemma~2.4]{Xu2022}
recorded that the initial segments maximize twice the number of internal
edges, wrote $2e(I_n(m))$ as an explicit sum over the nonzero ternary
digits of $m$, and noted the increment $2s_3(m)$~\cite[Lemma~2.7]{Xu2022}.
In their work, the boundary is minimized over the sets $A$ such that both
$A$ and its complement are connected. By Lemma~\ref{lindsey} and
\eqref{edges}, the set $I_n(m)$ has the smallest boundary among the sets
of $m$ vertices. By the next lemma, which Xu, Zhou, and
Hsieh~\cite[Lemma~2.3]{Xu2022} stated for $q=3$, this set and its
complement are connected, so this restriction does not change the minimum.

\begin{lemma}
\label{conn}
For $1\le m<q^n$, the set $I_n(m)$ and its complement are both connected,
that is, any two vertices of one of these sets are joined by a path in
$Q_n^q$ whose vertices all lie in that set.
\end{lemma}

\begin{proof}
Replacing a nonzero digit of the label of a vertex by $0$ gives a
neighbor with a smaller label, so every vertex of $I_n(m)$ is joined to
the vertex $0$ inside $I_n(m)$. In the same way, replacing a digit
different from $q-1$ by $q-1$ joins every vertex of the complement to the
vertex with the label $q^n-1$ inside the complement.
\end{proof}

\begin{table}[H]
\centering
\caption{The case $q=3$: the boundary counts $\xi_3(m)=6m-2S_3(m)$ of
$Q_3^3$ for $0\le m\le13$.}
\label{tab1}
\renewcommand{\arraystretch}{1.1}
\begin{tabular}{@{}r c r r r@{}}
\toprule
$m$ & ternary digits of $m$ & $s_3(m)$ & $S_3(m)$ & $\xi_3(m)$\\
\midrule
 0 & $000$ & 0 &  0 &  0\\
 1 & $001$ & 1 &  0 &  6\\
 2 & $002$ & 2 &  1 & 10\\
 3 & $010$ & 1 &  3 & 12\\
 4 & $011$ & 2 &  4 & 16\\
 5 & $012$ & 3 &  6 & 18\\
 6 & $020$ & 2 &  9 & 18\\
 7 & $021$ & 3 & 11 & 20\\
 8 & $022$ & 4 & 14 & 20\\
 9 & $100$ & 1 & 18 & 18\\
10 & $101$ & 2 & 19 & 22\\
11 & $102$ & 3 & 21 & 24\\
12 & $110$ & 2 & 24 & 24\\
13 & $111$ & 3 & 26 & 26\\
\bottomrule
\end{tabular}
\end{table}

Table~\ref{tab1} lists the values for $q=3$, $n=3$, and $m\le13$, and the
remaining values follow from $\xi_3(m)=\xi_3(27-m)$. Its second column
gives the vertex with label $m$ as the string $x_2x_1x_0$ of its ternary
digits. The rows are therefore in lexicographic order, and the $m$
vertices in the rows above row $m$ form the set $I_3(m)$, which has
exactly $\xi_3(m)$ boundary edges. For $1\le m\le13$, the values of
$\xi_3(m)$ also appear in Table~II of Xu, Zhou, and Hsieh~\cite{Xu2022}.
In Section~\ref{sec:extra} we use these values to illustrate
Corollaries~\ref{blockmax} and~\ref{xi2m}. The numbers $s_3(k)$ form
the sequence A053735~\cite{OEIS-A053735}, and for $m\ge1$ the value $S_3(m)$ is the
term with index $m-1$ of the sequence A094345~\cite{OEIS-A094345}.

\section{The Delange function of base \texorpdfstring{$q$}{q} and its partial sums}
\label{sec:T}

This section defines the Delange function $T_q$ of base $q$ and proves
the sampling equation \eqref{samp0} in every base. It also states the
theorem of Kr\"uppel that contains \eqref{cell0} and collects the
properties of $T_q$ used later in the paper. Given integers
$i,m\ge0$, we write $r_i(m)$ for the remainder of $m$ modulo $q^{i+1}$,
taken in $\{0,1,\ldots,q^{i+1}-1\}$.

\begin{definition}
\label{delange}
Let $\tau_q$ be the continuous function on $\R$ with period $1$ that is
given on $[0,1]$ by
\begin{equation}
 \tau_q(x)=(q-1-2k)x+\frac{k(k+1)}{q}
 ~\text{for}~\frac kq\le x\le\frac{k+1}q,~k=0,1,\ldots,q-1.
 \label{tau}
\end{equation}
For $n\ge0$ let
\begin{equation}
 F_n(x)=\sum_{j=0}^{n-1}q^{-j}\tau_q(q^jx),
 ~
 T_q(x)=\sum_{j=0}^{\infty}q^{-j}\tau_q(q^jx),
 \label{T}
\end{equation}
so that $F_0=0$. In the notation of Galkin and
Galkina~\cite[equations~(1.1), (1.2), and~(2.2)]{GalkinGalkina2020},
$\tau_q=D_{q,0}$ and $T_q=D_q$. Following them, we call $T_q$ the Delange
function of base~$q$. We call $F_n$ its partial sum of order $n$.
\end{definition}

The two expressions that \eqref{tau} gives at a point $(k+1)/q$ with
$0\le k\le q-2$ are both equal to $(k+1)(q-1-k)/q$, and
$\tau_q(0)=\tau_q(1)=0$, so the definition is consistent. Moreover,
$\tau_q(k/q)=k(q-k)/q$ for $0\le k\le q$, and the slopes $q-1-2k$ of
$\tau_q$ on the intervals $[k/q,(k+1)/q]$ decrease as $k$ increases, so
$\tau_q$ is concave on $[0,1]$. The slope on
$[(q-1-k)/q,(q-k)/q]$ is the negative of the slope on $[k/q,(k+1)/q]$.
Hence $\tau_q(1-x)=\tau_q(x)$ for $0\le x\le1$, and since $\tau_q$ has period
$1$, it is even. A concave function on $[0,1]$ that is zero at both
endpoints and symmetric about $1/2$ is nonnegative and attains its
maximum at $1/2$, so
\[
 0\le\tau_q(x)\le\tau_q(1/2)=\frac1q\Bigl\lfloor\frac{q^2}4\Bigr\rfloor
 \le\frac q4~\text{for all real }x.
\]
For $q=2$ and $q=3$ we obtain the functions $\tau_2(x)=\dist(x,\Z)$ and
$\tau_3(x)=2\min\{\dist(x,\Z),1/3\}$ of Section~\ref{intro}, and $T_2$ is
Takagi's function~\cite{Takagi1903}. Since $\tau_q$ is bounded, the series
in \eqref{T} converges uniformly on $\R$, so $T_q$ is also continuous,
even, and periodic with period $1$, and $T_q(1-x)=T_q(x)$.

\begin{definition}
\label{ternint}
A $q$-adic rational is a real number of the form $m/q^n$ with $m\in\Z$
and $n\ge0$. For $n\ge0$ and $0\le m<q^n$, the interval
$[m/q^n,(m+1)/q^n]$ is an interval of level $n$, or a level $n$
interval. For $q=3$ we also say ternary rational and ternary interval.
\end{definition}

\begin{theorem}[Galkin and Galkina {\cite[equations~(2.4) and~(2.6) and Theorem~2]{GalkinGalkina2020}}]
\label{limitthm}
For every real $x$,
\begin{equation}
 T_q(x)=\tau_q(x)+\tfrac1qT_q(qx),
 \label{fe}
\end{equation}
and for every $n\ge0$,
\begin{equation}
 T_q(x)-F_n(x)=q^{-n}T_q(q^nx).
 \label{eq:tail}
\end{equation}
Moreover, $T_q\ge0$, and if $q$ is odd, then $T_q\le(q+1)/4$ with
$T_q(1/2)=(q+1)/4$.
\end{theorem}

Galkin and Galkina stated \eqref{eq:tail} for $n\ge1$, in the form
$D_q(x)=D_{q,n-1}(x)+q^{-n}D_q(q^nx)$ with $D_{q,n-1}=F_n$, and they
stated the bounds for $0\le x\le1$. For $n=0$, equation \eqref{eq:tail}
holds because $F_0=0$, and the bounds extend to $\R$ because $T_q$ has
period $1$. For $q=3$ we obtain $0\le T_3\le1$ and $T_3(1/2)=1$.

Write $\Z_q$ for the integers modulo $q$ and $\Z_q^n$ for the product of
$n$ copies of this group. Lev~\cite{Lev2015} used the functions
$\omega_q(x)=\sum_{j=0}^{\infty}q^{-j}\min\{\dist(q^jx,\Z),1/q\}$, so that
$T_3=2\omega_3$. For the unit vectors $e_1,\ldots,e_n$ of $\Z_3^n$, where $e_i$ has entry
$1$ in coordinate $i$ and $0$ elsewhere, his directed Cayley graph is the orientation of $Q_n^3$ in which each edge
$\{a,a+e_i\}$ appears once, as $a\to a+e_i$, and he noted that his count
of the directed edges leaving $A$ is half of
$|\partial A|$~\cite[page~2390]{Lev2015}. Indeed, for each $i$ the
directed edges $a\to a+e_i$ form cycles $a\to a+e_i\to a+2e_i\to a$ of length $3$, and
such a cycle leaves $A$ as often as it enters $A$. His Theorem~1.3 computes this count for the sets $I_n(m)$.
His equation~(1.1) says that these sets minimize the count. For
$\Z_3^n$, this is Lindsey's theorem.
After doubling, the two statements give the sampling equation for $q=3$.

In every base, the sampling equation follows from \eqref{xi} and from a
formula of Kr\"uppel~\cite[Theorem~5.2]{Kruppel2010b}. He proved that
$G(m/q^n)=(S(q)mn-qS(m))/q^n$ for every Knopp function
$G(x)=\sum_{j=0}^{\infty}q^{-j}g(q^jx)$ whose generator $g$ is continuous,
has period $1$, is zero at the integers, and is linear on each interval
$[k/q,(k+1)/q]$. Here $S(m)=\sum_{k=0}^{m-1}s(k)$, and $s(k)$ is the
weighted digit sum of $k$ in base $q$ in which a digit $c$ has the weight
$g(1/q)+g(c/q)-g((c+1)/q)$~\cite[equations~(5.4), (5.6),
and~(5.10)]{Kruppel2010b}. The generator $\tau_q$ of $T_q$ satisfies
these conditions, and for it a digit $c$ has the weight $2c/q$, because
$\tau_q(k/q)=k(q-k)/q$. Hence $s(m)=\frac2qs_q(m)$,
$S(m)=\frac2qS_q(m)$, and $S(q)=q-1$, and his formula reads as follows in
our notation.

\begin{theorem}[Kr\"uppel {\cite[Theorem~5.2]{Kruppel2010b}}]
\label{sampthm}
For $n\ge1$ and $0\le m\le q^n$,
\[
 T_q\!\left(\frac m{q^n}\right)=\frac{(q-1)nm-2S_q(m)}{q^n}.
\]
\end{theorem}

By \eqref{xi}, Theorem~\ref{sampthm} is the sampling equation
\begin{equation}
 \xi_n(m)=q^nT_q\!\left(\frac m{q^n}\right),~n\ge1,~0\le m\le q^n.
 \label{samp}
\end{equation}

The partial sums $F_n$ are the linear interpolants of the rescaled
boundary counts (see Figures~\ref{fig2} and~\ref{partialpic}).

\begin{theorem}
\label{Fn}
Let $n\ge1$. For $0\le m\le q^n$,
\begin{equation}
 F_n\!\left(\frac m{q^n}\right)=\frac{\xi_n(m)}{q^n},
 \label{interp}
\end{equation}
and on each interval $[m/q^n,(m+1)/q^n]$ of level $n$ the
function $F_n$ is linear with slope
\begin{equation}
 (q-1)n-2s_q(m).
 \label{slope}
\end{equation}
\end{theorem}

\begin{proof}
Let $0\le m<q^n$, and let $m=\sum_{i=0}^{n-1}x_iq^i$ with digits
$x_i\in\{0,1,\ldots,q-1\}$. Fix $j$ with $0\le j\le n-1$ and put
$i=n-1-j$. Then $r_i(m)=\sum_{l=0}^{i}x_lq^l$, so $m=Mq^{i+1}+r_i(m)$
with an integer $M\ge0$, and
\[
 x_iq^i\le r_i(m)\le r_i(m)+1\le(x_i+1)q^i .
\]
Let $m/q^n\le x\le(m+1)/q^n$. Since $q^jx=q^nx/q^{i+1}$ and
$m\le q^nx\le m+1$, the number $q^jx-M$ lies between $r_i(m)/q^{i+1}$ and
$(r_i(m)+1)/q^{i+1}$, and hence in $[x_i/q,(x_i+1)/q]$. By \eqref{tau}
and the periodicity of $\tau_q$,
\[
 q^{-j}\tau_q(q^jx)
 =q^{-j}\Bigl((q-1-2x_i)(q^jx-M)+\frac{x_i(x_i+1)}q\Bigr),
\]
which is a linear function of $x$ with slope $q-1-2x_i$. Summing over
$j$, we see that $F_n$ is linear on $[m/q^n,(m+1)/q^n]$ with slope
$\sum_{i=0}^{n-1}(q-1-2x_i)=(q-1)n-2s_q(m)$, which is \eqref{slope}. By
\eqref{diff}, this slope equals $\xi_n(m+1)-\xi_n(m)$, so
$F_n((m+1)/q^n)-F_n(m/q^n)=q^{-n}(\xi_n(m+1)-\xi_n(m))$ for $0\le m<q^n$.
Since $F_n(0)=0$ and $\xi_n(0)=0$ by \eqref{xi}, adding these equations
gives \eqref{interp} for $0\le m\le q^n$.
\end{proof}

\begin{example}
\label{ex:F2}
Let $q=3$ and $n=2$. The integers $m=0,1,\ldots,8$ have the ternary digit
sums $0$, $1$, $2$, $1$, $2$, $3$, $2$, $3$, and $4$, so by \eqref{slope}
the slopes of $F_2$ on the nine ternary intervals of level $2$ are $4$,
$2$, $0$, $2$, $0$, $-2$, $0$, $-2$, and $-4$. Since $F_2(0)=0$,
\eqref{interp} shows that $\xi_2(0),\xi_2(1),\ldots,\xi_2(9)$ are $0$,
$4$, $6$, $6$, $8$, $8$, $6$, $6$, $4$, and $0$. For instance,
$\xi_2(4)=8$: every set of four vertices of $Q_2^3$ has at least $8$
edges leaving it, and some set of four vertices has exactly $8$.
\end{example}

Written term by term, \eqref{interp} reads
\begin{equation}
 \xi_n(m)=\sum_{i=0}^{n-1}
 q^{i+1}\tau_q\!\left(\frac{r_i(m)}{q^{i+1}}\right),
 \label{coord-sum}
\end{equation}
because for $i=n-1-j$ the term with index $j$ of $q^nF_n(m/q^n)$ is
$q^{i+1}\tau_q(m/q^{i+1})=q^{i+1}\tau_q(r_i(m)/q^{i+1})$, by the periodicity
of $\tau_q$. For $q=3$ the term with index $i$ in \eqref{coord-sum} is
$2\min\{r_i(m),3^i,3^{i+1}-r_i(m)\}$, since
$\dist(m/3^{i+1},\Z)=\min\{r_i(m),3^{i+1}-r_i(m)\}/3^{i+1}$.

\begin{definition}
\label{intslope}
The slope of an interval $[m/q^n,(m+1)/q^n]$ of level $n$ is the
slope of $F_n$ on this interval. By \eqref{slope}, and since $F_0=0$, it
is the integer $(q-1)n-2s_q(m)$, which is even when $q$ is odd. For
$q=3$ it is $2(n-s_3(m))$.
\end{definition}

By \eqref{samp}, the slope of an interval of level $n$ is also the slope
of the line through the two points of the graph of $T_q$ above its
endpoints.

To illustrate \eqref{coord-sum}, let $n=2$ and $q\ge3$. Figure~\ref{cubepic}
shows $Q_2^q$ for $q=3$, $4$, and~$5$, with the vertex $m=x_0+qx_1$ in
column~$x_0$ and row~$x_1$. Two vertices are adjacent exactly when they lie
in the same row or in the same column, so each row and each column is a
complete graph on $q$ vertices, which we draw as one straight line through
its vertices. The edges in a row change coordinate~$0$, and those in a
column change coordinate~$1$. If a row or a column contains $k$ vertices of
a set $A$, then exactly $k(q-k)$ of its edges lie in $\partial A$. In each
panel the filled vertices form $I_2(q+2)$, which consists of the top row
and the vertices $q$ and $q+1$. By Lemma~\ref{lindsey} and \eqref{edges},
this set has $\xi_2(q+2)$ boundary edges. The top row is black, and the
rows that contain no vertex of the set are gray. A row that contains $k$
vertices of the set with $0<k<q$ is red and dashed, a column of this kind
is green and dash dotted, and the number $k(q-k)$ is written next to it.
Below each panel we write $\xi_2(q+2)$ as the sum of the red numbers plus
the sum of the green numbers. The second row contains two vertices of the
set, which gives $2(q-2)$ boundary edges in coordinate~$0$. The first two
columns contain two vertices of the set and the other $q-2$ columns
contain one, which gives $4(q-2)+(q-2)(q-1)=(q-2)(q+3)$ boundary edges in
coordinate~$1$. These are the two terms of \eqref{coord-sum}, since
$r_0(q+2)=2$, $r_1(q+2)=q+2$, and \eqref{tau} with $k=1$ gives
\[
 q\tau_q\!\left(\frac2q\right)=2(q-2)~\text{and}~
 q^2\tau_q\!\left(\frac{q+2}{q^2}\right)=(q-2)(q+3).
\]
Hence $\xi_2(q+2)=(q-2)(q+5)$, which is $8$, $18$, and $30$ for $q=3$,
$4$, and $5$.

\begin{figure}[H]
\centering
\includegraphics{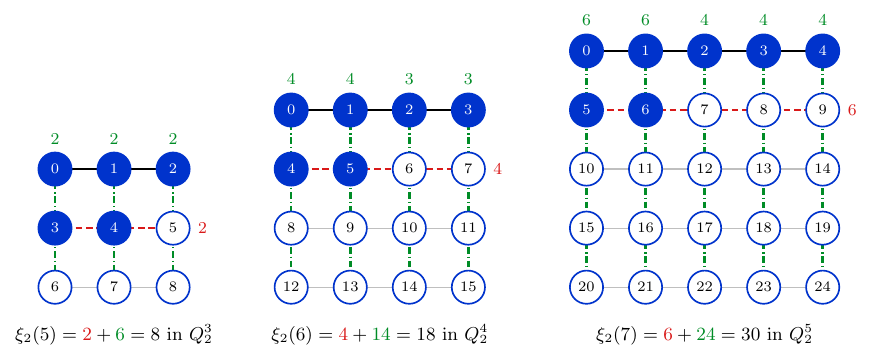}
\caption{The initial segments $I_2(5)$, $I_2(6)$, and $I_2(7)$ in the
Hamming graphs $Q_2^3$, $Q_2^4$, and $Q_2^5$. Below each panel the boundary
count is split into the two terms of \eqref{coord-sum}.}
\label{cubepic}
\end{figure}

\begin{figure}[H]
\centering
\includegraphics{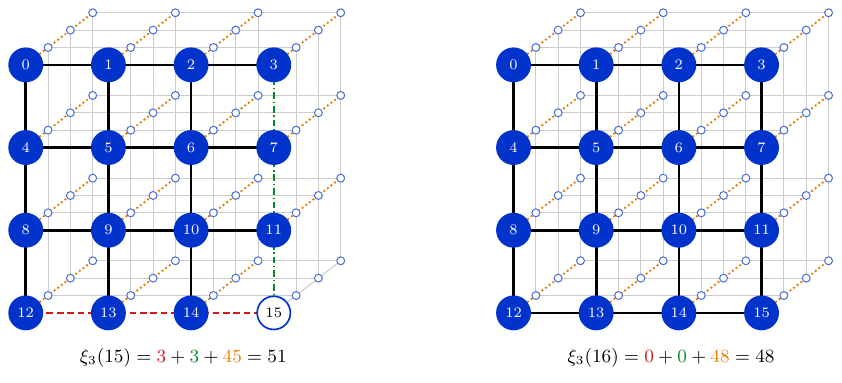}
\caption{The initial segments $I_3(15)$ (left) and $I_3(16)$ (right) in the
Hamming graph $Q_3^4$. By \eqref{diff}, the larger set has three fewer
boundary edges.}
\label{Q3pic}
\end{figure}

Since the change in \eqref{diff} can be negative, a larger set may have a
smaller boundary. Figure~\ref{Q3pic} shows an example in $Q_3^4$, drawn
in the same way, with the vertex $m=x_0+4x_1+16x_2$ in column~$x_0$,
row~$x_1$, and layer~$x_2$. The layer $x_2=0$ is in front,
and the vertices of the three layers behind it are small circles without
labels. The four vertices with given $x_0$ and $x_1$ are pairwise adjacent,
and the edges between them change coordinate~$2$. We draw them as one
straight line, which is orange and dotted when it contains one, two, or
three vertices of the set. The set $I_3(16)$ is the front layer, a copy of
$Q_2^4$, and $I_3(15)$ is this layer without the vertex $15=(033)_4$. In
both panels every orange line contains one vertex of the set and three
boundary edges. The vertex $15$ has the six neighbors $3$, $7$, $11$, $12$,
$13$, and $14$ in $I_3(15)$ and the three neighbors $31$, $47$, and $63$
behind it, so $\xi_3(16)=\xi_3(15)-6+3=48$, in agreement with
$3\cdot3-2s_4(15)=-3$.
Below each panel the boundary count is split by coordinate into the three
terms of \eqref{coord-sum}.

Equation \eqref{samp} appears in the literature in other forms as well.
In every base $q$, Allaart~\cite[Section~3]{Allaart2014} wrote $h_q$ for
$T_q/2$, noted that the partial sum of order $n$ of his series has slope
$\frac{q-1}2n-s_q(m)$ on $(m/q^n,(m+1)/q^n)$, which is half of
\eqref{slope}, and derived
$h_q(k/q^n)=q^{-n}\bigl(\frac{q-1}2kn-S_q(k)\bigr)$ for
$0\le k\le q^n$~\cite[equation~(27)]{Allaart2014}. In a question from
evolutionary biology, Mohanty et~al.~\cite[equation~(3.3)]{Mohanty2023}
computed the largest number of edges among $m$ vertices of a Hamming
graph. The Delange functions enter their computation through Delange's
formula for $S_q$, which writes $S_q(m)$ as $\frac{q-1}2m\log_qm$ plus $m$
times a continuous function of $\log_qm$ with period $1$ that is
expressed through $T_q$ (see Definition~\ref{Psidef}).

Summing \eqref{samp} over $m$ gives the total of the boundary counts.
Since $F_n$ is linear on each interval of level $n$ and
$F_n(0)=F_n(1)=0$, equation \eqref{interp} gives
$\int_0^1F_n=q^{-n}\sum_{m=0}^{q^n}F_n(m/q^n)=q^{-2n}\sum_{m=0}^{q^n}\xi_n(m)$.
In the same way, since $\tau_q$ is linear on each interval $[k/q,(k+1)/q]$
and $\tau_q(k/q)=k(q-k)/q$,
\[
 \int_0^1\tau_q(x)\,dx=\frac1q\sum_{k=1}^{q-1}\frac{k(q-k)}q=\frac{q^2-1}{6q}.
\]
Each function $\tau_q(q^jx)$ has period $q^{-j}$, so its integral over
$[0,1]$ is the same number. Therefore
\begin{equation}
 \sum_{m=0}^{q^n}\xi_n(m)=q^{2n}\int_0^1F_n(x)\,dx
 =q^{2n}\cdot\frac{q^2-1}{6q}\sum_{j=0}^{n-1}q^{-j}
 =\frac{q+1}6\,(q^{2n}-q^n).
 \label{edgesum}
\end{equation}
The left side is also the smallest value of the sum $\sum_{uv}|f(u)-f(v)|$
over the edges $uv$ of $Q_n^q$, where $f$ runs through the bijections from
the vertex set onto $\{0,1,\ldots,q^n-1\}$. Indeed, an edge $uv$ with
$f(u)<f(v)$ contributes $f(v)-f(u)$, the number of integers $k$ with
$f(u)<k\le f(v)$, and these are exactly the integers $k$ for which the
edge leaves the set $\{w:f(w)<k\}$ of $k$ vertices. Hence the sum is
$\sum_{k=1}^{q^n-1}|\partial\{w:f(w)<k\}|$. Since
at least $\xi_n(k)$ edges leave a set of $k$ vertices, the sum is at
least $\sum_{k=1}^{q^n-1}\xi_n(k)$, with equality when $f$ is the
labeling of Section~\ref{sec:graph}, since $|\partial I_n(k)|=\xi_n(k)$ for every $k$.
Lindsey~\cite[page~515]{Lindsey1964} found the value
$\frac16(q^n-1)q^n(q+1)$ of this minimum, which is the right side of
\eqref{edgesum}. Since $F_n\to T_q$ uniformly on $[0,1]$, dividing
\eqref{edgesum} by $q^{2n}$ and letting $n\to\infty$ gives
$\int_0^1T_q=(q+1)/6$, the constant term of the Fourier series in
Theorem~\ref{fourier}.

For odd $q$, equation \eqref{coord-sum} gives the largest of the values
$\xi_n(m)$. Let $q$ be odd. The function $\tau_q$ attains its maximum
$(q^2-1)/(4q)$ exactly on the interval $[(q-1)/(2q),(q+1)/(2q)]$, on
which its slope is zero. Hence the term with index $i$ in
\eqref{coord-sum} is at most $q^i(q^2-1)/4$, with equality if and only if
$(q-1)q^i/2\le r_i(m)\le(q+1)q^i/2$. Since
$\sum_{i=0}^{n-1}q^i(q^2-1)/4=(q+1)(q^n-1)/4$, we obtain
\begin{equation}
 \max_{0\le m\le q^n}\xi_n(m)=\frac{(q+1)(q^n-1)}4
 ~\text{for odd}~q,
 \label{peak}
\end{equation}
attained at $m=(q^n\pm1)/2$, and only there. Indeed, if $\xi_n(m)$ is this
maximum, then equality holds in every term, so
$(q-1)q^i/2\le r_i(m)\le(q+1)q^i/2$ for every $i<n$. Hence
$r_{n-1}(m)>0$ and $m<q^n$. Let $x_0,\ldots,x_{n-1}$ be the base $q$
digits of $m$. For $i=0$ the condition says $x_0\in\{(q-1)/2,(q+1)/2\}$.
For $1\le i<n$ we have $r_i(m)=x_iq^i+r_{i-1}(m)$ with
$0<r_{i-1}(m)<q^i$, so $x_i\le(q-3)/2$ would give $r_i(m)<(q-1)q^i/2$,
and $x_i\ge(q+1)/2$ would give $r_i(m)>(q+1)q^i/2$. Therefore
$x_i=(q-1)/2$ for $1\le i<n$, and $m=(q^n\pm1)/2$. Conversely, for these
two integers
$r_i(m)=(q^{i+1}\pm1)/2$ for every $i<n$, and this number lies between
$(q-1)q^i/2$ and $(q+1)q^i/2$. For $q=3$ the maximum \eqref{peak} is
$3^n-1$. For even $q$ the largest boundary count is given in
Corollary~\ref{peakeven}.

Figures~\ref{fig1} and~\ref{fig2} show the boundary counts $\xi_n(m)$ for
$0\le m\le q^n/2$ and the function $F_n$ on $[0,1]$ for $q=3$, $4$, and
$5$ with $n=15$, $12$, and $10$. By \eqref{interp}, the graphs in
Figure~\ref{fig1} are the left halves of those in Figure~\ref{fig2} with
both axes multiplied by $q^n$. The dashed lines in
Figure~\ref{fig1} mark the largest boundary count, which is $3^{15}-1$
for $q=3$ and $3(5^{10}-1)/2$ for $q=5$ by \eqref{peak}, and
$6(4^{12}-1)/5$ for $q=4$ by Corollary~\ref{peakeven}. The corners of
all these graphs come from carries. By \eqref{diff}, the piecewise linear
function with the values $\xi_n(m)$ at the integers $m$ changes its slope
at $m$ by $2(s_q(m-1)-s_q(m))=2((q-1)v-1)$, where
$v$ is the exponent of $q$ in $m$, that is, the number of carries when
$1$ is added to $m-1$. Thus the slope decreases by $2$ at each $m$ with
$q\nmid m$ and increases by $2((q-1)v-1)$ at each multiple of $q$, and
the large increases at the multiples of high powers of $q$ produce the
visible dips.

Each panel of Figure~\ref{fig2} enlarges a level $2$
interval $[c/q^2,(c+1)/q^2]$ with $s_q(c)=q-1$, namely $[4/9,5/9]$,
$[3/8,7/16]$, and $[12/25,13/25]$. By \eqref{slope}, the slope of $F_2$
on this interval is $2(q-1)-2s_q(c)=0$, so by \eqref{interp}, $F_2$ is
equal to $\xi_2(c)/q^2$ there, which is $8/9$, $9/8$, and $36/25$,
respectively. Since $F_n(x)=F_2(x)+q^{-2}F_{n-2}(q^2x)$ by the definition
of $F_n$, and $F_{n-2}$ has period $1$, the graph of $F_n$ on this
interval is a copy of the graph of $F_{n-2}$ scaled by $q^{-2}$ in both
directions and raised by $\xi_2(c)/q^2$. Since
$\tau_q\ge0$, we have $0\le F_{n-2}\le T_q$ by \eqref{eq:tail}, so this copy
lies in the box
drawn in the main panel, between the heights $\xi_2(c)/q^2$ and
$\xi_2(c)/q^2+q^{-2}\max T_q$. The upper height is $1$, $6/5$, and $3/2$,
respectively, which is the maximum of $T_q$ (see Theorem~\ref{gmax}) and
the highest mark on the vertical axis.

\begin{figure}[H]
  \centering
  \includegraphics{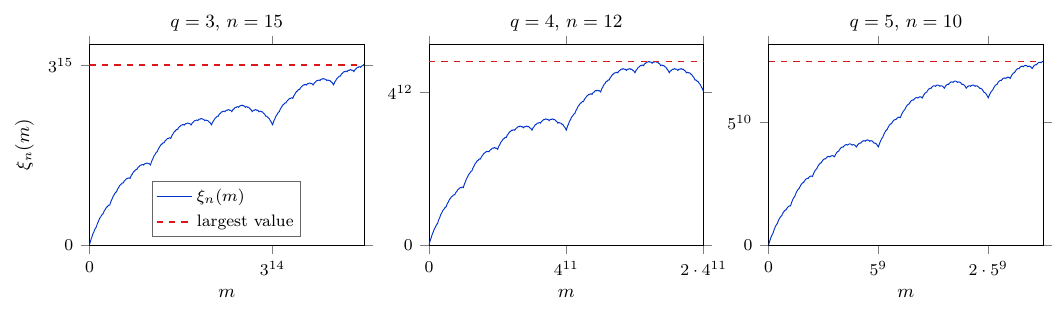}
  \caption{The minimum number $\xi_n(m)$ of edges leaving a set of
  $m$ vertices in $Q_n^q$ for $0\le m\le q^n/2$, for $q=3$, $4$, and $5$
  with $n=15$, $12$, and $10$, plotted at the multiples of $3^7$, $4^6$,
  and $5^5$ and at the integers $m$ at which $\xi_n(m)$ is largest. The
  dashed lines mark this largest value (see \eqref{peak} and
  Corollary~\ref{peakeven}).}
  \label{fig1}
\end{figure}

\begin{figure}[H]
  \centering
  \includegraphics{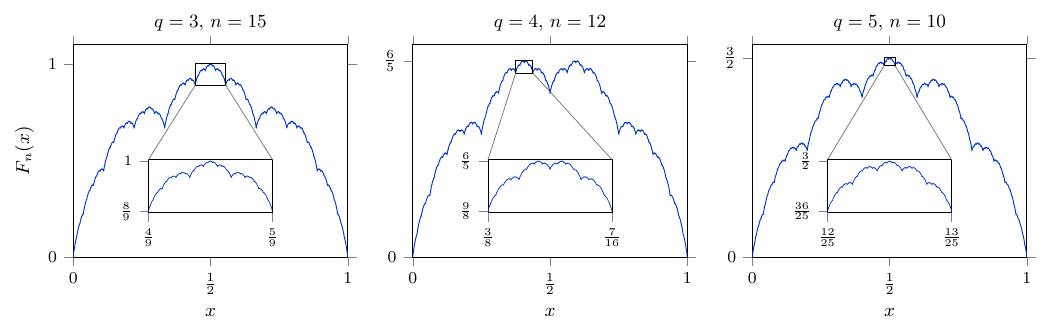}
\caption{The functions $F_{15}$ for $q=3$, $F_{12}$ for $q=4$, and
$F_{10}$ for $q=5$ (see Theorem~\ref{Fn}), plotted at the multiples of
$3^{-8}$, $4^{-6}$, and $5^{-5}$, with enlarged views of the level $2$
intervals $[4/9,5/9]$, $[3/8,7/16]$, and $[12/25,13/25]$, plotted at the
multiples of $3^{-10}$, $4^{-8}$, and $5^{-7}$.}
  \label{fig2}
\end{figure}

The functional equation \eqref{fe} describes the graph of $T_q$ above each
of the $q$ intervals of level $1$ in terms of the whole graph. If we
write $x=(a+y)/q$ with $a\in\{0,1,\ldots,q-1\}$ and $0\le y\le1$, then
$T_q(qx)=T_q(y)$, and \eqref{tau} gives
\begin{equation}
 T_q\!\left(\frac{a+y}q\right)
 =\frac{a(q-a)+(q-1-2a)\,y+T_q(y)}q.
 \label{branches}
\end{equation}
For $q=3$ the right side is $(T_q(y)+2y)/3$ for $a=0$, $(T_q(y)+2)/3$ for
$a=1$, and $(T_q(y)+2-2y)/3$ for $a=2$.
On every interval of level $n$ we find the same structure, with the slope
of the interval in place of the slopes $q-1-2a$. We state the resulting
equation as Proposition~\ref{prop-cell} below. For $q=2$, its two
equations are Kr\"uppel's equations for Takagi's
function~\cite[Proposition~2.1]{Kruppel2007}. For $n\ge1$ and every $q$,
its first part is the functional equation~(5.12) of
Kr\"uppel~\cite[Theorem~5.2]{Kruppel2010b}, which he proved for the Knopp
functions described before Theorem~\ref{sampthm}, and its second part
follows from the first and the symmetry $T_q(1-t)=T_q(t)$. We state his
equation for $T_q$. In his notation the coefficient of $t$ is
$(S(q)n-qs(m))/q^n$, which is $((q-1)n-2s_q(m))q^{-n}$ because
$S(q)=q-1$ and $s(m)=\frac2qs_q(m)$.

\begin{figure}[H]
\centering
\includegraphics{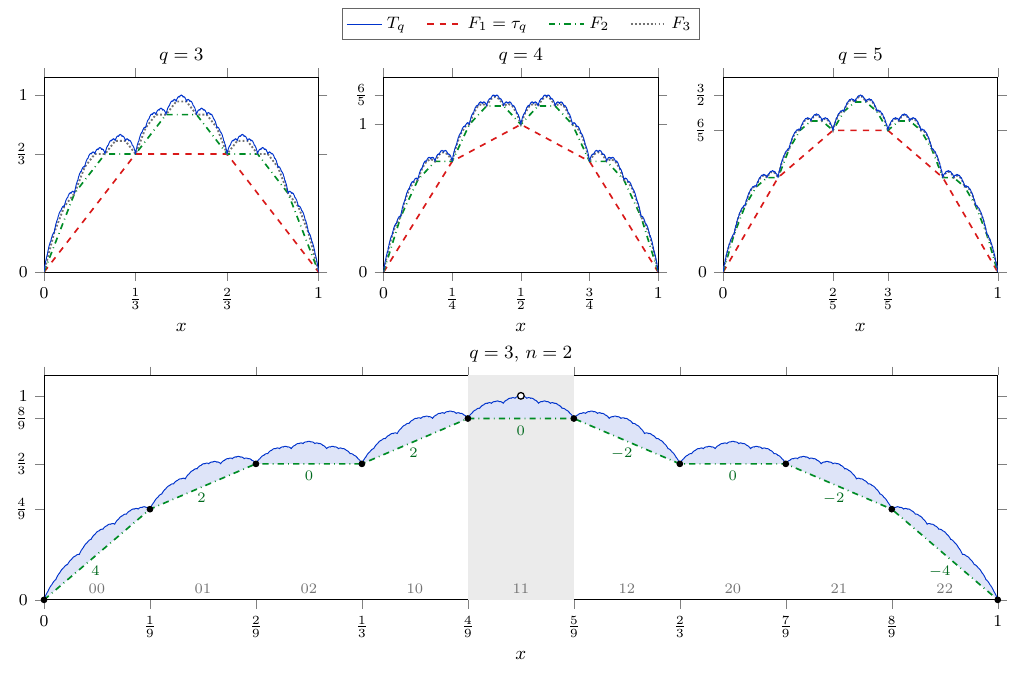}
\caption{Upper panels: $T_q$ and the partial sums $F_1=\tau_q$, $F_2$,
and $F_3$ for $q=3$, $4$, and $5$ (see Theorem~\ref{Fn}). Lower panel:
$T_q$ and $F_2$ for $q=3$, with the slope of $F_2$ on each ternary
interval of level $2$ (see Example~\ref{ex:F2}).}
\label{partialpic}
\end{figure}

\begin{theorem}[Kr\"uppel {\cite[Theorem~5.2]{Kruppel2010b}}]
\label{kruppelcell}
Let $n\ge1$, $0\le m<q^n$, and $0\le t\le1$. Then
\[
 T_q\!\left(\frac{m+t}{q^n}\right)
 =T_q\!\left(\frac m{q^n}\right)
  +q^{-n}\bigl(((q-1)n-2s_q(m))\,t+T_q(t)\bigr).
\]
\end{theorem}

\begin{proposition}
\label{prop-cell}
Let $n\ge0$ and $0\le t\le1$. For $0\le m<q^n$,
\begin{equation}
 T_q\!\left(\frac{m+t}{q^n}\right)
 =T_q\!\left(\frac m{q^n}\right)
  +q^{-n}\bigl(((q-1)n-2s_q(m))\,t+T_q(t)\bigr),
 \label{cell}
\end{equation}
and for $1\le m\le q^n$,
\begin{equation}
 T_q\!\left(\frac{m-t}{q^n}\right)
 =T_q\!\left(\frac m{q^n}\right)
  +q^{-n}\bigl((2s_q(m-1)-(q-1)n)\,t+T_q(t)\bigr).
 \label{cell2}
\end{equation}
\end{proposition}

\begin{proof}
For $n\ge1$, equation \eqref{cell} is Theorem~\ref{kruppelcell}. For
$n=0$ equation \eqref{cell}
reads $T_q(t)=T_q(0)+T_q(t)$. To prove \eqref{cell2}, we apply \eqref{cell}
with $q^n-m$ in place of $m$. Since $T_q$ has period $1$ and $T_q(1-x)=T_q(x)$,
this gives
$T_q((m-t)/q^n)=T_q(m/q^n)+q^{-n}\bigl(((q-1)n-2s_q(q^n-m))\,t+T_q(t)\bigr)$.
Moreover, $s_q(q^n-m)=(q-1)n-s_q(m-1)$, because $q^n-m=(q^n-1)-(m-1)$ and
all $n$ base $q$ digits of $q^n-1$ equal $q-1$.
\end{proof}

Equation \eqref{cell} also follows from Theorems~\ref{limitthm}
and~\ref{Fn}. Let $n\ge1$, $0\le m<q^n$, $0\le t\le1$, and
$x=(m+t)/q^n$. Since $T_q$ is zero at the integer $m$, we have
$F_n(m/q^n)=T_q(m/q^n)$ by \eqref{eq:tail}. By Theorem~\ref{Fn}, $F_n$ is
linear on $[m/q^n,(m+1)/q^n]$ with slope $(q-1)n-2s_q(m)$, and hence
$F_n(x)=T_q(m/q^n)+q^{-n}((q-1)n-2s_q(m))\,t$. By \eqref{eq:tail} and the
periodicity of $T_q$, we have $T_q(x)-F_n(x)=q^{-n}T_q(m+t)=q^{-n}T_q(t)$.
Adding the last two equations gives \eqref{cell}. Thus, on every interval
$[m/q^n,(m+1)/q^n]$ of level $n$, the function $T_q$ is a straight line of
slope $(q-1)n-2s_q(m)$ plus a copy of $T_q$ scaled by $q^{-n}$. At the
endpoints the copy contributes nothing, so
$T_q((m+1)/q^n)-T_q(m/q^n)=((q-1)n-2s_q(m))q^{-n}$, and summing these
increments from $T_q(0)=0$ gives $q^nT_q(m/q^n)=(q-1)nm-2S_q(m)$, which
is \eqref{samp} again, by \eqref{xi}.

The upper panels of Figure~\ref{partialpic} show the graphs of $T_q$ and
of $F_1$, $F_2$, and $F_3$ for $q=3$, $4$, and $5$. By
Theorem~\ref{Fn} and \eqref{samp}, the straight lines form the graph of
$F_n$. Let $q=3$. Since $0\le T_q\le1$ and $T_q(1/2)=1$ by
Theorem~\ref{limitthm}, the copy
$T_q-F_n=3^{-n}T_q(3^nx)$ lies between $0$ and $3^{-n}$ by
\eqref{eq:tail}, with the value $3^{-n}$ at $1/2$, since $3^n/2-1/2$ is an
integer. Take $n=2$, as in the lower panel. The strings at the bottom of
the panel are the integers $m=0,1,\ldots,8$ written with two ternary
digits. The slopes of the nine intervals of level $2$ and the values
$\xi_2(m)$ are those of Example~\ref{ex:F2}, and the dots are the points
$(m/9,\xi_2(m)/9)$. Since $T_q-F_2=T_q(9x)/9$ by
\eqref{eq:tail}, the shaded region between the graphs of $F_2$ and $T_q$ has
height $T_q(t)/9$ at $x=(m+t)/9$. On the middle interval $[4/9,5/9]$, shaded
gray, the slope is $0$, and \eqref{cell} with $n=2$ and $m=4$ reads
$T_q((4+t)/9)=8/9+T_q(t)/9$ for $0\le t\le1$. In other words, the graph of
$T_q$ there is
the whole graph scaled by $1/9$ in both directions and raised by $8/9$,
with its top at $(1/2,1)$ (the open circle). On $[0,1/9]$, \eqref{cell}
reads $T_q(t/9)=(4t+T_q(t))/9$, so there the same copy sits on a line of
slope $4$.

The upper middle panel of Figure~\ref{partialpic} shows these functions
for $q=4$. Here $\tau_q$ takes its maximum $1$ only at $1/2$, where it has a
corner. Since $\tau_q$ and $T_q$ are zero at the integers $4^j/2$ with
$j\ge1$, the definition of $F_n$ and \eqref{eq:tail} give
$T_q(1/2)=F_n(1/2)=\tau_q(1/2)=1$ for every $n\ge1$. By \eqref{slope},
the slope $3n-2s_4(m)$ of $F_n$ on an interval of level $n$ is odd when
$n$ is odd. Hence $F_1$ and $F_3$ have no
horizontal pieces, while $F_2$ is constant on the four intervals
$[m/16,(m+1)/16]$ with $s_4(m)=3$, that is, with $m=3$, $6$, $9$, and
$12$. The copy $T_q-F_n=4^{-n}T_q(4^nx)$ lies between $0$ and
$\frac65\,4^{-n}$, because the maximum of $T_4$ is $6/5$ (see
Theorem~\ref{gmax}). The upper right panel shows these functions for
$q=5$. Here $\tau_q$ is constant, equal to its maximum $6/5$, on
$[2/5,3/5]$, and the copy $T_q-F_n=5^{-n}T_q(5^nx)$ lies between $0$ and
$\frac32\,5^{-n}$ by Theorem~\ref{limitthm}.

Let $d$ be an integer, let $\zeta$ be an integer with $0\le\zeta\le q-1$,
and let $0\le y\le1$. Adding $d(\zeta+y)/q$ to both sides of
\eqref{branches} with $a=\zeta$ gives
\begin{equation}
 T_q\Bigl(\frac{\zeta+y}q\Bigr)+d\,\frac{\zeta+y}q
 =\frac{\zeta(q-\zeta+d)+T_q(y)+(d+q-1-2\zeta)\,y}q.
 \label{tiltbranch}
\end{equation}

\begin{definition}
\label{digslope}
A digit sequence of a point $x\in[0,1]$ is a sequence
$\xi_1,\xi_2,\ldots$ of integers with $0\le\xi_j\le q-1$ for every $j$
and $x=\sum_{j=1}^{\infty}\xi_jq^{-j}$. The digit slopes of this sequence
are the integers
\[
 d_n=\sum_{j=1}^{n}(q-1-2\xi_j),~n\ge0,
\]
so that $d_0=0$. Each of the points $0$ and $1$, and each point of
$(0,1)$ that is not a $q$-adic rational, has exactly one digit sequence.
A $q$-adic rational in $(0,1)$ has exactly two digit sequences: the
terminating one, in which all digits from some index on are $0$, and one
in which all digits from some index on are $q-1$. For $x\in[0,1)$, the
digit slopes $d_n(x)$ of $x$ are the digit slopes of the digit sequence
of $x$, where we take the terminating one when $x$ is a $q$-adic rational.
\end{definition}

Let $\xi_1,\xi_2,\ldots$ be a digit sequence of $x$ with digit slopes
$d_n$, and for $n\ge0$ put $a_n=\sum_{j=1}^{n}\xi_jq^{-j}$. Then
$q^na_n$ is an integer with $0\le q^na_n<q^n$, and $\xi_1,\ldots,\xi_n$
are its base $q$ digits, written with $n$ digits. Hence
$s_q(q^na_n)=\xi_1+\cdots+\xi_n$ and
\begin{equation}
 d_n=(q-1)n-2s_q(q^na_n),
 \label{dnsq}
\end{equation}
so that, by Definition~\ref{intslope}, $d_n$ is the slope of the interval
$[a_n,a_n+q^{-n}]$ of level $n$. This interval contains $x$, because
$0\le x-a_n=\sum_{j=n+1}^{\infty}\xi_jq^{-j}\le q^{-n}$. Each difference
$d_n-d_{n-1}=q-1-2\xi_n$ is one of the integers
$q-1,q-3,\ldots,-(q-1)$. Therefore $|d_n|\le(q-1)n$, and $d_n-(q-1)n$ is
even. In particular, the digit slopes are even when $q$ is odd. We use
their halves in the sections in which $q=3$.

\begin{definition}
\label{halfslope}
Let $q$ be odd, let $x\in[0,1)$, and let $\xi_1,\xi_2,\ldots$ be the digit
sequence of $x$, where we take the terminating one when $x$ is a $q$-adic rational. For $n\ge0$ we put
$\kappa_n(x)=d_n(x)/2=\sum_{j=1}^{n}\bigl((q-1)/2-\xi_j\bigr)$. For $q=3$
we have $\kappa_n(x)=\sum_{j=1}^{n}(1-\xi_j)$, which is the number of
zeros minus the number of twos among $\xi_1,\ldots,\xi_n$.
\end{definition}

Since $\xi_1,\ldots,\xi_n$ are the digits of $k=\lfloor q^nx\rfloor$,
equation \eqref{dnsq} gives $d_n(x)=(q-1)n-2s_q(k)$, which is the slope of
the interval of level $n$ with $k/q^n\le x<(k+1)/q^n$. For $q=3$ this is $2\kappa_n(x)=2(n-s_3(k))$, and $d_j(x)\ge2$, that is,
$\kappa_j(x)\ge1$, means that there are more zeros than twos among
$\xi_1,\ldots,\xi_j$.

\begin{proposition}
\label{digform}
Let $\xi_1,\xi_2,\ldots$ be a digit sequence of a point $x\in[0,1]$ with
digit slopes $d_n$, and let $a_n=\sum_{j=1}^{n}\xi_jq^{-j}$. Then, for
every $N\ge1$,
\begin{equation}
 T_q(a_N)=\sum_{n=1}^{N}\frac{(q-\xi_n+d_{n-1})\,\xi_n}{q^n}
 =\frac{q+1}4-\frac{q-1}{4q}\sum_{n=1}^{N-1}\frac{d_n^2}{q^n}
  -\frac{(d_N+1)^2+q}{4q^N},
 \label{digsum}
\end{equation}
and
\begin{equation}
 T_q(x)=\sum_{n=1}^{\infty}\frac{(q-\xi_n+d_{n-1})\,\xi_n}{q^n}
 =\frac{q+1}4-\frac{q-1}{4q}\sum_{n=1}^{\infty}\frac{d_n^2}{q^n}.
 \label{dsq}
\end{equation}
\end{proposition}

\begin{proof}
Let $n\ge1$ and $k=q^{n-1}a_{n-1}$, so that $0\le k<q^{n-1}$ and
$a_n=(k+\xi_n/q)/q^{n-1}$. For an integer $j$ with $0\le j\le q-1$ we
have $T_q(j/q)=\tau_q(j/q)=j(q-j)/q$, by \eqref{fe} and $T_q(j)=0$. Hence
\eqref{cell} with $n-1$ in place of $n$, $m=k$, and $t=\xi_n/q$,
together with \eqref{dnsq}, gives
\[
 T_q(a_n)-T_q(a_{n-1})
 =q^{-(n-1)}\Bigl(d_{n-1}\,\frac{\xi_n}q+\frac{(q-\xi_n)\,\xi_n}q\Bigr)
 =\frac{(q-\xi_n+d_{n-1})\,\xi_n}{q^n}.
\]
Summing over $n=1,\ldots,N$ from $T_q(a_0)=T_q(0)=0$ gives the first equation
in \eqref{digsum}. Since $d_n-d_{n-1}=q-1-2\xi_n$, we have
$\xi_n=\frac12(q+d_{n-1})-\frac12(1+d_n)$ and
$q-\xi_n+d_{n-1}=\frac12(q+d_{n-1})+\frac12(1+d_n)$, so
\[
 4T_q(a_N)=\sum_{n=1}^{N}\frac{(q+d_{n-1})^2}{q^n}
 -\sum_{n=1}^{N}\frac{(1+d_n)^2}{q^n}
 =q+\sum_{n=1}^{N-1}\frac{(q+d_n)^2-q(1+d_n)^2}{q^{n+1}}
 -\frac{(1+d_N)^2}{q^N},
\]
where we used $d_0=0$. Now $(q+d)^2-q(1+d)^2=(q-1)(q-d^2)$ for every
real $d$, and $\sum_{n=1}^{N-1}(q-1)q^{-n}=1-q^{1-N}$. This gives the
second equation in \eqref{digsum}. Since $|d_n|\le(q-1)n$, both series
in \eqref{dsq} converge absolutely, and the last term of \eqref{digsum}
tends to $0$ as $N\to\infty$. Moreover, $T_q(a_N)\to T_q(x)$, because
$0\le x-a_N\le q^{-N}$ and $T_q$ is continuous. Letting $N\to\infty$ in
\eqref{digsum} gives \eqref{dsq}.
\end{proof}

We state the theorem of Galkin and Galkina on the maximum of $T_q$ in our
notation.

\begin{theorem}[Galkin and Galkina {\cite[Theorem~2 and Corollary~1]{GalkinGalkina2020}}]
\label{gmax}
If $q$ is odd, then the maximum of $T_q$ is $(q+1)/4$, and the only point
of $[0,1]$ at which it is attained is $1/2$. If $q$ is even, then the
maximum of $T_q$ is $q(q+2)/(4(q+1))$, and it is attained exactly at the
points $x\in[0,1]$ that have a digit sequence in which, for every
$i\ge0$, the pair $(\xi_{2i+1},\xi_{2i+2})$ is $(q/2-1,q/2)$ or
$(q/2,q/2-1)$.
\end{theorem}

Galkin and Galkina stated the theorem for the maximum over $[0,1]$, which
is the maximum over $\R$ because $T_q$ has period $1$. For even $q$ their
Theorem~2 says that the maximum is attained only at the points whose base
$q$ expansion has the form above. Their Corollary~1 describes the set of
these points as the set obtained from $[0,1]$ by dividing every remaining
interval into $q^2$ equal parts and keeping the parts with the numbers
$q(q-1)/2+1$ and $q(q+1)/2$, so the maximum is attained at each of these
points. They bounded each term of the series \eqref{T} when $q$ is odd, and each
sum of two consecutive terms when $q$ is even. We give another proof,
based on \eqref{dsq}.

\begin{proof}
Let $\xi_1,\xi_2,\ldots$ be a digit sequence of $x\in[0,1]$ with digit
slopes $d_n$. By \eqref{dsq}, $T_q(x)\le(q+1)/4$, with equality if and
only if $d_n=0$ for every $n\ge1$, that is, $q-1-2\xi_n=0$ for every
$n\ge1$. For odd $q$ this means that every digit is $(q-1)/2$, and then
$x=\sum_{j=1}^{\infty}\frac{q-1}2q^{-j}=\frac12$. Let $q$ be even. Then
$d_n-n$ is even, so $d_n^2\ge1$ for every odd $n$, and \eqref{dsq} gives
\[
 T_q(x)\le\frac{q+1}4-\frac{q-1}{4q}\sum_{i=0}^{\infty}q^{-2i-1}
 =\frac{q+1}4-\frac1{4(q+1)}=\frac{q(q+2)}{4(q+1)},
\]
with equality if and only if $d_n=0$ for every even $n$ and $|d_n|=1$
for every odd $n$. Since $d_{2i+1}-d_{2i}=q-1-2\xi_{2i+1}$ and
$d_{2i+2}-d_{2i+1}=q-1-2\xi_{2i+2}$, this holds if and only if, for
every $i\ge0$, either $\xi_{2i+1}=q/2-1$ and $\xi_{2i+2}=q/2$, or
$\xi_{2i+1}=q/2$ and $\xi_{2i+2}=q/2-1$. The digits of such a sequence are neither eventually all $0$ nor
eventually all $q-1$.
Thus the points of maximum are not $q$-adic rationals, and each of them
has only this digit sequence.
\end{proof}

For even $q$, equation \eqref{digsum} gives the largest boundary count.

\begin{corollary}
\label{peakeven}
Let $q$ be even and $n\ge1$. Then
\[
 \max_{0\le m\le q^n}\xi_n(m)=
 \begin{cases}
 \dfrac{q(q+2)(q^n-1)}{4(q+1)}&\text{if $n$ is even},\\[3mm]
 \dfrac{q\bigl((q+2)q^n-q\bigr)}{4(q+1)}&\text{if $n$ is odd}.
 \end{cases}
\]
Let $0\le m<q^n$, and let $d_j=d_j(m/q^n)$, $0\le j\le n$, be the digit
slopes of $m/q^n$. Then $\xi_n(m)$ equals this maximum if and only if
$d_j=0$ for every even $j<n$, $|d_j|=1$ for every odd $j<n$, and
$d_n\in\{0,-2\}$ when $n$ is even, $d_n=-1$ when $n$ is odd.
\end{corollary}

\begin{proof}
Since $\xi_n(q^n)=0$, it suffices to consider $0\le m<q^n$. The
terminating digit sequence of $m/q^n$ consists of the base $q$ digits of
$m$, written with $n$ digits and starting with the digit of the highest
power, followed by zeros, so $m/q^n=a_n$ for this digit sequence, and
\eqref{samp} and \eqref{digsum} give
\[
 \xi_n(m)=q^nT_q(a_n)=\frac{(q+1)q^n}4
 -\frac{q-1}{4q}\sum_{j=1}^{n-1}q^{n-j}d_j^2-\frac{(d_n+1)^2+q}4.
\]
Since $q$ is even, $d_j-j$ is even. Hence $d_j^2\ge1$ for every odd $j$,
and $(d_n+1)^2\ge1$ when $n$ is even. Consequently $\xi_n(m)$ is at most
the value of the right side in which $d_j^2$ is replaced by $0$ for even
$j$ and by $1$ for odd $j$, and $(d_n+1)^2$ by $1$ for even $n$ and by
$0$ for odd $n$. Equality holds exactly under the conditions of the
corollary. These conditions hold when the digits of $m$, in the same
order, are alternately $q/2-1$ and $q/2$, starting with $q/2-1$, except
that the last digit is $q/2$ when $n$ is odd, because then $d_j=1$ for
every odd $j<n$, $d_j=0$ for every even $j\le n$, and $d_n=-1$ when $n$
is odd. It remains to compute the bound. The sum of
$q^{n-j}$ over the odd $j<n$ is $q(q^n-1)/(q^2-1)$ when $n$ is even and
$q^2(q^{n-1}-1)/(q^2-1)$ when $n$ is odd. For even $n$ the bound is
therefore
$\frac{(q+1)(q^n-1)}4-\frac{q^n-1}{4(q+1)}=\frac{q(q+2)(q^n-1)}{4(q+1)}$,
and for odd $n$ it is
$\frac{(q+1)q^n-q}4-\frac{q^n-q}{4(q+1)}
=\frac{q((q+2)q^n-q)}{4(q+1)}$.
\end{proof}

Equation \eqref{eq:tail} also gives the values of $T_q$ at rational points
in closed form. Lagarias~\cite[Theorem~3.3]{Lagarias2012} noted that
Takagi's function maps rational numbers to rational numbers; in every
base the same argument gives a formula.

\begin{proposition}
\label{ratval}
Let $0\le x<1$ be rational, let $\xi_1,\xi_2,\ldots$ be the digit
sequence of $x$, where we take the terminating one when $x$ is a $q$-adic rational, and let $N\ge0$ and $p\ge1$ be integers such
that $\xi_{j+p}=\xi_j$ for all $j>N$.
Let $y$ be the point with the digits $\xi_{N+1},\xi_{N+2},\ldots$. Then
\begin{equation}
 T_q(x)=F_N(x)+q^{-N}T_q(y),~
 T_q(y)=\frac{q^p}{q^p-1}F_p(y),
 \label{ratval1}
\end{equation}
and in particular $T_q(x)$ is rational. If $y=m/(q^p-1)$ with
$0\le m<q^p-1$, then
\begin{equation}
 T_q\Bigl(\frac m{q^p-1}\Bigr)
 =\frac{\bigl((q-1)pm-2S_q(m)\bigr)(q^p-1)+\bigl((q-1)p-2s_q(m)\bigr)m}{(q^p-1)^2}.
 \label{ratval2}
\end{equation}
\end{proposition}

\begin{proof}
The number $q^Nx-y$ is an integer, so $T_q(q^Nx)=T_q(y)$, and the first
equation in \eqref{ratval1} is \eqref{eq:tail} with $n=N$. The digits of
$y$ have period $p$, so $q^py-y$ is an integer and $T_q(q^py)=T_q(y)$; now
\eqref{eq:tail} with $n=p$ and $y$ in place of $x$ gives
$T_q(y)=F_p(y)+q^{-p}T_q(y)$, which is the second equation. The partial
sums $F_N$ and $F_p$ take rational values at rational points because
$\tau_q$ does, so $T_q(x)$ is rational.

If $0\le m<q^p-1$ and $y=m/(q^p-1)$, then $q^py=m+y$, so the digits of
$y$ repeat the $p$ base $q$ digits of $m$, and $y=(m+y)/q^p$ lies in the
level $p$ interval with left endpoint $m/q^p$. By \eqref{cell} with
$n=p$ and $t=y$ and by \eqref{samp},
\[
 T_q(y)=\frac{\xi_p(m)}{q^p}+q^{-p}\Bigl(\bigl((q-1)p-2s_q(m)\bigr)y+T_q(y)\Bigr).
\]
Solving for $T_q(y)$ and inserting $\xi_p(m)=(q-1)pm-2S_q(m)$ from
\eqref{xi} and $y=m/(q^p-1)$ gives \eqref{ratval2}.
\end{proof}

For instance, let $q=3$. Then $1/4=(0.\overline{02})_3$,
$1/8=(0.\overline{01})_3$, and
$1/5=(0.\overline{0121})_3$ give $T_q(1/4)=3/4$, $T_q(1/8)=17/32$, and
$T_q(1/5)=7/10$, and the point $31/40=(0.\overline{2022})_3$ of
Section~\ref{sec:inf} has $T_q(31/40)=549/800$. For $q=2$ the point
$1/3=(0.\overline{01})_2$ has $m=1$, $p=2$, $S_2(1)=0$, and $s_2(1)=1$,
and \eqref{ratval2} gives $T_q(1/3)=2/3$. For $q=4$ the point
$2/5=6/15=(0.\overline{12})_4$ has $m=6$, $p=2$, $S_4(6)=9$, and
$s_4(6)=3$, and \eqref{ratval2} gives $T_q(2/5)=6/5$. By
Theorem~\ref{gmax}, these two values are the maxima of $T_q$ for $q=2$
and for $q=4$.

\begin{proposition}
\label{incr}
\begin{enumerate}
\item[(a)] For every $h\in\R$,
 $\sup_{x\in\R}|T_q(x+h)-T_q(x)|=T_q(h)$.
\item[(b)] For every integer $k\ge1$, $T_q(q^{-k})=(q-1)k\,q^{-k}$, and
 for every integer $r\ge0$,
 \[
  T_q\Bigl(\frac1{2q^r}\Bigr)=q^{-r}\Bigl(\frac{(q-1)r}2+T_q\bigl(\tfrac12\bigr)\Bigr),
 \]
 where $T_q(1/2)=(q+1)/4$ if $q$ is odd and $T_q(1/2)=q/4$ if $q$ is even.
 For $q=3$ this reads $T_q\bigl(1/(2\cdot3^r)\bigr)=(r+1)3^{-r}$.
\item[(c)] For $0<h\le1$ and $N=\lfloor\log_q(1/h)\rfloor$,
 \begin{equation}
  (q-1)Nh\le T_q(h)\le(q-1)Nh+\frac{q(q+1)}4\,h.
  \label{Thbd}
 \end{equation}
 In particular, as $h\downarrow0$,
 \begin{equation}
  T_q(h)=(q-1)h\log_q(1/h)+O(h).
  \label{Thlog}
 \end{equation}
 For $q=3$, \eqref{Thbd} reads $2Nh\le T_q(h)\le2Nh+3h$.
\end{enumerate}
\end{proposition}

\begin{proof}
We claim that $\tau_q(x+y)\le\tau_q(x)+\tau_q(y)$ for all real $x$ and $y$.
Since $\tau_q$ has period $1$ and $\tau_q(0)=0$, it suffices to prove this
for $0<x<1$ and $0<y<1$. Suppose first that $x+y\le1$. Since $\tau_q$ is
concave on $[0,1]$ and $\tau_q(0)=0$, we have
$\tau_q(x)\ge\frac x{x+y}\tau_q(x+y)$ and $\tau_q(y)\ge\frac y{x+y}\tau_q(x+y)$,
and adding the two inequalities gives the claim. Suppose now that
$x+y>1$. Then $1-x$ and $1-y$ are positive with sum less than $1$. Since
$\tau_q$ is even and has period $1$, the first case gives
$\tau_q(x+y)=\tau_q(2-x-y)\le\tau_q(1-x)+\tau_q(1-y)=\tau_q(x)+\tau_q(y)$.
Summing the claim with the positive weights in \eqref{T} gives
$T_q(x+y)\le T_q(x)+T_q(y)$. With $y=h$ this gives
$T_q(x+h)-T_q(x)\le T_q(h)$, and with $x+h$ and $-h$ in place of $x$ and
$y$ it gives $T_q(x)-T_q(x+h)\le T_q(-h)=T_q(h)$. Hence
$|T_q(x+h)-T_q(x)|\le T_q(h)$ for all $x$, with equality at $x=0$ because
$T_q(0)=0$. This proves (a).

By \eqref{tau}, $\tau_q(x)=(q-1)x$ for $0\le x\le1/q$. Let $h>0$, and let
$k\ge0$ be an integer with $q^kh\le1$. For $j<k$ we have $q^jh\le1/q$, so
the term of \eqref{T} with index $j$ equals $(q-1)h$, and \eqref{eq:tail}
gives $T_q(h)=(q-1)kh+q^{-k}T_q(q^kh)$. For $h=q^{-k}$ the last term is
zero, because $T_q(1)=0$, and for $h=1/(2q^r)$ and $k=r$ it is
$q^{-r}T_q(1/2)$. This gives the two equations in (b). If $q$ is odd, then
$T_q(1/2)=(q+1)/4$ by Theorem~\ref{limitthm}. If $q$ is even, then $q/2$
is an integer, so \eqref{fe} gives
$T_q(1/2)=\tau_q(1/2)+\frac1qT_q(q/2)=\tau_q(1/2)=q/4$. This proves (b).

Let $0<h\le1$. Then $N\ge0$ and $q^{-N-1}<h\le q^{-N}$, and
$0\le T_q\le(q+1)/4$ by Theorem~\ref{gmax}. With $k=N$, the equation for $T_q(h)$
just proved gives
$0\le T_q(h)-(q-1)Nh\le\frac{q+1}4q^{-N}<\frac{q(q+1)}4h$.
Since $|N-\log_q(1/h)|<1$, \eqref{Thlog} follows.
\end{proof}

\begin{corollary}
\label{holder}
For every $0<\alpha<1$ there is a constant $K_\alpha$ such that
$|T_q(x)-T_q(y)|\le K_\alpha|x-y|^\alpha$ for all real $x,y$. The function
$T_q$ is not Lipschitz.
\end{corollary}

\begin{proof}
Since $h^{1-\alpha}\log(1/h)$ is bounded for small $h>0$, parts (a)
and (c) of Proposition~\ref{incr} give the H\"older bound for small
$|x-y|$. For the remaining pairs the bound holds
because $T_q$ is bounded. By part (b), the quotients
$T_q(q^{-k})/q^{-k}=(q-1)k$
are unbounded, so no Lipschitz constant exists.
\end{proof}

For Takagi's function, Lagarias~\cite[Theorem~8.1]{Lagarias2012} showed
that $|T_2(x+h)-T_2(x)|\le2h\log_2(1/h)$ for $0\le x\le x+h\le1$ and small
$h$. His proof compares $T_2$ with its partial sum of order $n$, where
$2^{-n}\le h\le2^{1-n}$, and gives this inequality for $0<h\le1/4$.
Together with the equation $T_2(2^{-n})=n2^{-n}$, which he noted after
the theorem, it gives Corollary~\ref{holder} for $q=2$. The survey of Allaart and
Kawamura~\cite{AllaartKawamura2012} describes further results on the
continuity of Takagi's function.

We use the following inequality several times. The function $\tau_q$ is
linear on each interval $[k/q,(k+1)/q]$, and its values $k(q-k)/q$ at the
points $k/q$ with $0<k<q$ are positive. Hence $\tau_q(t)>0$ for $0<t<1$,
and since no term of the series \eqref{T} is negative,
\begin{equation}
 T_q(t)\ge\tau_q(t)>0~\text{for}~0<t<1.
 \label{Tpositive}
\end{equation}

The next lemma gives the smallest positive number at which a line through
the origin with a positive integer slope meets the graph of $T_q$.

\begin{lemma}
\label{firstcross}
Let $d\ge1$ be an integer, and write $d=(k-1)(q-1)+b$ with integers
$k\ge1$ and $1\le b\le q-1$. Then $T_q(w)>dw$ for $0<w<(q-b)/q^k$, and
$T_q(w)=dw$ for $w=(q-b)/q^k$.
\end{lemma}

\begin{proof}
Suppose first that $k=1$, so that $d=b$, and let $0<w\le(q-b)/q$. Write
$w=(a+y)/q$ with an integer $a$ such that $0\le a\le q-b-1$ and with
$0<y\le1$. By \eqref{branches},
\[
 q\bigl(T_q(w)-bw\bigr)=a(q-a-b)+(q-1-2a-b)\,y+T_q(y).
\]
The first two terms on the right form a linear function of $y$ whose
values $a(q-a-b)$ at $y=0$ and $(a+1)(q-a-b-1)$ at $y=1$ are
nonnegative, so this function is nonnegative on $[0,1]$. If
$y<1$, then $T_q(y)>0$ by \eqref{Tpositive}, and hence $T_q(w)>bw$. If
$y=1$, then $T_q(y)=0$ and $q(T_q(w)-bw)=(a+1)(q-a-b-1)$. This number is
positive for $a<q-b-1$, and it is zero for $a=q-b-1$, that is, for
$w=(q-b)/q$.

Let now $k\ge2$ and $0<w\le(q-b)/q^k$. For $0\le j\le k-2$ we have
$q^jw\le1/q$, so $\tau_q(q^jw)=(q-1)q^jw$ by \eqref{tau}, and therefore
$F_{k-1}(w)=(q-1)(k-1)w$. By equation \eqref{eq:tail},
$T_q(w)=(q-1)(k-1)w+q^{1-k}T_q(q^{k-1}w)$, so that
$T_q(w)-dw=q^{1-k}\bigl(T_q(v)-bv\bigr)$ with $v=q^{k-1}w$. Since
$0<v\le(q-b)/q$, the case $k=1$ gives the claim.
\end{proof}

The following proposition gives a lower bound for $T_q$ on the left half
of the period (see Figure~\ref{fig5}). For Takagi's function,
Kr\"uppel~\cite[Lemma~6.1]{Kruppel2007} proved an inequality of the same
kind on the interval $(0,1/3)$: there Takagi's function is at least $2x$,
with equality exactly at the points $(1-4^{-m})/3$ with $m\ge1$.

\begin{proposition}
\label{T2x}
For $0\le x\le1/2$ we have $T_q(x)\ge\lfloor(q+1)/2\rfloor\,x$. If $q$ is
odd, equality holds exactly at $0$, at $1/2$, and at the points
$(1-q^{-r})/2$ with $r\ge1$. If $q$ is even, equality holds exactly at
$0$ and at $1/2$. For $q=3$ the inequality reads $T_3(x)\ge2x$.
\end{proposition}

\begin{proof}
Let $H(x)=T_q(x)-\lfloor(q+1)/2\rfloor\,x$, and let
$b=\lfloor(q+1)/2\rfloor$. Then $1\le b\le q-1$ because $q\ge2$, and
$q-b=\lfloor q/2\rfloor$. Lemma~\ref{firstcross} with $k=1$ and $d=b$
shows that $H(x)>0$ for $0<x<\lfloor q/2\rfloor/q$ and that $H$ is zero
at $\lfloor q/2\rfloor/q$. Moreover, $H(0)=0$. If $q$ is even, then
$\lfloor q/2\rfloor/q=1/2$, and the proposition is proved.

Let $q$ be odd. Then $\lfloor q/2\rfloor/q=(q-1)/(2q)=(1-q^{-1})/2$. For
$(q-1)/(2q)\le x\le1/2$ put $y=qx-(q-1)/2$, so that $0\le y\le1/2$ and
$1/2-y=q(1/2-x)$. Equation \eqref{branches} with $a=(q-1)/2$ gives
$T_q(x)=\bigl((q^2-1)/4+T_q(y)\bigr)/q$, and
$\frac{q+1}2x=\bigl((q^2-1)/4+\frac{q+1}2y\bigr)/q$, so that
$H(x)=H(y)/q$. Now let $(q-1)/(2q)<x<1/2$. We apply the map
$u\mapsto qu-(q-1)/2$ repeatedly, starting at $x$, and stop at the first
point $x_s$ that lies in $[0,(q-1)/(2q)]$. This happens after finitely
many steps $s\ge1$, since each step multiplies the distance from $1/2$ by
$q$. The points before $x_s$ lie in the open interval from $(q-1)/(2q)$
to $1/2$, so $x_s$ lies in $(0,(q-1)/(2q)]$, and $H(x)=q^{-s}H(x_s)$.
Hence $H(x)\ge0$, with equality exactly when $x_s=(q-1)/(2q)$, that is,
when $x$ is obtained from $(q-1)/(2q)$ by $s$ applications of the inverse
map $u\mapsto(u+(q-1)/2)/q$. This map sends $(1-q^{-r})/2$ to
$(1-q^{-r-1})/2$, so these points are $(1-q^{-r})/2$ with $r=s+1\ge2$.
Together with the zeros $0$ and $(1-q^{-1})/2$ of $H$ on
$[0,(q-1)/(2q)]$ and with $H(1/2)=0$, which holds because
$T_q(1/2)=(q+1)/4$, this proves the proposition.
\end{proof}

\begin{figure}[H]
\centering
\includegraphics{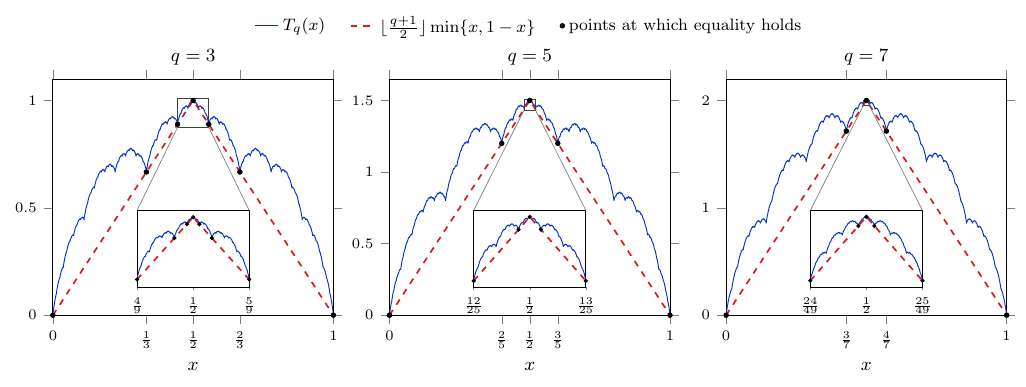}
\caption{Proposition~\ref{T2x} for $q=3$, $5$, and $7$: the graphs of
$T_q$ and $\lfloor(q+1)/2\rfloor\min\{x,1-x\}$, with enlarged views near
$1/2$.}
\label{fig5}
\end{figure}

Let $q$ be odd in this paragraph. In base $q$, the equality points
$(1-q^{-r})/2$ of Proposition~\ref{T2x} are written with $r$ digits
$(q-1)/2$ followed by zeros. By
Proposition~\ref{T2x} and the symmetry $T_q(1-x)=T_q(x)$, the graph of $T_q$
lies above the graph of $\frac{q+1}2\min\{x,1-x\}$ on $[0,1]$ and meets it
exactly at $0$, $1/2$, $1$, and $(1\pm q^{-r})/2$, $r\ge1$. The points
$(1\pm q^{-r})/2$ accumulate at $1/2$. For $q=3$ they are $1/3$, $4/9$,
$13/27,\ldots$ and $2/3$, $5/9$, $14/27,\ldots$, for $q=5$ they are
$2/5$, $12/25$, $62/125,\ldots$ and $3/5$, $13/25$, $63/125,\ldots$, and
for $q=7$ they are $3/7$, $24/49$, $171/343,\ldots$ and $4/7$, $25/49$,
$172/343,\ldots$. By the proof of
Proposition~\ref{T2x}, on each interval between consecutive meeting
points, except the first and the last of these intervals, the gap between
the graphs is a copy of the gap one step farther from $1/2$, reduced by
the factor $q$ in both directions. The insets of Figure~\ref{fig5} show
this on $[4/9,5/9]$, $[12/25,13/25]$, and $[24/49,25/49]$ for $q=3$, $5$,
and $7$. In this figure the graphs of $T_q$ are drawn through the values
at the multiples of $1/(2\cdot3^7)$, $1/(2\cdot5^5)$, and $1/(2\cdot7^4)$,
and in the insets at the multiples of $1/(2\cdot3^9)$, $1/(2\cdot5^7)$,
and $1/(2\cdot7^6)$.

For even $q$, Proposition~\ref{T2x} and the symmetry $T_q(1-x)=T_q(x)$
show that the two graphs meet only at $0$, $1/2$, and $1$. In this case
$1/2$ is a $q$-adic rational, so it is a strict local minimum of $T_q$ by
Theorem~\ref{thm1}, and near $1/2$ the graph of $T_q$ lies above the
horizontal line at height $T_q(1/2)=q/4$.

For a word
$\xi_1\ldots\xi_n$ over $\{0,1,\ldots,q-1\}$ we call the integers
$d_j=\sum_{i=1}^{j}(q-1-2\xi_i)$, $0\le j\le n$, the digit slopes of the
word, in agreement with Definition~\ref{digslope}. For $q=3$ the words
with $d_n=0$ are those with as many zeros as twos.

\begin{lemma}
\label{trinom}
For $n\ge0$ let $C_n$ be the number of words of length $n$ over
$\{0,1,\ldots,q-1\}$ with digit slope $d_n=0$. These are the words whose
letters add up to $(q-1)n/2$. If $(q-1)n$ is odd, then $C_n=0$, and if
$(q-1)n$ is even, then $C_n$ is the coefficient of $z^{(q-1)n/2}$ in
$(1+z+\cdots+z^{q-1})^n$. For every $n\ge1$ and every integer $k$, the
number of words of length $n$ over $\{0,1,\ldots,q-1\}$ with $d_n=k$ is
at most $q^n\sqrt{\pi q}/(2\sqrt n)$. Moreover,
\[
 \sqrt n\,\frac{C_n}{q^n}\to\sqrt{\frac6{\pi(q^2-1)}}
\]
as $n\to\infty$ through the integers $n$ for which $(q-1)n$ is even. For
$q=3$ the number $C_n$ is the coefficient of $z^n$ in $(1+z+z^2)^n$, the
central trinomial coefficient, $C_0,C_1,C_2,\ldots$ are
$1,1,3,7,19,51,\ldots$~\cite{OEIS-A002426}, and
$\sqrt n\,C_n/3^n\to\sqrt3/(2\sqrt\pi)$. For $q=2$ we have
$C_{2m}=\binom{2m}m$ for $m\ge0$, and $\sqrt{2m}\,C_{2m}/4^m\to\sqrt{2/\pi}$.
\end{lemma}

For $q=3$, the Online Encyclopedia of Integer
Sequences~\cite{OEIS-A002426} also lists the limit, and for every $q$,
Eger~\cite{Eger2014} obtained it from a local limit theorem of Gnedenko
and Kolmogorov. We include the short proof.

\begin{proof}
Since $d_n=(q-1)n-2(\xi_1+\cdots+\xi_n)$, a word of length $n$ has
$d_n=0$ exactly when its letters add up to $(q-1)n/2$. This is not
possible when $(q-1)n$ is odd, and the number of words of length
$n$ whose letters add up to an integer $m$ is the coefficient of $z^m$ in
$(1+z+\cdots+z^{q-1})^n$. For real $\theta$ put
\[
 \Upsilon(\theta)=\frac1q\sum_{\xi=0}^{q-1}e^{i(q-1-2\xi)\theta}.
\]
The number of words of length $n$ with $d_n=k$ is the coefficient of
$w^k$ in $\bigl(\sum_{\xi=0}^{q-1}w^{q-1-2\xi}\bigr)^n$, and taking
$w=e^{i\theta}$ shows that this number is
\begin{equation}
 \frac{q^n}{2\pi}\int_{-\pi}^{\pi}\Upsilon(\theta)^ne^{-ik\theta}\,d\theta.
 \label{slopeint}
\end{equation}
The terms of $\Upsilon(\theta)$ with the indices $\xi$ and $q-1-\xi$ are
complex conjugates, so $\Upsilon$ is real and
$\Upsilon(\theta)=\frac1q\sum_{\xi=0}^{q-1}\cos((q-1-2\xi)\theta)$. All
exponents $q-1-2\xi$ have the parity of $q-1$, so
$\Upsilon(\theta+\pi)=(-1)^{q-1}\Upsilon(\theta)$. The two terms with $\xi=0$ and
$\xi=1$ add up to $2e^{i(q-2)\theta}\cos\theta$, and each of the other
$q-2$ terms has absolute value $1$, so
$q|\Upsilon(\theta)|\le q-2+2|\cos\theta|$. For $|\theta|\le\pi/2$ we have
$1-\cos\theta=2\sin^2(\theta/2)\ge2\theta^2/\pi^2$, because
$\sin x\ge2x/\pi$ for $0\le x\le\pi/2$. Hence
\begin{equation}
 |\Upsilon(\theta)|\le1-\frac{4\theta^2}{q\pi^2}
 \le e^{-4\theta^2/(q\pi^2)},~|\theta|\le\frac\pi2.
 \label{psibound}
\end{equation}
Since $|\Upsilon|$ has period $\pi$, the absolute value of \eqref{slopeint}
is at most
\[
 \frac{q^n}{2\pi}\int_{-\pi}^{\pi}|\Upsilon(\theta)|^n\,d\theta
 =\frac{q^n}{\pi}\int_{-\pi/2}^{\pi/2}|\Upsilon(\theta)|^n\,d\theta
 \le\frac{q^n}{\pi}\int_{-\infty}^{\infty}e^{-4n\theta^2/(q\pi^2)}\,d\theta
 =\frac{q^n\sqrt{\pi q}}{2\sqrt n},
\]
which is the bound for the number of words with $d_n=k$.

Now let $(q-1)n$ be even. Then $\Upsilon(\theta+\pi)^n=\Upsilon(\theta)^n$, and
\eqref{slopeint} with $k=0$ gives
$C_n/q^n=\frac1\pi\int_{-\pi/2}^{\pi/2}\Upsilon(\theta)^n\,d\theta$. The
substitution $\theta=t/\sqrt n$ turns this into
\[
 \sqrt n\,\frac{C_n}{q^n}
 =\frac1\pi\int_{-\pi\sqrt n/2}^{\pi\sqrt n/2}
 \Upsilon\bigl(t/\sqrt n\bigr)^n\,dt.
\]
We have $\sum_{\xi=0}^{q-1}(q-1-2\xi)^2=q(q^2-1)/3$, because
$\sum_{\xi=0}^{q-1}\xi=q(q-1)/2$ and
$\sum_{\xi=0}^{q-1}\xi^2=q(q-1)(2q-1)/6$. Since
$\cos x=1-x^2/2+O(x^4)$, it follows that
$\Upsilon(\theta)=1-(q^2-1)\theta^2/6+O(\theta^4)$ as $\theta\to0$. Hence
the integrand tends to $e^{-(q^2-1)t^2/6}$ for every $t$, and by
\eqref{psibound} its absolute value is at most $e^{-4t^2/(q\pi^2)}$.
Dominated convergence gives
\[
 \sqrt n\,\frac{C_n}{q^n}\to\frac1\pi\int_{-\infty}^{\infty}e^{-(q^2-1)t^2/6}\,dt
 =\sqrt{\frac6{\pi(q^2-1)}}
\]
as $n\to\infty$ through the integers $n$ for which $(q-1)n$ is even. For
$q=3$ the limit is $\sqrt3/(2\sqrt\pi)$, and for $q=2$ it is
$\sqrt{2/\pi}$. A word over $\{0,1\}$ of length $2m$ has $d_{2m}=0$
exactly when it contains $m$ ones, so that $C_{2m}=\binom{2m}m$ for $q=2$.
\end{proof}

\section{Infinite derivatives and Dini derivatives}
\label{sec:inf}

By Theorem~\ref{nodiff} below, $T_q$ has no finite derivative at any
point, but the slopes of $T_q$ can still tend to $+\infty$ or to
$-\infty$. For Takagi's function, Allaart and
Kawamura~\cite[Theorem~3.1 and Corollary~3.2]{AllaartKawamura2010} and
Kr\"uppel~\cite[Propositions~3.1 and~4.5]{Kruppel2010a} characterized
the points with an infinite derivative from one side or from both sides
in terms of the binary digits. Kr\"uppel had earlier given sufficient
conditions~\cite[Section~5]{Kruppel2007}. Ferrera, G\'omez Gil, and
Llorente~\cite[Section~3]{FerreraGomezGilLlorente2019} obtained the
corresponding characterization for the Takagi--van~der~Waerden functions
of every base. This section proves Theorem~\ref{thm2}, which characterizes the
infinite derivatives of $T_q$ in every base $q$. For odd $q$ we also
compute the four Dini derivatives of $T_q$ at the midpoints of the
intervals of each level and at one further rational point, and for $q=3$
we bound the difference of the two right Dini derivatives from below.

The proof of Corollary~\ref{holder} shows that $\Delta(0,h)=T_q(h)/h$ is
unbounded as $h\downarrow0$. Delange proved in every base that $T_q$ has
no finite derivative at any point. His proof uses the slopes of his
function between the endpoints of intervals that contain the point.
Kr\"uppel~\cite[Proposition~5.4]{Kruppel2010b} derived from his
functional equation that every Knopp function of the kind described in
Section~\ref{sec:T} with a nonzero generator has no finite derivative
from either side at any point. In Theorem~\ref{nodiff} and its
proof, ``derivative'' means a finite derivative, with $h$ tending to $0$
from both sides.

\begin{theorem}[Delange {\cite[Section~3]{Delange1975}}]
\label{nodiff}
The function $T_q$ has no finite derivative at any point.
\end{theorem}

Delange stated the theorem for his auxiliary function, which equals
$-T_q/2$ (see Section~\ref{intro}). He showed that this function has no
derivative at any point of $(0,1)$ and no left derivative at $1$. Since
this function has period $1$, it has no derivative at any point.

We give another proof, based on \eqref{cell} and the symmetry
$T_q(1-t)=T_q(t)$. By these two facts, if a point of an interval of level
$n$ differs from its reflection in the midpoint of the interval, then
the slope of $T_q$ between the two points is the slope of the interval,
an integer. At a point with a
finite derivative these integers would be eventually constant. This is
impossible for even $q$, and for odd $q$ it leaves only the midpoints of
intervals, which we exclude by a computation at $1/2$.

\begin{proof}
Let $x_0\in\R$. By periodicity we may
assume that $x_0\in[0,1)$. Suppose that $T_q$ has a derivative $T_q'(x_0)$ at $x_0$,
and let $\xi_1,\xi_2,\ldots$ be the base $q$ digits of $x_0$, with the
terminating expansion if $x_0$ is a $q$-adic rational. For $n\ge0$ let
$m_n=\lfloor q^nx_0\rfloor$ and $t_n=q^nx_0-m_n$, so that
$m_{n+1}=qm_n+\xi_{n+1}$ and
$t_n=\sum_{j=1}^{\infty}\xi_{n+j}q^{-j}$. We write
$y_n=(m_n+1-t_n)/q^n$ for the reflection of $x_0$ in the midpoint of the
interval $[m_n/q^n,(m_n+1)/q^n]$ of level $n$. By \eqref{dnsq}, the slope
of this interval is the digit slope $d_n=d_n(x_0)$. By \eqref{cell} with
$t=t_n$ and with $t=1-t_n$, and since $T_q(1-t_n)=T_q(t_n)$,
\[
 T_q(y_n)-T_q(x_0)=d_n\,(y_n-x_0).
\]
Suppose that $t_n\ne1/2$, and hence $y_n\ne x_0$, for all large $n$.
Since $|y_n-x_0|\le q^{-n}$, the integers $\Delta(x_0,y_n)=d_n$ converge
to $T_q'(x_0)$, so they are eventually constant, and
$q-1-2\xi_{n+1}=d_{n+1}-d_n=0$ for all large $n$. This is impossible when
$q$ is even, because $q-1-2\xi_{n+1}$ is then odd. When $q$ is odd, it
gives $\xi_{n+1}=(q-1)/2$ for all large $n$, and then
$t_n=\sum_{j=1}^{\infty}\frac{q-1}2q^{-j}=\frac12$ for all large $n$, a
contradiction. Hence $t_n=1/2$ for infinitely many $n$. If $q$ is even
and $t_N=1/2$ for some $N$, then $q^{N+1}x_0=qm_N+q/2$ is an integer, and
so $t_n=0$ for every $n>N$. Therefore $q$ is odd, and there is an
$N\ge0$ with $t_N=1/2$, so that $x_0=(m_N+1/2)/q^N$.

It remains to exclude these midpoints for odd $q$. By \eqref{cell} with
$n=N$ and $m=m_N$, $T_q(x)$ is a linear function of $x$ plus
$q^{-N}T_q(q^Nx-m_N)$ on the level $N$ interval
$[m_N/q^N,(m_N+1)/q^N]$, which contains $x_0$ in its interior. Since
$q^Nx_0-m_N=1/2$, it suffices to show that $T_q$ is not differentiable at
$1/2$. By \eqref{tau} with $k=(q-1)/2$, the function $\tau_q$ is constant
on the interval $[1/2-1/(2q),1/2+1/(2q)]$. Since $(q-1)/2$ is an integer
and $T_q$ has period $1$, we have $T_q(q/2+qh)=T_q(1/2+qh)$ for every real
$h$. Thus for $0<|h|\le1/(2q)$, \eqref{fe} at the points $1/2+h$ and
$1/2$ gives
\[
 T_q(1/2+h)-T_q(1/2)=\tfrac1q\bigl(T_q(1/2+qh)-T_q(1/2)\bigr).
\]
Hence the slope $Q(h)=(T_q(1/2+h)-T_q(1/2))/h$ satisfies $Q(h)=Q(qh)$ for
$0<|h|\le1/(2q)$. Since $T_q(1)=0$ and $T_q(1/2)=(q+1)/4$ by
Theorem~\ref{limitthm}, we have $Q(1/2)=-(q+1)/2$. From \eqref{branches}
with $a=(q+1)/2$ and $y=1/2$ we get
\[
 T_q\Bigl(\frac12+\frac1q\Bigr)
 =\frac1q\Bigl(\frac{q^2-1}4-1+\frac{q+1}4\Bigr)=\frac{q+1}4-\frac1q,
\]
so $Q(1/q)=-1$. By induction on $k$, the equation $Q(h)=Q(qh)$ gives
$Q(q^{-k})=-1$ and $Q(q^{-k}/2)=-(q+1)/2$ for every $k\ge1$. Both
$q^{-k}$ and $q^{-k}/2$ tend to $0$ as $k\to\infty$, and $Q$ has
different limits along these two sequences, because $(q+1)/2\ge2$.
Therefore $T_q'(1/2)$ does not exist, and neither does $T_q'(x_0)$.
\end{proof}

\begin{example}
\label{ex:quarter}
Let $q=3$. By Theorem~\ref{nodiff}, $T_q$ has no finite derivative at
$1/4=(0.\overline{02})_3$. At this point the slopes of $T_q$ show that $T_q$
has no derivative from either side, finite or infinite. Let $k\ge0$. The
ternary interval of level $2k$ that contains $1/4$ is
$[m/9^k,(m+1)/9^k]$ with $m=(9^k-1)/4$, whose word consists of $k$
strings $02$. Its slope is $0$, and $1/4=(m+1/4)/9^k$. Hence \eqref{cell}
gives
\[
 T_q\Bigl(\frac14+\frac u{9^k}\Bigr)-T_q\Bigl(\frac14\Bigr)
 =\frac1{9^k}\Bigl(T_q\Bigl(\frac14+u\Bigr)-T_q\Bigl(\frac14\Bigr)\Bigr),
 ~-\frac14\le u\le\frac34,
\]
so that $\Delta(1/4,1/4+u/9^k)=\Delta(1/4,1/4+u)$ for $u\ne0$. With
$T_q(1/4)=3/4$, $T_q(1/2)=1$, $T_q(1)=T_q(0)=0$, and $T_q(1/6)=2/3$, this slope is
$1$ for $u=1/4$, $-1$ for $u=3/4$, $3$ for $u=-1/4$, and $1$ for
$u=-1/12$. Letting $k\to\infty$, we see that $\Delta(1/4,1/4+h)$ has no
limit as $h\downarrow0$ and no limit as $h\uparrow0$.
\end{example}

For Takagi's function,
Kr\"uppel~\cite[Proposition~3.2]{Kruppel2007} proved that
$(T_q(a+h)-T_q(a))/(|h|\log_2(1/|h|))\to1$ as $h\to0$ at every dyadic
rational $a$, and this also follows from \eqref{cusp1} below with $q=2$.

\begin{proposition}
\label{cusp}
At every $q$-adic rational $a$,
\begin{equation}
 T_q(a+h)-T_q(a)=(q-1)|h|\log_q\frac1{|h|}+O(|h|)~(h\to0),
 \label{cusp1}
\end{equation}
with an implied constant depending on $a$. Consequently $a$ is a strict
local minimum of $T_q$, with $(T_q)'_+(a)=+\infty$ and $(T_q)'_-(a)=-\infty$.
\end{proposition}

\begin{proof}
By periodicity we may take $a=k/q^n\in[0,1)$. If $a=0$, the claim
follows from \eqref{Thlog} and the evenness of $T_q$.
Otherwise $1\le k<q^n$, and for $0<h\le q^{-n}$ we apply \eqref{cell}
with $m=k$ and $t=q^nh$ and get
\[
 T_q(a+h)-T_q(a)=\bigl((q-1)n-2s_q(k)\bigr)h+q^{-n}T_q(q^nh),
\]
where $q^{-n}T_q(q^nh)=(q-1)h\log_q(1/h)-(q-1)nh+O(h)$ by
\eqref{Thlog}. If $-q^{-n}\le h<0$, then \eqref{cell2} with $m=k$ and
$t=q^n|h|$ gives the same expression with
$(2s_q(k-1)-(q-1)n)|h|$ in place of $((q-1)n-2s_q(k))h$ and with $|h|$
in place of $h$ in the last term. In both cases we
obtain \eqref{cusp1}, and the three consequences follow because the
logarithmic term dominates the remainder.
\end{proof}

\begin{example}
\label{ex:cusp}
Let $q=3$. The point $1/3$ is a ternary rational, so by Proposition~\ref{cusp} it is
a strict local minimum of $T_q$, with $(T_q)'_+(1/3)=+\infty$ and
$(T_q)'_-(1/3)=-\infty$. The level $1$ interval $[1/3,2/3]$
has slope $0$, and the interval $[0,1/3]$ to its left has slope $2$.
Therefore \eqref{cell} and \eqref{cell2} with $n=m=1$ give
$T_q(1/3+t/3)=2/3+T_q(t)/3$ and $T_q(1/3-t/3)=2/3+(T_q(t)-2t)/3$ for
$0\le t\le1$. We take $t=3^{1-k}$ and use $T_q(3^{1-k})=2(k-1)3^{1-k}$,
which is Proposition~\ref{incr}(b) for $k\ge2$ and $T_q(1)=0$ for $k=1$.
For every $k\ge1$ this gives
\[
 T_q\Bigl(\frac13+3^{-k}\Bigr)-T_q\Bigl(\frac13\Bigr)=2(k-1)\,3^{-k},
 ~
 T_q\Bigl(\frac13-3^{-k}\Bigr)-T_q\Bigl(\frac13\Bigr)=2(k-2)\,3^{-k}.
\]
With $h=3^{-k}$ these increments are $2h\log_3(1/h)-2h$ and
$2h\log_3(1/h)-4h$. Thus the two sides of the cusp at $1/3$ have the
same leading term $2h\log_3(1/h)$ and differ already in the term of
order~$h$.
\end{example}

\begin{definition}
\label{dini}
Let $x\in\R$. The right and left derivatives of $T_q$ at $x$ are the limits
\[
 (T_q)'_+(x)=\lim_{h\downarrow0}\frac{T_q(x+h)-T_q(x)}h,~
 (T_q)'_-(x)=\lim_{h\uparrow0}\frac{T_q(x+h)-T_q(x)}h
\]
in the extended real line, when they exist. We write $T_q'(x)=+\infty$ if
both equal $+\infty$, and $T_q'(x)=-\infty$ if both equal $-\infty$. The
four Dini derivatives of $T_q$ at $x$ are
\begin{align*}
 D^+T_q(x)&=\limsup_{h\downarrow0}\frac{T_q(x+h)-T_q(x)}h,&
 D_+T_q(x)&=\liminf_{h\downarrow0}\frac{T_q(x+h)-T_q(x)}h,\\
 D^-T_q(x)&=\limsup_{h\uparrow0}\frac{T_q(x+h)-T_q(x)}h,&
 D_-T_q(x)&=\liminf_{h\uparrow0}\frac{T_q(x+h)-T_q(x)}h.
\end{align*}
\end{definition}

The Dini derivatives always exist in the extended real line.

Throughout the rest of this section we fix a point $x\in(0,1)$ that is
not a $q$-adic rational, with digit sequence $\xi_1,\xi_2,\ldots$, so
that $x=\sum_{j=1}^{\infty}\xi_jq^{-j}$. For $n\ge0$ we write
\[
 k_n=\lfloor q^nx\rfloor,~ t_n=q^nx-k_n\in(0,1),~
 a_n=\frac{k_n}{q^n},~ c_n=a_n+q^{-n},
\]
so that $[a_n,c_n]$ is the interval of level $n$ containing
$x$, and $k_n$ is the integer whose base $q$ digits, written from the
highest power of $q$ down and with leading zeros allowed, are
$\xi_1,\ldots,\xi_n$. The point $t_n=\sum_{j=1}^{\infty}\xi_{n+j}q^{-j}$
has the digits $\xi_{n+1},\xi_{n+2},\ldots$. In this notation the digit
slopes of Definition~\ref{digslope} are
\begin{equation}
 d_n(x)=\sum_{j=1}^n(q-1-2\xi_j)=(q-1)n-2s_q(k_n),
 \label{dn}
\end{equation}
and $d_n(x)$ is the slope \eqref{slope} of $F_n$ on $[a_n,c_n]$, which
also equals $\Delta(a_n,c_n)$. For $q=3$ we have
$d_n(x)=2\sum_{j=1}^n(1-\xi_j)=2\bigl(n-s_3(k_n)\bigr)$.

\begin{lemma}
\label{lem:subtract}
Let $n\ge0$ and $y=(k_n+u)/q^n$ with $u\in[0,1]$, $u\ne t_n$. Then
\begin{equation}
 \frac{T_q(y)-T_q(x)}{y-x}
 =d_n(x)+\frac{T_q(u)-T_q(t_n)}{u-t_n}.
 \label{sub}
\end{equation}
In particular,
\begin{equation}
 \frac{T_q(x)-T_q(a_n)}{x-a_n}=d_n(x)+\frac{T_q(t_n)}{t_n},
 ~
 \frac{T_q(c_n)-T_q(x)}{c_n-x}=d_n(x)-\frac{T_q(1-t_n)}{1-t_n},
 \label{ends}
\end{equation}
and the slope of $T_q$ between $x$ and its reflection in
the midpoint of $[a_n,c_n]$, the point $(k_n+1-t_n)/q^n$, is
$d_n(x)$ when $t_n\ne1/2$. If moreover $n\ge1$ and $\xi_n\ne0$,
then for $0<s\le q^{-n}$,
\begin{equation}
 \frac{T_q(a_n)-T_q(a_n-s)}{s}=d_n(x)+2-\frac{T_q(q^ns)}{q^ns}.
 \label{leftnb}
\end{equation}
\end{lemma}

\begin{proof}
We obtain \eqref{sub} by subtracting the two instances of \eqref{cell}
with $m=k_n$ for $x$ and $y$, in which the coefficient
$(q-1)n-2s_q(k_n)$ is $d_n(x)$ by \eqref{dn}. With $T_q(0)=T_q(1)=0$ and
$T_q(1-t)=T_q(t)$, the cases $u=0$ and $u=1$ give \eqref{ends}, and the case
$u=1-t_n$ gives the claim about the reflection. If $\xi_n\ne0$, the
integer $k_n$ ends in a nonzero base $q$ digit, so
$s_q(k_n-1)=s_q(k_n)-1$. By \eqref{slope}, the level $n$ interval
$[a_n-q^{-n},a_n]$ then has slope $(q-1)n-2s_q(k_n-1)=d_n(x)+2$.
Now \eqref{cell2} with $m=k_n$ and $t=q^ns$ gives \eqref{leftnb}.
\end{proof}

\begin{remark}
\label{rem:subtract}
Subtracting the instances of \eqref{cell} for two points $(k_n+u)/q^n$
and $(k_n+u')/q^n$ of $[a_n,c_n]$ with $u\ne u'$ shows that the slope of
$T_q$ between them is $d_n(x)+(T_q(u)-T_q(u'))/(u-u')$. Since
$0\le T_q\le(q+1)/4$ by Theorem~\ref{gmax}, the slope of $T_q$ between two
points of $[a_n,c_n]$ at distance more than $q^{-n}/q$ lies between
$d_n(x)-q(q+1)/4$ and $d_n(x)+q(q+1)/4$. For $q=3$ these two numbers are
$d_n(x)-3$ and $d_n(x)+3$.
\end{remark}

For $y\in[0,1]$, $y\ne x$, let $n(y)$ be the largest $n$ with
$y\in[a_n,c_n]$. It exists because the intervals $[a_n,c_n]$ shrink to
$\{x\}$. Since $x$ is an interior point of each of these intervals,
$n(y)\to\infty$ as $y\to x$. If $y>x$, then $y>c_{n(y)+1}$, and if
$y<x$, then $y<a_{n(y)+1}$.

\begin{definition}
\label{runlength}
Let $x\in(0,1)$ be a point that is not a $q$-adic rational. A run of
zeros of $x$ is a maximal string of consecutive digits of $x$
equal to $0$, and its length is the number of these digits. Runs of the
digit $q-1$ are defined in the same way. The runs of zeros of $x$ have
bounded length if there is an integer $b\ge0$ such that each of them has
at most $b$ digits, and similarly for the runs of the digit $q-1$. For
$q=3$ we call the runs of the digit $q-1$ the runs of twos.
\end{definition}

\begin{lemma}
\label{lem:runs}
Let $b\ge0$ be an integer, and suppose that every run of zeros and every
run of the digit $q-1$ of $x$ has length at most $b$. Then for every
$y\in[0,1]$, $y\ne x$,
\begin{equation}
 \left|\frac{T_q(y)-T_q(x)}{y-x}-d_{n(y)}(x)\right|
 \le(q-1)(b+2)+\frac{q(q+1)}4.
 \label{runs1}
\end{equation}
For $q=3$ the right side is $2b+7$.
\end{lemma}

\begin{proof}
Let $n=n(y)$, write $y=(k_n+u)/q^n$ with $u\in[0,1]$, and let
$i=\xi_{n+1}$ be the first digit of $t_n$. The number $qt_n-i\in(0,1)$
has the digits $\xi_{n+2},\xi_{n+3},\ldots$. Any run of zeros or of the
digit $q-1$ at the start of these digits has length at most $b$, so
$qt_n-i$ lies at distance at
least $q^{-b-1}$ from $0$ and from $1$. Hence $t_n$ lies at distance at
least $q^{-b-2}$ from both endpoints of $[i/q,(i+1)/q]$. The point $y$
is not in the level $n+1$ interval $[a_{n+1},c_{n+1}]$, so
$u\notin(i/q,(i+1)/q)$ and $h=|u-t_n|\ge q^{-b-2}$. By
Proposition~\ref{incr}(a), the last term in \eqref{sub} is at most
$T_q(h)/h$ in absolute value, and since $\lfloor\log_q(1/h)\rfloor\le b+2$,
the bound \eqref{Thbd} gives $T_q(h)/h\le(q-1)(b+2)+q(q+1)/4$.
\end{proof}

Let $n_1<n_2<\cdots$ be the positions of the nonzero digits of $x$, as
in Theorem~\ref{thm2}, and put
\[
 \Gamma_k=d_{n_k}(x)-(q-1)(n_{k+1}-n_k)+(q-1)\log_q(n_{k+1}-n_k),
\]
so that the first condition in Theorem~\ref{thm2}(b)
reads $\Gamma_k\to+\infty$. In the proof of Theorem~\ref{thm2} we also
use the bound \eqref{Thbd} divided by $h$, with $t$ in place of $h$:
\begin{equation}
 \frac{T_q(t)}{t}=(q-1)\Bigl\lfloor\log_q\frac1t\Bigr\rfloor+\theta,
 ~0\le\theta\le\frac{q(q+1)}4,~0<t\le1,
 \label{theta}
\end{equation}
where the number $\theta$ depends on $t$.

To bound $\Delta(x,y)$ from below for the first statements in
Theorem~\ref{thm2}(a) and (b), we split the interval between $x$ and $y$
at a $q$-adic rational close to $x$, so that
$\Delta(x,y)$ is a weighted mean of the slopes of $T_q$ over the two
pieces, and we bound these slopes by Lemma~\ref{lem:subtract},
\eqref{cell}, and \eqref{theta}. In (b) this mean depends on the ratio of the lengths of the
two pieces, and the ratio for which it is smallest gives the term
$(q-1)\log_q(n_{k+1}-n_k)$ in $\Gamma_k$.

\begin{proof}[Proof of Theorem~\ref{thm2}]
Replacing $x$ by $1-x$ replaces each digit $\xi_j$ by $q-1-\xi_j$, so
it changes $d_n(x)$ into $-d_n(x)$ and exchanges the positions $n_k$
with the positions $m_k$. It also gives $(T_q)'_+(1-x)=-(T_q)'_-(x)$ and
$(T_q)'_-(1-x)=-(T_q)'_+(x)$, because $T_q(1-x+h)-T_q(1-x)=T_q(x-h)-T_q(x)$. Therefore the
second statements in (a) and (b) follow from the first. Part
(c) follows from (a) and (b) once we know that the first condition in
(b) implies $d_n(x)\to+\infty$ and, by symmetry, the second implies
$d_n(x)\to-\infty$. Indeed, if $n_k\le n<n_{k+1}$, the digits
in positions $n_k+1,\ldots,n$ are zeros, so $d_n(x)\ge d_{n_k}(x)$, and
$d_{n_{k+1}}(x)\ge d_{n_{k+1}-1}(x)-(q-1)$. Put $g=n_{k+1}-n_k\ge1$. Since
$q^{g-1}\ge2^{g-1}\ge g$, we have $g-\log_qg\ge1$, and it follows that
$d_n(x)\ge d_{n_k}(x)-(q-1)\ge\Gamma_k$ for
$n_k\le n\le n_{k+1}$. It remains to prove the first statements in (a)
and (b).

(a) Suppose that $d_n(x)\to+\infty$. Let $x<y\le1$, and put $n=n(y)$
and $N=n+1$, so that $y>c_N$. Then $\Delta(x,y)$ is a weighted mean of
$\Delta(x,c_N)$ and $\Delta(c_N,y)$. Let $\ell\ge0$ be the number of
consecutive digits $q-1$ immediately after position $N$, which is finite
because $x$ is not a $q$-adic rational. The next digit is at most $q-2$,
so $1-t_N\ge q^{-\ell-1}$, and by \eqref{ends} and \eqref{theta},
\[
 \frac{T_q(c_N)-T_q(x)}{c_N-x}
 \ge d_N(x)-(q-1)(\ell+1)-q(q+1)/4=d_{N+\ell}(x)-(q-1)-q(q+1)/4,
\]
where the last equality holds because the $\ell$ digits $q-1$ lower the
digit slope by $(q-1)\ell$. If $y\le c_N+q^{-N}$, then $y$ lies in the
level $N$ interval $[c_N,c_N+q^{-N}]$, which is contained in $[a_n,c_n]$
because $c_N<y\le c_n$ and both $c_N$ and $c_n$ are multiples of
$q^{-N}$. Since adding $1$ to an integer raises its base $q$ digit sum by
at most $1$, the slope $(q-1)N-2s_q(k_N+1)$ of this interval is at least
$d_N(x)-2$, and \eqref{cell} on this interval gives
$\Delta(c_N,y)\ge d_N(x)-2$. If
$y>c_N+q^{-N}$, then $c_N$ and $y$ lie in $[a_n,c_n]$ at distance more
than $q^{-N}=q^{-n}/q$, so $\Delta(c_N,y)\ge d_n(x)-q(q+1)/4$ by
Remark~\ref{rem:subtract}. Since $q-1+q(q+1)/4\ge2$, in both cases
we obtain
\[
 \frac{T_q(y)-T_q(x)}{y-x}\ge\min_{i\ge n(y)}d_i(x)-(q-1)-q(q+1)/4,
\]
where the minimum exists because $d_i(x)\to+\infty$. Since
$n(y)\to\infty$ as $y\downarrow x$, we get $(T_q)'_+(x)=+\infty$.

Conversely, suppose that $(T_q)'_+(x)=+\infty$, and fix $n$. If $t_n<1/2$, the
reflection $(k_n+1-t_n)/q^n$ of $x$ in the midpoint of $[a_n,c_n]$ lies in
$(x,c_n]$, and by Lemma~\ref{lem:subtract} the slope of $T_q$ between $x$
and this reflection is $d_n(x)$. If $t_n\ge1/2$, then $\Delta(x,c_n)\le
d_n(x)$ by \eqref{ends}. In both cases we have $d_n(x)\ge\Delta(x,y)$ for
some $y$ with $x<y\le x+q^{-n}$, and since $(T_q)'_+(x)=+\infty$, it follows
that $d_n(x)\to+\infty$.

(b) Suppose that $\Gamma_k\to+\infty$. Then $d_n(x)\to+\infty$, as shown
above. Let $0\le y<x$, and put $n=n(y)$ and $N=n+1$, so that $y<a_N$.
Since $y\ge a_n$, the point $a_N$ is not $a_n$, so $\xi_N\ne0$ and
$N=n_k$ for some $k$. Put $g=n_{k+1}-n_k$ and $r=g-1$. After position
$N$ the digit sequence continues with $r$ zeros and then a nonzero digit,
so $q^{-g}<t_N<q^{-r}$ and $\lfloor\log_q(1/t_N)\rfloor=r$. Let
$h=x-a_N=q^{-N}t_N$ and $s=a_N-y$.

Suppose first that $s\le q^{-N}$, and put $\lambda=s/h$ and
$v=q^Ns=\lambda t_N$. By
\eqref{ends}, \eqref{leftnb}, and
\eqref{theta},
\[
 \Delta(a_N,x)=\frac{T_q(x)-T_q(a_N)}{h}=d_N(x)+(q-1)r+\theta_1,
 ~
 \Delta(y,a_N)=\frac{T_q(a_N)-T_q(y)}{s}=d_N(x)+2-(q-1)M-\theta_2,
\]
where $M=\lfloor\log_q(1/v)\rfloor$ and $0\le\theta_1,\theta_2\le q(q+1)/4$.
Since $M\le\log_q(1/v)=\log_q(1/t_N)-\log_q\lambda<r+1-\log_q\lambda$,
$\theta_1\ge0$, and $\theta_2\le q(q+1)/4$, this gives
\[
 \Delta(a_N,x)\ge d_N(x)+(q-1)r,
 ~
 \Delta(y,a_N)>d_N(x)-(q-1)r+(q-1)\log_q\lambda-(q-3+q(q+1)/4).
\]
Since $x-y=h+s$ and $s=\lambda h$, we have
$\Delta(y,x)=(\Delta(a_N,x)+\lambda\,\Delta(y,a_N))/(1+\lambda)$, and
therefore
\begin{equation}
 \frac{T_q(x)-T_q(y)}{x-y}
 >d_N(x)-(q-1)r
 +\frac{(q-1)(2r+\lambda\log_q\lambda)-(q-3+q(q+1)/4)\lambda}{1+\lambda}.
 \label{lam}
\end{equation}
We claim that
\begin{equation}
 \frac{2r+\lambda\log_q\lambda}{1+\lambda}>\log_q(r+1)-2
 ~\text{for all }\lambda>0.
 \label{lamclaim}
\end{equation}
To prove \eqref{lamclaim}, we note that $\log_2w<w$ for every $w>0$,
because $2^w>w$, that $0\le\log_qw\le\log_2w$ for every $w\ge1$, and
that $r\ge\log_q(r+1)$, because $q^r\ge2^r\ge r+1$. If $\lambda\le1$,
then $\lambda\log_q(1/\lambda)\le\lambda\log_2(1/\lambda)<1$, so the
fraction in \eqref{lamclaim} is greater than $(2r-1)/(1+\lambda)$. This
number is at least $(2r-1)/2\ge\log_q(r+1)-2$ if $r\ge1$, and at least
$-1>\log_q(r+1)-2$ if $r=0$. If
$\lambda\ge1$, then $\log_q\lambda<\lambda<1+\lambda$, and hence
$\lambda\log_q\lambda/(1+\lambda)
=\log_q\lambda-\log_q\lambda/(1+\lambda)>\log_q\lambda-1$. Thus the
fraction in \eqref{lamclaim} is greater than
$2r/(1+\lambda)+\log_q\lambda-1$. If $\lambda\ge r+1$, this number is at
least $\log_q(r+1)-1$. If $\lambda<r+1$, then $r\ge1$, and we put
$w=(r+1)/\lambda>1$. Since $1+\lambda\le2\lambda$, we have
$2r/(1+\lambda)\ge r/\lambda=wr/(r+1)\ge w/2$, and
$\log_q\lambda=\log_q(r+1)-\log_qw$, where
$\log_qw\le\log_2w=1+\log_2(w/2)<1+w/2$. Hence
$2r/(1+\lambda)+\log_q\lambda-1>\log_q(r+1)-2$, which proves
\eqref{lamclaim}. Since $q-3+q(q+1)/4>0$ and $\lambda/(1+\lambda)<1$, the
fraction in \eqref{lam} is greater than
$(q-1)(\log_q(r+1)-2)-(q-3+q(q+1)/4)$ by \eqref{lamclaim}, and \eqref{lam}
gives
\[
 \frac{T_q(x)-T_q(y)}{x-y}>d_N(x)-(q-1)r+(q-1)\log_q(r+1)-(3q-5+q(q+1)/4)
 =\Gamma_k-(2q-4+q(q+1)/4).
\]
Suppose now that $s>q^{-N}$. Then $\xi_N\ge2$ and
$y\in[a_n,a_N-q^{-N})$, because
$a_N-a_n=\xi_Nq^{-N}$. Thus $y$ and $a_N$ lie in $[a_n,c_n]$ at
distance more than $q^{-n}/q$, and Remark~\ref{rem:subtract} gives
$\Delta(y,a_N)\ge d_n(x)-q(q+1)/4$. By
\eqref{ends}, $\Delta(a_N,x)\ge d_N(x)$, so
$\Delta(y,x)\ge\min\{d_n(x),d_N(x)\}-q(q+1)/4$. In both cases we have
\[
 \frac{T_q(y)-T_q(x)}{y-x}
 \ge\min\{\Gamma_k-(2q-4+q(q+1)/4),\,d_n(x)-q(q+1)/4,\,d_N(x)-q(q+1)/4\},
\]
where $n=n(y)$ and $N=n_k$ tend to infinity as $y\uparrow x$. Hence
$(T_q)'_-(x)=+\infty$.

For the converse, fix $k\ge1$, let $N=n_k$, $g=n_{k+1}-n_k$, $r=g-1$, and
$h=x-a_N$, and put $\lambda=g$ and $y_k=a_N-\lambda h$.
Since $t_N<q^{-r}$ and $g=r+1\le2^r\le q^r$, we have
$\lambda h\le q^{-N}$, so
$y_k$ lies in $[a_N-q^{-N},a_N)$ and the formulas for $\Delta(a_N,x)$ and
$\Delta(y,a_N)$ apply with $y=y_k$. Now
$\Delta(a_N,x)\le d_N(x)+(q-1)r+q(q+1)/4$, and since
$M>\log_q(1/v)-1>r-\log_q\lambda-1$, we also have
$\Delta(y_k,a_N)<d_N(x)-(q-1)r+(q-1)\log_q\lambda+q+1$. Therefore
\begin{align*}
 \frac{T_q(x)-T_q(y_k)}{x-y_k}
 &<d_N(x)-(q-1)r
 +\frac{2(q-1)r+q(q+1)/4+(q-1)\lambda\log_q\lambda+(q+1)\lambda}{1+\lambda}\\
 &=d_N(x)-(q-1)r+\frac{(3q-1)r+q(q+1)/4+q+1}{r+2}
 +\frac{(q-1)(r+1)}{r+2}\log_q(r+1),
\end{align*}
where we used $\lambda=r+1$. Since $(r+1)/(r+2)<1$ and
$d_N(x)-(q-1)r+(q-1)\log_q(r+1)=\Gamma_k+q-1$, the right side is at most
\[
 \Gamma_k+q-1+\frac{(3q-1)r+q(q+1)/4+q+1}{r+2}
 =\Gamma_k+\frac{(4q-2)r+q(q+1)/4+3q-1}{r+2}.
\]
Now $q(q+1)/2+4-(4q-2)=(q-3)(q-4)/2\ge0$, because $q$ is an integer. Thus
$4q-2\le q(q+1)/2+4$ and $q(q+1)/4+3q-1\le q(q+1)/4+4q-2\le2(q(q+1)/2+4)$, so
the last fraction is at most $q(q+1)/2+4$. Moreover,
$0<x-y_k=(g+1)h<(r+2)q^{-N-r}\le2q^{-N}$, because
$r+2\le2^{r+1}\le2q^r$. Thus, for every $k\ge1$,
the point $y_k=a_{n_k}-(n_{k+1}-n_k)(x-a_{n_k})$ satisfies
\begin{equation}
 \Delta(y_k,x)<\Gamma_k+q(q+1)/2+4~\text{and}~0<x-y_k\le2q^{-n_k}.
 \label{ykbound}
\end{equation}
For $q=3$ the first inequality reads $\Delta(y_k,x)<\Gamma_k+10$.
If $(T_q)'_-(x)=+\infty$, then $\Delta(y_k,x)\to+\infty$ as $k\to\infty$,
because $x-y_k\to0$, and hence $\Gamma_k\to+\infty$.
\end{proof}

Let $q=2$. Then $d_{n_k}(x)=n_k-2k$, because $k$ of the first $n_k$
digits of $x$ are ones. Hence
$\Gamma_k=2n_k-n_{k+1}-2k+\log_2(n_{k+1}-n_k)$. Allaart and
Kawamura~\cite[Theorem~3.1]{AllaartKawamura2010} wrote $a_k$ for the
positions $n_k$ and $b_k$ for the positions $m_k$. Their condition for
$(T_q)'_-(x)=+\infty$ is $a_{k+1}-2a_k+2k-\log_2(a_{k+1}-a_k)\to-\infty$,
which is $\Gamma_k\to+\infty$, and their condition for
$(T_q)'_+(x)=+\infty$ is $a_k-2k\to\infty$, which is equivalent to
$d_n(x)\to+\infty$ because $d_n(x)\ge d_{n_k}(x)$ for
$n_k\le n<n_{k+1}$. Their conditions for $(T_q)'_+(x)=-\infty$ and
$(T_q)'_-(x)=-\infty$ are, in this order, the same two conditions with
$b_k$ in place of $a_k$, which are the second conditions in
Theorem~\ref{thm2}(b) and~(a), because $d_{m_k}(x)=2k-m_k$. Thus
Theorem~\ref{thm2} with $q=2$ is their theorem.

Example~\ref{ex:squares} illustrates Theorem~\ref{thm2}(c) at an irrational
point with runs of zeros of unbounded length.

\begin{example}
\label{ex:squares}
Let $x=\sum_{k=1}^{\infty}q^{-k^2}$. Its digit in position $j$ is $1$
if $j$ is a square and $0$ otherwise. Then $x$ is irrational,
$T_q'(x)=+\infty$, and $T_q'(1-x)=-\infty$. For $q=3$ this point is
$(0.100100001\ldots)_3$.
\end{example}

\begin{proof}
The runs of zeros of $x$ grow without bound, and $x$ has infinitely many
digits $1$, so its digit sequence is not eventually periodic. Thus $x$ is
irrational, and in particular $x$ is not a $q$-adic rational. The
nonzero digits of $x$ are at the positions $n_k=k^2$. In the digit
slopes, each digit $0$ adds $q-1$ and each digit $1$ adds $q-3$, so
$d_{n_k}(x)=(q-1)(k^2-k)+(q-3)k=(q-1)k^2-2k$. With $n_{k+1}-n_k=2k+1$
this gives
\[
 \Gamma_k=(q-1)k^2-2qk-(q-1)+(q-1)\log_q(2k+1)\to+\infty.
\]
By Theorem~\ref{thm2}(c), $T_q'(x)=+\infty$. Replacing $x$ by $1-x$
changes the sign of every slope of $T_q$, as in the proof of
Theorem~\ref{thm2}, so $T_q'(1-x)=-\infty$.
\end{proof}

For $q=2$, Corollary~\ref{rat}(a) below is the criterion for Takagi's
function in terms of bounded runs of binary digits.
Kr\"uppel~\cite[Proposition~5.3]{Kruppel2007} assumed that both the runs
of zeros and the runs of ones are bounded, while Allaart and
Kawamura~\cite[Example~3.4]{AllaartKawamura2010} and
Kr\"uppel~\cite[Remark~4.7]{Kruppel2010a} gave the version that needs only
one kind of bounded run. Corollary~\ref{rat}(b) treats the rational
points, as Kr\"uppel~\cite[Proposition~5.4]{Kruppel2007} did for Takagi's
function, and Figure~\ref{fig6} illustrates it for $q=3$.

\begin{definition}
\label{balanced}
A rational point $x\in(0,1)$ that is not a $q$-adic rational has an
eventually periodic digit sequence. Such a point is balanced if the sum
of the digits in a period of this sequence is $(q-1)/2$ times the length
of the period. For $q=3$ this means that the period contains as many
zeros as twos.
\end{definition}

\begin{corollary}
\label{rat}
Let $x\in(0,1)$ be a point that is not a $q$-adic rational.
\begin{enumerate}
\item[(a)] If the runs of zeros of $x$ have bounded
 length, then $T_q'(x)=+\infty$ if and only if $d_n(x)\to+\infty$. If
 the runs of the digit $q-1$ have bounded length, then $T_q'(x)=-\infty$
 if and only if $d_n(x)\to-\infty$.
\item[(b)] Let $x$ be rational, with an eventually periodic digit
 sequence of period length $p$ and digit sum $\nu$ over a period. Then
 $T_q'(x)=+\infty$ if $2\nu<(q-1)p$ and $T_q'(x)=-\infty$ if
 $2\nu>(q-1)p$. If $2\nu=(q-1)p$, that is, if $x$ is balanced, then all
 four Dini derivatives
 of $T_q$ at $x$ are finite, and $T_q$ has neither a finite nor an infinite
 derivative at $x$. For $q=3$ the three cases are $\nu<p$, $\nu>p$, and
 $\nu=p$.
\end{enumerate}
\end{corollary}

\begin{proof}
(a) Suppose that the runs of zeros have length at most $b$. Then
$n_{k+1}-n_k\le b+1$ and $\Gamma_k\ge d_{n_k}(x)-(q-1)(b+1)$. Thus
$d_n(x)\to+\infty$ implies $\Gamma_k\to+\infty$, and
Theorem~\ref{thm2}(c) gives $T_q'(x)=+\infty$. The
converse is part of Theorem~\ref{thm2}(a). The second statement follows
from the first, applied to $1-x$, as in the proof of Theorem~\ref{thm2}.

(b) Over each full period the digit slope changes by $(q-1)p-2\nu$, and
the digits before the period and those of an incomplete period change
$d_n(x)$ by a bounded amount, so
$d_n(x)=\bigl(q-1-2\nu/p\bigr)\,n+O(1)$. The period is not a string of
zeros or a string of digits $q-1$, because the digit sequences that end
in this way belong to $q$-adic rationals. Consequently the runs of zeros
and the runs of the digit $q-1$ have bounded
length, and (a) gives the two infinite cases. If $2\nu=(q-1)p$, then
$d_n(x)$ is bounded, and \eqref{runs1} shows that $\Delta(x,y)$ is bounded
for $y\in[0,1]$, $y\ne x$. Hence the four Dini derivatives are finite and $T_q$
has no infinite derivative at $x$, while Theorem~\ref{nodiff}
excludes a finite derivative.
\end{proof}

\begin{figure}[H]
\centering
\includegraphics{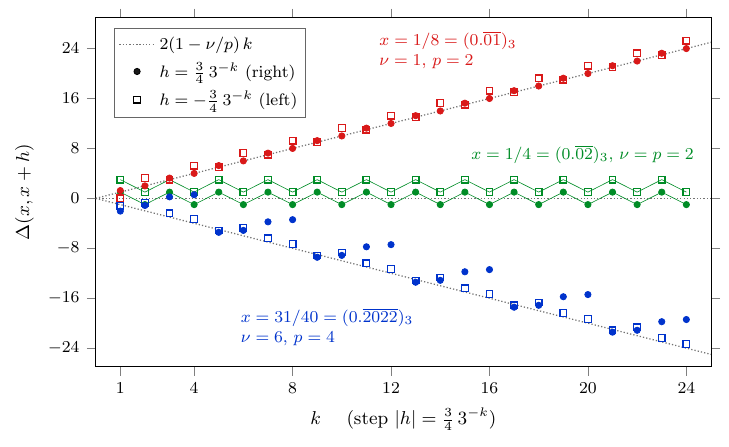}
\caption{The case $q=3$: the slopes $\Delta(x,x+h)$ of $T_q$ at $x=1/8$,
$1/4$, and $31/40$, for $h=\pm\frac34\,3^{-k}$.}
\label{fig6}
\end{figure}

For instance, let $q=3$. The periods of $1/8=(0.\overline{01})_3$ and
$31/40=(0.\overline{2022})_3$ have the digit sums $1<2$ and $6>4$, so
Corollary~\ref{rat}(b) gives $T_q'(1/8)=+\infty$ and $T_q'(31/40)=-\infty$.
By the proof of Corollary~\ref{rat}, the digit slopes grow like
$2(1-\nu/p)\,n$, that is, like $n$ at $1/8$ and like $-n$ at $31/40$.
Figure~\ref{fig6} shows the slopes $\Delta(x,x+h)$ at these two points and
at the balanced point $1/4$, for $h=\pm\frac34\,3^{-k}$ and $1\le k\le24$.
At $1/8$ and $31/40$ the slopes of $T_q$ follow the dotted lines $k$ and
$-k$. At $1/4$ they stay bounded without a limit: from the right they
alternate between $1$ and $-1$, and from the left between $3$ and $1$.
These are the four Dini derivatives of Proposition~\ref{quarter}. The factor
$\frac34$ is chosen for the point $1/4$: for $k=1$ and $k=2$ the points
$1/4+h$ are $1/2$ and $1/3$ on the right and $0$ and $1/6$ on the left,
and by Example~\ref{ex:quarter} the slopes $\Delta(1/4,1/4+h)$ at these
points are $1$, $-1$, $3$, and $1$ and do not change when $k$ is replaced
by $k+2$. At $1/8$ and
$31/40$ we use the same steps for comparison.

In a base $q\ge4$ the digits $1,\ldots,q-2$ also change the digit
slopes, except the middle digit $(q-1)/2$ when $q$ is odd. For instance, the period $1,2$ has $p=2$ and $\nu=3$. For $q=5$,
this period represents $7/24=(0.\overline{12})_5$, whose expansion
contains no digit $0$. Since $2\nu=6<8=(q-1)p$,
Corollary~\ref{rat}(b) gives $T_q'(7/24)=+\infty$. For $q=4$, the same
period represents $2/5=(0.\overline{12})_4$. Now $2\nu=6=(q-1)p$, so
$2/5$ is balanced and the four Dini derivatives of $T_q$ at $2/5$ are
finite.

In the next example we let $q=3$ and construct an irrational point whose
runs of zeros and runs of twos have bounded length, so that
Corollary~\ref{rat}(a) applies to it. Its ternary digits are obtained by
modifying the sequence of residues of $s_3(k)$ modulo $3$, and the point
is a ternary version of the example that
Kr\"uppel~\cite[Section~7.1]{Kruppel2007} built from the Thue--Morse
sequence.

\begin{example}
\label{ex:mod3}
Let $q=3$. For $k\ge0$ let $\epsilon_k\in\{0,1,2\}$ be the residue of
$s_3(k)$ modulo $3$,
and for $k\ge1$ let
\[
 \zeta_k=
 \begin{cases}
  2,&\text{if }k=3^\ell+2\text{ for some }\ell\ge1,\\
  \epsilon_k,&\text{otherwise},
 \end{cases}
 ~
 x=\sum_{k=1}^{\infty}\zeta_k3^{-k}.
\]
Then $x$ is irrational, $T_q'(x)=-\infty$, and $T_q'(1-x)=+\infty$.
\end{example}

\begin{proof}
Since $\epsilon_{3j+r}\equiv\epsilon_j+r\pmod3$ for $r=0,1,2$, each triple
$(\epsilon_{3j},\epsilon_{3j+1},\epsilon_{3j+2})$ is a permutation
of $0,1,2$, so the sum of $1-\epsilon_k$ over each triple is $0$. Since a
partial triple contributes between $-1$ and $1$, and the term $k=0$
contributes $1$ but is absent from $x$, we have
$|\sum_{k=1}^n(1-\epsilon_k)|\le2$ for all $n$. At a modified position
$k=3^\ell+2$ we have $s_3(k)=3$, so the original digit is $\epsilon_k=0$,
and each modification lowers the digit slope by $4$. Hence $d_n(x)$ is
$2\sum_{k=1}^n(1-\epsilon_k)$ minus $4$ times the number of modified
positions at most $n$. Since the sum is bounded and this number tends to
infinity with $n$, it follows that $d_n(x)\to-\infty$.

Within a triple no two consecutive digits are equal, so no three consecutive
digits of the unmodified sequence are equal. Around a modified position
$3^\ell+2$, the five digits
$\epsilon_{3^\ell},\ldots,\epsilon_{3^\ell+4}$ are $1,2,0,2,0$, and after
the modification they are $1,2,2,2,0$. Modified positions are at least
$6$ apart, so these strings of five digits do not overlap. In each
string the new run of twos $2,2,2$ lies between a $1$ and a $0$. Thus runs
of zeros still have length at most $2$, and runs of twos have length at
most $3$. In particular the digits are neither eventually $0$ nor
eventually $2$, so $x$ is not a ternary rational. Now Corollary~\ref{rat}(a)
gives $T_q'(x)=-\infty$, and $T_q'(1-x)=+\infty$ follows from $T_q(1-u)=T_q(u)$.
It remains to prove that $x$ is irrational.

Suppose, to the contrary, that $(\zeta_k)$ is eventually periodic with
period $p$. Then $2p$ is also a period, and $2p\ne3$. We claim that for
large $M$, neither $3^M-1$ nor $3^M-1+2p$ is a
modified position. Indeed, the first would require $3^M-3^\ell=3$, and the
second $3^M-3^\ell=3-2p$. Since $2p\ne3$, both right sides are nonzero, and
hence $M\ne\ell$ in every solution. Because
$|3^M-3^\ell|\ge2\cdot3^{\max\{M,\ell\}-1}$ for $M\ne\ell$, each equation
has only finitely many solutions $(M,\ell)$. Periodicity then gives
$\epsilon_{3^M-1}=\epsilon_{3^M-1+2p}$ for all large $M$. The
ternary digits of $3^M-1$ are $M$ twos, so $s_3(3^M-1)=2M$, while
$3^M-1+2p=3^M+(2p-1)$ has digit sum $1+s_3(2p-1)$ once $3^M>2p$. Thus
$2M\equiv1+s_3(2p-1)\pmod3$ for all large $M$, which is impossible,
because the left side takes every residue as $M$ varies. This
contradiction shows that the expansion is not eventually periodic, and
hence $x$ is irrational.
\end{proof}

For odd $q$ we determine the four Dini derivatives at the balanced point
$(q-1)/(2(q+1))$, whose digit sequence has the period $(q-3)/2,(q+1)/2$.
For $q=3$ this point is $1/4=(0.\overline{02})_3$, and for $q=5$ and $q=7$
it is $1/3$ and $3/8$ (see Figure~\ref{fig7}). By
Proposition~\ref{quarter} below, the upper and the lower Dini derivative
differ on each side, so $T_q$ has no derivative from either side there,
finite or infinite.
Kr\"uppel~\cite[Lemma~6.2]{Kruppel2007} carried out the corresponding
computation for Takagi's function at $1/3=(0.\overline{01})_2$.

Galkin and Galkina~\cite{GalkinGalkina2020} found the largest value of
$T_q(x)+x$ for odd and for even bases. We state their result for odd bases.

\begin{lemma}[Galkin and Galkina {\cite[Theorem~4]{GalkinGalkina2020}}]
\label{ggmax}
Let $q$ be odd. The maximum of $T_q(x)+x$ on $[0,1]$ is $(q+3)/4$. It is
attained exactly at the points that have a digit sequence with the
following two properties. Each of its digits is $(q-3)/2$, $(q-1)/2$, or
$(q+1)/2$, and its digits different from $(q-1)/2$ alternate
$(q+1)/2,(q-3)/2,(q+1)/2,(q-3)/2,\ldots$, beginning with $(q+1)/2$. There
may be infinitely many such digits, finitely many, or none; in the last
two cases the digit sequence ends with infinitely many digits $(q-1)/2$.
For $q=3$ the maximum is $3/2$, and the digits different from $1$
alternate $2,0,2,0,\ldots$, beginning with $2$.
\end{lemma}

The value of this maximum and the points at which it is attained are also
given by the case $c=1$ of \eqref{tilt2} in Lemma~\ref{lem:tilt} below. For
even $q$ the maximum is $(q+3)/4+1/(4(q+1))$, by \eqref{tilt2} with
$c=1$, where $a=0$ and $b=1$ if $q\ge4$, and $a=1$ and $b=0$ if $q=2$.
For $q=2$ this number is $4/3$. None of the points of
Lemma~\ref{ggmax} is a $q$-adic rational, so the digits of $1-x$ are the
complements $q-1-\xi_j$ of those of $x$, and since $T_q(1-x)=T_q(x)$,
Lemma~\ref{ggmax} says that, for odd $q$,
\begin{equation}
 T_q(u)\le u+\frac{q-1}4~(0\le u\le1),
 \label{quarter1}
\end{equation}
with equality exactly at the points $u$ whose digits are all $(q-3)/2$,
$(q-1)/2$, or $(q+1)/2$ and whose digits different from $(q-1)/2$
alternate $(q-3)/2,(q+1)/2,(q-3)/2,(q+1)/2,\ldots$, beginning with
$(q-3)/2$. For $q=3$ the inequality reads $T_q(u)\le u+1/2$, and equality
holds exactly at the points whose digits different from $1$ alternate
$0,2,0,2,\ldots$, beginning with $0$.

\begin{proposition}
\label{quarter}
Let $q$ be odd. At $x=(q-1)/(2(q+1))$,
\[
 D^+T_q(x)=1,~ D_+T_q(x)=-\frac{q-1}2,~
 D^-T_q(x)=\frac{q+3}2,~ D_-T_q(x)=1.
\]
For $q=3$ the point is $x=1/4$, and the four Dini derivatives are $1$,
$-1$, $3$, and $1$.
\end{proposition}

\begin{proof}
Let $x=(q-1)/(2(q+1))$. Since
\[
 \frac{q-3}2\,q^{-1}+\frac{q+1}2\,q^{-2}=\frac{(q-1)^2}{2q^2}
 =\bigl(1-q^{-2}\bigr)x,
\]
the digit sequence of $x$ has the period $(q-3)/2,(q+1)/2$. Thus
\eqref{quarter1} holds with equality at $u=x$, that is,
\begin{equation}
 T_q(x)=x+\frac{q-1}4=\frac{(q-1)(q+3)}{4(q+1)}
 =\frac{q+3}2\,x=\frac{q-1}2\,(1-x).
 \label{quartval}
\end{equation}

By \eqref{quarter1} and \eqref{quartval}, $T_q(u)-T_q(x)\le u-x$ for
$0\le u\le1$. Hence $\Delta(x,u)\le1$ for
$x<u\le1$ and $\Delta(u,x)\ge1$ for $0\le u<x$. We claim that
\begin{equation}
 T_q(u)\ge\frac{q+3}2\,u~(0\le u\le x),
 ~
 T_q(u)\ge\frac{q-1}2\,(1-u)~(x\le u\le1).
 \label{quartlow}
\end{equation}
By \eqref{quartval}, the first inequality gives
$T_q(x)-T_q(u)\le\frac{q+3}2(x-u)$, so that $\Delta(u,x)\le(q+3)/2$ for
$0\le u<x$, and the second gives $T_q(u)-T_q(x)\ge-\frac{q-1}2(u-x)$, so
that $\Delta(x,u)\ge-(q-1)/2$ for $x<u\le1$.

To prove \eqref{quartlow}, put $H(v)=T_q(v)-\frac{q-1}2v$ for
$0\le v\le1-x$ and $P(u)=T_q(u)-\frac{q+3}2u$ for $0\le u\le x$. The first
inequality in \eqref{quartlow} says that $P\ge0$, and since
$T_q(1-u)=T_q(u)$, the second says that $H\ge0$.
Lemma~\ref{firstcross} with $k=1$ and $d=b=(q-1)/2$ shows that
$H(v)\ge0$ for $0\le v\le(q+1)/(2q)$. If $q\ge5$, then
$(q+3)/2\le q-1$, and Lemma~\ref{firstcross} with $k=1$ and
$d=b=(q+3)/2$ shows that $P(u)\ge0$ for $0\le u\le(q-3)/(2q)$. If $q=3$,
then $(q-3)/(2q)=0$ and $P(0)=0$. Let $(q+1)/(2q)\le v\le1-x$ and
$y=qv-(q+1)/2$. Then $0\le y\le q(1-x)-(q+1)/2=x$, and
\eqref{branches} with $a=(q+1)/2$ gives $qT_q(v)=(q^2-1)/4-2y+T_q(y)$. Since
$q\cdot\frac{q-1}2v=(q^2-1)/4+\frac{q-1}2y$, we obtain $qH(v)=P(y)$. Let
now $(q-3)/(2q)\le u\le x$ and $y=qu-(q-3)/2$. Then
$0\le y\le qx-(q-3)/2=1-x$, and \eqref{branches} with $a=(q-3)/2$ gives
$qT_q(u)=(q^2-9)/4+2y+T_q(y)$. Since
$q\cdot\frac{q+3}2u=(q^2-9)/4+\frac{q+3}2y$, we obtain $qP(u)=H(y)$.

The functions $H$ and $P$ are continuous. Let $\mu$ be the smallest of
all the values of $H$ on $[0,1-x]$ and of $P$ on $[0,x]$, and suppose
that $\mu<0$. If $\mu=H(v)$ for some $v$, then $v>(q+1)/(2q)$, and the
point $y=qv-(q+1)/2$ of $[0,x]$ has $P(y)=q\mu<\mu$. If $\mu=P(u)$ for
some $u$, then $u>(q-3)/(2q)$, and the point $y=qu-(q-3)/2$ of
$[0,1-x]$ has $H(y)=q\mu<\mu$. Both cases contradict the choice of
$\mu$. Hence $\mu\ge0$, which proves \eqref{quartlow}.

\begin{figure}[H]
\centering
\includegraphics{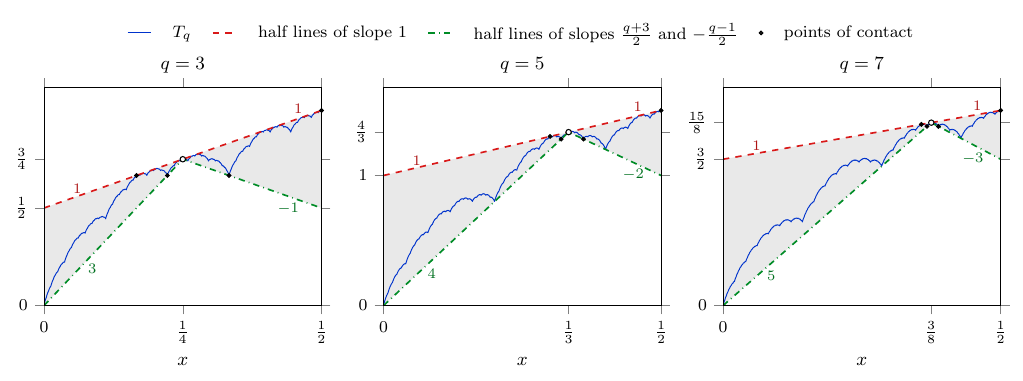}
\caption{Proposition~\ref{quarter} for $q=3$, $5$, and $7$: the graph of
$T_q$ over $[0,1/2]$ and the four half lines whose slopes are the Dini
derivatives of $T_q$ at $(q-1)/(2(q+1))$.}
\label{fig7}
\end{figure}

It remains to show that each of the four bounds is attained at points
arbitrarily close to $x$. Let $K=\frac{q-3}2\,q+\frac{q+1}2$ be the integer
whose two base $q$ digits are the first two digits of $x$. Then
$s_q(K)=q-1$, so the level $2$ interval $[K/q^2,(K+1)/q^2]$ has slope
zero, and $x=(K+x)/q^2$. Hence \eqref{cell} gives
$T_q((K+u)/q^2)=T_q(K/q^2)+q^{-2}T_q(u)$ for $0\le u\le1$, and therefore
\[
 \frac{T_q(x+h/q^2)-T_q(x)}{h/q^2}=\frac{T_q(x+h)-T_q(x)}h
 ~\text{whenever $h\ne0$ and }x+h\in[0,1].
\]
By \eqref{quartval} and $T_q(0)=T_q(1)=0$, we have
$\Delta(x,1)=-T_q(x)/(1-x)=-(q-1)/2$ and $\Delta(0,x)=T_q(x)/x=(q+3)/2$. All
digits of $1/2$ are $(q-1)/2$, and the digits of
$(q-2)/(2q)=\frac{q-3}{2q}+\frac1{2q}$ are $(q-3)/2$ followed by digits
$(q-1)/2$ only. Both points therefore satisfy \eqref{quarter1} with
equality, and $(q-2)/(2q)<x<1/2$. Hence $\Delta(x,1/2)=1$ and
$\Delta((q-2)/(2q),x)=1$. Dividing $h$
repeatedly by $q^2$ preserves each of these four slopes and sends $h$ to
$0$, so the four bounds are the four Dini derivatives.
\end{proof}

Let $q$ be odd and $x=(q-1)/(2(q+1))$. By the proof of
Proposition~\ref{quarter}, the graph of $T_q$ over $[0,1]$ stays in the
regions between the four half lines from $(x,T_q(x))$ with the slopes $1$
and $-(q-1)/2$ to the right and $(q+3)/2$ and $1$ to the left. With the
integer $K$ of that proof, the slope of $T_q$ between $x$ and
$(K+u)/q^2$ equals the slope between $x$ and $u$ for $0\le u\le1$,
$u\ne x$, as shown there. For $u=0$ and $u=1$ this shows that the graph touches the half
lines of slopes $(q+3)/2$ and $-(q-1)/2$ at $K/q^2$ and $(K+1)/q^2$, and
the proof shows that it touches the two half lines of slope $1$ at
$(q-2)/(2q)$ and $1/2$. Figure~\ref{fig7} shows the graph over $[0,1/2]$ and the half lines
for $q=3$, $5$, and $7$, with the regions shaded, the point $(x,T_q(x))$
drawn as a circle, and these four points marked. For $q=3$ the points are
$1/6$, $2/9$, $1/3$, and $1/2$. For $q=5$ they are $3/10$, $8/25$,
$9/25$, and $1/2$, and for $q=7$ they are $5/14$, $18/49$, $19/49$, and
$1/2$.

Let $q=3$. Then $x=1/4$ and $T_q(x)=3/4$, and besides $1/6$, $2/9$,
$1/3$, and $1/2$, the graph of $T_q$ also touches the left half line of
slope $1$ at $1/5$. Indeed,
$T_q(1/3)=T_q(2/9)=2/3$ by \eqref{T}, since $3^j/3$ has distance $1/3$ and
then $0$ from $\Z$, and $3^j\cdot2/9$ has distance $2/9$, $1/3$, and then
$0$. The value $T_q(1/5)=7/10$ follows by summing
the geometric series over the period $1/5,3/5,4/5,2/5$ of
$x\mapsto3x\bmod1$:
\[
 T_q(1/5)=\frac{2/5+2/9+2/45+2/81}{1-3^{-4}}=\frac7{10}.
\]
With $T_q(1/4)=3/4$ and $T_q(1/2)=1$, these values give $\Delta(1/5,1/4)=1$,
$\Delta(2/9,1/4)=3$, $\Delta(1/4,1/3)=-1$, and $\Delta(1/4,1/2)=1$. By
\eqref{cell}, $T_q((2+t)/9)=2/3+T_q(t)/9$ for $0\le t\le1$, so the scaling by
$1/9$ about $(1/4,3/4)$ maps the graph over $[0,1/2]$ and the four half
lines onto their parts over $[2/9,5/18]$. Among the points of contact
there are $11/45$, $20/81$, $7/27$, and $5/18$, the images of $1/5$,
$2/9$, $1/3$, and $1/2$.

For $q=2$ the formulas of Proposition~\ref{quarter} are false. In this
case $(q-1)/(2(q+1))=1/6$. By \eqref{cell} with $n=1$ and
$m=0$, and with $n=3$ and $m=1$, and since $T_q(1/8)=3/8$ by
Theorem~\ref{sampthm}, we have $T_q(1/8+v/4)=3/8+T_q(v)/4$ for
$0\le v\le1/2$. The map $v\mapsto1/8+v/4$ leaves $1/6$ unchanged and does
not change the slope of $T_q$ between two points of $[0,1/2]$. Its
iterates send $0$ to points $v<1/6$ arbitrarily close to $1/6$, and
$\Delta(v,1/6)=\Delta(0,1/6)=3$ at these points, because $T_q(1/6)=1/2$ by
Proposition~\ref{ratval}. Hence $D^-T_q(1/6)\ge3>5/2=(q+3)/2$. The proof
of Proposition~\ref{quarter} uses odd $q$ in \eqref{quarter1}, which comes
from Lemma~\ref{ggmax}, and in the digits $(q-3)/2$ and $(q+1)/2$ of $x$.

\begin{proposition}
\label{midpt}
Let $q$ be odd, let $N\ge0$ and $0\le m<q^N$, and let $x=(m+1/2)/q^N$.
Then
\[
 D^+T_q(x)=d_N(x)-1,~ D_+T_q(x)=d_N(x)-\frac{q+1}2,~
 D^-T_q(x)=d_N(x)+\frac{q+1}2,~ D_-T_q(x)=d_N(x)+1,
\]
where $d_N(x)=(q-1)N-2s_q(m)$ is the slope of $[m/q^N,(m+1)/q^N]$. In
particular, $T_q$ has no derivative from either side at $x$, finite or
infinite. For $q=3$ the four Dini derivatives are $d_N(x)-1$,
$d_N(x)-2$, $d_N(x)+2$, and $d_N(x)+1$, with $d_N(x)=2(N-s_3(m))$.
\end{proposition}

\begin{proof}
For $0<|h|\le1/2$ put $Q(h)=(T_q(1/2+h)-T_q(1/2))/h$. Let $0<h\le1/2$ and
$v=1/2-h$, so that $0\le v<1/2$ and $T_q(1/2+h)=T_q(v)$. By
\eqref{quarter1}, $T_q(v)\le v+(q-1)/4=(q+1)/4-h$, and by
Proposition~\ref{T2x},
$T_q(v)\ge\frac{q+1}2v=\frac{q+1}4-\frac{q+1}2h$. Since $T_q(1/2)=(q+1)/4$,
we obtain $-(q+1)/2\le Q(h)\le-1$. In the
proof of Theorem~\ref{nodiff} we saw that $Q(h)=Q(qh)$ for
$0<|h|\le1/(2q)$, and that $Q(q^{-k})=-1$ and $Q(q^{-k}/2)=-(q+1)/2$ for
every $k\ge1$. Hence the upper and the lower limit of $Q(h)$ as
$h\downarrow0$ are $-1$ and $-(q+1)/2$. Since $T_q(1/2-h)=T_q(1/2+h)$, we
have $Q(-h)=-Q(h)$, so
the upper and the lower limit as $h\uparrow0$ are $(q+1)/2$ and $1$. This
is the claim for $x=1/2$.

Now let $x=(m+1/2)/q^N$ and $0<|h|\le q^{-N}/2$. Then $x+h$ lies in
$[m/q^N,(m+1)/q^N]$, and \eqref{cell} with $t=1/2+q^Nh$ and with $t=1/2$
gives
\[
 \frac{T_q(x+h)-T_q(x)}h=(q-1)N-2s_q(m)+Q(q^Nh).
\]
Thus the four Dini derivatives of $T_q$ at $x$ are those at $1/2$ increased
by $(q-1)N-2s_q(m)=d_N(x)$. Since $(q+1)/2\ne1$, the upper and the lower
Dini derivative differ on each side, so $T_q$ has no derivative from
either side at $x$, finite or infinite.
\end{proof}

For even $q$ the points $(m+1/2)/q^N$ are $q$-adic rationals, because
$q^{N+1}(m+1/2)/q^N=qm+q/2$ is an integer. By Proposition~\ref{cusp} they
are strict local minima of $T_q$, with the right derivative $+\infty$ and
the left derivative $-\infty$.

\begin{example}
\label{ex:midpt}
Let $q=3$. The point $1/6=(0+1/2)/3$ is the midpoint of $[0,1/3]$, so
Proposition~\ref{midpt} applies with $N=1$ and $m=0$, where
$d_1(1/6)=2$. It gives
\[
 D^+T_q(1/6)=1,~D_+T_q(1/6)=0,~D^-T_q(1/6)=4,~D_-T_q(1/6)=3.
\]
Since $D_-T_q(1/6)>0$, we have $T_q(y)<T_q(1/6)$ for every $y<1/6$ close enough
to $1/6$. Since $D^+T_q(1/6)>0$, there are points $y>1/6$ arbitrarily close
to $1/6$ with $T_q(y)>T_q(1/6)$, so $1/6$ is not a local maximum of $T_q$, as
Example~\ref{ex:localmax} below also shows. The left panel of
Figure~\ref{midptpic} shows the graph of $T_q$ on $[0,1/3]$ and the four
half lines from $(1/6,2/3)$ with these slopes. On each side the graph lies between the two half lines and
touches each of them at points that accumulate at $1/6$, for instance the
half line of slope $1$ at $5/18$ and the half line of slope $0$ at $2/9$
and $1/3$. For $|h|\le1/6$ we have $\tau_q(1/2+h)=2/3$, so \eqref{fe} gives
$T_q(1/2+3h)-1=3(T_q(1/2+h)-1)$, and \eqref{branches} with $a=0$
then gives $T_q(1/6+h)-2/3=3\bigl(T_q(1/6+h/3)-2/3\bigr)$. Hence the
enlargement by the factor $3$ about $(1/6,2/3)$ maps the part of this
picture over $[1/9,2/9]$ onto the whole picture.
\end{example}

The other two panels of Figure~\ref{midptpic} show the same objects for
$q=5$ and $q=7$. Here Proposition~\ref{midpt} with $N=1$ and $m=0$ gives
the Dini derivatives $q-2$, $(q-3)/2$, $(3q-1)/2$, and $q$ at the
midpoint $1/(2q)$ of $[0,1/q]$, and \eqref{cell} with $n=1$, $m=0$, and
$t=1/2$ gives $T_q(1/(2q))=(3q-1)/(4q)$. In each panel the dots mark the
points $1/(2q)\pm1/q^2$ and $1/(2q)\pm1/(2q^2)$, at which the graph
touches the half lines by the proof of Proposition~\ref{midpt}. They also
mark the points $0$ and $1/q$, at which the graph touches the half lines
of slopes $(3q-1)/2$ and $(q-3)/2$, because $T_q(0)=0$ and
$T_q(1/q)=(q-1)/q$.

\begin{figure}[H]
\centering
\includegraphics{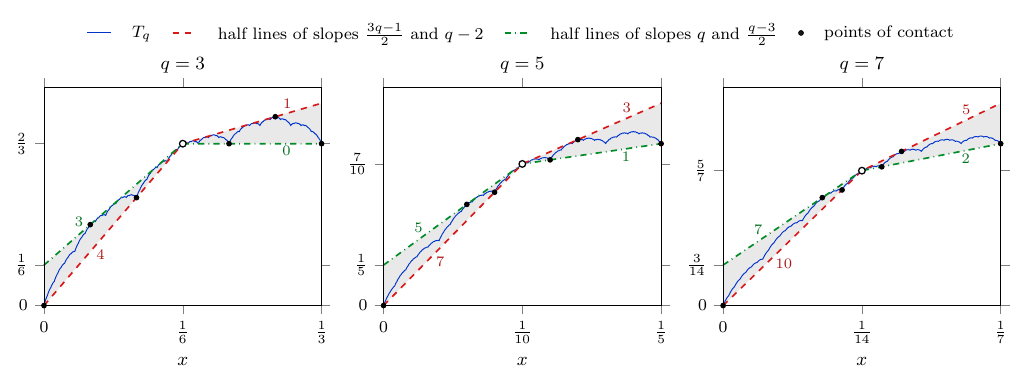}
\caption{Proposition~\ref{midpt} for $q=3$, $5$, and $7$: the half lines
from the midpoint of $[0,1/q]$ whose slopes are the Dini derivatives of
$T_q$ there (see Example~\ref{ex:midpt} for $q=3$).}
\label{midptpic}
\end{figure}

For odd $q$, the upper and the lower right Dini derivative differ by
$(q-1)/2$ at the midpoints of the intervals of each level and by
$(q+1)/2$ at the point of Proposition~\ref{quarter}, that is, by $1$ and
$2$ for $q=3$. For Takagi's function,
Kr\"uppel~\cite[Proposition~6.3]{Kruppel2007} stated that the two right
Dini derivatives differ by at least $2$ wherever both are finite, with
equality at $1/3$. By his Lemma~6.2 these derivatives are $0$ and $-1$ at
$1/3$, so they differ by $1$ there, as they do for $T_q$ at the midpoints.
Hence his statement does not hold. In his proof the value $2$ of this
difference at $1/3$ is attributed to Lemma~6.2.
Kr\"uppel~\cite[Proposition~5.1]{Kruppel2007} used the reflections of a
point in the midpoints of the dyadic intervals that contain it, and the
proof of Theorem~\ref{dinigap} uses the reflections in the midpoints of
the ternary intervals of each level that contain the point.

Let $q=3$. For $k\ge0$ let $x_k$ be the point whose ternary expansion is
periodic with the period consisting of the digit $0$, $k$ digits $1$,
and the digit $2$; summing the geometric series gives
$x_k=(3^{k+1}+1)/(2(3^{k+2}-1))$, so that $x_0=1/4$ and $x_1=5/26$.

\begin{theorem}
\label{dinigap}
Let $q=3$, let $x\in\R$, and suppose that $D^+T_q(x)$ and $D_+T_q(x)$ are
finite. Then
$D^+T_q(x)-D_+T_q(x)\ge1$, with equality exactly when $x-\lfloor x\rfloor$ is
the midpoint of a ternary interval. If $x-\lfloor x\rfloor$ is not such a
midpoint, then
$D^+T_q(x)-D_+T_q(x)\ge9/5$, and $D^+T_q(x)-D_+T_q(x)\ge2$ unless the ternary
digits of $x-\lfloor x\rfloor$ that are different from $1$ are infinitely
many and eventually alternate $0,2,0,2,\ldots$. The constant $9/5$ cannot
be replaced by a larger number, because
\[
 D^+T_q(x_k)-D_+T_q(x_k)=\frac3{2(1-x_k)}~(k\ge0),
\]
which tends to $9/5$ as $k\to\infty$. The same statements hold for
$D^-T_q(x)-D_-T_q(x)$, and $D^-T_q(1-x_k)-D_-T_q(1-x_k)=3/(2(1-x_k))$ for
$k\ge0$.
\end{theorem}

\begin{proof}
By periodicity we may assume $0\le x<1$. Since $D_+T_q=+\infty$ at the
ternary rationals by Proposition~\ref{cusp}, $x$ is not a ternary
rational, and $0<x<1$. We use its digits $\xi_1,\xi_2,\ldots$ and the
notation $k_n$, $t_n$, $a_n$, $c_n$ fixed before Lemma~\ref{lem:subtract}.

Suppose first that $x$ has finitely many digits $2$. If it has infinitely
many digits $0$, then $d_n(x)\to+\infty$, so $(T_q)'_+(x)=+\infty$ by
Theorem~\ref{thm2}(a), contrary to the assumption. Otherwise the digits
of $x$ are eventually $1$, $x$ is the midpoint of a ternary interval, and
$D^+T_q(x)-D_+T_q(x)=1$ by Proposition~\ref{midpt}.

Suppose next that $x$ has finitely many digits $0$ and infinitely many
digits $2$. Then $d_n(x)\to-\infty$, and $\Delta(x,c_n)\le d_n(x)$ by the
second equation in \eqref{ends}, because $T_q\ge0$. Since $c_n>x$ and $c_n\to x$, it follows
that $D_+T_q(x)=-\infty$, contrary to the assumption.

It remains to treat the points with infinitely many digits $0$ and
infinitely many digits $2$. Let $R$ be the set of $n\ge0$ such that the
first digit different from $1$ after the place $n$ is a $0$. For $n\in R$
we have $t_n<(0.\overline1)_3=1/2$, so the reflection of $x$ in the
midpoint of $[a_n,c_n]$ lies in $(x,c_n]$, and by
Lemma~\ref{lem:subtract} the slope of $T_q$ between $x$ and this
reflection is $d_n(x)$. Since $R$ is infinite and these reflections
tend to $x$, the numbers $d_n(x)$ with large $n\in R$ lie in
$[D_+T_q(x)-1,D^+T_q(x)+1]$, and every even integer that equals $d_n(x)$ for
infinitely many $n\in R$ lies in $[D_+T_q(x),D^+T_q(x)]$. If two different
even integers are each equal to $d_n(x)$ for infinitely many $n\in R$, then
$D^+T_q(x)-D_+T_q(x)\ge2$.

Suppose finally that this is not the case. Since the numbers $d_n(x)$ with
large $n\in R$ take only finitely many values, there is an even integer
$\beta$ with $d_n(x)=\beta$ for all large $n\in R$. Let $j_1<j_2<\cdots$ be the places of the digits different
from $1$, and put $\beta_i=d_{j_i}(x)$ and $\beta_0=0$. Then $d_n(x)=\beta_{i-1}$ for
$j_{i-1}\le n<j_i$ (with $j_0=0$), and these $n$ belong to $R$ exactly
when $\xi_{j_i}=0$. Thus $\beta_{i-1}=\beta$ for every large $i$ with
$\xi_{j_i}=0$. Let $i$ be such an index. Then $\beta_i=\beta+2$, so
$\xi_{j_{i+1}}=2$ and $\beta_{i+1}=\beta$. If $\xi_{j_{i+2}}$ were $2$, then
$\beta_{i+2}=\beta-2$. Digits $1$ do not change the digit slope and digits $2$
lower it, so the digit slope would stay below $\beta$ from the place $j_{i+2}$
up to the next digit $0$, which exists. This contradicts the fact, shown
above, that $d_n(x)=\beta$ just before every digit $0$ at a large place. Hence $\xi_{j_{i+2}}=0$, and by induction the digits
different from $1$ alternate $0,2,0,2,\ldots$ from the place $j_i$ on.
Put $N=j_{i+1}$ and $x'=t_N$, the point with the digits
$\xi_{N+1},\xi_{N+2},\ldots$. Its digits different from $1$ are
infinitely many and alternate $0,2,0,2,\ldots$, beginning with $0$. By
\eqref{cell}, $\Delta(x,x+h)=d_N(x)+\Delta(x',x'+3^Nh)$ for all small
$h$, so the Dini derivatives of $T_q$ at $x$ are those at $x'$ increased by
$d_N(x)$. Let $n$ be a place at which $x'$ has a digit $2$.
Each digit $0$ of $x'$ is followed by a digit $2$ before the next $0$, so
$d_n(x')=0$, and the point $t_n(x')$ with the digits
of $x'$ after the place $n$ again has digits different from $1$
alternating $0,2,0,2,\ldots$, beginning with $0$. Hence
$T_q(t_n(x'))=t_n(x')+1/2$ by \eqref{quarter1}, and $t_n(x')\ge1/6$, since
the smallest such point is $1/6=(0.0\overline1)_3$. Moreover
$t_n(x')<1/2=(0.\overline1)_3$, because the digits of $t_n(x')$ different
from $1$ are infinitely many and the first of them is a $0$. We apply
\eqref{sub} to $x'$ with $u=1/2$. The midpoint of the level $n$ interval
of $x'$ lies to the right of $x'$, and since $d_n(x')=0$,
the slope of $T_q$ between $x'$ and this midpoint is
\[
 \frac{T_q(1/2)-T_q(t_n(x'))}{1/2-t_n(x')}=\frac{1/2-t_n(x')}{1/2-t_n(x')}=1.
\]
By the second equation in \eqref{ends}, with $T_q(1-t)=T_q(t)$, the slope of
$T_q$ between $x'$ and the right endpoint of its level $n$ interval is
\[
 -\frac{T_q(t_n(x'))}{1-t_n(x')}=-\frac{t_n(x')+1/2}{1-t_n(x')}\le
 -\frac{1/6+1/2}{1-1/6}=-\frac45,
\]
because $t\mapsto(t+1/2)/(1-t)$ is increasing. There are infinitely many
such places $n$, so $D^+T_q(x')\ge1$ and $D_+T_q(x')\le-4/5$, and therefore
$D^+T_q(x)-D_+T_q(x)\ge9/5$.

This proves the inequalities and the characterization of equality. To
compute $D^+T_q(x_k)-D_+T_q(x_k)$ for $k\ge1$, let $p=k+2$, and let $K$ be
the integer whose $p$ ternary digits are $0$, $k$ digits $1$, and $2$.
Then $s_3(K)=p$, so the level $p$ interval $[K/3^p,(K+1)/3^p]$ has
slope $0$ and contains $x_k=(K+x_k)/3^p$. By \eqref{cell},
$T_q((K+u)/3^p)=T_q(K/3^p)+3^{-p}T_q(u)$ for $0\le u\le1$, and hence
$\Delta(x_k,x_k+3^{-p}h)=\Delta(x_k,x_k+h)$ whenever $h\ne0$ and
$x_k+h\in[0,1]$. Put $r=(K+1)/3^p=x_k+3^{-p}(1-x_k)$. Every $u$ with
$x_k<u\le1$ can be written as $u=x_k+3^{-ip}(v-x_k)$ with an integer
$i\ge0$ and $r<v\le1$, and then $\Delta(x_k,u)=\Delta(x_k,v)$. Hence the
slopes $\Delta(x_k,u)$ with $x_k<u\le1$ are the slopes with $r<u\le1$,
and each of them is attained at points arbitrarily close to $x_k$. Let
$\gamma=(x_k+1/2)/(1-x_k)$. Then $T_q(x_k)=x_k+1/2=\gamma(1-x_k)$ by
\eqref{quarter1}, and $\gamma\le1$ because $x_k\le1/4$. For $u>x_k$, the
inequality $\Delta(x_k,u)\ge-\gamma$ is equivalent to
$T_q(u)\ge\gamma(1-u)$. The digits of $r$
are $0$, $k-1$ digits $1$, and $2$, so $3r=1/2+3^{-k}/2$, and
\eqref{branches} with $a=0$, the symmetry, and Proposition~\ref{T2x} give
\[
 T_q(r)=2r+\tfrac13T_q\bigl(\tfrac12-\tfrac{3^{-k}}2\bigr)
 =\tfrac13+3^{-k-1}+\tfrac13\bigl(1-3^{-k}\bigr)=\tfrac23.
\]
Since $\Delta(x_k,r)=\Delta(x_k,1)=-\gamma$, we have
$\gamma(1-r)=T_q(r)=2/3$. For $1/6\le u\le1/3$, equation
\eqref{branches} with $a=0$ and Proposition~\ref{T2x} give
$T_q(u)=2u+T_q(1-3u)/3\ge2u+2(1-3u)/3=2/3$. Since $1/6<r<1/4$, it follows
that $T_q(u)\ge2/3=\gamma(1-r)\ge\gamma(1-u)$ for $r\le u\le1/3$. For
$1/4\le u\le1$ we saw in the proof of Proposition~\ref{quarter} that
$T_q(u)\ge1-u\ge\gamma(1-u)$. Thus $\Delta(x_k,u)\ge-\gamma$ for
$r<u\le1$, with equality at $u=1$, and therefore $D_+T_q(x_k)=-\gamma$.
By \eqref{quarter1} and $T_q(x_k)=x_k+1/2$, we also have
$\Delta(x_k,u)\le1$ for $u>x_k$, with equality at $u=1/2$, since
$T_q(1/2)=1/2+1/2$. Hence $D^+T_q(x_k)=1$, and
$D^+T_q(x_k)-D_+T_q(x_k)=1+\gamma=3/(2(1-x_k))$ for $k\ge1$. For $k=0$,
Proposition~\ref{quarter} gives $D^+T_q(x_0)-D_+T_q(x_0)=2=3/(2(1-x_0))$.
As $k\to\infty$, $x_k$ tends to $1/6$ and $3/(2(1-x_k))$ to $9/5$.

The statements on the left Dini derivatives follow from those on the
right Dini derivatives. Indeed, $T_q(1-x-h)=T_q(x+h)$ gives
$D^-T_q(x)=-D_+T_q(1-x)$ and $D_-T_q(x)=-D^+T_q(1-x)$, so that
$D^-T_q(x)-D_-T_q(x)=D^+T_q(1-x)-D_+T_q(1-x)$. Moreover, the point $1-x$ is a
midpoint exactly when $x$ is, and the digits of $1-x$ are the complements
of those of $x$.
\end{proof}

The proof of Theorem~\ref{dinigap} uses $q=3$ in its last case, where the
digits different from the middle digit $1$ are $0$ and $2$, which change
the digit slope by $2$ and by $-2$. For $q\ne3$ we do not know the
smallest value of $D^+T_q(x)-D_+T_q(x)$ over the points $x$ at which both
right Dini derivatives are finite (see Section~\ref{sec:questions}).

For $q=2$ the points of the next proposition were studied by
Kr\"uppel~\cite[Section~7.2]{Kruppel2007} and by Allaart and
Kawamura~\cite[Section~2]{AllaartKawamura2010}. These points also give
simple examples in which $d_n(x)\to+\infty$ while the
condition in Theorem~\ref{thm2}(b) that involves the runs of zeros fails,
so that $(T_q)'_+(x)=+\infty$ but $(T_q)'_-(x)\ne+\infty$ (see
Figure~\ref{sparsepic} for $q=3$, $4$, and $5$).

\begin{proposition}
\label{sparse}
Let $p_1<p_2<\cdots$ be positive integers, not eventually consecutive,
and let $x=\sum_{j=1}^{\infty}(q-1)q^{-p_j}$. Then $(T_q)'_+(x)=+\infty$ if
and only if $p_k-2k\to+\infty$, and $T_q'(x)=+\infty$ if and only if
\[
 2p_k-p_{k+1}-2k+\log_q(p_{k+1}-p_k)\to+\infty.
\]
If $p_{k+1}\ge4p_k$ for all $k$, then $(T_q)'_+(x)=+\infty$, but
$D_-T_q(x)=-\infty$.
\end{proposition}

\begin{proof}
The digits of $x$ are $q-1$ at the positions $p_j$ and $0$ elsewhere,
and $x$ is not a $q$-adic rational, because the gaps $p_{j+1}-p_j$ are
not eventually equal to $1$. Its nonzero digits are at
the positions $n_k=p_k$, and
$d_{p_k}(x)=(q-1)(p_k-k)-(q-1)k=(q-1)(p_k-2k)$, because $p_k-k$ of the
first $p_k$ digits are $0$ and $k$ of them are $q-1$. Between consecutive
positions $p_k$ the digit slope $d_n(x)$ increases with $n$, so
$d_n(x)\to+\infty$ if and only if $p_k-2k\to+\infty$. Here
$\Gamma_k=(q-1)(p_k-2k)-(q-1)(p_{k+1}-p_k)+(q-1)\log_q(p_{k+1}-p_k)$ is
$q-1$ times the quantity displayed in the statement, and
Theorem~\ref{thm2}(a) and (c) give the two characterizations.

Now let $p_{k+1}\ge4p_k$ for all $k$. Then $p_k\ge4^{k-1}$, so
$p_k-2k\to+\infty$ and $(T_q)'_+(x)=+\infty$. Fix $j$ and put $N=p_j$,
$t=t_N$, and $h=x-a_N=q^{-N}t$. The first nonzero digit of $t$ stands at
the position $p_{j+1}-N$, and not all later digits of $t$ are $q-1$, so
$t<q^{-(p_{j+1}-N-1)}\le q^{-2}$, where we used
$p_{j+1}-N\ge3p_j\ge3$. Hence $qh<q^{-N}$, the point
$y_j=a_N-qh$ lies in $[a_N-q^{-N},a_N)$, and $\xi_N=q-1\ne0$, so
\eqref{ends} and \eqref{leftnb} apply. With $T_q(qt)=qT_q(t)-q(q-1)t$, which
follows from \eqref{fe} because $\tau_q(t)=(q-1)t$ for $0\le t\le1/q$,
they give
\[
 \frac{T_q(x)-T_q(a_N)}{h}=d_N(x)+\frac{T_q(t)}t,
 ~
 \frac{T_q(a_N)-T_q(y_j)}{qh}=d_N(x)+2-\frac{T_q(qt)}{qt}
 =d_N(x)+q+1-\frac{T_q(t)}t,
\]
and $\Delta(y_j,x)$, which is the weighted mean of these two slopes with
weights $1$ and $q$, is
\begin{align*}
 d_N(x)+q-\frac{q-1}{q+1}\cdot\frac{T_q(t)}{t}
 &<(q-1)(p_j-2j)+q-\frac{(q-1)^2}{q+1}\bigl(p_{j+1}-p_j-2\bigr)\\
 &\le-\frac{2(q-1)(q-2)}{q+1}\,p_j-2(q-1)j+q+\frac{2(q-1)^2}{q+1}.
\end{align*}
The first inequality holds by \eqref{Thbd}, because
$T_q(t)/t\ge(q-1)\lfloor\log_q(1/t)\rfloor>(q-1)(\log_q(1/t)-1)
>(q-1)(p_{j+1}-p_j-2)$, and the
second one holds because $p_{j+1}-p_j\ge3p_j$. The right side tends to
$-\infty$ as $j\to\infty$, and $x-y_j=(q+1)h\to0$, so $D_-T_q(x)=-\infty$.
For $q=3$ we have $y_j=a_N-3h$, and the right side is $-p_j-4j+5$.
\end{proof}

\begin{example}
\label{ex:sparse}
Let $p_k=k^2$ and $x=\sum_{k=1}^{\infty}(q-1)q^{-k^2}$. Its nonzero
digits are $q-1$ and occur at the same positions as the digits $1$ in
Example~\ref{ex:squares}. The
numbers $p_k$ are not eventually consecutive, and the quantity displayed
in Proposition~\ref{sparse} is $k^2-4k-1+\log_q(2k+1)$, which tends to
$+\infty$. By Proposition~\ref{sparse}, $T_q'(x)=+\infty$.
\end{example}

\begin{figure}[H]
\centering
\includegraphics{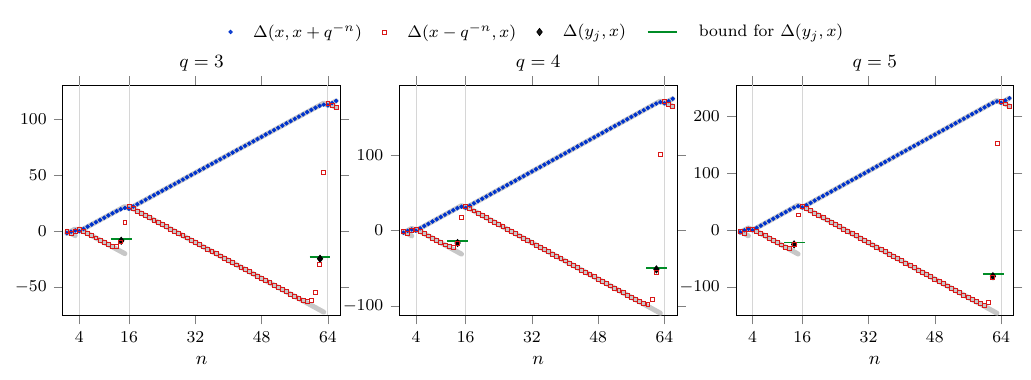}
\caption{Proposition~\ref{sparse} for $q=3$, $4$, and $5$: the slopes
$\Delta(x,x+q^{-n})$ and $\Delta(x-q^{-n},x)$ of $T_q$ at
$x=\sum_{j=1}^{\infty}(q-1)q^{-4^{j-1}}$ for $1\le n\le66$.}
\label{sparsepic}
\end{figure}

To see how the slopes on the two sides behave, let $p_j=4^{j-1}$, so that
$x=(q-1)(q^{-1}+q^{-4}+q^{-16}+q^{-64}+\cdots)$. By
Proposition~\ref{sparse}, $(T_q)'_+(x)=+\infty$ and $D_-T_q(x)=-\infty$. Let
$p_j\le n<p_{j+1}$. Then $a_n=a_{p_j}$, and the slope of $T_q$ between $a_n$
and $c_n$ is the digit slope $d_n(x)=(q-1)(n-2j)$. By \eqref{leftnb} at
level $p_j$, Proposition~\ref{incr}(b), and $T_q(1)=0$, the slope between
$a_n-q^{-n}$ and $a_n$ is $(q-1)(2p_j-2j-n)+2$. Moving the three points
$a_n-q^{-n}$, $a_n$, and $c_n$ by $h=x-a_n$ gives $x-q^{-n}$, $x$, and
$x+q^{-n}$. By Proposition~\ref{incr}(a), this shift changes each of the
two slopes by at most $2q^nT_q(h)$. As
$(q-1)q^{-p_{j+1}}<h<q^{1-p_{j+1}}$, the bound \eqref{Thbd} gives
\[
 2q^nT_q(h)\le2\Bigl((q-1)(p_{j+1}-1)+\frac{q(q+1)}4\Bigr)q^{n+1-p_{j+1}},
\]
which is small unless $n$ is close to $p_{j+1}$. Thus, until $n$ comes
close to $p_{j+1}$, the slope $\Delta(x,x+q^{-n})$ stays close to the
digit slope $(q-1)(n-2j)$, which increases with $n$, and
$\Delta(x-q^{-n},x)$ stays close to $(q-1)(2p_j-2j-n)+2$, which decreases.
Let $B_j$ be the right side of the last display in the proof of
Proposition~\ref{sparse}. This is the bound for $\Delta(y_j,x)$ at
$y_j=a_{p_j}-qh$. For $q=3$, $4$, and $5$ and for $j=2$
and $j=3$, the slope on the left falls below $B_j$ before $n$ reaches
$p_{j+1}$. Indeed, the number $(q-1)(2p_j-2j-n)+2$ and the bound for
$2q^nT_q(h)$ show that $\Delta(x-q^{-n},x)<B_2$ at $n=9$ and
$\Delta(x-q^{-n},x)<B_3$ at $n=39$ when $q=3$, at $n=10$ and $n=43$ when
$q=4$, and at $n=10$ and $n=46$ when $q=5$. Figure~\ref{sparsepic} shows both families of slopes for these
three bases and $1\le n\le66$, and its wide gray lines show the digit
slopes $d_n(x)$ and the numbers $(q-1)(2p_j-2j-n)+2$. It also shows the
slopes $\Delta(y_j,x)$ (diamonds), plotted at $n=\log_q(1/(x-y_j))$, with
their bounds $B_j$ (bars) for $j=2,3$. These bounds are $-7$ and $-23$
for $q=3$, $-14$ and $-244/5$ for $q=4$, and $-65/3$ and $-233/3$ for
$q=5$. The vertical lines mark the positions $p_2=4$, $p_3=16$, and
$p_4=64$ of the digits $q-1$.

\section{Local extrema and the maximum on each interval of level \texorpdfstring{$n$}{n}}
\label{sec:max}

In this section we prove Theorem~\ref{cellmax} by means of equation
\eqref{cell}, and then Theorem~\ref{thm1}. The following lemma expresses
$T_q$ minus a linear function in terms of the digit slopes of
Definition~\ref{digslope}.

\begin{lemma}
\label{tiltform}
Let $d_1,d_2,\ldots$ be the digit slopes of a digit sequence of a point
$t\in[0,1]$, and let $c$ be a real number. Then
\begin{equation}
 T_q(t)-ct=\frac{(q-c)^2+q}{4q}
 -\frac{q-1}{4q}\sum_{k=1}^{\infty}\frac{(d_k-c)^2}{q^k}.
 \label{tiltsq}
\end{equation}
\end{lemma}

\begin{proof}
Let $\xi_1,\xi_2,\ldots$ be the digits of the sequence, so that
$d_k=\sum_{j=1}^{k}(q-1-2\xi_j)$. Since $|q-1-2\xi_j|\le q-1$, the double
series $\sum_{k=1}^{\infty}\sum_{j=1}^{k}q^{-k}(q-1-2\xi_j)$ converges
absolutely, and we may sum over $k$ first. Since
$\sum_{k=j}^{\infty}q^{-k}=q^{1-j}/(q-1)$,
$\sum_{j=1}^{\infty}q^{-j}=1/(q-1)$, and
$\sum_{j=1}^{\infty}\xi_jq^{-j}=t$, this gives
\[
 \sum_{k=1}^{\infty}\frac{d_k}{q^k}
 =\frac{q}{q-1}\sum_{j=1}^{\infty}\frac{q-1-2\xi_j}{q^j}
 =\frac{q(1-2t)}{q-1}.
\]
By the last display, the expansion $(d_k-c)^2=d_k^2-2cd_k+c^2$, and
\eqref{dsq},
\[
 \frac{q-1}{4q}\sum_{k=1}^{\infty}\frac{(d_k-c)^2}{q^k}
 =\frac{q+1}4-T_q(t)-\frac{c(1-2t)}2+\frac{c^2}{4q},
\]
which is \eqref{tiltsq}, because
$\frac{(q-c)^2+q}{4q}=\frac{q+1}4-\frac c2+\frac{c^2}{4q}$.
\end{proof}

Equation \eqref{tiltsq} gives the maximum on $[0,1]$ of $T_q(t)$ plus an
integer multiple of $t$, together with all points at which this maximum
is attained.

\begin{lemma}
\label{lem:tilt}
Let $b$ be an integer with $0\le b\le q-2$. Then
\begin{equation}
 \max_{0\le t\le1}\bigl(T_q(t)-bt\bigr)=
 \begin{cases}
 \dfrac{(q-b)^2+q}{4q}&\text{if $q$ is odd and $b$ is even},\\[3mm]
 \dfrac{(q-b)^2+q}{4q}-\dfrac1{4q}&\text{if $q$ and $b$ are odd},\\[3mm]
 \dfrac{(q-b)^2+q}{4q}-\dfrac1{4(q+1)}
 &\text{if $q$ and $b$ are even},\\[3mm]
 \dfrac{(q-b)^2+q}{4q}-\dfrac1{4q(q+1)}
 &\text{if $q$ is even and $b$ is odd}.
 \end{cases}
 \label{tilt0}
\end{equation}
Let $a\ge0$ be an integer and $c=a(q-1)+b$. Then
\begin{align}
 \max_{0\le t\le1}\bigl(T_q(t)-ct\bigr)
 &=q^{-a}\max_{0\le t\le1}\bigl(T_q(t)-bt\bigr),
 \label{tilt1}\\
 \max_{0\le t\le1}\bigl(T_q(t)+ct\bigr)
 &=c+q^{-a}\max_{0\le t\le1}\bigl(T_q(t)-bt\bigr).
 \label{tilt2}
\end{align}
The maximum in \eqref{tilt1} is attained exactly at the points
$t\in[0,1]$ that have a digit sequence whose first $a$ digits are $0$
and whose digit slopes satisfy $|d_k-c|\le1$ for every $k>a$. The
maximum in \eqref{tilt2} is attained exactly at the points $1-t$ for
these $t$. None of these points is a $q$-adic rational. If $q$ is odd
and $b$ is even, then the maximum in \eqref{tilt1} is attained only at
$t=(q-b)/(2q^{a+1})$, and the maximum in \eqref{tilt2} only at
$t=1-(q-b)/(2q^{a+1})$. In the other three cases of \eqref{tilt0}, each
of the two maxima is attained at uncountably many points. For $q=3$ and
$c=2a$ we have $b=0$, so that the two maxima are $3^{-a}$ and
$2a+3^{-a}$, attained only at $1/(2\cdot3^a)$ and at
$1-1/(2\cdot3^a)$, respectively.
\end{lemma}

For $c=0$, the lemma is Theorem~\ref{gmax}. Galkin and
Galkina~\cite[Theorem~4]{GalkinGalkina2020} proved the case $c=1$ of
\eqref{tilt2}, together with the points at which the maximum is
attained, in another way. They also determined the concave hull of $T_q$
on $[0,1]$, the smallest concave function on $[0,1]$ that is at least
$T_q$~\cite[Theorem~9]{GalkinGalkina2020}. Since the maximum of
$T_q(t)-ct$ over $[0,1]$ is the maximum of this hull minus $ct$, their
description of the hull determines the values of the maxima in
\eqref{tilt1} and \eqref{tilt2}.

\begin{proof}
Suppose first that $a=0$, so that $c=b$. Let $d_1,d_2,\ldots$ be the
digit slopes of a digit sequence $\xi_1,\xi_2,\ldots$ of a point
$t\in[0,1]$. Since $d_k-(q-1)k=-2\sum_{j=1}^{k}\xi_j$ is even, the
integer $d_k-b$ is odd exactly when $(q-1)k-b$ is odd, and then
$(d_k-b)^2\ge1$. By \eqref{tiltsq} with $b$ in place of $c$,
\begin{equation}
 T_q(t)-bt\le\frac{(q-b)^2+q}{4q}-\frac{q-1}{4q}\sum_{k}\frac1{q^k},
 \label{tiltpar}
\end{equation}
where the sum is over the integers $k\ge1$ for which $(q-1)k-b$ is odd.
Equality holds in \eqref{tiltpar} exactly when $|d_k-b|\le1$ for every $k\ge1$. If $q$ is
odd, then $(q-1)k-b$ has the parity of $b$, so the sum is $0$ when $b$
is even and $\sum_{k=1}^{\infty}q^{-k}=1/(q-1)$ when $b$ is odd. If $q$
is even, then $(q-1)k-b$ has the parity of $k-b$, so the sum is
$\sum_{i=0}^{\infty}q^{-2i-1}=q/(q^2-1)$ when $b$ is even and
$\sum_{i=1}^{\infty}q^{-2i}=1/(q^2-1)$ when $b$ is odd. Multiplied by
$(q-1)/(4q)$, these four numbers are $0$, $1/(4q)$, $1/(4(q+1))$, and
$1/(4q(q+1))$, so that the right side of \eqref{tiltpar} is that of
\eqref{tilt0}. Hence \eqref{tilt0} follows once we find a digit sequence
with $|d_k-b|\le1$ for every $k\ge1$.

We have $d_1=q-1-2\xi_1$ and $d_k=d_{k-1}+q-1-2\xi_k$ for $k\ge2$. The
condition $|d_k-b|\le1$ means that $d_k=b$ when $(q-1)k-b$ is even, and
that $d_k$ is $b-1$ or $b+1$ when $(q-1)k-b$ is odd. Let $q$ be odd and
$b$ even. Then the condition is $d_k=b$ for every $k\ge1$, which holds
exactly when $\xi_1=(q-1-b)/2$ and $\xi_k=(q-1)/2$ for every $k\ge2$,
that is, when
\[
 t=\frac{q-1-b}{2q}+\sum_{k=2}^{\infty}\frac{q-1}{2q^k}
 =\frac{q-b}{2q}.
\]
Let $q$ and $b$ be odd. Then the condition holds for $k=1$ exactly when
$\xi_1$ is $(q-b)/2$ or $(q-2-b)/2$. If $k\ge2$ and $d_{k-1}=b+1$, it
holds for $k$ exactly when $\xi_k$ is $(q-1)/2$ or $(q+1)/2$, and if
$d_{k-1}=b-1$, exactly when $\xi_k$ is $(q-1)/2$ or $(q-3)/2$. Let $q$
be even. For $k=1$ the condition is $\xi_1=(q-1-b)/2$ when $b$ is odd,
and it is that $\xi_1$ is $(q-b)/2$ or $(q-2-b)/2$ when $b$ is even. If
$k\ge2$ and $d_{k-1}=b$, the condition holds for $k$ exactly when
$\xi_k$ is $q/2-1$ or $q/2$. If $d_{k-1}=b+1$, it holds exactly when
$\xi_k=q/2$, and if $d_{k-1}=b-1$, exactly when $\xi_k=q/2-1$.

Since $0\le b\le q-2$, all the digits named here are integers between
$0$ and $q-1$. Hence a digit sequence with $|d_k-b|\le1$ for every
$k\ge1$ exists in each case, which proves \eqref{tilt0}. If $q$ is odd
and $b$ is even, there is exactly one such sequence, and it belongs to
the point $(q-b)/(2q)$. In the other three cases there are infinitely
many indices $k$ at which two digits are possible, whatever the earlier
digits are, so there are uncountably many such sequences. By
Definition~\ref{digslope}, the digits of every digit sequence of a
$q$-adic rational in $[0,1]$ are all $0$ or all $q-1$ from some index
on, and then $d_k\to+\infty$ or $d_k\to-\infty$, respectively, so that $|d_k-b|\le1$
fails for some $k$. Therefore
no point at which the maximum in \eqref{tilt0} is attained is a $q$-adic
rational, each of these points has only one digit sequence, and
different digit sequences with $|d_k-b|\le1$ for every $k\ge1$ belong to
different points. This proves the statements on the maximum in
\eqref{tilt1} for $a=0$.

Suppose now that $a\ge1$. For $y\ge0$ we have $\tau_q(y)\le(q-1)y$, with
equality exactly for $0\le y\le1/q$. Indeed, $\tau_q(y)=(q-1)y$ on
$[0,1/q]$ by \eqref{tau}, the slopes of $\tau_q$ on $[1/q,1]$ are at most
$q-3$, and for $y>1$ we have $\tau_q(y)\le q/4<(q-1)y$. Hence
$F_a(t)=\sum_{j=0}^{a-1}q^{-j}\tau_q(q^jt)\le(q-1)at$ for $0\le t\le1$,
with equality exactly when $q^{a-1}t\le1/q$, that is, when
$t\le q^{-a}$. Since $T_q$ has period $1$ and $b\ge0$, we have
$T_q(y)-by\le T_q(y-\lfloor y\rfloor)-b(y-\lfloor y\rfloor)$ for every
$y\ge0$. With $y=q^at$, this, \eqref{eq:tail}, the bound for $F_a$, and
$c=(q-1)a+b$ give
\[
 T_q(t)-ct=F_a(t)-(q-1)at+q^{-a}\bigl(T_q(q^at)-bq^at\bigr)
 \le q^{-a}\max_{0\le v\le1}\bigl(T_q(v)-bv\bigr)
\]
for $0\le t\le1$, and equality holds exactly when $t\le q^{-a}$ and
$T_q(v)-bv$ attains its maximum over $[0,1]$ at $v=q^at$. This proves
\eqref{tilt1}.

Let $0\le t\le q^{-a}$ and $v=q^at$. The digit sequences of $t$ whose first
$a$ digits are $0$ are the sequences that consist of $a$ digits $0$
followed by a digit sequence of $v$. Conversely, a point of $[0,1]$
that has a digit sequence whose first $a$ digits are $0$ is at most
$\sum_{j=a+1}^{\infty}(q-1)q^{-j}=q^{-a}$. If $d_1,d_2,\ldots$ are the
digit slopes of a digit sequence of $v$, then the digit slopes of the
corresponding digit sequence of $t$ are $(q-1)k$ at the indices
$k\le a$ and $(q-1)a+d_k$ at the indices $a+k$ with $k\ge1$, and
$|(q-1)a+d_k-c|=|d_k-b|$. By the case $a=0$, the maximum in
\eqref{tilt1} is therefore attained exactly at the points described in
the lemma. Since $t=v/q^a$ is a $q$-adic rational exactly when $v$ is,
and different points $v$ give different points $t$, the case $a=0$ also
shows that none of these points is a $q$-adic rational, that there are
uncountably many of them in the last three cases of \eqref{tilt0}, and
that $(q-b)/(2q^{a+1})$ is the only one when $q$ is odd and $b$ is even.

Since $T_q(1-t)=T_q(t)$, we have $T_q(1-t)+c(1-t)=c+T_q(t)-ct$ for
$0\le t\le1$. Hence the maximum in \eqref{tilt2} is $c$ plus the maximum
in \eqref{tilt1}, and it is attained exactly at the points $1-t$ for
which $T_q(t)-ct$ is the maximum in \eqref{tilt1}. The point $1-t$ is a
$q$-adic rational exactly when $t$ is. For $q=3$ and $b=0$, the right
side of \eqref{tilt0} is $(9+3)/12=1$.
\end{proof}

For an integer $c$ we write
\[
 \Lambda(c)=\max_{0\le t\le1}\bigl(T_q(t)-ct\bigr).
\]
If $c\ge0$ and $c=a(q-1)+b$ with integers $a\ge0$ and $0\le b\le q-2$,
then $\Lambda(c)$ is $q^{-a}$ times the number in \eqref{tilt0}, by
\eqref{tilt1}, and $\Lambda(-c)=c+\Lambda(c)$ by \eqref{tilt2}. In
particular, $\Lambda(0)$ is the maximum of $T_q$, and $\Lambda(2k)=3^{-k}$
for $q=3$ and $k\ge0$. Every $\Lambda(c)$ is positive, and
$\Lambda(c')<\Lambda(c)$ for integers $c<c'$. Indeed, by \eqref{Thbd} we
have $T_q(t)>c't$ for all sufficiently small $t>0$, so the maximum
$\Lambda(c')$ is positive and is attained at a point $t>0$, and at this
point $T_q(t)-c't<T_q(t)-ct\le\Lambda(c)$.

Lemma~\ref{lem:tilt} and \eqref{cell} locate the maximum of $T_q$ on every
interval of level $n$. When $q$ is odd, the slope of every such interval
is even, and so is $b$, because $q-1$ is even. When $q\ge4$ is even,
both parities of $b$ occur.

\begin{proof}[Proof of Theorem~\ref{cellmax}]
Let $c=a(q-1)+b$, the absolute value of the slope of $J$. Every point of
$J$ is $x=(m+t)/q^n$ with $0\le t\le1$, and $x$ is a $q$-adic rational
exactly when $t$ is one. Suppose that $t$ is not a $q$-adic rational,
and let $\xi_1,\ldots,\xi_n$ be the digits with
$m=\sum_{j=1}^{n}\xi_jq^{n-j}$. Then the digit sequence of $x$ consists
of $\xi_1,\ldots,\xi_n$ followed by the digit sequence of $t$. By
\eqref{dnsq}, $d_n(x)=(q-1)n-2s_q(m)$, which is the slope of $J$. Hence,
by the definition of the digit slopes,
\begin{equation}
 d_{n+k}(x)=(q-1)n-2s_q(m)+d_k(t),~k\ge0.
 \label{cellshift}
\end{equation}
By \eqref{cell} and since $T_q(1)=0$,
\[
 T_q\Bigl(\frac{m+t}{q^n}\Bigr)
 =T_q\Bigl(\frac m{q^n}\Bigr)+q^{-n}\bigl(T_q(t)\pm ct\bigr)
 ~\text{and}~
 T_q\Bigl(\frac{m+1}{q^n}\Bigr)=T_q\Bigl(\frac m{q^n}\Bigr)\pm cq^{-n}
\]
for $0\le t\le1$, where the sign is $+$ if the slope of $J$ is $c$ and
$-$ if it is $-c$.

Suppose first that the slope of $J$ is $-c$. Then
$T_q((m+1)/q^n)\le T_q(m/q^n)$, and by \eqref{tilt1} the maximum of
$T_q$ on $J$ is
$T_q(m/q^n)+q^{-n-a}\max_{0\le t\le1}\bigl(T_q(t)-bt\bigr)$, which is
\eqref{cellmax1}. By Lemma~\ref{lem:tilt}, it is attained exactly at the
points $x=(m+t)/q^n$ for which $t$ is not a $q$-adic rational, the first
$a$ digits of $t$ are $0$, and $|d_k(t)-c|\le1$ for every $k>a$. Since
$d_{n+k}(x)=d_k(t)-c$ by \eqref{cellshift}, these are the points
described in the theorem.

Suppose now that the slope of $J$ is $c$. Then
$T_q((m+1)/q^n)\ge T_q(m/q^n)$, and by \eqref{tilt2} the maximum of
$T_q$ on $J$ is
\[
 T_q\Bigl(\frac m{q^n}\Bigr)
 +q^{-n}\Bigl(c+q^{-a}\max_{0\le t\le1}\bigl(T_q(t)-bt\bigr)\Bigr)
 =T_q\Bigl(\frac{m+1}{q^n}\Bigr)
 +q^{-n-a}\max_{0\le t\le1}\bigl(T_q(t)-bt\bigr),
\]
which is \eqref{cellmax1}. By Lemma~\ref{lem:tilt}, it is attained
exactly at the points $x=(m+1-t)/q^n$ for which $t$ is not a $q$-adic rational, the first $a$ digits of $t$ are $0$, and $|d_k(t)-c|\le1$ for
every $k>a$. Let $\eta_1,\eta_2,\ldots$ be the digit sequence of such a
point $t$. Since $\sum_{j=1}^{\infty}(q-1)q^{-j}=1$ and $1-t$ is not a
$q$-adic rational, the digit sequence of $1-t$ is
$q-1-\eta_1,q-1-\eta_2,\ldots$. Hence the first $a$ digits of $1-t$ are
$q-1$, and $d_k(1-t)=-d_k(t)$ for every $k\ge0$. Thus
\eqref{cellshift} with $1-t$ in place of $t$ gives
$d_{n+k}(x)=c-d_k(t)$, so these are again the points described in the
theorem.

If $q$ is odd, then $(q-1)n$, $2s_q(m)$, and $a(q-1)$ are even, so $b$
is even. By Lemma~\ref{lem:tilt}, the maximum of $T_q(t)-ct$ over $[0,1]$
is then attained only at $t=(q-b)/(2q^{a+1})$, and the maximum of
$T_q(t)+ct$ only at $t=1-(q-b)/(2q^{a+1})$, which gives \eqref{cellmax2}.
If $q$ is even, then by the same lemma each of these two maxima is
attained at uncountably many points $t$. For $q=3$ the slope of $J$ is
$2n-2s_3(m)=2d$, so that $a=|d|$ and $b=0$. The maximum in \eqref{tilt0}
is then $(9+3)/12=1$, and $(q-b)/(2q^{n+a+1})=1/(2\cdot3^{n+|d|})$.
\end{proof}

\begin{example}
\label{ex:cellmax}
Let $q=3$. For the interval $[0,1/9]$ we have $n=2$, $m=0$, and $d=2$.
Since $T_q(1/9)=4/9$ by Proposition~\ref{incr}(b), \eqref{cellmax1} gives
$\max_{0\le x\le1/9}T_q(x)=4/9+3^{-4}=37/81$, and by \eqref{cellmax2} this
maximum is attained only at $1/9-1/162=17/162$. For $[4/9,5/9]$ we have
$m=4$ and $d=0$, and $T_q(4/9)=T_q(5/9)=8/9$ by Example~\ref{ex:F2} and
\eqref{eq:tail}. Hence the maximum of $T_q$ on $[4/9,5/9]$ is
$8/9+1/9=1$, attained only at $4/9+1/18=1/2$, in agreement with
Theorem~\ref{limitthm}.

Let now $q=4$. For the interval $[0,1/4]$ we have $n=1$ and $m=0$, and
the slope of the interval is $3$, so that $a=1$ and $b=0$. Since
$T_q(1/4)=3/4$ by Proposition~\ref{incr}(b) and the maximum of $T_q$ is
$6/5$ by Theorem~\ref{gmax}, \eqref{cellmax1} gives
$\max_{0\le x\le1/4}T_q(x)=3/4+4^{-2}\cdot6/5=33/40$. The first digit of
every point in the interior of $[0,1/4]$ is $0$. By
Theorem~\ref{cellmax}, the maximum is attained exactly at the points of
the interval that are not $4$-adic rationals, whose second digit is $3$,
and whose digit slopes satisfy $|d_k|\le1$ for every $k>2$. For these
points $d_2=3+(3-6)=0$, and each further digit $\xi_k$ changes the digit
slope by $3-2\xi_k$, which is one of the numbers $3$, $1$, $-1$, and
$-3$. Hence these are the points whose digits are $0,3$ followed by
pairs $1,2$ or $2,1$. One of them is $17/80$, whose digits are
$0,3,1,2,1,2,\ldots$.
\end{example}

Figure~\ref{fig4} illustrates Lemma~\ref{lem:tilt} in its top row and
Theorem~\ref{cellmax} in its bottom row, for $q=3$, $4$, and $5$. The top
row shows the functions $T_q(t)-ct$ for $c=0$, $2$, and $4$ when $q$ is
$3$ or $5$, and for $c=0$, $1$, and $2$ when $q=4$, with dotted lines at
the heights of their maxima, which are given by \eqref{tilt0} and
\eqref{tilt1}. For $q=3$ we have $c=2a$ and $b=0$, so that the maxima are
$1$, $1/3$, and $1/9$, attained only at $1/2$, $1/6$, and $1/18$. For
$q=5$ the three values of $c$ give $(a,b)=(0,0)$, $(0,2)$, and $(1,0)$,
so that the maxima are $3/2$, $7/10$, and $3/10$, attained only at
$1/2$, $3/10$, and $1/10$. In these panels the dots mark the points at
which the maxima are attained. For $q=4$ we have $a=0$ and $c=b$, and by
\eqref{tilt0} the maxima are $6/5$, $4/5$, and $9/20$. By the lemma, each
of them is attained exactly at the uncountably many points whose digit slopes
satisfy $|d_k-c|\le1$ for every $k\ge1$. For each $c$ the dots mark
sixteen of them, which overlap at the scale of the panel.

\begin{figure}[H]
\centering
\includegraphics{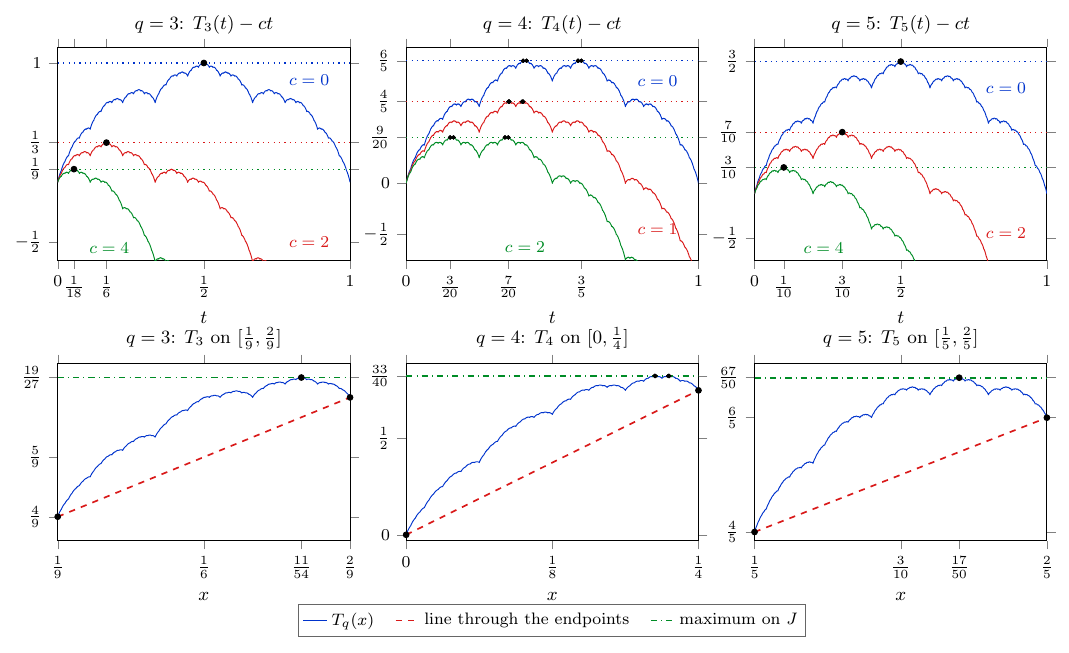}
\caption{Lemma~\ref{lem:tilt} (top row) and Theorem~\ref{cellmax} (bottom
row) for $q=3$, $4$, and $5$.}
\label{fig4}
\end{figure}

The bottom row shows $T_q$ on an interval $J$ of level $n$, the line
through the points of the graph above the endpoints of $J$, and the
horizontal line at the height of the maximum \eqref{cellmax1}. For $q=3$
and $J=[1/9,2/9]$ we have $n=2$, $m=1$, and $d=1$, so that the line
through the endpoints has slope $2d=2$. By Theorem~\ref{cellmax}, the maximum of $T_q$ on $J$ is
$T_q(2/9)+3^{-3}=19/27$, attained only at $11/54$, the midpoint of the
interval $[5/27,2/9]$ of level $n+|d|=3$ that ends at the higher
endpoint $2/9$. By \eqref{cell2}, the graph read from $2/9$ toward $1/9$
is the curve for $c=2$ in the top left panel, scaled by $1/9$ in both
directions and raised by $2/3$. For $q=4$ we take the interval
$J=[0,1/4]$ of Example~\ref{ex:cellmax}, on which the maximum is $33/40$.
The dots on the horizontal line mark the sixteen points of $J$ whose
digits are $0,3$ followed by four pairs $1,2$ or $2,1$ and then by the
pair $1,2$ repeated. One of them is $17/80$. At this scale only two separate
dots can be seen. For $q=5$ and $J=[1/5,2/5]$ we have $n=1$ and $m=1$, so
that the slope of $J$ is $4-2=2$, $a=0$, and $b=2$. Since
$T_q(1/5)=\tau_q(1/5)=4/5$ and $T_q(2/5)=\tau_q(2/5)=6/5$ by \eqref{fe},
equations \eqref{cellmax1} and \eqref{tilt0} give the maximum
$6/5+5^{-1}\cdot7/10=67/50$, and by \eqref{cellmax2} it is attained only
at $2/5-3/50=17/50$.

Theorem~\ref{thm1} combines Proposition~\ref{qmin} on the local minima of
$T_q$ with Propositions~\ref{oddq} and~\ref{evenq} on the local maxima for
odd and for even $q$.

\begin{proposition}\label{qmin}
The local minima of $T_q$ are exactly the $q$-adic rationals, and all of them
are strict.
\end{proposition}

\begin{proof}
By Proposition~\ref{cusp}, every $q$-adic rational is a strict local
minimum. Kr\"uppel~\cite[Proposition~4.1]{Kruppel2007} proved the
proposition for Takagi's function, and for the other points we follow
his proof. Let $x$ be a real number that is not a $q$-adic rational. Since
$T_q$ has period $1$, we may assume that $0<x<1$. Let $n\ge0$, let
$[m/q^n,(m+1)/q^n]$ be the interval of level $n$ that contains $x$, and
write $x=(m+t)/q^n$ with $0<t<1$. By \eqref{dnsq}, the slope of this
interval is the digit slope $d_n(x)$. By \eqref{cell} and
Definition~\ref{intslope},
\[
 T_q(x)=T_q\Bigl(\frac m{q^n}\Bigr)+q^{-n}\bigl(d_n(x)t+T_q(t)\bigr)
 ~\text{and}~
 T_q\Bigl(\frac{m+1}{q^n}\Bigr)=T_q\Bigl(\frac m{q^n}\Bigr)+q^{-n}d_n(x),
\]
where the second equation is the case $t=1$ of \eqref{cell}, because $T_q(1)=0$. Since
$T_q(t)>0$ by \eqref{Tpositive},
\[
 T_q(x)>T_q\Bigl(\frac m{q^n}\Bigr)+q^{-n}d_n(x)t
 =(1-t)T_q\Bigl(\frac m{q^n}\Bigr)+tT_q\Bigl(\frac{m+1}{q^n}\Bigr),
\]
and the right side is at least the smaller of $T_q(m/q^n)$ and
$T_q((m+1)/q^n)$. Thus $T_q$ is smaller than $T_q(x)$ at one of the two
endpoints of the interval, which are at distance less than $q^{-n}$ from
$x$. Since $n$ is arbitrary, $x$ is not a local minimum.
\end{proof}
\begin{proposition}
\label{oddq}
Let $q\ge3$ be odd. For an integer $d\ge0$ write $2d=a(q-1)+b$ with
integers $a\ge0$ and $0\le b\le q-2$. Then $b$ is even, the maximum of
$T_q(t)-2dt$ on $[0,1]$ equals $((q-b)^2+q)/(4q^{a+1})$, and it is
attained only at $t=(q-b)/(2q^{a+1})$. The local maxima of $T_q$ are
exactly the points
$(K+1/2)/q^N+k$, where $k$, $K$, and $N$ are integers with $N\ge0$,
$0\le K<q^N$, and $2s_q(K)=(q-1)N$, and all of them are strict.
\end{proposition}

If $q=3$, then $b=0$ and $a=d$, and the condition on $K$ reads
$s_3(K)=N$.

\begin{proof}
Since $2d$ and $a(q-1)$ are even, $b$ is even, and the first statement is
the case $c=2d$ of Lemma~\ref{lem:tilt}, with \eqref{tilt1} and the first
case of \eqref{tilt0}.

Let $N\ge0$ and $0\le K<q^N$ with $2s_q(K)=(q-1)N$. Then the interval
$[K/q^N,(K+1)/q^N]$ has slope $0$, and \eqref{cell} gives
$T_q((K+u)/q^N)=T_q(K/q^N)+q^{-N}T_q(u)$ for $0\le u\le1$. By
Theorem~\ref{gmax}, $T_q$ attains its maximum on $[0,1]$ only at the
interior point $u=1/2$. Thus $(K+1/2)/q^N$ is a strict local maximum, and
by periodicity so is $(K+1/2)/q^N+k$ for every integer $k$.

Conversely, let $x$ be a local maximum of $T_q$. By periodicity we may
assume that $0\le x<1$. By Proposition~\ref{qmin}, every $q$-adic rational
is a strict local minimum, so $x$ is not a $q$-adic rational. Choose
$\delta>0$ such that $T_q(x')\le T_q(x)$ whenever $|x'-x|<\delta$, fix $n$
with $q^{-n}<\delta$, and let $J=[m/q^n,(m+1)/q^n]$ be the interval of
level $n$ that contains $x$. Then $T_q$ attains its maximum on $J$ at $x$.
Write the absolute value of the slope of $J$ as $a(q-1)+b$ with integers
$a\ge0$ and $0\le b\le q-2$, as in Theorem~\ref{cellmax}. By that
theorem, $b$ is even and $x$ is the point \eqref{cellmax2}. Put
$N=n+a+1$. Suppose first that the slope of $J$ is at most $0$, so that
$2s_q(m)=(q-1)n+a(q-1)+b$. Then $x=(K+1/2)/q^N$ with
$K=mq^{a+1}+(q-1-b)/2$. The base $q$ digits of $K$ are those of $m$,
followed by $a$ zeros and by the digit $(q-1-b)/2$, so
$2s_q(K)=2s_q(m)+q-1-b=(q-1)N$. Suppose now that the slope of $J$ is at
least $0$, so that $2s_q(m)=(q-1)n-a(q-1)-b$. Then $x=(K+1/2)/q^N$ with
$K=(m+1)q^{a+1}-(q+1-b)/2$. The base $q$ digits of $K$ are those of $m$,
followed by $a$ digits $q-1$ and by the digit $(q-1+b)/2$, so
$2s_q(K)=2s_q(m)+2a(q-1)+q-1+b=(q-1)N$. In both cases $x$ has the
stated form.
\end{proof}

For odd $q$ and every interval $J$ of level $n$, the computation in this
proof shows that the point \eqref{cellmax2}, at which $T_q$ attains its
maximum on $J$, is the midpoint of an
interval of level $n+a+1$ and slope $0$ that lies in $J$. If the slope
of $J$ is a multiple of $q-1$, as it always is for $q=3$, then $b=0$,
and this point is also the midpoint of the interval of level $n+a$ that
lies in $J$ and contains the endpoint of $J$ at which $T_q$ is larger.

Figure~\ref{q5pic} illustrates Proposition~\ref{oddq} for $q=5$, where
$\tau_5$ has the slopes $4,2,0,-2,-4$. For $d=0,1,2,3$ we have
$2d=4a+b$ with $(a,b)=(0,0)$, $(0,2)$, $(1,0)$, and $(1,2)$, so by
Proposition~\ref{oddq} the function $T_5(t)-2dt$ attains its maximum on
$[0,1]$ only at $1/2$, $3/10$, $1/10$, and $3/50$, and these maxima are
$3/2$, $7/10$, $3/10$, and $7/50$. The right panel shows these functions
on $[0,3/5]$, with dotted lines at their maxima. Since \eqref{cell} with
$n=1$ and $m=0$ gives $T_5(t/5)=(4t+T_5(t))/5$, we have
$T_5(t/5)-2(d+2)t/5=(T_5(t)-2dt)/5$ for $0\le t\le1$. Thus on $[0,1/5]$
the curves for $d=2$ and $d=3$ are those for $d=0$ and $d=1$ scaled by
$1/5$ in both directions. The left panel shows $T_5$ and its local
maxima $(K+1/2)/5^N$ in $(0,1)$ with $N\le2$. They are $1/2$, which
comes from $(N,K)=(0,0)$, $(1,2)$, and $(2,12)$, and the points $9/50$,
$17/50$, $33/50$, and $41/50$, which come from $N=2$ and $K=4,8,16,20$.

\begin{figure}[H]
\centering
\includegraphics{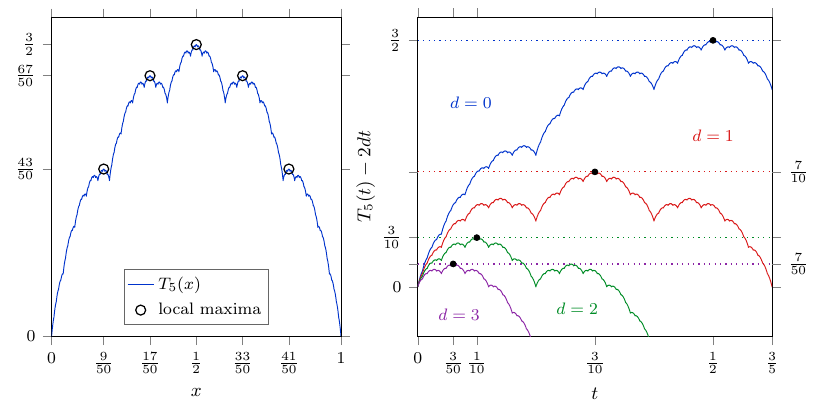}
\caption{Proposition~\ref{oddq} for $q=5$. Left: $T_5$ and its local
maxima $(K+1/2)/5^N$ with $N\le2$. Right: the functions $T_5(t)-2dt$ for
$d=0,1,2,3$.}
\label{q5pic}
\end{figure}
\begin{proposition}\label{evenq}
Let $q$ be even, and let $x\in[0,1)$. Then $x$ is a local maximum of
$T_q$ if and only if there are a digit sequence $\xi_1,\xi_2,\ldots$ of $x$
and an integer $N\ge0$ such that the digit slopes $d_n$ of this sequence
satisfy $d_{2N}=0$ and, for every $i\ge N$, the pair
$(\xi_{2i+1},\xi_{2i+2})$ is $(q/2-1,q/2)$ or $(q/2,q/2-1)$. No local
maximum of $T_q$ is strict.
\end{proposition}

Since $T_q$ has period $1$, the local maxima of $T_q$ in $\R$ are the points
$x+k$ with $x$ as in the proposition and integers $k$. By \eqref{dnsq} and
Definition~\ref{intslope}, the condition $d_{2N}=0$ says that the interval
of level $2N$ with left endpoint $\sum_{j=1}^{2N}\xi_jq^{-j}$ has slope
$0$. For $q=2$ the two pairs are $(0,1)$ and $(1,0)$, and we obtain the
description of the local maxima of Takagi's function that we recalled
after Theorem~\ref{thm1}.

\begin{proof}
Let $\xi_1,\xi_2,\ldots$ be a digit sequence of $x$ with digit slopes
$d_n$. For $n\ge0$ put $K_n=\sum_{j=1}^{n}\xi_jq^{n-j}$ and
$u_n=q^nx-K_n=\sum_{j=1}^{\infty}\xi_{n+j}q^{-j}$. Then $K_n$ is an integer
with $0\le K_n\le(q-1)\sum_{j=1}^{n}q^{n-j}=q^n-1$, and $K_n/q^n$ is the
number $a_n$ of \eqref{dnsq}. Moreover, $0\le u_n\le1$, the sequence
$\xi_{n+1},\xi_{n+2},\ldots$ is a digit sequence of $u_n$, and its digit
slopes are the numbers $d_{n+k}-d_n$ with $k\ge0$. By \eqref{dnsq} and
\eqref{cell},
\begin{equation}
 T_q\Bigl(\frac{K_n+t}{q^n}\Bigr)=T_q\Bigl(\frac{K_n}{q^n}\Bigr)
 +q^{-n}\bigl(d_nt+T_q(t)\bigr)~\text{for}~0\le t\le1,
 \label{evcell}
\end{equation}
and $x=(K_n+u_n)/q^n$.

Suppose first that $d_{2N}=0$ and that, for every $i\ge N$, the pair
$(\xi_{2i+1},\xi_{2i+2})$ is $(q/2-1,q/2)$ or $(q/2,q/2-1)$. Then the digit
sequence $\xi_{2N+1},\xi_{2N+2},\ldots$ of $u_{2N}$ has the property stated
in Theorem~\ref{gmax}, so $T_q$ attains its maximum on $[0,1]$ at
$u_{2N}$. Each pair $(\xi_{2i+1},\xi_{2i+2})$ with $i\ge N$ contains the
two different digits $q/2-1$ and $q/2$, so the digits of this sequence
are neither all $0$ nor all $q-1$ from some index on, and by
Definition~\ref{digslope} the point $u_{2N}$ is not a $q$-adic rational.
Hence $0<u_{2N}<1$, and by \eqref{evcell} with $n=2N$ the maximum of $T_q$
on the interval $[K_{2N}/q^{2N},(K_{2N}+1)/q^{2N}]$ is attained at its
interior point $x$. Thus $x$ is a local maximum. For
$j\ge N$ let $x_j$ be the point of $[0,1]$ with the digit sequence that we
obtain from $\xi_1,\xi_2,\ldots$ by exchanging $\xi_{2j+1}$ and
$\xi_{2j+2}$. The new sequence has the same first $2N$ digits and satisfies the same
two conditions with the same $N$. Thus the argument above shows that
$T_q$ attains its maximum on the same interval of level $2N$ at $x_j$,
so $T_q(x_j)=T_q(x)$. Since $|\xi_{2j+1}-\xi_{2j+2}|=1$, we
have $|x_j-x|=q^{-2j-1}-q^{-2j-2}>0$, and this distance tends to $0$ as
$j\to\infty$. Therefore $x$ is not a strict local maximum.

Conversely, let $x$ be a local maximum of $T_q$. By
Proposition~\ref{qmin}, every $q$-adic rational is a strict local minimum,
so $x$ is not a $q$-adic rational. Hence $x$ has only one digit sequence,
no $u_n$ is a $q$-adic rational, and $0<u_n<1$ for every $n$. Choose
$\delta>0$ such that $T_q(x')\le T_q(x)$ whenever $|x'-x|<\delta$, and fix
$n$ with $q^{-n}<\delta$. The interval $[K_n/q^n,(K_n+1)/q^n]$ contains
$x$ and has length $q^{-n}$, so $T_q$ attains its maximum on it at $x$,
and $d_n$ is its slope by \eqref{dnsq}. Write $|d_n|=a(q-1)+b$ with integers $a\ge0$ and
$0\le b\le q-2$. By Theorem~\ref{cellmax}, the digit slopes of $x$ satisfy
$|d_j|\le1$ for every $j>n+a$. Each difference
$d_j-d_{j-1}=q-1-2\xi_j$ is odd because $q$ is even, so $d_j$ has the
parity of $j$. Therefore $d_j=0$ for every even
$j>n+a$, and $d_j$ is $1$ or $-1$ for every odd $j>n+a$. Choose $N$ with
$2N>n+a$. Then $d_{2N}=0$. For $i\ge N$, the number
$q-1-2\xi_{2i+1}=d_{2i+1}-d_{2i}$ is $1$ or $-1$, so $\xi_{2i+1}$ is
$q/2-1$ or $q/2$, and
$0=d_{2i+2}-d_{2i}=2(q-1)-2(\xi_{2i+1}+\xi_{2i+2})$ gives
$\xi_{2i+1}+\xi_{2i+2}=q-1$. Hence the pair $(\xi_{2i+1},\xi_{2i+2})$ is
$(q/2-1,q/2)$ or $(q/2,q/2-1)$. By the first part of the proof, $x$ is
not a strict local maximum. Since $T_q$ has period $1$, no local maximum
of $T_q$ is strict.
\end{proof}

\begin{example}\label{ex:evenq}
Let $q=4$ and $x=17/80$. Since $17/80=3/16+(2/5)/16$ and
$2/5=\sum_{i=0}^{\infty}(4^{-2i-1}+2\cdot4^{-2i-2})$, the digit sequence of
$x$ is $0,3,1,2,1,2,\ldots$. Its digit slopes satisfy $d_2=(3-0)+(3-6)=0$, and for
every $i\ge1$ the pair $(\xi_{2i+1},\xi_{2i+2})$ is $(1,2)=(q/2-1,q/2)$.
By Proposition~\ref{evenq} with $N=1$, the point $17/80$ is a local maximum
of $T_q$ that is not strict. Here \eqref{evcell} with $n=2$ and
Theorem~\ref{gmax} give
\[
 T_q\Bigl(\frac{17}{80}\Bigr)=T_q\Bigl(\frac3{16}\Bigr)
 +\frac1{16}T_q\Bigl(\frac25\Bigr)=\frac34+\frac1{16}\cdot\frac65
 =\frac{33}{40},
\]
where
$T_q(3/16)=\tau_q(3/16)+\frac14\tau_q(3/4)=\frac9{16}+\frac3{16}=\frac34$.
The point $287/1280=17/80+3/256$ has the digit sequence
$0,3,2,1,1,2,1,2,\ldots$, in which the third and the fourth digit are
exchanged, and $T_q(287/1280)=33/40$ as well.
\end{example}

\begin{proof}[Proof of Theorem~\ref{thm1}]
The first statement of the theorem is Proposition~\ref{qmin}, and the
estimate at a $q$-adic rational is \eqref{cusp1}.

Let $q$ be odd. Since $T_q$ attains its maximum on $[0,1]$ only at
$u=1/2$ by Theorem~\ref{gmax}, Proposition~\ref{oddq} shows that the
local maxima of $T_q$ are the points of the theorem and that all of them
are strict. They are indexed by integers, so they form a countable set.

Let $q$ be even. By Theorem~\ref{gmax}, the points $u\in[0,1]$ at which
$T_q$ attains its maximum on $[0,1]$ are the points that have a digit
sequence $\xi_1,\xi_2,\ldots$ in which every pair
$(\xi_{2i+1},\xi_{2i+2})$ is $(q/2-1,q/2)$ or $(q/2,q/2-1)$. With
$\eta_i=\xi_{2i+1}$ we have $\xi_{2i+2}=q-1-\eta_i$, so these points are
the numbers displayed in the theorem. Each pair contains the two
different digits $q/2-1$ and $q/2$, so the digits of such a sequence are
neither all $0$ nor all $q-1$ from some index on. By
Definition~\ref{digslope}, none of these points is a $q$-adic rational.
Hence $0<u<1$, and each of them has only this digit sequence.
Let $x\in[0,1)$ be a local maximum of $T_q$, and let $\xi_1,\xi_2,\ldots$
and $N$ be as in Proposition~\ref{evenq}. Put $n=2N$,
$m=\sum_{j=1}^{n}\xi_jq^{n-j}$, and $u=\sum_{j=1}^{\infty}\xi_{n+j}q^{-j}$.
Then $0\le m<q^n$ and $x=(m+u)/q^n$. Moreover, $2s_q(m)=(q-1)n$ by
\eqref{dnsq}, because $d_n=0$, and $T_q$ attains its maximum on $[0,1]$ at
$u$. Conversely, let $n\ge0$, let $0\le m<q^n$ with $2s_q(m)=(q-1)n$, and
let $u$ be a point at which $T_q$ attains its maximum on $[0,1]$. Then $n$
is even, because $q-1$ is odd. The $n$ base $q$ digits of $m$, starting
with the digit of the highest power, followed by the digit sequence of
$u$, form a digit sequence of the point $x=(m+u)/q^n$, and $0\le x<1$. Its
digit slope with the index $n$ is $(q-1)n-2s_q(m)=0$, and its pairs of
digits after the first $n$ digits are those of $u$. By
Proposition~\ref{evenq} with $N=n/2$, the point $x$ is a local maximum.
Since $T_q$ has period $1$, the local maxima of $T_q$ are therefore the points
of the theorem, and by Proposition~\ref{evenq} none of them is strict.
There are uncountably many sequences $\eta_0,\eta_1,\ldots$ with
$\eta_i\in\{q/2-1,q/2\}$. Different sequences give different points $u$,
because each of these points has only one digit sequence. Hence already
the local maxima $u$, for which $n=0$, form an uncountable set.

In every base the $q$-adic rationals form a countable set that is dense
in $\R$. The local maxima are dense as well. Indeed, by
Theorem~\ref{cellmax}, every interval of level $n$ contains a point that
is not a $q$-adic rational at which $T_q$ attains its maximum on the
interval. Such a point is an interior point of the interval, and hence a
local maximum of $T_q$. Now let
$I$ be a nonempty open interval of $\R$, and choose $n\ge0$ such that
$2q^{-n}$ is less than the length of $I$. Then $I$ contains an interval
$[j/q^n,(j+1)/q^n]$ with $j=kq^n+m$, where $k$ and $0\le m<q^n$ are
integers. The interval $[m/q^n,(m+1)/q^n]$ contains a local maximum $x$
of $T_q$, and since $T_q$ has period $1$, the point $x+k$ of $I$ is a
local maximum of $T_q$.
\end{proof}

\begin{example}
\label{ex:localmax}
Let $q=3$. By Theorem~\ref{thm1}, and since $1/2=0.111\ldots$ in base
$3$, the local maxima of $T_q$ in $(0,1)$ are the points
$0.\xi_1\ldots\xi_n111\ldots$ with $\xi_1+\cdots+\xi_n=n$.
For instance, $1/2$, $5/18=0.02111\ldots$, and $13/18=0.20111\ldots$ are
local maxima, but the midpoint $1/6=0.0111\ldots$ of $[0,1/3]$ is not,
since the sum of its first $n$ digits is $n-1$ for every $n\ge1$. Every
ternary interval in $[0,1/3]$ that contains $1/6$ has slope $2$, and by
Theorem~\ref{cellmax} the maximum of $T_q$ on $[0,1/3]$ is attained only
at $1/3-1/18=5/18$.
\end{example}

Ferrera, G\'omez Gil, and
Llorente~\cite[Theorem~1.2]{FerreraGomezGilLlorente2022} showed that for
the Takagi--van~der~Waerden function $\sum_{j=0}^{\infty}3^{-j}\dist(3^jx,\Z)$, whose
generator has a corner at $1/2$, every midpoint $(m+1/2)/3^n$ is a local
maximum. By Example~\ref{ex:localmax}, this is not true for $T_3$, whose
generator $\tau_3$ is constant on $[1/3,2/3]$.

Let $q$ be odd. Each local maximum of $T_q$ in $(0,1)$ has infinitely many
representations $(K+1/2)/q^N$ with $2s_q(K)=(q-1)N$, because appending
the digit $(q-1)/2$ to $K$ raises $N$ by one and $s_q(K)$ by $(q-1)/2$
and leaves the point unchanged. Removing a final digit $(q-1)/2$ from $K$
reverses this step, so in the shortest representation either $N=0$ or the
last digit of $K$ is not $(q-1)/2$. Conversely, if $N\ge1$ and the last
digit of $K$ is not $(q-1)/2$, then $2K+1$ is not divisible by $q$, and
$N$ cannot be smaller. Indeed, a representation of the same point with a
smaller exponent would write it as an odd integer divided by twice a
smaller power of $q$, and then $2K+1$ would be this odd integer times a
positive power of $q$. Since $2s_q(K)=(q-1)N$, the interval
$[K/q^N,(K+1)/q^N]$ has slope $0$, and \eqref{cell} with $t=1/2$,
\eqref{samp}, and $T_q(1/2)=(q+1)/4$ give
\[
 T_q\Bigl(\frac{K+1/2}{q^N}\Bigr)=T_q\Bigl(\frac K{q^N}\Bigr)+\frac{q+1}{4q^N}
 =\frac{\xi_N(K)+(q+1)/4}{q^N},~N\ge1.
\]
Now let $q=3$, so that the value is
$(\xi_N(K)+1)/3^N$. No new point arises for $N=1$, since the only $K<3$ with
$s_3(K)=1$ is $K=1$, which gives $1/2$ again. For $N=2$ we get $5/18$ and
$13/18$ from $K=2,6$, and for $N=3$ the points $11/54$, $23/54$, $31/54$,
and $43/54$ from $K=5,11,15,21$. Together with $1/2$, these are all the
local maxima in $(0,1)$ with $N\le3$ in the shortest form. By
\eqref{count}, Table~\ref{tab1}, and the symmetry $\xi_n(m)=\xi_n(3^n-m)$,
the values at the six points with $N=2$ and $N=3$ are $7/9$, $7/9$,
$19/27$, $25/27$, $25/27$, and $19/27$. The left panel of
Figure~\ref{extpic} shows them together with $1/2$ and the local minima
$m/9$ with $0\le m\le9$. The level $2$ interval $[2/9,1/3]$ has slope
$0$, since $s_3(2)=2$. On this interval \eqref{cell} gives
$T_q((2+t)/9)=2/3+T_q(t)/9$ for $0\le t\le1$, so $t\mapsto(2+t)/9$ maps
the local maxima in $(0,1)$ to local maxima in $(2/9,1/3)$. For instance, it maps
$1/2$, $5/18$, and $13/18$ to $5/18$ and to the local maxima $41/162$ and
$49/162$ with $N=4$. It also maps the local minima $m/9$ to the local
minima $(18+m)/81$.

The middle panel of Figure~\ref{extpic} shows the case $q=4$. By
Theorem~\ref{gmax}, the maximum $6/5$ of $T_q$ is attained exactly at
the points $u$ of $[0,1]$ whose digit pairs $(\xi_{2i+1},\xi_{2i+2})$ are
all $(1,2)$ or $(2,1)$. The smallest of them is
$2/5=(0.1212\ldots)_4$ and the largest is $3/5=(0.2121\ldots)_4$, while
$T_q(1/2)=1$ by Proposition~\ref{incr}(b). By Theorem~\ref{thm1}, the local maxima of $T_q$ in $(0,1)$
are the points $(m+u)/4^N$ with $2s_4(m)=3N$, so $N$ is even. For $N=2$
the condition $s_4(m)=3$ gives $m=3$, $6$, $9$, and $12$. The digits of
$6$ and $9$ are $1,2$ and $2,1$, so the points $(6+u)/16$ and $(9+u)/16$
are again global maxima. The black dots at height $6/5$ mark these
points for $u=2/5$ and $u=3/5$. Two of them are $2/5=(6+2/5)/16$ and
$3/5=(9+3/5)/16$. The red dots mark the points
$(3+u)/16$ and $(12+u)/16$ for the same two values of $u$. The intervals $[3/16,1/4]$ and $[3/4,13/16]$ that
contain them have slope $6-2\cdot3=0$, and $T_q(3/16)=T_q(3/4)=3/4$ by
\eqref{samp} and \eqref{xi}. Hence \eqref{cell} gives the value
$3/4+\frac1{16}\cdot\frac65=33/40$ at these points. This is the maximum
of $T_q$ on $[0,1/4]$ found in Example~\ref{ex:cellmax}, attained exactly
at the points $(3+u)/16$, where $u$ runs through all points at which
$T_q$ attains its maximum on $[0,1]$. The interval $[3/8,7/16]$ of level
$2$ has slope $0$ because $s_4(6)=3$, and $T_q(3/8)=9/8$ by \eqref{samp}
and \eqref{xi}. Hence \eqref{cell} gives $T_q((6+t)/16)=9/8+T_q(t)/16$
for $0\le t\le1$, so the graph of $T_q$ on this interval is the whole
graph scaled by $1/16$ in both directions and raised by $9/8$.

The right panel of Figure~\ref{extpic} shows the case $q=5$. Here the
value at $(K+1/2)/5^N$ is $(\xi_N(K)+3/2)/5^N$ for $N\ge1$. Again $N=1$
gives no new point, since the only $K<5$ with $s_5(K)=2$ is $K=2$, which
gives $1/2$ again. For $N=2$ the condition $s_5(K)=4$, with a last digit
of $K$ other than $2$, gives $K=4,8,16,20$, that is, the points $9/50$,
$17/50$, $33/50$, and $41/50$ of Figure~\ref{q5pic}. By \eqref{xi} and
the symmetry $T_q(1-x)=T_q(x)$, their values are $43/50$, $67/50$,
$67/50$, and $43/50$. The panel shows these points together with $1/2$
and the local minima $m/25$ with $0\le m\le25$.

\begin{figure}[H]
\centering
\includegraphics{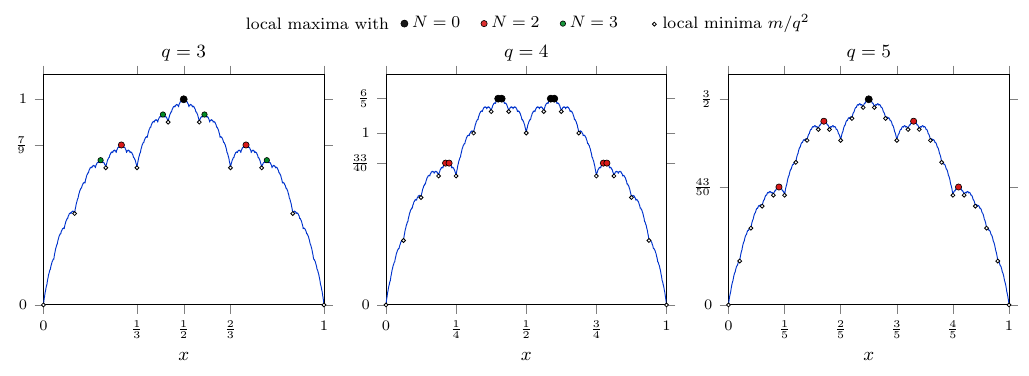}
\caption{Theorem~\ref{thm1}. The graphs of $T_q$ for $q=3$, $4$, and $5$,
with some local maxima and the local minima $m/q^2$.}
\label{extpic}
\end{figure}

\section{The slopes at almost every point}
\label{sec:ae}

By Proposition~\ref{incr}, the largest slope of $T_q$ over a step of length
$h$ is $T_q(h)/h=(q-1)\log_q(1/h)+O(1)$, attained at the integers, and by
Proposition~\ref{cusp} the slope $\Delta(a,a+h)$ at a $q$-adic rational $a$
differs from $T_q(h)/h$ by $O(1)$ as $h\downarrow0$. At almost every point
the slopes are much smaller. For Takagi's function,
K\^ono~\cite[Theorem~5]{Kono1987} stated a law of the iterated logarithm:
at almost every point the slopes over steps $h$, divided by
$\sqrt{2\log_2(1/h)\log\log\log_2(1/h)}$, have the upper limit $1$ and the
lower limit $-1$ as $h\to0$. Gamkrelidze~\cite[Theorem~2]{Gamkrelidze1990}
proved this law by another method and gave this form of the
normalization, and his central limit
theorem~\cite[Theorem~1]{Gamkrelidze1990} states that the slopes over a
fixed step $h$, divided by $\sqrt{\log_2(1/h)}$, have the standard normal
limit distribution as $h\downarrow0$. Both authors wrote the increments as sums of Rademacher functions,
the functions $1-2\xi_j$ of the binary digits $\xi_j$ of $x$. K\^ono then applied the law of the iterated
logarithm of Kolmogorov, and Gamkrelidze applied that of Hartman and
Wintner and the central limit theorem. Nakano and
Takei~\cite[Theorem~2.3]{NakanoTakei2026} proved a central limit theorem
and a law of the iterated logarithm for the slopes of the
Takagi--van~der~Waerden functions $\sum_{k=0}^{\infty}r^{-k}\dist(r^kx,\Z)$
of every integer base $r\ge2$, and the variance of the limit distribution
in their central limit theorem depends on the parity of $r$.
Nakano~\cite{Nakano2026} obtained an almost sure invariance principle,
that is, an approximation by a Brownian motion that holds almost surely,
for weighted versions of these functions. In this section we use the same
method and the bounds of Section~\ref{sec:inf} for the difference between
the slopes of $T_q$ and the digit slopes to prove a law of the iterated
logarithm and a central limit theorem for the slopes of $T_q$ in every
base $q$. Our proofs use the fact that the digit slopes $d_n(x)$ are sums
of independent, identically distributed random variables when $x$ is
chosen uniformly in $[0,1)$. For $q=2$ the normalizations in
Theorems~\ref{liltheorem} and~\ref{clttheorem} below are those above. We
also show that the sets of points with an infinite derivative have
Lebesgue measure zero and Hausdorff dimension $1$. For Takagi's function,
Allaart and Kawamura~\cite[Corollary~3.9]{AllaartKawamura2010} proved that
these sets have Hausdorff dimension $1$, and Ferrera, G\'omez Gil, and
Llorente~\cite[Theorems~4.3 and~4.4]{FerreraGomezGilLlorente2019} proved
both statements for the Takagi--van~der~Waerden function of every base.

\begin{remark}
\label{pow2}
Let $q=2^K$ with an integer $K\ge1$. Then
\begin{equation}
 \tau_q(x)=\sum_{i=0}^{K-1}2^{K-1-2i}\dist(2^ix,\Z)~\text{for all real}~x.
 \label{tautwo}
\end{equation}
Indeed, both sides are continuous, have period $1$, and are zero at $0$.
Let $0\le k\le q-1$, and let $b_1,\ldots,b_K$ be the binary digits of
$k$, written with $K$ digits and starting with the digit of $2^{K-1}$. On
$[k/q,(k+1)/q]$ the function $\dist(2^ix,\Z)$ with $0\le i\le K-1$ is
linear with slope $2^i(1-2b_{i+1})$, so the right side of \eqref{tautwo}
is linear with slope
$\sum_{i=0}^{K-1}2^{K-1-i}(1-2b_{i+1})=2^K-1-2k$, which is the slope
$q-1-2k$ of $\tau_q$ by \eqref{tau}. We insert \eqref{tautwo} into
\eqref{T} and write $m=Kj+i$ for the term with the indices $j$ and $i$.
This gives
\[
 T_q(x)=\sum_{m=0}^{\infty}2^{-m}\gamma_m\dist(2^mx,\Z),
\]
where $\gamma_m=2^{K-1-i}$ when $i$ is the remainder of $m$ modulo $K$. Thus
$1\le\gamma_m\le2^{K-1}$ for every $m$, and
$\sum_{m=0}^{Kn-1}\gamma_m^2=n\sum_{i=0}^{K-1}4^{K-1-i}=n(q^2-1)/3$ for every
$n\ge1$.
\end{remark}

Let $q$ be a power of $2$ in this paragraph. Allaart~\cite[Theorem~5.2]{Allaart2009}
proved a law of the iterated logarithm for the series
$\sum_{m=0}^{\infty}2^{-m}\gamma_m\dist(2^mx,\Z)$ with bounded numbers
$\gamma_m\ge0$ whose squares have an infinite sum. K\^ono~\cite[Theorem~5]{Kono1987}
had stated this theorem earlier, and Allaart~\cite[Remark~5.4]{Allaart2009}
pointed out a gap in K\^ono's proof. For $0<h\le q^{-16}$ put
\[
 \psi(h)=\sqrt{\frac{2(q^2-1)}3\log_q\frac1h\cdot\log\log\log_q\frac1h}.
\]
By Remark~\ref{pow2} their theorem applies to $T_q$, and the last statement
of the remark shows that the quotient of Allaart's normalization by
$|h|\psi(|h|)$ tends to $1$. Thus, for almost every $x$, their theorem
gives the upper limit $1$ and the lower limit $-1$ of
$\Delta(x,x+h)/\psi(|h|)$ as $h\to0$ from both sides together. Allaart's
proof~\cite[page~253]{Allaart2009} treats $h\downarrow0$ and
$h\uparrow0$ separately, so it gives these two limits for each side,
which is Theorem~\ref{liltheorem} for these $q$. In the same way, the
central limit theorem of Nakano~\cite[Corollary~2.2(i)]{Nakano2026} for
weighted Takagi--van~der~Waerden functions gives
Theorem~\ref{clttheorem}. Nakano assumed that the sum of the squares of
the first $n$ coefficients tends to infinity, and that the square of the
$n$th coefficient is at most a constant times this sum raised to the
power $1-\delta$, for some $\delta\in(0,1]$. Since
$1\le\gamma_m\le2^{K-1}$ for every $m$, both assumptions hold for $T_q$,
the second one with $\delta=1$. Nakano divided the slopes by the square
root of a nonincreasing function of $|h|$ whose value at $2^{-n}$ is the
sum of the squares of the first $n$ coefficients, and by the last
statement of Remark~\ref{pow2} the quotient of this square root by
$\sigma(|h|)$ tends to $1$, where
$\sigma(h)=\sqrt{\frac{q^2-1}3\log_q(1/h)}$ for $0<h<1$.

We use the Hausdorff dimension as Fraser~\cite[Section~1.2]{Fraser2021}
defined it.

\begin{definition}
\label{hausdorff}
Let $F$ be a subset of $\R$ or of $\R^2$, and let $s\ge0$. For a set $U$
let $\operatorname{diam}U$ be its diameter. The Hausdorff measure of
dimension $s$ of $F$ is the limit, as $r\downarrow0$, of the infimum
of the sums $\sum_{i=1}^{\infty}(\operatorname{diam}U_i)^s$ over all sequences of
sets $U_1,U_2,\ldots$ with $\operatorname{diam}U_i\le r$ that cover $F$.
This limit exists in $[0,\infty]$, since the infimum does not decrease as
$r$ decreases. The Hausdorff dimension of $F$ is the
infimum of the numbers $s\ge0$ for which this measure is $0$.
\end{definition}

The lower bounds for Hausdorff dimensions in Propositions~\ref{infdim}
and~\ref{onesided}, in Example~\ref{ex:nineblocks}, and in
Theorem~\ref{heightdim}(b) come from the
following form of the mass distribution principle, in which intervals of
level $k\ell$ take the place of balls (see
Fraser~\cite[Lemma~3.4.2]{Fraser2021}).

\begin{lemma}
\label{massdist}
Let $\ell\ge1$ be an integer and $s>0$, let $F\subseteq[0,1]$ be a
compact set, and let $\mu$ be a Borel probability measure with
$\mu(F)=1$. Suppose that the numbers $q^{k\ell s}\mu(J)$ are bounded,
where $k\ge0$ and $J$ runs through the intervals of level $k\ell$.
Then the
Hausdorff measure of dimension $s$ of $F$ is positive, and $F$ has
Hausdorff dimension at least $s$.
\end{lemma}

\begin{proof}
Let $M$ be an upper bound for these numbers, and let $k\ge0$. We have $\mu([0,1])=1$, and $[0,1]$ is the union of the intervals of level
$k\ell$. A closed set $U$ with $\operatorname{diam}U<q^{-k\ell}$ meets at most
two of them. Hence every such set satisfies $\mu(U)\le2M\cdot q^{-k\ell s}$. In
particular $\mu(U)=0$ if
$\operatorname{diam}U=0$, and if
$q^{-(k+1)\ell}\le\operatorname{diam}U<q^{-k\ell}$, then
\[
 \mu(U)\le2M\cdot q^{\ell s}\,q^{-(k+1)\ell s}
 \le2M\cdot q^{\ell s}(\operatorname{diam}U)^s.
\]
If sets $U_1,U_2,\ldots$ of diameter less than $1$ cover $F$, then their
closures $\overline{U_i}$ cover $F$ and have the same diameters, so
\[
 1=\mu(F)\le\sum_{i=1}^{\infty}\mu(\overline{U_i})
 \le2M\cdot q^{\ell s}\sum_{i=1}^{\infty}(\operatorname{diam}U_i)^s.
\]
Hence the Hausdorff measure of dimension $s$ of $F$ is at least
$q^{-\ell s}/(2M)$. Since $(\operatorname{diam}U)^t\ge(\operatorname{diam}U)^s$
when $0\le t<s$ and $\operatorname{diam}U\le1$, the Hausdorff measures of
all dimensions $t<s$ of $F$ are positive as well, and $F$ has Hausdorff
dimension at least $s$.
\end{proof}

In this section $|Q|$ denotes the Lebesgue measure of a measurable set
$Q\subseteq[0,1)$. For a point $x\in(0,1)$ that is not a $q$-adic rational
we use the notation of Section~\ref{sec:inf}: $x$ has the digit sequence
$\xi_1,\xi_2,\ldots$ and the digit slopes
$d_n(x)=\sum_{j=1}^{n}(q-1-2\xi_j)$, the interval of level $n$ containing
$x$ is $[a_n,c_n]$ with $a_n=k_n/q^n$,
$c_n=a_n+q^{-n}$, and $k_n=\lfloor q^nx\rfloor$, the number $t_n=q^nx-k_n$ has
the digits $\xi_{n+1},\xi_{n+2},\ldots$, and for $y\in[0,1]$, $y\ne x$, the
number $n(y)$ is the largest $n$ with $y\in[a_n,c_n]$. In addition, for
$n\ge0$ let $\ell_n(x)$ be the number of consecutive digits equal to
$\xi_{n+1}$, counted from position $n+1$ on, if $\xi_{n+1}\in\{0,q-1\}$,
and $\ell_n(x)=0$ otherwise. For $q=2$ every digit is $0$ or $q-1$, so
that $\ell_n(x)\ge1$ for every $n$.

\begin{lemma}
\label{localrun}
Let $x\in(0,1)$ be a point that is not a $q$-adic rational, and let
$n\ge0$.
\begin{enumerate}
\item[(a)] If $y\in[0,1]$ and $n(y)=n$, then
 $|\Delta(x,y)-d_n(x)|\le(q-1)(\ell_{n+1}(x)+2)+q(q+1)/4$.
\item[(b)] $d_n(x)-(q-1)(\ell_n(x)+1)-q(q+1)/4\le\Delta(x,c_n)\le d_n(x)$ and
 $d_n(x)\le\Delta(x,a_n)\le d_n(x)+(q-1)(\ell_n(x)+1)+q(q+1)/4$.
\end{enumerate}
For $q=3$ the bound in (a) is $2\ell_{n+1}(x)+7$, and the bounds in (b)
are $d_n(x)-2\ell_n(x)-5$ and $d_n(x)+2\ell_n(x)+5$.
\end{lemma}

\begin{proof}
(a) Write $y=(k_n+u)/q^n$ with $u\in[0,1]$, and let $i=\xi_{n+1}$. The
number $qt_n-i\in(0,1)$ has the digits $\xi_{n+2},\xi_{n+3},\ldots$. If
$\xi_{n+2}=0$, then the first nonzero digit of $qt_n-i$ stands at
position $\ell_{n+1}(x)+1$, so $qt_n-i\ge q^{-\ell_{n+1}(x)-1}$, and
$1-(qt_n-i)>1-1/q\ge1/q$. If $\xi_{n+2}=q-1$, then the same holds with
$qt_n-i$ and $1-(qt_n-i)$ exchanged, because the digits of $1-(qt_n-i)$
are $q-1-\xi_{n+2},q-1-\xi_{n+3},\ldots$. If $\xi_{n+2}$ is
neither $0$ nor $q-1$, then
$qt_n-i\in[1/q,1-1/q]$ and $\ell_{n+1}(x)=0$. In all cases $qt_n-i$ lies at
distance at least $q^{-\ell_{n+1}(x)-1}$ from $0$ and from $1$, so $t_n$ lies
at distance at least $q^{-\ell_{n+1}(x)-2}$ from both endpoints of
$[i/q,(i+1)/q]$. Since $n(y)=n$, the point $y$ does not lie in
$[a_{n+1},c_{n+1}]$, so $u\notin[i/q,(i+1)/q]$ and
$v=|u-t_n|\ge q^{-\ell_{n+1}(x)-2}$. By Proposition~\ref{incr}(a), the last
term in \eqref{sub} is at most $T_q(v)/v$ in absolute value, and by
\eqref{theta} with $\lfloor\log_q(1/v)\rfloor\le \ell_{n+1}(x)+2$ this is at
most $(q-1)(\ell_{n+1}(x)+2)+q(q+1)/4$.

(b) By \eqref{ends}, $\Delta(x,c_n)=d_n(x)-T_q(1-t_n)/(1-t_n)$ and
$\Delta(x,a_n)=d_n(x)+T_q(t_n)/t_n$, and both fractions are nonnegative.
The number $t_n$ has the digits $\xi_{n+1},\xi_{n+2},\ldots$. If we
replace $qt_n-i$, $\xi_{n+2}$, and $\ell_{n+1}(x)$ by $t_n$, $\xi_{n+1}$,
and $\ell_n(x)$, the argument of (a) shows that $t_n$ and $1-t_n$
are at least $q^{-\ell_n(x)-1}$. Hence \eqref{theta} bounds both
fractions by $(q-1)(\ell_n(x)+1)+q(q+1)/4$.
\end{proof}

We regard the digits $\xi_1,\xi_2,\ldots$ as functions of $x\in[0,1)$. For
every $n$ and every word of length $n$ over $\{0,1,\ldots,q-1\}$, the set
of $x$ whose first $n$ digits form this word is an interval of length
$q^{-n}$. Hence, with respect to Lebesgue measure on $[0,1)$, the digits are
independent random variables, each of which takes the values
$0,1,\ldots,q-1$ with probability $1/q$, and $d_n(x)$ is a sum of $n$
independent, identically distributed random variables $q-1-2\xi_j$.
Each of them has mean $0$, because the steps $q-1-2\xi$ of the digits
$\xi$ and $q-1-\xi$ add up to $0$, and it has the variance
\begin{equation}
 \frac1q\sum_{\xi=0}^{q-1}(q-1-2\xi)^2
 =(q-1)^2-2(q-1)^2+\frac{2(q-1)(2q-1)}3=\frac{q^2-1}3,
 \label{stepvar}
\end{equation}
where we used $\sum_{\xi=0}^{q-1}\xi=q(q-1)/2$ and
$\sum_{\xi=0}^{q-1}\xi^2=q(q-1)(2q-1)/6$. This variance is $1$ for $q=2$
and $8/3$ for $q=3$. Hartman and Wintner~\cite{HartmanWintner1941} proved
the law of the iterated
logarithm for independent variables with mean $0$ whose variances add up
to at least a constant multiple of $n$ and whose tail probabilities are
bounded, uniformly in $n$, by a constant times the tail probabilities of
one distribution with finite second moment. We state the case of
identically distributed variables with finite positive variance
$\sigma^2$, and we write $\sqrt{2\sigma^2n\log\log n}$ in place of their
normalization $\sqrt{2n\sigma^2\log\log(n\sigma^2)}$, since the ratio of
the two tends to $1$.

\begin{theorem}[Hartman and Wintner {\cite[Section~2]{HartmanWintner1941}}]
\label{lil}
Let $Z_1,Z_2,\ldots$ be independent, identically distributed random
variables with mean $0$ and finite positive variance $\sigma^2$, and put
$S_n=Z_1+\cdots+Z_n$. Then, with probability $1$,
\[
 \limsup_{n\to\infty}\frac{S_n}{\sqrt{2\sigma^2n\log\log n}}=1
 ~\text{and}~
 \liminf_{n\to\infty}\frac{S_n}{\sqrt{2\sigma^2n\log\log n}}=-1.
\]
\end{theorem}

Durrett~\cite[Theorem~3.4.1]{Durrett2019} stated the central limit
theorem for the variables $S_n/(\sigma\sqrt n)$, and we state the
equivalent form for the probabilities that they are at most $t$.

\begin{theorem}[central limit theorem {\cite[Theorem~3.4.1]{Durrett2019}}]
\label{clt}
Under the hypotheses of Theorem~\ref{lil}, for every real $t$,
\[
 \lim_{n\to\infty}\mathbb{P}\bigl(S_n\le t\sigma\sqrt n\bigr)=\Phi(t),
 ~\text{where}~
 \Phi(t)=\frac1{\sqrt{2\pi}}\int_{-\infty}^te^{-v^2/2}\,dv.
\]
\end{theorem}

Applied to $Z_j=q-1-2\xi_j$, for which $S_n=d_n(x)$ and
$\sigma^2=(q^2-1)/3$ by \eqref{stepvar}, Theorem~\ref{lil} says that for
almost every $x\in[0,1)$
\begin{equation}
 \limsup_{n\to\infty}\frac{d_n(x)}{\varphi(n)}=1
 ~\text{and}~
 \liminf_{n\to\infty}\frac{d_n(x)}{\varphi(n)}=-1,
 ~\text{where}~
 \varphi(n)=\sqrt{\frac{2(q^2-1)}3\,n\log\log n}.
 \label{lildigits}
\end{equation}
For $q=3$ we have $\varphi(n)=\sqrt{\frac{16}3\,n\log\log n}$, and for
$q=2$ we have $\varphi(n)=\sqrt{2n\log\log n}$.

\begin{lemma}
\label{aeruns}
For almost every $x\in(0,1)$ there is an $n_0$ such that
$\ell_n(x)\le2\log_qn+1$ for all $n\ge n_0$.
\end{lemma}

\begin{proof}
Let $n\ge1$ and $k\ge0$. If $\ell_n(x)\ge k$, then the digits
$\xi_{n+1},\ldots,\xi_{n+k}$ are all $0$ or all $q-1$, so
$|\{x:\ell_n(x)\ge k\}|\le2q^{-k}$. Let $Q_n$ be the set of $x$ with
$\ell_n(x)\ge\lceil2\log_qn\rceil$, so that $|Q_n|\le2n^{-2}$. For every
$N$, the set of $x$ that lie in infinitely many of the sets $Q_n$ is
contained in $\bigcup_{n=N}^{\infty}Q_n$, whose measure is at most
$2\sum_{n=N}^{\infty}n^{-2}$ and tends to $0$ as $N\to\infty$. Hence this
set has measure zero, and every other $x$ satisfies
$\ell_n(x)<\lceil2\log_qn\rceil\le2\log_qn+1$ for all large $n$.
\end{proof}

\begin{theorem}
\label{liltheorem}
For $0<h\le q^{-16}$ put
\[
 \psi(h)=\sqrt{\frac{2(q^2-1)}3\log_q\frac1h\cdot\log\log\log_q\frac1h}.
\]
For almost every $x\in\R$,
\[
 \limsup_{h\downarrow0}\frac{\Delta(x,x+h)}{\psi(h)}=1,~
 \liminf_{h\downarrow0}\frac{\Delta(x,x+h)}{\psi(h)}=-1,~
 \limsup_{h\uparrow0}\frac{\Delta(x,x+h)}{\psi(-h)}=1,~
 \liminf_{h\uparrow0}\frac{\Delta(x,x+h)}{\psi(-h)}=-1.
\]
For $q=3$ the factor $2(q^2-1)/3$ is $16/3$, and for $q=2$ it is $2$.
\end{theorem}

\begin{proof}
Since $T_q$ has period $1$, it suffices to consider $x\in(0,1)$. We exclude
the $q$-adic rationals, the set of measure zero in Lemma~\ref{aeruns}, and
the set of measure zero on which \eqref{lildigits} fails, and we fix a
point $x$ outside these sets. Let $\varphi$ be as in \eqref{lildigits}.
Then $\psi(q^{-n})=\varphi(n)$ for $n\ge16$, $\psi$ is decreasing, and
$\varphi(n+m)/\varphi(n)\to1$ when $m=O(\log n)$. Let $0<\varepsilon<1$, and
let $n_0\ge16$ be such that for all $n\ge n_0$
\[
 -(1+\varepsilon)\varphi(n)\le d_n(x)\le(1+\varepsilon)\varphi(n)
 ~\text{and}~
 \ell_n(x)\le2\log_qn+1.
\]

Let $y\in[0,1]$, $y\ne x$, be close to $x$, and put $n=n(y)$ and
$h=|y-x|$. Then $n\to\infty$ as $y\to x$, and $h\le q^{-n}$ because $x$
and $y$ lie in $[a_n,c_n]$; hence $\psi(h)\ge\varphi(n)$. For $n\ge n_0$,
Lemma~\ref{localrun}(a) gives
\[
 \Delta(x,y)\le d_n(x)+(q-1)(\ell_{n+1}(x)+2)+q(q+1)/4
 \le(1+\varepsilon)\varphi(n)+(q-1)\bigl(2\log_q(n+1)+3\bigr)+q(q+1)/4,
\]
and in the same way
$\Delta(x,y)\ge-(1+\varepsilon)\varphi(n)-(q-1)(2\log_q(n+1)+3)-q(q+1)/4$.
Since $\log n/\varphi(n)\to0$, the upper limit of $\Delta(x,y)/\psi(h)$ as
$y\to x$ from either side is at most $1+\varepsilon$, and the lower limit
is at least $-1-\varepsilon$.

By \eqref{lildigits} there are infinitely many $n\ge n_0$ with
$d_n(x)\ge(1-\varepsilon)\varphi(n)$. For such $n$ we take $y=c_n$, and
Lemma~\ref{localrun}(b) gives
\[
 \Delta(x,c_n)\ge d_n(x)-(q-1)(\ell_n(x)+1)-q(q+1)/4
 \ge(1-\varepsilon)\varphi(n)-(q-1)(2\log_qn+2)-q(q+1)/4.
\]
The proof of Lemma~\ref{localrun}(b) shows that
$1-t_n\ge q^{-\ell_n(x)-1}$, so $h=c_n-x=q^{-n}(1-t_n)$ satisfies
$q^{-n-\ell_n(x)-1}\le h\le q^{-n}$, and therefore
$\varphi(n)\le\psi(h)\le\varphi(n+2\log_qn+2)=(1+o(1))\varphi(n)$. Hence
$\Delta(x,c_n)/\psi(c_n-x)\ge1-2\varepsilon$ for all large $n$ in this
sequence, and since $c_n\downarrow x$, the upper limit as $h\downarrow0$ is
at least $1-2\varepsilon$. For the infinitely many $n$ with
$d_n(x)\le-(1-\varepsilon)\varphi(n)$ we have
$\Delta(x,c_n)\le d_n(x)\le-(1-\varepsilon)\varphi(n)$ and the same bounds on
$\psi(c_n-x)$, so the lower limit as $h\downarrow0$ is at most
$-1+2\varepsilon$. On the left we use $y=a_n$, which tends to $x$ from
below. Since $t_n\ge q^{-\ell_n(x)-1}$, we have
$q^{-n-\ell_n(x)-1}\le x-a_n\le q^{-n}$, and Lemma~\ref{localrun}(b) gives
$d_n(x)\le\Delta(x,a_n)\le d_n(x)+(q-1)(\ell_n(x)+1)+q(q+1)/4$. The same
choices of $n$ give the upper limit at least $1-2\varepsilon$ and the
lower limit at most $-1+2\varepsilon$ as $h\uparrow0$, and letting
$\varepsilon\to0$ completes the proof.
\end{proof}

\begin{corollary}
\label{aedini}
For almost every $x\in\R$,
\[
 D^+T_q(x)=D^-T_q(x)=+\infty~\text{and}~D_+T_q(x)=D_-T_q(x)=-\infty.
\]
In particular, the set of points at which $T_q$ has a derivative from at
least one side, finite or infinite, has Lebesgue measure zero.
\end{corollary}

\begin{proof}
At every $x$ for which the four equations of Theorem~\ref{liltheorem} hold,
$\Delta(x,x+h)$ is unbounded above and below as $h\downarrow0$ and as
$h\uparrow0$, because $\psi(h)\to\infty$. Thus the four Dini derivatives
have the stated values, and since $D^+T_q(x)\ne D_+T_q(x)$ and
$D^-T_q(x)\ne D_-T_q(x)$, the function $T_q$ has no derivative at $x$
from either side, finite or infinite.
\end{proof}

For Takagi's function, Allaart and
Kawamura~\cite[page~659]{AllaartKawamura2010} noted that, by the law of
the iterated logarithm, it has no infinite derivative from either side at
almost every point.

\begin{theorem}
\label{clttheorem}
For $0<h<1$ put $\sigma(h)=\sqrt{\frac{q^2-1}3\log_q(1/h)}$. Then for every
real $t$,
\[
 \lim_{h\downarrow0}\bigl|\{x\in[0,1):\Delta(x,x+h)\le t\sigma(h)\}\bigr|
 =\Phi(t).
\]
For $q=3$ we have $\sigma(h)=\sqrt{\frac83\log_3(1/h)}$, and for $q=2$ we
have $\sigma(h)=\sqrt{\log_2(1/h)}$.
\end{theorem}

\begin{proof}
Let $0<h<q^{-2}$ and $n=\lfloor\log_q(1/h)\rfloor$, so that $n\ge2$,
$q^{-n-1}<h\le q^{-n}$, and $\sigma(h)/\sqrt{(q^2-1)n/3}\to1$ as
$h\downarrow0$.
Let $x\in(0,1-h)$ be a point that is not a $q$-adic rational, let
$y=x+h$, and let $n'=n(y)$. Since $c_{n+1}-x\le q^{-n-1}<h$, the point $y$
lies to the right of $c_{n+1}$, so $n'\le n$. Let $j=j_n(x)$ be the number
of digits equal to $q-1$ at the end of $\xi_1,\ldots,\xi_n$, that is, the
largest $j\le n$ with $\xi_{n-j+1}=\cdots=\xi_n=q-1$. If $j<n$, put
$n''=n-j-1$. Then $\xi_{n''+1}\ne q-1$, so
$t_{n''}\le(\xi_{n''+1}+1)/q\le1-1/q$
and $c_{n''}-x=q^{-n''}(1-t_{n''})\ge q^{-n''-1}=q^{-n+j}\ge h$. Hence
$y\le c_{n''}$ and $n'\ge n''$. In all cases $n-n'\le j+1$, and therefore
$|d_{n'}(x)-d_n(x)|\le(q-1)(n-n')\le(q-1)(j+1)$, because every step of the
digit slopes is at most $q-1$ in absolute value. Together with
Lemma~\ref{localrun}(a),
\[
 |\Delta(x,x+h)-d_n(x)|\le(q-1)\bigl(j_n(x)+\ell_{n'+1}(x)+3\bigr)+q(q+1)/4.
\]
The set of $x$ with $j_n(x)\ge k$ has measure at most $q^{-k}$. If
$j_n(x)<k$, then $n'\ge n-k$, so that $\ell_{n'+1}(x)\ge k'$ implies
$\ell_{i+1}(x)\ge k'$ for some $i$ with $\max\{0,n-k\}\le i\le n$. By the
proof of Lemma~\ref{aeruns}, the set of $x$ for which such an $i$ exists
has measure at most $2(k+1)q^{-k'}$. Hence, for all integers $k,k'\ge1$
and all
$0<h<q^{-2}$,
\begin{equation}
 \bigl|\{x:|\Delta(x,x+h)-d_n(x)|>(q-1)(k+k'+3)+q(q+1)/4\}\bigr|
 \le q^{-k}+2(k+1)q^{-k'},
 \label{Etight}
\end{equation}
where $x$ runs through the points of $(0,1-h)$ that are not $q$-adic rationals.

Now fix $t$ and $\delta>0$. By Theorem~\ref{clt} applied to the variables
$Z_i=q-1-2\xi_i$, which have variance $(q^2-1)/3$ by \eqref{stepvar} and
the sum $Z_1+\cdots+Z_n=d_n(x)$, we have
$|\{x:d_n(x)\le c\sqrt{(q^2-1)n/3}\}|\to\Phi(c)$ for every real $c$.
Since $\sigma(h)/\sqrt{(q^2-1)n/3}\to1$ and $\Phi$ is continuous and
nondecreasing, also $|\{x:d_n(x)\le c\sigma(h)\}|\to\Phi(c)$ as
$h\downarrow0$. If
$d_n(x)\le(t-\delta)\sigma(h)$ and $|\Delta(x,x+h)-d_n(x)|\le\delta\sigma(h)$,
then $\Delta(x,x+h)\le t\sigma(h)$, and if $\Delta(x,x+h)\le t\sigma(h)$
and $|\Delta(x,x+h)-d_n(x)|\le\delta\sigma(h)$, then
$d_n(x)\le(t+\delta)\sigma(h)$. Given $\eta>0$, choose $k$ and then $k'$ so
that the right side of \eqref{Etight} is less than $\eta$; since
$\sigma(h)\to\infty$, we have $\delta\sigma(h)\ge(q-1)(k+k'+3)+q(q+1)/4$ for
all small
$h$, and then $|\{x:|\Delta(x,x+h)-d_n(x)|>\delta\sigma(h)\}|<\eta$. The
excluded points, the $q$-adic rationals and the interval $[1-h,1)$, have measure at most $h$. Hence
\[
 \Phi(t-\delta)-\eta
 \le\liminf_{h\downarrow0}\bigl|\{x:\Delta(x,x+h)\le t\sigma(h)\}\bigr|
 \le\limsup_{h\downarrow0}\bigl|\{x:\Delta(x,x+h)\le t\sigma(h)\}\bigr|
 \le\Phi(t+\delta)+\eta,
\]
and letting $\eta\to0$ and $\delta\to0$ gives the claim.
\end{proof}

\begin{example}
\label{ex:clt}
Theorem~\ref{clttheorem} with $t=1$ and $t=-1$ shows that the Lebesgue
measure of the set of the points $x\in[0,1)$ with
$-\sigma(h)<\Delta(x,x+h)\le\sigma(h)$ tends to $\Phi(1)-\Phi(-1)$ as
$h\downarrow0$, and numerically $\Phi(1)-\Phi(-1)\approx0.6827$. With
$t=0$, it shows that the measure of the set of the points with
$\Delta(x,x+h)\le0$ tends to $\Phi(0)=1/2$.
\end{example}

For $h=q^{-n}$ and $x=m/q^N$ with $N\ge n$, the number $q^Nh$ is an
integer, so by \eqref{eq:tail} $T_q(x+h)-T_q(x)=F_N(x+h)-F_N(x)$, and by
\eqref{slope} and the periodicity of $F_N$ the slope $\Delta(x,x+h)$ is
$(q-1)N$ minus twice the mean of $s_q(i)$ over the $q^{N-n}$ integers $i$
with $m\le i<m+q^{N-n}$, where $i$ is replaced by $i-q^N$ when
$i\ge q^N$. Figure~\ref{slopedistpic} shows the distribution of these
exact values for $q=3$,
$h=3^{-12}$, and $N=14$, for $q=4$, $h=4^{-9}$, and $N=11$, and for $q=5$,
$h=5^{-8}$, and $N=10$, at all points $m/q^N$ with $0\le m<q^N$, in the
bins $[2b-1,2b+1)$ with integers $b$. The curves are the normal densities
with the variances $\sigma(h)^2=32$, $45$, and $64$ of
Theorem~\ref{clttheorem}.

\begin{figure}[H]
\centering
\includegraphics{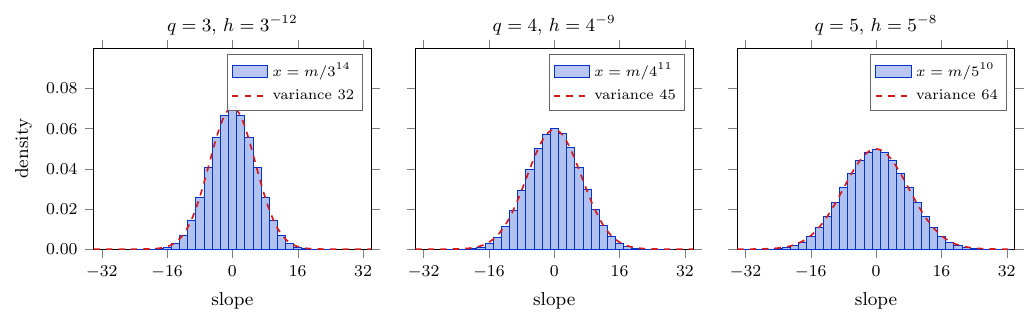}
\caption{Theorem~\ref{clttheorem} for $q=3$, $4$, and $5$: the
distribution of the slopes $\Delta(x,x+h)$ for $h=3^{-12}$ at the points
$x=m/3^{14}$, for $h=4^{-9}$ at the points $x=m/4^{11}$, and for
$h=5^{-8}$ at the points $x=m/5^{10}$, in bins of width $2$, and the
normal densities of the theorem.}
\label{slopedistpic}
\end{figure}

We use the numbers $C_n$ of Lemma~\ref{trinom}, which count the words of
length $n$ over
$\{0,1,\ldots,q-1\}$ with digit slope $d_n=0$. In the rest of this
section, the slope sum of a word $w_1\ldots w_m$ over
$\{0,1,\ldots,q-1\}$, or of some of its letters, is the sum of the steps
$q-1-2w_i$ of these letters. For $q=3$ it is twice the number of zeros
minus twice the number of twos.

\begin{proposition}
\label{infdim}
The sets $\{x\in\R:T_q'(x)=+\infty\}$ and $\{x\in\R:T_q'(x)=-\infty\}$ are
dense in $\R$, have Lebesgue measure zero, and have Hausdorff
dimension $1$.
\end{proposition}

\begin{proof}
By the symmetry $T_q(1-x)=T_q(x)$, we have $T_q'(1-x)=-T_q'(x)$ whenever one of
the two derivatives exists, so it suffices to treat $S=\{x:T_q'(x)=+\infty\}$.
A point whose digit sequence ends with the period $0,0,1$ has period
length $3$ and digit sum $1$ over a period, and $2<3(q-1)$. By
Corollary~\ref{rat}(b), every such point lies in $S$, and these points
are dense in $[0,1]$. Since $T_q$ has period $1$, the set $S$ is dense in
$\R$. By
Corollary~\ref{aedini}, $S$ has measure zero. For the dimension, let
$m\ge5$ and let $\mathcal W_m$ be the set of words $w_1\ldots w_m$ over
$\{0,1,\ldots,q-1\}$
with $w_1=0$ and $w_m=1$ in which the letters $w_2,\ldots,w_{m-1}$ have a
positive slope sum. Replacing every letter $w$ by $q-1-w$ changes the
sign of the slope sum, so there are as many words of
length $m-2$ with a positive slope sum as with a negative one, and
the number $\#\mathcal W_m$ of words in $\mathcal W_m$ is
$(q^{m-2}-C_{m-2})/2$. For $n\ge1$ we have $C_n\le q^{n-1}$, because the
first $n-1$ letters of a word with slope sum $0$ determine its
last letter. Hence
$\#\mathcal W_m\ge(q-1)q^{m-3}/2\ge q^{m-4}\ge q$. The letters $w_1=0$ and
$w_m=1$ have the steps $q-1$ and $q-3$, so every word of $\mathcal W_m$ has
a slope sum of at least $2q-3\ge1$. Let
$\mathcal E_m$ be the set of $x\in[0,1)$ whose digit sequence consists of infinitely
many words of $\mathcal W_m$ written one after another. A
point $x\in \mathcal E_m$ is not a
$q$-adic rational, because every word contains the digit $0$, which is
not $q-1$, and ends with the nonzero digit $1$. Its runs
of zeros have length at most $m-1$, because every word ends
with $1$. Moreover, $d_{km}(x)\ge k$ for every $k$, so
$d_n(x)\ge k-(q-1)m$
for $km\le n<(k+1)m$ and $d_n(x)\to+\infty$. Hence $\mathcal E_m\subseteq S$ by
Corollary~\ref{rat}(a).

We show that the Hausdorff dimension of $\mathcal E_m$ is at least
$s=\log\#\mathcal W_m/(m\log q)$; then the dimension of $S$ is at least $s$, and
$s\ge(m-4)/m$ tends to $1$ as $m\to\infty$. Choose a sequence of words of
$\mathcal W_m$ at random, with the choices at different positions independent
and each word of $\mathcal W_m$ having probability $1/\#\mathcal W_m$.
Let $\mu$ be the image of this product measure under the map that sends a
sequence of words of $\mathcal W_m$ to the point with that digit sequence. This map is
continuous, since sequences with the same first $k$ words are sent to
points at distance at most $q^{-km}$. Hence its image $\mathcal E_m$ is
compact, and $\mu$ is a Borel probability measure with
$\mu(\mathcal E_m)=1$. Let $k\ge0$, and let $J$ be an interval of
level $km$. The points of $\mathcal E_m$ in $J$ are not $q$-adic rationals, so they have the same first $km$ digits. If there are such
points, these digits form $k$ words of $\mathcal W_m$, and
$\mu(J)=(\#\mathcal W_m)^{-k}=q^{-kms}$; otherwise $\mu(J)=0$. By
Lemma~\ref{massdist} with $\ell=m$, the Hausdorff dimension of
$\mathcal E_m$ is at least $s$.
\end{proof}

The proof of Theorem~\ref{slopegrowth} below uses the following one sided
form of Lemma~\ref{localrun}(a), in which only the runs of the digit $q-1$
of $x$ enter the bound for the points to the right of $x$.

\begin{lemma}
\label{siderun}
Let $x\in(0,1)$ be a point that is not a $q$-adic rational, let $n\ge0$,
and let $j\ge0$ be the number of consecutive digits $q-1$ of $x$ counted from
position $n+2$ on, so that $j=0$ if $\xi_{n+2}\ne q-1$. If $x<y\le1$ and
$n(y)=n$, then
\[
 q^{-n-j-2}<y-x\le q^{-n}~\text{and}~
 \bigl|\Delta(x,y)-d_n(x)\bigr|\le(q-1)(j+1)+q(q+1)/4.
\]
For $q=3$ the last bound is $2j+5$.
\end{lemma}

\begin{proof}
Write $y=(k_n+u)/q^n$ with $u\in[0,1]$, and let $i=\xi_{n+1}$, so that
$i/q<t_n<(i+1)/q$. Since $n(y)=n$, the point $y$ does not lie in
$[a_{n+1},c_{n+1}]$, so $u\notin[i/q,(i+1)/q]$, and since $u>t_n$, we have
$u>(i+1)/q$. The number $qt_n-i$ has the digits
$\xi_{n+2},\xi_{n+3},\ldots$. If $j\ge1$, then the digits of $1-(qt_n-i)$
are $q-1-\xi_{n+2},q-1-\xi_{n+3},\ldots$, of which exactly the first $j$ are
zero, so $1-(qt_n-i)\ge q^{-j-1}$. If $j=0$, then $\xi_{n+2}\le q-2$, so
$qt_n-i<1-1/q$ and again $1-(qt_n-i)>q^{-j-1}$. Hence $v=u-t_n$ satisfies
\[
 1>v>\frac{i+1}q-t_n=\frac{1-(qt_n-i)}q\ge q^{-j-2},
\]
and $y-x=q^{-n}v$ gives the first claim. By \eqref{sub} and
Proposition~\ref{incr}(a), $|\Delta(x,y)-d_n(x)|\le T_q(v)/v$, and by
\eqref{theta} with $\lfloor\log_q(1/v)\rfloor\le j+1$ this is at most
$(q-1)(j+1)+q(q+1)/4$.
\end{proof}

Since $|d_n(x)|\le(q-1)n$, we write the limit of $d_n(x)/n$ in the next
theorem as $(q-1)\varrho$ with $-1\le\varrho\le1$. For $q=3$ the condition
$d_n(x)/n\to(q-1)\varrho$ means that the number of zeros minus the number
of twos among the first $n$ digits of $x$, divided by $n$, tends to
$\varrho$.

\begin{theorem}
\label{slopegrowth}
Let $-1\le\varrho\le1$, and let $x\in(0,1)$ be a point that is not a
$q$-adic rational.
\begin{enumerate}
\item[(a)] If $-1<\varrho<1$, then $d_n(x)/n\to(q-1)\varrho$ if and only if
 \[
  \lim_{h\to0}\frac{\Delta(x,x+h)}{\log_q(1/|h|)}=(q-1)\varrho.
 \]
\item[(b)] If $-1<\varrho\le1$, then $d_n(x)/n\to(q-1)\varrho$ if and only
 if $\Delta(x,x+h)/\log_q(1/|h|)\to(q-1)\varrho$ as $h\downarrow0$. If
 $-1\le\varrho<1$, then $d_n(x)/n\to(q-1)\varrho$ if and only if
 $\Delta(x,x+h)/\log_q(1/|h|)\to(q-1)\varrho$ as $h\uparrow0$.
\item[(c)] If $x=\sum_{k=1}^{\infty}q^{-2^k}$, then $d_n(x)/n\to q-1$, but
 $\Delta(x,x+h)/\log_q(1/|h|)$ has no limit as $h\uparrow0$.
\end{enumerate}
\end{theorem}

\begin{proof}
Let $-1<\varrho\le1$, and suppose first that $d_n(x)/n\to(q-1)\varrho$.
Fix $\varepsilon$ with $0<\varepsilon<1+\varrho$, and choose $N$ such that
$|d_i(x)-(q-1)\varrho i|\le(q-1)\varepsilon i$ for all $i\ge N$. If
$i\ge N$ and
the digits of $x$ in the positions $i+1,\ldots,i+j$ are all $q-1$, then
$(q-1)j=d_i(x)-d_{i+j}(x)
\le(q-1)(\varrho+\varepsilon)i-(q-1)(\varrho-\varepsilon)(i+j)$,
that is, $j\le2\varepsilon i/(1+\varrho-\varepsilon)$. Now let $x<y\le1$,
put $n=n(y)$, and let $j$ be as in Lemma~\ref{siderun}. Since
$n(y)\to\infty$ as $y\downarrow x$, we may assume that $n\ge N$, so that
$j\le2\varepsilon(n+1)/(1+\varrho-\varepsilon)$. By Lemma~\ref{siderun},
$n\le\log_q(1/(y-x))<n+j+2$ and
$|\Delta(x,y)-d_n(x)|\le(q-1)(j+1)+q(q+1)/4$. Since $|\varrho|\le1$, it
follows that
\begin{align*}
 \Bigl|\Delta(x,y)-(q-1)\varrho\log_q\frac1{y-x}\Bigr|
 &\le|\Delta(x,y)-d_n(x)|+|d_n(x)-(q-1)\varrho n|
 +(q-1)\Bigl(\log_q\frac1{y-x}-n\Bigr)\\
 &<(q-1)(2j+3)+q(q+1)/4+(q-1)\varepsilon n.
\end{align*}
Dividing by $\log_q(1/(y-x))\ge n$ shows that the upper limit of
$|\Delta(x,y)/\log_q(1/(y-x))-(q-1)\varrho|$ as $y\downarrow x$ is at most
$(q-1)\bigl(4\varepsilon/(1+\varrho-\varepsilon)+\varepsilon\bigr)$, and
letting $\varepsilon\to0$ gives
$\Delta(x,y)/\log_q(1/(y-x))\to(q-1)\varrho$.

Suppose now that $\Delta(x,y)/\log_q(1/(y-x))\to(q-1)\varrho$ as
$y\downarrow x$. Call $n\ge0$ good if $\xi_{n+1}\ne q-1$. There are infinitely
many good $n$, because $x$ is not a $q$-adic rational. For a good $n$ we have
$t_n<1-1/q$, so $q^{-n-1}<c_n-x\le q^{-n}$, and by \eqref{ends} and
\eqref{theta} with $\lfloor\log_q(1/(1-t_n))\rfloor=0$,
$0\le d_n(x)-\Delta(x,c_n)=T_q(1-t_n)/(1-t_n)\le q(q+1)/4$. Since
$c_n\downarrow x$ and $n\le\log_q(1/(c_n-x))<n+1$, the hypothesis gives
$d_n(x)=(q-1)\varrho n+o(n)$ as $n\to\infty$ through the good $n$. Now let
$n$ be large and not good, and let the positions $i+1,\ldots,i+M$ form
the run of the digit $q-1$ of $x$ that contains position $n+1$, so that
$i\le n<i+M$. The runs of the digit $q-1$ of $x$ are finite, so
$i\to\infty$ as $n\to\infty$, and we may assume that $i\ge1$. Then
$\xi_i\ne q-1$ and $\xi_{i+M+1}\ne q-1$, so $i-1$ and $i+M$ are good.
Since $d_{i+M}(x)=d_{i-1}(x)+(q-1-2\xi_i)-(q-1)M$, we obtain
$(q-1)M=-(q-1)\varrho(M+1)+O(1)+o(i+M)$, that is,
$(1+\varrho)M=O(1)+o(i+M)$, and as $1+\varrho>0$, this gives $M=o(i)$.
Moreover, $|d_n(x)-d_{i+M}(x)|\le(q-1)M$ and $|n-(i+M)|\le M$, so
$|d_n(x)-(q-1)\varrho n|\le|d_{i+M}(x)-(q-1)\varrho(i+M)|+2(q-1)M=o(n)$.
Therefore $d_n(x)/n\to(q-1)\varrho$. This proves the first statement
in (b).

The digits of $1-x$ are $q-1-\xi_1,q-1-\xi_2,\ldots$, so
$d_n(1-x)=-d_n(x)$, and $T_q(1-x+h)-T_q(1-x)=T_q(x-h)-T_q(x)$ gives
$\Delta(1-x,1-x+h)=-\Delta(x,x-h)$. It follows that the second statement
in (b) for $x$ and $\varrho$ is the first statement for $1-x$ and
$-\varrho$. Part (a) follows from (b), because the limit as $h\to0$ is
$(q-1)\varrho$ if and only if the limits as $h\downarrow0$ and as
$h\uparrow0$ are $(q-1)\varrho$.

(c) The digit of $x$ in position $j$ is $1$ if $j=2^k$ for some $k\ge1$,
and $0$ otherwise. Hence $x$ is not a $q$-adic rational,
$d_n(x)=(q-1)n-2\lfloor\log_2n\rfloor$ for $n\ge1$, and
$d_n(x)/n\to q-1$. The nonzero digits stand at the positions $n_k=2^k$,
and $d_{n_k}(x)=(q-1)2^k-2k$, so the numbers
$\Gamma_k=d_{n_k}(x)-(q-1)(n_{k+1}-n_k)+(q-1)\log_q(n_{k+1}-n_k)$ of
Section~\ref{sec:inf} are
\[
 \Gamma_k=(q-1)2^k-2k-(q-1)2^k+(q-1)\log_q2^k
 =-k\bigl(2-(q-1)\log_q2\bigr)\le(q-1)k.
\]
By \eqref{ykbound}, the points $y_k=a_{n_k}-(n_{k+1}-n_k)(x-a_{n_k})$
satisfy $\Delta(y_k,x)<\Gamma_k+q(q+1)/2+4\le(q-1)k+q(q+1)/2+4$ and
$0<x-y_k\le2q^{-2^k}$, so that $\log_q(1/(x-y_k))\ge2^k-1$. Hence the
quotient in (a) at the step $h=y_k-x$ is less than
$((q-1)k+q(q+1)/2+4)/(2^k-1)$, which tends to $0$ as $k\to\infty$. On the
other hand, let $n=2^k-1$ with $k\ge1$. Then
$\xi_{n+1}=1$, so $1/q<t_n<2/q$ and $q^{-n-1}<x-a_n<q^{-n}$, and
Lemma~\ref{localrun}(b) gives $\Delta(x,a_n)\ge d_n(x)$. Since
$d_n(x)=(q-1)(2^k-1)-2(k-1)>0$ and $\log_q(1/(x-a_n))<n+1$, the quotient
in (a) at
the step $h=a_n-x$ is greater than $d_n(x)/(n+1)$, which tends to $q-1$.
Consequently the quotient has no limit as $h\uparrow0$.
\end{proof}

For the point $x$ in part (c), $2-(q-1)\log_q2>0$ if and only if
$2^{q-1}<q^2$, that is, if and only if $q\le6$. For these bases
$\Gamma_k\to-\infty$, and the quotient is negative at the steps $h=y_k-x$
for all large $k$. For $q\ge7$ we have $\Gamma_k\to+\infty$, so that
$T_q'(x)=+\infty$ by Theorem~\ref{thm2}(c), although the quotient has no
limit as $h\uparrow0$.

\begin{example}
\label{ex:slopegrowth}
Let $q=3$. The point $1/26=(0.\overline{001})_3$ is not a ternary rational,
and for
$k\ge0$ we have $d_{3k}(1/26)=4k$, $d_{3k+1}(1/26)=4k+2$, and
$d_{3k+2}(1/26)=4k+4$. Thus $d_n(1/26)/n\to4/3$, which is $(q-1)\varrho$
with $\varrho=2/3$, and
Theorem~\ref{slopegrowth}(a) gives
\[
 \lim_{h\to0}\frac{\Delta(1/26,1/26+h)}{\log_3(1/|h|)}=\frac43.
\]
In particular $T_q'(1/26)=+\infty$. Now let $q=2$. The point
$1/7=(0.\overline{001})_2$ is not a dyadic rational, and
$d_{3k}(1/7)=k$ for $k\ge0$, so $d_n(1/7)/n\to1/3$. By
Theorem~\ref{slopegrowth}(a) with $\varrho=1/3$, the slopes of Takagi's
function between $1/7$ and $1/7+h$, divided by $\log_2(1/|h|)$, tend to
$1/3$ as $h\to0$.
\end{example}

Allaart and Kawamura~\cite[Definition~5.1 and
Theorem~5.4]{AllaartKawamura2010} proved a statement of the same kind for
Takagi's function. At a point that is not a dyadic rational, the quotient
of the slope of Takagi's function over a step $h$ and $\log_2(1/|h|)$ has
a limit as $h\to0$ if and only if the proportion of ones among the first
$n$ binary digits has a limit and, when this limit is $0$ or $1$, the
positions of the ones or of the zeros, respectively, form a sequence in
which the ratio of consecutive terms tends to $1$. The limit of the
quotient is then the limit of the proportion of zeros minus that of the
ones, which for $q=2$ is the limit of $d_n(x)/n$.

Whether $d_n(x)/n\to(q-1)\varrho$ depends only on the proportions of the
digits $0,1,\ldots,q-1$ among the first $n$ digits of $x$.
Volkmann~\cite{Volkmann1958} determined the Hausdorff dimension of the
sets of points defined by conditions on these proportions. He treated
every integer base at least $2$ and arbitrary sets of possible proportions.
We state his theorem for nonempty closed sets, for which the maximum
exists.

\begin{theorem}[Volkmann {\cite[Satz~3]{Volkmann1958}}]
\label{volkmann}
Let $\mathcal S$ be a nonempty closed set of points
$z=(z_0,z_1,\ldots,z_{q-1})\in\R^q$
with $z_i\ge0$ for $0\le i\le q-1$ and $\sum_{i=0}^{q-1}z_i=1$. For
$x\in(0,1]$ and $n\ge1$ let
$z^{(n)}(x)$ be the point whose coordinates are the proportions of the
digits $0,1,\ldots,q-1$ among the first $n$ digits of the base $q$
expansion of $x$ with infinitely many nonzero digits. Then the set of
points $x\in(0,1]$ such that the limit of every convergent subsequence of
$(z^{(n)}(x))_{n\ge1}$ lies in $\mathcal S$ has Hausdorff dimension
\[
 \max_{z\in\mathcal S}\Bigl(-\sum_{i=0}^{q-1}z_i\log_qz_i\Bigr),
\]
where $0\log_q0=0$.
\end{theorem}

\begin{definition}
\label{rates}
For $-1\le\varrho\le1$ let $\mathcal R_\varrho$ be the set of points
$x\in(0,1)$ that are not $q$-adic rationals and satisfy
$d_n(x)/n\to(q-1)\varrho$.
\end{definition}

\begin{corollary}
\label{spectrum}
\begin{enumerate}
\item[(a)] The Hausdorff dimension of $\mathcal R_\varrho$ is the largest
 value of $-\sum_{i=0}^{q-1}z_i\log_qz_i$ over the numbers
 $z_0,\ldots,z_{q-1}\ge0$ with $\sum_{i=0}^{q-1}z_i=1$ and
 $\sum_{i=0}^{q-1}(q-1-2i)z_i=(q-1)\varrho$. This value
 is $0$ if $\varrho=\pm1$. If $-1<\varrho<1$, then the equation
 \begin{equation}
  \frac{\sum_{m=0}^{q-1}mu^m}{\sum_{m=0}^{q-1}u^m}
  =\frac{(q-1)(1+\varrho)}2
  \label{specu}
 \end{equation}
 has exactly one solution $u>0$. The largest value is attained only at
 the point whose coordinates are $z_i=u^{q-1-i}/\sum_{m=0}^{q-1}u^m$
 for $0\le i\le q-1$, and it equals
 \[
  \log_q\Bigl(\sum_{m=0}^{q-1}u^m\Bigr)-\frac{(q-1)(1+\varrho)}2\log_qu.
 \]
 For $q=2$ we have $u=(1+\varrho)/(1-\varrho)$, and the value is
 $-z_0\log_2z_0-z_1\log_2z_1$ with $z_0=(1+\varrho)/2$ and
 $z_1=(1-\varrho)/2$. For $q=3$ we have
 $u=\bigl(\varrho+\sqrt{4-3\varrho^2}\bigr)/(2(1-\varrho))$ and
 $z_1=(\sqrt{4-3\varrho^2}-1)/3$, and the value is
 $\log_3(u+1+1/u)-\varrho\log_3u$.
\item[(b)] As a function of $\varrho$, the Hausdorff dimension of
 $\mathcal R_\varrho$ is even, real analytic, and strictly concave on
 $(-1,1)$, and it decreases strictly on $[0,1]$ from $1$ at $\varrho=0$ to
 $0$ at $\varrho=1$.
\item[(c)] If $0<\varrho<1$, then $T_q'(x)=+\infty$ for every
 $x\in\mathcal R_\varrho$, and if $-1<\varrho<0$, then $T_q'(x)=-\infty$ for
 every $x\in\mathcal R_\varrho$. The set $\mathcal R_0$ has Lebesgue
 measure $1$.
\end{enumerate}
\end{corollary}

\begin{proof}
(a) For $-1\le\varrho\le1$ let $\mathcal S_\varrho$ be the set of points
$z$ of Theorem~\ref{volkmann} with
$\sum_{i=0}^{q-1}(q-1-2i)z_i=(q-1)\varrho$. This set is closed, and it
contains the point with $z_0=(1+\varrho)/2$, $z_{q-1}=(1-\varrho)/2$, and
all other coordinates $0$. Let $x\in(0,1)$ be a point that is not a
$q$-adic rational. Then the expansion of $x$ in Theorem~\ref{volkmann} is its
digit sequence, and
$\sum_{i=0}^{q-1}(q-1-2i)z^{(n)}_i(x)=d_n(x)/n$, because $nz^{(n)}_i(x)$
of the first $n$ digits are equal to $i$. The
points $z^{(n)}(x)$ lie in the compact set of the points $z$ of
Theorem~\ref{volkmann}, and $\sum_{i=0}^{q-1}(q-1-2i)z_i$ depends
continuously on $z$. Hence
$d_n(x)/n\to(q-1)\varrho$ if and only if the limit of every convergent
subsequence of $(z^{(n)}(x))_{n\ge1}$ lies in $\mathcal S_\varrho$. Thus
$\mathcal R_\varrho$ is the set of Theorem~\ref{volkmann} with
$\mathcal S=\mathcal S_\varrho$, from which the $q$-adic rationals in
$(0,1]$ are removed. A countable set does not change the Hausdorff
dimension, so Theorem~\ref{volkmann} gives the first statement in (a).

If $\varrho=1$, then every point $z$ of $\mathcal S_\varrho$ has $z_0=1$,
because $q-1-2i<q-1$ for $i\ge1$, and if $\varrho=-1$, then it has
$z_{q-1}=1$. In both cases the maximum is $0$. Let $-1<\varrho<1$. For
$u>0$ put $P(u)=\sum_{m=0}^{q-1}u^m$ and
$g(u)=\sum_{m=0}^{q-1}mu^m/P(u)$, the left side of \eqref{specu}. Then
\[
 ug'(u)
 =\frac{\bigl(\sum_{m=0}^{q-1}m^2u^m\bigr)\bigl(\sum_{m=0}^{q-1}u^m\bigr)
 -\bigl(\sum_{m=0}^{q-1}mu^m\bigr)^2}{P(u)^2}
 =\frac1{P(u)^2}\sum_{0\le m<m'\le q-1}(m'-m)^2u^{m+m'}>0,
\]
where the second equality follows by expanding the numerator as the sum of
the terms $(m^2-mm')u^{m+m'}$ over all $0\le m,m'\le q-1$. The terms with
$m=m'$ are zero. For $m<m'$, the two terms with the indices $m,m'$ and
$m',m$ add up to $(m'-m)^2u^{m+m'}$. Thus $g$ increases strictly.
Moreover, $g(u)\to0$ as $u\downarrow0$ and $g(u)\to q-1$ as
$u\to\infty$, and $0<(q-1)(1+\varrho)/2<q-1$. Hence \eqref{specu} has
exactly one solution $u>0$. Let $w$ be the point with
$w_i=u^{q-1-i}/P(u)$ for $0\le i\le q-1$. Its coordinates are positive
and add up to $1$, and $\sum_{i=0}^{q-1}(q-1-i)w_i=g(u)$. Every point $z$
of Theorem~\ref{volkmann} satisfies
$\sum_{i=0}^{q-1}(q-1-2i)z_i=2\sum_{i=0}^{q-1}(q-1-i)z_i-(q-1)$, so
it lies in $\mathcal S_\varrho$ if and only if
$\sum_{i=0}^{q-1}(q-1-i)z_i=(q-1)(1+\varrho)/2$. In particular
$w\in\mathcal S_\varrho$. Now let $z\in\mathcal S_\varrho$. Since
$\log v\le v-1$ for $v>0$, with equality only for $v=1$,
\[
 \sum_{i:z_i>0}z_i\log\frac{w_i}{z_i}
 \le\sum_{i:z_i>0}(w_i-z_i)=\sum_{i:z_i>0}w_i-1\le0,
\]
where the sums run over the indices $0\le i\le q-1$ with $z_i>0$.
Equality holds in both places only if $w_i=z_i$ whenever $z_i>0$ and
$z_i>0$ for all $i$, that is, only if $z=w$. Since
$\log_qw_i=(q-1-i)\log_qu-\log_qP(u)$, it follows that
\[
 -\sum_{i=0}^{q-1}z_i\log_qz_i\le-\sum_{i=0}^{q-1}z_i\log_qw_i
 =\log_qP(u)-\frac{(q-1)(1+\varrho)}2\log_qu,
\]
with equality only for $z=w$. This proves the statements in (a) for
every $q$. For $q=2$, equation \eqref{specu} reads
$u/(1+u)=(1+\varrho)/2$, so $u=(1+\varrho)/(1-\varrho)$ and
$w=\bigl((1+\varrho)/2,(1-\varrho)/2\bigr)$. For $q=3$, equation
\eqref{specu} reads $(1-\varrho)u^2-\varrho u-(1+\varrho)=0$, and its
positive root is the number $u$ stated in (a). The value is
$\log_3(u^2+u+1)-(1+\varrho)\log_3u=\log_3(u+1+1/u)-\varrho\log_3u$.
Finally, $w=(u^2,u,1)/(u^2+u+1)$ satisfies $w_0w_2=w_1^2$,
$w_0+w_2=1-w_1$, and $w_0-w_2=\varrho$. Hence
$(1-w_1)^2-\varrho^2=4w_1^2$, that is,
$3w_1^2+2w_1-(1-\varrho^2)=0$, which gives
$w_1=(\sqrt{4-3\varrho^2}-1)/3$.

(b) Let $P$ and $g$ be as in the proof of (a). The function
$u\mapsto2g(u)/(q-1)-1$ is real analytic with a positive derivative on
$(0,\infty)$ and maps $(0,\infty)$ onto $(-1,1)$. Hence the solution $u$
of \eqref{specu} is a strictly increasing and real analytic function of
$\varrho\in(-1,1)$, with $u=1$ at $\varrho=0$ because $g(1)=(q-1)/2$.
Moreover, $P(1/u)=u^{-(q-1)}P(u)$ and
$g(1/u)=\sum_{m=0}^{q-1}mu^{q-1-m}/P(u)=q-1-g(u)$, so replacing $u$ by
$1/u$ replaces $\varrho$ by $-\varrho$, and it does not change the value
in (a), since
$\log_qP(1/u)+\frac{(q-1)(1-\varrho)}2\log_qu
=\log_qP(u)-\frac{(q-1)(1+\varrho)}2\log_qu$. The value in (a) is
therefore a real analytic and even function of $\varrho$, equal to
$\log_qP(1)=1$ at $\varrho=0$. Write $u'$ for the derivative of $u$ with
respect to $\varrho$. Since $uP'(u)/P(u)=g(u)=(q-1)(1+\varrho)/2$, the
chain rule gives
\[
 \frac{d}{d\varrho}\Bigl(\log_qP(u)-\frac{(q-1)(1+\varrho)}2\log_qu\Bigr)
 =\frac{(q-1)(1+\varrho)u'}{2u\log q}-\frac{q-1}2\log_qu
 -\frac{(q-1)(1+\varrho)u'}{2u\log q}
 =-\frac{q-1}2\log_qu,
\]
which decreases strictly in $\varrho$. It follows that the function is
strictly concave, and that it decreases strictly on $[0,1)$, where
$u\ge1$. Finally, the dimension is $0$ at $\varrho=1$ and positive for
$0\le\varrho<1$, because the point $w$ of $\mathcal S_\varrho$ has
$q\ge2$ positive coordinates. Therefore it decreases strictly on $[0,1]$.

(c) Let $0<|\varrho|<1$ and $x\in\mathcal R_\varrho$. By
Theorem~\ref{slopegrowth}(a),
$\Delta(x,x+h)=((q-1)\varrho+o(1))\log_q(1/|h|)$
as $h\to0$, and this tends to $+\infty$ if $\varrho>0$ and to $-\infty$ if
$\varrho<0$. Almost every $x\in(0,1)$ satisfies $|d_n(x)|\le2\varphi(n)$
for all large $n$ by \eqref{lildigits}, and $\varphi(n)/n\to0$. Since
the $q$-adic rationals form a countable set, $\mathcal R_0$ has Lebesgue
measure $1$.
\end{proof}

For instance, for $q=3$ the points $1/8=(0.\overline{01})_3$ and
$31/40=(0.\overline{2022})_3$ of Figure~\ref{fig6} lie in
$\mathcal R_{1/2}$ and in $\mathcal R_{-1/2}$,~respectively. Both sets have
Hausdorff dimension approximately $0.8203$ by Corollary~\ref{spectrum}(a)
and~(b), since $u=(1+\sqrt{13})/2$ for $\varrho=1/2$. For $q=2$ the point
$1/7$ of Example~\ref{ex:slopegrowth} lies in $\mathcal R_{1/3}$, and by
Corollary~\ref{spectrum}(a) with $z_0=2/3$ and $z_1=1/3$ this set has
Hausdorff dimension $\log_23-2/3\approx0.9183$.

For Takagi's function, Allaart and
Kawamura~\cite[Corollary~3.8]{AllaartKawamura2010} proved a statement of
the same kind as Corollary~\ref{spectrum}(c): the derivative is $+\infty$
at the points that are not dyadic rationals and whose proportion of ones
among the first $n$ binary digits tends to a limit in $(0,1/2)$, and it is
$-\infty$ when this limit lies in $(1/2,1)$.
They~\cite[Corollary~3.9]{AllaartKawamura2010} derived from it that the
sets of points at which Takagi's function has an infinite derivative have
Hausdorff dimension~$1$. In the same way, Corollary~\ref{spectrum} gives a second proof that the sets in
Proposition~\ref{infdim} have Hausdorff dimension $1$, because the
dimension of $\mathcal R_\varrho$ tends to $1$ as $\varrho\to0$.

At a point that is not a $q$-adic rational the derivative of $T_q$ can be
infinite from exactly one side. For instance, at the points $x$ of
Proposition~\ref{sparse} with $p_{k+1}\ge4p_k$ for all $k$ we have
$(T_q)'_+(x)=+\infty$ and $D_-T_q(x)=-\infty$, and $(T_q)'_-(x)\ne-\infty$ by
Theorem~\ref{thm2}(a), because $d_n(x)\to+\infty$.
Kr\"uppel~\cite[Remark~1.1]{Kruppel2010a} observed that at the points that
are not dyadic rationals, Takagi's function has the derivative $+\infty$
from the right whenever it has the derivative $+\infty$ from the left, and
that the converse fails. Proposition~\ref{onesided} contains the same
statement for $T_q$ in every base.

\begin{definition}
\label{onesets}
Let $\mathcal O_+$ be the set of points $x\in\R$ with $(T_q)'_+(x)=+\infty$
at which $T_q$ has no infinite left derivative, and let $\mathcal O_-$ be
the set of points $x\in\R$ with $(T_q)'_-(x)=-\infty$ at which $T_q$ has no
infinite right~derivative.
\end{definition}

\begin{proposition}
\label{onesided}
\begin{enumerate}
\item[(a)] The points of $\R$ at which $T_q$ has an infinite derivative from
 exactly one side are the points of $\mathcal O_+\cup\mathcal O_-$, and
 $\mathcal O_-=\{1-x:x\in\mathcal O_+\}$. Let $x\in(0,1)$ be a point that
 is not a $q$-adic rational. If $(T_q)'_-(x)=+\infty$, then $(T_q)'_+(x)=+\infty$,
 and if $(T_q)'_+(x)=-\infty$, then $(T_q)'_-(x)=-\infty$. Moreover,
 $x\in\mathcal O_+$ if and only if $d_n(x)\to+\infty$ and the first
 condition in Theorem~\ref{thm2}(b) fails.
\item[(b)] The sets $\mathcal O_+$ and $\mathcal O_-$ are dense in $\R$,
 have Lebesgue measure zero, and have Hausdorff dimension $1$. Every
 $x\in\mathcal O_+$ satisfies $D^-T_q(x)=+\infty$, and $T_q$ has no left
 derivative at $x$, finite or infinite. The points $x\in\mathcal O_+$
 with $D_-T_q(x)=-\infty$ form a dense set of Hausdorff dimension $1$.
\end{enumerate}
\end{proposition}

\begin{proof}
(a) Let $x\in(0,1)$ be a point that is not a $q$-adic rational. By
Theorem~\ref{thm2}(b) and~(c), $(T_q)'_-(x)=+\infty$ holds exactly when
$T_q'(x)=+\infty$, and then $(T_q)'_+(x)=+\infty$. In the same way,
$(T_q)'_+(x)=-\infty$ implies $T_q'(x)=-\infty$, and hence $(T_q)'_-(x)=-\infty$.
Theorem~\ref{thm2}(a) shows that $(T_q)'_+(x)=+\infty$ holds exactly when
$d_n(x)\to+\infty$, and that $(T_q)'_-(x)\ne-\infty$ in this case, since
$(T_q)'_-(x)=-\infty$ would imply $d_n(x)\to-\infty$. Hence
$x\in\mathcal O_+$ if and only if $d_n(x)\to+\infty$ and
$(T_q)'_-(x)\ne+\infty$. By Theorem~\ref{thm2}(b), the condition
$(T_q)'_-(x)\ne+\infty$ means that the first condition there fails.

Now let $x\in\R$ be a point at which $T_q$ has an infinite derivative from
exactly one side. At a $q$-adic rational both one sided derivatives are
infinite by Proposition~\ref{cusp}, and $T_q$ has period $1$, so we may
assume that $x\in(0,1)$ is not a $q$-adic rational. If $(T_q)'_+(x)$ is
infinite, then $(T_q)'_+(x)\ne-\infty$, since otherwise $(T_q)'_-(x)=-\infty$;
hence $x\in\mathcal O_+$. If $(T_q)'_-(x)$ is infinite, then
$(T_q)'_-(x)\ne+\infty$ by the first implication, and $x\in\mathcal O_-$. At
the points of $\mathcal O_+\cup\mathcal O_-$ the function $T_q$ has an
infinite derivative from exactly one side by definition. Finally,
$T_q(1-x+h)-T_q(1-x)=T_q(x-h)-T_q(x)$ gives $(T_q)'_+(1-x)=-(T_q)'_-(x)$ and
$(T_q)'_-(1-x)=-(T_q)'_+(x)$, so $\mathcal O_-=\{1-x:x\in\mathcal O_+\}$.

(b) At the points of $\mathcal O_+$ and $\mathcal O_-$ the function $T_q$
has an infinite derivative from one side, so these sets have Lebesgue
measure zero by Corollary~\ref{aedini}. Let $x\in\mathcal O_+$. By
periodicity we may assume that $0\le x<1$. Since $(T_q)'_-(x)\ne-\infty$,
Proposition~\ref{cusp} shows that $x$ is not a $q$-adic rational, so
$0<x<1$, and $d_n(x)\to+\infty$ by (a). By Lemma~\ref{localrun}(b),
$\Delta(x,a_n)\ge d_n(x)$, where $a_n<x$ and $a_n\to x$, so
$D^-T_q(x)=+\infty$. Since $(T_q)'_-(x)\ne+\infty$, the function $T_q$ has
no left derivative at $x$, finite or infinite.

Let $x\in(0,1)$, $n\ge0$, $0\le p<q^n$, and $x'=(p+x)/q^n$. If $h\ne0$
and $x+h\in[0,1]$, then \eqref{cell} at the points $t=x+h$ and $t=x$ gives
\[
 \Delta\bigl(x',x'+q^{-n}h\bigr)=(q-1)n-2s_q(p)+\Delta(x,x+h).
\]
Hence the one sided derivatives and the Dini derivatives of $T_q$ at $x'$
are those at $x$ increased by $(q-1)n-2s_q(p)$. In particular,
$x'\in\mathcal O_+$ if and only if $x\in\mathcal O_+$, and
$D_-T_q(x')=-\infty$ if and only if $D_-T_q(x)=-\infty$. If $\mathcal O_+$
contains a point $x\in(0,1)$ with $D_-T_q(x)=-\infty$, it follows that
every interval of every level contains such points, and by periodicity
they are dense in $\R$. Then $\mathcal O_+$ and
$\mathcal O_-=\{1-x:x\in\mathcal O_+\}$ are dense as well.

For every even $m\ge6$ it remains to find a set
$\mathcal K_m\subseteq(0,1)$ whose points $x$ lie in $\mathcal O_+$ and
satisfy $D_-T_q(x)=-\infty$, and whose Hausdorff dimension is at least a
number $s$ that tends to $1$ as $m\to\infty$. Since $x\mapsto1-x$
preserves Hausdorff dimension, this also gives the dimension of
$\mathcal O_-$. Let $\mathcal B_m$ be the set of words $w_1\ldots w_m$ over
$\{0,1,\ldots,q-1\}$ with $w_1=w_m=1$ and $w_2=w_{m-1}=q-2$ such that
the letters $w_3,\ldots,w_{m-2}$ have the slope sum $0$, and let
$\#\mathcal B_m$ be the number of these words. The letters $1$ and $q-2$
have the steps $q-3$ and $3-q$, so every word of $\mathcal B_m$ has the
slope sum $0$. For $q=3$ these are the words with
$w_1=w_2=w_{m-1}=w_m=1$ in which $w_3,\ldots,w_{m-2}$ contain as many zeros
as twos.
Let $\mathcal K_m$ be the set of points $x\in[0,1)$ whose digit sequence
consists of words of $\mathcal B_m$ written one after another, with a
string of $2\cdot3^{j-1}$ zeros inserted after the first $3^j$ words, for
every $j\ge1$.

Now let $x\in\mathcal K_m$. Every word contains the nonzero digit $1$ and
the digit $q-2$, which is not $q-1$, every word has the slope sum $0$,
and every inserted zero has the step $q-1$. Hence $x$ is not a $q$-adic
rational, $0<x<1$, and $d_n(x)$ is at least $q-1$ times the number of
inserted zeros among the first $n$ digits, minus $(q-1)m$, so that
$d_n(x)\to+\infty$. The digit just before the $j$th inserted string is
the last digit of a word. This digit is $1$, so its position is $n_k$ for
some $k$, where $n_1<n_2<\cdots$ are the positions of the nonzero digits
of $x$. At this position
$d_{n_k}(x)=(q-1)\sum_{i=1}^{j-1}2\cdot3^{i-1}=(q-1)(3^{j-1}-1)$, and the
next nonzero digit is the first digit of the next word, so
$n_{k+1}-n_k=2\cdot3^{j-1}+1$. Since
$2\cdot3^{j-1}+1\le3^j\le q^{2j}$, the numbers $\Gamma_k$ of
Section~\ref{sec:inf} satisfy
\[
 \Gamma_k=(q-1)\bigl(3^{j-1}-1\bigr)-(q-1)\bigl(2\cdot3^{j-1}+1\bigr)
 +(q-1)\log_q\bigl(2\cdot3^{j-1}+1\bigr)\le-(q-1)\bigl(3^{j-1}+2-2j\bigr)
\]
at these positions. The right side tends to $-\infty$ as $j\to\infty$, so
$\Gamma_k\to-\infty$ along the indices $k$ of these positions. In
particular, $\Gamma_k$ does not tend to $+\infty$, and
$x\in\mathcal O_+$ by (a). By \eqref{ykbound}, the points
$y_k=a_{n_k}-(n_{k+1}-n_k)(x-a_{n_k})$ satisfy
$\Delta(y_k,x)<\Gamma_k+q(q+1)/2+4$
and $0<x-y_k\le2q^{-n_k}$ for every $k\ge1$, so $D_-T_q(x)=-\infty$.

For $n\ge1$, at most $n/m$ of the first $n$ digits of $x$ belong to
inserted strings. Indeed, if the $j$th string is the last one that begins
at a position at most $n$, then $n>m\cdot3^j$, because the first $3^j$
words come before it, while the first $j$ strings have
$\sum_{i=1}^{j}2\cdot3^{i-1}<3^j$ digits. Consequently the first $n$
digits of $x$ contain at least $n(1-1/m)/m-1$ complete words. Consider the map
that sends a sequence of words of $\mathcal B_m$ to the point of
$\mathcal K_m$ constructed from it, and let $\mu$ be the image under this
map of the product measure in which each word has weight
$1/\#\mathcal B_m$ at every place. This map is continuous, since sequences
with the same first $i$ words are sent to points at distance at most
$q^{-im}$. Hence its image $\mathcal K_m$ is compact, and $\mu$ is a
Borel probability measure with $\mu(\mathcal K_m)=1$. Let $n\ge0$, and
let $J$ be an interval of level $n$. The
points of $\mathcal K_m$ are not $q$-adic rationals, so those in $J$ have
the same first $n$ digits, and hence the same first $i$ words for some
$i\ge n(1-1/m)/m-1$. Therefore
\[
 \mu(J)\le(\#\mathcal B_m)^{1-n(1-1/m)/m}=\#\mathcal B_m\,q^{-ns},
 ~\text{where}~
 s=\Bigl(1-\frac1m\Bigr)\frac{\log_q\#\mathcal B_m}m.
\]
Since $m\ge6$, the set $\mathcal B_m$ contains the word in which the
letters $w_3,\ldots,w_{m-2}$ are $0,q-1,0,q-1,\ldots,0,q-1$ and the word in
which they are $q-1,0,q-1,0,\ldots,q-1,0$. Hence $\#\mathcal B_m\ge2$ and
$s>0$, and Lemma~\ref{massdist} with $\ell=1$ shows that $\mathcal K_m$
has Hausdorff dimension at least $s$.

Finally, a word of length $m-4$ has the slope sum $0$ exactly when its
first half has slope sum $r$ and its second half has slope sum $-r$ for
some integer $r$. Replacing every letter $w$ by $q-1-w$ shows that among the
words of length $(m-4)/2$ there are as many with the slope sum $-r$ as
with the slope sum $r$. The slope sum of a word of length $(m-4)/2$ is
one of the $(q-1)(m-4)/2+1$ integers $r$ with $|r|\le(q-1)(m-4)/2$ that
have the parity of $(q-1)(m-4)/2$. It follows that $\#\mathcal B_m$ is
the sum of the squares of $(q-1)(m-4)/2+1$ numbers whose sum is
$q^{(m-4)/2}$, and the Cauchy--Schwarz inequality gives
$\#\mathcal B_m\ge q^{m-4}/\bigl((q-1)(m-4)/2+1\bigr)$.
Therefore
\[
 s\ge\Bigl(1-\frac1m\Bigr)
 \frac{m-4-\log_q\bigl((q-1)(m-4)/2+1\bigr)}m,
\]
which tends to $1$ as $m\to\infty$.
\end{proof}

\section{Level sets}
\label{sec:level}

In this section we prove Theorem~\ref{levelthm}, determine $L(3/4)$ for
$q=3$, and prove further results on the level sets of $T_q$. For Takagi's
function, the level set at the maximum value, which
Kahane~\cite[Section~1]{Kahane1959} described by the binary digits, has
Hausdorff dimension $1/2$~\cite[Section~1]{deAmo2011}, and de Amo
et~al.~\cite[Theorem~3.4]{deAmo2011} proved that no level set has a larger
dimension. For $q=2$ both parts of Theorem~\ref{levelthm} are known (see
Section~\ref{intro}). We follow the argument of
Allaart~\cite[Theorem~1.1]{AllaartLevelSets2014} for part (b) in every
base. For $q\le3$ we follow it also for part (a), and then it gives the
bound of Corollary~\ref{levelcount}. In every base
we prove that $L(y)$ is finite for almost every $y$ in another way, with
the help of Theorem~\ref{thm2}.

We recall that $\Lambda(c)=\max_{0\le t\le1}(T_q(t)-ct)$ for integers $c$.
By Theorem~\ref{gmax}, the maximum $\Lambda(0)$ of $T_q$ is $(q+1)/4$ if
$q$ is odd and $q(q+2)/(4(q+1))$ if $q$ is even, so that $\Lambda(0)=1$ for
$q=3$ and $\Lambda(0)=2/3$ for $q=2$. Since $T_q$ is continuous and
$T_q(0)=0$, it takes every value in $[0,\Lambda(0)]$ on $[0,1]$.

\begin{definition}
\label{levelset}
For $0\le y\le\Lambda(0)$, the level set of $T_q$ at height $y$ is
$L(y)=\{x\in[0,1]:T_q(x)=y\}$.
\end{definition}

By \eqref{Tpositive}, $L(0)=\{0,1\}$. By Theorem~\ref{gmax},
$L(\Lambda(0))=\{1/2\}$ if $q$ is odd, and if $q$ is even, then
$L(\Lambda(0))$ consists of the points that have a digit sequence in which
every pair $(\xi_{2i+1},\xi_{2i+2})$ with $i\ge0$ is $(q/2-1,q/2)$ or
$(q/2,q/2-1)$. In every base there are sets of this kind on which $T_q$ is
constant.

\begin{definition}
\label{Eset}
Let $e$ be an integer with $0\le e<(q-1)/2$, and put $e'=q-1-e$, so that
$e<e'$. Let $E$ be the set of points of $[0,1]$ that have a digit sequence
consisting of infinitely many strings $ee'$ and $e'e$, that is, a digit
sequence in which every pair $(\xi_{2i+1},\xi_{2i+2})$ with $i\ge0$ is
$(e,e')$ or $(e',e)$. For $k\ge0$ let $W_k$ be the set
of integers $0\le w<q^{2k}$ whose word of length $2k$, with leading zeros
allowed, consists of
$k$ strings $ee'$ and~$e'e$, and for $w\in W_k$
let $E_w=[w/q^{2k},(w+1)/q^{2k}]$. Each integer $w$ lies in at most one of the
sets $W_k$. Indeed, for $k\ge1$ the word of length $2k$ of an integer
$w\in W_k$ does not begin with the string $00$, as it would if
$w<q^{2k-2}$. We also put
\[
 y_e=\frac{(e+1)(q-e)}{q+1}.
\]
The sets $E$, $W_k$, and $E_w$ depend on $e$, which we do not show in the
notation. For $q=3$ and for $q=2$ the only choice is $e=0$. For $q=3$ the
two strings are $02$ and $20$, $y_0=3/4$, and $E_w=[w/9^k,(w+1)/9^k]$ for
$w\in W_k$.
\end{definition}

Exchanging $e$ and $e'$ changes neither $E$ nor the sets $W_k$, so the
condition $e<(q-1)/2$ excludes only the case $e=e'=(q-1)/2$ for odd $q$,
in which $E$ is the single point $1/2$.

Let $g$ be an integer with $0\le g\le q-1$, and put $g'=q-1-g$. By
\eqref{branches} with $y=0$, $T_q(g'/q)=g'(q-g')/q$, and by \eqref{branches}
with $a=g$ and $y=g'/q$,
\[
 T_q\Bigl(\frac{gq+g'}{q^2}\Bigr)
 =\frac{g(q-g)q+(q-1-2g)g'+g'(q-g')}{q^2}
 =\frac{(q-g)(gq+g')}{q^2}=\frac{(q-1)(g+1)(q-g)}{q^2},
\]
because $(q-1-2g)g'+g'(q-g')=g'(q-g)$ and $gq+g'=(q-1)(g+1)$. The last
expression does not change when $g$ is replaced by $g'$, since
$(g'+1)(q-g')=(q-g)(g+1)$. For $g=e$ it equals $(1-q^{-2})y_e$. Hence
$T_q(m/q^2)=(1-q^{-2})y_e$ for $m=eq+e'$ and for $m=e'q+e$. Both integers
have $s_q(m)=q-1$, so equation~\eqref{cell} with
$n=2$ gives
\begin{equation}
 T_q\Bigl(\frac{m+t}{q^2}\Bigr)-y_e=\frac{T_q(t)-y_e}{q^2},
 ~m\in\{eq+e',e'q+e\},\ 0\le t\le1.
 \label{Estep}
\end{equation}
For $q=3$ this reads $T_q((m+t)/9)-3/4=(T_q(t)-3/4)/9$ for $m\in\{2,6\}$.

\begin{lemma}
\label{Efacts}
The function $T_q$ equals $y_e$ on $E$, and $E$ is the union of its images
under the maps $t\mapsto(eq+e'+t)/q^2$ and $t\mapsto(e'q+e+t)/q^2$. For
every $k\ge0$ and every $w\in W_k$,
\begin{equation}
 T_q\Bigl(\frac{w+u}{q^{2k}}\Bigr)-y_e=\frac{T_q(u)-y_e}{q^{2k}},~0\le u\le1.
 \label{Eblock}
\end{equation}
For every $k\ge0$,
the intervals $E_w$ with $w\in W_k$ cover $E$, and any two of them are at
distance at least $\bigl((e'-e)(q-1)-1\bigr)q^{-2k}$, which is
$3\cdot9^{-k}$ for $q=3$ and $0$ for $q=2$. The set $E$ is closed and
contains no $q$-adic rational. Its Hausdorff
measure of dimension $s=\log2/(2\log q)$ is at most $1$, at least $1/2$
if $q\ge3$, and at least $1/4$ if $q=2$, so that
$E$ has Hausdorff dimension $s$.
\end{lemma}

\begin{proof}
A digit sequence that consists of strings $ee'$ and $e'e$ contains both
digits $e$ and $e'$ infinitely often, so its digits are neither
eventually all $0$ nor eventually all $q-1$. Hence no point of $E$ is a
$q$-adic rational, and every point of $E$ has only one digit sequence.
The word of $w\in W_k$ consists of $k$ strings $ee'$ and $e'e$, so
$(w+u)/q^{2k}$ is obtained from $u\in[0,1]$ by $k$ applications of the maps
$t\mapsto(eq+e'+t)/q^2$ and $t\mapsto(e'q+e+t)/q^2$, and $k$ applications
of \eqref{Estep} give \eqref{Eblock}.
Every point $x\in E$ has the form $(m+t)/q^2$ with $m\in\{eq+e',e'q+e\}$
and $t\in E$, and every point of this form lies in $E$; thus $E$ is the
union of its images under the two maps. For every $k\ge0$ the first $2k$
digits of $x$ form the word of an integer $w\in W_k$, and then
$x=(w+u)/q^{2k}$ with $u\in E$. Hence the intervals $E_w$ with $w\in W_k$
cover $E$, and \eqref{Eblock} together with $0\le T_q\le\Lambda(0)$ gives
$|T_q(x)-y_e|\le(\Lambda(0)+y_e)q^{-2k}$ for all
$k\ge0$, so that $T_q(x)=y_e$.

Any two of the intervals $E_w$ with $w\in W_k$ are at distance at least
$\bigl((e'-e)(q-1)-1\bigr)q^{-2k}$. Indeed, for $k\ge1$ and $w\in W_{k-1}$
the two intervals $E_{w'}$ with $w'\in W_k$ inside $E_w$ are
$E_{q^2w+eq+e'}$ and $E_{q^2w+e'q+e}$, and they are separated by a gap of
length $\bigl((e'q+e)-(eq+e')-1\bigr)q^{-2k}=\bigl((e'-e)(q-1)-1\bigr)q^{-2k}$,
while two intervals $E_{w'}$ with $w'\in W_k$ inside different
intervals $E_w$ with $w\in W_{k-1}$ are, by induction on $k$, at distance
at least $\bigl((e'-e)(q-1)-1\bigr)q^{2-2k}$.

Now let $x_1,x_2,\ldots$ be points of $E$ that converge to a point $x$.
Since there are only two strings, we can choose strings
$b_1,b_2,\ldots\in\{ee',e'e\}$ one after the other such that, for every
$k\ge1$, infinitely many of the points $x_i$ have a digit sequence that
begins with $b_1b_2\ldots b_k$. The point $z$ whose digit sequence is
$b_1b_2b_3\ldots$ lies in $E$, and for every $k$ it lies in the same
interval of level $2k$ as infinitely many of the points $x_i$. Hence
$|x-z|\le q^{-2k}$ for every $k$, so that $x=z\in E$. Thus $E$ is closed.

It remains to bound the Hausdorff measure of dimension $s$ of $E$. We
have $q^{2s}=2$. The covers by the $2^k$ intervals $E_w$ with $w\in W_k$
give $\sum_{w\in W_k}(q^{-2k})^s=1$, so this measure is at most $1$. For
the lower bound, consider the product measure on the
sequences of strings $ee'$ and $e'e$ in which each string has weight
$1/2$ at every position, and let $\mu$ be its image under the map that
sends a sequence to the point of $E$ with that digit sequence. This map is continuous, since sequences
that agree in their first $k$ strings are sent to points at distance at
most $q^{-2k}$, so $\mu$ is a Borel probability measure on $E$. Since no
point of $E$ is a $q$-adic rational, the map sends a sequence into $E_w$
exactly when its first $k$ strings form the word of $w$, and so $\mu$
gives the mass $2^{-k}$ to each of the $2^k$ intervals $E_w$ with
$w\in W_k$.

Suppose first that $q\ge3$. Then $(e'-e)(q-1)\ge2$, so any two of the
intervals $E_w$ with $w\in W_k$ are at distance at least $q^{-2k}$, and a
closed set $U$ with $\operatorname{diam}U<q^{-2k}$ meets at most one of
them. Hence $\mu(U)\le2^{-k}$. Thus $\mu(U)=0$ if $\operatorname{diam}U=0$,
and $\mu(U)\le2^{-k}=2(q^{-2k-2})^s\le2(\operatorname{diam}U)^s$ if
$q^{-2k-2}\le\operatorname{diam}U<q^{-2k}$, so that
$\mu(U)\le2(\operatorname{diam}U)^s$ for every closed set $U$ with
$\operatorname{diam}U<1$. If sets $U_1,U_2,\ldots$ of diameter less than
$1$ cover $E$, then so do their closures, which have the same diameters,
and the countable subadditivity of $\mu$ gives
$1=\mu(E)\le\sum_{i=1}^{\infty}\mu(\overline{U_i})\le2\sum_{i=1}^{\infty}(\operatorname{diam}U_i)^s$.
Hence the Hausdorff measure of dimension $s$ of $E$ is at least $1/2$.
Suppose now that $q=2$. A closed set $U$ with $\operatorname{diam}U<4^{-k}$
meets at most two intervals of level $2k$, so $\mu(U)\le2\cdot2^{-k}$, and
the same argument gives $\mu(U)\le4(\operatorname{diam}U)^s$ for every
closed set $U$ with $\operatorname{diam}U<1$ and the lower bound $1/4$.
\end{proof}

Since $T_q=y_e$ on $E$, the level set $L(y_e)$ contains $E$, and its
Hausdorff dimension is at least $\log2/(2\log q)$. If $q$ is even and $e=q/2-1$, then
$y_e=\Lambda(0)$ and $E=L(\Lambda(0))$, by Theorem~\ref{gmax}; for $q=2$
this is the level set at the maximum value of Takagi's function, and
$\log2/(2\log q)=1/2$. For this set Galkin and
Galkina~\cite[Corollary~1]{GalkinGalkina2020} showed that its Hausdorff
dimension is $1/(2\log_2q)$, which is $\log2/(2\log q)$. If $q$ is odd and
$e=(q-3)/2$, then $E$ contains the point $(q-1)/(2(q+1))$ of
Proposition~\ref{quarter}, whose digits are alternately $(q-3)/2$ and
$(q+1)/2$, and $y_e$ is the value \eqref{quartval} of $T_q$ at this point.
For $q=3$ this point is $1/4=(0.\overline{02})_3$, and $y_e=3/4$. The next
example shows that for $q=9$ a level set contains a closed set of
Hausdorff dimension larger than $\log2/(2\log q)$.

\begin{example}
\label{ex:nineblocks}
Let $q=9$. The four strings $417$, $471$, $354$, and $534$ have the digit
slopes $(d_1,d_2,d_3)=(0,6,0)$, $(0,-6,0)$, $(2,0,0)$, and $(-2,0,0)$,
respectively, so that $d_3=0$ and $\sum_{j=1}^{3}9^{-j}d_j^2=4/9$ for
each of them. Let $F$ be the set of points of $[0,1]$ that have a digit
sequence consisting of infinitely many of these four strings. For
$x\in F$, the second equation in \eqref{dsq} gives
\[
 T_q(x)=\frac{10}4-\frac8{36}\cdot\frac49\sum_{i=0}^{\infty}9^{-3i}
 =\frac{437}{182}.
\]
No point of $F$ is a $9$-adic rational, and $F$ is closed, by the
argument for $E$ in the proof of Lemma~\ref{Efacts}. Consider the product
measure on the sequences of these strings in which each string has weight
$1/4$ at every position, and let $\mu$ be its image under the map that
sends a sequence to the point of $F$ with that digit sequence. Then
$\mu$ is a Borel probability measure with $\mu(F)=1$, and for every
$k\ge0$ it gives each interval of level $3k$ the mass $4^{-k}$ or $0$.
Since $9^{3ks}4^{-k}=1$ for $s=\log4/(3\log9)$, Lemma~\ref{massdist} with
$\ell=3$ shows that $F$ has Hausdorff dimension at least
$2\log2/(3\log9)$. Hence the level set $L(437/182)$ has Hausdorff
dimension at least $2\log2/(3\log9)$, which is larger than
$\log2/(2\log9)$.
\end{example}

From now on, $x\in[0,1)$ has the digit sequence
$\xi_1,\xi_2,\ldots$, where we take the terminating one when $x$ is a
$q$-adic rational, and for $n\ge0$ we put $a_n=\lfloor q^nx\rfloor/q^n$,
$c_n=a_n+q^{-n}$, and $t_n=q^n(x-a_n)$, as in Section~\ref{sec:inf}. Thus
$[a_n,c_n]$ is the interval of level $n$ with $a_n\le x<c_n$, the point
$t_n\in[0,1)$ has the digits $\xi_{n+1},\xi_{n+2},\ldots$, and the digit
slope $d_n(x)=\sum_{j=1}^{n}(q-1-2\xi_j)$ of Definition~\ref{digslope} is
the slope of $[a_n,c_n]$, by \eqref{dnsq}. For a word $u=u_1\ldots u_k$
over $\{0,1,\ldots,q-1\}$ and $0\le j\le k$ we write
$d_j(u)=\sum_{i=1}^{j}(q-1-2u_i)$ for its digit slopes, as in
Section~\ref{sec:T}, and $|u|=k$ for its length. For odd $q$ we recall from
Definition~\ref{halfslope} that $\kappa_n(x)=d_n(x)/2$. For $q=3$, the number
$\kappa_n(x)=\sum_{j=1}^{n}(1-\xi_j)$ is the number of zeros minus the
number of twos among the first $n$ digits of $x$. Thus $d_n(x)=0$ exactly
when $\kappa_n(x)=0$, that is, when these digits contain as many zeros as
twos.

\begin{remark}
\label{rem:signs}
The digit slope $d_0$ of every point and of every word is $0$, and
two consecutive digit slopes $d_{j-1}$ and $d_j$ differ by at most
$q-1$. For $q=2$ they differ by exactly $1$, and for $q=3$ they are even
and differ by at most $2$. Hence, if $q\le3$, then for any two
consecutive indices $n<n'$ with $d_n=d_{n'}=0$, the digit slopes $d_j$
with $n<j<n'$ all have the same sign, and if $N$ is the largest index with
$d_N=0$, then so do the digit slopes $d_j$ with $j>N$. For $q\ge4$
this is false: for $q=4$ the word $13$ has the digit slopes $d_1=1$ and
$d_2=-2$. By \eqref{cell} and \eqref{dnsq},
\begin{equation}
 T_q(a_n+q^{-n}u)=T_q(a_n)+q^{-n}\bigl(d_n(x)u+T_q(u)\bigr),~0\le u\le1.
 \label{lcell}
\end{equation}
If $d_n(x)=0$, then \eqref{lcell} says that the graph of $T_q$ above
$[a_n,c_n]$ is a copy of the graph of $T_q$ over $[0,1]$, scaled by
$q^{-n}$. Hence $T_q([a_n,c_n])=[T_q(a_n),T_q(a_n)+q^{-n}\Lambda(0)]$, since
$T_q$ maps $[0,1]$ onto $[0,\Lambda(0)]$. Moreover, $T_q$ takes the same value at $x$
and at the point $2a_n+q^{-n}-x$, the reflection of $x$ in the midpoint
of $[a_n,c_n]$, because $T_q(1-u)=T_q(u)$. If $x$ is not a $q$-adic rational,
the digits of the reflected point are
$\xi_1,\ldots,\xi_n,q-1-\xi_{n+1},q-1-\xi_{n+2},\ldots$.
\end{remark}

In the results on $L(3/4)$ that follow, up to the description of
Figure~\ref{levelpic} after the proof of Proposition~\ref{assouadlevel},
we let $q=3$, so that $e=0$, $y_e=3/4$, and $E$ is the set of points of
$[0,1]$ whose ternary expansion consists of infinitely many strings $02$
and $20$. Theorem~\ref{threequarters} rests on the finite list of cases in
Table~\ref{statetab}, which is special to $q=3$, and for other bases we do
not know a description of the level set at the heights $y_e$.
Figure~\ref{levelpic} shows the level set $L(3/4)$.

\begin{theorem}
\label{threequarters}
Let $q=3$, and let $p=337/972$, whose ternary expansion is $0.100100\overline{20}$, so that
$1-p=635/972$. With the sets $E$ and $W_k$ of Definition~\ref{Eset},
\[
 L(3/4)=E\cup\Bigl\{\frac{w+p}{9^k},\frac{w+1-p}{9^k}:\ k\ge0,\ w\in W_k\Bigr\}.
\]
In particular $L(3/4)$ is the union of $E$ and a countable set, and its
Hausdorff dimension is $\log2/\log9$.
\end{theorem}

\begin{proof}
First we show that the right side is contained in $L(3/4)$. By
Lemma~\ref{Efacts}, $T_q=3/4$ on $E$, and \eqref{Eblock} reads
$T_q((w+u)/9^k)-3/4=(T_q(u)-3/4)/9^k$ for $w\in W_k$ and $0\le u\le1$.
Hence it suffices to show that $T_q(p)=3/4$; then also $T_q(1-p)=T_q(p)=3/4$.
The first six digits of $p$ form the word of $252=(100100)_3$, and the
remaining digits give $t_6=(0.\overline{20})_3=3/4$. Since $s_3(252)=2$,
equation~\eqref{cell} with $n=6$ gives
\[
 T_q(p)=T_q\Bigl(\frac{252}{729}\Bigr)+\frac{8\cdot\frac34+T_q(3/4)}{729}
 =\frac{60}{81}+\frac{27}{4\cdot729}=\frac34,
\]
where $T_q(252/729)=T_q(28/81)=\xi_4(28)/81=60/81$ by \eqref{samp}, because
$S_3(28)=S_3(27)+s_3(27)=82$, and $T_q(3/4)=3/4$ because $3/4\in E$.

Now let $x\in L(3/4)$. By \eqref{samp} the values of $T_q$ at ternary
rationals have denominators that are powers of $3$, so $x$ is not a
ternary rational, its digits are unique, and $t_n\in(0,1)$ for all $n$.
Put $y_n=3^n(3/4-T_q(a_n))$ for $n\ge0$. By \eqref{lcell} with
$u=t_n$,
\[
 y_n=2\kappa_nt_n+T_q(t_n),
\]
so $y_n$ lies in the range $\mathcal T(2\kappa_n)$ of the function
$u\mapsto2\kappa_nu+T_q(u)$ on $[0,1]$. For $k\ge0$ we have
$\mathcal T(2k)=[0,2k+3^{-k}]$: the minimum $0$ is attained at $u=0$
because $T_q\ge0$, and the maximum is given by the last statement of Lemma~\ref{lem:tilt}. Replacing $u$
by $1-u$ and using $T_q(1-u)=T_q(u)$ shows that $\mathcal T(2k)=2k+\mathcal T(-2k)$,
that is, $\mathcal T(2k)=[2k,3^{-|k|}]$ for $k<0$. Moreover,
\eqref{lcell} with $u=\xi_{n+1}/3$ gives
$T_q(a_{n+1})=T_q(a_n)+3^{-n-1}\bigl(2\kappa_n\xi_{n+1}+3T_q(\xi_{n+1}/3)\bigr)$,
and $3T_q(\xi/3)$ is $0$ for $\xi=0$ and $2$ for $\xi\in\{1,2\}$. Hence
$\kappa_{n+1}=\kappa_n+1-\xi_{n+1}$ and
\begin{equation}
 y_{n+1}=\begin{cases}
   3y_n,&\xi_{n+1}=0,\\
   3y_n-2\kappa_n-2,&\xi_{n+1}=1,\\
   3y_n-4\kappa_n-2,&\xi_{n+1}=2.
  \end{cases}
 \label{lrec}
\end{equation}
Since $3/4=T_q(a_n)+3^{-n}y_n$ and
$T_q([a_n,c_n])=T_q(a_n)+3^{-n}\mathcal T(2\kappa_n)$ by \eqref{lcell},
the condition $y_n\in\mathcal T(2\kappa_n)$ says exactly that $T_q$ takes
the value $3/4$ somewhere on $[a_n,c_n]$. We start from
$(\kappa_0,y_0)=(0,3/4)$ and follow \eqref{lrec} digit by digit, keeping
only the digits $\xi_{n+1}$ for which
$y_{n+1}\in\mathcal T(2\kappa_{n+1})$. For each pair $(\kappa_n,y_n)$ that
occurs with $\kappa_n\ge0$, Table~\ref{statetab} gives the pairs
$(\kappa_{n+1},y_{n+1})$ produced by the three digits, and the excluded
pairs are printed in gray. For instance, in the row of $(1,\frac94)$ the
digit $0$ gives $(2,\frac{27}4)$, and $\frac{27}4$ exceeds the largest
element $4+\frac19$ of $\mathcal T(4)$.

\begin{table}[H]
\centering
\caption{The pairs $(\kappa_n,y_n)$ with $\kappa_n\ge0$ in the proof of
Theorem~\ref{threequarters} and the pairs $(\kappa_{n+1},y_{n+1})$ given
by the three digits. The gray pairs, those with
$y_{n+1}\notin \mathcal T(2\kappa_{n+1})$, are excluded.}
\label{statetab}
\newcommand{\gr}[1]{\textcolor{black!50}{#1}}
\renewcommand{\arraystretch}{1.25}
\setlength{\tabcolsep}{8pt}
\begin{tabular}{@{}c@{\qquad}ccc@{}}
\toprule
 & \multicolumn{3}{c}{$(\kappa_{n+1},y_{n+1})$}\\
\cmidrule(l){2-4}
$(\kappa_n,y_n)$ & $\xi_{n+1}=0$ & $\xi_{n+1}=1$ & $\xi_{n+1}=2$\\
\midrule
$(0,\frac34)$ & $(1,\frac94)$ & $(0,\frac14)$ & $(-1,\frac14)$\\
$(1,\frac94)$ & \gr{$(2,\frac{27}4)$} & \gr{$(1,\frac{11}4)$} & $(0,\frac34)$\\
$(0,\frac14)$ & $(1,\frac34)$ & \gr{$(0,-\frac54)$} & $(-1,-\frac54)$\\
$(1,\frac34)$ & $(2,\frac94)$ & \gr{$(1,-\frac74)$} & \gr{$(0,-\frac{15}4)$}\\
$(2,\frac94)$ & \gr{$(3,\frac{27}4)$} & $(2,\frac34)$ & \gr{$(1,-\frac{13}4)$}\\
$(2,\frac34)$ & $(3,\frac94)$ & \gr{$(2,-\frac{15}4)$} & \gr{$(1,-\frac{31}4)$}\\
$(3,\frac94)$ & $(4,\frac{27}4)$ & \gr{$(3,-\frac54)$} & \gr{$(2,-\frac{29}4)$}\\
$(4,\frac{27}4)$ & \gr{$(5,\frac{81}4)$} & \gr{$(4,\frac{41}4)$} & $(3,\frac94)$\\
\bottomrule
\end{tabular}
\end{table}

Here we used $\mathcal T(0)=[0,1]$, $\mathcal T(2)=[0,7/3]$, $\mathcal T(4)=[0,4+1/9]$, $\mathcal T(6)=[0,6+1/27]$,
$\mathcal T(8)=[0,8+1/81]$, $\mathcal T(10)=[0,10+1/243]$, and $\mathcal T(-2)=[-2,1/3]$. The pairs
with $\kappa_n<0$ are reduced to this table by symmetry. Indeed, if $x$ is
replaced by $1-x$, then every digit $\xi_j$ is replaced by $2-\xi_j$,
$\kappa_n$ changes sign, and $y_n$ is replaced by $y_n-2\kappa_n$,
because the interval of level $n$ containing $1-x$ is $[1-c_n,1-a_n]$ and
$T_q(1-c_n)=T_q(c_n)=T_q(a_n)+2\kappa_n3^{-n}$. Thus the pair
$(-1,\tfrac14)$ behaves like $(1,\tfrac94)$ with complemented digits,
and $(-1,-\tfrac54)$ behaves like $(1,\tfrac34)$. Figure~\ref{statepic}
shows, as a diagram, the digits of Table~\ref{statetab} that are not
excluded, together with the pairs with $\kappa_n<0$ obtained in this way.

\begin{figure}[H]
\centering
\includegraphics{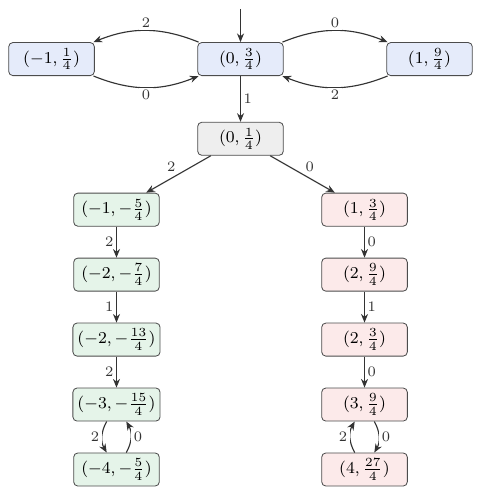}
\caption{The case $q=3$: the pairs $(\kappa_n,y_n)$ in the proof of
Theorem~\ref{threequarters}. An arrow labeled $\xi$ leads from
$(\kappa_n,y_n)$ to $(\kappa_{n+1},y_{n+1})$ when $\xi_{n+1}=\xi$.}
\label{statepic}
\end{figure}

Consequently the digit sequence of $x$ is an infinite path in this diagram
starting at $(0,\tfrac34)$. From $(0,\tfrac34)$ the digits $0$ and $2$
are followed by the forced digits $2$ and $0$, respectively, which lead
back to $(0,\tfrac34)$. The digit $1$ leads to $(0,\tfrac14)$, from
where the digit $0$ is followed by the forced digits
$0,1,0,0,2,0,2,0,2,\ldots$, and the digit $2$ by their complements
$2,1,2,2,0,2,0,2,0,\ldots$. Hence either the digit sequence of $x$
consists of infinitely many strings $02$ and $20$, so that $x\in E$, or
it consists of $k\ge0$ such strings, forming the word of some $w\in W_k$,
followed by $1,0,0,1,0,0,2,0,2,0,\ldots$ or by $1,2,2,1,2,2,0,2,0,2,
\ldots$, so that $x=(w+p)/9^k$ or $x=(w+1-p)/9^k$. In
Figure~\ref{statepic}, the first case consists of the paths that stay in
the two loops at the top, and the second of the paths that leave them by
the digit $1$ and end in one of the two cycles at the bottom.

Finally, the set of the points $(w+p)/9^k$ and $(w+1-p)/9^k$ with $k\ge0$
and $w\in W_k$ is countable, so $L(3/4)$ and
$E$ have the same Hausdorff dimension, which is $\log2/\log9$ by
Lemma~\ref{Efacts}.
\end{proof}

\begin{example}
\label{ex:quarterlevel}
Let $q=3$.
Theorem~\ref{threequarters} determines the level set at height $1/4$:
it consists of the two points $3p-1=13/324$ and $2-3p=311/324$. Indeed,
the points of $E$ and the points $(w+p)/9^k$ and $(w+1-p)/9^k$ with
$k\ge1$ and $w\in W_k$ have the first two digits $02$ or $20$, so they
lie in $[2/9,1/3]\cup[2/3,7/9]$, and $T_q(1/3)=T_q(2/3)=2/3$. Hence $p$ and
$1-p$ are the only points of $L(3/4)$ in $[1/3,2/3]$. By \eqref{branches},
$T_q((1+t)/3)=(2+T_q(t))/3$ for $0\le t\le1$, so $(1+t)/3\in L(3/4)$ exactly
when $t\in L(1/4)$. Therefore $L(1/4)=\{3p-1,2-3p\}$. Figure~\ref{quarterpic}
shows the two points of $L(1/4)$ and the points $p$ and $1-p$, and the
dotted lines mark the middle third $[1/3,2/3]$.
\end{example}

\begin{figure}[t]
\centering
\includegraphics{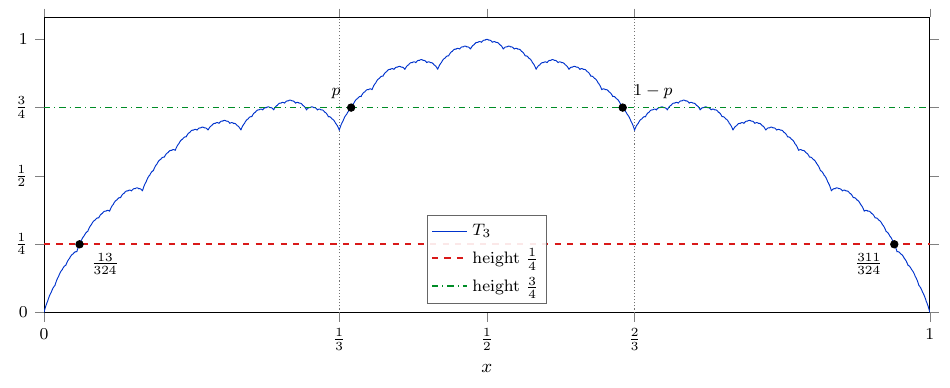}
\caption{The case $q=3$: the level set at height $1/4$ and the points $p$ and $1-p$ of
$L(3/4)$ (see Example~\ref{ex:quarterlevel}).}
\label{quarterpic}
\end{figure}

To measure sets at small scales we use, besides the Hausdorff
dimension, the following notions of
dimension~\cite[Sections~1.2 and~2.1]{Fraser2021}.

\begin{definition}
\label{assouad}
Let $F$ be a nonempty bounded subset of $\R$ or of $\R^2$. For $r>0$ let
$N_r(F)$ be the smallest number of sets of diameter at most $r$ that
cover $F$.

\Needspace{8\baselineskip}
The lower and upper box dimensions of $F$ are the lower and
upper limits of $\log N_r(F)/\log(1/r)$ as $r\downarrow0$. The Assouad
dimension of $F$ is the infimum of the numbers $\alpha\ge0$ for which
there is a constant $C$ such that
\begin{equation}
 N_r\bigl(F\cap B(z,R)\bigr)\le C\Bigl(\frac Rr\Bigr)^{\alpha}
 ~\text{for all }z\in F\text{ and }0<r<R,
 \label{acount}
\end{equation}
where $B(z,R)$ is the closed ball of radius $R$ about $z$.
\end{definition}

\begin{remark}
\label{rem:assouad}
Box dimensions describe covers of the whole set $F$ by sets of diameter
at most $r$ as $r\downarrow0$, and the Assouad dimension describes covers
of $F\cap B(z,R)$ by such sets, uniformly in $z\in F$ and in the two
scales $r<R$. Fraser~\cite[Section~1.2]{Fraser2021}
noted that the Hausdorff dimension of $F$ is at most its lower box
dimension, which is at most its upper box dimension. If $R$ is at least the diameter
of $F$, then $F\cap B(z,R)=F$, so \eqref{acount} shows that the upper box
dimension is at most the Assouad dimension.
\end{remark}

The countable part of $L(3/4)$ does not change the Hausdorff dimension,
but it could change the Assouad dimension, since the countable set
$\{0\}\cup\{1/k:k\ge1\}$ has Assouad
dimension~$1$~\cite[Theorem~2.1.1]{Fraser2021}. By \eqref{Estep} and
Theorem~\ref{threequarters}, the compact set $F=L(3/4)$ satisfies
\[
 F=\Bigl\{\frac{2+t}9:t\in F\Bigr\}\cup\Bigl\{\frac{6+t}9:t\in F\Bigr\}
 \cup\{p,1-p\},
\]
where $\{p,1-p\}\subseteq(1/3,2/3)$. The right side, as a function of
$F$, multiplies the Hausdorff distance between two nonempty compact
subsets of $\R$ by at most $1/9$, where this distance is the smallest
$r\ge0$ such that each of the two sets lies within distance $r$ of the
other. Hence $L(3/4)$ is the only
nonempty compact set $F\subseteq\R$ that satisfies this equation. Thus
$L(3/4)$ is the set that K\"aenm\"aki and
Lehrb\"ack~\cite{KaenmakiLehrback2017} associated with the maps $t\mapsto(2+t)/9$ and $t\mapsto(6+t)/9$ and
the condensation set $\{p,1-p\}$, the set added to the two images in this
equation. In their definition of the condensation open set
condition~\cite[page~175]{KaenmakiLehrback2017}, a misprint requires that
the images of the open set under two different maps meet. They noted that
without the condensation set their condition is the open set condition,
which requires these images to be disjoint, and we use their condition in
this form. It holds with the open set
$(0,1)$: the images $(2/9,1/3)$ and $(2/3,7/9)$ of $(0,1)$ are disjoint
subsets of $(0,1)$, and $p$ and $1-p$ lie in $(0,1)$ outside the closures of
these images. Hence their theorem~\cite[Theorem~4.1]{KaenmakiLehrback2017}
gives the first statement of the next proposition. We also include a short direct
proof.

\begin{proposition}
\label{assouadlevel}
Let $q=3$.
The level set $L(3/4)$ has Assouad dimension $\log2/\log9$. In
particular, its lower and upper box dimensions are $\log2/\log9$.
\end{proposition}

\begin{proof}
Put $s=\log2/\log9$, so that $9^s=2$. For $k\ge0$ let
$P_k$ be the set of the $2(2^k-1)$ points $(w+p)/9^j$ and $(w+1-p)/9^j$
with $j<k$ and $w\in W_j$. By Theorem~\ref{threequarters} and
Lemma~\ref{Efacts}, $L(3/4)$ is contained in the union of $P_k$ and the intervals $E_w$ with $w\in W_k$,
and by \eqref{Eblock}, $L(3/4)\cap E_w$ is the image of $L(3/4)$ under
$u\mapsto(w+u)/9^k$.

We claim that every point of $P_k$ has distance at least
$\frac{13}{108}9^{-k}$ from every other point of $P_k$ and from every
interval $E_w$ with $w\in W_k$. Consider the point $(w+p)/9^j$ with
$j<k$ and $w\in W_j$. It lies in $E_w$, at distance $p\,9^{-j}$ from the
left end and $(1-p)9^{-j}$ from the right end, so every point outside
$E_w$, and every interval disjoint from $E_w$, has distance at least
$p\,9^{-j}$ from $(w+p)/9^j$. Let $j\le i\le k$ and $m\in W_i$, and let
$m_0\in W_j$ be the integer whose word consists of the first $2j$ letters
of the word of $m$, so that $E_m\subseteq E_{m_0}$. If $m_0\ne w$, then
$E_{m_0}$ is disjoint from $E_w$ by Lemma~\ref{Efacts}. If $m_0=w$ and
$i>j$, then the next two letters of the word of $m$ form the string $02$
or $20$, so that $E_m$ lies in $E_{9w+2}=w/9^j+9^{-j}[2/9,1/3]$ or in
$E_{9w+6}=w/9^j+9^{-j}[2/3,7/9]$. This applies to every interval
$E_{w'}$ with $w'\in W_k$, where $i=k>j$, and to the points $(m+p)/9^i$
and $(m+1-p)/9^i$ of $P_k$ with $i\ge j$, which lie in $E_m$. For $i=j$
and $m=w$ these two points are $(w+p)/9^j$ itself and $(w+1-p)/9^j$. The
points of $P_k$ with an index smaller than $j$ lie outside $E_w$. Each of
them has the digit $1$ at an odd place less than $2j$. These points are not
ternary rationals, because the digits of $p$ and of $1-p$ are not eventually
constant. The first $2j$ digits of a point of $E_w$ that is not a ternary
rational form the word of $w$, which has no letter $1$. Since
$p=337/972$, $p-1/3=13/972$, $2/3-p=311/972$, and $1-2p=298/972$, the distance in
question is at least $\frac{13}{972}9^{-j}\ge\frac{13}{108}9^{-k}$.
The point $(w+1-p)/9^j$ is treated in the same way, since the reflection of $E_w$ in its midpoint exchanges it with $(w+p)/9^j$ and exchanges $E_{9w+2}$ with $E_{9w+6}$.

Let $0<\delta<1$, and choose $j\ge1$ with $9^{-j}\le\delta<9^{1-j}$. Then
$L(3/4)$ is covered by the $2^j$ intervals $E_w$ with $w\in W_j$ and the
$2(2^j-1)$ points of $P_j$, so $N_\delta(L(3/4))<3\cdot2^j=6\cdot9^{s(j-1)}
<6\delta^{-s}$.

\begin{figure}[H]
\centering
\includegraphics{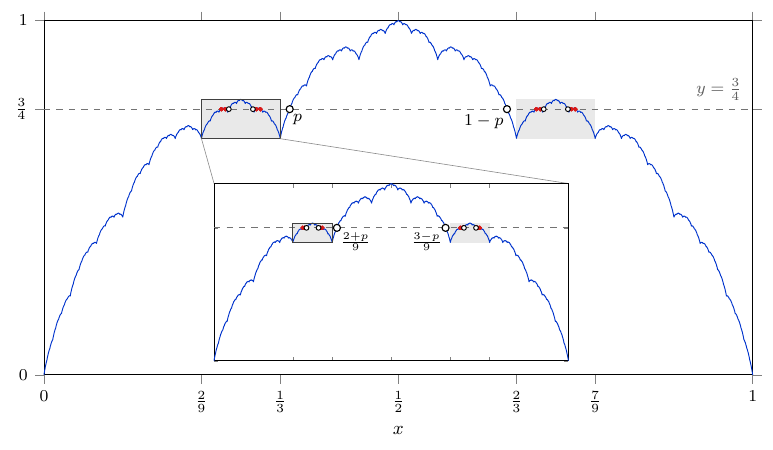}
\caption{The case $q=3$: the level set $L(3/4)$ on the graph of $T_q$, with points of $E$ (dots)
and the points $(w+p)/9^k$ and $(w+1-p)/9^k$ with $w\in W_k$ (circles), for
$k\le1$ in the main panel and $k\le2$ in the inset. The inset shows the
left rectangle enlarged.}
\label{levelpic}
\end{figure}

Now let $x\in L(3/4)$ and $0<r<R$. Since $L(3/4)\subseteq[0,1]$, replacing
$R>1$ by $1$ does not change $L(3/4)\cap[x-R,x+R]$, and $N_r(L(3/4))=1$
for $r\ge1$. Hence we may assume that $r<R\le1$, and we choose $k\ge0$
with $9^{-k-1}<R\le9^{-k}$. The interval
$[x-R,x+R]$ has length less than $3\cdot9^{-k}$, so it meets at most one
interval $E_w$ with $w\in W_k$, and by the claim it contains at most
$2R\cdot\frac{108}{13}9^k+1<18$ points of $P_k$. As $9^kr<9^kR\le1$,
we obtain
\[
 N_r\bigl(L(3/4)\cap[x-R,x+R]\bigr)<18+N_{9^kr}(L(3/4))
 <18+6(9^kr)^{-s}=18+6(9^kR)^{-s}\Bigl(\frac Rr\Bigr)^s
 <30\Bigl(\frac Rr\Bigr)^s,
\]
because $(9^kR)^{-s}<9^s=2$ and $R/r>1$. Hence the Assouad dimension of
$L(3/4)$ is at most $s$. It is at least the upper box dimension, which is at
least the lower box dimension, which is at least the Hausdorff dimension
$s$.
\end{proof}

In Figure~\ref{levelpic}, the dots are points of $E$. In the main panel,
the circles are the points $p$, $1-p$, $(w+p)/9$,
and $(w+1-p)/9$ with $w\in W_1=\{2,6\}$, and each shaded rectangle is the
product of one of the intervals $E_2=[2/9,1/3]$ and $E_6=[2/3,7/9]$ with
the range $[2/3,7/9]$ of $T_q$ over it. By
\eqref{Estep}, the scaling by $1/9$ about $(1/4,3/4)$ maps the graph of
$T_q$ over $[0,1]$ onto the part of
the graph above $[2/9,1/3]$, and the map $t\mapsto(2+t)/9$ maps $L(3/4)$
onto $L(3/4)\cap[2/9,1/3]$. The inset, which enlarges the left
rectangle, therefore repeats the main picture. In it, $(2+p)/9$ and $(3-p)/9$
take the place of $p$ and $1-p$, and the points $(w+p)/81$ and $(w+1-p)/81$
with $w\in\{20,24\}$ take the place of $(w+p)/9$ and $(w+1-p)/9$ with
$w\in\{2,6\}$. For every $k\ge0$ and $w\in W_k$, the proof of
Proposition~\ref{assouadlevel} uses the fact that $u\mapsto(w+u)/9^k$
maps $L(3/4)$ onto $L(3/4)\cap E_w$.

We return to an arbitrary base $q$. The proof of Theorem~\ref{levelthm}
uses two facts about the indices $n$ with $d_n(x)=0$. The first,
Lemma~\ref{dzero},
follows from a criterion for sums of independent random variables to be
$0$ infinitely often. Durrett proved for the sums of independent,
identically distributed random vectors in $\R^d$ that three conditions
are equivalent. We state, for $d=1$, the implication that we use.

\begin{theorem}[Durrett {\cite[Theorem~5.4.3]{Durrett2019}}]
\label{durrett}
Let $Z_1,Z_2,\ldots$ be independent, identically distributed random
variables with real values, and put $S_0=0$ and $S_n=Z_1+\cdots+Z_n$ for
$n\ge1$. If $\sum_{n=0}^{\infty}\mathbb{P}(S_n=0)=\infty$, then with probability
$1$ we have $S_n=0$ for infinitely many $n$.
\end{theorem}

\begin{lemma}
\label{dzero}
For almost every $x\in[0,1)$, we have $d_n(x)=0$ for infinitely many $n$.
\end{lemma}

\begin{proof}
For every word $u$ of length $n$ over $\{0,1,\ldots,q-1\}$, the set of
$x\in[0,1)$
whose first $n$ digits form $u$ is an interval of length $q^{-n}$. Hence,
with respect to Lebesgue measure on $[0,1)$, the digits $\xi_1,\xi_2,\ldots$
are independent random variables, each of which takes the values
$0,1,\ldots,q-1$ with probability $1/q$. Thus $d_n(x)$ is the sum $S_n$ of
Theorem~\ref{durrett} for $Z_j=q-1-2\xi_j$, and the set of $x$ with
$d_n(x)=0$ has measure $C_n/q^n$, where $C_n$ is the number of words
of length $n$ over $\{0,1,\ldots,q-1\}$ with digit slope $d_n=0$. By
Lemma~\ref{trinom}, $\sqrt n\,C_n/q^n$ converges to a positive limit as
$n\to\infty$ through the even integers, so
$\sum_{n=0}^{\infty}C_n/q^n=\infty$. By Theorem~\ref{durrett}, almost every
$x\in[0,1)$ has $d_n(x)=0$ for infinitely many $n$.
\end{proof}

The second fact, Lemma~\ref{flip}, is that the digits of a point $x$
between two indices $n<n'$ with $d_n(x)=d_{n'}(x)=0$ may be replaced by
their complements $q-1-\xi$ without changing the value of $T_q$. Such a replacement changes only the
signs of some digit slopes, and the second equation in \eqref{dsq}
contains only their squares. For $q=3$, Figure~\ref{foldpic} illustrates
the lemma.

\begin{lemma}
\label{flip}
Let $x\in[0,1)$ be a point that is not a $q$-adic rational, and let
$0=n_0<n_1<n_2<\cdots$ be indices $n$ with $d_n(x)=0$, finitely or
infinitely many; if there are finitely many, say $n_0<\cdots<n_K$, put
$n_{K+1}=\infty$. Let $x'$ be the point obtained from $x$ by replacing,
for some of the indices $k$, every digit $\xi_j$ with $n_k<j\le n_{k+1}$
by $q-1-\xi_j$. Then $x'$ is not a $q$-adic rational and $T_q(x')=T_q(x)$.
Moreover, $d_j(x')=d_j(x)$ if $n_k<j\le n_{k+1}$ and the digits
with these indices are kept, and $d_j(x')=-d_j(x)$ if they are
replaced. For $q=3$ the same holds with $\kappa_j$ in place of $d_j$.
\end{lemma}

\begin{proof}
If $n_{k+1}<\infty$, then
$\sum_{j=n_k+1}^{n_{k+1}}(q-1-2\xi_j)=d_{n_{k+1}}(x)-d_{n_k}(x)=0$, and
replacing every $\xi_j$ by $q-1-\xi_j$ changes the sign of each term of
this sum, so the complements satisfy the same equation. In particular
neither these digits nor their complements are all $0$ or all $q-1$, and
so the digits of $x'$ are not eventually all $0$ or all $q-1$ if there
are infinitely many indices $n_k$. If there are finitely many, the digits of
$x'$ after the place $n_K$ are those of $x$ or their complements, and
again they are not eventually all $0$ or all $q-1$, because $x$ is not a
$q$-adic rational. Since every digit sequence of a $q$-adic rational is
eventually all $0$ or all $q-1$, the point $x'$ is not a $q$-adic
rational, and these digits form its digit sequence.

Now let $n_k<j\le n_{k+1}$. By the first paragraph,
$d_{n_k}(x')=d_{n_k}(x)=0$, so $d_j(x')$ and $d_j(x)$ are the sums of the
steps $q-1-2\xi$ over the digits $\xi$ of $x'$ and of $x$ with indices
$n_k<i\le j$, and replacing these digits by their complements changes the
sign of the sum. Hence $d_j(x')$ is $d_j(x)$ or $-d_j(x)$ depending on
whether the digits with indices $n_k<i\le n_{k+1}$ are kept or replaced,
and in both cases $d_j(x')^2=d_j(x)^2$. Since every $j\ge1$ satisfies
$n_k<j\le n_{k+1}$ for some $k$, the second equation in \eqref{dsq} gives
$T_q(x')=T_q(x)$.
\end{proof}

For $q=3$, the point $x=(0.0220\overline{01})_3$ of Figure~\ref{foldpic}
has $d_n(x)=0$ exactly for $n=0$, $2$, and $4$, and $\kappa_j(x)<0$ only
for $j=3$.
Replacing $\xi_3=2$ and $\xi_4=0$ by their complements gives
$x'=(0.0202\overline{01})_3$, with
$\kappa_j(x')=|\kappa_j(x)|$ for all $j$. In the left panel, the shaded
strip marks these two digits, and the dashed line shows $\kappa_j(x')$
where it differs from $\kappa_j(x)$. By Lemma~\ref{flip}, the eight points
obtained from $x$ by complementing the digits $\xi_j$ in some of the
ranges $1\le j\le2$, $3\le j\le4$, and $j\ge5$ lie in $L(T_q(x))$. The
right panel shows the four of them that lie in the ternary interval
$[2/9,1/3]$ of level $2$ containing $x$. They form two pairs, and the two
points of each pair are reflections of each other in the midpoint of one
of the ternary intervals $[20/81,21/81]$ and $[24/81,25/81]$, which have
level $4$ and slope $0$. The other four lie in $[2/3,7/9]$, because
$\xi_1$ and $\xi_2$ are complemented for them. The line $y=T_q(x)$ also meets the graph at points that are
not obtained from $x$ in this way, for instance in $(21/81,22/81)$, since
$T_q(21/81)<T_q(x)<T_q(22/81)$.

\begin{figure}[H]
\centering
\includegraphics{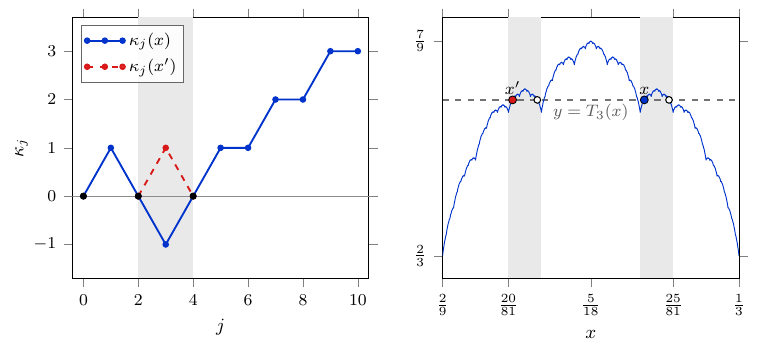}
\caption{The case $q=3$. Left: $\kappa_j(x)$ for $x=(0.0220\overline{01})_3$ and
$\kappa_j(x')$ for $x'=(0.0202\overline{01})_3$. Right: $T_q$ on $[2/9,1/3]$
and four points of $L(T_q(x))$ given by Lemma~\ref{flip}.}
\label{foldpic}
\end{figure}

\begin{proof}[Proof of Theorem~\ref{levelthm}(b)]
Let $1\le n_1<n_2<\cdots$ be indices $n$ with $d_n(x)=0$ and
$2\xi_{n+1}\ne q-1$, and
put $n_0=0$. For a sequence $\epsilon=(\epsilon_1,\epsilon_2,\ldots)$ of
zeros and ones, let $x_\epsilon$ be the point obtained from $x$ by
replacing the digits $\xi_j$ with $n_k<j\le n_{k+1}$ by $q-1-\xi_j$ for
every $k\ge1$ with $\epsilon_k=1$. By Lemma~\ref{flip},
$T_q(x_\epsilon)=T_q(x)$, and $x_\epsilon$ is not a $q$-adic rational, so the
digits that define $x_\epsilon$ form its only digit sequence. If
$\epsilon\ne\epsilon'$, choose $k$ with $\epsilon_k\ne\epsilon'_k$. Then
the digits of $x_\epsilon$ and $x_{\epsilon'}$ at the place $n_k+1$ are
$\xi_{n_k+1}$ and $q-1-\xi_{n_k+1}$ in some order, and these digits differ
because $2\xi_{n_k+1}\ne q-1$. Hence $x_\epsilon\ne x_{\epsilon'}$, and
$L(T_q(x))$ contains uncountably many points.

By Lemma~\ref{dzero}, almost every $x$ has $d_n(x)=0$ for infinitely
many $n$, and the $q$-adic rationals form a countable set. This proves the
last statement for even $q$. Let $q$ be odd. Suppose that $\xi_{n+1}$ is different
from $(q-1)/2$ for only finitely many indices $n$ with $d_n(x)=0$. Then,
from some point on, $\xi_{n+1}=(q-1)/2$ for every such $n$. This gives
$d_{n+1}(x)=0$, so the digits of $x$ are eventually
$(q-1)/2$. Such points form a countable set, so almost
every $x$ has the stated property for odd $q$ as well.
\end{proof}

\begin{example}
\label{ex:levelthm}
Let $q=3$.
The point $1/4=(0.\overline{02})_3$ has $d_n(1/4)=0$ exactly for
$n=0,2,4,\ldots$, and its digit $\xi_{n+1}$ is $0$ for each of these $n$.
By Theorem~\ref{levelthm}(b),
$L(T_q(1/4))=L(3/4)$ is uncountable, in agreement with
Theorem~\ref{threequarters}. The point $5/18=(0.02\overline1)_3$ has
$d_n(5/18)=0$ exactly for $n=0,2,3,4,\ldots$, but its digit $\xi_{n+1}$ is
$1$ for each of these $n\ge2$, so the theorem does not apply to this
point. In fact $L(T_q(5/18))=L(7/9)$
has only four points (see Theorem~\ref{cardmod}).
\end{example}

We prove Theorem~\ref{levelthm}(a) using one family of intervals of slope
$0$ and countably many families of two other kinds of intervals. Within
each family, the lengths of their images under $T_q$ have a finite sum. Since there are countably many families, almost every height
lies in the images of only finitely many members of each family. Let $y$
be such a height, and suppose that $y$ is not a value of $T_q$ at a
$q$-adic rational and that $L(y)$ is infinite. The set $L(y)$ is a closed
subset of $[0,1]$, so it contains a point $x$ that is the limit of a
sequence of other points of $L(y)$. Definition~\ref{pairs} collects the
intervals of the first family, and Lemma~\ref{aplus} gives the sum of
their lengths. If $d_n(x)=0$ for infinitely many $n$, then the point of
$L(y)$ given by Lemma~\ref{foldzero} lies in infinitely many of these
intervals, so that $y$ lies in the images of infinitely many of them.
When $d_n(x)=0$ for finitely many $n$ and $|d_n(x)|$ does not tend to
$\infty$, the point $x$ itself lies in infinitely many members of one
family of the second kind, and Lemma~\ref{slopevisits} gives a bound for
the sums of these families. In the remaining case $d_n(x)$ tends to $+\infty$ or to $-\infty$, and $T_q$
does not have an infinite derivative at $x$, because $x$ is the limit of a
sequence of other points of $L(y)$. By Theorem~\ref{thm2}(c), the digits
of the point $x$ or of the point $1-x$ then contain infinitely many
strings of zeros whose lengths are at least $(d-C)/(q-1)-1$, where $d$ is
the digit slope of this point at the place before the string and $C$ is a
constant. Such strings give infinitely many members of one family of the
third kind that contain this point, and the sums of the families of this
kind are bounded by means of Lemma~\ref{stayabove}. For $q\le3$,
Proposition~\ref{finitelevel} below uses the first family together with
Lemma~\ref{monopos} and bounds the number of points of $L(y)$ in terms of
the number of members of this family whose images under $T_q$ contain
$y$.

The proofs of Lemmas~\ref{aplus} and~\ref{slopevisits} use the following
bound for sums of the numbers $q^{-|w|}$ over sets of words $w$.

\begin{lemma}
\label{wordsum}
Let $\mathcal W$ be a set of words over $\{0,1,\ldots,q-1\}$ such that no
word of $\mathcal W$ consists of the first letters of another word of
$\mathcal W$. Then the sets of the points of $[0,1)$ whose digits begin
with $w$, for $w\in\mathcal W$, are pairwise disjoint intervals of lengths
$q^{-|w|}$, and $\sum_{w\in\mathcal W}q^{-|w|}\le1$.
\end{lemma}

\begin{proof}
Let $w\in\mathcal W$ have length $k$, and let $m$ be the integer whose
word of length $k$, with leading zeros allowed, is $w$. The first $k$
digits of a point $x\in[0,1)$ form the word of $\lfloor q^kx\rfloor$ of
length $k$, so they form $w$ exactly when $m/q^k\le x<(m+1)/q^k$. If a
point lay in the intervals of two different words of $\mathcal W$, both
words would consist of the first digits of this point. Since they are
different, they would have different lengths, and the shorter one would
consist of the first letters of the longer one. Hence the intervals are
pairwise disjoint subsets of $[0,1)$, and the sum of their lengths is at
most $1$.
\end{proof}

\begin{definition}
\label{pairs}
Let $\mathcal A$ be the set of pairs $(n,m)$ of integers with $n\ge0$
and $0\le m<q^n$ such that the word $v$ of $m$, of length $n$ with leading
zeros allowed, satisfies $d_n(v)=0$ and $d_{j+1}(v)\ge0$ for every index
$j<n$ with $d_j(v)=0$. For $y\in[0,\Lambda(0)]$ let $\mathcal A(y)$ be the set
of pairs
$(n,m)\in\mathcal A$ such that $T_q$ takes the value $y$ on
$[m/q^n,(m+1)/q^n]$.
\end{definition}

Let $n\ge0$ and $0\le m<q^n$, and let $v$ be the word of $m$ of length
$n$. For $j\le n$, the left endpoint of the interval of level $j$
containing $[m/q^n,(m+1)/q^n]$ has the first $j$ letters of $v$ as its
digits, so this interval has slope $d_j(v)$, by \eqref{dnsq}. Thus
$(n,m)\in\mathcal A$ means that $[m/q^n,(m+1)/q^n]$ has slope $0$ and
that, whenever an interval of level $j<n$ containing it has slope $0$,
the interval of level $j+1$ containing it has a nonnegative slope. If
$q\le3$, the second condition holds exactly when every interval
containing $[m/q^n,(m+1)/q^n]$ has a nonnegative slope, by
Remark~\ref{rem:signs}. For $q=3$ this means that for every $j\le n$ the
first $j$ letters of $v$ contain at least as many zeros as twos, and that
$v$ has as many zeros as twos. For $(n,m)\in\mathcal A$ we have
$T_q([m/q^n,(m+1)/q^n])=[T_q(m/q^n),T_q(m/q^n)+q^{-n}\Lambda(0)]$ by
Remark~\ref{rem:signs}.

\begin{lemma}
\label{aplus}
If $q$ is odd, then $\sum_{(n,m)\in\mathcal A}q^{-n}=2q/(q-1)$, and if $q$
is even, then $\sum_{(n,m)\in\mathcal A}q^{-n}=2$. In particular,
$\sum_{(n,m)\in\mathcal A}3^{-n}=3$ for $q=3$.
\end{lemma}

\begin{proof}
Let $\mathcal U$ be the set of the nonempty words $u$ over
$\{0,1,\ldots,q-1\}$ with $d_{|u|}(u)=0$ and $d_j(u)\ne0$ for $0<j<|u|$,
and let
$\mathcal U^+$ be the set of the words $u\in\mathcal U$ with
$d_1(u)\ge0$. No word of $\mathcal U$ consists of the first letters of
another word $u\in\mathcal U$, since its length would then be an index
$j$ with $0<j<|u|$ and $d_j(u)=0$. By Lemma~\ref{wordsum}, the sets
of the points of $[0,1)$ whose digits begin with $u$, for
$u\in\mathcal U$, are pairwise disjoint intervals of lengths $q^{-|u|}$.
A point $x\in[0,1)$ satisfies $d_n(x)=0$ for some $n\ge1$ exactly when
its digits begin with a word of $\mathcal U$, namely with its first $n$
digits for the smallest such $n$. Hence the union of these
intervals is the set of the points $x\in[0,1)$ with $d_n(x)=0$ for some
$n\ge1$, which has measure $1$ by Lemma~\ref{dzero}. Thus
$\sum_{u\in\mathcal U}q^{-|u|}=1$. A word of
length $1$ lies in $\mathcal U$ exactly when its letter $\xi$ satisfies
$q-1-2\xi=0$, so there is exactly one such word if $q$ is odd, and none
if $q$ is even. For odd $q$ this word lies in $\mathcal U^+$. For a word $u\in\mathcal U$ of length at least
$2$ we have $d_1(u)\ne0$, and replacing every letter $\xi$ of $u$ by
$q-1-\xi$ gives a word of $\mathcal U$ of the same length whose digit
slopes are those of $u$ with the opposite signs. Exactly one of these two
words lies in $\mathcal U^+$. Therefore
\[
 \varrho=\sum_{u\in\mathcal U^+}q^{-|u|}
 =\begin{cases}
 \dfrac1q+\dfrac12\Bigl(1-\dfrac1q\Bigr)=\dfrac{q+1}{2q}&\text{if $q$ is odd},\\[3mm]
 \dfrac12&\text{if $q$ is even}.
 \end{cases}
\]
Now let $(n,m)\in\mathcal A$ with word $v$, and let
$0=j_0<j_1<\cdots<j_k=n$ be the indices $j$ with $d_j(v)=0$. For
$1\le i\le k$ the
letters of $v$ with the indices $j_{i-1}<j\le j_i$ form a word
$u_i\in\mathcal U$, and $d_1(u_i)=d_{j_{i-1}+1}(v)\ge0$, so
$u_i\in\mathcal U^+$. Conversely, for $k\ge0$ and
$u_1,\ldots,u_k\in\mathcal U^+$, the word $v=u_1\ldots u_k$ has
$d_j(v)=0$ exactly for $j=0$ and for the places $j$ where the words $u_i$
end, it satisfies the two
conditions of Definition~\ref{pairs}, and it determines $k$ and
$u_1,\ldots,u_k$. Since all terms are nonnegative, we obtain
\[
 \sum_{(n,m)\in\mathcal A}q^{-n}
 =\sum_{k=0}^{\infty}\Bigl(\sum_{u\in\mathcal U^+}q^{-|u|}\Bigr)^k
 =\frac1{1-\varrho},
\]
which is $2q/(q-1)$ if $q$ is odd and $2$ if $q$ is even.
\end{proof}

For a point $x\in(0,1)$ that is not a $q$-adic rational, the following
lemma gives a point $x'$ of $L(T_q(x))$ that has $d_n(x')=0$ exactly for
the indices $n$ with $d_n(x)=0$ and satisfies $d_{n+1}(x')\ge0$ for each
of these $n$ (see the left panel of Figure~\ref{foldpic}).

\begin{lemma}
\label{foldzero}
Let $x\in(0,1)$ be a point that is not a $q$-adic rational, let
$0=n_0<n_1<n_2<\cdots$ be all indices $n$ with $d_n(x)=0$, and if there
are finitely many, say $n_0<\cdots<n_K$, put $n_{K+1}=\infty$. Let $x'$
be the point
obtained from $x$ by replacing every digit $\xi_j$ with
$n_k<j\le n_{k+1}$ by $q-1-\xi_j$, for every $k$ with $d_{n_k+1}(x)<0$.
Then $x'$ is not a $q$-adic rational, $T_q(x')=T_q(x)$, the indices $n$
with $d_n(x')=0$ are the indices $n_k$, and
$(n,\lfloor q^nx'\rfloor)\in\mathcal A(T_q(x))$ for every such $n$. If
$q\le3$, then $d_j(x')=|d_j(x)|$ for every $j\ge0$.
\end{lemma}

\begin{proof}
By Lemma~\ref{flip}, $x'$ is not a $q$-adic rational, $T_q(x')=T_q(x)$, and
$d_j(x')$ is $d_j(x)$ or $-d_j(x)$ for every $j$, so that the indices $j$
with $d_j(x')=0$ are the indices $n_k$. Moreover,
$d_{n_k+1}(x')=|d_{n_k+1}(x)|\ge0$ for every $k$, because the digits with
the indices $n_k<j\le n_{k+1}$ are replaced exactly when
$d_{n_k+1}(x)<0$. Let $n$ be an index with $d_n(x')=0$, let
$m=\lfloor q^nx'\rfloor$, and let $v$ be the word of $m$ of length
$n$, which consists of the first $n$ digits of $x'$. Then
$d_j(v)=d_j(x')$ for $j\le n$, so $d_n(v)=0$, and $d_{j+1}(v)\ge0$ for
every index $j<n$ with $d_j(v)=0$. Thus $(n,m)\in\mathcal A$, and $T_q$
takes the
value $T_q(x)$ at the point $x'$ of $[m/q^n,(m+1)/q^n]$. Finally, let
$q\le3$. By Remark~\ref{rem:signs}, the digit slopes $d_j(x)$ with
$n_k<j<n_{k+1}$ are nonzero and all have the same sign, so they are
negative exactly when the digits with the indices $n_k<j\le n_{k+1}$ are
replaced. Hence $d_j(x')=|d_j(x)|$ for every $j$.
\end{proof}

For the points $x$ such that $d_n(x)=0$ for finitely many $n$ and
$|d_n(x)|$ does not tend to $\infty$, the proof of
Theorem~\ref{levelthm}(a) uses the following bound (see Figure~\ref{pathpic}).

\begin{lemma}
\label{slopevisits}
Let $N\ge0$ and $K\ge1$ be integers, and for $n>N$ let $\mathcal V_n$ be
the set of the words $v$ of length $n$ over $\{0,1,\ldots,q-1\}$ with
$d_j(v)\ne0$ for $N<j\le n$ and $|d_n(v)|\le K$. Then
$\sum_{n=N+1}^{\infty}q^{-n}\,\#\mathcal V_n\le Kq^K$.
\end{lemma}

\begin{proof}
Let $\mathcal Z$ be the set of the words $z$ of length greater than $N$
with $d_{|z|}(z)=0$ and $d_j(z)\ne0$ for $N<j<|z|$. No word of
$\mathcal Z$ consists of the first letters of another word
$z\in\mathcal Z$, since its length would then be an index $j$ with
$N<j<|z|$ and $d_j(z)=0$.
Hence
$\sum_{z\in\mathcal Z}q^{-|z|}\le1$, by Lemma~\ref{wordsum}.

Let $n>N$ and $v\in\mathcal V_n$, and let $\zeta=\lfloor(q+1)/2\rfloor$
if $d_n(v)>0$ and $\zeta=q-1-\lfloor(q+1)/2\rfloor$ if $d_n(v)<0$. Then
$q-1-2\zeta$ has the sign opposite to that of $d_n(v)$, and its absolute
value is $2$ if $q$ is odd and $1$ if $q$ is even. If $q$ is odd, then
$d_n(v)$ is even, because $d_n(v)-(q-1)n$ is even. Thus $|d_n(v)|$ is a
multiple of $|q-1-2\zeta|$, and we let $\hat v$ be the word that consists of
$v$ followed by $|d_n(v)|/|q-1-2\zeta|$ letters $\zeta$. The absolute
values of the digit slopes of $\hat v$ with the indices $n,n+1,\ldots,|\hat v|$
decrease in equal steps from $|d_n(v)|$ to $0$, and
$d_j(\hat v)=d_j(v)\ne0$ for $N<j\le n$. Hence $|\hat v|$ is the only index
$j>N$ with $d_j(\hat v)=0$, so that $\hat v\in\mathcal Z$, and
$n<|\hat v|\le n+K$, because $1\le|d_n(v)|\le K$. Different words
$v\in\mathcal V_n$ give different words $\hat v$, because $v$ consists of the
first $n$ letters of $\hat v$. Since $q^{-n}\le q^Kq^{-|z|}$ for every word
$z$ with $|z|\le n+K$, we obtain
\[
 q^{-n}\,\#\mathcal V_n\le q^K\sum_{z\in\mathcal Z,\ n<|z|\le n+K}q^{-|z|}.
\]
We add these inequalities for all $n>N$. Each word $z\in\mathcal Z$ occurs
on the right side for at most $K$ values of $n$, so the sum of the left
sides is at most $Kq^K\sum_{z\in\mathcal Z}q^{-|z|}\le Kq^K$.
\end{proof}

Figure~\ref{pathpic} shows an example for $q=3$, where the letters $0$,
$1$, and $2$ change the digit slope by $2$, $0$, and $-2$, respectively.
Let $N=4$ and $K=4$. The shaded strip marks the places $4<j\le10$, the
dash dotted lines mark the heights $K$ and $-K$, and the black dots mark
the indices $j$ with $d_j=0$. The word $v=0220000121$ lies in
$\mathcal V_{10}$, because its digit slopes $d_j(v)$ in the strip are not
$0$ and $|d_{10}(v)|=4=K$. For this word, the proof of
Lemma~\ref{slopevisits} gives $\hat v=v22$. The dashed line shows the
digit slopes of $\hat v$ with the indices $10$, $11$, and $12$, which
decrease from $4$ to $0$ in two equal steps, so that $12$ is the only
index $j>4$ with $d_j(\hat v)=0$.

\begin{figure}[t]
\centering
\includegraphics{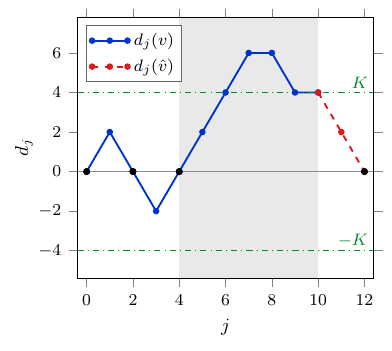}
\caption{The case $q=3$: the digit slopes of a word $v$ as in
Lemma~\ref{slopevisits} and of the word $\hat v$ of its proof.}
\label{pathpic}
\end{figure}

The proof of Theorem~\ref{levelthm}(a) uses the next lemma for the
points $x$ with $d_n(x)\to+\infty$.

\begin{lemma}
\label{stayabove}
Let $h\ge1$ and $d\ge1$ be integers, and for $i\ge0$ let
$\mathcal P_i(h,d)$ be the set of the words $w$ of length $i$ over
$\{0,1,\ldots,q-1\}$ with $h+d_j(w)\ge1$ for $0\le j\le i$ and
$h+d_i(w)=d$. Then $\sum_{i=0}^{\infty}q^{-i}\,\#\mathcal P_i(h,d)\le q^h$.
\end{lemma}

\begin{proof}
For a word $w=w_1\ldots w_i$ let $w^*$ be the word
$(q-1-w_i)(q-1-w_{i-1})\ldots(q-1-w_1)$. Reading the letters from the
last to the first reverses the order of the steps $q-1-2w_j$, and
replacing each letter $\xi$ by $q-1-\xi$ changes their signs, so that
$d_j(w^*)=d_{i-j}(w)-d_i(w)$ for $0\le j\le i$.

\Needspace{8\baselineskip}
If $w\in\mathcal P_i(h,d)$, then $d+d_j(w^*)=h+d_{i-j}(w)\ge1$ for
$0\le j\le i$ and $d+d_i(w^*)=h$, so that $w^*\in\mathcal P_i(d,h)$.
Since $(w^*)^*=w$, the map $w\mapsto w^*$ is a bijection from
$\mathcal P_i(h,d)$ onto $\mathcal P_i(d,h)$, and it suffices to show that
$\sum_{i=0}^{\infty}q^{-i}\,\#\mathcal P_i(d,h)\le q^h$.

Choose $x$ uniformly in $[0,1)$. As in the proof of Lemma~\ref{dzero},
the digits $\xi_1,\xi_2,\ldots$ of $x$ are independent, and each of them
takes the values $0,1,\ldots,q-1$ with probability $1/q$. Let $V$ be the
set of the indices $i\ge0$ with $d+d_j(x)\ge1$ for $0\le j\le i$ and
$d+d_i(x)=h$. An index $i$ lies in $V$ exactly when the first $i$ digits
of $x$ form a word of $\mathcal P_i(d,h)$, which has the probability
$q^{-i}\,\#\mathcal P_i(d,h)$. Hence the sum to be bounded is the expected
number of elements of $V$. For $k\ge1$ let $E_k$ be the event that $V$
has at least $k$ elements. If $i$ is the $k$th element of $V$ and the $h$
digits $\xi_{i+1},\ldots,\xi_{i+h}$ are all equal to $q-1$, then
$d+d_{i+j}(x)=h-(q-1)j$ for $1\le j\le h$. These numbers are smaller
than $h$, and the last of them is at most $0$, so $V$ has no element
larger than $i$. Whether $i$ is the $k$th element of $V$ depends only on
$\xi_1,\ldots,\xi_i$, which are independent of the next $h$ digits.
Summing over $i$, we see that $V$ has exactly $k$ elements with
probability at least $q^{-h}\,\mathbb{P}(E_k)$. Hence
$\mathbb{P}(E_{k+1})\le(1-q^{-h})\,\mathbb{P}(E_k)$, and by induction
$\mathbb{P}(E_k)\le(1-q^{-h})^{k-1}$ for every $k\ge1$. Therefore the
expected number of elements of $V$ is
\[
 \sum_{k=1}^{\infty}\mathbb{P}(E_k)\le\sum_{k=1}^{\infty}(1-q^{-h})^{k-1}=q^h.
\]
\end{proof}

\begin{proof}[Proof of Theorem~\ref{levelthm}(a)]
For a word $v$ of length $n$ over $\{0,1,\ldots,q-1\}$ we write $J(v)$
for the interval $[m/q^n,(m+1)/q^n]$ of level $n$, where $v$ is the word
of $m$. Its slope is $d_n(v)$, by \eqref{dnsq}, and it contains every
point of $[0,1)$ whose first $n$ digits form $v$. By \eqref{cell}, the set
$T_q(J(v))$ consists of the numbers
$T_q(m/q^n)+q^{-n}\bigl(d_n(v)u+T_q(u)\bigr)$ with $0\le u\le1$. Since
$0\le T_q\le\Lambda(0)$ and $T_q$ is continuous, $T_q(J(v))$ is a closed
interval of length at most $q^{-n}\bigl(|d_n(v)|+\Lambda(0)\bigr)$.

We use the following fact three times. If $J_1,J_2,\ldots$ is a countable
family of intervals such that the lengths of the intervals $T_q(J_i)$
have a finite sum, then the set of the heights $y$ that lie in $T_q(J_i)$
for infinitely many $i$ has Lebesgue measure zero. Indeed, for every
$i_0$ this set is contained in the union of the intervals $T_q(J_i)$ with
$i>i_0$. The measure of this union is at most the sum of the lengths of
these intervals, and this sum tends to $0$ as $i_0\to\infty$.

The first family consists of the intervals $J=[m/q^n,(m+1)/q^n]$ with
$(n,m)\in\mathcal A$. For each of them, $T_q(J)$ is a closed interval
of length $q^{-n}\Lambda(0)$, by Remark~\ref{rem:signs}. The number of
elements of $\mathcal A(y)$ is
\[
 \sum_{(n,m)\in\mathcal A}\mathbf 1_{T_q([m/q^n,(m+1)/q^n])}(y),
\]
where $\mathbf 1_B$ is the indicator function of a set $B$. By monotone
convergence and Lemma~\ref{aplus}, its integral over
$0\le y\le\Lambda(0)$ is $\Lambda(0)\sum_{(n,m)\in\mathcal A}q^{-n}$,
which is $q(q+1)/(2(q-1))$ if $q$ is odd and $q(q+2)/(2(q+1))$ if $q$ is
even. Hence $\mathcal A(y)$ is finite for almost every $y$. For $q=3$ the
integral is $3$, so $\mathcal A(y)$ has on average three elements over
$0\le y\le1$.

The second family depends on two integers $N\ge0$ and $K\ge1$. It
consists of the intervals $J(v)$ with $v\in\mathcal V_n$ and $n>N$, in the
notation of Lemma~\ref{slopevisits}. For $v\in\mathcal V_n$ the interval
$T_q(J(v))$ has length at most $q^{-n}(K+\Lambda(0))$, so the sum of these
lengths is at most $(K+\Lambda(0))Kq^K$.

The third family depends on an integer $C\ge0$ and on a word $v_0$ with
$d_{|v_0|}(v_0)\ge1$. Put $h=d_{|v_0|}(v_0)$, and for $d\ge1$ let
$\ell(d)=\max\{0,\lceil(d-C)/(q-1)\rceil-1\}$, so that
\[
 \ell(d)\ge\frac{d-C}{q-1}-1~\text{and}~(q-1)\ell(d)\le d.
\]
The family is indexed by the triples $(i,d,w)$ with $i\ge0$, $d\ge1$, and
$w\in\mathcal P_i(h,d)$, in the notation of Lemma~\ref{stayabove}, and its
member with the index $(i,d,w)$ is the interval $J(v_0w0^{\ell(d)})$ of the
word that consists of $v_0$, then $w$, and then $\ell(d)$ letters $0$. This
word has length $|v_0|+i+\ell(d)$, and its last digit slope is
$h+d_i(w)+(q-1)\ell(d)=d+(q-1)\ell(d)\le2d$. Hence the interval
$T_q(J(v_0w0^{\ell(d)}))$ has length at most
\[
 q^{-|v_0|-i-\ell(d)}\bigl(2d+\Lambda(0)\bigr)
 \le q^{1+C/(q-1)-|v_0|}\,q^{-i}\,q^{-d/(q-1)}\bigl(2d+\Lambda(0)\bigr).
\]
By Lemma~\ref{stayabove}, the sum of these lengths over all indices
$(i,d,w)$ is at most
$q^{1+C/(q-1)-|v_0|+h}\sum_{d=1}^{\infty}(2d+\Lambda(0))q^{-d/(q-1)}$, which
is finite.

There are countably many families of the second and of the third kind.
Hence there is a set $Y\subseteq[0,\Lambda(0)]$ of Lebesgue measure zero
that contains the values of $T_q$ at the $q$-adic rationals, the heights
$y$ for which $\mathcal A(y)$ is infinite, and the heights that lie in
$T_q(J)$ for infinitely many members $J$ of one of these families. Let
$y\in[0,\Lambda(0)]\setminus Y$, and suppose that $L(y)$ is infinite.
Since $L(y)$ is a closed subset of $[0,1]$, it contains a point $x$ that
is the limit of a sequence of other points of $L(y)$. This point is not
a $q$-adic rational, because $y\notin Y$, and so $0<x<1$.

Suppose first that $d_n(x)=0$ for infinitely many $n$. The point $x'$ of
Lemma~\ref{foldzero} has $d_n(x')=0$ exactly for the indices $n$ with
$d_n(x)=0$, and each of these indices gives an element of
$\mathcal A(y)$. Different indices give different pairs, so
$\mathcal A(y)$ is infinite, which is impossible because $y\notin Y$.

Suppose now that $d_n(x)=0$ for finitely many $n$ and that there is an
integer $K\ge1$ such that $|d_n(x)|\le K$ for infinitely many $n$. Let $N$
be the largest index with $d_N(x)=0$. For each $n>N$ with
$|d_n(x)|\le K$, the
word $v=\xi_1\ldots\xi_n$ lies in $\mathcal V_n$, and $x\in J(v)$, so
$y\in T_q(J(v))$. Thus $y$ lies in $T_q(J)$ for infinitely many members $J$ of
the second family with these $N$ and $K$, which is impossible.

In the remaining case $|d_n(x)|\to\infty$. Two consecutive digit slopes
differ by at most $q-1$, so they have the same sign when both have
absolute value at least $q$. Hence $d_n(x)\to+\infty$ or
$d_n(x)\to-\infty$. The point $1-x$ has the digits $q-1-\xi_j$ and the
digit slopes $-d_n(x)$, and it is the limit of a sequence of other points
of $L(y)$, because $T_q(1-z)=T_q(z)$ for all $z\in[0,1]$. Replacing $x$ by
$1-x$ if necessary, we may assume that $d_n(x)\to+\infty$. Let
$n_1<n_2<\cdots$ be the positions of the nonzero digits of $x$, and put
\[
 \Gamma_k=d_{n_k}(x)-(q-1)(n_{k+1}-n_k)+(q-1)\log_q(n_{k+1}-n_k).
\]
If $\Gamma_k\to+\infty$, then $T_q'(x)=+\infty$ by Theorem~\ref{thm2}(c).
In this case $(T_q(z)-T_q(x))/(z-x)>0$ for all $z\ne x$ sufficiently close to
$x$, so that $T_q(z)\ne y$ for these $z$, contrary to the choice of $x$.
Hence there is an integer $C\ge0$ such that $\Gamma_k\le C$ for infinitely
many $k$. Choose $N\ge1$ with $d_j(x)\ge1$ for all $j\ge N$, and put
$v_0=\xi_1\ldots\xi_N$ and $h=d_N(x)$. Let $k$ be an index with
$\Gamma_k\le C$ and $n_k\ge N$, and put $d=d_{n_k}(x)$ and
$w=\xi_{N+1}\ldots\xi_{n_k}$. Then $w\in\mathcal P_{n_k-N}(h,d)$, because
$h+d_j(w)=d_{N+j}(x)\ge1$ for $0\le j\le n_k-N$. Since
$\log_q(n_{k+1}-n_k)\ge0$, the inequality $\Gamma_k\le C$ gives
$(q-1)(n_{k+1}-n_k)\ge d-C$, and therefore $n_{k+1}-n_k-1\ge\ell(d)$. The
digits of $x$ at the places $n_k<j<n_{k+1}$ are $0$, so the first
$n_k+\ell(d)$ digits of $x$ form the word $v_0w0^{\ell(d)}$. Thus $x$ lies in
the member $J(v_0w0^{\ell(d)})$ of the third family for $C$ and $v_0$, and
$y$ lies in its image under $T_q$. Different indices $k$ give words $w$ of
different lengths, so $y$ lies in $T_q(J)$ for infinitely many members $J$
of this family, which again is impossible. Therefore $L(y)$ is finite.
\end{proof}

The following lemma extends Proposition~\ref{T2x} to lines of larger slope.

\begin{lemma}
\label{tiltbelow}
Let $k\ge1$ be an integer, let $c=\lfloor(q+1)/2\rfloor+(k-1)(q-1)$, and
let $0\le u\le1$. If $0<u<1/(2q^{k-1})$, then $T_q(u)\ge cu$, with
equality if and only if $q$ is odd and $u=(1-q^{-r})/(2q^{k-1})$ for
some integer $r\ge1$. Moreover, $T_q(u)=cu$ for $u=0$ and for
$u=1/(2q^{k-1})$. If $q$ is odd, then $T_q(u)\le cu$ for
$1/(2q^{k-1})\le u\le1$. If $q$ is even, then $T_q(u)\le cu$ for
$(q+2)/(2(q+1)q^{k-1})\le u\le1$. Consequently, if $q$ is odd, then
$T_q(u)\le cu$ if and only if $u=0$, or $u\ge1/(2q^{k-1})$, or
$u=(1-q^{-r})/(2q^{k-1})$ for some integer $r\ge1$. For $q=3$ we have
$c=2k$.
\end{lemma}

\begin{proof}
Let $0\le u\le q^{1-k}$ and $v=q^{k-1}u$, so that $0\le v\le1$. As in the
proof of Lemma~\ref{firstcross}, we have $F_{k-1}(u)=(q-1)(k-1)u$, and
equation \eqref{eq:tail} gives
\[
 T_q(u)-cu=q^{1-k}\bigl(T_q(v)-\lfloor(q+1)/2\rfloor\,v\bigr).
\]
If $0\le u\le1/(2q^{k-1})$, then $0\le v\le1/2$, and
Proposition~\ref{T2x} gives the statements of the lemma for these $u$.
By \eqref{tilt1} with $k-1$ in place of $a$ and $b=0$, and by
Theorem~\ref{gmax}, we have for every $u\in[0,1]$
\[
 T_q(u)-(k-1)(q-1)u\le q^{1-k}\max_{0\le t\le1}T_q(t),
\]
where the maximum is $(q+1)/4$ if $q$ is odd and $q(q+2)/(4(q+1))$ if $q$
is even. The right side is at most $\lfloor(q+1)/2\rfloor\,u$ if and only
if $u\ge1/(2q^{k-1})$ when $q$ is odd, and if and only if
$u\ge(q+2)/(2(q+1)q^{k-1})$ when $q$ is even. For these $u$ we obtain
$T_q(u)\le cu$. For odd $q$, the statements proved so far together give
the description of all $u$ with $T_q(u)\le cu$.
\end{proof}

For even $q$ the inequality $T_q(u)\le cu$ can fail between the points
$1/(2q^{k-1})$ and $(q+2)/(2(q+1)q^{k-1})$. For instance, if $q=2$ and
$k=1$, then $c=1$, these points are $1/2$ and $2/3$, and Takagi's
function takes the value $2/3$ at $3/5=(0.\overline{1001})_2$, by
Proposition~\ref{ratval} with $N=0$ and $p=4$. The second part of the
proof gives $T_q(u)\le cu$ for all $u\ge1/(2q^{k-1})$ only for odd $q$,
because then the maximum $(q+1)/4$
of $T_q$ equals $\lfloor(q+1)/2\rfloor/2$, while for even $q$ the maximum
$q(q+2)/(4(q+1))$ is larger than $\lfloor(q+1)/2\rfloor/2=q/4$.

For $q\le3$, the next two lemmas show that $T_q$ is strictly increasing on
the set of points that are not $q$-adic rationals and whose digit slopes
$d_j$ with $j\ge1$ are all positive. The first lemma holds in every base.

\begin{lemma}
\label{tiltpos}
Let $k\ge0$ be an integer, put $c=\lfloor(q+1)/2\rfloor+k(q-1)$, and let
$w\in(0,1)$ be a point that is not a
$q$-adic rational and satisfies $d_j(w)\ge-(q-1)k$ for all $j\ge1$. Then
\[
 T_q(w)+cw\le c.
\]
For $q=3$ the hypothesis reads $\kappa_j(w)\ge-k$ for all $j\ge1$, and the
conclusion reads $T_q(w)+(2k+2)w\le2k+2$.
\end{lemma}

\begin{proof}
Let $u_0=1/(2q^k)$ if $q$ is odd, and $u_0=(q+2)/(2(q+1)q^k)$ if $q$ is
even. The digit sequence of $u_0$ consists of $k$ digits $0$ followed, for
odd $q$, by the digits $(q-1)/2,(q-1)/2,\ldots$, and for even $q$ by the
digits $q/2,q/2-1,q/2,q/2-1,\ldots$. Thus $d_j(u_0)=(q-1)k$ for $j\ge k$
if $q$ is odd, and for $i\ge0$ we have $d_{k+2i}(u_0)=(q-1)k$ and
$d_{k+2i+1}(u_0)=(q-1)k-1$ if $q$ is even. The point $v=1-w$ is not a
$q$-adic rational, its digits are those of $w$ replaced by their
complements, and so $d_j(v)=-d_j(w)\le(q-1)k$ for all $j\ge1$. Suppose
that $v<u_0$. Then at the first place $j$ where the digits of the two
points differ, the digit of $v$ is the smaller one. Hence $j>k$, and
$d_{j-1}(v)=d_{j-1}(u_0)$. If $q$ is odd, the $j$th digit of $v$ is less
than $(q-1)/2$, so $d_j(v)\ge(q-1)k+2$. If $q$ is even and $j-k$ is odd,
the $j$th digit of $v$ is at most $q/2-1$, so $d_j(v)\ge(q-1)k+1$. If $q$
is even and $j-k$ is even, the $j$th digit of $v$ is at most $q/2-2$, so
$d_j(v)\ge(q-1)k-1+3$. Each case contradicts $d_j(v)\le(q-1)k$.
Therefore $v\ge u_0$, and Lemma~\ref{tiltbelow} with $k+1$ in
place of $k$ gives $T_q(v)\le cv$. Since $T_q(w)=T_q(1-w)=T_q(v)$, this is
the claim.
\end{proof}

\begin{lemma}
\label{monopos}
Let $q\le3$, and let $x<x'$ be points of $[0,1)$ that are not $q$-adic rationals and
satisfy $d_j(x)\ge1$ and $d_j(x')\ge1$ for all $j\ge1$. Then
$T_q(x)<T_q(x')$. For $q=3$ the hypothesis reads $\kappa_j(x)\ge1$ and
$\kappa_j(x')\ge1$ for all $j\ge1$.
\end{lemma}

\begin{proof}
Both points have the first digit $0$, because $d_1=q-1-2\xi_1\ge1$ and
$q\le3$, and their digit sequences differ.
Let $n\ge1$ be the largest index such that $x$ and $x'$ have the same first
$n$ digits, so that $d=d_n(x)=d_n(x')\ge1$, and let $t$ and
$t'$ be the points whose digits are those of $x$ and of $x'$ after the
place $n$. Then $t<t'$, the first digits $\xi$ of $t$ and $\xi'$ of $t'$
satisfy $\xi<\xi'$, and by \eqref{lcell}, applied to $x$ and to $x'$,
\[
 q^n\bigl(T_q(x')-T_q(x)\bigr)=f(t')-f(t),~f(v)=dv+T_q(v).
\]
For integers $0\le i\le q$ put $G_i=i(q-i+d)$, so that
$G_{i+1}-G_i=d+q-1-2i$. We show that
\begin{equation}
 G_\xi<qf(t)\le G_{\xi+1}~\text{and}~G_{\xi'}<qf(t').
 \label{monosteps}
\end{equation}
The lemma follows from \eqref{monosteps}. Indeed, for $i\le\xi'$ we have
$G_{i+1}-G_i\ge d+q-1-2\xi'=d_{n+1}(x')\ge1$, so that
$G_{\xi+1}\le G_{\xi'}$ and $qf(t)\le G_{\xi+1}\le G_{\xi'}<qf(t')$.

Put $c=d+q-1-2\xi=d_{n+1}(x)$ and $w=qt-\xi$. Then $c\ge1$, and $w$ is a
point of $(0,1)$ that is not a $q$-adic rational and has the digits
$\xi_{n+2},\xi_{n+3},\ldots$ of $x$, so that
$d_j(w)=d_{n+1+j}(x)-c\ge1-c$ for $j\ge1$. By \eqref{branches} with $a=\xi$
and $y=w$,
\[
 qf(t)=d(\xi+w)+\xi(q-\xi)+(q-1-2\xi)w+T_q(w)=G_\xi+cw+T_q(w).
\]
Since $c\ge1$, $w>0$, and $T_q(w)\ge0$, we have $qf(t)>G_\xi$. The same
computation for $x'$, with $d_{n+1}(x')\ge1$ in place of $c$, gives
$qf(t')>G_{\xi'}$. It remains to show that $T_q(w)+cw\le c$, because
$G_{\xi+1}=G_\xi+c$. If $q=3$, then $c$ and the numbers $d_j(w)$ are even.
Thus $c=2k+2$ with an integer $k\ge0$, and $d_j(w)\ge1-c$ implies
$d_j(w)\ge-2k$. If $q=2$, then $c=k+1$ with $k=c-1\ge0$, and
$d_j(w)\ge-k$. In both cases $c=\lfloor(q+1)/2\rfloor+k(q-1)$ and
$d_j(w)\ge-(q-1)k$ for all $j\ge1$, so Lemma~\ref{tiltpos} gives
$T_q(w)+cw\le c$.
\end{proof}

For $q=4$ and for $q=5$ the statement of Lemma~\ref{monopos} is false. We
compute the values of $T_q$ below with Proposition~\ref{ratval}. For
$q=5$, the points $x=(0.03\overline2)_5=7/50$ and
$x'=(0.1000\overline2)_5=251/1250$ have the digit slopes $4,2,2,\ldots$
and $2,6,10,14,14,\ldots$, which are all positive, and $x<x'$, but
$T_q(x)=41/50$ is larger than $T_q(x')=1017/1250$. For $q=4$, the points
$x=(0.02\overline{21})_4=13/80$ and $x'=(0.1000\overline{12})_4=161/640$
have the digit slopes $3,2,1,2,1,\ldots$ and $1,4,7,10,11,10,11,\ldots$,
and $T_q(x)=31/40$ is larger than $T_q(x')=493/640$.

The statement of Lemma~\ref{monopos} fails for the Takagi--van~der~Waerden
function $\sum_{j=0}^{\infty}3^{-j}\dist(3^jx,\Z)$:
Allaart~\cite[Remark~4.7]{AllaartLevelSets2014} gave two points,
$17/108$ and $37/108$, at which all slopes of the partial sums are
positive and the function takes the same value $3/8$.
Allaart~\cite[Lemma~4.6]{AllaartLevelSets2014} also showed that between two such points
there is an interval on which a partial sum has slope zero and whose
image under the function contains their common value, and he used this in
the proof of his Proposition~4.9. For Takagi's function, Lagarias and
Maddock~\cite[Theorem~4.8]{LagariasMaddock2012} proved that the function
is nondecreasing on the set of points whose first $n$ binary digits
contain more zeros than ones for every $n\ge1$, and strictly increasing on
this set away from a countable set of rationals. Their countable set is
contained in the set of the rationals whose binary digits are eventually
all $0$ or eventually alternate between $0$ and $1$. Thus for $q=2$ their theorem
contains Lemma~\ref{monopos} for all points whose binary digits do not
eventually alternate between $0$ and $1$. Lemma~\ref{tiltpos} plays the
role of the claim in the proof of part~(1) of their theorem, and they
attribute that proof to Allaart.
Allaart~\cite[Proposition~3.2]{AllaartHowLarge2012} proved that Takagi's
function is strictly increasing on the set of the points of $[0,1/2)$
that lie in no interval $[m/2^n,(m+1)/2^n]$ of slope $0$ with $n\ge1$. A
point of $[0,1)$ that is not a dyadic rational and whose digit slopes are
all positive lies in this set, because its first digit is $0$ and the
interval of level $n$ that contains it has the slope $d_n\ge1$. Hence his
theorem contains Lemma~\ref{monopos} for $q=2$. For $q=3$,
Lemma~\ref{monopos} is the corresponding statement for $T_q$ and the
ternary digits.

\begin{proposition}
\label{finitelevel}
Let $q\le3$, let $y\in[0,\Lambda(0)]$ be a number that is not a value of
$T_q$ at a $q$-adic rational, and suppose that $\mathcal A(y)$ is finite, with $r$ elements.
Then $L(y)$ has at most $2^{r+1}-2$ points.
\end{proposition}

\begin{proof}
Let $x\in L(y)$. Then $x$ is not a $q$-adic rational, so $0<x<1$. Let
$x'$ be the point of Lemma~\ref{foldzero}. It is not a $q$-adic rational,
$T_q(x')=y$, the indices $n$ with $d_n(x')=0$ are those with $d_n(x)=0$,
and $d_j(x')=|d_j(x)|$ for all $j$, because $q\le3$. Distinct indices $n$
with $d_n(x')=0$ give distinct pairs
$(n,\lfloor q^nx'\rfloor)\in\mathcal A(y)$, so $d_n(x)=0$ for finitely
many $n$. If there are $z$ such indices $n$, they divide the indices
$j\ge1$ into the $z$ ranges $n_k<j\le n_{k+1}$ of Lemma~\ref{foldzero},
the last of which is infinite. The point $x$ is obtained from $x'$ by
replacing the digits in some of these ranges by their complements, and
the ranges depend only on the indices $n$ with $d_n(x')=0$. Therefore at
most $2^z$ points of $L(y)$ give the same point $x'$.

Let $N$ be the largest index with $d_N(x')=0$, let
$M=\lfloor q^Nx'\rfloor$, so that $(N,M)\in\mathcal A(y)$, and let
$u=q^Nx'-M$, the point whose digits are those of $x'$ after the place
$N$. It is not a $q$-adic rational, and $d_j(u)=d_{N+j}(x')\ge1$ for
$j\ge1$, because $d_{N+j}(x')\ge0$ and $d_{N+j}(x')\ne0$ by the choice
of $N$. By
\eqref{lcell}, applied to $x'$, $T_q(u)=q^N\bigl(y-T_q(M/q^N)\bigr)$. Suppose another point of $L(y)$ produces a point $x''$ with the same pair
$(N,M)$. Then $q^Nx''-M$ has the same properties as $u$ and the same
value under $T_q$, so $q^Nx''-M=u$ by Lemma~\ref{monopos}, and hence
$x''=x'$.

Now order the elements of $\mathcal A(y)$ as $A_1,\ldots,A_r$ so that
their first entries do not decrease. By the previous paragraph, for each
$i$ at most one point $x'$ arises with $(N,M)=A_i$. Every other index $n$
with $d_n(x')=0$ gives an element $(n,\lfloor q^nx'\rfloor)$ of
$\mathcal A(y)$ with $n<N$, which therefore comes before $A_i$. Hence
$d_n(x')=0$ for at most $i$ indices $n$, and at most $2^i$ points of
$L(y)$ give $x'$.
Therefore $L(y)$ has at most $\sum_{i=1}^r2^i=2^{r+1}-2$ points.
\end{proof}

For $q=4$ Proposition~\ref{finitelevel} is false (see
Example~\ref{ex:base4} below). Its proof uses $q\le3$ in two places:
Lemma~\ref{foldzero} gives $d_j(x')=|d_j(x)|$, so that $d_j(u)\ge1$ for
every $j\ge1$, and Lemma~\ref{monopos} shows that $u$ is determined by
$T_q(u)$.

\begin{corollary}
\label{levelcount}
Let $q\le3$. For every integer $M\ge2$, the set of heights
$y\in[0,\Lambda(0)]$ for which
$L(y)$ has at least $M$ points has Lebesgue measure at most
\[
 \frac3{\log_2(M+2)-1}~\text{if }q=3,~\text{and at most}~
 \frac4{3(\log_2(M+2)-1)}~\text{if }q=2.
\]
\end{corollary}

\begin{proof}
Let $Y$ be this set. In every base it is the union of the compact sets
$T_q(J_1)\cap\cdots\cap T_q(J_M)$ over all choices of pairwise disjoint
closed intervals $J_1,\ldots,J_M\subseteq[0,1]$ with rational endpoints,
so it is a Borel set. Put $t=\log_2(M+2)-1$, so that $t\ge1$ and
$2^{r+1}-2<M$ for every integer $r<t$. By Proposition~\ref{finitelevel}, a
height $y$ that is not a value of $T_q$ at a $q$-adic rational and for
which $\mathcal A(y)$ has fewer than $t$ elements does not lie in $Y$.
Hence $Y$ is contained, up to a countable set, in the set of $y$ for which
$\mathcal A(y)$ has at least $t$ elements. In the proof of
Theorem~\ref{levelthm}(a) we saw that the number of elements of
$\mathcal A(y)$ has the integral $\Lambda(0)\sum_{(n,m)\in\mathcal A}q^{-n}$
over $0\le y\le\Lambda(0)$, which is $3$ for $q=3$ and $4/3$ for $q=2$.
Hence this set, and with it $Y$, has measure at most $3/t$ for $q=3$ and
at most $4/(3t)$ for $q=2$.
\end{proof}

For $q\ge4$ the measure of this set tends to $0$ as $M\to\infty$, because
these sets decrease as $M$ increases and their intersection, the set of
the heights with an infinite level set, has measure zero by
Theorem~\ref{levelthm}(a). The proof of Corollary~\ref{levelcount} uses
Proposition~\ref{finitelevel}, and for $q\ge4$ we know no bound of order
$1/\log M$.

\begin{lemma}
\label{levelnull}
Let $q\ge3$ and $0\le y\le\Lambda(0)$. For every integer $k\ge0$, the level
set $L(y)$ meets
at most $(q^2-2)^k$ intervals of level $2k$. For $q=3$ this number is
$7^k$.
\end{lemma}

\begin{proof}
Let $J=[m/q^n,(m+1)/q^n]$ be an interval of level $n$, and let
$d=(q-1)n-2s_q(m)$ be its slope. For $i=0,1,\ldots,q-1$ let
$J_i=[(qm+i)/q^{n+1},(qm+i+1)/q^{n+1}]$. These are the intervals of level
$n+1$ contained in $J$, and $J_i$ has the slope $d+q-1-2i$, since
$s_q(qm+i)=s_q(m)+i$. By \eqref{cell} with $(i+t)/q$ in place of $t$, and
by \eqref{branches} with $a=i$ and $y=t$,
\begin{equation}
 q^{n+1}\left(T_q\!\left(\frac{qm+i+t}{q^{n+1}}\right)
 -T_q\!\left(\frac m{q^n}\right)\right)=i(q-i+d)+(d+q-1-2i)\,t+T_q(t)
 ~\text{for }0\le t\le1.
 \label{subrange}
\end{equation}
We claim that if $|d|\ge2$, then no value of $T_q$ on $J_0$ is a value of
$T_q$ on $J_{q-1}$. The reflection $x\mapsto1-x$ maps $J$ onto the interval
$[(q^n-m-1)/q^n,(q^n-m)/q^n]$, whose slope is $-d$ because
$s_q(q^n-1-m)=(q-1)n-s_q(m)$. It maps $J_0$ and $J_{q-1}$ onto the last
and the first of the intervals of level $n+1$ contained in this interval,
and $T_q(1-x)=T_q(x)$. Hence it suffices to prove the claim for $d\ge2$. By
Theorems~\ref{limitthm} and~\ref{gmax}, $0\le T_q\le\Lambda(0)\le(q+1)/4$,
where for even $q$ the last inequality holds because $q(q+2)<(q+1)^2$. For $i=0$ the right
side of \eqref{subrange} is therefore at most $d+q-1+(q+1)/4$ for every
$t$. For $i=q-1$ it is $(q-1)(d+1)+(d-q+1)t+T_q(t)$, which is at least
$(q-1)(d+1)$ if $d\ge q-1$ and at least $(q-1)(d+1)+d-q+1=qd$ if
$d<q-1$. Now
\[
 (q-1)(d+1)-(d+q-1)=(q-2)d\ge2(q-2)~\text{and}~
 qd-(d+q-1)=(q-1)(d-1)\ge q-1,
\]
and both $2(q-2)$ and $q-1$ are larger than $(q+1)/4$, because $q\ge3$.
Thus every value of $T_q$ on $J_0$ is smaller than every value of $T_q$ on
$J_{q-1}$. This proves the claim.

Now let $n$ be even. Then $d$ is even, because $d-(q-1)n$ is even, and
$L(y)$ meets at most $q^2-2$ of the $q^2$ intervals of level $n+2$
contained in $J$. Indeed, each interval $J_i$ contains $q$ of them. If
$d\ne0$, then $|d|\ge2$, so by the claim $L(y)$ does not meet both $J_0$
and $J_{q-1}$, and it meets at most $q(q-1)\le q^2-2$ of these intervals.
If $d=0$, then $J_0$ and $J_{q-1}$ have the slopes $q-1$ and $-(q-1)$,
whose absolute value is at least $2$, and the claim for these two
intervals of level $n+1$ shows that $L(y)$ meets at most $q-1$ of the
intervals of level $n+2$ contained in each of them, hence at most
$q^2-2$ in all. Since every interval of level $2k+2$ is contained in an
interval of level $2k$ and $[0,1]$ is the only interval of level $0$,
induction on $k$ proves the lemma.
\end{proof}

\begin{corollary}
\label{levelbox}
Let $q\ge3$. Every level set of $T_q$ has Lebesgue measure zero. Its upper
box
dimension, and hence its Hausdorff dimension, is at most
$\log(q^2-2)/(2\log q)$. For $q=3$ this bound is $\log7/\log9$.
\end{corollary}

\begin{proof}
Let $0\le y\le\Lambda(0)$. The set $L(y)$ is not empty, because $T_q$
takes every value in $[0,\Lambda(0)]$ on $[0,1]$. By
Lemma~\ref{levelnull}, for every $k\ge0$ it is covered by at most
$(q^2-2)^k$ intervals of length $q^{-2k}$. Hence its Lebesgue measure is
at most $(1-2q^{-2})^k$ for every $k\ge0$, and it is zero. In the
notation of Definition~\ref{assouad}, we also have
$N_r(L(y))\le(q^2-2)^k$ for every $r\ge q^{-2k}$. If $0<r<q^{-2}$ and
$k\ge2$ is the integer with $q^{-2k}\le r<q^{2-2k}$, then
$\log N_r(L(y))\le k\log(q^2-2)$ and $\log(1/r)>2(k-1)\log q$, so that
\[
 \frac{\log N_r(L(y))}{\log(1/r)}\le\frac{k\log(q^2-2)}{2(k-1)\log q}.
\]
As $r\downarrow0$, the integer $k$ tends to infinity, and the right side
tends to $\log(q^2-2)/(2\log q)$. Thus the upper box dimension of $L(y)$
is at most this number, and by Remark~\ref{rem:assouad} so is its
Hausdorff dimension.
\end{proof}

For $q=2$ the statement of Lemma~\ref{levelnull} is false: Takagi's
function takes the value $1/2$ at $1/4$, $1/2$, and $3/4$, by
Theorem~\ref{sampthm}, so $L(1/2)$ meets all four intervals of level $2$,
while $q^2-2=2$. The proof of the lemma uses $q\ge3$ when it compares
$2(q-2)$ and $q-1$ with $(q+1)/4$. By the theorem of de Amo
et~al.~\cite[Theorem~3.4]{deAmo2011}, every level set of Takagi's
function has Hausdorff dimension at most $1/2$, and therefore Lebesgue
measure zero. Our proof of the bound for the upper box dimension in
Corollary~\ref{levelbox} uses Lemma~\ref{levelnull}, so it does not cover
$q=2$.

\begin{remark}
\label{rem:levelbox}
Let $q=3$. The upper box dimension of every level set of $T_q$ is in fact at most
$\log2/\log9$, and hence so is its Hausdorff dimension. By
Theorem~\ref{threequarters} and Proposition~\ref{assouadlevel}, both
dimensions of $L(3/4)$ are equal to $\log2/\log9$. The proof of the bound
$\log2/\log9$ that we know relies on a finite calculation carried out
with a computer. It would be the only proof in this paper that uses a
computer, so we omit it and leave it to the interested reader.
\end{remark}

For real $y$ let
$\Theta(y)$ be the Lebesgue measure of $\{x\in[0,1]:T_q(x)\le y\}$. Part~(a)
of the next proposition uses the heights $y_e$ and the
intervals $E_w$ of Definition~\ref{Eset}.

\begin{proposition}
\label{distfn}
Let $\alpha=1-\log2/(2\log q)$.
\begin{enumerate}
\item[(a)] Let $e$ be an integer with $0\le e<(q-1)/2$, let $y_e$ be the
 height of Definition~\ref{Eset}, and put
 $r_e=\max\{y_e,\Lambda(0)-y_e\}$. For $0<r\le r_e$,
 \[
 \Theta(y_e+r)-\Theta(y_e-r)\ge\Bigl(\frac r{q^2r_e}\Bigr)^{\alpha}.
 \]
 Consequently, there are no constants $C$ and $\alpha'>\alpha$ such that
 $\Theta(y+h)-\Theta(y)\le Ch^{\alpha'}$ for all real $y$ and $h>0$. For
 $q=3$ we have $e=0$, $y_e=r_e=3/4$, and $\alpha=1-\log2/\log9$, and the
 bound is $(4r/27)^\alpha$.
\item[(b)] The function $\Theta$ is continuous and strictly increasing on
 $[0,\Lambda(0)]$, with $\Theta(0)=0$ and $\Theta(\Lambda(0))=1$.
\item[(c)] There is a Borel set $Y_\infty\subseteq[0,\Lambda(0)]$ of
 Lebesgue measure $0$ such that $T_q(x)\in Y_\infty$ for almost every
 $x\in[0,1]$.
\end{enumerate}
\end{proposition}

By (c), the function $T_q$ maps a set of full Lebesgue measure into a set
of Lebesgue measure zero. Allaart~\cite[Corollary~1.2]{AllaartLevelSets2014} deduced the
same fact for the Takagi--van~der~Waerden functions from his results on
their level sets, and the proof of (c) is his argument. Since Takagi's function is one of these functions, the fact is known for
$q=2$.

\begin{proof}
(a) Since $y_e$ is a value of $T_q$ by Lemma~\ref{Efacts}, and
$0\le T_q\le\Lambda(0)$, we have $|T_q-y_e|\le r_e$ on $[0,1]$. Let
$k\ge0$. By \eqref{Eblock}, $|T_q-y_e|\le r_eq^{-2k}$ on each of the $2^k$
intervals $E_w$ with $w\in W_k$. These intervals have length $q^{-2k}$,
any two of them have at most one common point, and
$2^kq^{-2k}=q^{-2k\alpha}$ because $q^{2\alpha}=q^2/2$. Hence the set of
$x\in[0,1]$ with $|T_q(x)-y_e|\le r_eq^{-2k}$ has measure at least
$q^{-2k\alpha}$. Now let $0<r\le r_e$, and choose $k\ge0$ with
$r_eq^{-2k-2}<r\le r_eq^{-2k}$, so that $q^{-2k-2}\ge r/(q^2r_e)$. The
difference $\Theta(y_e+r)-\Theta(y_e-r)$ is the measure of the set of
$x\in[0,1]$ with $y_e-r<T_q(x)\le y_e+r$, and since $r_eq^{-2k-2}<r$, this
set contains the one above with $k+1$ in place of $k$. Hence
$\Theta(y_e+r)-\Theta(y_e-r)\ge q^{-2(k+1)\alpha}\ge(r/(q^2r_e))^\alpha$. If
$\Theta(y+h)-\Theta(y)\le Ch^{\alpha'}$ held for all $y$ and $h>0$ with
constants $C$ and $\alpha'>\alpha$, then
$(r/(q^2r_e))^\alpha\le C(2r)^{\alpha'}$ for all small $r>0$, which
is false.

(b) Let $0\le y\le\Lambda(0)$. For $h>0$, the difference
$\Theta(y+h)-\Theta(y-h)$ is the measure of the set of $x\in[0,1]$ with
$y-h<T_q(x)\le y+h$. As $h$ decreases to $0$, these sets decrease, and their
intersection is $L(y)$, which has measure zero. For $q\ge3$ this is part
of Corollary~\ref{levelbox}, and for $q=2$ it holds because every level
set of Takagi's function has Hausdorff dimension at most $1/2$, by the
theorem of de Amo et~al.~\cite[Theorem~3.4]{deAmo2011}. Since the measure
of the intersection of a decreasing sequence of measurable subsets of
$[0,1]$ is the limit of their measures, $\Theta(y+h)-\Theta(y-h)\to0$ as
$h\to0$, and the nondecreasing function $\Theta$ is continuous at $y$.
Moreover, $\Theta(0)=0$ and $\Theta(\Lambda(0))=1$, because $L(0)=\{0,1\}$
and $T_q\le\Lambda(0)$. Let $0\le a<b\le\Lambda(0)$. The
set of $x\in(0,1)$ with $a<T_q(x)<b$ is open, because $T_q$ is continuous,
and it is not empty, because $T_q$ takes every value between $T_q(0)=0$ and
its maximum $\Lambda(0)$. Hence it has positive measure, and
$\Theta(b)>\Theta(a)$.

(c) Let $Y_\infty$ be the set of heights $y\in[0,\Lambda(0)]$ for which
$L(y)$ is infinite. In the proof of Corollary~\ref{levelcount} we saw that
for every $M$ the set of $y$ for which $L(y)$ has at least $M$ points is a
Borel set, so $Y_\infty$ is a Borel set. It has measure $0$ by
Theorem~\ref{levelthm}(a), and by Theorem~\ref{levelthm}(b), $L(T_q(x))$ is
uncountable, so that $T_q(x)\in Y_\infty$, for almost every $x\in[0,1]$.
\end{proof}

Lagarias and Maddock~\cite[Definition~1.2]{LagariasMaddock2012} defined
the local level set of a point for Takagi's function as the set of the
points obtained from it by complementing the binary digits in some of
the strings of digits that lie between two consecutive places $j$ at
which the first $j$ digits contain as many zeros as ones, or after the
last such place if there are finitely many. They showed that every level
set is the disjoint union of the local level sets that it
contains~\cite[Theorem~1.3]{LagariasMaddock2012}. For $q=2$ these places
$j$ are those at which the digit slope $d_j$ of the point is $0$, and in
every base the next definition makes the same construction with the
base $q$ digits and the indices $j$ with $d_j=0$.

\begin{definition}
\label{locset}
Let $x\in[0,1)$ be a point that is not a $q$-adic rational, and let
$0=n_0<n_1<n_2<\cdots$ be all the indices $n$ with $d_n(x)=0$; if there
are finitely many, say $n_0<\cdots<n_K$, put $n_{K+1}=\infty$. For each
$k$, the digits $\xi_j$ with $n_k<j\le n_{k+1}$ form the string of $x$
after $n_k$, and the digits $q-1-\xi_j$ with $n_k<j\le n_{k+1}$ form the
complement of this string. The local level set $L^{\mathrm{loc}}_x$ of
$x$ is the set of the points obtained from $x$ by replacing some of
these strings by their complements.
\end{definition}

By Lemma~\ref{flip}, the points of $L^{\mathrm{loc}}_x$ are not $q$-adic rationals, and $L^{\mathrm{loc}}_x\subseteq L(T_q(x))$.

\begin{lemma}
\label{locstruct}
Let $x\in[0,1)$ be a point that is not a $q$-adic rational.
\begin{enumerate}
\item[(a)] For every point $x'\in L^{\mathrm{loc}}_x$, the indices $n$ with
$d_n(x')=0$ are those with $d_n(x)=0$, and
$L^{\mathrm{loc}}_{x'}=L^{\mathrm{loc}}_x$. Consequently, two
local level sets are equal or disjoint,
and for every $y\in[0,\Lambda(0)]$ the points of $L(y)$ that are not $q$-adic rationals form the disjoint union of the local level sets contained in
$L(y)$. If $q\le3$, then a point $x'\in[0,1)$ that is not a $q$-adic rational lies in $L^{\mathrm{loc}}_x$ if and only if
$|d_j(x')|=|d_j(x)|$ for every $j\ge1$.
\item[(b)] Exactly one point $x'$ of $L^{\mathrm{loc}}_x$ satisfies
$d_{n+1}(x')\ge0$ for every index $n$ with $d_n(x')=0$. It is the point
$x'$
of Lemma~\ref{foldzero}, and it is the smallest point of
$L^{\mathrm{loc}}_x$. If $q\le3$, it is the only point $x'$ of
$L^{\mathrm{loc}}_x$ with $d_j(x')\ge0$ for every $j\ge1$.
\item[(c)] If $d_n(x)=0$ and $2\xi_{n+1}\ne q-1$ for exactly $i$ indices
$n$, then $L^{\mathrm{loc}}_x$ has $2^i$ points. If $d_n(x)=0$ and
$2\xi_{n+1}\ne q-1$ for infinitely many $n$, then $L^{\mathrm{loc}}_x$ is
uncountable.
\end{enumerate}
For $q=3$ the condition $2\xi_{n+1}\ne q-1$ reads $\xi_{n+1}\ne1$, and for
even $q$ it holds for every digit.
\end{lemma}

\begin{proof}
Let $x'\in L^{\mathrm{loc}}_x$. By Lemma~\ref{flip}, $d_j(x')$ is
$d_j(x)$ or $-d_j(x)$ for every $j$, so the indices $n$ with $d_n(x')=0$
are the same numbers $n_0<n_1<\cdots$ as for $x$, and for each $k$ the
string of $x'$ after $n_k$ is
the string of $x$ after $n_k$ or its complement. Since the complement of
the complement of a string is the string itself, the points obtained
from $x'$ by replacing some of its strings by their complements are the
points obtained from $x$ in this way, that is,
$L^{\mathrm{loc}}_{x'}=L^{\mathrm{loc}}_x$. If two local level sets have a
common point $x'$, both are equal to $L^{\mathrm{loc}}_{x'}$. A point of
$L(y)$ that is not a $q$-adic rational lies in $[0,1)$, since $1$ is a
$q$-adic rational, and it lies in its own local level set, which is
contained in $L(y)$ by Lemma~\ref{flip}. Hence the points of $L(y)$ that are not
$q$-adic rationals form the disjoint union of the local level sets
contained in $L(y)$.

Now let $q\le3$. If $x'\in L^{\mathrm{loc}}_x$, then
$|d_j(x')|=|d_j(x)|$ for all $j$ by Lemma~\ref{flip}. Conversely, let
$x'\in[0,1)$ be a point that is not a $q$-adic rational and satisfies
$|d_j(x')|=|d_j(x)|$ for all $j\ge1$. Then the indices $n$ with
$d_n(x')=0$ are the same numbers $n_0<n_1<\cdots$ as for $x$. By
Remark~\ref{rem:signs}, the numbers $d_j(x)$
with $n_k<j<n_{k+1}$ are nonzero and have the same sign, and so do the
numbers $d_j(x')$. As $d_j(x)$ and $d_j(x')$ are zero for
$j=n_0,n_1,\ldots$, either $d_j(x')=d_j(x)$ for all $j$ with
$n_k\le j\le n_{k+1}$, or $d_j(x')=-d_j(x)$ for all these $j$.
The digits of $x$ are $\xi_j=\bigl(q-1+d_{j-1}(x)-d_j(x)\bigr)/2$, and the
digits of $x'$ are given by the same formula. Hence the string of $x'$
after $n_k$ is the string of $x$ after $n_k$ or its complement, and
$x'\in L^{\mathrm{loc}}_x$. This proves (a).

Let $x'$ be the point obtained from $x$ by replacing the strings
after those indices $n_k$ for which $d_{n_k+1}(x)<0$ by their
complements. This is the point $x'$ of Lemma~\ref{foldzero}. By
Lemma~\ref{flip}, $d_{n_k+1}(x')=-d_{n_k+1}(x)>0$ if the string after
$n_k$ is replaced, and $d_{n_k+1}(x')=d_{n_k+1}(x)\ge0$ if it is kept.
Since the indices $n$ with $d_n(x')=0$ are those with $d_n(x)=0$, we
obtain $d_{n+1}(x')\ge0$ for each of them. Now let $x''$ be a point of
$L^{\mathrm{loc}}_x$ different from $x'$. By (a), each string of $x''$ is
the corresponding string of $x'$ or its complement. Let $j$ be the first
place at which the digits of $x'$ and $x''$ differ, and let
$n_k<j\le n_{k+1}$. The string of $x''$ after $n_k$ is the complement of
that of $x'$, and the two strings are different. Let $\xi'$ be the digit
of $x'$ at the place $n_k+1$, so that
$d_{n_k+1}(x')=q-1-2\xi'\ge0$. If this number were $0$, then $n_{k+1}$
would be $n_k+1$, and the string of $x'$ after $n_k$ would be the
single digit $(q-1)/2$, which is its own complement. Thus
$2\xi'<q-1$, the digit $q-1-\xi'$ of $x''$ at the place $n_k+1$ is
larger than $\xi'$, and $j=n_k+1$. Consequently
$d_{n_k+1}(x'')=-d_{n_k+1}(x')<0$, so $x'$ is the only point of
$L^{\mathrm{loc}}_x$ with the property stated in (b). Moreover, the digits
of $x'$ and $x''$ before the place $j$ agree, the digit of $x'$ at the
place $j$ is the smaller one, and neither point is a $q$-adic rational,
so $x'<x''$. If $q\le3$, then $d_j(x')=|d_j(x)|\ge0$ for every $j$ by
Lemma~\ref{foldzero}, and every point $x''\in L^{\mathrm{loc}}_x$ with
$d_j(x'')\ge0$ for all $j\ge1$ has the property stated in (b), so that
$x''=x'$. This proves (b).

If $n$ is an index with $d_n(x)=0$ and $2\xi_{n+1}=q-1$, then also
$d_{n+1}(x)=0$, so the string after $n$ is the single digit $(q-1)/2$,
and replacing it by its complement changes nothing. The other strings
are those after the indices $n$ with $d_n(x)=0$ and
$2\xi_{n+1}\ne q-1$, so they begin with a digit $\xi$ with
$2\xi\ne q-1$. If the strings in two different sets of such strings
are replaced by their complements, the two digit sequences obtained differ at
the first digit of a string that belongs to only one of the two sets. By
Lemma~\ref{flip}, these are the digit sequences of points that are not
$q$-adic rationals, and such a point has no other digit sequence, so the
two points are different. Therefore the points of $L^{\mathrm{loc}}_x$
correspond bijectively to the sets of strings of $x$ that do not consist
of the single digit $(q-1)/2$, which proves (c).
\end{proof}

For example, let $q=3$, and let $x=(0.0220\overline{01})_3$ be the point
of Figure~\ref{foldpic}. The indices $n$ with $d_n(x)=0$ are $0$, $2$,
and $4$, and they are followed by the digits $0$, $2$, and $0$. By
Lemma~\ref{locstruct}(c),
$L^{\mathrm{loc}}_x$ consists of the eight points obtained from $x$ by
complementing some of its three strings, and by
Lemma~\ref{locstruct}(b), the smallest of them is the point
$x'=(0.0202\overline{01})_3$ of the left panel.

By the second equation in \eqref{dsq}, two points $x$ and $x'$ of
$[0,1)$ with $|d_j(x')|=|d_j(x)|$ for every $j\ge1$ satisfy $T_q(x')=T_q(x)$
in every base. For $q=4$ such a point $x'$ need not lie in
$L^{\mathrm{loc}}_x$, so that the last statement of
Lemma~\ref{locstruct}(a) is false for $q=4$. Indeed, the points
$x=(0.13\overline{12})_4=37/80$ and $x'=(0.11\overline{21})_4=7/20$ have
the digit slopes $1,-2,-1,-2,-1,\ldots$ and $1,2,1,2,1,\ldots$, and
$T_q(x)=T_q(x')=23/20$. The only index $n$ with $d_n(x)=0$ is $0$, so
$L^{\mathrm{loc}}_x=\{x,1-x\}=\{37/80,43/80\}$, and this set does not
contain $x'$.

The next definition extends the numbers $y_n$ of the proof of
Theorem~\ref{threequarters} to every base, every height, and every
sequence of digits.

\begin{definition}
\label{excl}
Let $0\le y\le\Lambda(0)$, and let $\xi_1,\xi_2,\ldots$ be digits in
$\{0,1,\ldots,q-1\}$.
For $n\ge0$ put $a_n=\sum_{j=1}^n\xi_jq^{-j}$, $c_n=a_n+q^{-n}$,
$d_n=\sum_{j=1}^n(q-1-2\xi_j)$, and $y_n=q^n(y-T_q(a_n))$. For an integer
$d$ let $\mathcal T(d)$ be the range of the function
$u\mapsto du+T_q(u)$ on $[0,1]$. The digit $\xi_{n+1}$ is excluded
for the height $y$ if $y_{n+1}\notin \mathcal T(d_{n+1})$.
\end{definition}

If $\xi_1,\xi_2,\ldots$ are the digits of a point $x\in[0,1)$, then
$[a_n,c_n]$ is the interval of level $n$ with $a_n\le x<c_n$, and
$d_n=d_n(x)$. Since $T_q\ge0$ and $T_q(0)=0$, the smallest value of
$du+T_q(u)$ on $[0,1]$ is $0$ for $d\ge0$, and its largest value is
$\Lambda(-d)=d+\Lambda(d)$. Thus $\mathcal T(d)=[0,d+\Lambda(d)]$ for $d\ge0$, and
since $T_q(1-u)=T_q(u)$, we have
$\mathcal T(d)=d+\mathcal T(-d)=[d,\Lambda(|d|)]$ for $d<0$. For $q=3$ and an integer
$k$ this gives $\mathcal T(2k)=[0,2k+3^{-k}]$ if $k\ge0$
and $\mathcal T(2k)=[2k,3^{-|k|}]$ if $k<0$.
Moreover, $d_{n+1}=d_n+q-1-2\xi_{n+1}$. The interval $[a_n,c_n]$ has the
slope $d_n$ by \eqref{dnsq}, and $T_q(\xi/q)=\xi(q-\xi)/q$ for
$\xi\in\{0,1,\ldots,q-1\}$ by \eqref{branches} with $y=0$. Hence
\eqref{cell} with $t=\xi_{n+1}/q$ gives
$T_q(a_{n+1})=T_q(a_n)+q^{-n-1}(q-\xi_{n+1}+d_n)\,\xi_{n+1}$, and therefore
\begin{equation}
 y_{n+1}=qy_n-\bigl(q-\xi_{n+1}+d_n\bigr)\,\xi_{n+1}
 \label{digstep}
\end{equation}
at every height $y$. For $q=3$, $y=3/4$, and $d_n=2\kappa_n$ this is
\eqref{lrec}. Thus the pair $(d_{n+1},y_{n+1})$
depends only on $(d_n,y_n)$ and $\xi_{n+1}$, and so does the question
whether $\xi_{n+1}$ is excluded. We say that the digit $\xi_{n+1}$ leads
from the pair $(d_n,y_n)$ to the pair $(d_{n+1},y_{n+1})$, and
that it is excluded at $(d_n,y_n)$ if it is excluded for the height
$y$.

\begin{lemma}
\label{digpaths}
Let $0\le y\le\Lambda(0)$, and let $\xi_1,\xi_2,\ldots$ be digits in
$\{0,1,\ldots,q-1\}$.
\begin{enumerate}
\item[(a)] For every $n\ge0$, $T_q$ takes the value $y$ somewhere on
$[a_n,c_n]$ if and only if $y_n\in\mathcal T(d_n)$.
\item[(b)] Suppose that $y$ is not a value of $T_q$ at a $q$-adic rational.
Then no point of $L(y)$ is a $q$-adic rational, and $\xi_1,\xi_2,\ldots$
are the digits of a point of $L(y)$ if and only if no digit is excluded.
\item[(c)] Let $y$ be as in (b), let $y'\in[0,\Lambda(0)]$ also be a
height that
is not a value of $T_q$ at a $q$-adic rational, and let $n,n'\ge0$. Let
$\xi'_1,\ldots,\xi'_{n'}$ be digits, and let $d'_{n'}$ and $y'_{n'}$
be the numbers of Definition~\ref{excl} for these digits and the height
$y'$, so that $(d'_0,y'_0)=(0,y')$. Suppose that no digit among
$\xi_1,\ldots,\xi_n$ is excluded for the height $y$, that no digit among
$\xi'_1,\ldots,\xi'_{n'}$ is excluded for the height $y'$, and that
$(d_n,y_n)=(d'_{n'},y'_{n'})$. Then $L(y)$ has as many points
whose digits begin with $\xi_1,\ldots,\xi_n$ as $L(y')$ has points whose
digits begin with $\xi'_1,\ldots,\xi'_{n'}$, where the condition on the
digits is empty if $n=0$ or $n'=0$.
\end{enumerate}
\end{lemma}

\begin{proof}
Since $q^na_n$ is the integer whose word consists of
$\xi_1,\ldots,\xi_n$, \eqref{cell} and \eqref{dnsq} give
$T_q(a_n+q^{-n}u)=T_q(a_n)+q^{-n}(d_nu+T_q(u))$ for $0\le u\le1$. Thus
the set of the values of $T_q$ on $[a_n,c_n]$ is $T_q(a_n)+q^{-n}\mathcal T(d_n)$,
which proves (a). If a point of $L(y)$ were a $q$-adic rational, $y$ would
be a value of $T_q$ at a $q$-adic rational. A point $x\in L(y)$ lies in the
intervals $[a_n,c_n]$ given by its digits, so no digit is excluded, by
(a). Conversely, if no digit is excluded, then $T_q$ takes the value $y$ on
every interval $[a_n,c_n]$, by (a) and since $y_0=y\in\mathcal T(0)$. These
intervals contain the point $x$ with the digits $\xi_1,\xi_2,\ldots$, and
their lengths tend to $0$, so $T_q(x)=y$ by the continuity of $T_q$.
Consequently $x$ is not a $q$-adic rational, and $\xi_1,\xi_2,\ldots$ is
its only digit sequence. This proves (b). By \eqref{digstep}, the pair
$(d_{j+1},y_{j+1})$ and the question whether $\xi_{j+1}$ is excluded
depend only on $(d_j,y_j)$ and $\xi_{j+1}$. Hence the sequences of
digits that continue $\xi_1,\ldots,\xi_n$ without excluded digits for the
height $y$ are the same as those that continue $\xi'_1,\ldots,\xi'_{n'}$
without excluded digits for the height $y'$, and (c) follows from (b).
\end{proof}

For $q\le3$ the next lemma determines the values of $T_q$ at the points of
Lemma~\ref{monopos}.

\begin{lemma}
\label{posvalues}
Let $q\le3$. The function $T_q$ maps the set of the points $u\in[0,1)$ that
are not
$q$-adic rationals and satisfy $d_j(u)\ge1$ for all $j\ge1$ onto
$(0,(q-1)/q]$. For $q=3$ the condition reads $\kappa_j(u)\ge1$ for all
$j\ge1$, and the image is $(0,2/3]$.
\end{lemma}

\begin{proof}
Let $u$ be a point of this set. Its first digit is $0$, since
$d_1(u)=q-1-2\xi_1\ge1$ and $q\le3$. Hence $0<u<1/q$, and $w=qu$ is a point
of $(0,1)$ that is not a $q$-adic rational and satisfies
$d_j(w)=d_{j+1}(u)-(q-1)\ge2-q$ for $j\ge1$. If $q=2$, this says that
$d_j(w)\ge0$. If $q=3$, it says
that $d_j(w)\ge-1$, and then $d_j(w)\ge0$, because the digit slopes are
even. By \eqref{branches} with $a=0$ and Lemma~\ref{tiltpos} with
$k=0$, where $\lfloor(q+1)/2\rfloor=q-1$ because $q\le3$, we obtain
$qT_q(u)=T_q(w)+(q-1)w\le q-1$. Moreover, $T_q(u)>0$ by \eqref{Tpositive}.

Conversely, let $0<y\le(q-1)/q$. In the notation of
Definition~\ref{excl} we choose digits $\xi_1,\xi_2,\ldots$ such that $d_n\ge1$ and
$0<y_n\le d_n$ for every $n\ge1$. We take $\xi_1=0$, so that
$d_1=q-1$ and $y_1=qy$. Suppose that $\xi_1,\ldots,\xi_n$ are chosen,
where $n\ge1$. For integers $0\le i\le q$ put $G_i=i(q-i+d_n)$, so that
$G_0=0$, $G_q=qd_n$, and $G_{i+1}-G_i=d_n+q-1-2i$. Since
$G_0<qy_n\le G_q$, there is a largest integer $i$ with $0\le i\le q-1$
and $G_i<qy_n$, and then $qy_n\le G_{i+1}$. We take $\xi_{n+1}=i$. By
\eqref{digstep}, $y_{n+1}=qy_n-G_i$ and $d_{n+1}=G_{i+1}-G_i$, so that
$0<y_{n+1}\le d_{n+1}$, and $d_{n+1}\ge1$ because $d_{n+1}$ is a positive
integer. Since
$(0,d_n]\subseteq \mathcal T(d_n)$ for $n\ge1$, and $y_0=y\in\mathcal T(0)$ because
$y\le(q-1)/q=T_q(1/q)\le\Lambda(0)$ by Proposition~\ref{incr}(b),
Lemma~\ref{digpaths}(a) shows that
$T_q$ takes the value $y$ on every interval $[a_n,c_n]$. The point $u$ with
the digits $\xi_1,\xi_2,\ldots$ lies in all these intervals, so $T_q(u)=y$
by continuity. The digits are not eventually $q-1$, because $d_n\ge1$
for all $n\ge1$. If $\xi_j=0$ for all $j>n$, then $y_{n+i}=q^iy_n$ and
$d_{n+i}=d_n+(q-1)i$ for $i\ge0$, which contradicts
$y_{n+i}\le d_{n+i}$ for large $i$. Since the digit sequences of a
$q$-adic rational are eventually $0$ or eventually $q-1$, $u$ is not a
$q$-adic rational, and
$d_j(u)=d_j\ge1$ for all $j\ge1$.
\end{proof}

For $q=4$ Lemma~\ref{posvalues} is false (see Example~\ref{ex:base4}).
Its proof uses $q\le3$ to show that the first digit of $u$ is $0$, in the
step from $d_j(w)\ge2-q$ to $d_j(w)\ge0$, and for
$\lfloor(q+1)/2\rfloor=q-1$.

\begin{lemma}
\label{lastzero}
Let $q\le3$, and let $y\in[0,\Lambda(0)]$ be a number that is not a value
of $T_q$ at a $q$-adic rational. For a point $x\in L(y)$ with
$d_j(x)=0$ for finitely many $j$ and with $d_j(x)\ge0$ for all $j\ge1$,
let $n$ be the largest index with $d_n(x)=0$, and let
$m=\lfloor q^nx\rfloor$. Then $x\mapsto(n,m)$ is a bijection from the
set of these points onto the set of the pairs $(n,m)\in\mathcal A$ with
\begin{equation}
 T_q(m/q^n)<y<T_q(m/q^n)+(q-1)q^{-n-1}.
 \label{locsum}
\end{equation}
For $q=3$ the right side of \eqref{locsum} is
$T_q(m/3^n)+\frac23\,3^{-n}$.
\end{lemma}

\begin{proof}
Let $x$ be such a point, let $n$ be the largest index with $d_n(x)=0$,
and let $m=\lfloor q^nx\rfloor$. The word $v$ of $m$ consists of the first $n$
digits of $x$, so $d_j(v)=d_j(x)\ge0$ for $j\le n$ and $d_n(v)=0$. Hence
$(n,m)\in\mathcal A$. The point $u=q^nx-m$, whose
digits are those of $x$ after the place $n$, is not a $q$-adic rational,
and $d_j(u)=d_{n+j}(x)\ge1$ for $j\ge1$, because
$d_{n+j}(x)\ge0$ and $d_{n+j}(x)\ne0$ by the choice of $n$. By
\eqref{lcell}, $y=T_q(x)=T_q(m/q^n)+q^{-n}T_q(u)$, and $0<T_q(u)\le(q-1)/q$ by
Lemma~\ref{posvalues}. We have $T_q(1/q)=(q-1)/q$ by \eqref{branches}. If
$T_q(u)=(q-1)/q$, then \eqref{cell} gives
$y=T_q((qm+1)/q^{n+1})$, contrary to the hypothesis on $y$. Hence $(n,m)$
satisfies \eqref{locsum}. If another point $x'$ of this kind gives the
same pair, then $q^nx'-m$ has the same properties as $u$ and the same
value under $T_q$, so $q^nx'-m=u$ by Lemma~\ref{monopos}, and $x'=x$.

Conversely, let $(n,m)\in\mathcal A$ satisfy \eqref{locsum}, and let $v$
be the word of $m$ of length $n$. By Lemma~\ref{posvalues} there is a
point $u\in[0,1)$ that is not a $q$-adic rational, with $d_j(u)\ge1$ for
all $j\ge1$ and $T_q(u)=q^n(y-T_q(m/q^n))$. The digits of $x=(m+u)/q^n$
are the letters of $v$ followed by the digits of $u$. By
Definition~\ref{pairs}, $d_n(v)=0$ and $d_{j+1}(v)\ge0$ for every index
$j<n$ with $d_j(v)=0$. Since $q\le3$, the digit slopes of $v$ between two
consecutive indices $i<i'$ with $d_i(v)=d_{i'}(v)=0$ have the same sign,
by Remark~\ref{rem:signs}, so $d_j(v)\ge0$ for $j\le n$. Hence
$d_j(x)\ge0$ for $j\le n$, $d_n(x)=0$, and $d_{n+j}(x)=d_j(u)\ge1$ for
$j\ge1$. Thus $n$ is the largest index with $d_n(x)=0$, and
$\lfloor q^nx\rfloor=m$.
Moreover, $T_q(x)=T_q(m/q^n)+q^{-n}T_q(u)=y$ by \eqref{cell}, since
the interval $[m/q^n,(m+1)/q^n]$ has slope $0$.
\end{proof}

For $q=4$ Lemma~\ref{lastzero} is false as well (see the next example).
Its proof uses $q\le3$ through Lemma~\ref{posvalues},
Lemma~\ref{monopos}, and Remark~\ref{rem:signs}.

\begin{example}
\label{ex:base4}
Let $q=4$. The point $1/3=(0.\overline1)_4$ is not a $4$-adic rational,
and $d_j(1/3)=j$ for every $j\ge0$. Proposition~\ref{ratval} with $N=0$
and $p=1$ gives $T_q(1/3)=\frac43\tau_q(1/3)=\frac43\cdot\frac56=\frac{10}9$, which is larger
than $(q-1)/q=3/4$. Thus Lemma~\ref{posvalues} is false for $q=4$. Put
$y=10/9$. By \eqref{samp}, $4^nT_q(m/4^n)$ is an integer for $n\ge1$ and
$0\le m\le4^n$, so $y$ is not a value of $T_q$ at a $4$-adic rational.
The only index $n$ with $d_n(1/3)=0$ is $0$, and \eqref{locsum} for the pair
$(0,0)$ reads $0<y<3/4$, which is false. Thus Lemma~\ref{lastzero} fails
at the point $1/3$ of the level set $L(y)$.

We use the notation of Definition~\ref{excl}. By
\eqref{tilt0} and \eqref{tilt1}, $\Lambda(1)=4/5$ and $\Lambda(3)=3/10$,
and $\Lambda(c)\le\Lambda(1)<1$ for every integer $c\ge1$. By
\eqref{digstep}, the digits $0$, $1$, $2$, and $3$ lead from $(0,y)$ to
$(3,40/9)$, $(1,13/9)$, $(-1,4/9)$, and $(-3,13/9)$, and the first and the
last of these digits are excluded, because $40/9>3+\Lambda(3)$ and
$13/9>\Lambda(3)$. Let $d\ge1$. From $(d,(3d+10)/9)$ the four digits lead
to $(d+3,(12d+40)/9)$, $(d+1,(3d+13)/9)$, $(d-1,(4-6d)/9)$, and
$(d-3,(13-15d)/9)$. The digit $0$ is excluded because $(12d+40)/9>d+4$,
the digit $2$ because $(4-6d)/9<0\le d-1$, and, if $d\ge2$, the digit $3$
because $(13-15d)/9$ is less than $0$ and less than $d-3$. The digit $1$
is not excluded, since $0<(3d+13)/9\le d+1$, and for $d=1$ the digit $3$
leads to $(-2,-2/9)$, which is not excluded either, since
$\mathcal T(-2)=[-2,\Lambda(2)]$. Hence, by Lemma~\ref{digpaths}(b), $1/3$ is the
only point of $L(y)$ whose digits begin with $1,1$, and every other point
of $L(y)$ whose digits begin with $1$ begins with $1,3$. The digits $2$
and $0$ lead from $(0,y)$ to $(-1,4/9)$ and then to $(2,16/9)$, the pair
reached after $1,1$. By \eqref{digstep}, the digits that can follow $2,0$
are therefore those that can follow $1,1$, and $25/48=(0.20\overline1)_4$
is the only point of $L(y)$ whose digits begin with $2,0$. Since
$T_q(1-x)=T_q(x)$ and the digits of $1-x$ are the complements of those of
$x$, the points of $L(y)$ whose digits begin with $2$ are the points $1-x$
with $x\in L(y)$ beginning with $1$, and those beginning with $1,3$ are
the points $1-x$ with $x$ beginning with $2,0$. Therefore
$L(y)=\{1/3,23/48,25/48,2/3\}$.

Each of these four points $x$ satisfies $|d_n(x)|=n$ for every $n\ge0$.
If $(n,m)\in\mathcal A(y)$, then $T_q$ takes the value $y$ at a point $x$
of $[m/4^n,(m+1)/4^n]$. This point lies in $L(y)$ and is not a $4$-adic rational, so this interval is the interval of level $n$ that contains
$x$, and its slope $d_n(x)$ is zero only if $n=0$. Thus
$\mathcal A(y)=\{(0,0)\}$, and Proposition~\ref{finitelevel} with $r=1$
would give at most $2$ points, while $L(y)$ has four.
\end{example}

Lagarias and Maddock~\cite[Theorem~1.6]{LagariasMaddock2012} proved that
for Takagi's function the expected number of local level sets contained
in the level set at a height chosen uniformly in its range $[0,2/3]$ is
$3/2$. Allaart~\cite{AllaartHowLarge2012} gave another proof of this
value. For $q=2$ the next theorem gives this value again, and it
shows that the expected number is $2$ for $q=3$.

\begin{theorem}
\label{expected}
Let $q\le3$. The number of local level sets contained in $L(y)$ is at
least $1$ for
$0<y<\Lambda(0)$. We give this number the value $+\infty$ if $L(y)$
contains
infinitely many local level sets. As a function of $y\in[0,\Lambda(0)]$,
it is
Lebesgue measurable, and its integral over $[0,\Lambda(0)]$ is $2$ if
$q=3$ and $1$ if $q=2$. Thus, at a height chosen uniformly in
$[0,\Lambda(0)]$, the expected number of local level sets is $2$ for
$q=3$, where $\Lambda(0)=1$, and $3/2$ for $q=2$, where
$\Lambda(0)=2/3$.
\end{theorem}

\begin{proof}
Let $t^*=1/2$ and $p=1$ if $q=3$, and let $t^*=1/3$ and $p=2$ if $q=2$.
The digits of $t^*$ are $1,1,1,\ldots$ in the first case and
$0,1,0,1,\ldots$ in the second, and $T_q(t^*)=\Lambda(0)$ by
Theorem~\ref{gmax}. For $k\ge0$ let $m_k=\lfloor q^{kp}t^*\rfloor$. The
word of $m_k$ of length $kp$ consists of the first $kp$ digits of $t^*$.
Its digit slopes are $0$ or $1$, and the last one is $0$, so
$(kp,m_k)\in\mathcal A$. The digits of $t^*$ after the place $kp$ are
again the digits of $t^*$, so \eqref{lcell} gives
$\Lambda(0)=T_q(m_k/q^{kp})+q^{-kp}\Lambda(0)$, that is,
$T_q(m_k/q^{kp})=\Lambda(0)(1-q^{-kp})$. Moreover,
$\Lambda(0)(1-q^{-p})=(q-1)/q$ in both cases. Now let $0<y<\Lambda(0)$,
and let $k\ge0$ be the integer with
$\Lambda(0)(1-q^{-kp})<y\le\Lambda(0)(1-q^{-(k+1)p})$. Then
$q^{kp}\bigl(y-T_q(m_k/q^{kp})\bigr)$ is a number in $(0,(q-1)/q]$. By
Lemma~\ref{posvalues} there is a point $u\in[0,1)$ that is not a $q$-adic
rational, with $d_j(u)\ge1$ for all $j\ge1$, at which $T_q$ takes this
value. The interval $[m_k/q^{kp},(m_k+1)/q^{kp}]$ has slope $0$, so by
\eqref{cell} the point $x=(m_k+u)/q^{kp}$ satisfies
$T_q(x)=T_q(m_k/q^{kp})+q^{-kp}T_q(u)=y$, and it is not a $q$-adic rational. Hence $L(y)$
contains the local level set of $x$, by Lemma~\ref{flip}.

Now let $y\in[0,\Lambda(0)]$ be a height that is not a value of $T_q$ at a
$q$-adic rational and for which $\mathcal A(y)$ is finite. By
Lemma~\ref{locstruct}(a) and~(b), the local level sets contained in
$L(y)$ correspond to the points $x\in L(y)$ with $d_j(x)\ge0$ for
all $j\ge1$. For such a point $x$, the point $x'$ of
Lemma~\ref{foldzero} is $x$ itself, so every index $n$ with $d_n(x)=0$
gives a pair $(n,\lfloor q^nx\rfloor)\in\mathcal A(y)$. Different such
indices give different pairs, so $d_n(x)=0$ for finitely many $n$. By
Lemma~\ref{lastzero}, the number of local level sets
contained in $L(y)$ is therefore the number of pairs $(n,m)\in\mathcal A$
that satisfy \eqref{locsum}.

As a function of $y\in[0,\Lambda(0)]$, this number of pairs is a sum of
indicator functions of intervals, so it is a Borel function. The interval
of heights given by \eqref{locsum} has length $(q-1)q^{-n-1}$ and lies in
$[0,\Lambda(0)]$, since $T_q(m/q^n)+(q-1)q^{-n-1}=T_q((qm+1)/q^{n+1})$ by
\eqref{cell}. By monotone convergence and Lemma~\ref{aplus}, the integral
of this function over $[0,\Lambda(0)]$ is
$\sum_{(n,m)\in\mathcal A}(q-1)q^{-n-1}$, which is $\frac23\cdot3=2$ for
$q=3$ and $\frac12\cdot2=1$ for $q=2$. The values of $T_q$ at $q$-adic
rationals form a countable set. In the proof of
Theorem~\ref{levelthm}(a) we saw that the number of elements of
$\mathcal A(y)$ is a sum of indicator functions of closed intervals with
a finite integral over $[0,\Lambda(0)]$, so the heights $y$ for which
$\mathcal A(y)$ is infinite form a Borel set of measure zero. Outside
these two sets, the number of local level sets contained in $L(y)$ is
equal to the number of pairs above. Since Lebesgue measure is complete,
the number of local level sets is a Lebesgue measurable function of $y$
with the same integral over $[0,\Lambda(0)]$.
\end{proof}

For $q\ge4$ we do not know the integral in Theorem~\ref{expected}. Its
proof counts the local level sets with Lemma~\ref{lastzero}, which is
false for $q=4$ (see Example~\ref{ex:base4}).

\begin{figure}[H]
\centering
\includegraphics{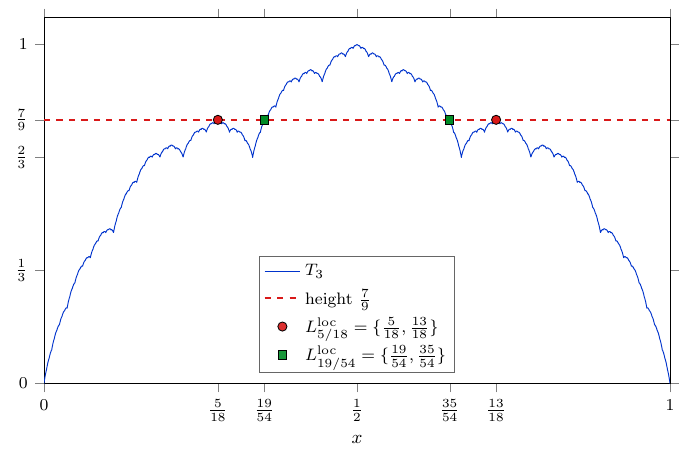}
\caption{The case $q=3$: the two local level sets contained in $L(7/9)$
(see Example~\ref{ex:expected}).}
\label{locpic}
\end{figure}

\begin{example}
\label{ex:expected}
Let $q=3$ and $y=7/9$. Since $0<y<1$, Theorem~\ref{expected} shows that
$L(7/9)$ contains at least one local level set, and in fact it contains
exactly two. By Theorem~\ref{cardmod}, $L(7/9)$ has four points. They are
$5/18=(0.02\overline1)_3$, $19/54=(0.100\overline1)_3$,
$35/54=(0.122\overline1)_3$, and $13/18=(0.20\overline1)_3$, since
\eqref{cell}, the symmetry $T_q(1-x)=T_q(x)$, and
Proposition~\ref{incr}(b) give
$T_q(5/18)=\bigl(2\cdot\frac56+T_q(\frac56)\bigr)/3=\bigl(\frac53+T_q(\frac16)\bigr)/3=7/9$
and $T_q(19/54)=T_q(\frac13)+T_q(\frac1{18})/3=7/9$, and the same symmetry
gives $T_q(35/54)=T_q(19/54)$ and $T_q(13/18)=T_q(5/18)$. The indices
$n$ with $d_n(5/18)=0$ are $0,2,3,4,\ldots$, and the strings after them
are $02,1,1,\ldots$, so
$L^{\mathrm{loc}}_{5/18}=\{5/18,13/18\}$. The indices $n$ with
$d_n(19/54)=0$ are $0$ and $1$, and the strings after them are $1$ and
$00111\ldots$, so
$L^{\mathrm{loc}}_{19/54}=\{19/54,35/54\}$. Figure~\ref{locpic} shows the four points, and
the two points of each local level set have the same mark. In this
figure the graph of $T_3$ is drawn through its values at the multiples
of $1/3000$, computed from the first $36$ terms of the series~\eqref{T}.
\end{example}

For Takagi's function, Allaart~\cite{AllaartHowLarge2012} proved that if
$L(y)$ is finite, then the number of its points is twice the number of
the intervals $[m/2^n,(m+1)/2^n]$ of slope $0$ with $n\ge0$ and
$T_2(m/2^n)\le y\le T_2(m/2^n)+2^{-n-1}$. For $q\le3$ the next theorem
counts the points of the finite level sets at the
heights outside a countable set.

\begin{theorem}
\label{cardmod}
Let $q\le3$, let $y\in[0,\Lambda(0)]$ be a height outside the countable
set of the values of
$T_q$ at the $q$-adic rationals and at the midpoints $(m+1/2)/q^n$ of the
intervals of all levels $n\ge0$, and suppose that $L(y)$ is finite. Then
\begin{equation}
 \#L(y)=\sum_{(n,m)}2^{k+1},
 \label{locformula}
\end{equation}
where the sum is over the pairs $(n,m)\in\mathcal A$ that satisfy
\eqref{locsum}, $v$ is the word of $m$, and $k$ is the number of indices
$j$ with $0\le j<n$ and $d_j(v)=0$ such that the letter $\zeta$ of $v$ at
the place $j+1$ satisfies $2\zeta\ne q-1$. For $q=2$ the midpoints are
$2$-adic rationals, and $k$ is the number of all indices $j$ with
$0\le j<n$ and $d_j(v)=0$.
\begin{enumerate}
\item[(a)] Let $q=2$. Then $\#L(y)\equiv2\pmod4$ if $y<1/2$, and
$\#L(y)\equiv0\pmod4$ if $y>1/2$.
\item[(b)] Let $q=3$. Then $\#L(y)\equiv2\pmod4$. The heights
$3^{-i}/2$ and
$4/9+3^{-4-i}/2$ with $i\ge0$ lie outside this countable set, and
\[
 \#L(3^{-i}/2)=2,~\#L(4/9+3^{-4-i}/2)=4i+6.
\]
Hence the numbers of points of the finite level sets at the heights
outside this countable set are exactly $2,6,10,14,\ldots$. The height
$7/9=T_q(5/18)$ is not a value of $T_q$ at a ternary rational, and $L(7/9)$
has four points. Here $5/18$ is the midpoint of $[2/9,1/3]$.
\end{enumerate}
\end{theorem}

\begin{proof}
No point of $L(y)$ is a $q$-adic rational. If $q=3$, no point of $L(y)$
has digits that are eventually $1$, because a point with the digits
$\xi_1,\ldots,\xi_n,1,1,\ldots$ is the midpoint of an interval of
level $n$. Suppose that a point $x\in L(y)$ satisfies $d_n(x)=0$ for
infinitely many $n$. If $q=2$, every such $n$ satisfies
$2\xi_{n+1}\ne q-1$. If $q=3$, then, as in the proof of
Theorem~\ref{levelthm}(b), infinitely many of these indices $n$ are
followed by a digit different from $1$, since
the digits of $x$ are not eventually $1$. By Lemma~\ref{locstruct}(c), $L^{\mathrm{loc}}_x$
is then uncountable, and $L^{\mathrm{loc}}_x\subseteq L(y)$ by
Lemma~\ref{flip}, contrary to the assumption that $L(y)$ is finite. Hence
every point $x\in L(y)$ satisfies $d_n(x)=0$ for finitely many $n$. By
Lemma~\ref{locstruct}(a) and~(b), $L(y)$ is the disjoint union of the
local level sets of its points $x$ with $d_j(x)\ge0$ for all $j\ge1$,
and by Lemma~\ref{lastzero} these points correspond to the pairs
$(n,m)\in\mathcal A$ that satisfy \eqref{locsum}. If $x$ corresponds to
$(n,m)$ and $v$ is the word of $m$, then the indices $j$ with $d_j(x)=0$
are the indices $j\le n$ with $d_j(v)=0$, and for $j<n$ the digit
$\xi_{j+1}$ of $x$ is the letter of $v$ at the place $j+1$. The largest
index $n$ with $d_n(x)=0$ is followed by a digit $\xi_{n+1}$ with
$2\xi_{n+1}\ne q-1$, because $d_{n+1}(x)\ne0$. Therefore
$L^{\mathrm{loc}}_x$ has $2^{k+1}$
points by Lemma~\ref{locstruct}(c), and \eqref{locformula} follows.

We have $k=0$ exactly when every letter $\zeta$ of the word $v$ of $m$
satisfies $2\zeta=q-1$.
Indeed, if $k=0$, then every index $j<n$ with $d_j(v)=0$ is followed by
such a letter, which gives $d_{j+1}(v)=0$, so by induction on $j$ all
letters have this property. Conversely, if all letters of $v$ are such
letters, then every $j<n$ satisfies $d_j(v)=0$ and is followed by such a
letter.

Let $q=2$. No letter $\zeta$ satisfies $2\zeta=1$, so $k=0$ only for the
pair $(n,m)=(0,0)$, and for this pair \eqref{locsum} reads $0<y<1/2$.
All other terms of \eqref{locformula} are divisible by $4$. The height
$1/2=T_q(1/2)$ does not satisfy the hypothesis. This proves (a).

Let $q=3$ in the rest of the proof. The word that consists of $n$
letters $1$ is the word of $m=(3^n-1)/2$, and, as in the proof of
Theorem~\ref{expected}, $(n,m)\in\mathcal A$ and $T_q(m/3^n)=1-3^{-n}$,
so that \eqref{locsum} reads
$1-3^{-n}<y<1-3^{-n-1}$. These intervals with $n\ge0$ are disjoint, and
their union is $(0,1)$ without the points
$1-3^{-n-1}=T_q\bigl((3^{n+1}-1)/(2\cdot3^{n+1})\bigr)$, which are values of
$T_q$ at ternary rationals. The heights $0=T_q(0)$ and $1=T_q(1/2)$ do not
satisfy the hypothesis either, so exactly one term of \eqref{locformula}
equals $2$, and all other terms are divisible by $4$. This proves that
$\#L(y)\equiv2\pmod4$.

By \eqref{samp} and \eqref{xi}, $T_q(m/3^n)$ is an even multiple of
$3^{-n}$, and hence of $3^{-N}$ for every $N\ge n$, if $n\ge1$ and
$0\le m\le3^n$. Every ternary rational in $[0,1]$ has this form, and $T_q$
has period $1$. By \eqref{cell} with $t=1/2$, the value of $T_q$ at the
midpoint $(m+1/2)/3^n$ of a ternary interval of level $n\ge1$ is an
integer multiple of $3^{-n}$, and $T_q(1/2)=1$. The heights $3^{-i}/2$ and
$4/9+3^{-4-i}/2$ have even denominators in lowest terms, so they are
outside the countable set of the theorem. For $N\ge2$ the numbers
$3^N\cdot7/9$ and $3^N/3$ are odd integers, so the heights $7/9$ and $1/3$
are not values of $T_q$ at ternary rationals.

We count the points of these level sets with Lemma~\ref{digpaths}(b)
and~(c), in the notation of Definition~\ref{excl}, and we say that a
point begins with a word if its digits begin with this word. If $0<y<2/3$
and $y$ is not a value of $T_q$ at a ternary rational, then the digit $1$
is excluded at $(0,y)$, because it leads to $(0,3y-2)$ by
\eqref{digstep},
and $3y-2<0$, while $\mathcal T(0)=[0,1]$. Moreover, since $T_q(1-x)=T_q(x)$ and the
digits of $1-x$ are the complements of those of $x$, the map
$x\mapsto1-x$ sends the points of $L(y)$ that begin with $0$ onto those
that begin with $2$. The digit $0$ leads from $(0,y)$ to $(2,3y)$.

Let $y=3^{-i}/2$ with $i\ge0$. By \eqref{samp} and \eqref{xi},
$y_n-3^{n-i}/2=-3^nT_q(a_n)$ is an even integer for $n\ge1$ and every
choice of the digits. If $1\le n<i$ and the first $n$ digits are $0$,
then $(d_n,y_n)=(2n,3^{n-i}/2)$, and the digits $1$ and $2$ are
excluded, because they give $y_{n+1}<0$ by \eqref{digstep}, while
$\mathcal T(2n)$ and
$\mathcal T(2n-2)$ contain no negative numbers. Thus every point of $L(y)$ that
begins with $0$ begins with $\max\{i,1\}$ digits $0$. Now consider the
digits of a point of $L(y)$ that begins with $0$, let
$n\ge\max\{i,1\}$, and suppose that $d=d_n\ge2$. Then
$3y_n-1/2$ is an integer. By \eqref{digstep}, the digit $0$ is not
excluded
only if $3y_n\le d+2+3^{-d/2-1}$, the digit $1$ only if
$d+2\le3y_n\le2d+2+3^{-d/2}$, and the digit $2$ only if
$3y_n\ge2d+2$. Two of these conditions together would put $3y_n$ in
$[d+2,d+2+1/9]$ or in $[2d+2,2d+2+1/3]$, or would
contradict each other. These intervals contain no number whose
difference from $1/2$ is an integer, so at most one digit is not
excluded. If $d=2$, the digit $2$ is not excluded only if
$3y_n-6\in\mathcal T(0)$, that is, $2\le y_n\le7/3$, which is impossible, because
$y_n-1/2$ is an integer. Consequently the digit that is not excluded
leads to a pair with $d_{n+1}\ge2$. By induction, at most one point
of $L(y)$ begins with $0$. There is one, because $L(y)$ is not empty and
$x\mapsto1-x$ exchanges the points that begin with $0$ and with $2$.
Thus $\#L(y)=2$.

Next let $y=4/9+\delta$ with $\delta=3^{-4-i}/2$ and $i\ge0$, so that
$0<9\delta\le1/18$ and $y<2/3$. We use \eqref{digstep} and the ranges
$\mathcal T(-2)=[-2,1/3]$, $\mathcal T(0)=[0,1]$, $\mathcal T(2)=[0,7/3]$, $\mathcal T(4)=[0,4+1/9]$, and
$\mathcal T(6)=[0,6+1/27]$. The digit $0$ leads from $(0,y)$ to
$(2,4/3+3\delta)$. At this pair the digit $2$ is excluded, and the digits
$0$ and $1$ lead to $(4,4+9\delta)$ and $(2,9\delta)$. The only digit
that is not excluded at $(4,4+9\delta)$ is $2$, and it leads to
$(2,2+27\delta)$. For $0\le j\le i$, the digit $0$ is excluded at
$(2,2+3^{3+j}\delta)$, the digit $2$ leads from this pair to
$(0,3^{4+j}\delta)=(0,3^{j-i}/2)$, and the digit $1$ leads to
$(2,2+3^{4+j}\delta)$ and is excluded exactly when $j=i$. The pair
$(2,9\delta)$ is also reached from $(0,3^{-3-i}/2)$ by the digit $0$. By
the case just treated, exactly one point of $L(3^{-3-i}/2)$ begins with
$0$, so exactly one point of $L(y)$ begins with $01$, by
Lemma~\ref{digpaths}(c). Every point of $L(y)$ that begins with $00$
begins with $002$ followed by $j$ digits $1$ and a digit $2$, for some
$0\le j\le i$, and for each $j$ there are $\#L(3^{j-i}/2)=2$ such points,
by Lemma~\ref{digpaths}(c) with $n'=0$. Therefore $L(y)$ has $2i+3$
points that begin with $0$, and $\#L(y)=4i+6$.

Finally, \eqref{cell} gives $T_q(5/18)=T_q(2/9)+T_q(1/2)/9=7/9$, since
$s_3(2)=2$ and $T_q(2/9)=2/3$. At $(0,7/9)$ no digit is excluded, and the
digits $0$, $1$, and $2$ lead to $(2,7/3)$, $(0,1/3)$, and $(-2,1/3)$.
The only digit that is not excluded at $(2,7/3)$ is $2$, which leads to
$(0,1)$. At $(0,1)$ only the digit $1$ is not excluded, and it leads to
$(0,1)$ again. Hence $5/18=(0.02\overline1)_3$ is the only point of
$L(7/9)$ that begins with $0$, and by symmetry $13/18$ is the only one
that begins with $2$. The digit $1$ leads to $(0,1/3)$, so by
Lemma~\ref{digpaths}(c) with $n'=0$, $L(7/9)$ has as many points that
begin with $1$ as $L(1/3)$ has points. From $(0,1/3)$ the digit $0$ leads
to $(2,1)$, where only the digit $0$ is not excluded. This digit leads to
$(4,3)$. There only the digit $1$ is not excluded, and it leads to
$(4,3)$ again. Thus $1/18=(0.00\overline1)_3$ is the only point of
$L(1/3)$ that begins with $0$. Since $1/3<2/3$, no point of $L(1/3)$
begins with $1$, and by symmetry $L(1/3)=\{1/18,17/18\}$. By
\eqref{branches} with $a=1$, $T_q((1+t)/3)=(T_q(t)+2)/3$ for $0\le t\le1$, so
the points of $L(7/9)$ that begin with $1$ are the points $(1+t)/3$ with
$t\in L(1/3)$. It follows that $L(7/9)=\{5/18,19/54,35/54,13/18\}$.
\end{proof}

For $q=4$ formula \eqref{locformula} is false. At the height $y=10/9$ of
Example~\ref{ex:base4}, the set $L(y)$ has four points and
$\mathcal A(y)=\{(0,0)\}$. Every pair $(n,m)\in\mathcal A$ that satisfies
\eqref{locsum} lies in $\mathcal A(y)$, because $T_q$ takes every value
between $T_q(m/q^n)$ and $T_q((qm+1)/q^{n+1})=T_q(m/q^n)+(q-1)q^{-n-1}$ on
$[m/q^n,(m+1)/q^n]$, and the pair $(0,0)$ does not satisfy
\eqref{locsum}. Hence the right side of \eqref{locformula} is $0$. The
proof of Theorem~\ref{cardmod} uses $q\le3$ through Lemma~\ref{lastzero}.

For $q=3$, the last statement of Theorem~\ref{cardmod}(b) shows that the
values of $T_q$ at the midpoints cannot be omitted from its countable
set, since $7/9$ is such a value and $L(7/9)$ has four points. For
Takagi's function, Allaart~\cite[Theorem~4.5]{AllaartFinite2012} proved that the heights $y$
for which $L(y)$ has exactly two points form a set of Lebesgue measure
between $5/12$ and $35/72$.

\begin{corollary}
\label{twopoints}
Let $q=3$. The heights $y\in[0,1]$ for which $L(y)$ has exactly two
points form a
set of Lebesgue measure at least $2/81$.
\end{corollary}

\begin{proof}
For $n\ge0$ let $m_n=(3^n-1)/2$, whose word of length $n$ consists of $n$
letters $1$. In the proof of Theorem~\ref{cardmod}(b) we saw that
$(n,m_n)\in\mathcal A$, that the interval of heights given by
\eqref{locsum} for the pair $(n,m_n)$ is $(1-3^{-n},1-3^{-n-1})$, and
that these intervals are disjoint and cover $(0,1)$ except for the points
$1-3^{-n}$ with $n\ge1$. For the pair $(n,m_n)$ the number $k$ in
Theorem~\ref{cardmod} is $0$, since every letter of its word is $1$.

Let $U$ be the union of the intervals of heights given by \eqref{locsum}
for the pairs $(n,m)\in\mathcal A$ with $m\ne m_n$. Let $y\in(0,1)$ be a
height that is not of the form $1-3^{-n}$, lies outside the countable set
of Theorem~\ref{cardmod}, and has a finite level set. If $y\notin U$, then
exactly one pair satisfies \eqref{locsum}, namely the pair $(n,m_n)$ whose
interval contains $y$, and \eqref{locformula} gives $\#L(y)=2$. If $y\in
U$, then at least two pairs satisfy \eqref{locsum}, and $\#L(y)\ge4$. By
Theorem~\ref{levelthm}(a), the heights in $[0,1]$ without these three
properties form a set of measure zero. Since Lebesgue measure is
complete, the set of heights $y$ with $\#L(y)=2$ is Lebesgue measurable,
and its measure is that of $(0,1)\setminus U$.

It remains to show that $U$ has measure at most $1-2/81$. The interval
of heights given by \eqref{locsum} for $(n,m)$ has length
$\frac23\,3^{-n}$. By Lemma~\ref{aplus}, the intervals of all pairs in
$\mathcal A$ have total length $2$, and those of the pairs $(n,m_n)$ have
total length $\sum_{n=0}^{\infty}\frac23\,3^{-n}=1$, so the intervals that make
up $U$ have total length $1$. Two of them overlap. The words $02$ and
$012$ give the pairs $(2,2)$ and $(3,5)$ in $\mathcal A$. By
Proposition~\ref{incr}(b) and the symmetry $T_q(1-x)=T_q(x)$ we have
$T_q(2/3)=T_q(1/3)=2/3$, and \eqref{branches} gives $T_q(5/9)=8/9$ and
$T_q(2/9)=T_q(5/27)=2/3$. Hence the intervals of these two pairs are
$(2/3,20/27)$ and $(2/3,56/81)$, and the second lies in the first.
Therefore $U$ has measure at most $1-2/81$.
\end{proof}

The proof of Corollary~\ref{twopoints} uses \eqref{locformula}, which is
false for $q=4$. For $q\ge4$ we do not know whether the heights $y$ for
which $L(y)$ has exactly two points form a set of positive Lebesgue measure.

At the heights that Theorem~\ref{cardmod} excludes, a level set can be
countably infinite. For Takagi's function, Lagarias and
Maddock~\cite[Theorem~7.1]{LagariasMaddock2012} proved that the level
set at the height $1/2$ is countably infinite. The next example and
Figure~\ref{twothirdspic} show such a level set for $q=3$.

\begin{figure}[H]
\centering
\includegraphics{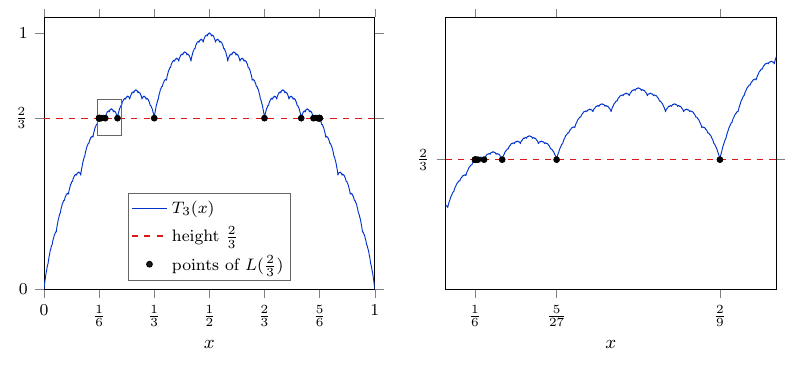}
\caption{The case $q=3$: the level set $L(2/3)$ of
Example~\ref{ex:twothirds}. Left: on
$[0,1]$. Right: near $1/6$.}
\label{twothirdspic}
\end{figure}

\begin{example}
\label{ex:twothirds}
Let $q=3$. The height $2/3$ is the value of $T_q$ at the ternary rational
$1/3$ and at the midpoint $1/6$ of $[0,1/3]$, and
\[
 L(2/3)=\Bigl\{\frac16,\frac13,\frac23,\frac56\Bigr\}
 \cup\Bigl\{\frac16+\frac1{2\cdot3^k},~\frac56-\frac1{2\cdot3^k}:
 k\ge2\Bigr\}.
\]
In particular, $L(2/3)$ is countably infinite.
\end{example}

\begin{proof}
Every term of \eqref{T} is nonnegative, and the term with $j=0$ is
positive unless $t$ is an integer, so for $0\le t\le1$ we have $T_q(t)=0$
only at $t=0$ and $t=1$. By \eqref{branches}, $T_q((1+t)/3)=(T_q(t)+2)/3$ for
$0\le t\le1$, so the points of $L(2/3)$ in $[1/3,2/3]$ are $1/3$ and
$2/3$. Again by \eqref{branches}, a point $t/3$ with $0\le t\le1$ lies in
$L(2/3)$ exactly when $t$ lies in the set $\mathcal M$ of the numbers
$t\in[0,1]$ with $T_q(t)=2-2t$. Since $T_q\le1$ by Lemma~\ref{lem:tilt}, the
set $\mathcal M$ lies in $[1/2,1]$. For $t=(2+u)/3$ with $0\le u\le1$,
\eqref{branches} gives $T_q(t)-2+2t=T_q(u)/3$, so $\mathcal M\cap[2/3,1]=\{2/3,1\}$.
For $t=(1+u)/3$ with $0\le u\le1$, it gives $T_q(t)-2+2t=(T_q(u)-2+2u)/3$, so
$t\in\mathcal M$ exactly when $u\in\mathcal M$. Hence $\mathcal M$ consists
of $2/3$, $1$, and the numbers $(1+u)/3$ with $u\in\mathcal M$. Since
$(1+u)/3-1/2=(u-1/2)/3$ and $(1+1)/3=2/3$, induction on $K$ shows that for
every $K\ge1$ the set $\mathcal M$ consists of $1$, the numbers
$1/2+1/(2\cdot3^{k+1})$ with $0\le k<K$, and the numbers
$1/2+3^{-K}(u-1/2)$ with $u\in\mathcal M$, which lie in the interval of
length $3^{-K}$ with midpoint $1/2$. As $1/2\in\mathcal M$ because
$T_q(1/2)=1$, we obtain
$\mathcal M=\{1/2,1\}\cup\{1/2+1/(2\cdot3^{k+1}):k\ge0\}$. The points of $L(2/3)$ in
$[0,1/3]$ are the numbers $t/3$ with $t\in\mathcal M$, and by the symmetry
$T_q(1-x)=T_q(x)$ the points in $[2/3,1]$ are the numbers $1-t/3$ with
$t\in\mathcal M$.
\end{proof}

In Figure~\ref{twothirdspic}, the graph of $T_3$ is drawn from its exact
values at the points $m/3^7$ on the left and $m/3^{10}$ on the right, the
dashed line marks the height $2/3$, and the dots are the points of
$L(2/3)$ listed in Example~\ref{ex:twothirds} with $k\le9$.

By Theorem~\ref{levelthm}(a), the heights $y$ for which $L(y)$ is
uncountable form a set of Lebesgue measure zero. Theorem~\ref{heightdim}
below shows that in every base this set has Hausdorff dimension $1$ and
is residual in $[0,\Lambda(0)]$, meaning that its complement in this
interval is a countable union of nowhere dense sets.

\begin{definition}
\label{baire}
A set of real numbers is nowhere dense if its closure contains no
interval of positive length, and perfect if it is closed and has no
isolated points. A subset of an interval $J$ is of the first category if
it is the union of countably many nowhere dense sets, and residual in $J$
if its complement in $J$ is of the first~category.
\end{definition}

Garg~\cite[Theorem~1]{Garg1963} proved that a continuous function $f$ on
a closed and bounded interval $J$ is monotone on no subinterval of $J$ if
and only if all its level sets are nowhere dense and its level sets are
nonempty and perfect at the heights of a set that is residual in $f(J)$.
We state the part that we use.

\begin{theorem}[Garg {\cite[Theorem~1]{Garg1963}}]
\label{garg}
Let $f$ be a continuous function on a closed and bounded interval $J$,
and suppose that $f$ is monotone on no subinterval of $J$ of positive
length. Then the set of the numbers $y$ for which $\{x\in J:f(x)=y\}$ is
nonempty and perfect is residual in $f(J)$.
\end{theorem}

\begin{lemma}
\label{nomono}
The function $T_q$ is monotone on no subinterval of $[0,1]$ of positive
length.
\end{lemma}

\begin{proof}
Such a subinterval contains an interval $[m/q^n,(m+1)/q^n]$ of some
level $n$. Let $d$ be its slope.
By \eqref{cell}, $T_q(m/q^n+q^{-n}t)=T_q(m/q^n)+q^{-n}f(t)$ for
$0\le t\le1$, where $f(t)=dt+T_q(t)$, so it suffices to show that $f$ is
not monotone on $[0,1]$. If $d\ge0$, then $f(0)=0$ and $f(1)=d$ are
smaller than the maximum $\Lambda(-d)=d+\Lambda(d)$ of $f$ on $[0,1]$,
because $\Lambda(d)>0$, and a monotone function on $[0,1]$ takes its
maximum at $0$ or at $1$. If $d<0$, then
$f(1-t)=d+|d|t+T_q(t)$, since $T_q(1-t)=T_q(t)$, and the case of $|d|$ shows
that $f$ is not monotone.
\end{proof}

The proof of Theorem~\ref{heightdim}(b) below uses a set of points on
which $T_q$ is strictly increasing. Lemma~\ref{monopos} gives such a set
for $q\le3$, but its statement is false for $q=4$ and for $q=5$. The next
two lemmas show that in every base $T_q$ is strictly increasing on the
set of the points that are not $q$-adic rationals and whose digit slopes
$d_j$ with $j\ge1$ are all at least $\lfloor(q+1)/2\rfloor$. The first
lemma contains Lemma~\ref{tiltpos}, which is the case
$c=\lfloor(q+1)/2\rfloor+k(q-1)$ with an integer $k\ge0$.

\begin{lemma}
\label{abovehalf}
Let $c\ge\lfloor(q+1)/2\rfloor$ be an integer, and let $w\in[0,1]$ be a
point that has a digit sequence whose digit slopes $d_1,d_2,\ldots$
satisfy $c+d_j\ge\lfloor(q+1)/2\rfloor$ for every $j\ge1$. Then
\[
 T_q(w)+cw\le c.
\]
\end{lemma}

\begin{proof}
Put $h=\lfloor(q+1)/2\rfloor$, so that $2h\ge q$. Let $M$ be the
supremum of the numbers $T_q(w)+cw-c$ over all pairs $(c,w)$ that satisfy
the hypothesis of the lemma. There are such pairs, for instance $c=h$
and $w=0$, and $T_q(w)+cw-c\le T_q(w)\le\Lambda(0)$ for each of them, so $M$
is a real number. Let $(c,w)$ be such a pair, and let
$\zeta_1,\zeta_2,\ldots$ be a digit sequence of $w$ whose digit slopes
$d_j$ satisfy $c+d_j\ge h$ for every $j\ge1$. Put $\zeta=\zeta_1$,
$y=\sum_{j=1}^{\infty}\zeta_{j+1}q^{-j}$, and $c'=c+q-1-2\zeta=c+d_1$.
Then $w=(\zeta+y)/q$ and $c'\ge h$. The digit sequence
$\zeta_2,\zeta_3,\ldots$ of $y$ has the digit slopes $d_{j+1}-d_1$, and
$c'+d_{j+1}-d_1=c+d_{j+1}\ge h$ for $j\ge1$. Hence the pair $(c',y)$
satisfies the hypothesis of the lemma, and $T_q(y)+c'y-c'\le M$. By
\eqref{tiltbranch} with $c$ in place of $d$,
\[
 q\bigl(T_q(w)+cw-c\bigr)
 =\zeta(q-\zeta+c)+c'-qc+\bigl(T_q(y)+c'y-c'\bigr)
 \le(q-1-\zeta)(\zeta+1-c)+M,
\]
because $\zeta(q-\zeta+c)+c'=(\zeta+1)(q-1-\zeta+c)$ and
$(\zeta+1)(q-1-\zeta+c)-qc=(q-1-\zeta)(\zeta+1-c)$. Now
$q-1-\zeta\ge0$. Moreover, $c'\ge h$ gives $2\zeta\le c+q-1-h$, and
$q-1-h\le h-1\le c-1$, so that $2\zeta\le2c-1$ and therefore
$\zeta+1\le c$. Hence $q(T_q(w)+cw-c)\le M$ for every pair $(c,w)$ that
satisfies the hypothesis, so $qM\le M$. Since $q\ge2$, it follows that
$M\le0$.
\end{proof}

\begin{lemma}
\label{monoabove}
Let $x<x'$ be points of $[0,1)$ that are not $q$-adic rationals and
satisfy $d_j(x)\ge\lfloor(q+1)/2\rfloor$ and
$d_j(x')\ge\lfloor(q+1)/2\rfloor$ for all $j\ge1$. Then $T_q(x)<T_q(x')$.
\end{lemma}

For $q\le3$ this is Lemma~\ref{monopos}, because for $q=3$ the digit
slopes are even.

\begin{proof}
Put $h=\lfloor(q+1)/2\rfloor$. Let $n\ge0$ be the largest index such
that $x$ and $x'$ have the same first $n$ digits, put
$d=d_n(x)=d_n(x')$, so that $d\ge h$ if $n\ge1$ and $d=0$ if $n=0$, and
let $t$ and $t'$ be the points whose digits are those of $x$ and of $x'$
after the place $n$. Then $t<t'$, the first digits $\xi$ of $t$ and
$\xi'$ of $t'$ satisfy $\xi<\xi'$, and by \eqref{lcell}, applied to $x$
and to $x'$,
\[
 q^n\bigl(T_q(x')-T_q(x)\bigr)=f(t')-f(t),~f(v)=dv+T_q(v).
\]
For integers $0\le i\le q$ put $G_i=i(q-i+d)$, so that
$G_{i+1}-G_i=d+q-1-2i$. Put $c=d+q-1-2\xi=d_{n+1}(x)$ and $w=qt-\xi$.
Then $c\ge h$, and $w$ is a point of $(0,1)$ that is not a $q$-adic rational and has the digits $\xi_{n+2},\xi_{n+3},\ldots$ of $x$, so that
$c+d_j(w)=d_{n+1+j}(x)\ge h$ for $j\ge1$. By \eqref{tiltbranch},
\[
 qf(t)=G_\xi+cw+T_q(w).
\]
Since $c\ge1$, $w>0$, and $T_q(w)\ge0$, we have $qf(t)>G_\xi$, and
Lemma~\ref{abovehalf} gives $qf(t)\le G_\xi+c=G_{\xi+1}$. The same
computation for $x'$ gives $qf(t')>G_{\xi'}$. For $\xi<i<\xi'$ we have
$G_{i+1}-G_i>d+q-1-2\xi'=d_{n+1}(x')>0$, so that
$G_{\xi+1}\le G_{\xi'}$. Therefore
$qf(t)\le G_{\xi+1}\le G_{\xi'}<qf(t')$, and $T_q(x)<T_q(x')$.
\end{proof}

Allaart~\cite[page~1169]{AllaartLevelSets2014} noted that the
Takagi--van~der~Waerden functions are monotone on no interval, so that by
Theorem~\ref{garg} their level sets are uncountable at the heights of a
residual set. Part~(a) of the next theorem is the same statement for
$T_q$ in every base. For Takagi's function, Lagarias and
Maddock~\cite[Theorems~1.4 and~1.5]{LagariasMaddock2011} proved that the
points $x$ for which the level set containing $x$ has positive Hausdorff
dimension form a set of Hausdorff dimension~$1$, and that the same holds
for the heights whose level sets have positive Hausdorff dimension.
Part~(b) contains the same statements for $T_q$ in every base, and in its
proof we follow their proofs~\cite[Sections~6 and~7]{LagariasMaddock2011},
with Lemma~\ref{monoabove} for the monotonicity that they use. In part~(b) we
write $\mathcal N_n$, for $n\ge0$, for the set of the words $v$ of length
$n$ over $\{0,1,\ldots,q-1\}$ with $d_n(v)=0$ and $d_j(v)\ge0$ for
$0\le j\le n$. This set is empty if $(q-1)n$ is odd, because
$d_n(v)-(q-1)n$ is even. If $q\le3$, the words in $\mathcal N_n$ are the
words of the integers $m$ with $(n,m)\in\mathcal A$, by
Definition~\ref{pairs} and Remark~\ref{rem:signs}.

\begin{theorem}
\label{heightdim}
\begin{enumerate}
\item[(a)] The heights $y\in[0,\Lambda(0)]$ for which $L(y)$ is nonempty
and
perfect form a set that is residual in $[0,\Lambda(0)]$. In particular,
$L(y)$ is
uncountable for every $y\in[0,\Lambda(0)]$ outside a set of the first
category.
\item[(b)] For every integer $\ell\ge4$ such that $(q-1)\ell$ is even,
the heights $y\in[0,\Lambda(0)]$ for which
$L(y)$ has Hausdorff dimension at least $\log2/(\ell\log q)$ form a set of
Hausdorff dimension at least
\[
 \frac{\log\#\mathcal N_{\ell-2}}{\ell\log q},
\]
and so do the points $x\in[0,1]$ for which $L(T_q(x))$ has Hausdorff
dimension at least $\log2/(\ell\log q)$. The set of the heights $y$ for
which $L(y)$ has positive Hausdorff dimension and the set of the heights
$y$ for which $L(y)$ is uncountable both have Lebesgue measure zero and
Hausdorff dimension $1$. The points $x\in[0,1]$ for which $L(T_q(x))$ has
positive Hausdorff dimension form a set of Hausdorff dimension $1$.
\end{enumerate}
\end{theorem}

\begin{proof}[Proof of (a)]
The function $T_q$ is continuous on $[0,1]$, and
$T_q([0,1])=[0,\Lambda(0)]$, because
$0\le T_q\le\Lambda(0)$, the value $\Lambda(0)$ is attained, and $T_q(0)=0$. By
Lemma~\ref{nomono}, Theorem~\ref{garg} applies to $T_q$ on $[0,1]$ and
gives the first statement. The second follows, because a nonempty perfect
set $P$ of real numbers is uncountable. Indeed, $P$ is closed and
therefore a complete metric space, and each set $\{p\}$ with $p\in P$ is
closed and has empty interior in $P$, since $P$ has no isolated points.
By the Baire category theorem, $P$ is not the union of countably many
such sets, so $P$ is uncountable.
\end{proof}

\begin{proof}[Proof of (b)]
Let $\ell\ge4$ be an integer such that $(q-1)\ell$ is even. The set
$\mathcal N_{\ell-2}$ is not empty: if $q$ is odd, it contains the word
of $\ell-2$ letters $(q-1)/2$, and if $q$ is even, then $\ell$ is even,
and it contains the word in which the letters $q/2-1$ and $q/2$
alternate, beginning with $q/2-1$. Let $\mathcal G_\ell$ be the set of
the points $x\in[0,1)$ whose digit sequence consists of words $0v(q-1)$ written one
after another, where each $v$ is a word in $\mathcal N_{\ell-2}$, and
$0v(q-1)$ is the word of length $\ell$ that consists of the letter $0$,
the letters of $v$, and the letter $q-1$. Such a word satisfies
$d_1(0v(q-1))=q-1$, $d_{j+1}(0v(q-1))=q-1+d_j(v)\ge q-1$ for
$1\le j\le\ell-2$, and $d_\ell(0v(q-1))=0$. Hence, for every point
$x\in\mathcal G_\ell$, the indices $n$ with $d_n(x)=0$ are the multiples
of $\ell$, each of them is followed by the digit $0$, and $d_j(x)\ge q-1$
for all other $j\ge1$.
Moreover, $x$ is not a $q$-adic rational, since each word contains the
digits $0$ and $q-1$.

We claim that
\begin{equation}
 T_q(x')-T_q(x)\ge\bigl(q-\Lambda(0)\bigr)q^{-\ell}(x'-x)~\text{for}~
 x,x'\in\mathcal G_\ell\text{ with }x<x'.
 \label{bilip}
\end{equation}
Here $q-\Lambda(0)>0$, because $\Lambda(0)\le(q+1)/4$ by
Theorem~\ref{gmax}.
Let $x<x'$ be points of $\mathcal G_\ell$, and let $k\ge0$ be the largest
integer such that $x$ and $x'$ have the same first $k\ell$ digits. Since
$d_{k\ell}(x)=d_{k\ell}(x')=0$, \eqref{lcell} gives
$T_q(x')-T_q(x)=q^{-k\ell}(T_q(z')-T_q(z))$ and $x'-x=q^{-k\ell}(z'-z)$, where $z$
and $z'$ are the points whose digits are those of $x$ and of $x'$ after
the place $k\ell$. These points lie in $\mathcal G_\ell$, $z<z'$, and
their first $\ell$ digits differ. Since $z'-z\le1$, it suffices to show
that $T_q(z')-T_q(z)\ge(q-\Lambda(0))q^{-\ell}$. Let $w$ and $w'$ be the
integers whose words of length $\ell-1$ consist of the first $\ell-1$ digits of $z$
and of $z'$. The $\ell$th digit of both points is $q-1$, so $w\ne w'$, and
$w<w'$ because $z<z'$. Let $t^*$ be the point $1/2$ if $q$ is odd, and
the point with the digits $q/2-1,q/2,q/2-1,q/2,\ldots$ if $q$ is even.
Its digit slopes are $0$ or $1$, and it is not a $q$-adic rational. Put
$u=(w+t^*)/q^{\ell-1}$ and $u'=(w'+t^*)/q^{\ell-1}$. The digits of $u$ are
the letters of the word of $w$ followed by the digits of $t^*$. Hence
$u$ is not a $q$-adic rational, $d_j(u)=d_j(z)\ge q-1$
for $1\le j\le\ell-1$, and $d_j(u)=q-1+d_{j-\ell+1}(t^*)\ge q-1$ for
$j\ge\ell$, and the same holds for $u'$. Since $q-1\ge\lfloor(q+1)/2\rfloor$ and
$u<u'$, Lemma~\ref{monoabove} gives
$T_q(u)<T_q(u')$. The intervals of level $\ell-1$ with the left endpoints
$w/q^{\ell-1}$ and $w'/q^{\ell-1}$ have the slope
$d_{\ell-1}(z)=d_{\ell-1}(z')=q-1$, so \eqref{cell} gives
\[
 T_q(u)=T_q\bigl(w/q^{\ell-1}\bigr)+q^{1-\ell}\bigl((q-1)t^*+T_q(t^*)\bigr),
\]
and the same equation holds for $u'$ and $w'$. Therefore
$T_q(w'/q^{\ell-1})-T_q(w/q^{\ell-1})=T_q(u')-T_q(u)>0$. By \eqref{samp},
$q^{\ell-1}T_q(w/q^{\ell-1})$ and $q^{\ell-1}T_q(w'/q^{\ell-1})$ are
integers, so
$T_q(w'/q^{\ell-1})-T_q(w/q^{\ell-1})\ge q^{1-\ell}$. Let $t$ be the point
whose digits are those of $z$ after the place $\ell$, so that
$z=\bigl(w+(q-1+t)/q\bigr)/q^{\ell-1}$. By \eqref{cell} and by
\eqref{branches} with $a=q-1$,
\[
 q^{\ell-1}\Bigl(T_q(z)-T_q\bigl(w/q^{\ell-1}\bigr)\Bigr)
 =\frac{(q-1)(q-1+t)}q+\frac{(q-1)-(q-1)t+T_q(t)}q
 =q-1+\frac{T_q(t)}q,
\]
and the same holds for $z'$ and $w'$, with the point whose digits are
those of $z'$ after the place $\ell$ in place of $t$. Since
$0\le T_q\le\Lambda(0)$, we obtain
$T_q(z')-T_q(z)\ge q^{1-\ell}-\Lambda(0)q^{-\ell}$, which proves
\eqref{bilip}. If $q$ is odd, the two integers above are even by
\eqref{xi}, and the same argument gives \eqref{bilip} with the constant
$2q-\Lambda(0)$, which is $5$ for $q=3$.

Put $s=\log\#\mathcal N_{\ell-2}/(\ell\log q)$. If $\mathcal N_{\ell-2}$
has only one element, then $s=0$, and the first statement of~(b) holds
because every set has Hausdorff dimension at least $0$. Suppose that
$\mathcal N_{\ell-2}$ has at least two elements, so that $s>0$. Consider
the product measure on the sequences of words $0v(q-1)$ as above in which
every word has the weight $q^{-\ell s}=1/\#\mathcal N_{\ell-2}$ at every
position, and let $\mu$ be its image under the map that sends such a sequence to
the point with that digit sequence. The map is continuous, since
sequences that agree in their first $k$ words are sent to points at
distance at most $q^{-k\ell}$. It follows that $\mathcal G_\ell$ is
compact and that $\mu$ is a Borel probability measure with
$\mu(\mathcal G_\ell)=1$. Since the points of $\mathcal G_\ell$ are not
$q$-adic rationals, all points of $\mathcal G_\ell$ in an interval
$J$ of level $k\ell$ have the same first $k\ell$ digits. If there are
such points, these digits form $k$ words $0v(q-1)$, and
$\mu(J)=q^{-k\ell s}$. Otherwise $\mu(J)=0$. By Lemma~\ref{massdist}, the Hausdorff measure of
dimension $s$ of $\mathcal G_\ell$ is positive, and $\mathcal G_\ell$ has
Hausdorff dimension at least $s$.

Now let $x\in\mathcal G_\ell$. The strings of $x$ after the indices $n$
with $d_n(x)=0$ are its words $0v(q-1)$, so $L^{\mathrm{loc}}_x$ consists
of the points
obtained from $x$ by replacing some of these words by their complements.
Consider the product measure on the sequences
$\epsilon_1,\epsilon_2,\ldots$ of zeros and ones in which $0$ and $1$
have weight $1/2$ at every position, and let $\mu'$ be its image under
the map that sends such a sequence to the point obtained from $x$ by
replacing the $k$th word by its complement for every $k$ with
$\epsilon_k=1$. As before, the map is continuous, $L^{\mathrm{loc}}_x$ is
compact, and $\mu'$ is a Borel probability measure with
$\mu'(L^{\mathrm{loc}}_x)=1$. The digit of the image at the place
$(k-1)\ell+1$ is $0$ if $\epsilon_k=0$ and $q-1$ if $\epsilon_k=1$. By
Lemma~\ref{flip}, the points of $L^{\mathrm{loc}}_x$ are not $q$-adic rationals, so all points of $L^{\mathrm{loc}}_x$ in an interval
$J$ of level $k\ell$ are the images of sequences with the same
$\epsilon_1,\ldots,\epsilon_k$. Thus $\mu'(J)\le2^{-k}$, and
Lemma~\ref{massdist} with $\mu'$ in place of $\mu$ and
$\log2/(\ell\log q)$ in place of $s$ shows that $L^{\mathrm{loc}}_x$ has
Hausdorff dimension at least $\log2/(\ell\log q)$. Since
$L^{\mathrm{loc}}_x\subseteq L(T_q(x))$ by Lemma~\ref{flip}, the same
holds for $L(T_q(x))$. Hence the points $x\in[0,1]$ for which $L(T_q(x))$
has Hausdorff dimension at least $\log2/(\ell\log q)$ form a set that
contains $\mathcal G_\ell$, and this set has Hausdorff dimension at least $s$.

Let $\delta>0$. Suppose that sets $U_1,U_2,\ldots$ of diameter at most
$\delta$ cover $T_q(\mathcal G_\ell)$. The sets $\mathcal G_\ell\cap T_q^{-1}(U_i)$
cover $\mathcal G_\ell$. By \eqref{bilip}, each of these sets has diameter
at most $q^\ell/(q-\Lambda(0))$ times the diameter of $U_i$, and therefore
at most $q^\ell\delta/(q-\Lambda(0))$. Therefore the Hausdorff measure of
dimension $s$ of $T_q(\mathcal G_\ell)$ is at least
$\bigl((q-\Lambda(0))q^{-\ell}\bigr)^s$ times that of $\mathcal G_\ell$,
which is positive. Since $T_q(\mathcal G_\ell)$ is
contained in the set of the heights $y$ for which $L(y)$ has Hausdorff
dimension at least $\log2/(\ell\log q)$, this proves the first statement
of~(b).

The heights $y$ for which $L(y)$ has positive Hausdorff dimension form a
set that contains $T_q(\mathcal G_\ell)$ for every $\ell\ge4$ such that
$(q-1)\ell$ is even, and the points $x$ for which $L(T_q(x))$ has
positive Hausdorff dimension form a set that contains $\mathcal G_\ell$
for every such $\ell$. By the first statement of~(b), both sets have
Hausdorff dimension $1$ if $\log\#\mathcal N_n/(n\log q)$ tends to $1$ as
$n\to\infty$ through the integers $n$ for which $(q-1)n$ is even, because
$(\ell-2)/\ell$ tends to $1$. Let $n\ge1$ be such an integer, and let
$v=v_1\ldots v_n$ be one of the $C_n$ words of length $n$ with
$d_n(v)=0$, in the notation of Lemma~\ref{trinom}. Let
$i\in\{0,\ldots,n-1\}$ be the smallest index with
$d_i(v)=\min_{0\le r<n}d_r(v)$. Since $d_n(v)=0=d_0(v)$,
the word $v_{i+1}\ldots v_nv_1\ldots v_i$ has the digit slopes
$d_{i+j}(v)-d_i(v)\ge0$ for $j\le n-i$ and
$d_{j-n+i}(v)-d_i(v)\ge0$ for $n-i<j\le n$, and its last digit slope is
$0$, so it lies in $\mathcal N_n$. Each word in $\mathcal N_n$ arises in
this way from at most $n$ words $v$, namely from the words obtained from
it by moving some of its first letters to its end. Therefore
$\#\mathcal N_n\ge C_n/n$. Since $\sqrt n\,C_n/q^n$ converges
to a positive limit by Lemma~\ref{trinom}, we have
$\log(C_n/n)=n\log q-\frac32\log n+O(1)$, and
$C_n/n\le\#\mathcal N_n\le q^n$, so the quotient tends to $1$.

Finally, a set of positive Hausdorff dimension is uncountable, so the
heights $y$ for which $L(y)$ has positive Hausdorff dimension form a
subset of the heights for which $L(y)$ is uncountable, and both sets have
Hausdorff dimension $1$. Both have Lebesgue measure zero, because $L(y)$
is finite for almost every $y$ by Theorem~\ref{levelthm}(a).
\end{proof}

\begin{example}
\label{ex:heightdim}
Let $q=3$ and $\ell=6$. The set $\mathcal N_4$ consists of the nine words
$1111$, $0211$, $0121$, $0112$, $1021$, $1012$, $1102$, $0202$, and
$0022$. These are the words of length $4$ over $\{0,1,2\}$ with as many
zeros as twos whose first $j$ letters contain at least as many zeros as
twos for every $j$, and they are the words of the pairs
$(4,m)\in\mathcal A$. By Theorem~\ref{heightdim}(b), the heights $y$ for
which $L(y)$ has Hausdorff dimension at least $\log2/(6\log3)$ form a set
of Hausdorff dimension at least $\log9/(6\log3)=1/3$.
\end{example}

\section{Extra edge connectivity of the Hamming graphs}
\label{sec:extra}

In this section we prove Theorem~\ref{lamdig} and further statements on
the boundary counts of the Hamming graphs, which are the numbers
$q^nT_q(m/q^n)$ by \eqref{samp}.

\begin{definition}
\label{blocks}
Let $N\ge1$, $1\le k\le N$, and $0\le a<q^{N-k}$. The block of $a$ is the
range of sizes $aq^k\le m\le(a+1)q^k$.
\end{definition}

If we write sizes with $N$ base $q$ digits, allowing leading zeros, then
the block of $a$ consists of the sizes whose leading $N-k$ digits form the
expansion of $a$ with $N-k$ digits, together with the size $(a+1)q^k$. By
\eqref{split} below, $\xi_N((a+1)q^k)-\xi_N(aq^k)=Dq^k$, where
$D=(q-1)(N-k)-2s_q(a)$ is the slope of the interval
$[a/q^{N-k},(a+1)/q^{N-k}]$ of level $N-k$ (see
Definition~\ref{intslope}). By Corollary~\ref{blockmax}, the largest
boundary count in a block exceeds the larger of the boundary counts at
its ends by the largest of the numbers $\xi_k(u)-|D|u$ with
$0\le u\le q^k$, and the next lemma gives this number.

\begin{lemma}
\label{blocktilt}
Let $k\ge1$, let $c\ge0$ be an integer, and write $c=e(q-1)+b$ with
integers $e\ge0$ and $0\le b\le q-2$. For an integer $u$ with
$0\le u<q^k$ let $d_j=d_j(u/q^k)$, $1\le j\le k$, be the first $k$ digit
slopes of $u/q^k$. Then
\begin{equation}
 \xi_k(u)-cu=\frac{q^{k-1}\bigl((q-c)^2+q\bigr)}4
 -\frac{q-1}{4q}\sum_{j=1}^{k-1}q^{k-j}(d_j-c)^2
 -\frac{(d_k+1-c)^2+q}4 .
 \label{blockform}
\end{equation}
If $e\ge k$, then the maximum of $\xi_k(u)-cu$ over the integers
$0\le u\le q^k$ is $0$. It is attained only at $u=0$ when $c>(q-1)k$, and
exactly at $u=0$ and $u=1$ when $c=(q-1)k$. If $e<k$, then
\begin{equation}
 \max_{0\le u\le q^k}\bigl(\xi_k(u)-cu\bigr)
 =q^{k-e}\max_{0\le t\le1}\bigl(T_q(t)-bt\bigr)-
 \begin{cases}
 \frac{q+1}4&\text{if $q$ is odd and $b$ is even},\\[1mm]
 \frac{q-1}4&\text{if $q$ and $b$ are odd},\\[1mm]
 \frac{q(q+2)}{4(q+1)}&\text{if $q$ and $k-e-b$ are even},\\[1mm]
 \frac{q^2}{4(q+1)}&\text{if $q$ is even and $k-e-b$ is odd},
 \end{cases}
 \label{blocktilt1}
\end{equation}
where the maximum over $t$ is given by \eqref{tilt0}, and the maximum
over $u$ is attained exactly at the integers $u$ with $0\le u<q^{k-e}$
whose digit slopes satisfy $|d_j-c|\le1$ for every $j$ with $e<j<k$ and
$|d_k+1-c|\le1$. If $e<k$, $q$ is odd, and $b$ is even, then the maximum
\eqref{blocktilt1} equals
$\frac14\bigl(q^{k-e-1}((q-b)^2+q)-q-1\bigr)$, and it is attained exactly
at the two integers $u=\tfrac12\bigl((q-b)q^{k-e-1}\pm1\bigr)$.
\end{lemma}

For $c=0$ the lemma gives the largest boundary count of $Q_k^q$, which is
\eqref{peak} for odd $q$ and Corollary~\ref{peakeven} for even $q$.

\begin{proof}
Let $0\le u<q^k$, and let $\eta_1,\ldots,\eta_k$ be the base $q$ digits
of $u$, written with $k$ digits and starting with the digit of the
highest power. The terminating digit sequence of $u/q^k$ consists of
$\eta_1,\ldots,\eta_k$ followed by zeros. Hence its digit slopes are
$d_1,\ldots,d_k$ and $d_{k+i}=d_k+(q-1)i$ for $i\ge0$. We use the sums
$\sum_{i=0}^{\infty}q^{-i}=q/(q-1)$,
$\sum_{i=0}^{\infty}iq^{-i}=q/(q-1)^2$, and
$\sum_{i=0}^{\infty}i^2q^{-i}=q(q+1)/(q-1)^3$. With $z=d_k-c$ they give
\[
 \sum_{j=k}^{\infty}\frac{(d_j-c)^2}{q^j}
 =q^{-k}\sum_{i=0}^{\infty}\frac{(z+(q-1)i)^2}{q^i}
 =\frac{q^{1-k}}{q-1}\bigl(z^2+2z+q+1\bigr)
 =\frac{q^{1-k}}{q-1}\bigl((d_k+1-c)^2+q\bigr).
\]
We insert this into \eqref{tiltsq} with $t=u/q^k$, multiply by $q^k$, and
use $\xi_k(u)=q^kT_q(u/q^k)$, which is \eqref{samp}. This gives the formula \eqref{blockform}.

Suppose first that $e\ge k$, that is, $c\ge(q-1)k$. By \eqref{xi},
$\xi_k(u)-cu=((q-1)k-c)u-2S_q(u)\le0$ for $0\le u\le q^k$. Equality
holds if and only if $u=0$, or $c=(q-1)k$ and $S_q(u)=0$, and
$S_q(u)=0$ exactly when $u\le1$.

Suppose now that $e<k$, so that $c<(q-1)k$. The value of $\xi_k(u)-cu$
at $u=q^k$ is $-cq^k\le0$, and its value at $u=1$ is $(q-1)k-c>0$.
Hence the maximum is attained only at integers $u<q^k$, and we may use
\eqref{blockform}. For $1\le j\le e$ we have
$d_j\le(q-1)j\le c$, so $(d_j-c)^2\ge((q-1)j-c)^2$, with equality if and
only if $d_j=(q-1)j$, that is, $\eta_1=\cdots=\eta_j=0$. For $e<j<k$ the
integer $d_j-c$ has the parity of $(q-1)j-c$, so $(d_j-c)^2$ is at least
$0$ or at least $1$, depending on whether $(q-1)j-c$ is even or odd, and
this bound is attained if and only if $|d_j-c|\le1$. In the same way,
$(d_k+1-c)^2$ is at least $1$ or at least $0$, depending on whether
$(q-1)k-c$ is even or odd, and this bound is attained if and only if
$|d_k+1-c|\le1$. Let $U$ be the set of the integers $u$ with
$0\le u<q^k$ for which all these bounds are attained. Since
$\eta_1=\cdots=\eta_e=0$ means that $u<q^{k-e}$, the set $U$ consists of
the integers described in the lemma. By \eqref{blockform}, the number
$\xi_k(u)-cu$ is the same for all $u\in U$, and it is smaller for every
other $u<q^k$.

We claim that $U$ is not empty. Indeed, by Lemma~\ref{lem:tilt} there is
a digit sequence whose first $e$ digits are $0$ and whose digit slopes
$d_j$ satisfy $|d_j-c|\le1$ for every $j>e$, namely the digit sequence of
a point at which $T_q(t)-ct$ attains its maximum over $[0,1]$. Let $u$ be
the integer whose $k$ digits $\eta_1,\ldots,\eta_k$ are the first $k$
digits of this sequence. Then $d_1,\ldots,d_k$ are also the first $k$
digit slopes of the sequence. If $d_k\le c$, then $u$ belongs to $U$. Let
$d_k=c+1$. We have
$d_{k-1}\le c+1$. This holds by the choice of the sequence when $k-1>e$,
and because $d_{k-1}=(q-1)e\le c$ when $k-1=e$. Therefore
$q-1-2\eta_k=d_k-d_{k-1}\ge0$, so that $\eta_k\le(q-1)/2<q-1$, and
$\eta_k+1$ is a digit. The integer $u+1$, whose digits are
$\eta_1,\ldots,\eta_{k-1},\eta_k+1$, has the digit slopes
$d_1,\ldots,d_{k-1}$ and $d_k-2=c-1$, and it belongs to $U$. Hence the
maximum of $\xi_k(u)-cu$ over $0\le u\le q^k$ is attained exactly at the
integers of $U$.

It remains to compute the maximum. For $0\le u\le q^{k-e}$, equation
\eqref{xi} gives $\xi_k(u)-cu=\xi_{k-e}(u)-bu$, because
$(q-1)k-c=(q-1)(k-e)-b$. Since $U$ lies in this range, the maximum of
$\xi_k(u)-cu$ over $0\le u\le q^k$ equals the maximum of
$\xi_{k-e}(u)-bu$ over $0\le u\le q^{k-e}$. It therefore suffices to
prove \eqref{blocktilt1} for $e=0$, with $k-e$ in place of $k$. Let
$e=0$, so that $c=b$, and let $u\in U$. In \eqref{blockform} we then have
$(d_j-b)^2=1$ for the indices $j<k$ for which $(q-1)j-b$ is odd and
$(d_j-b)^2=0$ for the other indices $j<k$, and $(d_k+1-b)^2$ is $0$ if
$(q-1)k-b$ is odd and $1$ if it is even. By the proof of
Lemma~\ref{lem:tilt}, the maximum of $T_q(t)-bt$ over $[0,1]$ is the right
side of \eqref{tiltpar}, in which the sum is over all integers $j\ge1$
for which $(q-1)j-b$ is odd. Multiplying \eqref{tiltpar} by $q^k$ and
comparing with \eqref{blockform}, we obtain
\[
 \max_{0\le u\le q^k}\bigl(\xi_k(u)-bu\bigr)
 =q^k\max_{0\le t\le1}\bigl(T_q(t)-bt\bigr)
 +\frac{q-1}{4q}\sum_{j}q^{k-j}-\frac{(d_k+1-b)^2+q}4,
\]
where the sum is over the integers $j\ge k$ for which $(q-1)j-b$ is odd.
If $q$ is odd, then $(q-1)j-b$ has the parity of $b$. For even $b$ the
sum is empty and $(d_k+1-b)^2=1$, so that the last two terms add up to
$-\frac{q+1}4$. For odd $b$ the sum is
$\sum_{j=k}^{\infty}q^{k-j}=q/(q-1)$ and $(d_k+1-b)^2=0$, so that the
last two terms add up to $\frac14-\frac q4=-\frac{q-1}4$. If $q$ is
even, then $(q-1)j-b$ has the parity of $j-b$. When $k-b$ is even, the
sum is $\sum_{i=0}^{\infty}q^{-2i-1}=q/(q^2-1)$ and $(d_k+1-b)^2=1$, and
the last two terms add up to
$\frac1{4(q+1)}-\frac{q+1}4=-\frac{q(q+2)}{4(q+1)}$. When $k-b$ is odd,
the sum is $\sum_{i=0}^{\infty}q^{-2i}=q^2/(q^2-1)$ and
$(d_k+1-b)^2=0$, and the last two terms add up to
$\frac q{4(q+1)}-\frac q4=-\frac{q^2}{4(q+1)}$. This proves
\eqref{blocktilt1}.

Finally, let $e<k$, let $q$ be odd, and let $b$ be even. Then $c$ and all
digit slopes are even, so that $u\in U$ if and only if
$\eta_1=\cdots=\eta_e=0$, $d_j=c$ for $e<j<k$, and $d_k$ is $c$ or $c-2$.
Since $d_e=(q-1)e=c-b$, this holds exactly when the digits after
$\eta_e$ are $(q-1-b)/2$, followed by $k-e-2$ digits $(q-1)/2$, followed
by $(q-1)/2$ or $(q+1)/2$ if $k-e\ge2$, and exactly when $\eta_k$ is
$(q-1-b)/2$ or $(q+1-b)/2$ if $k-e=1$. In both cases these are the two
integers $\tfrac12\bigl((q-b)q^{k-e-1}\pm1\bigr)$. By \eqref{tilt0}, the
maximum of $T_q(t)-bt$ is $((q-b)^2+q)/(4q)$, and \eqref{blocktilt1} gives
the value of the maximum.
\end{proof}

The top row of Figure~\ref{Q4pic} illustrates the next corollary for
$q=3$, $4$, and $5$.

\begin{corollary}
\label{blockmax}
Let $N\ge1$, $1\le k\le N$, and $0\le a<q^{N-k}$. Let
$D=(q-1)(N-k)-2s_q(a)$, and write $|D|=e(q-1)+b$ with integers $e\ge0$
and $0\le b\le q-2$. Then
\begin{equation}
 \max_{aq^k\le m\le(a+1)q^k}\xi_N(m)
 =\max\bigl\{\xi_N(aq^k),\xi_N((a+1)q^k)\bigr\}
 +\max_{0\le u\le q^k}\bigl(\xi_k(u)-|D|u\bigr),
 \label{blockmax1}
\end{equation}
where the last maximum is given by Lemma~\ref{blocktilt} with $c=|D|$.
The maximum over the block is attained exactly at the sizes
$m=(a+1)q^k-u$ if $D\ge0$, and exactly at the sizes $m=aq^k+u$ if
$D\le0$, where $u$ runs through the integers at which $\xi_k(u)-|D|u$
attains its maximum over $0\le u\le q^k$. In particular, if
$|D|\ge(q-1)k$, the last term of \eqref{blockmax1} is $0$. If
$|D|>(q-1)k$, the maximum is attained only at $m=(a+1)q^k$ when $D>0$
and only at $m=aq^k$ when $D<0$. If $|D|=(q-1)k$, it is attained exactly
at the two sizes $(a+1)q^k-1$ and $(a+1)q^k$ when $D>0$, and exactly at
the two sizes $aq^k$ and $aq^k+1$ when $D<0$.

If $q$ is odd, then $b$ is even. In this case, if $|D|<(q-1)k$, the
last term of \eqref{blockmax1} is
$\frac14\bigl(q^{k-e-1}((q-b)^2+q)-q-1\bigr)$, and the maximum is
attained exactly at the two sizes
$m=(a+1)q^k-\tfrac12\bigl((q-b)q^{k-e-1}\pm1\bigr)$ if $D\ge0$, and
exactly at the two sizes
$m=aq^k+\tfrac12\bigl((q-b)q^{k-e-1}\pm1\bigr)$ if $D\le0$; for $D=0$
the two descriptions give the same sizes. For $q=3$, let
$d=(N-k)-s_3(a)$. Then
$D=2d$, $e=|d|$, and $b=0$, so that the last term of \eqref{blockmax1}
is $\max\{3^{k-|d|}-1,0\}$. For $0\le d\le k$ the maximum is attained
exactly at the two sizes $m=(a+1)3^k-\tfrac12(3^{k-d}\pm1)$, and for
$-k\le d\le0$ exactly at the two sizes $m=a3^k+\tfrac12(3^{k-|d|}\pm1)$.
\end{corollary}

\begin{proof}
Write $m=aq^k+i$ with $0\le i\le q^k$. By \eqref{samp} and
\eqref{cell} at level $N-k$, with $a$ in place of $m$ and $t=i/q^k$, we
have
\begin{equation}
 \xi_N(aq^k+i)=\xi_N(aq^k)+Di+\xi_k(i),~0\le i\le q^k,
 \label{split}
\end{equation}
since $q^N\cdot q^{-(N-k)}T_q(i/q^k)=q^kT_q(i/q^k)=\xi_k(i)$. Taking
$i=q^k$ gives $\xi_N((a+1)q^k)-\xi_N(aq^k)=Dq^k$, since $\xi_k(q^k)=0$.

If $D\le0$, then $\xi_N(aq^k)\ge\xi_N((a+1)q^k)$, and \eqref{split}
reads $\xi_N(aq^k+i)=\xi_N(aq^k)+\xi_k(i)-|D|i$. If $D\ge0$, then
$\xi_N((a+1)q^k)\ge\xi_N(aq^k)$, and with $u=q^k-i$ and
$\xi_k(q^k-u)=\xi_k(u)$ we obtain
\[
 \xi_N\bigl((a+1)q^k-u\bigr)=\xi_N(aq^k)+Dq^k-Du+\xi_k(u)
 =\xi_N\bigl((a+1)q^k\bigr)+\xi_k(u)-|D|u,~0\le u\le q^k.
\]
In both cases we obtain \eqref{blockmax1} and the sizes at which the
maximum is attained. For $D=0$ the two descriptions give the same sizes,
because $\xi_k(q^k-u)=\xi_k(u)$. The statements for $|D|\ge(q-1)k$ are
the case $e\ge k$ of Lemma~\ref{blocktilt}.

Let $q$ be odd. Then $(q-1)(N-k)$ and $e(q-1)$ are even, so $D$ and $b$
are even. If $|D|<(q-1)k$, then $e<k$, and the last statement of
Lemma~\ref{blocktilt} gives the last term of \eqref{blockmax1} and the
two integers $u$. For $q=3$ we have $|D|=2|d|=|d|(q-1)$, so that $e=|d|$
and $b=0$. For $|d|<k$ the last term of \eqref{blockmax1} is therefore
$\frac14(12\cdot3^{k-|d|-1}-4)=3^{k-|d|}-1$, and the two integers $u$ are
$\tfrac12(3^{k-|d|}\pm1)$. For $|d|=k$ the number $3^{k-|d|}-1$ is $0$,
and the two numbers $\tfrac12(3^{k-|d|}\pm1)$ are $1$ and $0$, so that
the statements for $q=3$ agree with the case $|D|=(q-1)k$. For $|d|>k$ we
have $|D|>(q-1)k$, and the last term is $0$.
\end{proof}

Take, for example, $q=3$, $N=3$, $k=2$, and $a=0$, so that the block is
$0\le m\le9$ and $d=1$. Corollary~\ref{blockmax} gives the maximum
$\xi_3(9)+2=20$ over this block, attained exactly at $m=7$ and $m=8$, in
agreement with Table~\ref{tab1}. Since $0<d<k$, this maximum exceeds both
$\xi_3(0)$ and $\xi_3(9)$. In an even base the maximum over a block can be
attained at more than two sizes. Let $q=2$, $N=4$, $k=3$, and $a=0$, so that the
block is $0\le m\le8$ in the hypercube $Q_4^2$, and $D=1$, $e=1$, and
$b=0$. By \eqref{xi}, the numbers $\xi_3(u)-u$ with $0\le u\le8$ are
$0$, $2$, $2$, $2$, $0$, $0$, $-2$, $-4$, and $-8$. Their maximum is
$2$, in agreement with \eqref{blocktilt1}, which gives
$4\cdot\frac23-\frac23$ because the maximum of $T_q$ is $2/3$ for $q=2$.
It is attained at $u=1,2,3$, whose digits $001$, $010$, and $011$ have
the digit slopes $1,2,1$, then $1,0,1$, and then $1,0,-1$. By
Corollary~\ref{blockmax}, the maximum of $\xi_4(m)$ over $0\le m\le8$ is
$\xi_4(8)+2=10$, attained exactly at $m=5,6,7$.

\begin{corollary}
\label{xi2m}
Let $n\ge1$. For $0\le m\le q^n/2$,
\begin{equation}
 \xi_n(m)\ge\Bigl\lfloor\frac{q+1}2\Bigr\rfloor m.
 \label{xi2m1}
\end{equation}
If $q$ is odd, equality holds exactly for $m=0$ and for
$m=(q^n-q^{n-r})/2$ with $1\le r\le n$. If $q$ is even, equality holds
exactly for $m=0$ and for $m=q^n/2$. Consequently every set $A$ of
vertices of $Q_n^q$ with $|A|\le q^n/2$ has at least
$\lfloor(q+1)/2\rfloor\,|A|$ boundary edges, that is,
$|\partial A|\ge\lfloor(q+1)/2\rfloor\,|A|$, and equality requires $|A|$
to be one of these sizes. For $q=3$ these inequalities read
$\xi_n(m)\ge2m$ and $|\partial A|\ge2|A|$.
\end{corollary}

\begin{proof}
By \eqref{samp},
$\xi_n(m)-\lfloor(q+1)/2\rfloor m=q^n\bigl(T_q(x)-\lfloor(q+1)/2\rfloor x\bigr)$
with $x=m/q^n\le1/2$, and Proposition~\ref{T2x} gives the inequality and
the points $x$ at which equality holds. If $q$ is even, these points are
$0$ and $1/2$, that is, $m=0$ and $m=q^n/2$. Let $q$ be odd. Since $q^n$
is odd, the point $1/2$ is not of the form $m/q^n$. Moreover,
$(1-q^{-r})/2=m/q^n$ holds for an integer $m$ exactly when $r\le n$ and
$m=(q^n-q^{n-r})/2$. The last statement follows from
$|\partial A|\ge\xi_n(|A|)$.
\end{proof}

By \eqref{xi}, inequality \eqref{xi2m1} is equivalent to
$2S_q(m)\le\bigl((q-1)n-\lfloor(q+1)/2\rfloor\bigr)m$ for
$0\le m\le q^n/2$. Allaart~\cite{Allaart2014} studied inequalities of
this kind for the summatory functions of digit sums. For odd $q$, the
equality sizes $(q^n-q^{n-r})/2$ of Corollary~\ref{xi2m} are written in
base $q$ as $r$ digits
$(q-1)/2$ followed by $n-r$ zeros. For $q=3$ and $n=3$ these are the
sizes $9$, $12$, and $13$, with ternary digits $100$, $110$, and $111$,
and together with $m=0$ they are exactly the sizes in Table~\ref{tab1}
with $\xi_3(m)=2m$.

The inequality $|\partial A|\ge\lfloor(q+1)/2\rfloor\,|A|$ of
Corollary~\ref{xi2m} is already known: very recently, Abiad
et~al.~\cite[Proposition~5.1]{Abiad2026} determined the isoperimetric
number of every Cartesian product of complete graphs, that is, the
smallest ratio $|\partial A|/|A|$ over the nonempty sets $A$ that
contain at most half of the vertices. For $Q_n^q$ this number is
$\lceil q/2\rceil=\lfloor(q+1)/2\rfloor$, and for even $q$ they quote this
value from earlier work~\cite[Section~5.2]{Abiad2026}. By \eqref{samp},
their result gives $T_q(m/q^n)\ge\lfloor(q+1)/2\rfloor\,m/q^n$ for
$0\le m\le q^n/2$. These points are dense in $[0,1/2]$, so by the
continuity of $T_q$ their result also gives the inequality of
Proposition~\ref{T2x}.

Let $q=3$ in this paragraph. The equality sizes then also follow from a
decomposition of $\xi_n$ due to Xu, Zhou, and Hsieh. For $p\ge1$ and $0\le t\le3^{p-1}$,
\eqref{xi} gives $\xi_p(t)-2t=2\bigl((p-1)t-S_3(t)\bigr)$. This number
is zero for $t=0$, and also for $t=3^{p-1}$ because
$S_3(3^{p-1})=(p-1)3^{p-1}$. If $p\ge2$, it equals $\xi_{p-1}(t)$, which
is positive for $0<t<3^{p-1}$ because $Q_{p-1}^3$ is connected. With
$p=n$, the equality sizes $m\le3^{n-1}$ are therefore $0$ and $3^{n-1}$.
Now put $w_j=(3^n-3^{n-1-j})/2$ for $0\le j\le n-1$, so that
$w_0=3^{n-1}$, $w_{n-1}=(3^n-1)/2$, and $w_{j+1}-w_j=3^{n-j-2}$. For
$0\le j\le n-2$ and $0\le t\le3^{n-j-2}$, Xu, Zhou, and
Hsieh~\cite[Theorem~3.8]{Xu2022} show that
$\xi_n(w_j+t)=\xi_n(w_j)+\xi_{n-j-1}(t)$. This equation also follows
from \eqref{split} with $N=n$, $k=n-1-j$, and $a=(3^{j+1}-1)/2$, since
then $a3^k=w_j$ and $s_3(a)=j+1=N-k$, so that $D=0$. Hence, with
$p=n-j-1$,
\[
 \xi_n(w_j+t)-2(w_j+t)=\xi_n(w_j)-2w_j+\xi_p(t)-2t,
 ~0\le t\le3^{p-1}=w_{j+1}-w_j.
\]
By the first part of this paragraph, $\xi_p(t)-2t\ge0$, with equality
exactly at $t=0$ and $t=3^{p-1}$. Since $\xi_n(w_0)=2w_0$, induction on
$j$ shows that $\xi_n(w_j)=2w_j$ for $0\le j\le n-1$ and that on each
range $w_j\le m\le w_{j+1}$ equality in \eqref{xi2m1} holds only at its
two ends. Thus for $3^{n-1}\le m\le(3^n-1)/2$ equality holds exactly at
the sizes $w_j$, which are the sizes $(3^n-3^{n-r})/2$ of
Corollary~\ref{xi2m} with $r=j+1$.

Extra edge connectivity is used to assess the reliability of
interconnection networks~\cite{Xu2022}.

\begin{definition}
\label{extra}
Let $n\ge1$ and $1\le h\le q^n/2$. The $h$-extra edge connectivity
$\lambda_h(Q_n^q)$ is the smallest number of edges whose removal
disconnects $Q_n^q$ into components that all have at least $h$ vertices.
\end{definition}

\begin{lemma}
\label{extramin}
For $n\ge1$ and $1\le h\le q^n/2$,
\[
 \lambda_h(Q_n^q)=\min\{\xi_n(m):h\le m\le q^n-h\}.
\]
In particular $\lambda_h(Q_n^q)\ge\lfloor(q+1)/2\rfloor\,h$. If $q$ is odd,
equality holds exactly for $h=(q^n-q^{n-r})/2$ with $1\le r\le n$. If $q$
is even, equality holds exactly for $h=q^n/2$.
\end{lemma}

\begin{proof}
If the removal of a set of edges disconnects $Q_n^q$ into components that
all have at least $h$ vertices, then this set contains the boundary of
each of these components, and since there are at least two components,
each component $A$ satisfies $h\le|A|\le q^n-h$. By Definition~\ref{cube},
the boundary of $A$ has at least $\xi_n(|A|)$ edges. Therefore
$\lambda_h(Q_n^q)\ge\min\{\xi_n(m):h\le m\le q^n-h\}$. We claim that
equality holds. Indeed, by Lemma~\ref{conn}, both $I_n(m)$ and its
complement are connected for $1\le m<q^n$, so for $h\le m\le q^n-h$ the
boundary of $I_n(m)$, which has $\xi_n(m)$ edges by Lemma~\ref{lindsey}
and \eqref{edges}, is such an edge set. By the symmetry
$\xi_n(m)=\xi_n(q^n-m)$, the minimum is attained at some $m$ with
$h\le m\le q^n/2$. For such $m$, Corollary~\ref{xi2m} gives
$\xi_n(m)\ge\lfloor(q+1)/2\rfloor\,m\ge\lfloor(q+1)/2\rfloor\,h$, and
both inequalities are equalities exactly when $m=h$ and $h$ is one of the
equality sizes of that corollary.
\end{proof}

For $q=3$, Xu, Zhou, and Hsieh~\cite[Section~III]{Xu2022} gave formulas
and an algorithm for this minimum. For $n=3$ and $h\in\{7,8\}$,
Table~\ref{tab1} and the symmetry $\xi_3(m)=\xi_3(27-m)$ give
$\lambda_h(Q_3^3)=\xi_3(9)=18$, in agreement with their Theorem~3.4
and with Ma, Zhang, Meng, and Ma~\cite[page~77]{Ma2021}. Tables~I and~II
of Xu, Zhou, and Hsieh~\cite{Xu2022} give
$\lambda_7(Q_3^3)=\lambda_8(Q_3^3)=20=\xi_3(7)=\xi_3(8)$ instead, but the
subcube $I_3(9)$ has only $18$ boundary edges, since each of its nine
vertices has exactly two neighbors outside it.

\begin{figure}[H]
\centering
\includegraphics{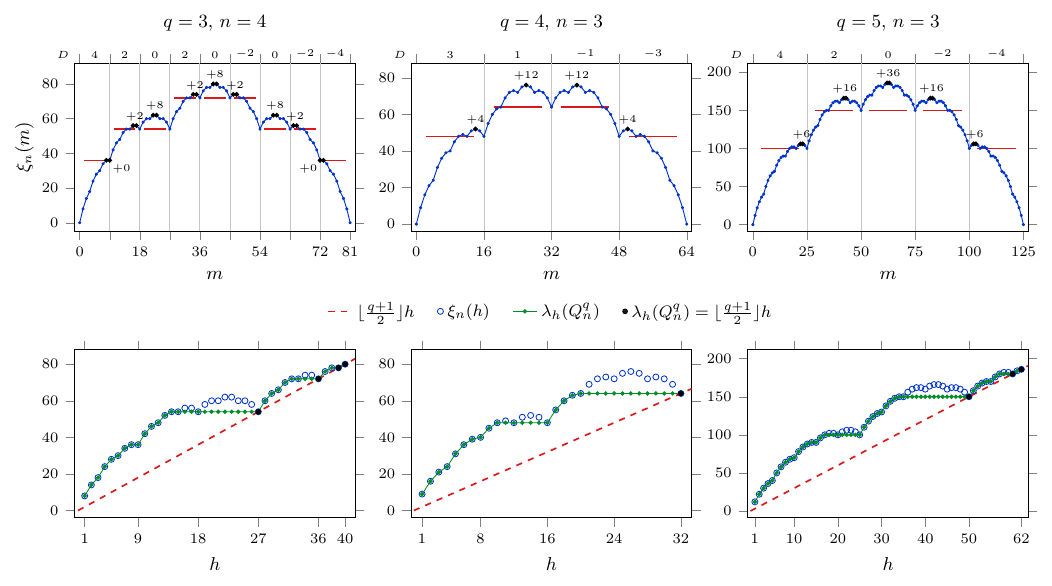}
\caption{Corollary~\ref{blockmax}, Lemma~\ref{extramin}, and
Theorem~\ref{lamdig} for $q=3$, $4$, and $5$, with $n=4$, $3$, and $3$. Top: the boundary counts $\xi_n(m)$
and their maxima over the blocks with $k=2$. Bottom: the $h$-extra edge
connectivity $\lambda_h(Q_n^q)$, together with $\xi_n(h)$ and
$\lfloor(q+1)/2\rfloor h$, for $1\le h\le q^n/2$.}
\label{Q4pic}
\end{figure}

We illustrate Corollary~\ref{blockmax}, Lemma~\ref{extramin}, and
Theorem~\ref{lamdig} for $Q_4^3$, $Q_3^4$, and $Q_3^5$ in
Figure~\ref{Q4pic}. In the top row $k=2$, and $N$ is $4$, $3$, and $3$ in the three panels.
The vertical lines separate the blocks $aq^2\le m\le(a+1)q^2$. The
numbers $D=(q-1)(N-2)-2s_q(a)$ are written along the top. For $q=3$
they are $4,2,0,2,0,-2,0,-2,-4$, and \eqref{split} with $i=9$ gives $\xi_4(9a+9)-\xi_4(9a)=9D$, so that the
endpoint values $\xi_4(9a)$ are $0$, $36$, $54$, $54$, $72$, $72$, $54$,
$54$, $36$, $0$. By Corollary~\ref{blockmax}, the maximum over a block
exceeds the larger endpoint value, drawn as a horizontal segment, by the
last term of \eqref{blockmax1}. Black dots mark the sizes where the
maximum is attained, and the excess is written next to them. For $q=3$
the excess is $8$ if $D=0$, $2$ if $|D|=2$, and $0$ if $|D|=4$. For
instance, the block of $a=1$ has its maximum $\xi_4(18)+2=56$ exactly at
$m=16$ and $m=17$, and the maximum $72+8=80$ of the block of $a=4$
is attained at $m=40$ and $m=41$ and equals the largest boundary count
$3^4-1$ of \eqref{peak}. For $q=5$ the excess is $36$, $16$, and $6$ for
$|D|=0$, $2$, and $4$, and the maximum $150+36=186$ of the middle block
is attained at $m=62$ and $m=63$ and equals the largest boundary count $(5+1)(5^3-1)/4$ of
\eqref{peak}. For $q=4$ the excess is $12$ for $|D|=1$ and $4$ for
$|D|=3$, and each block maximum is attained at only one size, namely at
$m=14$, $26$, $38$, and $50$.

In the bottom row we compare $\lambda_h(Q_n^q)$ with $\xi_n(h)$ and with
$\lfloor(q+1)/2\rfloor h$ for $1\le h\le q^n/2$. For $q=3$ we have
$\lambda_h(Q_4^3)=2h$ exactly for $h=27,36,39,40$, the sizes
$(81-3^{4-r})/2$ with $1\le r\le4$, and $\lambda_h(Q_4^3)<\xi_4(h)$
exactly for $h=16,17$, for $19\le h\le26$, and for $h=34,35$, where an
open circle lies above a green dot. For example,
$\lambda_{16}(Q_4^3)=\xi_4(18)=54$, while $\xi_4(16)=56$. For $q=4$ we
have $\lambda_h(Q_3^4)=2h$ only for $h=32$, and
$\lambda_h(Q_3^4)<\xi_3(h)$ exactly for $h=11$, for $13\le h\le15$, and
for $21\le h\le31$. For $q=5$ we have $\lambda_h(Q_3^5)=3h$ exactly for
$h=50,60,62$, the sizes $(125-5^{3-r})/2$ with $1\le r\le3$, and
$\lambda_h(Q_3^5)<\xi_3(h)$ exactly for $h=18,19$, for $21\le h\le24$,
for $36\le h\le49$, and for $h=58,59$. For $q=3$, $4$, and
$5$ the smallest of these sizes are $16$, $11$, and $18$, which are the
values given by the last statement of Theorem~\ref{lamdig}.

By \eqref{samp} and \eqref{cell}, $\lambda_h(Q_n^q)=\xi_n(h)$ holds
exactly when $T_q(h/q^n)$ is the minimum of $T_q$ on $[h/q^n,1/2]$, and
Theorem~\ref{rightmin} treats this minimum for every left endpoint (see
Figure~\ref{runpic} for $q=3$, $4$, and $5$).

\begin{figure}[H]
\centering
\includegraphics{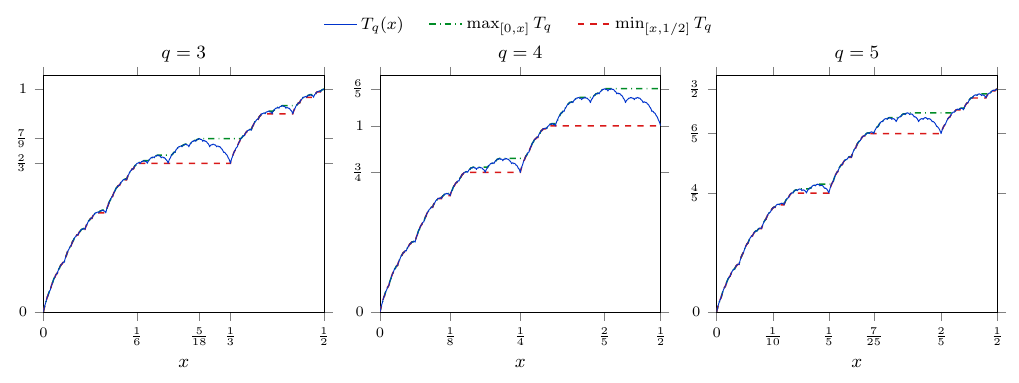}
\caption{Theorems~\ref{rightmin} and~\ref{runmax} for $q=3$, $4$, and
$5$: the graph of $T_q$ on $[0,1/2]$ together with $\max_{[0,x]}T_q$ and
$\min_{[x,1/2]}T_q$ as functions of $x$. The first is equal to $T_q$
exactly at the points of Theorem~\ref{runmax}, and the second
exactly at the points of Theorem~\ref{rightmin}.}
\label{runpic}
\end{figure}

\begin{theorem}
\label{rightmin}
Let $0\le x\le1/2$, let $\xi_1,\xi_2,\ldots$ be the digit sequence of
$x$, where we take the terminating one when $x$ is a $q$-adic rational,
and for $j\ge0$ let $t_j=q^jx-\lfloor q^jx\rfloor$. Then
\begin{equation}
 T_q(x)=\min_{x\le y\le1/2}T_q(y)
 \label{rightmin1}
\end{equation}
if and only if $d_j(x)\ge qt_j$ for every $j\ge1$ such that
$\xi_i\ne(q-1)/2$ for some $i\le j$ and $\xi_i\ne0$ for some $i>j$. For
$q=3$ this condition holds if and only if $d_j(x)\ge2$ for every $j\ge1$
such that $\xi_i=0$ for some $i\le j$ and $\xi_i\ne0$ for some $i>j$.
\end{theorem}

Here $t_j=\sum_{i=1}^{\infty}\xi_{j+i}q^{-i}$, so that $0\le t_j<1$ and
$qt_j=\xi_{j+1}+t_{j+1}$, and $t_j>0$ exactly when $\xi_i\ne0$ for some
$i>j$. Thus the inequality $d_j(x)\ge qt_j$ says that
$\xi_{j+1}<d_j(x)$, or that $\xi_{j+1}=d_j(x)$ and $t_{j+1}=0$. If $q$
is even, then $(q-1)/2$ is not an integer,
so that the theorem requires the inequality for every $j\ge1$ with
$t_j>0$.

\begin{proof}
For $x=0$ and for $x=1/2$ equation \eqref{rightmin1} holds, and no index
$j\ge1$ has the two properties named in the theorem, so that both
conditions hold as well. Let $0<x<1/2$. For
$j\ge0$ put $k_j=\lfloor q^jx\rfloor$, $a_j=k_j/q^j$, and
$c_j=a_j+q^{-j}$, so that $t_j=q^jx-k_j$, $[a_j,c_j]$ is the interval
of level $j$ with $a_j\le x<c_j$, and its slope is $d_j(x)$, by
\eqref{dnsq}. We write $d_j$ for $d_j(x)$. By \eqref{cell},
\begin{equation}
 T_q(a_j+q^{-j}u)=T_q(a_j)+q^{-j}\bigl(d_ju+T_q(u)\bigr),~0\le u\le1.
 \label{cellagain}
\end{equation}
Let $p\ge1$ be the smallest index with $\xi_p\ne(q-1)/2$. If $q$ is even,
then $p=1$, and if $q$ is odd, then $p$ exists because
$\sum_{j=1}^{\infty}\frac{q-1}2q^{-j}=\frac12\ne x$. The condition of the
theorem says that $d_j\ge qt_j$ for every $j\ge p$ with $t_j>0$.

We first reduce \eqref{rightmin1} to the interval $[x,c_p]$. The digits
$\xi_1,\ldots,\xi_{p-1}$ are equal to $(q-1)/2$, so
$a_{p-1}=(1-q^{1-p})/2$ and $d_{p-1}=0$, and \eqref{cellagain} with
$j=p-1$ reads $T_q(a_{p-1}+q^{1-p}u)=T_q(a_{p-1})+q^{1-p}T_q(u)$. In this
equation the point $1/2$ corresponds to $u=1/2$, the point $x$ to
$u=t_{p-1}<1/2$, and the point $c_p$ to $u=(\xi_p+1)/q$. Since
$\xi_p\le qt_{p-1}<q/2$ and $\xi_p\ne(q-1)/2$, we have
$\xi_p+1\le q/2$, so that $c_p\le1/2$. We claim that
$T_q(u)\ge T_q((\xi_p+1)/q)$ for $(\xi_p+1)/q\le u\le1/2$. If
$(\xi_p+1)/q=1/2$, there is nothing to prove. Otherwise such a
point is $u=(m+y)/q$ with an integer $m$ such that $\xi_p+1\le m<q/2$
and with $0\le y\le1$. Then $q-1-2m\ge0$, and \eqref{branches} gives
\[
 T_q(u)=\frac{m(q-m)+(q-1-2m)\,y+T_q(y)}q\ge\frac{m(q-m)}q
 \ge\frac{(\xi_p+1)(q-\xi_p-1)}q=T_q\Bigl(\frac{\xi_p+1}q\Bigr),
\]
because the numbers $m(q-m)$ increase with $m$ for $0\le m\le q/2$. Hence
$T_q\ge T_q(c_p)$ on $[c_p,1/2]$, and therefore
\eqref{rightmin1} holds if and only if $T_q(x)\le T_q(y)$ for all
$y\in[x,c_p]$.

Suppose that $d_j<qt_j$ for some $j\ge p$ with $t_j>0$. Then
$x<c_j\le c_p$. Put $w=1-t_j$, so that $0<w<1$ and $T_q(w)=T_q(t_j)$. By
\eqref{cellagain} with $u=1$ and with $u=t_j$,
\[
 T_q(x)-T_q(c_j)=q^{-j}\bigl(T_q(t_j)-d_j(1-t_j)\bigr)
 =q^{-j}\bigl(T_q(w)-d_jw\bigr).
\]
If $d_j\le0$, this number is positive by \eqref{Tpositive}. If
$d_j\ge1$, then $d_j<qt_j<q$, so that $1\le d_j\le q-1$ and
$0<w<(q-d_j)/q$, and Lemma~\ref{firstcross} with $k=1$ and $b=d_j$ gives
$T_q(w)>d_jw$. In both cases $T_q(c_j)<T_q(x)$. Thus \eqref{rightmin1} fails,
which proves necessity.

For sufficiency, assume that $d_i\ge qt_i$ for every $i\ge p$ with
$t_i>0$. Since $\xi_p+1\le q/2$, we have
$d_p=q-1-2\xi_p\ge1$. If $t_p=0$, then $x=a_p$, and \eqref{cellagain}
with $j=p$ gives $T_q(y)\ge T_q(x)$ for $x\le y\le c_p$. Let $t_p>0$, and put
$b_i=q^{-i}\lceil q^ix\rceil$ for $i\ge p$. Then $b_p=c_p$, the numbers
$b_i$ are nonincreasing and tend to $x$, and $b_i>x$ exactly when
$t_i>0$, in which case $b_i=c_i$.

Let $i\ge p$ with $b_i>x$, and let $g$ be the integer with
$b_{i+1}=a_i+g/q^{i+1}$. Then $g=\xi_{i+1}+1$ if $t_{i+1}>0$ and
$g=\xi_{i+1}$ if $t_{i+1}=0$. In both cases $1\le g\le q$, because
$qt_i=\xi_{i+1}+t_{i+1}>0$, and
$g\le d_i$ because $d_i$ is an integer with
$d_i\ge qt_i=\xi_{i+1}+t_{i+1}$. If $g=q$, then $b_{i+1}=b_i$, and
\eqref{thirdsup} below is trivial. Let $g\le q-1$ and
$b_{i+1}\le z\le b_i$. Then
$z=a_i+q^{-i}u$ with $g/q\le u\le1$, and $u=(m+y)/q$ with an integer $m$
such that $g\le m\le q-1$ and with $0\le y\le1$. With
$V(m)=m(d_i+q-m)$, equations \eqref{cellagain} and \eqref{branches} give
\[
 q^{i+1}\bigl(T_q(z)-T_q(a_i)\bigr)=q\bigl(d_iu+T_q(u)\bigr)
 =V(m)+\bigl(V(m+1)-V(m)\bigr)\,y+T_q(y)\ge\min\{V(m),V(m+1)\},
\]
since $V(m+1)-V(m)=d_i+q-1-2m$ and $T_q(y)\ge0$. For every integer $m$
with $g\le m\le q$ we have $V(m)-V(g)=(m-g)(d_i-g+q-m)\ge0$. Hence
$q^{i+1}(T_q(z)-T_q(a_i))\ge V(g)$, and $V(g)=q^{i+1}(T_q(b_{i+1})-T_q(a_i))$ by
\eqref{cellagain} with $u=g/q$, because $T_q(g/q)=g(q-g)/q$. Therefore
\begin{equation}
 T_q(z)\ge T_q(b_{i+1}),~b_{i+1}\le z\le b_i,
 \label{thirdsup}
\end{equation}
and in particular $T_q(b_i)\ge T_q(b_{i+1})$. If $b_i=x$, then also
$b_{i+1}=x$. Thus $T_q(b_p)\ge T_q(b_{p+1})\ge\cdots$, and since $b_i$ tends to
$x$ and $T_q$ is continuous, $T_q(b_i)\ge T_q(x)$ for every $i\ge p$. Finally,
every $y$ with $x<y\le c_p$ lies in $[b_{i+1},b_i]$ for some $i\ge p$ with
$b_i>x$, and then $T_q(y)\ge T_q(b_{i+1})\ge T_q(x)$ by \eqref{thirdsup}.

Let now $q=3$ and $0<x<1/2$. Then $\xi_p=0$ and the digits before
$\xi_p$ are $1$, so that $\xi_i=0$ for some $i\le j$ exactly when
$j\ge p$. If $d_j\ge3t_j$ for an index $j\ge p$ with $t_j>0$, then
$d_j>0$, and $d_j\ge2$ because $d_j$ is even. Conversely, suppose that
$d_j\ge2$ for every $j\ge p$ with $t_j>0$, and let $j$ be such an index.
If $d_j\ge4$, then $d_j>3>3t_j$. Let $d_j=2$. If $\xi_{j+1}\le1$, then
$3t_j=\xi_{j+1}+t_{j+1}<2$. If $\xi_{j+1}=2$, then
$d_{j+1}=d_j+2-2\xi_{j+1}=0$, so that $t_{j+1}=0$ by the assumption for
the index $j+1$, and $3t_j=2$. In all cases $d_j\ge3t_j$.
\end{proof}

In the notation of this proof, if $1\le d_j\le q-1$, then by
\eqref{cellagain} and Lemma~\ref{firstcross} the point
$a_j+d_j/q^{j+1}$ is the largest point $y<c_j$ with
$T_q(y)=T_q(c_j)$, and $T_q>T_q(c_j)$ between this point and $c_j$. The
inequality $d_j\ge qt_j$ says that $x$ does not lie in this open
interval.

\begin{example}
\label{ex:rightmin}
Let $q=3$. The point $1/8=(0.\overline{01})_3$ has $d_j(1/8)\ge2$ for
every $j\ge1$.
By Theorem~\ref{rightmin}, $T_q(y)\ge T_q(1/8)=17/32$ for $1/8\le y\le1/2$,
where the value $T_q(1/8)$ is the one computed after
Proposition~\ref{ratval}. The point $7/26=(0.\overline{021})_3$ has
$\xi_1=0$, $\xi_3\ne0$, and $d_2(7/26)=0$, so the condition of the
theorem fails for $j=2$, and $T_q(7/26)$ is not the minimum of $T_q$ on
$[7/26,1/2]$. Indeed, Proposition~\ref{ratval} with $N=0$ and $p=3$
gives $T_q(7/26)=\frac{27}{26}F_3(7/26)=10/13$, which is larger than
$T_q(1/3)=2/3$.

Let now $q=5$. The point $7/25$ has the digits $1,2$ followed by zeros.
Here only $j=1$ has the two properties named in the theorem, and
$d_1(7/25)=2=5t_1$, so $T_q(7/25)=6/5$ is the minimum of $T_q$ on
$[7/25,1/2]$. This minimum is also attained at $2/5$. The point $3/10$
has the
digits $1,2,2,2,\ldots$, so that $d_j(3/10)=2$ for every $j\ge1$, but
$d_1(3/10)<5t_1=5/2$. Thus the condition fails for $j=1$, and indeed
$T_q(3/10)=13/10$ is larger than $T_q(2/5)=6/5$. In base $5$ the inequalities
$d_j(x)\ge2$ therefore do not suffice.
\end{example}

By the symmetry $T_q(1-x)=T_q(x)$, Theorem~\ref{rightmin} also describes
the points $x\in[1/2,1]$ with $T_q(x)=\min_{1/2\le y\le x}T_q(y)$. With
$h_j$ as in Theorem~\ref{lamdig}, \eqref{diff} gives
$\xi_j(h_j+1)-\xi_j(h_j)=(q-1)j-2s_q(h_j)$, so for $q=3$ the inequality
$s_3(h_j)\le j-1$ of Theorem~\ref{lamdig} says that adding the vertex
$h_j$ to $I_j(h_j)$ increases the boundary.

\begin{proof}[Proof of Theorem~\ref{lamdig}]
Put $x=h/q^n$, so that $0<x\le1/2$. The terminating digit sequence of
$x$ consists of $x_{n-1},\ldots,x_0$ followed by zeros. Hence
$\lfloor q^jx\rfloor=h_j$, $t_j=(h-q^{n-j}h_j)/q^{n-j}$, and
$d_j(x)=(q-1)j-2s_q(h_j)$ for $j\le n$, by \eqref{dnsq}, and no
$j\ge n$ is followed by a nonzero digit. Consequently \eqref{lamdig1} is
the inequality $d_j(x)\ge qt_j$, and the condition of Theorem~\ref{lamdig} is the
condition of Theorem~\ref{rightmin} for $x$. By the last statement of
Theorem~\ref{rightmin}, for $q=3$ it has the second form given in Theorem~\ref{lamdig}.

By Lemma~\ref{extramin} and the symmetry
$\xi_n(m)=\xi_n(q^n-m)$,
\[
 \lambda_h(Q_n^q)=\min\{\xi_n(m):h\le m\le q^n-h\}
 =\min\{\xi_n(m):h\le m\le q^n/2\}.
\]
We claim that this minimum equals $q^n\min_{x\le y\le1/2}T_q(y)$. The
points $m/q^n$ with $h\le m\le q^n/2$ lie in $[x,1/2]$, so by
\eqref{samp} the minimum is at least $q^n\min_{[x,1/2]}T_q$. Conversely,
let $x\le y<1/2$ and $m=\lfloor q^ny\rfloor$. Then
$h\le m<q^n/2$, and by \eqref{cell} and $T_q\ge0$ we have
$T_q(y)\ge\min\{T_q(m/q^n),T_q((m+1)/q^n)\}$. If $m+1\le q^n/2$, both
values occur among the numbers $\xi_n(m)/q^n$ considered. Otherwise $q$
is odd, $m=(q^n-1)/2$, and $T_q((m+1)/q^n)=T_q(m/q^n)$ by symmetry. If
$y=1/2$, then either $q$ is even and $T_q(1/2)=\xi_n(q^n/2)/q^n$, or $q$ is
odd and $T_q(1/2)$ is the maximum of $T_q$. This proves the
claim. It follows that $\lambda_h(Q_n^q)=\xi_n(h)$ if and only if
$T_q(x)=\min_{[x,1/2]}T_q$, and Theorem~\ref{rightmin} gives the first
statement.

Let $m'$ with $h\le m'\le q^n/2$ attain the minimum above. Then
$T_q(m'/q^n)\le T_q(y)$ for all $y\in[m'/q^n,1/2]$, so $m'$ satisfies the
condition. Hence $h^*$ exists and
$h\le h^*\le m'$. Since $h^*$
satisfies the condition, $\xi_n(h^*)=q^n\min_{[h^*/q^n,1/2]}T_q\le
\xi_n(m')=\lambda_h(Q_n^q)\le\xi_n(h^*)$. This proves the statement on
$h^*$.

For the last statement we write $\eta_i=x_{n-i}$, $1\le i\le n$, for the
digits of $h$, starting with the digit of the highest power, and
$d_j=\sum_{i=1}^{j}(q-1-2\eta_i)=(q-1)j-2s_q(h_j)$. Since
$h-q^{n-j}h_j$ is $\eta_{j+1}q^{n-j-1}$ plus an integer that is
nonnegative, less than $q^{n-j-1}$, and zero exactly when
$\eta_{j+2}=\cdots=\eta_n=0$, inequality \eqref{lamdig1} is false if and
only if
\begin{equation}
 d_j<\eta_{j+1},~\text{or}~d_j=\eta_{j+1}~\text{and one of the digits}~
 \eta_{j+2},\ldots,\eta_n~\text{is not}~0.
 \label{lamfail}
\end{equation}
We say that $h$ fails at $j$ if $1\le j<n$, the digits
$\eta_1,\ldots,\eta_j$ are not all equal to $(q-1)/2$, one of the digits
$\eta_{j+1},\ldots,\eta_n$ is not $0$, and \eqref{lamfail} holds. Thus $h$
satisfies the condition of Theorem~\ref{lamdig} if and only if it fails at no
$j$. By \eqref{lamfail}, a size that fails at $j$ has
$d_j\le\eta_{j+1}\le q-1$. Moreover, if $d_j\ge0$, then \eqref{lamfail}
implies that one of the digits $\eta_{j+1},\ldots,\eta_n$ is not $0$.

Let $\ell$ be the number of leading zeros of $h$, so that
$\eta_1=\cdots=\eta_\ell=0$, $\eta_{\ell+1}\ge1$, and $0\le\ell<n$. Then
$d_j=(q-1)j$ for $j\le\ell$, and
$d_j\ge(q-1)\ell-(q-1)(j-\ell)=(q-1)(2\ell-j)$ for $j\ge\ell$, with
equality if and only if $\eta_{\ell+1}=\cdots=\eta_j=q-1$. Suppose that
$h$ fails at $j$. If $j<\ell$, then $\eta_{j+1}=0<d_j$, contrary to
\eqref{lamfail}. Hence $j\ge\ell$, and
$(q-1)(2\ell-j)\le d_j\le q-1$ gives $j\ge2\ell-1$. If $j=2\ell-1$, then
$d_j=q-1=\eta_{j+1}$, so that
$\eta_{\ell+1}=\cdots=\eta_{j+1}=q-1$, and by \eqref{lamfail} one of the
digits $\eta_{j+2},\ldots,\eta_n$ is not $0$, which requires
$n\ge j+2=2\ell+1$. If $j\ge2\ell$, then $n\ge j+1\ge2\ell+1$ as well.
We have shown that a size that fails at some index has at most
$\lfloor(n-1)/2\rfloor$ leading zeros and fails only at indices
$j\ge2\ell-1$. For $n=1$ there is no index $j$ with $1\le j<n$. A size
with fewer than $\lfloor(n-1)/2\rfloor$ leading zeros is at least
$q^{n-\lfloor(n-1)/2\rfloor}$, and a size with exactly
$\lfloor(n-1)/2\rfloor$ leading zeros is less than this number. Hence, if
some size with exactly $\lfloor(n-1)/2\rfloor$ leading zeros fails at some
index, then the smallest such size is the smallest size that does not
satisfy the condition. From now on we consider only sizes with exactly
$\lfloor(n-1)/2\rfloor$ leading zeros, so that
$\ell=\lfloor(n-1)/2\rfloor$.

Suppose first that $n=2\ell+1$ with $\ell\ge1$, and let $h'$ be the size
whose digits are $\ell$ zeros, the digit $\lfloor(q+1)/2\rfloor$, then
$\ell-1$ digits $q-1$, and the digit $2\lfloor q/2\rfloor-1$. A size fails
only at $j=2\ell-1$ or at $j=2\ell$. If $h$ fails at $j=2\ell$, then
\eqref{lamfail} reads $d_{2\ell}<\eta_n$, and
$d_{2\ell}=(q-1)\ell+\sum_{i=\ell+1}^{2\ell}(q-1-2\eta_i)
=2\sum_{i=\ell+1}^{2\ell}(q-1-\eta_i)$ is an even number that is at most
$\eta_n-1\le q-2$, hence at most $2\lfloor q/2\rfloor-2$. The terms
$q-1-\eta_i$ are nonnegative, so $q-1-\eta_{\ell+1}\le\lfloor q/2\rfloor-1$,
that is, $\eta_{\ell+1}\ge\lfloor(q+1)/2\rfloor$. If
$\eta_{\ell+1}=\lfloor(q+1)/2\rfloor$, then the other terms are zero, so that
$\eta_{\ell+2}=\cdots=\eta_{2\ell}=q-1$, $d_{2\ell}=2\lfloor q/2\rfloor-2$, and
$\eta_n\ge2\lfloor q/2\rfloor-1$. Hence $h\ge h'$. If $h$ fails at
$j=2\ell-1$, then, as we saw, $\eta_{\ell+1}=\cdots=\eta_{2\ell}=q-1$ and
$\eta_n\ge1$. Again $h\ge h'$: this is clear if
$q-1>\lfloor(q+1)/2\rfloor$, and otherwise $q\le3$ and the last digit
$2\lfloor q/2\rfloor-1$ of $h'$ is $1$. The size $h'$ fails at $j=2\ell$,
because its first $2\ell$ digits contain a zero, which is not $(q-1)/2$, and
because $d_{2\ell}=2\lfloor q/2\rfloor-2$ is nonnegative and less than its
last digit. Moreover, $h'<q^{\ell+1}\le q^n/2$. Thus $h'$ is the
smallest size that does not satisfy the condition, and
\[
 h'=\Bigl\lfloor\frac{q+1}2\Bigr\rfloor q^\ell+(q^\ell-q)
 +2\Bigl\lfloor\frac q2\Bigr\rfloor-1,
\]
which is $\frac12(q+3)q^\ell-2$ for odd $q$ and $\frac12(q+2)q^\ell-1$ for
even $q$.

Suppose now that $n=2\ell+2$ and $q\ge3$, and let $h'$ be the size whose
digits are $\ell$ zeros, the digit $1$, then $\ell$ digits $q-1$, and the digit
$q-2$, so that $h'=q^{\ell+1}+(q^{\ell+1}-q)+q-2=2q^{\ell+1}-2$. Let $h$ be a
size with $\eta_{\ell+1}=1$ that fails at $j$. Then $j\ge \ell+1$, because
$j=\ell$ would give $d_\ell=(q-1)\ell\le\eta_{\ell+1}=1$ and hence $\ell=0=j$. For
$\ell+1\le j\le2\ell+1$ we have
\[
 d_j=(q-1)\ell+q-3+\sum_{i=\ell+2}^{j}(q-1-2\eta_i)\ge(q-1)(2\ell+1-j)+q-3,
\]
with equality if and only if $\eta_{\ell+2}=\cdots=\eta_j=q-1$. If
$j\le2\ell$, then $d_j\ge2q-4$, and $d_j\le q-1$ gives $q=3$, $j=2\ell$, and
$d_j=2$. Then $\eta_{\ell+2}=\cdots=\eta_{2\ell}=2$ by the case of equality, and
\eqref{lamfail} gives $\eta_{2\ell+1}=2$ and $\eta_n\ge1$. If $j=2\ell+1$, then
\eqref{lamfail} reads $d_{2\ell+1}<\eta_n$, and
$d_{2\ell+1}=q-3+2\sum_{i=\ell+2}^{2\ell+1}(q-1-\eta_i)\le\eta_n-1\le q-2$ gives
$\eta_{\ell+2}=\cdots=\eta_{2\ell+1}=q-1$ and $\eta_n\ge q-2$. In both cases
$h\ge h'$. Every size with $\ell$ leading zeros and $\eta_{\ell+1}\ge2$ is
larger than $h'$. The size $h'$ has $d_{2\ell+1}=q-3<q-2=\eta_n$, and its
first $2\ell+1$ digits are all equal to $(q-1)/2$ only if $\ell=0$ and $q=3$.
Moreover, $h'\le q^n/2$ because $(q^{\ell+1}-2)^2\ge0$. Hence, unless $n=2$
and $q=3$, the size $h'$ fails at $j=2\ell+1$, and it is the smallest size
that does not satisfy the condition. If $n=2$ and $q=3$, then $\ell=0$, so
that a size that fails at some index has no leading zero, and $h'=4$ is
the largest size and does not fail. Hence every size satisfies the
condition.

Finally, let $n=2\ell+2$ and $q=2$, so that $\eta_{\ell+1}=1$. If $\ell=0$, then
$h=2$ is the only size with no leading zero, and it fails at no index
because its last digit is $0$. Hence every size satisfies the condition.
Let $\ell\ge1$, and let $h'=2^{\ell+2}-3$ be
the size whose digits are $\ell$ zeros, then $\ell$ digits $1$, and the digits
$0$ and $1$. For $\ell\le j\le2\ell+1$ we have $d_j=2\ell-j+2z_j$, where $z_j$
is the number of zeros among $\eta_{\ell+2},\ldots,\eta_j$. The size $h'$
fails at $j=2\ell-1$, because $d_{2\ell-1}=1=\eta_{2\ell}$ and $\eta_n=1$. Let
$h<h'$ be a size with $\ell$ leading zeros that fails at $j$, so that
$2\ell-1\le j\le2\ell+1$. Either one of the digits
$\eta_{\ell+2},\ldots,\eta_{2\ell}$ is $0$, or these digits are $1$ and
$\eta_{2\ell+1}=\eta_n=0$. In the second case a nonzero digit follows
$\eta_j$ only if $j=2\ell-1$, and then $d_j=1=\eta_{2\ell}$ and
$\eta_{2\ell+1}=\eta_n=0$, contrary to \eqref{lamfail}. In the first case
$\ell\ge2$. If $z_j\ge1$, then $d_j\ge2\ell-j+2\ge1$. Here \eqref{lamfail}
requires $d_j\le\eta_{j+1}\le1$, so that $j=2\ell+1$ and $d_j=1$, but for
$j=2\ell+1$ it reads $d_j<\eta_n\le1$. If $z_j=0$, then $j=2\ell-1$, the
digit $\eta_{2\ell}$ is $0$, and $d_j=1>\eta_{2\ell}$, contrary to
\eqref{lamfail}. Hence no size below $h'$ fails, and $h'\le2^n/2$ is
the smallest size that does not satisfy the condition.
\end{proof}

For $q=3$, the inequality $s_3(h_j)\le j-1$ of Theorem~\ref{lamdig} says
that the first $j$ digits of $h$ contain more zeros than twos. For $n=3$
the only sizes with $\lambda_h(Q_3^3)\ne\xi_3(h)$ are $7=(021)_3$ and
$8=(022)_3$, whose first two digits contain as many zeros as twos. For
$n=4$ the exceptions are the sizes $16$ and $17$, whose digits begin with
$012$, the sizes $19$ to $26$, whose digits begin with $02$, and the
sizes $34$ and $35$, whose digits begin with $102$, as in the bottom left
panel of Figure~\ref{Q4pic}, and $h^*$ is $18$, $27$, and $36$ for the
three groups. Xu, Zhou, and Hsieh~\cite[proof of Theorem~3.4]{Xu2022} show that
$\lambda_h(Q_n^3)=\xi_n(h)$ whenever $1\le h\le3^{(n+1)/2}-3$ with
$n\ge3$ odd, or $1\le h\le2\cdot3^{n/2}-3$ with $n\ge2$ even, and their
Theorem~3.9 reduces the sizes $3^{n-1}\le h\le(3^n-1)/2$ to smaller cubes.
For $n\ge4$ and $3^{\lceil n/2\rceil}<h\le3^{n-1}$, their Theorem~3.7
gives a formula for $\lambda_h(Q_n^3)$ in terms of smaller cubes, and Wu, Zhang, Wang,
Qiao, and Yang~\cite[Theorem~5.2]{Wu2026} obtain the same formula; the
networks $B_n^3$ that they study include $Q_n^3$~\cite[Section~2.2]{Wu2026}.
Xu, Zhou, and Hsieh~\cite[Theorem~3.4]{Xu2022} also give
$\lambda_h(Q_n^3)=(n-1)3^{(n+1)/2}$ for $3^{(n+1)/2}-3\le h\le3^{(n+1)/2}$
when $n\ge3$ is odd, and $\lambda_h(Q_n^3)=(2n-2)3^{n/2}$ for
$2\cdot3^{n/2}-3\le h\le2\cdot3^{n/2}$ when $n\ge4$ is even. These values
are $\xi_n(3^{(n+1)/2})$ and $\xi_n(2\cdot3^{n/2})$ by \eqref{xi}, and by
\eqref{diff} the numbers $\xi_n(h)$ exceed them by $2$ at the two sizes
just below $3^{(n+1)/2}$ and $2\cdot3^{n/2}$, respectively. It follows
that the ranges $1\le h\le3^{(n+1)/2}-3$ and $1\le h\le2\cdot3^{n/2}-3$
on which $\lambda_h(Q_n^3)=\xi_n(h)$ are best possible for $n\ge3$; for
odd $n\ge3$ the values of $\lambda_h(Q_n^3)$ at these two sizes are also
reported by Ma, Zhang, Meng, and Ma~\cite[page~71]{Ma2021}. The last
statement of Theorem~\ref{lamdig} finds these sizes again from the
digits of $h$. For $q=4$,
Tian, Zhang, and Feng~\cite[Theorem~3.1]{TianZhangFeng2022} proved that
$\lambda_h(Q_n^4)=\xi_n(h)$ for $1\le h\le c\,4^{\lfloor n/2\rfloor}-5+c$,
where $c=2$ if $n$ is even and $c=3$ if $n$ is odd, and their Lemma~4.2
shows that $\xi_n(h)>\xi_n(c\,4^{\lfloor n/2\rfloor})$ for the sizes $h$
with $c\,4^{\lfloor n/2\rfloor}-5+c<h<c\,4^{\lfloor n/2\rfloor}$. For
$n\ge2$ these two results give the smallest size in the last statement of
Theorem~\ref{lamdig} for $q=4$, which is $2\cdot4^{n/2}-2$ for even $n$
and $3\cdot4^{(n-1)/2}-1$ for odd $n$. Their Theorems~3.2 and~3.3
determine the sizes $h$ with $\lambda_h(Q_n^4)=\xi_n(h)$ in the ranges
$3\cdot4^{n-1}/5\le h\le4^{n-1}$ and $6\cdot4^{n-1}/5\le h\le2\cdot4^{n-1}$.
Ma, Zhang, Meng, and Ma~\cite[Theorem~1]{Ma2021} show for
$n\ge3$ that $\lambda_h(Q_n^3)$ is constant, equal to
$\xi_n(3^{\lceil n/2\rceil+r})=2(\lfloor n/2\rfloor-r)3^{\lceil n/2\rceil+r}$,
on the range
$3^{\lceil n/2\rceil+r}-\lfloor3^{2r+2}/2\rfloor\le h\le3^{\lceil
n/2\rceil+r}$ for odd $n$ and on the range
$3^{\lceil n/2\rceil+r}-\lfloor3^{2r+1}/2\rfloor\le h\le3^{\lceil
n/2\rceil+r}$ for even $n$,
where $0\le r<\lfloor n/2\rfloor$, and for $n=9$ their Table~2 lists the sizes $h$ in these ranges with
$\lambda_h(Q_n^3)=\xi_n(h)$. Theorem~\ref{lamdig} yields these sizes
directly from the digits of $h$. For $n=9$ and $r=0$ the range is
$239\le h\le243$; the sizes $239=(000022212)_3$, $240=(000022220)_3$, and
$243=(000100000)_3$ satisfy the condition, while the digits of $241$ and
$242$ begin with $00002222$, followed by a nonzero digit.

Tian, Zhang, and Feng~\cite[Theorem~6]{TianZhangFeng2024} later showed for
$q=4$ that $\lambda_h(Q_n^4)=\xi_n(g4^t)=(3ng-3gt-g(g-1))4^t$ on ranges of
sizes $h$ that end at $g4^t$, where either $g\in\{1,2\}$, $n\ge3$, and
$\lfloor n/2\rfloor+1\le t\le n-1$, or $g=3$, $n\ge7$, and
$\lfloor n/2\rfloor+1\le t\le n-3$. Their equations~(1)--(3) give the
lower ends of these ranges, and their Table~6 lists, for
$n\le9$, the sizes $h\le2\cdot4^{n-1}$ at which $\lambda_h(Q_n^4)$ takes
one of the values $(3ng-3gt-g(g-1))4^t$. In some cases the length $L_g$
printed in their Theorem~6 gives shorter ranges. For instance, for $n=7$,
$g=1$, and $t=4$, their equation~(1) and their Table~6 give the range
$250\le h\le256$, on which $\lambda_h(Q_7^4)=2304$, while $L_1=2$ gives
$254\le h\le256$. For $q=5$, Liang, Zhang, and
Yang~\cite[Theorems~4.1 and~4.2]{LiangZhangYang2022} proved for $n\ge3$
that $\lambda_h(Q_n^5)=4\cdot5^{n-1}$ for
$\lceil2\cdot5^{n-1}/3\rceil\le h\le5^{n-1}$ and
$\lambda_h(Q_n^5)=6\cdot5^{n-1}$ for
$\lceil4\cdot5^{n-1}/3\rceil\le h\le2\cdot5^{n-1}$, and they determined
the sizes $h$ in these ranges with $\lambda_h(Q_n^5)=\xi_n(h)$. By
\eqref{xi}, $4\cdot5^{n-1}=\xi_n(5^{n-1})$ and
$6\cdot5^{n-1}=\xi_n(2\cdot5^{n-1})$. In each of these ranges
Theorem~\ref{lamdig} yields the sizes $h$ with
$\lambda_h(Q_n^q)=\xi_n(h)$ directly from the digits of $h$.

For $q=3$ and $q=2$, Propositions~\ref{optcount} and~\ref{optcount2}
count the sizes $1\le h\le q^n/2$ with $\lambda_h(Q_n^q)=\xi_n(h)$. Let
$q=3$. For $k\ge0$ let $P(k)$ be the number of words of length $k$ over
$\{0,1,2\}$ in which,
for every $j\le k$, the first $j$ letters contain at least as many zeros
as twos. If the letters $0$, $1$,
and $2$ are read as the steps $(1,1)$, $(1,0)$, and $(1,-1)$, the words
counted by $P(k)$ become the lattice paths with $k$ steps that start at
height $0$ and never go below it, and $P(k)$ is the term with index $k+1$
of the sequence A005773~\cite{OEIS-A005773}. We include the short proof of
the formula for $P(k)$ below.

\begin{proposition}
\label{optcount}
Let $q=3$. Then $P(k)=(C_k+C_{k+1})/2$ for every $k\ge0$, and for
$n\ge1$ the number of the sizes $1\le h\le(3^n-1)/2$ with
$\lambda_h(Q_n^3)=\xi_n(h)$ is
\begin{equation}
 n+2\sum_{k=0}^{n-2}(n-1-k)P(k).
 \label{Ncount}
\end{equation}
This number multiplied by $\sqrt{\pi n}/3^n$ tends to $\sqrt3/2$ as
$n\to\infty$. Thus the proportion of these sizes among the sizes
$1\le h\le(3^n-1)/2$ is asymptotically $\sqrt{3/(\pi n)}$.
\end{proposition}

For $n=1,2,\ldots,7$ this number is $1$, $4$, $11$, $28$, $71$, $184$,
and $489$. For $n=4$ the $28$ sizes are the sizes $1\le h\le40$ other
than the twelve exceptions listed above.

\begin{proof}
Let $h$ have the word $x_{n-1}\ldots x_0$. By Theorem~\ref{lamdig},
$\lambda_h(Q_n^3)=\xi_n(h)$ exactly when, for every $j$ such that the
first $j$ letters of the word contain a zero and some later letter is
nonzero, the first $j$ letters contain more zeros than twos. The size $(3^n-1)/2$, whose word is $1\ldots1$, satisfies this
condition. Every size $1\le h<(3^n-1)/2$ has a zero in its word, and we
let $f$ be the position of the first zero and $g$ the position of the
last nonzero digit, both counted from the left. If $g<f$, the word is
$1\ldots10\ldots0$ with $f-1\ge1$ ones, because the digits before the
first zero are nonzero and the word does not exceed $1\ldots1$; the
condition holds, and we obtain the $n-1$ sizes $(3^n-3^{n-r})/2$ with
$1\le r\le n-1$. Let $g>f$, and suppose that $h$ satisfies the condition.
The first $f$
letters contain a zero, and a later letter is nonzero, so by the
condition they contain more zeros than twos. Since they contain exactly
one zero, the digits before the first zero are all $1$. Thus the word is
$1^{f-1}0v0^{n-g}$, where $v$ is the word formed by the digits in the
positions $f+1,\ldots,g$. It has length $m=g-f\ge1$, and its last letter
is nonzero. For $f\le j<g$ the first $j$ letters contain a zero,
and a later letter is nonzero; they contain more zeros than twos exactly
when the first $j-f$ letters of $v$ contain at least as many zeros as
twos. For $j\ge g$ all later letters are zeros, and for $j<f$ the first
$j$ letters contain no zero. Conversely, every word $1^{f-1}0v0^{n-g}$ of
this kind with $1\le f<g\le n$ is the word of a size $1\le h<(3^n-1)/2$
that satisfies the condition. For given $f$ and $m$ there are $2P(m-1)$
words $v$, since the first $m-1$ letters of $v$ form a word counted by
$P(m-1)$ and the last letter is $1$ or $2$. Therefore the number of the sizes is
\[
 1+(n-1)+\sum_{f=1}^{n-1}\sum_{m=1}^{n-f}2P(m-1)
 =n+2\sum_{m=1}^{n-1}(n-m)P(m-1),
\]
which is \eqref{Ncount} with $k=m-1$.

For the formula for $P(k)$ we count the words of length $k$ whose first $j$
letters contain more twos than zeros for some $j$. For a word $v$ and
$j\ge0$, the number of zeros minus the number of twos among the first $j$
letters of $v$ is $\kappa_j(v)=d_j(v)/2$, and each letter changes it by at
most $1$. If a word $v$ of length $k$ has $\kappa_k(v)\le-1$, then
$\kappa_j(v)=-1$ for some $j$. Exchanging the digits $0$ and $2$ shows that
there are as many words $v$ of length $k$ with $\kappa_k(v)<0$ as with
$\kappa_k(v)>0$, so $(3^k-C_k)/2$ of them have $\kappa_k(v)\le-1$. Now
consider a word $v$ of length $k$ with $\kappa_k(v)=r\ge0$ and
$\kappa_j(v)=-1$ for some $j$. Exchanging the digits $0$ and $2$ after
the first such $j$ gives a word $v'$ with $\kappa_k(v')=-2-r$, and this is
a bijection onto the words with $\kappa_k(v')=-2-r$. Summing over
$r\ge0$, the number of these words is the number of words $v$ with
$\kappa_k(v)\le-2$, which is $(3^k-C_k)/2-C'_k$, where $C'_k$ is the
number of words $v$ of length $k$ with $\kappa_k(v)=-1$. Hence
$P(k)=C_k+C'_k$. A word $v$ of length $k+1$ with $\kappa_{k+1}(v)=0$ ends
with the letter $1$, $0$, or $2$, and its first $k$ letters form a word
$u$ with $\kappa_k(u)$ equal to $0$, $-1$, or $1$, respectively;
exchanging the digits $0$ and $2$ shows that there are $C'_k$ words $v$
of each of the last two kinds.
Therefore
$C_{k+1}=C_k+2C'_k$, and $P(k)=(C_k+C_{k+1})/2$.

By Lemma~\ref{trinom}, $\sqrt k\,C_k/3^k\to\sqrt3/(2\sqrt\pi)$. Hence
$C_{k+1}/C_k\to3$, $P(k)/C_k\to2$, and $P(k)/P(k-1)\to3$. If $a_k>0$ and
$a_{k+1}/a_k\to3$, then $a_i$ is eventually strictly increasing and
tends to infinity, and the partial sums $A_i=\sum_{k=0}^{i}a_k$ satisfy
$A_i/a_i\to3/2$ by the theorem of Stolz and Ces\`aro, because
$(A_i-A_{i-1})/(a_i-a_{i-1})=1/(1-a_{i-1}/a_i)\to3/2$. Moreover,
$A_{i+1}/A_i=1+(a_{i+1}/a_i)(a_i/A_i)\to1+3\cdot\frac23=3$, so the same
statement applies to the sequence of the $A_i$. Let $a_k=P(k)$ and
$S(n)=\sum_{k=0}^{n-2}(n-1-k)P(k)=\sum_{i=0}^{n-2}A_i$. The two
statements give $A_{n-2}/P(n-2)\to3/2$ and $S(n)/A_{n-2}\to3/2$. Hence
$S(n)/P(n-2)\to9/4$ and $S(n)/P(n-1)\to3/4$. Since
$P(n-1)=(C_{n-1}+C_n)/2$ and $C_n/C_{n-1}\to3$, we have
$P(n-1)\sqrt{\pi n}/3^{n-1}\to\sqrt3$. Thus $n/P(n-1)\to0$, because
$n\sqrt n/3^{n-1}\to0$, and $(n+2S(n))/P(n-1)\to3/2$. Therefore the number $n+2S(n)$ of the sizes
satisfies $(n+2S(n))\sqrt{\pi n}/3^n\to\sqrt3/2$, and dividing by the
number $(3^n-1)/2$ of sizes gives the last statement.
\end{proof}

For $q=2$ the graph $Q_n^2$ is the hypercube of dimension $n$. For
$n\ge6$ and $1\le h\le2^{\lfloor n/2\rfloor}$, Li and
Yang~\cite[Theorem~6]{LiYang2013} proved that $\lambda_h(Q_n^2)$ is $nh$
minus the largest sum of the degrees in the subgraph induced by a set of
$h$ vertices, that is, $\xi_n(h)$.

\begin{proposition}
\label{optcount2}
Let $q=2$. For $n\ge1$ the number of the sizes $1\le h\le2^{n-1}$ with
$\lambda_h(Q_n^2)=\xi_n(h)$ is
\begin{equation}
 1+\sum_{k=0}^{n-2}\binom k{\lfloor k/2\rfloor}.
 \label{Ncount2}
\end{equation}
This number multiplied by $\sqrt{\pi n}/2^n$ tends to $\sqrt2/2$ as
$n\to\infty$. Thus the proportion of these sizes among the sizes
$1\le h\le2^{n-1}$ is asymptotically $\sqrt{2/(\pi n)}$.
\end{proposition}

For $n=1,2,\ldots,8$ this number is $1$, $2$, $3$, $5$, $8$, $14$, $24$, and $44$.
These numbers are the first terms of the sequence
A072100~\cite{OEIS-A072100}, whose entry contains the formula
\eqref{Ncount2}. For $n=4$
the five sizes are $1$, $2$, $3$, $4$, and $8$, and the smallest size
that is not counted is $5=2^{n/2+1}-3$, in agreement with
Theorem~\ref{lamdig}.

\begin{proof}
Let $\eta_1\ldots\eta_n$ be the word of $h$, starting with the digit of
the highest power, let $d_j=\sum_{i=1}^{j}(1-2\eta_i)$ be its digit
slopes, and let $g$ be the place of its last nonzero digit. Since $1/2$
is not a digit, the condition of Theorem~\ref{lamdig} concerns every
$j$ with $1\le j<g$. By \eqref{lamfail}, inequality \eqref{lamdig1} holds at such a $j$
if and only if either $d_j>\eta_{j+1}$, or $d_j=\eta_{j+1}$ and
$j+1=g$. We claim that these conditions hold for every $j<g$ if and only
if $d_j\ge1$ for $1\le j<g$. Suppose that they hold, and let $1\le j<g$.
Then $d_j\ge\eta_{j+1}\ge0$, and $d_j=0$ would give $\eta_{j+1}=0$ and
$j+1=g$, although $\eta_g=1$. Hence $d_j\ge1$. Conversely, let $d_j\ge1$
for $1\le j<g$, and let $1\le j<g$. If $d_j\ge2$ or $\eta_{j+1}=0$, then
$d_j>\eta_{j+1}$. Otherwise $d_j=1=\eta_{j+1}$, so that $d_{j+1}=0$, and
therefore $j+1=g$, because $d_{j+1}\ge1$ if $j+1<g$. This proves the
claim.

The size $2^{n-1}$ has $g=1$ and is counted. Let $g\ge2$. Then
$d_1\ge1$ means that $\eta_1=0$, so that $h<2^{n-1}$, and for
$2\le j\le g-1$ the number $d_j-1$ is the digit slope of the word
$\eta_2\ldots\eta_{g-1}$ at its place $j-1$. Hence the sizes $h$ with a
given $g\ge2$ that are counted are the sizes with the words
$0v10\ldots0$, where $v$ is a word of length $g-2$ over $\{0,1\}$ all of
whose digit slopes are nonnegative. Let $B(k)$ be the number of these
words of length $k$. Then the number of the counted sizes is
$1+\sum_{g=2}^{n}B(g-2)$.

For the number $B(k)$ we consider the words of length $k$ over $\{0,1\}$
that have a negative digit slope at some place. Each letter changes the
digit slope by $1$ or by $-1$, so such a word has a first place at which
the digit slope is $-1$, and every word with $d_k\le-1$ is of this kind.
Exchanging the letters $0$ and $1$ after this place is a bijection from
the words of this kind with $d_k=r\ge0$ onto the words with $d_k=-2-r$.
Summing over $r\ge0$, we see that the number of words of this kind is the
number of words with $d_k\le-1$ plus the number of words with
$d_k\le-2$. Hence $B(k)$ is the number of words with $d_k\ge0$ minus the
number of words with $d_k\le-2$. Exchanging the letters $0$ and $1$
everywhere changes the sign of $d_k$, so there are as many words with
$d_k\le-2$ as with $d_k\ge2$, and $B(k)$ is the number of words with
$d_k\in\{0,1\}$. Since $d_k$ is $k$ minus twice the
number of ones, these are the words with $\lfloor k/2\rfloor$ ones, and
$B(k)=\binom k{\lfloor k/2\rfloor}$. This proves \eqref{Ncount2}.

By Lemma~\ref{trinom} for $q=2$, we have $B(2m)=\binom{2m}m=C_{2m}$ and
$\sqrt{2m}\,C_{2m}/4^m\to\sqrt{2/\pi}$. Since
$B(2m+1)=\binom{2m+1}m=\frac12\binom{2m+2}{m+1}$, it follows that
$\sqrt k\,B(k)/2^k\to\sqrt{2/\pi}$ as $k\to\infty$, and hence
$B(k+1)/B(k)\to2$. Thus $B(k)$ is eventually strictly increasing and
tends to infinity, and as in the proof of Proposition~\ref{optcount}, the
partial sums $A_i=\sum_{k=0}^{i}B(k)$ satisfy $A_i/B(i)\to1/(1-1/2)=2$ by
the theorem of Stolz and Ces\`aro. The number of the
counted sizes is $1+A_{n-2}$, and $B(n-2)\sqrt{\pi n}/2^{n-2}\to\sqrt2$,
so we obtain $(1+A_{n-2})\sqrt{\pi n}/2^n\to\sqrt2/2$. Dividing by the number $2^{n-1}$ of
sizes gives the last statement.
\end{proof}

The proofs of Propositions~\ref{optcount} and~\ref{optcount2} use
$q\le3$ in two steps. Inequality \eqref{lamdig1} becomes a lower bound
for the digit slope alone, $d_j\ge2$ for $q=3$ by the last statement of
Theorem~\ref{rightmin} and $d_j\ge1$ for $q=2$ by the proof of
Proposition~\ref{optcount2}. The words are then counted by exchanging
letters after the first place with a negative digit slope, and this
exchange is a bijection because the first negative digit slope is always $-2$ for
$q=3$ and $-1$ for $q=2$. If $q\ge4$, a letter can change the digit
slope by $3$ or more, and \eqref{lamdig1} involves the digit after the
place $j$ as well as the digit slope at $j$ (see
Example~\ref{ex:rightmin} for $q=5$). We do not know a formula of the
kind of \eqref{Ncount} or \eqref{Ncount2} for this number in these
bases.

\section{The modulus of continuity}
\label{sec:X}

For $\delta\ge0$, the modulus of continuity of $T_q$ at $\delta$ is the
supremum of $|T_q(x)-T_q(y)|$ over all real $x$ and $y$ with
$|x-y|\le\delta$. By Proposition~\ref{incr}(a) for the steps $h$ with
$|h|\le\delta$, and since $T_q(-h)=T_q(h)$, the modulus of continuity
equals $\max_{0\le h\le\delta}T_q(h)$. This modulus exceeds $T_q(\delta)$ for some
$\delta$. Indeed, by Theorem~\ref{cellmax}, applied to the interval
$[0,1/q]$ of level $1$, whose slope is $q-1$, the maximum of $T_q$ on
$[0,1/q]$ is greater than $T_q(1/q)$. For instance, let $q=3$. Then the
maximum of $T_q$ on $[0,1/3]$ is $T_q(1/3)+3^{-2}=7/9$, attained only at
$1/3-1/18=5/18$ (see the left panel of Figure~\ref{runpic}). Thus for
$\delta=1/3$ the modulus of continuity is $T_q(5/18)=7/9$, while
$T_q(\delta)=2/3$.

In every base $q\ge2$ we study the set $X$ of the points
$\delta\in[0,1/2]$ at which the modulus equals $T_q(\delta)$.
Figure~\ref{runpic} shows the modulus on $[0,1/2]$ for $q=3$, $4$, and
$5$, computed from the exact values of $T_q$ at the multiples of
$1/(2\cdot3^9)$, $1/(5\cdot4^6)$, and $1/(2\cdot5^6)$. The next theorem
describes $X$ by the numbers $\Lambda(c)=\max_{0\le t\le1}(T_q(t)-ct)$ of
Section~\ref{sec:max}.

\begin{theorem}
\label{runmax}
Let $0\le\delta\le1/2$, let $\xi_1,\xi_2,\ldots$ be the digit sequence of
$\delta$, where we take the terminating one when $\delta$ is a $q$-adic rational. Put
$t_n=q^n\delta-\lfloor q^n\delta\rfloor$ and $d_n=d_n(\delta)$ for $n\ge0$. Then
$\max_{0\le h\le\delta}T_q(h)=T_q(\delta)$ if and only if
\begin{equation}
 T_q(t_n)+d_nt_n\ge\Lambda(d_n+2)~\text{for every }n\ge1\text{ with
 }\xi_n\ne0.
 \label{runmax1}
\end{equation}
In this case $d_n\ge-1$ for every $n\ge0$, so that $d_n\ge0$ for every
$n$ if $q$ is odd. For $q=3$, with $\kappa_n=\kappa_n(\delta)=d_n/2$,
condition \eqref{runmax1} holds if and only if
\begin{equation}
 T_q(t_n)+2\kappa_nt_n\ge3^{-\kappa_n-1}~\text{for every }n\ge1\text{ with
 }\xi_n\ne0,
 \label{runmaxq3}
\end{equation}
and in this case $\kappa_n\ge0$ for every $n\ge0$.
\end{theorem}

\begin{proof}
Put $m_n=\lfloor q^n\delta\rfloor$ and $a_n=m_n/q^n$, so that
$\delta=a_n+q^{-n}t_n$. The digits of $m_n$ are $\xi_1,\ldots,\xi_n$, and
$d_n=(q-1)n-2s_q(m_n)$ by \eqref{dnsq}. Equation \eqref{cell} gives
\begin{equation}
 T_q(a_n+q^{-n}u)=T_q(a_n)+q^{-n}\bigl(T_q(u)+d_nu\bigr),~0\le u\le1,~n\ge0,
 \label{runmax3}
\end{equation}
and for $u=t_n$ the left side is $T_q(\delta)$. Let $n\ge1$ with
$\xi_n\ne0$. For $1\le i\le\xi_n$ put $J_i=[(m_n-i)/q^n,(m_n-i+1)/q^n]$.
The last digit of $m_n$ is $\xi_n$, so $s_q(m_n-i)=s_q(m_n)-i$, and the
interval $J_i$ of level $n$ has slope $d_n+2i$. The intervals
$J_1,\ldots,J_{\xi_n}$ lie in $[0,\delta]$, and their union is
$[a_{n-1},a_n]$. By \eqref{cell} with $t=1$, across an interval of level
$n$ the function $T_q$ increases by $q^{-n}$ times the slope of the
interval, so $T_q$ is $T_q(a_n)-q^{-n}\sum_{j=1}^{i-1}(d_n+2j)$ at the
right endpoint of $J_i$. By \eqref{cell2}, the maximum of $T_q$ on an
interval of level $n$ with slope $c$ is its value at the right endpoint
plus $q^{-n}\Lambda(c)$. Therefore
\begin{equation}
 \max_{h\in J_i}T_q(h)=T_q(a_n)
 +q^{-n}\Bigl(\Lambda(d_n+2i)-\sum_{j=1}^{i-1}(d_n+2j)\Bigr),
 ~1\le i\le\xi_n.
 \label{runJ}
\end{equation}
By \eqref{runmax3} with $u=t_n$ and \eqref{runJ} with $i=1$, inequality
\eqref{runmax1} at the place $n$ says that $T_q(\delta)\ge\max_{J_1}T_q$.
Since $J_1\subseteq[0,\delta]$, this holds at every place $n$ with
$\xi_n\ne0$ if $T_q(\delta)=\max_{[0,\delta]}T_q$, so \eqref{runmax1} is
necessary. If
$d_n\ge-2$, then the numbers $d_n+2j$ with $j\ge1$ are nonnegative, and
$\Lambda(d_n+2i)\le\Lambda(d_n+2)$ for $i\ge1$, so \eqref{runJ} gives
\begin{equation}
 \max_{a_{n-1}\le h\le a_n}T_q(h)=T_q(a_n)+q^{-n}\Lambda(d_n+2)
 ~\text{if}~\xi_n\ne0~\text{and}~d_n\ge-2.
 \label{runcell}
\end{equation}

Conversely, suppose \eqref{runmax1}. We claim that $d_n\ge-1$ for every
$n\ge0$. Otherwise let $n$ be the smallest index with $d_n\le-2$. Then
$d_n<d_{n-1}$, so that $q-1-2\xi_n<0$ and $\xi_n\ne0$. By the definition
of $\Lambda$ we have $T_q(t_n)+d_nt_n\le\Lambda(-d_n)$, and
$\Lambda(-d_n)<\Lambda(d_n+2)$ because $-d_n>d_n+2$. This contradicts
\eqref{runmax1} at the place $n$ and proves the claim. Now let $0\le
h<\delta$, let $\eta_1,\eta_2,\ldots$ be the digit sequence of
$h$, and let $n$ be the first place with $\eta_n\ne\xi_n$. If
$\eta_n>\xi_n$, then $h\ge a_{n-1}+(\xi_n+1)q^{-n}>a_n+q^{-n}t_n=\delta$,
because $t_n<1$. Hence $\eta_n<\xi_n$, so that $\xi_n\ne0$ and $a_{n-1}\le
h<a_{n-1}+(\eta_n+1)q^{-n}\le a_n$. By \eqref{runcell}, \eqref{runmax1},
and \eqref{runmax3},
\[
 T_q(h)\le T_q(a_n)+q^{-n}\Lambda(d_n+2)\le T_q(\delta).
\]
This proves the equivalence. If $q$ is odd, then every $d_n$ is even,
and $d_n\ge-1$ means that $d_n\ge0$.

Let now $q=3$. If $\kappa_n\ge-1$, then
$\Lambda(d_n+2)=\Lambda(2\kappa_n+2)=3^{-\kappa_n-1}$, so that the
inequalities in \eqref{runmax1} and \eqref{runmaxq3} at the place $n$ are
the same. If \eqref{runmax1} holds, then $\kappa_n\ge0$ for every $n$,
and \eqref{runmaxq3} follows. Suppose \eqref{runmaxq3}. If $\kappa_n<0$
for some $n$, then for the smallest such $n$ we have $\kappa_{n-1}=0$,
$\xi_n=2$, and $\kappa_n=-1$, so that \eqref{runmaxq3} at the place $n$
reads $T_q(t_n)-2t_n\ge1$, although $T_q(t_n)-2t_n\le\Lambda(2)=1/3$.
Hence $\kappa_n\ge0$ for every $n$, and \eqref{runmax1} follows.
\end{proof}

If $q$ is odd, the points $1/(2q^a)$ with $a\ge0$ lie in $X$. Indeed, the
digit sequence of $1/(2q^a)$ consists of $a$ zeros followed by the digit
$(q-1)/2$ forever. At a place $n>a$ we have $d_n=(q-1)a$ and $t_n=1/2$,
and $T_q(1/2)+(q-1)a/2\ge T_q(1/2)=\Lambda(0)\ge\Lambda(d_n+2)$, so
\eqref{runmax1} holds. For $q=3$, Lev~\cite[equation~(1.5)]{Lev2015}
proved that the function $\omega_3=T_3/2$ of Section~\ref{sec:T}
satisfies $\omega_3(h)\le h\log_3(3/(2h))$ for $0<h\le1$, and he noted
that this bound is attained at the points $1/(2\cdot3^a)$ with $a\ge0$.
If $q$ is even, digit slopes equal to $-1$ occur
at points of $X$, as the second part of the next example shows.

\begin{example}
\label{ex:runmax}
Let $q=3$. The point $1/3$ does not satisfy \eqref{runmaxq3}, since its
only nonzero digit is $\xi_1=1$, with $\kappa_1=0$ and $t_1=0$. The point
$19/54=(0.100\overline1)_3$ satisfies \eqref{runmaxq3}, with equality at
$n=1$, where $\kappa_1=0$, $t_1=1/18$, and $T_q(1/18)=1/3$ by
Proposition~\ref{incr}(b); at the places $n\ge4$ we have $\kappa_n=2$ and
$t_n=1/2$. By \eqref{cell}, $T_q(19/54)=T_q(1/3)+T_q(1/18)/3=7/9$. Thus the
maximum $7/9$ of $T_q$ over $[0,1/3]$, which is attained at $5/18$, is also
attained at $19/54$, and it is the maximum of $T_q$ over $[0,19/54]$.

Let now $q=2$, so that $T_q$ is Takagi's function, $\Lambda(0)=2/3$, and
$\Lambda(c)=2^{-c}\cdot2/3$ for $c\ge0$. The digit sequence of $5/12$ is
$0,1$ followed by the pair $1,0$ forever. Its nonzero digits stand at the
place $2$ and at the odd places $n\ge3$. At the place $2$ we have $d_2=0$,
$t_2=2/3$, and $T_q(2/3)=2/3\ge\Lambda(2)=1/6$. At an odd place $n\ge3$ we
have $d_n=-1$, $t_n=1/3$, and $T_q(1/3)-1/3=1/3=\Lambda(1)$. Thus
\eqref{runmax1} holds, with equality at the places where the digit slope
is $-1$, and $5/12\in X$. By Theorem~\ref{gmax}, $5/12$ is the largest
point of $[0,1/2]$ at which Takagi's function attains its maximum $2/3$.
\end{example}

Since $qT_q(\xi/q)=\xi(q-\xi)$ for $\xi\in\{0,1,\ldots,q-1\}$ by
\eqref{branches}, equation \eqref{cell} at the level $j-1$ gives, in the
notation of the proof of Theorem~\ref{runmax},
$T_q(a_j)-T_q(a_{j-1})=q^{-j}(d_{j-1}+q-\xi_j)\xi_j$ for $j\ge1$. By
\eqref{runmax3} and the continuity of $T_q$, the sum of these equations
over $j>n$ gives
\[
 T_q(t_n)+d_nt_n=\sum_{\ell=1}^{\infty}q^{-\ell}
 \bigl(d_{n+\ell-1}+q-\xi_{n+\ell}\bigr)\xi_{n+\ell},
\]
so \eqref{runmax1} is a condition on the digits of $\delta$ alone.

\begin{proposition}
\label{modset}
In every base $q$ the set $X$ is compact, has Lebesgue measure~$0$ and
Hausdorff dimension~$1$, and contains no $q$-adic rational other than
$0$.
\end{proposition}

\begin{proof}
The function $\delta\mapsto\max_{0\le h\le\delta}T_q(h)$ is continuous, so
$X$ is closed. By Theorem~\ref{runmax}, $d_n(\delta)\ge-1$ for every
$\delta\in X$ and $n\ge0$. Hence for every $n\ge1$ the set $X$ is covered
by the intervals $[m/q^n,(m+1)/q^n]$ of level $n$ with $0\le m<q^n$ such
that all digit slopes of the word of the $n$ digits of $m$ are at least
$-1$. We show that the number of these words is at most
$(2q+1)q^n\sqrt{\pi q}/(2\sqrt n)$, so that the total length of the
covering intervals tends to $0$ as $n\to\infty$ and $X$ has Lebesgue
measure $0$.

Consider the words of length $n$ over $\{0,1,\ldots,q-1\}$ that have a
digit slope $d_j\le-2$ at some place $j\le n$, and for such a word let
$j_0$ be the first such place. Then $-q\le d_{j_0}\le-2$, because
$d_{j_0-1}\ge-1$ and no letter lowers the digit slope by more than $q-1$.
Replacing every letter $\xi$ after the place $j_0$ by $q-1-\xi$ maps the
set of these words onto itself and is its own inverse. It does not change
$j_0$ and $d_{j_0}$, and it replaces $d_n-d_{j_0}$ by $d_{j_0}-d_n$. Hence
there are as many of these words with $d_n<d_{j_0}$ as with
$d_n>d_{j_0}$. Every word with $d_n\le-q-1$ is one of these words and has
$d_n<d_{j_0}$, so the number of these words is at least twice the number
of words with $d_n\le-q-1$. Replacing every letter $\xi$ of a word by
$q-1-\xi$ changes the sign of $d_n$, so twice this number is the number
of words with $|d_n|\ge q+1$. Therefore the number of words of length $n$
all of whose digit slopes are at least $-1$ is at most the number of words
with $|d_n|\le q$, which is at most $(2q+1)q^n\sqrt{\pi q}/(2\sqrt n)$ by
Lemma~\ref{trinom}.

For $0\le y\le T_q(1/2)$ let $\delta_y$ be the smallest
$\delta\in[0,1/2]$ with $T_q(\delta)=y$, which exists because $T_q(0)=0$
and $T_q$ is continuous. If $T_q(h)>y$ for some $h<\delta_y$, then $T_q$
takes the value $y$ on $[0,h]$ by the intermediate value theorem, contrary
to the choice of $\delta_y$. Hence $\delta_y\in X$, and $T_q(X)$ contains
the interval $[0,T_q(1/2)]$, where $T_q(1/2)>0$ by \eqref{Tpositive}. By
Corollary~\ref{holder}, for every $0<\alpha<1$ the function $T_q$ maps
every set $U$ to a set of diameter at most
$K_\alpha(\operatorname{diam}U)^\alpha$. Hence, if sets $U_1,U_2,\ldots$
of diameter at most $r$ cover $X$, then the sets $T_q(U_i)$ cover
$T_q(X)$, have diameter at most $K_\alpha r^\alpha$, and satisfy
$\sum_{i=1}^{\infty}(\operatorname{diam}T_q(U_i))^{s/\alpha}\le
K_\alpha^{s/\alpha}\sum_{i=1}^{\infty}(\operatorname{diam}U_i)^s$ for every
$s\ge0$. Consequently the Hausdorff measure of dimension $s/\alpha$ of
$T_q(X)$ is $0$ whenever the Hausdorff measure of dimension $s$ of $X$ is
$0$, and the Hausdorff dimension of $T_q(X)$ is at most $1/\alpha$ times
that of $X$. Since $T_q(X)$ contains an interval of
positive length, its Hausdorff dimension is $1$, so $X$ has Hausdorff
dimension at least $\alpha$ for every $\alpha<1$, and hence~$1$.

Finally, let $\delta>0$ be a $q$-adic rational with $\delta\le1/2$, and
let $n$ be the place of the last nonzero digit of its terminating digit
sequence. Then $t_n=0$, so the left side of \eqref{runmax1} is $0$ at the
place $n$, while the right side is positive, and $\delta\notin X$ by
Theorem~\ref{runmax}.
\end{proof}

In Theorem~\ref{lexX} below, condition \eqref{runmax1} at a place $n$
with $\xi_n\ne0$ is replaced by the condition that the tail $t_n$ is at least
the smallest $t$ with $T_q(t)+d_nt=\Lambda(d_n+2)$. The next lemma gives the digits of the smallest point at which
$T_q(t)+dt$ takes a given value.

\begin{lemma}
\label{firstpass}
Let $d$ be an integer, and let $c$ be a real number with
$0<c\le\Lambda(-d)$, where $\Lambda(-d)$ is the maximum of $T_q(t)+dt$
over $0\le t\le1$. Put $(d_0,c_0)=(d,c)$, and for $j\ge0$ let
$\zeta_{j+1}$ be the smallest integer $\zeta$ with $0\le\zeta\le q-1$ and
\[
 \zeta(q-\zeta+d_j)+\Lambda(2\zeta+1-q-d_j)\ge qc_j,
\]
and put
\[
 d_{j+1}=d_j+q-1-2\zeta_{j+1},~
 c_{j+1}=qc_j-\zeta_{j+1}(q-\zeta_{j+1}+d_j).
\]
Then such an integer $\zeta$ exists and $0<c_{j+1}\le\Lambda(-d_{j+1})$
for every $j\ge0$, and the number $\sum_{j=1}^{\infty}\zeta_jq^{-j}$ is
the smallest $t\in[0,1]$ with $T_q(t)+dt=c$.
\end{lemma}

\begin{proof}
Let $e$ be an integer and $b$ a real number with $0<b\le\Lambda(-e)$.
The function $T_q(t)+et$ is continuous on $[0,1]$, it is $0$ at $t=0$, and
its maximum is $\Lambda(-e)\ge b$. Hence there is a smallest
$\beta\in[0,1]$ with $T_q(\beta)+e\beta=b$, and $\beta>0$ and
$T_q(t)+et<b$ for $0\le t<\beta$. By \eqref{tiltbranch}, for
$0\le\zeta\le q-1$ the maximum of $T_q(t)+et$ over the interval
$[\zeta/q,(\zeta+1)/q]$ is
$\bigl(\zeta(q-\zeta+e)+\Lambda(2\zeta+1-q-e)\bigr)/q$. Since the maximum
over $[0,1]$ is at least $b$, there is a smallest $\zeta$ for which this
number is at least $b$. By the choice of $\zeta$, the function is less
than $b$ on $[0,\zeta/q]$, also when $\zeta=0$, and it takes a value
at least $b$ on $[\zeta/q,(\zeta+1)/q]$, hence the value $b$ by the
intermediate value theorem. Thus $\zeta/q<\beta\le(\zeta+1)/q$. Put
$e'=e+q-1-2\zeta$ and $b'=qb-\zeta(q-\zeta+e)$. By \eqref{tiltbranch},
$T_q(t)+et-b=(T_q(y)+e'y-b')/q$ at $t=(\zeta+y)/q$ for $0\le y\le1$.
Since $T_q(t)+et<b$ at $t=\zeta/q$, we have $b'>0$. The choice of $\zeta$ gives $b'\le\Lambda(-e')$, and
$q\beta-\zeta$ is the smallest $y\in[0,1]$ with $T_q(y)+e'y=b'$.

We apply this with $(e,b)=(d_j,c_j)$ for $j=0,1,2,\ldots$. Then
$\zeta=\zeta_{j+1}$ and $(e',b')=(d_{j+1},c_{j+1})$, so by induction on
$j$ the integer $\zeta_{j+1}$ exists and $0<c_{j+1}\le\Lambda(-d_{j+1})$.
Let $\beta_j$ be the smallest $t\in[0,1]$ with $T_q(t)+d_jt=c_j$. We have
shown that $\beta_j=(\zeta_{j+1}+\beta_{j+1})/q$ for every $j\ge0$. Hence
$\beta_0=\sum_{j=1}^{J}\zeta_jq^{-j}+q^{-J}\beta_J$ for every $J\ge0$, and
$\beta_0=\sum_{j=1}^{\infty}\zeta_jq^{-j}$ because $0\le\beta_J\le1$.
\end{proof}

For instance, let $q=5$, $d=0$, and $c=\Lambda(2)=7/10$. Since
$\Lambda(-4)=4+\Lambda(4)=43/10\ge5c$, the first digit is $\zeta_1=0$,
and $(d_1,c_1)=(4,7/2)$. For this pair the left side of the inequality of
the lemma is $\Lambda(-8)=8+3/50$ for $\zeta=0$, it is
$8+\Lambda(-6)=14+7/50$ for $\zeta=1$, and it is
$14+\Lambda(-4)=18+3/10$ for $\zeta=2$, while $5c_1=35/2$. Hence
$\zeta_2=2$ and $(d_2,c_2)=(4,7/2)$ again, so that all further digits are
$2$. Thus the smallest $t$ with $T_q(t)=7/10$ is $(0.0\overline2)_5=1/10$.

Let $q=3$ in the rest of this paragraph, in Lemma~\ref{sublevel}, and in
Corollary~\ref{thresholds}, which describe the smallest $t$ with
$T_q(t)+d_nt=\Lambda(d_n+2)$ more closely in this base. For an integer $k$
and $0\le t\le1$ put $G_k(t)=T_q(t)+2kt$. By \eqref{tiltbranch} with
$d=2k$, for $0\le y\le1$,
\begin{equation}
 G_k\Bigl(\frac y3\Bigr)=\frac{G_{k+1}(y)}3,~
 G_k\Bigl(\frac{1+y}3\Bigr)=\frac{2+2k+G_k(y)}3,~
 G_k\Bigl(\frac{2+y}3\Bigr)=\frac{2+4k+G_{k-1}(y)}3.
 \label{Gdigit}
\end{equation}
If $k\ge0$, then $G_k\ge0$, $G_k(0)=0$, $G_k(1)=2k$, and the maximum
of $G_k$ is $2k+3^{-k}$ by the last statement of Lemma~\ref{lem:tilt}.

\begin{lemma}
\label{sublevel}
Let $q=3$, and let $k\ge1$ and $c$ be integers with $1\le c\le2k$. Put $(k_0,c_0)=(k,c)$,
and for $j\ge0$ let
\[
 (\zeta_{j+1},k_{j+1},c_{j+1})=
 \begin{cases}
  (0,\,k_j+1,\,3c_j),&\text{if }3c_j\le2k_j+2,\\
  (1,\,k_j,\,3c_j-2k_j-2),&\text{if }2k_j+3\le3c_j\le4k_j+2,\\
  (2,\,k_j-1,\,3c_j-4k_j-2),&\text{if }3c_j\ge4k_j+3.
 \end{cases}
\]
Then $k_j\ge1$ and $1\le c_j\le2k_j$ for every $j\ge0$, and the number
$\Xi(k,c)$ with the ternary digits $\zeta_1,\zeta_2,\ldots$ is not a ternary
rational and satisfies
\[
 \{t\in[0,1]:G_k(t)<c\}=[0,\Xi(k,c))~\text{and}~G_k(\Xi(k,c))=c.
\]
For every $J\ge0$, the digits of $\Xi(k,c)$ after the first $J$ are
those of $\Xi(k_J,c_J)$.
\end{lemma}

\begin{proof}
We show by induction on $j$ that $k_j\ge1$ and $1\le c_j\le2k_j$. Suppose
that this holds for $j$. In the first case $3\le
c_{j+1}=3c_j\le2k_j+2=2k_{j+1}$. In the second case $1\le
c_{j+1}\le2k_j=2k_{j+1}$. In the third case $6k_j\ge3c_j\ge4k_j+3$ gives
$k_j\ge2$, so that $k_{j+1}\ge1$, and $1\le c_{j+1}\le2k_j-2=2k_{j+1}$.

The digits $\zeta_j$ are not eventually $2$, since the third case lowers $k_j$
by $1$ and $k_j\ge1$. They are not eventually $0$, since otherwise
$3^{i+1}c_J\le2(k_J+i)+2$ for some $J$ and all $i\ge0$, which fails for
large $i$. Hence $\Xi(k,c)$ lies in $(0,1)$ and is not a ternary
rational, and $\zeta_1,\zeta_2,\ldots$ is its only ternary expansion. The last
statement of the lemma holds because the rule started at $(k_J,c_J)$
produces the digits $\zeta_{J+1},\zeta_{J+2},\ldots$.

For an integer $i$ and a real number $b$, let $A(i,b)$ be the set of the
numbers $t\in[0,1]$ with $G_i(t)<b$. By \eqref{Gdigit},
\begin{equation}
 \begin{aligned}
  A(i,b)\cap\bigl[0,\tfrac13\bigr]&=\tfrac13A(i+1,3b),\\
  A(i,b)\cap\bigl[\tfrac13,\tfrac23\bigr]&=\tfrac13\bigl(1+A(i,3b-2i-2)\bigr),\\
  A(i,b)\cap\bigl[\tfrac23,1\bigr]&=\tfrac13\bigl(2+A(i-1,3b-4i-2)\bigr).
 \end{aligned}
 \label{Asplit}
\end{equation}
Let $i\ge0$. By the properties of $G_i$ stated before the lemma, $A(i,b)$
is empty if $b\le0$, and $A(i,b)=[0,1]$ if $b>2i+3^{-i}$, in particular
if $i\ge1$ and $b\ge2i+1$ is an integer. If $b$ is an integer with
$1\le b\le2i$, then $A(i,b)$ contains $0$ but not $1$.

We claim that, with $\zeta=\zeta_{j+1}$,
\begin{equation}
 A(k_j,c_j)=\bigl[0,\tfrac\zeta3\bigr)\cup\tfrac13\bigl(\zeta+A(k_{j+1},c_{j+1})\bigr).
 \label{Astep}
\end{equation}
Indeed, write $(i,b)=(k_j,c_j)$ and consider the three sets in
\eqref{Asplit}. In the first case the second and the third set are empty,
since $3b-2i-2\le0$ and $3b-4i-2<0$. In the second case the first set is
$[0,1/3]$, since $3b\ge2(i+1)+1$, and the third set is empty, since
$3b-4i-2\le0$. In the third case the first two sets are $[0,1/3]$ and
$[1/3,2/3]$, since $3b-2i-2\ge2i+1$. In each case the remaining set is
$\frac13(\zeta+A(k_{j+1},c_{j+1}))$, and it contains the point $\zeta/3$, because
$0\in A(k_{j+1},c_{j+1})$. This proves \eqref{Astep}.

By \eqref{Astep} and induction on $J$, we have
$A(k,c)=[0,z_J)\cup(z_J+3^{-J}A(k_J,c_J))$ for every $J\ge0$, where
$z_J=\sum_{j=1}^J\zeta_j3^{-j}$. Since $A(k_J,c_J)$ contains $0$ but not $1$,
we get $[0,z_J)\subseteq A(k,c)\subseteq[0,z_J+3^{-J})$ for every $J$,
and therefore $[0,\Xi(k,c))\subseteq A(k,c)\subseteq[0,\Xi(k,c)]$. The set
$A(k,c)$ is open in $[0,1]$, because $G_k$ is continuous. If it contained
$\Xi(k,c)$, it would be $[0,\Xi(k,c)]$, a nonempty set different from
$[0,1]$ that is closed and open in the connected space $[0,1]$. Hence
$A(k,c)=[0,\Xi(k,c))$, so $G_k(\Xi(k,c))\ge c$, and
$G_k(\Xi(k,c))\le c$ by continuity, since $\Xi(k,c)>0$.
\end{proof}

\begin{corollary}
\label{thresholds}
Let $q=3$. For every integer $k\ge1$,
\[
 \{t\in[0,1]:G_k(t)<3^{-k-1}\}=\bigl[0,3^{-k-1}\Xi(2k+1,1)\bigr)
 ~\text{and}~
 G_k\bigl(3^{-k-1}\Xi(2k+1,1)\bigr)=3^{-k-1}.
\]
Moreover, $\{t\in[0,2/3]:T_q(t)<1/3\}=[0,1/18)$ and $T_q(1/18)=1/3$. The
number $1/18$ and the numbers $3^{-k-1}\Xi(2k+1,1)$ with $k\ge1$ are not
ternary rationals.
\end{corollary}

\begin{proof}
Let $k\ge1$ and $0\le i\le k$. In \eqref{Asplit} with $k+i$ in place of
$i$ and $b=3^{i-k-1}$, the second and the third set are empty, since
$3b\le1$ and $k+i-1\ge0$. Hence
$A(k+i,3^{i-k-1})=\frac13A(k+i+1,3^{i-k})$, and these equations for
$i=0,\ldots,k$ give $A(k,3^{-k-1})=3^{-k-1}A(2k+1,1)$.
Lemma~\ref{sublevel} with $2k+1$ and $1$ in place of $k$ and $c$ gives the
first equation, and \eqref{Gdigit} gives
$G_k(3^{-k-1}\Xi(2k+1,1))=3^{-k-1}G_{2k+1}(\Xi(2k+1,1))=3^{-k-1}$. For the
statements on $1/18$, \eqref{Asplit} gives
$A(0,1/3)\cap[0,1/3]=\frac13A(1,1)$ and $A(0,1/3)\cap[1/3,2/3]=\emptyset$.
The rule of Lemma~\ref{sublevel} started at $(1,1)$ gives the digit $0$
and the pair $(2,3)$, and then the digit $1$ and the pair $(2,3)$ again,
so that $\Xi(1,1)=(0.0\overline1)_3=1/6$. This gives
$\{t\in[0,2/3]:T_q(t)<1/3\}=[0,1/18)$, and $T_q(1/18)=1/3$ by
Proposition~\ref{incr}(b). Finally, the numbers $3^{-k-1}\Xi(2k+1,1)$ are
not ternary rationals by Lemma~\ref{sublevel}, and
$1/18=(0.00\overline1)_3$.
\end{proof}

For instance, for $q=3$ the rule of Lemma~\ref{sublevel} started at $(3,1)$ gives the digits $0$, $0$, $2$,
and then the digit $1$ forever, so that $\Xi(3,1)=(0.002\overline1)_3$ and
$3^{-2}\Xi(3,1)=5/486=(0.00002\overline1)_3$.

The proof of Lemma~\ref{sublevel} uses $q=3$ through the three equations
\eqref{Gdigit}. For $q\ne3$ we do not know whether the set of the points
$t\in[0,1]$ with $T_q(t)+dt<c$ is an interval $[0,u)$ for all integers
$d\ge1$ and $1\le c\le d$, as Lemma~\ref{sublevel} shows for $q=3$ and
even $d$.

\begin{theorem}
\label{lexX}
For an integer $d\ge-1$ let $\vartheta_d$ be the smallest $t\in[0,1]$
with $T_q(t)+dt=\Lambda(d+2)$. Let $0\le\delta\le1/2$, with $\xi_n$, $t_n$,
and $d_n$ as in Theorem~\ref{runmax}. Then $\delta\in X$ if and only if
$d_n\ge-1$ for every $n\ge1$ and $t_n\ge\vartheta_{d_n}$ for every
$n\ge1$ with $\xi_n\ne0$. No $\vartheta_d$ is a $q$-adic rational, and
the inequality $t_n\ge\vartheta_{d_n}$ holds exactly when the sequence
$\xi_{n+1},\xi_{n+2},\ldots$ agrees with the digit sequence of
$\vartheta_{d_n}$ or is larger at the first place where the two sequences
differ. For $q=3$ we have $\vartheta_0=1/18$ and
$\vartheta_{2k}=3^{-k-1}\Xi(2k+1,1)$ for every $k\ge1$, with the numbers
$\Xi(k,c)$ of Lemma~\ref{sublevel} (see Corollary~\ref{thresholds}), so that
$\delta\in X$ if and only if $\kappa_n\ge0$ for every $n\ge1$ and
$t_n\ge\vartheta_{2\kappa_n}$ for every $n\ge1$ with $\xi_n\ne0$.
\end{theorem}

By Lemma~\ref{firstpass} with $c=\Lambda(d+2)$, the digits of
$\vartheta_d$ can be computed one after the other.

\begin{proof}
Let $d\ge-1$. Then $-d\le d+2$, so that $0<\Lambda(d+2)\le\Lambda(-d)$. By
Lemma~\ref{firstpass} the number $\vartheta_d$ exists, and since
$T_q(t)+dt$ is $0$ at $t=0$, we have $\vartheta_d>0$ and
\begin{equation}
 T_q(t)+dt<\Lambda(d+2)~\text{for}~0\le t<\vartheta_d.
 \label{thetabelow}
\end{equation}
Suppose that $\vartheta_d$ is a $q$-adic rational. By \eqref{cusp1},
$T_q(\vartheta_d-h)-T_q(\vartheta_d)>dh$ for all sufficiently small $h>0$,
so that $T_q(t)+dt>T_q(\vartheta_d)+d\vartheta_d=\Lambda(d+2)$ at the points
$t=\vartheta_d-h$. This contradicts \eqref{thetabelow}. Hence
$\vartheta_d$ is not a $q$-adic rational.

Let $\delta\in X$. By Theorem~\ref{runmax}, $d_n\ge-1$ for every $n$,
and \eqref{runmax1} holds. By \eqref{thetabelow} with $d=d_n$,
inequality \eqref{runmax1} at a place $n$ with $\xi_n\ne0$ implies that
$t_n\ge\vartheta_{d_n}$.

Conversely, suppose that $d_n\ge-1$ for every $n\ge1$ and that
$t_n\ge\vartheta_{d_n}$ for every $n\ge1$ with $\xi_n\ne0$, and suppose
that $\delta\notin X$. The set of points $h\in[0,\delta]$ with $T_q(h)=\max_{[0,\delta]}T_q$ is
nonempty and compact. It does not contain $\delta$. Let $y$ be its largest
point, so $y<\delta$. Let $a_n$ be as in
the proof of Theorem~\ref{runmax}, and let $n$ be the first place at which
the digit sequence of $y$ differs from $\xi_1,\xi_2,\ldots$.
As in that proof, $\xi_n\ne0$ and $a_{n-1}\le y<a_n$. Since $d_n\ge-1$,
equation \eqref{runcell} gives $T_q(y)\le T_q(a_n)+q^{-n}\Lambda(d_n+2)$.
Put $z=a_n+q^{-n}\vartheta_{d_n}$. Then $a_n\le z\le
a_n+q^{-n}t_n=\delta$, and by \eqref{runmax3} with $u=\vartheta_{d_n}$,
\[
 T_q(z)=T_q(a_n)+q^{-n}\Lambda(d_n+2)\ge T_q(y)=\max_{[0,\delta]}T_q.
\]
Hence $z$ belongs to this set, although $z\ge a_n>y$. This contradiction shows that
$\delta\in X$.

For the statement on the digits, let $v_1,v_2,\ldots$ be the digit
sequence of $\vartheta=\vartheta_{d_n}$, which is unique and not
eventually $q-1$ because $\vartheta$ is not a $q$-adic rational. The
digits $\xi_{n+1},\xi_{n+2},\ldots$ of $t_n$ are not eventually $q-1$
either. Let $j$ be the first place with $\xi_{n+j}\ne v_j$. If $\xi_{n+j}>v_j$,
then $t_n\ge\sum_{i=1}^{j-1}v_iq^{-i}+(v_j+1)q^{-j}>\vartheta$. If
$\xi_{n+j}<v_j$, then
$t_n<\sum_{i=1}^{j-1}v_iq^{-i}+(\xi_{n+j}+1)q^{-j}\le\vartheta$. If the
two sequences agree, then $t_n=\vartheta$.

Let now $q=3$. Since $T_q(t)+2kt=G_k(t)$ and $\Lambda(2k+2)=3^{-k-1}$ for
$k\ge0$, Corollary~\ref{thresholds} gives $\vartheta_0=1/18$ and
$\vartheta_{2k}=3^{-k-1}\Xi(2k+1,1)$ for every $k\ge1$. Since
$d_n=2\kappa_n$, the conditions $d_n\ge-1$ and $t_n\ge\vartheta_{d_n}$
are the conditions $\kappa_n\ge0$ and $t_n\ge\vartheta_{2\kappa_n}$.
\end{proof}

\begin{example}
\label{ex:lexX}
Let $q=3$. Since $\vartheta_0=1/18=(0.00\overline1)_3$, at every place $n$ with
$\xi_n\ne0$ and $\kappa_n=0$ Theorem~\ref{lexX} requires that the digits
$\xi_{n+1},\xi_{n+2},\ldots$ agree with $0,0,1,1,\ldots$ or are larger at
the first place where the two sequences differ. The point
$5/18=(0.02\overline1)_3$ has $\kappa_1=1$ and $\kappa_n=0$ for $n\ge2$,
and after each place $n\ge2$ its digits are $1,1,1,\ldots$, which are
larger at the first place. Hence $5/18\in X$. The point $2/9=(0.02)_3$
has $\xi_2=2$ and $\kappa_2=0$, and after the place $2$ its digits are
$0,0,0,\ldots$, which are smaller at the third place. Hence
$2/9\notin X$. Indeed, by Theorem~\ref{cellmax} the maximum of $T_q$ on
$[1/9,2/9]$ is $T_q(2/9)+3^{-3}=19/27$, attained at $11/54$, while
$T_q(2/9)=2/3$.

Let now $q=2$. Lemma~\ref{firstpass} gives
$\vartheta_{-1}=(0.0\overline{01})_2=1/6$ and
$\vartheta_0=(0.00001000\overline{01})_2=25/768$. For the point $5/12$ of
Example~\ref{ex:runmax} we have $d_2=0$ and $t_2=2/3\ge\vartheta_0$ at
the place $2$, and $d_n=-1$ and $t_n=1/3\ge\vartheta_{-1}$ at the odd
places $n\ge3$. Thus Theorem~\ref{lexX} shows again that $5/12\in X$.
\end{example}

\begin{lemma}
\label{Xisol}
The point $0$ is not an isolated point of $X$. If $q$ is odd, then
$1/2\in X$, and $1/2$ is not an isolated point of $X$. A point
$\delta\in X$ with $0<\delta<1/2$ is an isolated point of $X$ if and only
if $\delta$ is a strict local maximum of $T_q$ and $T_q$ takes the value
$T_q(\delta)$ at some point of $[0,\delta)$.
\end{lemma}

\begin{proof}
For $0\le x\le1/2$ let $\Omega(x)=\max_{0\le h\le x}T_q(h)$. The function
$\Omega$ is continuous and nondecreasing, $T_q\le\Omega$ on $[0,1/2]$, and
$X$ is the set of the points $x\in[0,1/2]$ with $T_q(x)=\Omega(x)$. We use
the following fact several times. Let $0\le x<x'\le1/2$, and let $y$ be a
real number with $\Omega(x)<y\le\Omega(x')$. Then there is a point
$z\in X$ with $x<z\le x'$ and $T_q(z)=y$. Indeed, the set of the points
$h\in[0,x']$ with $T_q(h)\ge y$ is closed, and it is nonempty because
$\Omega(x')\ge y$. Let $z$ be its smallest point. Then $z>x$, because
$T_q\le\Omega(x)<y$ on $[0,x]$. Since $T_q<y$ on $[0,z)$ and
$T_q(z)\ge y$, continuity gives $T_q(z)=y$, so that $T_q(z)=\Omega(z)$ and
$z\in X$.

If $0<x'\le1/2$, then $\Omega(x')\ge T_q(x')>0=\Omega(0)$ by
\eqref{Tpositive}, and the fact with $x=0$ and $y=\Omega(x')$ gives a
point of $X$ in $(0,x']$. Hence $0$ is not an isolated point of $X$. Let
$q$ be odd. By Theorem~\ref{gmax}, $1/2$ is the only point of $[0,1/2]$
at which $T_q$ takes the value $\Omega(1/2)$, so that $1/2\in X$ and
$\Omega(x)<\Omega(1/2)$ for $0\le x<1/2$. For such an $x$ and a number
$y$ with $\Omega(x)<y<\Omega(1/2)$, the fact with $x'=1/2$ gives a point
$z\in X$ with $x<z\le1/2$ and $T_q(z)=y$, so that $z\ne1/2$. Hence $1/2$
is not an isolated point of $X$.

Let $\delta\in X$ with $0<\delta<1/2$. Suppose first that $\delta$ is a
strict local maximum of $T_q$ and that $T_q$ takes the value $T_q(\delta)$
at a point of $[0,\delta)$. Choose $\varepsilon>0$ so small that this
point lies in $[0,\delta-\varepsilon)$, that $\delta+\varepsilon<1/2$, and
that $T_q(h)<T_q(\delta)$ whenever $0<|h-\delta|<\varepsilon$. For these
$h$ we have $\Omega(h)\ge T_q(\delta)>T_q(h)$, so $h\notin X$, and
$\delta$ is an isolated point of $X$.

Suppose now that $\delta$ is an isolated point of $X$, and let
$\varepsilon>0$ be such that $0<\delta-\varepsilon$,
$\delta+\varepsilon<1/2$, and $\delta$ is the only point of $X$ in
$(\delta-\varepsilon,\delta+\varepsilon)$. If
$\Omega(\delta-\varepsilon)<\Omega(\delta)$, then the fact with
$x=\delta-\varepsilon$, $x'=\delta$, and a number $y$ with
$\Omega(\delta-\varepsilon)<y<\Omega(\delta)$ gives a point $z\in X$ with
$\delta-\varepsilon<z\le\delta$ and $T_q(z)=y<T_q(\delta)$, which is not
possible. Hence $\Omega(\delta-\varepsilon)=\Omega(\delta)=T_q(\delta)$, and
$T_q$ takes the value $T_q(\delta)$ at a point of $[0,\delta-\varepsilon]$.
If $\Omega(x')>\Omega(\delta)$ for some $x'\in(\delta,\delta+\varepsilon)$,
then the fact with $x=\delta$ and $y=\Omega(x')$ gives a point of $X$ in
$(\delta,x']$, which is not possible either. Hence
$\Omega(h)=T_q(\delta)$ for $\delta-\varepsilon\le h<\delta+\varepsilon$,
and $T_q(h)<\Omega(h)=T_q(\delta)$ for $0<|h-\delta|<\varepsilon$,
because $h\notin X$. Thus $\delta$ is a strict local maximum of $T_q$.
\end{proof}

\begin{proposition}
\label{Xperfect}
\begin{enumerate}
\item[(a)] If $q$ is even or $q=3$, then the set $X$ has no isolated
 points.
\item[(b)] If $q=m^2$ with an odd integer $m\ge3$, then
 $\frac12-(q-1)m/q^3$ is an isolated point of $X$. For $q=9$ this is the
 point $227/486$.
\end{enumerate}
\end{proposition}

Together with Proposition~\ref{modset}, part (a) shows that for even $q$
and for $q=3$ the set $X$ is a compact set without isolated points and
without interior points.

\begin{proof}
By Lemma~\ref{Xisol}, $0$ is not an isolated point of $X$ in any base. Let
$q$ be even. Then $1/2\notin X$ by Proposition~\ref{modset}, because $1/2$
is a $q$-adic rational. By Proposition~\ref{evenq}, no local maximum of
$T_q$ is strict, so by Lemma~\ref{Xisol} no point $\delta\in X$ with
$0<\delta<1/2$ is an isolated point of $X$. This proves (a) for even $q$.

Let now $q=3$, let $\delta\in X$ with $\delta>0$, and let $\xi_n$, $t_n$,
and $\kappa_n$ be as in Theorem~\ref{runmax}. By Proposition~\ref{modset},
$\delta$ is not a ternary rational, so it has infinitely many nonzero
digits, at the places $p_1<p_2<\cdots$. We construct points of $X$ other
than $\delta$ that are arbitrarily close to $\delta$. First we reduce the
condition \eqref{runmaxq3} at the places up to $p_i$ to one inequality for
the tail $t_{p_i}$.

Let $n\ge0$ and $\ell\ge1$, and suppose that
$\xi_{n+1}=\cdots=\xi_{n+\ell-1}=0$ and $\xi_{n+\ell}=e\in\{1,2\}$. Then
$\kappa_{n+\ell}=\kappa_n+\ell-e$, and \eqref{Gdigit}, applied $\ell-1$
times with the digit $0$ and then once with the digit $e$, gives
\begin{equation}
 G_{\kappa_n}(t_n)
 =3^{-\ell}\bigl(2+2e(\kappa_n+\ell-1)+G_{\kappa_{n+\ell}}(t_{n+\ell})\bigr).
 \label{Gblock}
\end{equation}
This holds for every point of $[0,1)$ in place of $\delta$. We call the
digits $\xi_{n+1},\ldots,\xi_{n+\ell}$ a string $(\ell,e)$. Put
$b_1=3^{-\kappa_{p_1}-1}$ and, for $i\ge1$,
\[
 b_{i+1}=\max\bigl\{3^\ell b_i-2e(\kappa_{p_i}+\ell-1)-2,\
 3^{-\kappa_{p_{i+1}}-1}\bigr\},
\]
where $(\ell,e)$ is the string that ends at $p_{i+1}$. By \eqref{Gblock},
$G_{\kappa_{p_i}}(t_{p_i})\ge b$ is equivalent to
$G_{\kappa_{p_{i+1}}}(t_{p_{i+1}})\ge3^\ell b-2e(\kappa_{p_i}+\ell-1)-2$.
By induction on $i$, the inequalities \eqref{runmaxq3} at the places
$p_1,\ldots,p_i$ together are therefore equivalent to
$G_{\kappa_{p_i}}(t_{p_i})\ge b_i$. Here $b_i$ depends only on the first
$p_i$ digits, and \eqref{runmaxq3} at the places after $p_i$ depends only
on $\kappa_{p_i}$ and on the digits after $p_i$. By Theorem~\ref{runmax},
a point $\delta'\in[0,1/2]$ whose first $p$ digits are given, the last of
them nonzero, therefore lies in $X$ exactly when the value of
$G_{\kappa_p(\delta')}$ at the tail of $\delta'$ after $p$ digits is at
least the number $b$ computed from these digits in the same way and
\eqref{runmaxq3} holds for $\delta'$ at the places after $p$.

We call $(k,b)$ a state if either $k\ge0$ and $b=3^{-k-1}$,
or $k\ge1$ and $b$ is an integer with $1\le b\le2k$. We claim
that $(\kappa_{p_i},b_i)$ is a state for every $i$. The map
$b\mapsto3^\ell b-2e(k+\ell-1)-2$ sends $3^{-k-1}$ to a negative
number if $\ell\le k+1$ and to an integer otherwise, and it sends
integers to integers and negative numbers to negative numbers. Hence
every $b_i$ is either $3^{-\kappa_{p_i}-1}$ or a positive integer. Since
$\delta\in X$, we have $b_i\le G_{\kappa_{p_i}}(t_{p_i})\le
2\kappa_{p_i}+3^{-\kappa_{p_i}}$, so an integer $b_i$ satisfies
$b_i\le2\kappa_{p_i}$ if $\kappa_{p_i}\ge1$ and $b_i=1$ if
$\kappa_{p_i}=0$. The second possibility does not occur. Indeed, $b_1$ is
not an integer, and if $b_{i+1}=1$ and $\kappa_{p_{i+1}}=0$, then
$\kappa_{p_i}+\ell-e=0$ and $3^\ell b_i-2e(\kappa_{p_i}+\ell-1)-2=1$, so
that $(\kappa_{p_i},\ell,e,b_i)$ is $(0,1,1,1)$, $(1,1,2,7/3)$, or
$(0,2,2,7/9)$. The last two values of $b_i$ are neither integers nor of
the form $3^{-\kappa_{p_i}-1}$, so $b_i=1$ and $\kappa_{p_i}=0$, and going
back we would reach $b_1=1$.

For a state $(k,b)$ we call $\Xi(k,b)$ the point of the state if $b$ is an
integer, and $\vartheta_{2k}$ otherwise. We claim that a point
$\delta'\in[0,1/2]$ lies in $X$ if, for some place $p$ at which its digit
is nonzero, the number $k=\kappa_p(\delta')$ and the number $b$ computed
from the first $p$ digits of $\delta'$ form a state $(k,b)$ and the tail
of $\delta'$ after $p$ digits is the point of this state. Indeed, $G_k$ takes the value $b$ at this point,
by Lemma~\ref{sublevel} and Corollary~\ref{thresholds}, so it remains to
check \eqref{runmaxq3} for $\delta'$ at the places after $p$. If $b$ is an
integer, let $k_j$, $c_j$, and $\zeta_j$ be as in Lemma~\ref{sublevel}
with $c=b$. The digits $0$, $1$, and $2$ change $\kappa_n(\delta')$ by
$1$, $0$, and $-1$, and so do the three cases of the rule for $k_j$, so
$\kappa_{p+j}(\delta')=k_j\ge1$. If $\zeta_j\ne0$, the tail of $\delta'$
after $p+j$ digits is $\Xi(k_j,c_j)$ by Lemma~\ref{sublevel}, and
$G_{k_j}(\Xi(k_j,c_j))=c_j\ge1>3^{-k_j-1}$, which is \eqref{runmaxq3} at
the place $p+j$. If $b=3^{-k-1}$ and $k\ge1$, the first $k+1$ digits of
$\vartheta_{2k}=3^{-k-1}\Xi(2k+1,1)$ are zeros, which carry no condition,
so that $\kappa_{p+k+1}(\delta')=2k+1$, and the remaining digits are those
of $\Xi(2k+1,1)$, so the integer case applies. If $b=1/3$ and $k=0$, then
the point of the state is $\vartheta_0=(0.00\overline1)_3$,
$\kappa_{p+2}(\delta')=2$, and $G_2(1/2)=3>3^{-3}$.

Now let $(k,b)=(\kappa_{p_i},b_i)$, and consider a string $(\ell,e)$ after
$p_i$. It leads to $k'=k+\ell-e$ in place of $\kappa_{p_{i+1}}$ and to the
number $b'=\max\{3^\ell b-2e(k+\ell-1)-2,3^{-k'-1}\}$ in place of
$b_{i+1}$. We distinguish three cases.

(a) If $b=3^{-k-1}$, then the strings $(1,1)$ and $(1,2)$ if $k\ge1$, and
$(1,1)$ and $(2,2)$ if $k=0$, lead to the state $(k',3^{-k'-1})$, since
$3^\ell b-2e(k+\ell-1)-2<0$ for them.

(b) If $b$ is an integer with $3b\le4k+2$, then the string $(1,1)$ leads
to $k'=k$ and $b'=\max\{3b-2k-2,3^{-k-1}\}$, where $3b-2k-2\le2k$, and the
string $(1,2)$ leads to $k'=k-1\ge0$ and $b'=3^{-k}$, since $3b-4k-2\le0$.
Both are states.

(c) If $b$ is an integer with $3b\ge4k+3$, then $6k\ge3b$ gives $k\ge2$,
and the string $(1,2)$ leads to the state $(k-1,3b-4k-2)$. The string
$(1,1)$ gives $b'\ge3b-2k-2\ge2k+1=2k'+1$. For $\ell\ge2$, the
inequalities $1\le e\le2$ and $3^\ell b\ge3^{\ell-1}(4k+3)\ge 12k+3^\ell$
give
\[
 b'-2k'\ge3^\ell b-6k-6\ell+2\ge6k+3^\ell-6\ell+2>0.
\]
Every string other than $(1,2)$ has $\ell\ge e$, so that $k'\ge k\ge2$,
and it gives an integer $b'\ge2k'+1$, which exceeds $2k'+3^{-k'}=\max
G_{k'}$.

Since $\delta\in X$, the string of $\delta$ that ends at $p_{i+1}$
satisfies $b_{i+1}\le G_{\kappa_{p_{i+1}}}(t_{p_{i+1}})$ and
$\kappa_{p_{i+1}}\ge0$. In case (c) it is therefore the string $(1,2)$,
and $\kappa_{p_{i+1}}=\kappa_{p_i}-1$. As the numbers $\kappa_{p_i}$ are
nonnegative, case (c) cannot occur for every $i$ from some index on, so
cases (a) and (b) occur for infinitely many $i$. For such an $i$, one of
the two strings named there differs from the string of $\delta$ after
$p_i$, and it leads to a state $(k',b')$. Let $\delta'$ be the point whose
digits are the first $p_i$ digits of $\delta$, then this string, and then
the digits of the point of the state $(k',b')$. Then
$|\delta'-\delta|\le3^{-p_i}$, and $\delta'\ne\delta$, since the digits of
these two points, which are not ternary rationals, differ after $p_i$. If
$\delta<1/2$, then $\delta'<1/2$ for large $i$. If $\delta=1/2$, then all
digits of $\delta$ are $1$, every $(\kappa_{p_i},b_i)$ equals $(0,1/3)$,
and the other string is $(2,2)$, so that $\delta'<\delta$. In both cases
$\delta'\in X$ by the preceding claim, so $\delta$ is not an isolated
point of $X$. This proves (a) for $q=3$.

For (b), let $q=m^2$ with an odd integer $m\ge3$, put $r=(q-1)/2$, and
let $\delta=\frac12-(q-1)m/q^3$. Since
$\frac12=\sum_{j=1}^{\infty}rq^{-j}$ and $(q-1)m/q^3=m/q^2-m/q^3$, the
digit sequence of $\delta$ is $r,r-m,r+m$ followed by the digit $r$
forever. Here $r-m\ge1$ and $r+m\le q-1$, because
$2(r-m-1)=(m-3)(m+1)\ge0$ and $m\le r$. The digit slopes of $\delta$ are
$d_1=0$, $d_2=2m$, and $d_n=0$ for $n\ge3$, and $0<\delta<1/2$.
Moreover, $\delta=(K+\frac12)/q^3$ for the integer $K$ with the digits
$r,r-m,r+m$, and $2s_q(K)=6r=3(q-1)$. By Proposition~\ref{oddq} with
$N=3$, the point $\delta$ is a strict local maximum of $T_q$.

We verify \eqref{runmax1} for $\delta$. All digits of $\delta$ are
nonzero. At a place $n\ge3$ we have $d_n=0$ and $t_n=1/2$, and
$T_q(1/2)=\Lambda(0)>\Lambda(2)$. At the place $2$ we have $d_2=2m$ and
$t_2=(r+m)/q+1/(2q)$. By \eqref{tiltbranch} with $d=2m$, $\zeta=r+m$, and
$y=1/2$, and since $T_q(1/2)=(q+1)/4$ and $q-r=r+1$,
\[
 T_q(t_2)+2mt_2=\frac{(r+m)(r+m+1)+(q+1)/4}q.
\]
Since $(r+m)(r+m+1)>r(r+1)=(q^2-1)/4$, this number is larger than
$(q+1)/4=\Lambda(0)\ge\Lambda(2m+2)$. At the place $1$ we have $d_1=0$
and $t_1=(r-m+t_2)/q$, and \eqref{tiltbranch} with $d=0$, $\zeta=r-m$,
and $y=t_2$ gives
\[
 T_q(t_1)=\frac{(r-m)(r+m+1)+T_q(t_2)+2mt_2}q
 =\frac{(r+m+1)\bigl(q(r-m)+r+m\bigr)+(q+1)/4}{q^2}.
\]
Now $r+m+1=(m+1)^2/2$ and
$q(r-m)+r+m=(q+1)r-(q-1)m=(q-1)(m-1)^2/2$, so that the product of these
two numbers is $(q-1)(m^2-1)^2/4=(q-1)^3/4$. Therefore
\[
 T_q(t_1)=\frac{(q-1)^3+q+1}{4q^2}=\frac{(q-2)^2+q}{4q}=\Lambda(2),
\]
where the last equation is the first case of \eqref{tilt0} with $b=2$.
Thus \eqref{runmax1} holds at every place, with equality at the place
$1$, and $\delta\in X$ by Theorem~\ref{runmax}.

In the notation of the proof of Theorem~\ref{runmax}, where $a_1=r/q$,
equation \eqref{runJ} with $n=1$ and $i=1$ shows that the maximum of
$T_q$ on the interval $[(r-1)/q,r/q]$ is $T_q(a_1)+q^{-1}\Lambda(2)$, and
by \eqref{runmax3} this number is $T_q(a_1)+q^{-1}T_q(t_1)=T_q(\delta)$.
Hence $T_q$ takes the value $T_q(\delta)$ at a point of $[(r-1)/q,r/q]$,
which lies in $[0,\delta)$ because $r/q<\delta$, and Lemma~\ref{Xisol}
shows that $\delta$ is an isolated point of $X$. For $q=9$ we have $m=3$,
$\delta=\frac12-24/729=227/486$, and $T_q(\delta)=389/162$, and by
Theorem~\ref{cellmax} the maximum of $T_q$ on $[3/9,4/9]$ is attained at
$4/9-7/162=65/162$.
\end{proof}

\section{Approximate convexity}
\label{sec:convex}

A convex function $f$ satisfies $f((x+y)/2)\le(f(x)+f(y))/2$ for all real
$x$ and $y$. Takagi's function
$T_2(x)=\sum_{j=0}^{\infty}2^{-j}\dist(2^jx,\Z)$ is far from convex, but
Boros~\cite{Boros2008} proved that its value at the midpoint of two real
numbers $x$ and $y$ exceeds the mean of its values at $x$ and $y$ by at
most $|x-y|/2$, which P\'ales had conjectured~\cite[Section~1]{Boros2008}.
From inequalities for the summatory functions of the digit sums,
Allaart~\cite[Theorems~3 and~11]{Allaart2014} derived this inequality and
its version for $T_q$, in which the constant $1/2$ is replaced by
$\lfloor(q+1)/2\rfloor/2$. In this section we prove two inequalities of this kind
for $T_q$ in every base, Theorems~\ref{qpoint} and~\ref{lamq}.

\begin{lemma}
\label{convlin}
For all integers $n\ge0$ and $m$ and for $0\le t\le1$,
\begin{equation}
 T_q\Bigl(\frac{m+t}{q^n}\Bigr)
 =(1-t)\,T_q\Bigl(\frac m{q^n}\Bigr)+t\,T_q\Bigl(\frac{m+1}{q^n}\Bigr)
 +q^{-n}T_q(t).
 \label{convlin1}
\end{equation}
\end{lemma}

\begin{proof}
If $m$ is replaced by its remainder modulo $q^n$, the points
$(m+t)/q^n$, $m/q^n$, and $(m+1)/q^n$ change by the same integer, so we
may assume that $0\le m<q^n$, because $T_q$ has period $1$. As
$T_q(1)=0$, equation \eqref{cell} with $t=1$ gives
$T_q((m+1)/q^n)=T_q(m/q^n)+q^{-n}\bigl((q-1)n-2s_q(m)\bigr)$. We multiply
this equation by $t$ and subtract it from \eqref{cell}. The terms with the
factor $(q-1)n-2s_q(m)$ cancel, and \eqref{convlin1} remains.
\end{proof}

For integers $r\ge0$ and $m$ we put
\[
 H_r(m)=q^rT_q\Bigl(\frac m{q^r}\Bigr).
\]
By \eqref{samp}, $H_r(m)$ is the boundary count $\xi_r(m)$ when $r\ge1$
and $0\le m\le q^r$. Since $T_q$ is zero at the integers, $H_0=0$, and
\eqref{fe} and \eqref{tau} give $H_1(m)=q\tau_q(m/q)=b(q-b)$, where $b$ is
the remainder of $m$ modulo $q$. In particular, $H_1$ has period $q$.

\begin{lemma}
\label{hstep}
\begin{enumerate}
\item[(a)] For all integers $m$ and $m'$, $H_1(m+m')\le H_1(m)+H_1(m')$.
\item[(b)] For all integers $r\ge1$, $t$, and $0\le b\le q$,
 \[
  H_r(qt+b)=(q-b)H_{r-1}(t)+bH_{r-1}(t+1)+b(q-b).
 \]
\item[(c)] For all integers $r\ge0$ and $t$,
 $H_r(t+1)+H_r(t-1)-2H_r(t)\ge-2$.
\end{enumerate}
\end{lemma}

\begin{proof}
(a) The proof of Proposition~\ref{incr} shows that $T_q(v+w)\le T_q(v)+T_q(w)$
for all real $v$ and $w$. With $v=m/q$ and $w=m'/q$, and after
multiplication by $q$, this is (a).

(b) We apply \eqref{convlin1} with $r-1$, $t$, and $b/q$ in place of $n$,
$m$, and $t$, and multiply by $q^r$. Thus
$H_r(qt+b)=(q-b)H_{r-1}(t)+bH_{r-1}(t+1)+qT_q(b/q)$, and
$qT_q(b/q)=H_1(b)=b(q-b)$ for $0\le b\le q$.

(c) We use induction on $r$. For $r=0$ the left side is $0$. Let $r\ge1$,
and write the given integer as $qt+b$ with integers $t$ and
$0\le b\le q-1$. Suppose first that $b\ge1$. Then $b-1$, $b$, and $b+1$
lie between $0$ and $q$, and in (b) the first two terms on the right side
form a polynomial of degree at most $1$ in $b$. Hence these terms cancel
in the following combination, and (b) gives
\[
 H_r(qt+b+1)+H_r(qt+b-1)-2H_r(qt+b)
 =(b+1)(q-b-1)+(b-1)(q-b+1)-2b(q-b)=-2.
\]
Suppose now that $b=0$. By (b), $H_r(qt)=qH_{r-1}(t)$,
$H_r(qt+1)=(q-1)H_{r-1}(t)+H_{r-1}(t+1)+q-1$, and, since
$qt-1=q(t-1)+q-1$, $H_r(qt-1)=H_{r-1}(t-1)+(q-1)H_{r-1}(t)+q-1$.
Therefore
\[
 H_r(qt+1)+H_r(qt-1)-2H_r(qt)
 =H_{r-1}(t+1)+H_{r-1}(t-1)-2H_{r-1}(t)+2(q-1)\ge2q-4\ge0,
\]
by the induction hypothesis and since $q\ge2$.
\end{proof}

\begin{lemma}
\label{convjen}
Let $r\ge1$, and let $Y_0,\ldots,Y_{q-1}$ be integers with sum $s$. Then
\[
 H_r(s)\le\sum_{v=0}^{q-1}H_{r-1}(Y_v)+\sum_{v=0}^{q-1}Y_v^2-\frac{s^2}q
 +\frac{q-1}q\,H_1(s).
\]
\end{lemma}

\begin{proof}
The differences $H_{r-1}(u+1)+(u+1)^2-H_{r-1}(u)-u^2$ do not decrease
as the integer $u$ increases, because the difference at $u$ minus the
difference at $u-1$ is $H_{r-1}(u+1)+H_{r-1}(u-1)-2H_{r-1}(u)+2$, which is
nonnegative by Lemma~\ref{hstep}(c).
Write $s=qt+b$ with integers $t$ and $0\le b\le q-1$, and let $Y$ be an
integer. If $Y>t$, then $H_{r-1}(Y)+Y^2-H_{r-1}(t)-t^2$ is the sum of the
$Y-t$ differences with $t\le u\le Y-1$, and each of them is at least the
difference with $u=t$. If $Y<t$, then $H_{r-1}(t)+t^2-H_{r-1}(Y)-Y^2$ is
the sum of the $t-Y$ differences with $Y\le u\le t-1$, and each of them
is at most the difference with $u=t$. Since $(t+1)^2-t^2=2t+1$, in both
cases, and also for $Y=t$,
\[
 H_{r-1}(Y)+Y^2\ge H_{r-1}(t)+t^2+(Y-t)\bigl(H_{r-1}(t+1)-H_{r-1}(t)+2t+1\bigr).
\]
We apply this inequality to $Y_0,\ldots,Y_{q-1}$ and add. The numbers
$Y_v-t$ with $0\le v\le q-1$ have the sum $s-qt=b$, so
\[
 \sum_{v=0}^{q-1}H_{r-1}(Y_v)+\sum_{v=0}^{q-1}Y_v^2
 \ge(q-b)H_{r-1}(t)+bH_{r-1}(t+1)+qt^2+2bt+b.
\]
By Lemma~\ref{hstep}(b), $(q-b)H_{r-1}(t)+bH_{r-1}(t+1)=H_r(s)-b(q-b)$.
Here $qt^2+2bt+b=s^2/q+b(q-b)/q$, and $b(q-b)=H_1(s)$. Hence the right
side of the last inequality is $H_r(s)+s^2/q-\frac{q-1}qH_1(s)$, and the
lemma follows.
\end{proof}

\begin{lemma}
\label{gapfree}
Let $M_0,\ldots,M_{q-1}$ be integers such that every integer between
the smallest and the largest of them is equal to one of them. Then
\[
 \sum_{0\le v<v'\le q-1}(M_{v'}-M_v)^2
 +(q-1)\,H_1\Bigl(\sum_{v=0}^{q-1}M_v+\frac{q(q-1)}2\Bigr)
 \le\frac{q^2(q^2-1)}{12}.
\]
\end{lemma}

\begin{proof}
Fix an integer $k$. For integers $M_0,\ldots,M_{q-1}$ put
$s=\sum_{v=0}^{q-1}M_v+q(q-1)/2$ and
\[
 \Omega(M_0,\ldots,M_{q-1})=\frac{q^2(q^2-1)}{12}
 -\sum_{0\le v<v'\le q-1}(M_{v'}-M_v)^2+(q-1)(s-qk)\bigl(s-q(k+1)\bigr).
\]
We claim that $\Omega\ge0$ whenever $M_0,\ldots,M_{q-1}$ satisfy the
hypothesis of the lemma. This proves the lemma. Indeed, for
$k=\lfloor s/q\rfloor$ the remainder of $s$ modulo $q$ is $s-qk$, so
$H_1(s)=(s-qk)(q(k+1)-s)$, and $\Omega\ge0$ is the asserted inequality.

The sum $\sum_{0\le v<v'\le q-1}(M_{v'}-M_v)^2$ equals
$q\sum_{v=0}^{q-1}M_v^2-\bigl(\sum_{v=0}^{q-1}M_v\bigr)^2$, in which each
$M_u^2$ has the coefficient $q-1$. The product $(s-qk)(s-q(k+1))$ equals
$s^2$ plus a polynomial of degree at most $1$ in $s$, so it contains each
$M_u^2$ with the coefficient $1$. Hence the coefficient of $M_u^2$ in
$\Omega$ is $-(q-1)+(q-1)=0$. Thus, when the other $q-1$ numbers are
fixed, $\Omega$ is a polynomial of degree at most $1$ in $M_u$.

We prove the claim by induction on the number of pairs $u<u'$ with
$M_u=M_{u'}$. Suppose first that there is no such pair. Then the numbers
$M_v$ are $c,c+1,\ldots,c+q-1$ in some order, for an integer $c$. In this
case
$\sum_{0\le v<v'\le q-1}(M_{v'}-M_v)^2=\sum_{d=1}^{q-1}(q-d)d^2=q^2(q^2-1)/12$
and $s=q(c+q-1)$, so
\[
 \Omega=(q-1)q^2(c+q-1-k)(c+q-2-k)\ge0,
\]
because a product of two consecutive integers is nonnegative. Suppose now
that $M_u=M_{u'}$ for some $u\ne u'$, and let $\alpha$ and $\beta$ be the
smallest and the largest of the numbers $M_v$. We replace $M_u$ by an
integer $x$ with $\alpha-1\le x\le\beta+1$ and leave the other numbers
unchanged. The other numbers still take every integer value from $\alpha$
to $\beta$, since the value $M_u$ is taken by $M_{u'}$. Thus the new
family satisfies the hypothesis of the lemma for each such $x$. On these
$x$ the function $\Omega$ is a polynomial of degree at most $1$, and
$\alpha-1\le M_u\le\beta+1$, so $\Omega(M_0,\ldots,M_{q-1})$ is at least the
smaller of the values of $\Omega$ for $x=\alpha-1$ and for $x=\beta+1$. In
both of these families the number $x$ is different from all the others,
so they have fewer pairs of equal numbers than the given family, and
$\Omega\ge0$ for them by the induction hypothesis.
\end{proof}

The next theorem contains the inequality of Boros, which is its case
$q=2$. For $q=3$ and $x_1\le x_2\le x_3$ the three distances
$|x_i-x_j|$ add up to $2(x_3-x_1)$, so the case $q=3$ of the theorem is
the inequality
\[
 T_q\Bigl(\frac{x+y+z}3\Bigr)\le\frac{T_q(x)+T_q(y)+T_q(z)}3+\frac23(z-x),
 ~x\le y\le z,
\]
which Lev~\cite[Theorem~1.4]{Lev2015} proved for his function
$\omega_3=T_3/2$ of Section~\ref{sec:T}.

\begin{theorem}
\label{qpoint}
For all real $x_1,\ldots,x_q$,
\begin{equation}
 T_q\Bigl(\frac{x_1+\cdots+x_q}q\Bigr)
 \le\frac1q\sum_{i=1}^{q}T_q(x_i)+\frac1q\sum_{1\le i<j\le q}|x_i-x_j|.
 \label{qpoint1}
\end{equation}
\end{theorem}

\begin{proof}
We show that for every integer $r\ge1$ and all integers
$x_1,\ldots,x_q$,
\begin{equation}
 H_r(x_1+\cdots+x_q)
 \le\sum_{i=1}^{q}H_{r-1}(x_i)+\sum_{1\le i<j\le q}|x_i-x_j|.
 \label{qpoint2}
\end{equation}
Divided by $q^r$, this is \eqref{qpoint1} at the points
$x_1/q^{r-1},\ldots,x_q/q^{r-1}$, and any $q$ numbers that are $q$-adic
rationals have this form with the same $r$. Since the $q$-adic rationals are
dense in $\R$ and both sides of \eqref{qpoint1} are continuous functions
of $(x_1,\ldots,x_q)$, inequality \eqref{qpoint1} then holds for all real
numbers.

We prove \eqref{qpoint2} by induction on $r$. Let $r=1$. Then the first
sum on the right side is $0$. Both sides are symmetric in
$x_1,\ldots,x_q$, so we may assume that $x_1\le\cdots\le x_q$. Put
$e_i=x_{i+1}-x_i$ for $1\le i\le q-1$. Then
$x_1+\cdots+x_q=qx_1+\sum_{i=1}^{q-1}(q-i)e_i$. The function $H_1$ has
period $q$, and the numbers $e_i$ are nonnegative integers, so
Lemma~\ref{hstep}(a) gives
\[
 H_1(x_1+\cdots+x_q)=H_1\Bigl(\sum_{i=1}^{q-1}(q-i)e_i\Bigr)
 \le\sum_{i=1}^{q-1}e_iH_1(q-i)=\sum_{i=1}^{q-1}i(q-i)e_i.
\]
The right side equals $\sum_{1\le j'<j\le q}(x_j-x_{j'})$, because
$x_j-x_{j'}$ is the sum of the numbers $e_i$ with $j'\le i<j$, and each
$e_i$ occurs in $i(q-i)$ of these differences.

Let $r\ge2$, and suppose that \eqref{qpoint2} holds with $r-1$ in place
of $r$. Let $x_1,\ldots,x_q$ be integers with sum $s$. For $1\le i\le q$
and every integer $v$ put
\[
 y_{i,v}=\Bigl\lfloor\frac{x_i+v}q\Bigr\rfloor
 ~\text{and}~Y_v=\sum_{i=1}^{q}y_{i,v}.
\]
Write $x_i=qa_i+b_i$, where $b_i$ is the remainder of $x_i$ modulo $q$.
For $0\le v\le q-1$ we have $y_{i,v}=a_i$ if $v<q-b_i$ and
$y_{i,v}=a_i+1$ if $v\ge q-b_i$. Hence $\sum_{v=0}^{q-1}y_{i,v}=x_i$, and
Lemma~\ref{hstep}(b) with $r-1$ in place of $r$ gives
\begin{equation}
 H_{r-1}(x_i)=\sum_{v=0}^{q-1}H_{r-2}(y_{i,v})+H_1(x_i).
 \label{qpoint3}
\end{equation}
Moreover, $\sum_{v=0}^{q-1}(2v+1-q)=0$, and therefore
\begin{equation}
 \sum_{v=0}^{q-1}(2v+1-q)\,y_{i,v}=\sum_{v=q-b_i}^{q-1}(2v+1-q)
 =b_i(q-b_i)=H_1(x_i).
 \label{qpoint4}
\end{equation}
If $x_i\le x_j$, then $y_{i,v}\le y_{j,v}$ for every $v$, so that
$\sum_{v=0}^{q-1}|y_{i,v}-y_{j,v}|=x_j-x_i$. Thus
$\sum_{v=0}^{q-1}|y_{i,v}-y_{j,v}|=|x_i-x_j|$ for all $i$ and $j$. We
apply the induction hypothesis to the integers $y_{1,v},\ldots,y_{q,v}$,
whose sum is $Y_v$, and add the resulting inequalities for $0\le v\le
q-1$. By \eqref{qpoint3} and the equation for the absolute values,
\begin{equation}
 \sum_{v=0}^{q-1}H_{r-1}(Y_v)
 \le\sum_{i=1}^{q}H_{r-1}(x_i)-\sum_{i=1}^{q}H_1(x_i)
 +\sum_{1\le i<j\le q}|x_i-x_j|.
 \label{qpoint5}
\end{equation}
Since $\sum_{v=0}^{q-1}Y_v=s$, Lemma~\ref{convjen} gives
\begin{equation}
 H_r(s)\le\sum_{v=0}^{q-1}H_{r-1}(Y_v)+\sum_{v=0}^{q-1}Y_v^2-\frac{s^2}q
 +\frac{q-1}q\,H_1(s).
 \label{qpoint6}
\end{equation}
By \eqref{qpoint5} and \eqref{qpoint6}, inequality \eqref{qpoint2} follows
once we show that
\begin{equation}
 q\sum_{v=0}^{q-1}Y_v^2-s^2+(q-1)H_1(s)\le q\sum_{i=1}^{q}H_1(x_i).
 \label{qpoint7}
\end{equation}

Put $M_v=Y_v-v$ for every integer $v$. We have $y_{i,v+1}\ge y_{i,v}$ and
$y_{i,v+q}=y_{i,v}+1$, so $Y_{v+1}\ge Y_v$ and $Y_{v+q}=Y_v+q$. Hence
$M_{v+1}\ge M_v-1$ and $M_{v+q}=M_v$ for every integer $v$. We claim that
$M_0,\ldots,M_{q-1}$ satisfy the hypothesis of Lemma~\ref{gapfree}. Let
$h$ be an integer that is at least the smallest and less than the largest
of these numbers. Since the sequence $(M_v)$ has period $q$, there are
integers $u<u'$ such that $M_u$ is the largest and $M_{u'}$ is the
smallest of the numbers $M_0,\ldots,M_{q-1}$. Let $w$ be the smallest
integer greater than $u$ with $M_w\le h$. Then $w\le u'$ and $M_{w-1}>h$,
so $M_w\ge M_{w-1}-1\ge h$. Thus $M_w=h$, and by the period $q$, the
number $h$ is one of $M_0,\ldots,M_{q-1}$. This proves the claim.

Since $Y_{v'}-Y_v=(M_{v'}-M_v)+(v'-v)$, we have
$(Y_{v'}-Y_v)^2=(M_{v'}-M_v)^2+2(v'-v)(Y_{v'}-Y_v)-(v'-v)^2$. We add
these equations for $0\le v<v'\le q-1$. On the left we obtain
$q\sum_{v=0}^{q-1}Y_v^2-s^2$. By \eqref{qpoint4} summed over $i$,
\[
 2\sum_{0\le v<v'\le q-1}(v'-v)(Y_{v'}-Y_v)
 =2q\sum_{v=0}^{q-1}vY_v-q(q-1)s
 =q\sum_{v=0}^{q-1}(2v+1-q)Y_v=q\sum_{i=1}^{q}H_1(x_i),
\]
and $\sum_{0\le v<v'\le q-1}(v'-v)^2=q^2(q^2-1)/12$. Therefore
\[
 q\sum_{v=0}^{q-1}Y_v^2-s^2
 =\sum_{0\le v<v'\le q-1}(M_{v'}-M_v)^2+q\sum_{i=1}^{q}H_1(x_i)
 -\frac{q^2(q^2-1)}{12}.
\]
Since $s=\sum_{v=0}^{q-1}M_v+q(q-1)/2$, Lemma~\ref{gapfree} shows that
$\sum_{0\le v<v'\le q-1}(M_{v'}-M_v)^2+(q-1)H_1(s)\le q^2(q^2-1)/12$,
and \eqref{qpoint7} follows.
\end{proof}

For points in $[0,1]$, Theorem~\ref{qpoint} is equivalent to an inequality
for digit sums that Mohanty et~al.~\cite[Theorem~5.1]{Mohanty2023} proved:
if $0\le n_1\le\cdots\le n_q$ are integers, then
\[
 \sum_{i=1}^{q}S_q(n_i)+\sum_{i=1}^{q-1}(q-i)\,n_i
 \le S_q\Bigl(\sum_{i=1}^{q}n_i\Bigr).
\]
Indeed, let $N\ge1$ be an integer with $n_q\le q^N$, and put
$x_i=n_i/q^N$. Then the mean of $x_1,\ldots,x_q$ is
$(n_1+\cdots+n_q)/q^{N+1}$, and $\sum_{1\le i<j\le
q}(n_j-n_i)=\sum_{i=1}^{q}(2i-q-1)n_i$. Hence, by Theorem~\ref{sampthm} at
the levels $N$ and $N+1$, the right side of \eqref{qpoint1} minus its left
side, multiplied by $q^{N+1}$, is twice the right side of the inequality
of Mohanty et~al.\ minus its left side. Since \eqref{qpoint1} is symmetric
in $x_1,\ldots,x_q$, both of its sides are continuous, and the points
$n/q^N$ are dense in $[0,1]$, inequality \eqref{qpoint1} for all
$x_1,\ldots,x_q\in[0,1]$ is equivalent to their inequality. Our proof of
Theorem~\ref{qpoint} does not use their inequality, and it covers all real
numbers. Lev's proofs of the cases $q=2$ and $q=3$ of Theorem~\ref{qpoint}
also use induction on the exponent $r$ of the denominators $q^r$, and they
treat the possible remainders of the points modulo $q$ one after the
other~\cite[Section~2]{Lev2015}.

When only two different points occur among $x_1,\ldots,x_q$,
Theorem~\ref{qpoint} reads as follows.

\begin{corollary}
\label{qsect}
For every integer $1\le k\le q-1$ and all real $x$ and $y$,
\begin{equation}
 T_q\Bigl(\frac{kx+(q-k)y}q\Bigr)
 \le\frac kq\,T_q(x)+\frac{q-k}q\,T_q(y)+\frac{k(q-k)}q\,|x-y|.
 \label{qsect1}
\end{equation}
\end{corollary}

\begin{proof}
We apply \eqref{qpoint1} with $x_1=\cdots=x_k=x$ and
$x_{k+1}=\cdots=x_q=y$. Exactly $k(q-k)$ of the pairs $i<j$ have
$|x_i-x_j|=|x-y|$, and for the other pairs $x_i=x_j$.
\end{proof}

The numbers $k(q-k)/q$ in \eqref{qsect1} are the values $\tau_q(k/q)$,
which equal $T_q(k/q)$ by \eqref{fe}. The next theorem extends
\eqref{qsect1} to all points of the interval between $x$ and $y$.

\begin{theorem}
\label{lamq}
For all real $x$, $y$, and $0\le\lambda\le1$,
\begin{equation}
 T_q\bigl(\lambda x+(1-\lambda)y\bigr)
 \le\lambda T_q(x)+(1-\lambda)T_q(y)+T_q(\lambda)|x-y|.
 \label{lamq1}
\end{equation}
\end{theorem}

\begin{proof}
We prove by induction on $n\ge0$ that \eqref{lamq1} holds for all real
$x$ and $y$ and for every $\lambda$ of the form $j/q^n$ with an integer
$0\le j\le q^n$. These $\lambda$ are dense in $[0,1]$, and both sides of
\eqref{lamq1} are continuous in $\lambda$, so the theorem follows.

For $\lambda=0$ and $\lambda=1$, which covers $n=0$, both sides of
\eqref{lamq1} are equal, because $T_q(0)=T_q(1)=0$. Let $n\ge1$ and $\lambda=j/q^n$
with $0\le j<q^n$. Put $k=\lfloor q\lambda\rfloor$ and $\mu=q\lambda-k$.
Then $0\le k\le q-1$, and $\mu=(j-kq^{n-1})/q^{n-1}$ lies in $[0,1)$ and
is an integer divided by $q^{n-1}$. For $i=k$ and $i=k+1$ put
$p_i=(ix+(q-i)y)/q$. Since $\lambda=(k+\mu)/q$, we have
\[
 \lambda x+(1-\lambda)y=\mu p_{k+1}+(1-\mu)p_k
 ~\text{and}~|p_{k+1}-p_k|=\frac{|x-y|}q.
\]
The induction hypothesis for $\mu$ and the points $p_{k+1}$ and $p_k$
gives
\begin{equation}
 T_q\bigl(\lambda x+(1-\lambda)y\bigr)
 \le\mu T_q(p_{k+1})+(1-\mu)T_q(p_k)+\frac{T_q(\mu)}q\,|x-y|.
 \label{lamq2}
\end{equation}
For $i=k$ and $i=k+1$ we have
\begin{equation}
 T_q(p_i)\le\frac iq\,T_q(x)+\frac{q-i}q\,T_q(y)+\tau_q\Bigl(\frac iq\Bigr)|x-y|.
 \label{lamq3}
\end{equation}
Indeed, for $1\le i\le q-1$ this is \eqref{qsect1}, because
$\tau_q(i/q)=i(q-i)/q$, and for $i=0$ and $i=q$ both sides are equal,
because $p_0=y$, $p_q=x$, and $\tau_q(0)=\tau_q(1)=0$. We insert
\eqref{lamq3} into \eqref{lamq2}. Since $\lambda=(\mu(k+1)+(1-\mu)k)/q$,
the coefficients of $T_q(x)$ and $T_q(y)$ become $\lambda$ and
$1-\lambda$. Since $\tau_q$ is linear on $[k/q,(k+1)/q]$, we have
$\mu\tau_q((k+1)/q)+(1-\mu)\tau_q(k/q)=\tau_q(\lambda)$, and
$T_q(\mu)=T_q(q\lambda)$ by the period $1$. Hence the coefficient of
$|x-y|$ becomes $\tau_q(\lambda)+\frac1qT_q(q\lambda)$, which is
$T_q(\lambda)$ by \eqref{fe}.
\end{proof}

For $q=3$, Theorem~\ref{lamq} is a theorem of
Lev~\cite[Theorem~1.5]{Lev2015}, which he derived from the inequality for
three points displayed before Theorem~\ref{qpoint} and a general theorem
on functions that satisfy an inequality of this
form~\cite[Theorems~1.9 and~1.12]{Lev2015}. For $\lambda=1/2$ we have
$T_q(1/2)=\lfloor(q+1)/2\rfloor/2$ by Proposition~\ref{incr}(b), and
\eqref{lamq1} is the version of the inequality of Boros for $T_q$ that
Allaart derived.
Allaart~\cite[Theorem~11]{Allaart2014} stated it for $x<y$ and for
$h_q(x)=\sum_{j=0}^{\infty}q^{-j}g_q(q^jx)$, with the constant
$\frac14\lfloor(q+1)/2\rfloor$. Here $g_q$ has period $1$ and
$g_q(x)=\int_0^x\bigl((q-1)/2-\lfloor qt\rfloor\bigr)\,dt$ for
$0\le x<1$. Thus $g_q$ is zero at the integers and has slope
$(q-1-2k)/2$ on $[k/q,(k+1)/q]$, so $g_q=\tau_q/2$ and $h_q=T_q/2$.

Figure~\ref{convexpic} illustrates Theorem~\ref{lamq} for $q=3$, $4$,
and $5$ on the interval $[x,y]=[1/q^2,(q-1)/q]$, which is not an interval
of level $n$ for any $n$. Each panel shows, as functions of the point
$z=\lambda x+(1-\lambda)y$ of $[x,y]$, the value $T_q(z)$, the linear
interpolation $\lambda T_q(x)+(1-\lambda)T_q(y)$, and the right side of
\eqref{lamq1}. All three are computed exactly at the points $x+j(y-x)/q^n$ with
$0\le j\le q^n$, where $n$ is the smallest integer with $q^n\ge2000$. The
dots mark the endpoints, at which both sides of \eqref{lamq1} are equal.

\begin{figure}[t]
\centering
\includegraphics{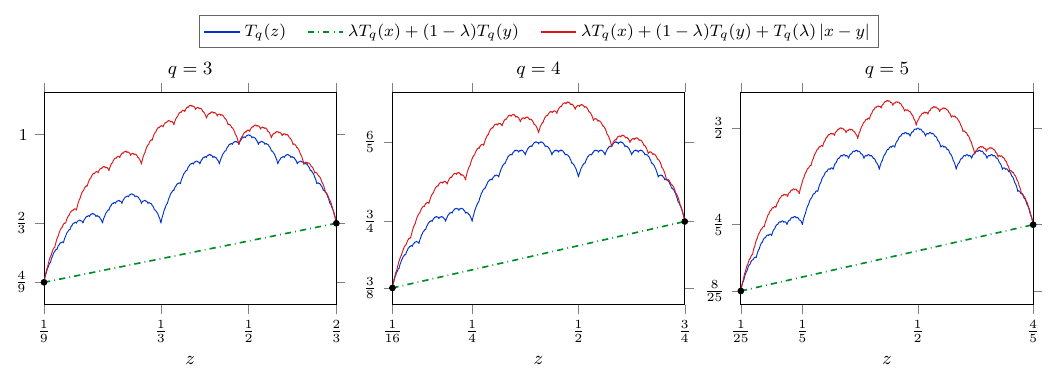}
\caption{Theorem~\ref{lamq} for $q=3$, $4$, and $5$ on the interval
$[1/q^2,(q-1)/q]$.}
\label{convexpic}
\end{figure}

\begin{proposition}
\label{convexeq}
\begin{enumerate}
\item[(a)] If $[x,y]$ is an interval $[m/q^n,(m+1)/q^n]$ with integers
 $n\ge0$ and $m$, then equality holds in \eqref{lamq1} for every
 $0\le\lambda\le1$.
\Needspace{8\baselineskip}
\item[(b)] For all real $u$ and $h\ge0$,
 \begin{equation}
  -2T_q(h)\le2T_q(u)-T_q(u-h)-T_q(u+h)\le2\Bigl\lfloor\frac{q+1}2\Bigr\rfloor h.
  \label{midconv}
 \end{equation}
 The left inequality is an equality for every $h$ when $u$ is an
 integer, and the right inequality is an equality when $[u-h,u+h]$ is an
 interval $[m/q^n,(m+1)/q^n]$ with integers $n\ge0$ and $m$. The constant
 $2\lfloor(q+1)/2\rfloor$, which is $4T_q(1/2)$, is the smallest possible.
 For $q=3$ it is $4$.
\item[(c)] For integers $n\ge1$ and $0\le a\le c\le b\le q^n$ with $a<b$,
 \begin{equation}
  \xi_n(c)\le\frac{(b-c)\xi_n(a)+(c-a)\xi_n(b)}{b-a}
  +(b-a)\,T_q\Bigl(\frac{b-c}{b-a}\Bigr),
  \label{cubeconv}
 \end{equation}
 with equality when $b-a=q^k$ and $q^k$ divides $a$. The last term is at
 most $b-a$ times the maximum of $T_q$, which is given in
 Theorem~\ref{gmax} and is $1$ for $q=3$. In terms of the summatory
 function $S_q$, inequality \eqref{cubeconv} reads
 \begin{equation}
  (b-c)S_q(a)+(c-a)S_q(b)-(b-a)S_q(c)
  \le\frac{(b-a)^2}2\,T_q\Bigl(\frac{b-c}{b-a}\Bigr).
  \label{digconv}
 \end{equation}
\end{enumerate}
\end{proposition}

\begin{proof}
(a) Let $x=m/q^n$ and $y=(m+1)/q^n$. With $t=1-\lambda$ we have
$\lambda x+(1-\lambda)y=(m+t)/q^n$, and \eqref{convlin1} gives
\[
 T_q\bigl(\lambda x+(1-\lambda)y\bigr)-\lambda T_q(x)-(1-\lambda)T_q(y)
 =q^{-n}T_q(1-\lambda)=T_q(\lambda)|x-y|.
\]

(b) By Proposition~\ref{incr}(b), $4T_q(1/2)$ is $q+1$ for odd $q$ and
$q$ for even $q$, so $4T_q(1/2)=2\lfloor(q+1)/2\rfloor$. The right
inequality is \eqref{lamq1} for the points $u-h$ and $u+h$ and
$\lambda=1/2$, multiplied by $2$. In the proof of Proposition~\ref{incr}
we saw that $T_q(v+w)\le T_q(v)+T_q(w)$ for all real $v,w$. With $v=u$ and
$w=\pm h$ this gives $T_q(u\pm h)\le T_q(u)+T_q(h)$, since
$T_q(-h)=T_q(h)$, and the sum of these two inequalities is the left
inequality. If $u$ is an integer, then $T_q(u)=0$ and $T_q(u\pm h)=T_q(h)$,
so the left inequality is an equality. If $[u-h,u+h]$ is an interval
$[m/q^n,(m+1)/q^n]$, then (a) with $\lambda=1/2$ gives
$T_q(u)-(T_q(u-h)+T_q(u+h))/2=2hT_q(1/2)$, so the right inequality is an
equality. For $u=h=1/2$ this shows that no smaller constant works.

(c) By \eqref{samp}, $\xi_n(j)=q^nT_q(j/q^n)$ for $0\le j\le q^n$. We apply
\eqref{lamq1} with $x=a/q^n$, $y=b/q^n$, and $\lambda=(b-c)/(b-a)$, so
that $\lambda x+(1-\lambda)y=c/q^n$ and $|x-y|=(b-a)/q^n$. After
multiplication by $q^n$ we obtain \eqref{cubeconv}. If $b-a=q^k$ and $q^k$
divides $a$, then $k\le n$ and $[x,y]$ is an interval of level $n-k$,
and (a) gives equality. If we insert $\xi_n(j)=(q-1)nj-2S_q(j)$ from
\eqref{xi} into \eqref{cubeconv}, the terms with the factor $n$
cancel, because $(b-c)a+(c-a)b=(b-a)c$, and \eqref{digconv} remains.
\end{proof}

For $a=m-k$, $b=m+k$, and $c=m$, inequality \eqref{digconv} reads
$S_q(m+k)+S_q(m-k)-2S_q(m)\le\lfloor(q+1)/2\rfloor k$, and \eqref{cubeconv}
reads $\xi_n(m)\le(\xi_n(m-k)+\xi_n(m+k))/2+\lfloor(q+1)/2\rfloor k$.
The first inequality is Allaart's digital
inequality~\cite[Theorem~3]{Allaart2014}, and Allaart showed that it is
strict for $k\ge1$ when $q$ is odd. Hence for odd $q$ the right
inequality in \eqref{midconv} is strict when $u$ and $h>0$ are $q$-adic
rationals with $0\le u-h$ and $u+h\le1$: if $u=m/q^n$ and $h=k/q^n$ with
integers $k\ge1$, $m-k\ge0$, and $m+k\le q^n$, then by
Theorem~\ref{sampthm}
\[
 2\Bigl\lfloor\frac{q+1}2\Bigr\rfloor h-\bigl(2T_q(u)-T_q(u-h)-T_q(u+h)\bigr)
 =\frac2{q^n}\Bigl(\Bigl\lfloor\frac{q+1}2\Bigr\rfloor k-S_q(m+k)-S_q(m-k)+2S_q(m)\Bigr)>0.
\]
When $q$ is even, the midpoint $u$ of an interval of level $n$ and half
its length $h$ are $q$-adic rationals, and by
Proposition~\ref{convexeq}(b) the right inequality in \eqref{midconv} is
an equality for them. The following example gives such a case.

\begin{example}
\label{ex:convfour}
Let $q=4$. By \eqref{tau}, $\tau_q(x)=x+1/2$ for $1/4\le x\le1/2$, and
\eqref{fe} gives $T_q(1/4)=\tau_q(1/4)=3/4$, $T_q(1/2)=\tau_q(1/2)=1$, and
$T_q(3/8)=\tau_q(3/8)+\frac14T_q(3/2)=\frac78+\frac14=\frac98$. Let $u=3/8$ and
$h=1/8$. Then $[u-h,u+h]=[1/4,1/2]$ is an interval of level $1$, and
$2T_q(u)-T_q(u-h)-T_q(u+h)=\frac94-\frac34-1=\frac12=4h$, in agreement with
Proposition~\ref{convexeq}(b), whose constant is $4$ for $q=4$. The
interval $[1/4,3/4]$ has the form $[m/4^n,(m+1)/4^n]$ for no integers $n$
and $m$. Its midpoint is $1/2$, and $T_q(3/4)=T_q(1/4)=3/4$. Here
Theorem~\ref{lamq} with $\lambda=1/2$ gives the bound
$T_q(1/2)\le\frac34+T_q(1/2)\cdot\frac12=\frac54$, which is larger than the
value $T_q(1/2)=1$.
\end{example}

By Proposition~\ref{convexeq}(a), on an interval $[x,y]$ of level $n$ the
difference between $T_q(\lambda x+(1-\lambda)y)$ and the linear
interpolation $\lambda T_q(x)+(1-\lambda)T_q(y)$ is $q^{-n}T_q(\lambda)$.
If the slope of $[x,y]$ is $0$ (see Definition~\ref{intslope}), the linear
interpolation is constant, so the maximum of $T_q$ on $[x,y]$ is
$T_q(x)+q^{-n}\max_{0\le t\le1}T_q(t)$, which is the value in
Theorem~\ref{cellmax}, and it is attained at the points
$\lambda x+(1-\lambda)y$ for which $T_q(\lambda)$ is the maximum of $T_q$.
By Theorem~\ref{gmax}, for odd $q$ the midpoint is the only such point.
If the slope of
$[x,y]$ is not $0$, Theorem~\ref{cellmax} gives both the value of the
maximum of $T_q$ on the interval and the points at which it is attained,
while Theorem~\ref{lamq} gives only the bound
$\max\{T_q(x),T_q(y)\}+q^{-n}\max_{0\le t\le1}T_q(t)$.

\begin{corollary}
\label{modulus}
For every $0\le\lambda\le1$,
\[
 \sup_{x<y}\frac{T_q(\lambda x+(1-\lambda)y)-\lambda T_q(x)-(1-\lambda)T_q(y)}{y-x}
 =T_q(\lambda),
\]
where the supremum is over all real $x<y$, and it is attained on every
interval $[x,y]$ of level $n$, for every $n\ge0$.
\end{corollary}

\begin{proof}
Theorem~\ref{lamq} gives the inequality $\le$, and
Proposition~\ref{convexeq}(a) gives equality for every interval of level
$n$.
\end{proof}

Boros proved his inequality because of its role in the theory of
approximately convex functions. H\'azy and
P\'ales~\cite[Theorem~4]{HazyPales2004} proved the following theorem for
functions on an open convex subset of a normed space, with an additive
constant in the hypothesis and twice this constant in the conclusion. We
state it for functions on $\R$ and with this constant equal to $0$.
Their error function is the bounded function
on $[0,1]$ whose value at $\lambda$ is $\lambda$ plus half its value at
$2\lambda$ for $0\le\lambda\le1/2$, and $1-\lambda$ plus half its value
at $2\lambda-1$ for $1/2\le\lambda\le1$ (their Proposition~1). Takagi's
function $T_2$ has this property, by its definition and its period $1$,
and only one bounded function has it, because the difference $g$ of two
such functions satisfies $|g(\lambda)|\le\frac12\sup|g|$ for every
$\lambda$; hence their error function is $T_2$.

\begin{theorem}[H\'azy and P\'ales {\cite[Theorem~4]{HazyPales2004}}]
\label{hp}
Let $\varepsilon\ge0$, and let $f\colon\R\to\R$ be bounded above on some
open interval and satisfy $f((x+y)/2)\le(f(x)+f(y))/2+\varepsilon|x-y|$
for all real $x$ and $y$. Then
\[
 f(\lambda x+(1-\lambda)y)\le\lambda f(x)+(1-\lambda)f(y)
 +2\varepsilon T_2(\lambda)|x-y|
\]
for all real $x$, $y$, and $0\le\lambda\le1$.
\end{theorem}

Boros's inequality says that $T_2$ itself satisfies the hypothesis of
Theorem~\ref{hp} with $\varepsilon=1/2$. P\'ales observed that this
makes the factor $T_2(\lambda)$ in the conclusion
optimal~\cite[Section~1]{Boros2008}. The conclusion of Theorem~\ref{hp}
for $f=T_2$ and $\varepsilon=1/2$ is inequality \eqref{lamq1} for $q=2$,
as Lev noted~\cite[equation~(1.10)]{Lev2015}. In every base, $T_q$
satisfies the hypothesis with $\varepsilon=T_q(1/2)$, so Theorem~\ref{hp}
gives the error term $2T_q(1/2)T_2(\lambda)|x-y|$, while by Corollary~\ref{modulus} the error
term $T_q(\lambda)|x-y|$ of Theorem~\ref{lamq} cannot be improved. Comparing the two error terms gives an
inequality between the Delange function of base $q$ and Takagi's
function.

\begin{proposition}
\label{twobases}
For every real $x$, $T_q(x)\le\lfloor(q+1)/2\rfloor T_2(x)$, with equality
at the integers and at the points $k+1/2$ with $k\in\Z$. For $q=3$ this
reads $T_q(x)\le2T_2(x)$.
\end{proposition}

\begin{proof}
The function $T_q$ is bounded, and by Theorem~\ref{lamq} with
$\lambda=1/2$ it satisfies the hypothesis of Theorem~\ref{hp} with
$\varepsilon=T_q(1/2)$. Since $T_q(0)=T_q(1)=0$, Theorem~\ref{hp} with $x=1$ and
$y=0$ gives $T_q(\lambda)\le2T_q(1/2)T_2(\lambda)$ for $0\le\lambda\le1$, and
$2T_q(1/2)=\lfloor(q+1)/2\rfloor$ by Proposition~\ref{incr}(b). Both
functions have period $1$, so $T_q\le\lfloor(q+1)/2\rfloor T_2$ on $\R$. At
the integers both sides are $0$, and $T_2(1/2)=1/2$, so both sides are
equal to $T_q(1/2)$ at the points $k+1/2$.
\end{proof}

Figure~\ref{basespic} shows the cases $q=3$, $4$, and $5$ of
Proposition~\ref{twobases}, in which the factor $\lfloor(q+1)/2\rfloor$ is
$2$, $2$, and $3$. The graph of $T_q$ is drawn from its exact values at
the points $i/q^n$ with $0\le i\le q^n$, where $n$ is the smallest integer
with $q^n\ge2000$, and the graph of $\lfloor(q+1)/2\rfloor T_2$ from the
exact values of $T_2$ at the points $i/2^{11}$ with $0\le i\le2^{11}$. The
dots mark the points $0$, $1/2$, and $1$, at which equality holds.

\begin{figure}[H]
\centering
\includegraphics{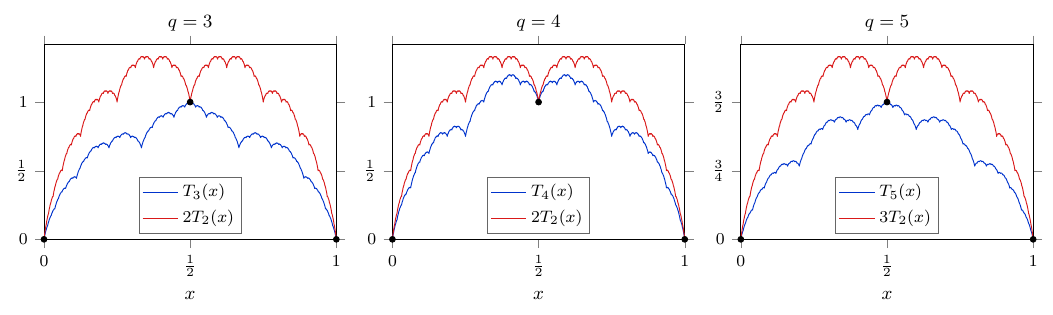}
\caption{Proposition~\ref{twobases} for $q=3$, $4$, and $5$.}
\label{basespic}
\end{figure}

\section{Functional equations}
\label{sec:feq}

In this section we show that each of the $q$ equations \eqref{feqr} below
determines $T_q$ among the functions on $\R$ that are bounded outside a
bounded interval. By a theorem of Barnsley, the $q$ equations \eqref{threeeqs} below together
determine $T_q$ among the continuous functions on $[0,1]$. Every collection
of $q-1$ of these equations has a continuous solution different from $T_q$. We also show that $T_q(x)$ is a
constant multiple of the derivative, with respect to the digit
probabilities, of the probability that a random number with independent
base $q$ digits is at most $x$. For
$q=2$ this is a theorem of Hata and Yamaguti, and for $q=3$ it is a result
of Okada, Sekiguchi, and Shiota.

For $r\in\{0,1,\ldots,q-1\}$ and a function $f\colon\R\to\R$ we consider
the equation
\begin{equation}
 f\Bigl(\frac{x+r}q\Bigr)=\tau_q\Bigl(\frac{x+r}q\Bigr)+\frac1qf(x)
 ~\text{for all real }x.
 \label{feqr}
\end{equation}
The function $T_q$ satisfies \eqref{feqr} for each $r$, because \eqref{fe}
at the point $(x+r)/q$ reads $T_q((x+r)/q)=\tau_q((x+r)/q)+T_q(x+r)/q$ and
$T_q(x+r)=T_q(x)$.

\begin{theorem}
\label{oneeq}
Let $r\in\{0,1,\ldots,q-1\}$, and let $f\colon\R\to\R$ satisfy
\eqref{feqr}. If $f$ is bounded outside some bounded interval, then
$f=T_q$.
\end{theorem}

\begin{proof}
Fix a real $y$, and put $y_0=y$ and $y_{k+1}=qy_k-r$ for $k\ge0$.
Equation \eqref{feqr} with $x=y_{k+1}$ gives
$f(y_k)=\tau_q(y_k)+f(y_{k+1})/q$ for every $k\ge0$, and these equations
with $k=0,\ldots,n-1$ give
\[
 f(y)=\sum_{k=0}^{n-1}q^{-k}\tau_q(y_k)+q^{-n}f(y_n),
 ~n\ge1.
\]
Put $c=r/(q-1)$. Since $y_{k+1}-c=q(y_k-c)$, we have
$y_k=q^ky-r(q^k-1)/(q-1)$, where $(q^k-1)/(q-1)=\sum_{i=0}^{k-1}q^i$
is an integer. Hence $\tau_q(y_k)=\tau_q(q^ky)$ by periodicity,
and the sum on the right is $F_n(y)$. If $y\ne c$, then
$|y_n|=|q^n(y-c)+c|\to\infty$, so $f(y_n)$
is bounded for all large $n$, and letting $n\to\infty$ gives $f(y)=T_q(y)$.
If $y=c$, then $y_1=y$, and the equation with $k=0$ gives
$f(c)=\frac q{q-1}\tau_q(c)$. Since $T_q$ also satisfies \eqref{feqr}, the
same holds for $T_q$, so $f(c)=T_q(c)$.
\end{proof}

\begin{example}
\label{ex:sym}
Theorem~\ref{oneeq} with $r=1$ gives the symmetry $T_q(1-x)=T_q(x)$. Let
$f(x)=T_q(1-x)$. Since $\tau_q$ is even and has period $1$, we have
$\tau_q((q-1-x)/q)=\tau_q((x+1)/q)$, and since $T_q$ has period $1$, we have
$T_q(q-1-x)=f(x)$. Hence \eqref{fe} at the point $(q-1-x)/q=1-(x+1)/q$
reads
\[
 f\Bigl(\frac{x+1}q\Bigr)=T_q\Bigl(\frac{q-1-x}q\Bigr)
 =\tau_q\Bigl(\frac{x+1}q\Bigr)+\frac1qf(x),
\]
which is \eqref{feqr} with $r=1$. As $f$ is bounded, the theorem gives
$f=T_q$. The hypothesis on $f$ cannot be omitted: for every real $c$ and
every $r$, the function $T_q(x)+c\bigl(x-r/(q-1)\bigr)$ satisfies
\eqref{feqr}, because $(x+r)/q-r/(q-1)=\bigl(x-r/(q-1)\bigr)/q$, and for
$c\ne0$ it is unbounded outside every bounded interval.
\end{example}

Restricted to $[0,1]$, the $q$ equations \eqref{feqr} involve only the
values of $f$ on $[0,1]$. By \eqref{tau}, for $b\in\{0,1,\ldots,q-1\}$
and $0\le y\le1$ we have
$\tau_q((b+y)/q)=\bigl(b(q-b)+(q-1-2b)\,y\bigr)/q$, so the equations become
\begin{equation}
 f\Bigl(\frac{b+y}q\Bigr)=\frac{b(q-b)+(q-1-2b)\,y+f(y)}q,
 ~0\le y\le1,~b=0,1,\ldots,q-1,
 \label{threeeqs}
\end{equation}
which are the equations \eqref{branches} with $f$ in place of $T_q$. For
$q=3$ they are $f(y/3)=(2y+f(y))/3$, $f((1+y)/3)=(2+f(y))/3$, and
$f((2+y)/3)=(2-2y+f(y))/3$.

For continuous functions, the $q$ equations \eqref{threeeqs} together
determine $T_q$ on $[0,1]$. Indeed, $T_q$ is a fractal interpolation
function in the sense of Barnsley~\cite[Example~1]{Barnsley1986}, with the
interpolation points $(k/q,k(q-k)/q)$, $0\le k\le q$, and all vertical
scaling factors equal to $1/q$. The graph of a continuous solution of \eqref{threeeqs} is the union of its
images under the $q$ maps
$(y,v)\mapsto\bigl((b+y)/q,(b(q-b)+(q-1-2b)y+v)/q\bigr)$ with
$0\le b\le q-1$. Let $A>0$ be so large that these
maps send the rectangle $[0,1]\times[-A,A]$ into itself and that the
rectangle contains this graph and the graph of $T_q$ over $[0,1]$. By a
theorem of Barnsley~\cite[Theorem~1]{Barnsley1986}, only one nonempty
closed subset of the rectangle has this property. Since $T_q$ is a
continuous solution of \eqref{threeeqs}, the two graphs are equal, and
the solution is $T_q$ on $[0,1]$.

\begin{proposition}
\label{threeeq}
For each of the $q$ equations \eqref{threeeqs} there is a continuous
function on $[0,1]$ that satisfies the other $q-1$ equations and differs
from $T_q$.
\end{proposition}

\begin{proof}
Fix $a\in\{0,1,\ldots,q-1\}$. We construct a continuous function
$\gamma$ on $[0,1]$ that is not identically zero and satisfies
\begin{equation}
 \gamma\Bigl(\frac{b+y}q\Bigr)=\frac1q\gamma(y),~0\le y\le1,
 \label{homog}
\end{equation}
for the $q-1$ values $b\in\{0,1,\ldots,q-1\}$ with $b\ne a$. Then
$f=T_q+\gamma$ satisfies the $q-1$ equations \eqref{threeeqs} with these
$b$, because they are linear in $f$ and $T_q$ satisfies them. Since
$\gamma$ is not identically zero, $f\ne T_q$.

Let $\rho(x)=\max\{0,\,1/(2q)-|x-(2a+1)/(2q)|\}$, the function that is
zero outside the interval $[a/q,(a+1)/q]$ and rises linearly to the
value $1/(2q)$ at its midpoint. For $x\in[0,1)$ let
$\xi_1,\xi_2,\ldots$ be the digit sequence of $x$, where we take the
terminating one when $x$ is a $q$-adic rational, let $j(x)$ be the
smallest $j\ge0$ with $\xi_{j+1}=a$, and put
\[
 \gamma(x)=q^{-j(x)}\rho\bigl(q^{j(x)}x-\lfloor q^{j(x)}x\rfloor\bigr)
\]
if $j(x)$ exists. Put $\gamma(x)=0$ if no digit of $x$ equals $a$, and put
$\gamma(1)=0$. Thus $\gamma$ is the sum of the scaled copies
$q^{-j}\rho(q^jx-m)$ of $\rho$ on the open intervals
$\bigl(m/q^j+a/q^{j+1},\,m/q^j+(a+1)/q^{j+1}\bigr)$. There is one copy for
each $j\ge0$ and each integer $0\le m<q^j$ for which none of the $j$ base
$q$ digits of $m$ is $a$. The function $\gamma$ is zero outside these intervals.

We claim that $\gamma$ is continuous. The open intervals are pairwise
disjoint, and each copy of $\rho$ is continuous, has slopes $\pm1$, and is
zero at the endpoints of its interval. Hence $\gamma$ is continuous at
every point of these intervals. Now let $x_0$ be a point outside all of
them, so that $\gamma(x_0)=0$, and let $J\ge0$. Then
$\gamma(y)\le\max\{|y-x_0|,q^{-J}/(2q)\}$ for all $y$ close enough to
$x_0$, because the copies on the intervals with $j\ge J$ are at most
$q^{-J}/(2q)$ and the finitely many intervals with $j<J$ that do not have
$x_0$ as an endpoint stay at a positive distance from $x_0$. Since $J$ is
arbitrary, $\gamma$ is continuous at $x_0$. Moreover,
$\gamma\bigl((2a+1)/(2q)\bigr)=\rho\bigl((2a+1)/(2q)\bigr)=1/(2q)$, so
$\gamma$ is not identically zero.

Let $b\ne a$ and $0\le y<1$. The digits of $x=(b+y)/q$ are $b$
followed by the digits of $y$. If no digit of $y$ equals
$a$, then no digit of $x$ equals $a$ and $\gamma(x)=0=\gamma(y)/q$. Otherwise
$j(x)=j(y)+1$, and $q^{j(x)}x-\lfloor q^{j(x)}x\rfloor=q^{j(y)}y-\lfloor q^{j(y)}y\rfloor$
because $qx=y+b$ with an integer $b$, so $\gamma(x)=\gamma(y)/q$. Now let
$y=1$, so that $\gamma(y)=0$ and $x=(b+1)/q$. We show that $\gamma(x)=0$.
If $b=q-1$, then $x=1$. Suppose now that $b\le q-2$. Then the digits of
$x$ are $b+1$ followed by zeros. If $b+1=a$, then $j(x)=0$ and
$\rho(x)=0$, since $x=a/q$ is an endpoint of the interval
$[a/q,(a+1)/q]$. If $b+1\ne a$ and $a\ne0$, then no digit of $x$ equals
$a$. If $b+1\ne a$ and $a=0$, then $j(x)=1$ and
$\gamma(x)=\rho(0)/q=0$. Thus \eqref{homog} holds for all $0\le y\le1$.
\end{proof}

Figure~\ref{twoeqpic} shows the graphs of $T_q$ and $T_q+\gamma$ for
$(q,a)=(3,1)$, $(4,1)$, and $(5,2)$, drawn through their exact values at
the multiples of $3^{-7}/2$, $4^{-6}/2$, and $5^{-5}/2$, with the interval
$[a/q,(a+1)/q]$ shaded. The two graphs agree outside the open intervals on
which $\gamma$ is positive. On $[a/q,(a+1)/q]$ their difference is $\rho$, whose largest value is
$1/(2q)$. On the interval of a scaled copy $q^{-j}\rho(q^jx-m)$, their
difference is that copy, whose largest value is $q^{-j}/(2q)$. Only the
copies with small $j$ are visible. In
every base, the function $T_q+\gamma$ does not satisfy the equation in
\eqref{threeeqs} with $b=a$. Indeed, since $T_q$
satisfies this equation, the equation for $T_q+\gamma$ would give
$\gamma((2a+1)/(2q))=\gamma(1/2)/q$ at $y=1/2$,
while $\gamma((2a+1)/(2q))=1/(2q)$ and $\gamma(1/2)/q\le1/(2q^2)$,
because $\gamma\le1/(2q)$ everywhere.

Adding the $q$ equations \eqref{feqr} gives the equation \eqref{replic1}
below. By Theorem~\ref{oneeq}, each equation \eqref{feqr} alone
determines $T_q$ among the bounded functions on $\R$. By
Proposition~\ref{replic}, equation \eqref{replic1} has the further
continuous bounded solutions $T_q+c\chi$ with real constants $c\ne0$,
where $\chi$ is a cosine series with the frequencies $1,q,q^2,\ldots$.

\begin{figure}[H]
\centering
\includegraphics{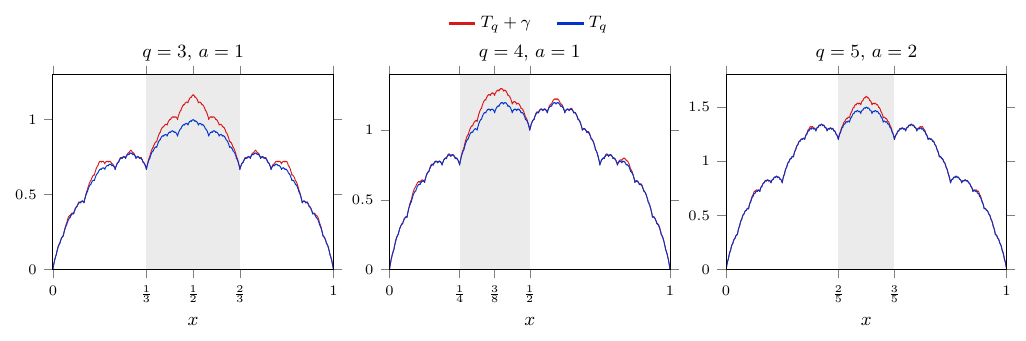}
\caption{Proposition~\ref{threeeq} for $(q,a)=(3,1)$, $(4,1)$, and
$(5,2)$: the function $T_q+\gamma$ of its proof satisfies the equations
in \eqref{threeeqs} with $b\ne a$, as $T_q$ does, but differs from $T_q$.}
\label{twoeqpic}
\end{figure}

\begin{proposition}
\label{replic}
For every real $x$,
\begin{equation}
 \sum_{r=0}^{q-1}T_q\Bigl(\frac{x+r}q\Bigr)=T_q(x)+\frac{q^2-1}6.
 \label{replic1}
\end{equation}
The function $\chi(x)=\sum_{k=0}^{\infty}q^{-k}\cos(2\pi q^kx)$ is
continuous on $\R$, bounded by $q/(q-1)$, and satisfies
$\sum_{r=0}^{q-1}\chi((x+r)/q)=\chi(x)$ for every real $x$. Hence
$T_q+c\chi$ satisfies \eqref{replic1} for every real constant $c$.
\end{proposition}

\begin{proof}
By \eqref{feqr}, $T_q((x+r)/q)=\tau_q((x+r)/q)+T_q(x)/q$ for
$r=0,1,\ldots,q-1$. If $0\le x\le1$, then
$\tau_q((x+r)/q)=\bigl(r(q-r)+(q-1-2r)\,x\bigr)/q$ by \eqref{tau}. Since
$\sum_{r=0}^{q-1}(q-1-2r)=0$ and
\[
 \sum_{r=0}^{q-1}r(q-r)=\frac{q^2(q-1)}2-\frac{(q-1)q(2q-1)}6
 =\frac{q(q^2-1)}6,
\]
the sum of the $q$ values $\tau_q((x+r)/q)$ is $(q^2-1)/6$. For general
$x$ write $x=k+x'$ with an integer $k$ and $0\le x'<1$, and let
$r'\in\{0,1,\ldots,q-1\}$ be the remainder of $k+r$ divided by $q$. Then
$(x+r)/q-(x'+r')/q$ is an integer, and $r'$ runs through
$0,1,\ldots,q-1$ as $r$ does. Since $\tau_q$ has period $1$, the sum of
the $q$ values $\tau_q((x+r)/q)$ is again $(q^2-1)/6$. Adding the $q$
equations gives \eqref{replic1}.

Since $|\cos|\le1$, the series for $\chi$ converges uniformly, so $\chi$
is continuous and $|\chi|\le\sum_{k=0}^{\infty}q^{-k}=q/(q-1)$. If
$k\ge1$ and $r\in\{0,1,\ldots,q-1\}$, the number $q^{k-1}r$ is an
integer, so $\cos(2\pi q^{k-1}(x+r))=\cos(2\pi q^{k-1}x)$. For $k=0$ the
sum of the $q$ numbers $\cos(2\pi(x+r)/q)$ is the real part of
$e^{2\pi ix/q}\sum_{r=0}^{q-1}e^{2\pi ir/q}$, which is zero because
$q\ge2$. Therefore
\[
 \sum_{r=0}^{q-1}\chi\Bigl(\frac{x+r}q\Bigr)
 =\sum_{k=1}^{\infty}q^{-k}\cdot q\cos(2\pi q^{k-1}x)=\chi(x).
\]
Since \eqref{replic1} is linear in $T_q$ apart from the constant, $T_q+c\chi$
satisfies it.
\end{proof}

For $q=2$, equation \eqref{replic1} on $[0,1]$ is the first of the four
functional equations of Takagi's function that
Sadowski~\cite[equation~(1)]{Sadowski2017} considered, and
Sadowski~\cite[Remark~4]{Sadowski2017} noted that the constant function
$1/2$ also satisfies it on $[0,1]$.

The next proposition holds for every solution of \eqref{threeeqs},
continuous or not. Its two limits are those that Proposition~\ref{cusp}
gives for the right and the left derivative of $T_q$.

\begin{proposition}
\label{anysol}
Let $f\colon[0,1]\to\R$ satisfy the $q$ equations \eqref{threeeqs}. Let
$n\ge0$ and $0\le m\le q^n$, and put $a=m/q^n$. If $m<q^n$, then
\[
 \lim_{k\to\infty}\frac{f(a+q^{-n-k})-f(a)}{q^{-n-k}}=+\infty,
\]
and if $m>0$, then
\[
 \lim_{k\to\infty}\frac{f(a)-f(a-q^{-n-k})}{q^{-n-k}}=-\infty.
\]
Moreover, $f$ has no finite right derivative at any $x\in[0,1)$ and no
finite left derivative at any $x\in(0,1]$.
\end{proposition}

\begin{proof}
Let $g=f-T_q$. By \eqref{threeeqs} and \eqref{branches},
$g((b+y)/q)=g(y)/q$ for $b\in\{0,1,\ldots,q-1\}$ and $0\le y\le1$. We
claim that
\begin{equation}
 g\Bigl(\frac{m+t}{q^n}\Bigr)=q^{-n}g(t),~n\ge0,~0\le m<q^n,~0\le t\le1.
 \label{gscale}
\end{equation}
We use induction on $n$. For $n=0$ there is nothing to prove. Let $n\ge1$,
and let $b$ be the first of the $n$ base $q$ digits of $m$, so that
$m=bq^{n-1}+m'$ with $0\le m'<q^{n-1}$. Then $(m+t)/q^n=(b+y)/q$ with
$y=(m'+t)/q^{n-1}\in[0,1]$, and $g((b+y)/q)=g(y)/q=q^{-n}g(t)$ by the
induction hypothesis. This proves \eqref{gscale}. The equations
$g(0)=g(0)/q$ and $g(1)=g((q-1+1)/q)=g(1)/q$ give $g(0)=g(1)=0$. For every
$n\ge0$, equation \eqref{gscale} with $m=0$ and $t=1$ gives
$g(q^{-n})=0$, and with $m=q^n-1$ and $t=0$ it gives $g(1-q^{-n})=0$.

Now let $a=m/q^n$ with $0\le m<q^n$, and let $k\ge1$. By \eqref{gscale},
$g(a)=q^{-n}g(0)=0$ and $g(a+q^{-n-k})=q^{-n}g(q^{-k})=0$, so the
slope of $f$ between $a$ and $a+q^{-n-k}$ equals that of $T_q$. Since
$T_q(q^{-k})=\xi_k(1)/q^k=(q-1)k\,q^{-k}$ by \eqref{samp} and \eqref{xi},
equation \eqref{cell} with $t=q^{-k}$ gives
\[
 \frac{T_q(a+q^{-n-k})-T_q(a)}{q^{-n-k}}
 =(q-1)n-2s_q(m)+\frac{T_q(q^{-k})}{q^{-k}}
 =(q-1)n-2s_q(m)+(q-1)k,
\]
which tends to $+\infty$. If $0<m\le q^n$, then
$a-q^{-n-k}=(m-1+(1-q^{-k}))/q^n$,
so $g(a-q^{-n-k})=q^{-n}g(1-q^{-k})=0$ and $g(a)=q^{-n}g(1)=0$, and
\eqref{cell2} with $t=q^{-k}$ gives
\[
 \frac{T_q(a)-T_q(a-q^{-n-k})}{q^{-n-k}}
 =(q-1)n-2s_q(m-1)-(q-1)k,
\]
which tends to $-\infty$.

It remains to prove the last statement. We follow the proof of
Kr\"uppel~\cite[Lemma~5.3 and Proposition~5.4]{Kruppel2010b}. By
\eqref{gscale}, $g=0$ at every $q$-adic rational of $[0,1]$, so $f=T_q$
there. Suppose
that $f$ has a finite right derivative $D$ at a point $x_0\in[0,1)$. For
every $n$ with $q^nx_0+4\le q^n$ let $j=\lfloor q^nx_0\rfloor$. One of
the two consecutive integers $j+1$ and $j+2$ is not congruent to $q-1$
modulo $q$. Let $i$ be the smaller one with this property, and put
$u_l=(i+l-1)/q^n$ for $l=1,2,3$. Then $0<u_l-x_0\le4q^{-n}$ and
$u_3\le1$, so
$f(u_l)=f(x_0)+D(u_l-x_0)+o(q^{-n})$ as $n\to\infty$, and the two slopes
$q^n(f(u_2)-f(u_1))$ and $q^n(f(u_3)-f(u_2))$ tend to $D$. Since $f=T_q$ at
the points $u_l$ and $T_q(1)=0$, equation \eqref{cell} with $t=1$ shows
that these slopes are $(q-1)n-2s_q(i)$ and $(q-1)n-2s_q(i+1)$. The last
base $q$ digit of $i$ is less than $q-1$, so $s_q(i+1)=s_q(i)+1$.
Hence the two slopes differ by $2$ for every such $n$, which contradicts
the fact that both tend to $D$.

This argument uses only the fact that $f=T_q$ at the $q$-adic rationals
of $[0,1]$. Since $T_q(1-x)=T_q(x)$, the function $x\mapsto f(1-x)$ also
has this property, so it has no finite right derivative at any point of
$[0,1)$, that is, $f$ has no finite left derivative at any point of
$(0,1]$.
\end{proof}

Random base $q$ expansions lead to equations of the same kind as
\eqref{branches}. Let $p_0,p_1,\ldots,p_{q-1}$ be positive numbers with sum $1$. Let
$\epsilon_1,\epsilon_2,\ldots$ be independent random digits, each equal
to $i$ with probability $p_i$ for $i=0,1,\ldots,q-1$. For $0\le x\le1$,
let $\Pi(x)$ be the probability that $\sum_{j=1}^{\infty}\epsilon_jq^{-j}$
is at most $x$,
\[
 \Pi(x)=\mathbb{P}\Bigl(\sum_{j=1}^{\infty}\epsilon_jq^{-j}\le x\Bigr),~0\le x\le1.
\]
The next lemma gives the equations for $\Pi$.

\begin{lemma}[Okada, Sekiguchi, and Shiota {\cite[Lemmas~2.1 and~2.2 and page~439]{OkadaSekiguchiShiota1996}}]
\label{lem:randdigits}
For $b\in\{0,1,\ldots,q-1\}$ and $0\le y\le1$,
\begin{equation}
 \Pi\Bigl(\frac{b+y}q\Bigr)=\sum_{i=0}^{b-1}p_i+p_b\Pi(y).
 \label{randeq}
\end{equation}
If $p_i=1/q$ for every $i$, then $\Pi(x)=x$ for $0\le x\le1$.
\end{lemma}

Okada, Sekiguchi, and Shiota stated the lemma for the function
$x\mapsto\mu([0,x])$, where $\mu$ is the probability measure on $[0,1]$
that gives each interval $[m/q^n,(m+1)/q^n)$ of level $n$ the product of
the probabilities $p_i$ of the $n$ base $q$ digits of $m$. This function
is $\Pi$, because the random number $\sum_{j=1}^{\infty}\epsilon_jq^{-j}$
lies in $[m/q^n,(m+1)/q^n)$ when $\epsilon_1,\ldots,\epsilon_n$ are the
digits of $m$, and otherwise only on an event of probability zero. Their
Lemma~2.1 gives a system of difference equations for this function, their
Lemma~2.2 shows that a continuous function satisfies this system exactly
when it satisfies \eqref{randeq}, and on page~439 they note that the
function is continuous and is $x\mapsto x$ when all $p_i$ are equal.

For a real number $s$ with $|s|<2/(q(q-1))$ let $\Pi_s$ be the function $\Pi$
for the probabilities
\[
 p_i=\frac1q+c_is,~c_i=\frac{q-1-2i}2,~i=0,1,\ldots,q-1.
\]
The numbers $p_i$ are positive because $|c_i|\le(q-1)/2$, and their sum
is $1$ because $\sum_{i=0}^{q-1}c_i=0$. For $s>0$ the digit $i$ has a
larger probability than the digit $q-1-i$ whenever $i<q-1-i$. For $q=3$
the three probabilities are $1/3+s$, $1/3$, and $1/3-s$, and for $q=2$
they are $1/2+s/2$ and $1/2-s/2$. By Lemma~\ref{lem:randdigits},
$\Pi_0(x)=x$. Hata and
Yamaguti~\cite[Theorem~4.6]{HataYamaguti1984} proved that for independent
binary digits equal to $0$ with probability $\alpha$, the derivative of
$\Pi(x)$ with respect to $\alpha$ at $\alpha=1/2$ is twice Takagi's
function. Sekiguchi and Shiota~\cite[Theorem~2.2]{SekiguchiShiota1991}
showed that $\Pi(x)$ is an analytic function of $\alpha$ and computed
all its derivatives, and Okada, Sekiguchi, and
Shiota~\cite[Theorem~5.1]{OkadaSekiguchiShiota1996} computed, for a
random expansion in any base with independent digits of arbitrary
probabilities, the derivatives of $\Pi(x)$ with respect to these
probabilities.

In the setting of Okada, Sekiguchi, and Shiota for $q=3$, the variables
are the probabilities of the digits $0$ and $1$, which are $1/3+s$ and
$1/3$ for $\Pi_s$, and the digit $2$ has the probability $1$ minus their
sum. Their theorem together with their Example~5.1 gives the derivative
of $\Pi_s(x)$ with respect to the first of these probabilities at $s=0$
as $\sum_{n=0}^{\infty}3^{-n}U(3^nx-\lfloor3^nx\rfloor)$, where $U(x)=3x$
for $0\le x\le1/3$, $U(x)=1$ for $1/3\le x\le2/3$, and $U(x)=3-3x$ for
$2/3\le x\le1$. Since $U=\frac32\tau_q$ on $[0,1]$, this derivative is
$\frac32T_q(x)$. Hence, for $q=3$ and every $x\in[0,1]$, the function
$s\mapsto\Pi_s(x)$ is differentiable at $s=0$, and
$T_q(x)=\frac23\lim_{s\to0}(\Pi_s(x)-x)/s$. In this sense $T_q$ measures
the change of $\Pi_s$ when the digit $0$ is favored over the digit $2$.

The results of Okada, Sekiguchi, and Shiota give the same limit in every
base. Fix $x\in[0,1)$ and regard $\Pi(x)$ as a function of
$(p_0,\ldots,p_{q-2})$, with $p_{q-1}=1-p_0-\cdots-p_{q-2}$. By their
Lemmas~5.1 and~5.3~\cite{OkadaSekiguchiShiota1996}, this function is
the limit of a sequence of functions that are analytic in
$(p_0,\ldots,p_{q-2})$, also for complex values, and that converge
uniformly near the point at which all $p_i$ equal $1/q$. Hence it is
analytic near this point. By their
Theorem~5.1~\cite{OkadaSekiguchiShiota1996}, its partial derivative with
respect to $p_k$ at this point, for $0\le k\le q-2$, is
$\sum_{n=0}^{\infty}q^{-n}U_k(q^nx-\lfloor q^nx\rfloor)$, where $U_k(t)$
is $q$ times the length of $[0,t]\cap[k/q,(k+1)/q]$ minus $q$ times the
length of $[0,t]\cap[(q-1)/q,1]$. For $q=3$ and $k=0$ this is the
function $U$ above. Since $\Pi_s$ is the function $\Pi$ for the
probabilities $p_i=1/q+c_is$, by the chain rule the derivative of
$\Pi_s(x)$ at $s=0$ is the same series with $\sum_{k=0}^{q-2}c_kU_k$ in
place of $U_k$. Since $c_{q-1}=-(c_0+\cdots+c_{q-2})$, the number
$\sum_{k=0}^{q-2}c_kU_k(t)$ is $q$ times the integral over $[0,t]$ of the
step function that equals $c_k$ on $[k/q,(k+1)/q)$ for
$0\le k\le q-1$. This step function is half the slope of $\tau_q$, and
$\tau_q(0)=0$, so that $\sum_{k=0}^{q-2}c_kU_k=\frac q2\tau_q$ on
$[0,1]$. Hence the derivative of $\Pi_s(x)$ at $s=0$ is $\frac q2T_q(x)$.

From Lemma~\ref{lem:randdigits} and \eqref{branches} alone, the next
proposition proves this limit again, with a bound for the difference.
For $q=3$ it contains the statement of Okada, Sekiguchi, and Shiota
described above. For $q=2$ the digit $0$ has the probability
$\alpha=1/2+s/2$, so the derivative with respect to $\alpha$ is twice the
derivative with respect to $s$, and the proposition gives the theorem of
Hata and Yamaguti stated above.

\begin{proposition}
\label{paramderq}
Let $M$ be the maximum of $T_q$. For $0<|s|<2/(q(q-1))$ and every
$x\in[0,1]$,
\begin{equation}
 \Bigl|\frac{\Pi_s(x)-x}{s}-\frac q2\,T_q(x)\Bigr|
 \le\frac{q^2M\,|s|}{2(2-q|s|)}.
 \label{paramest}
\end{equation}
Consequently the function $s\mapsto\Pi_s(x)$ is differentiable at $s=0$
for every $x\in[0,1]$, and
\[
 T_q(x)=\frac2q\lim_{s\to0}\frac{\Pi_s(x)-x}{s},
\]
where the convergence is uniform on $[0,1]$.
\end{proposition}

\begin{proof}
Fix $s$ with $0<|s|<2/(q(q-1))$, and recall that $p_i=1/q+c_is$ with
$c_i=(q-1-2i)/2$. Let $e$ be the supremum over $0\le x\le1$ of the left
side of \eqref{paramest}. It is finite, because $0\le\Pi_s\le1$ and $T_q$
is bounded, and it suffices to show that $e$ is at most the right side of
\eqref{paramest}.

Let $b\in\{0,1,\ldots,q-1\}$ and $0\le y\le1$. Since
$\sum_{i=0}^{b-1}c_i=\frac12\bigl(b(q-1)-b(b-1)\bigr)=b(q-b)/2$,
equation \eqref{randeq} gives
\[
 \Pi_s\Bigl(\frac{b+y}q\Bigr)-\frac{b+y}q
 =\frac{b(q-b)}2\,s+c_bs\,\Pi_s(y)+\frac{\Pi_s(y)-y}q,
\]
and \eqref{branches} gives
\[
 \frac q2\,T_q\Bigl(\frac{b+y}q\Bigr)
 =\frac{b(q-b)}2+c_b\,y+\frac1q\cdot\frac q2\,T_q(y).
\]
Dividing the first equation by $s$ and subtracting the second shows that
the value of $(\Pi_s(x)-x)/s-\frac q2T_q(x)$ at $x=(b+y)/q$ is
\[
 c_b\bigl(\Pi_s(y)-y\bigr)
 +\frac1q\Bigl(\frac{\Pi_s(y)-y}s-\frac q2\,T_q(y)\Bigr).
\]
The number in the large parentheses has absolute value at most $e$, so
that $|\Pi_s(y)-y|\le|s|\bigl(e+\frac q2T_q(y)\bigr)$. Since
$|c_b|\le(q-1)/2$ and $0\le T_q\le M$, the absolute value of the
displayed number is at most $\frac{q-1}2|s|\bigl(e+\frac q2M\bigr)+\frac1qe$.
Every point of $[0,1]$ is of the form $(b+y)/q$ with such $b$ and $y$.
Hence
\[
 e\le\frac{q-1}2\,|s|\,e+\frac{(q-1)q}4\,|s|\,M+\frac1q\,e,
 ~\text{that is,}~
 \frac{(q-1)(2-q|s|)}{2q}\,e\le\frac{(q-1)q}4\,|s|\,M.
\]
Here $2-q|s|>0$, because $q|s|<2/(q-1)\le2$. Dividing by the positive
factor of $e$ gives $e\le q^2M|s|/(2(2-q|s|))$, which is
\eqref{paramest}. The right side of \eqref{paramest} tends to $0$ as
$s\to0$ and does not depend on $x$. Since $\Pi_0(x)=x$, this proves the
remaining statements.
\end{proof}

\section{Fourier coefficients}
\label{sec:fourier}

To find the Fourier series of $T_q$, we write
$\widehat f(k)=\int_0^1f(x)e^{-2\pi ikx}\,dx$ for a continuous function
$f$ of period $1$ and $k\in\Z$. Delange~\cite{Delange1975} wrote $S_q$ in terms of a continuous function of
period $1$, and he determined the Fourier series of this
function~\cite[Section~4]{Delange1975}. It can be expressed through $T_q$ by the
function $\Psi$ of the next definition, which gives the difference
$S_q(m)/m-\frac{q-1}2\log_qm$ as a function of $\log_qm$.

\begin{definition}
\label{Psidef}
For $u\in\R$ let $\{u\}$ be the fractional part of $u$, and let
\[
 \Psi(u)=\frac{q-1}2\bigl(1-\{u\}\bigr)
 -\tfrac12\,q^{1-\{u\}}\,T_q\bigl(q^{\{u\}-1}\bigr).
\]
\end{definition}

The function $\Psi$ has period $1$, and it is continuous, since
$T_q(1/q)=(q-1)/q$ by \eqref{fe} and $T_q(1)=0$. If $m\ge1$ and
$q^{n-1}\le m<q^n$, then
$n=\log_qm+1-\{\log_qm\}$ and $m/q^n=q^{\{\log_qm\}-1}$, and \eqref{xi}
and \eqref{samp} give
\[
 S_q(m)=\frac{q-1}2\,nm-\tfrac12\,q^nT_q\Bigl(\frac m{q^n}\Bigr)
 =\frac{q-1}2\,m\log_qm+m\Psi(\log_qm).
\]
This is Delange's formula for $S_q$ in the form given by Galkin and
Galkina~\cite[equation~(1.3)]{GalkinGalkina2020}, with $\Psi$ as its
continuous function of period $1$. For $0<x\le1$ let
$f_q(x)=(q-1)\log_qx+T_q(x)/x$, the slope of $T_q$ between $0$ and $x$
minus $(q-1)\log_q(1/x)$. Then
$\Psi(\log_qx)=-\frac12f_q(x)$ for $0<x\le1$. Indeed, if $1/q\le x<1$,
then $\{\log_qx\}=1+\log_qx$, and the definition of $\Psi$ gives
$\Psi(\log_qx)=-\frac{q-1}2\log_qx-T_q(x)/(2x)$. At $x=1$ both sides are
zero. Neither side changes when $x$ is replaced by $x/q$: the function
$\Psi$ has period $1$, and for $0<x\le1/q$ we have $\tau_q(x)=(q-1)x$, so
that \eqref{fe} gives $T_q(x)/x=q-1+T_q(qx)/(qx)$ and hence $f_q(x)=f_q(qx)$.

For $n\ge0$, the partial sum of order $n$ of the Fourier series of a
continuous function $f$ of period $1$ is
$\sum_{k=-n}^{n}\widehat f(k)e^{2\pi ikx}$. The proof of the next theorem
uses the following theorem of Fej\'er. He stated it for bounded
integrable functions of period $2\pi$, at every point of continuity and
at every discontinuity of the first kind, where the limit is the mean of
the two one sided limits. We state the case of continuous functions of
period $1$.

\begin{theorem}[Fej\'er {\cite[Section~1]{Fejer1904}}]
\label{fejer}
Let $f$ be a continuous function of period $1$. For every real $x$, the
arithmetic mean of the partial sums of orders $0,1,\ldots,n-1$ of the
Fourier series of $f$ at $x$ tends to $f(x)$ as $n\to\infty$.
\end{theorem}

\begin{lemma}
\label{fzero}
Let $f$ be a continuous function of period $1$ with $\widehat f(k)=0$ for
every $k\in\Z$. Then $f(x)=0$ for every real $x$.
\end{lemma}

\begin{proof}
All partial sums of the Fourier series of $f$ are zero, and so are their
arithmetic means. By Theorem~\ref{fejer}, these means tend to $f(x)$ for
every real $x$.
\end{proof}

The next theorem gives the Fourier series of $T_q$ (see
Figure~\ref{fourpic}). For $q=2$ it is the Fourier series of Takagi's
function~\cite[Theorem~6.1]{Lagarias2012}.

\begin{theorem}
\label{fourier}
The mean of $T_q$ is $\widehat T_q(0)=(q+1)/6$. For $k\ne0$, write
$|k|=q^vm$ with integers $v\ge0$ and $m\ge1$ such that $q\nmid m$. Then
\begin{equation}
 \widehat T_q(k)=-\frac{q^{1-v}}{2\pi^2m^2}.
 \label{fcoef}
\end{equation}
Consequently
\begin{equation}
 T_q(x)=\frac{q+1}6-\frac q{\pi^2}
 \sum_{v=0}^{\infty}q^{-v}
 \sum_{\substack{m=1\\q\nmid m}}^{\infty}\frac{\cos(2\pi q^vmx)}{m^2},
 \label{fser}
\end{equation}
where the double series converges absolutely and uniformly on $\R$.
\end{theorem}

\begin{proof}
Since $\tau_q$ is linear on each interval $[l/q,(l+1)/q]$ with
$\tau_q(l/q)=l(q-l)/q$, its mean is
$\widehat\tau_q(0)=\frac1q\sum_{l=1}^{q-1}l(q-l)/q=(q^2-1)/(6q)$. Let
$k\ne0$. On the interval
$[l/q,(l+1)/q]$ with $0\le l\le q-1$, the function $\tau_q$ has the slope
$q-1-2l$, and an integration by parts gives
\[
 \int_{l/q}^{(l+1)/q}\tau_q(x)e^{-2\pi ikx}\,dx
 =\Bigl[\frac{\tau_q(x)e^{-2\pi ikx}}{-2\pi ik}\Bigr]_{l/q}^{(l+1)/q}
 -\frac{(q-1-2l)\bigl(e^{-2\pi ik(l+1)/q}-e^{-2\pi ikl/q}\bigr)}{(2\pi ik)^2}.
\]
We add these equations for $l=0,1,\ldots,q-1$. Since $\tau_q$ is
continuous and $\tau_q(0)=\tau_q(1)=0$, the first terms on the right add up
to zero. In the sum of the numbers
$(q-1-2l)\bigl(e^{-2\pi ik(l+1)/q}-e^{-2\pi ikl/q}\bigr)$, the number
$e^{-2\pi ikl/q}$ with $1\le l\le q-1$ has the coefficient
$(q-1-2(l-1))-(q-1-2l)=2$, and the numbers $e^0=1$ and $e^{-2\pi ik}=1$
have the coefficients $-(q-1)$ and $-(q-1)$. This sum is therefore
$-2q+2\sum_{l=0}^{q-1}e^{-2\pi ikl/q}$, and
\[
 \widehat\tau_q(k)=\frac{2q-2\sum_{l=0}^{q-1}e^{-2\pi ikl/q}}{(2\pi ik)^2}
 =-\frac{q-\sum_{l=0}^{q-1}e^{-2\pi ikl/q}}{2\pi^2k^2}
 =\begin{cases}
   0,&q\mid k,\\
   -q/(2\pi^2k^2),&q\nmid k.
  \end{cases}
\]
Indeed, if $q\mid k$, then all $q$ terms of the last sum are $1$. If
$q\nmid k$, then the last sum is a geometric sum whose ratio
$e^{-2\pi ik/q}\ne1$ has the $q$th power $1$, so it is zero.
Since the function $x\mapsto\tau_q(q^jx)$ has period $q^{-j}$, its
coefficient at frequency $k$ is zero unless $q^j\mid k$, and then it
equals $\widehat\tau_q(k/q^j)$. As the series \eqref{T} converges
uniformly, we may integrate it term by term. Let $|k|=q^vm$ with
$q\nmid m$. Then the terms with $j<v$ contribute nothing because $k/q^j$
is divisible by $q$, and the terms with $j>v$ contribute nothing because
$q^j\nmid k$. The term $j=v$ contributes
$q^{-v}\widehat\tau_q(\pm m)=q^{-v}\widehat\tau_q(m)$, since $\tau_q$ is even.
This proves \eqref{fcoef}, and the mean is
$\widehat T_q(0)=\frac{q^2-1}{6q}\sum_{j=0}^{\infty}q^{-j}=\frac{q+1}6$.

Since $\sum_{v=0}^{\infty}q^{-v}\sum_{m=1}^{\infty}m^{-2}<\infty$, the coefficients \eqref{fcoef}
are absolutely summable, and hence the Fourier series of $T_q$ converges
absolutely and uniformly. Its sum is continuous and has the same
coefficients as $T_q$. Lemma~\ref{fzero}, applied to the difference of the sum and $T_q$, shows
that the sum equals $T_q$. Pairing the frequencies $k$ and
$-k$ gives \eqref{fser}.
\end{proof}

\begin{example}
\label{ex:euler}
Theorem~\ref{fourier} gives Euler's formula
$\sum_{m=1}^{\infty}m^{-2}=\pi^2/6$. Since
$\sum_{v=0}^{\infty}q^{-v}=q/(q-1)$, the value $T_q(0)=0$ and \eqref{fser}
at $x=0$ give
\[
 \sum_{\substack{m=1\\q\nmid m}}^{\infty}\frac1{m^2}
 =\frac{(q^2-1)\pi^2}{6q^2}.
\]
The terms with $q\mid m$ of $\sum_{m=1}^{\infty}m^{-2}$ add up to
$q^{-2}$ times this sum, so
$\sum_{m=1}^{\infty}m^{-2}=\frac{q^2}{q^2-1}\cdot\frac{(q^2-1)\pi^2}{6q^2}
=\frac{\pi^2}6$. Now let $q$ be odd. Then $T_q(1/2)=(q+1)/4$ by
Theorem~\ref{limitthm}, and at $x=1/2$ the cosines in \eqref{fser} equal
$(-1)^m$, because $q^vm$ and $m$ have the same parity. Hence
\eqref{fser} at $x=1/2$ gives
\[
 \sum_{\substack{m=1\\q\nmid m}}^{\infty}\frac{(-1)^m}{m^2}
 =-\frac{(q^2-1)\pi^2}{12q^2}.
\]
Since $q$ is odd, the terms with $q\mid m$ of
$\sum_{m=1}^{\infty}(-1)^mm^{-2}$ add up to $q^{-2}$ times this sum, and
we obtain $\sum_{m=1}^{\infty}(-1)^{m-1}m^{-2}=\pi^2/12$ in the same
way. For $q=3$ the two displayed sums are $4\pi^2/27$ and
$-2\pi^2/27$.
\end{example}

\begin{figure}[H]
\centering
\includegraphics{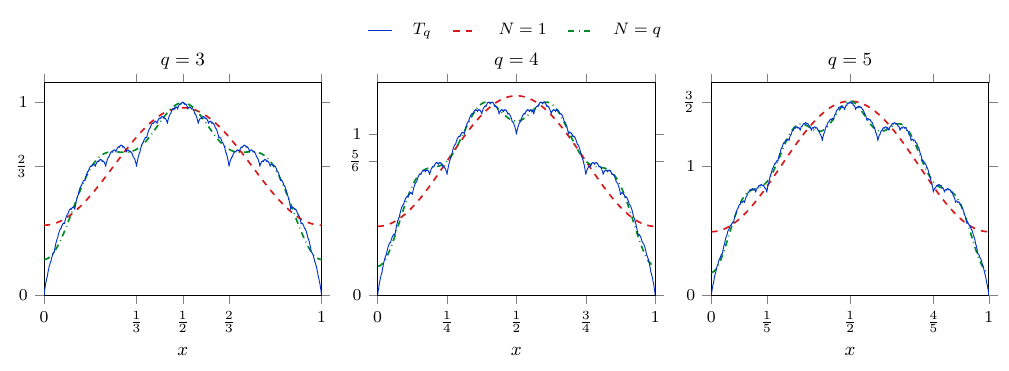}
\caption{Theorem~\ref{fourier} for $q=3$, $4$, and $5$: the function $T_q$
and the sums of $(q+1)/6$ and the terms of \eqref{fser} with $q^vm\le N$,
for $N=1$ and $N=q$.}
\label{fourpic}
\end{figure}

Every cosine in \eqref{fser} has a negative coefficient. If we keep
finitely many terms together with the constant $(q+1)/6$, the absolute
error is at most the sum of the absolute values of the omitted
coefficients. Equality holds at $x=0$, where all cosines equal $1$, so the
largest absolute error is attained there. For $q=3$, $4$, and $5$,
Figure~\ref{fourpic} compares $T_q$ with the sums of $(q+1)/6$ and the
terms of \eqref{fser} with $q^vm\le N$ for $N=1$ and $N=q$. For $q=3$ a
computation gives about $0.36$, $0.19$, $0.084$, and $0.036$ for the
largest absolute difference when $N=1$, $3$, $9$, and $27$. For $N=1$
and $N=q$ it gives about $0.43$ and $0.18$ when $q=4$, and about $0.49$
and $0.18$ when $q=5$.

If $|k|=q^vm$ with $q\nmid m$, then
\[
 |k|\,|\widehat T_q(k)|=\frac{q}{2\pi^2m},
\]
which equals $q/(2\pi^2)$ exactly when $|k|=q^v$ and tends to $0$ as
$|k|\to\infty$ with $q\nmid k$. In particular,
$\limsup_{k\to\infty}k|\widehat T_q(k)|=q/(2\pi^2)$.

The coefficients \eqref{fcoef} can also be obtained from a representation
of $T_q$ due to Galkin and
Galkina~\cite[Proposition~2 and equation~(2.7)]{GalkinGalkina2020},
\[
 T_q(x)=\{x\}(1-\{x\})+(q-1)\sum_{j=0}^{\infty}q^{-j}\{q^jx\}\bigl(1-\{q^jx\}\bigr).
\]
The function $\{x\}(1-\{x\})$ has the mean $1/6$ and the Fourier
coefficients $-1/(2\pi^2k^2)$ for $k\ne0$. Hence, if $q^j\mid k$, the
term with index $j$ of the sum has the Fourier coefficient
$-q^j/(2\pi^2k^2)$ at $k$. Integration term by term
gives $\widehat T_q(0)=\frac16+(q-1)\cdot\frac16\cdot\frac q{q-1}=\frac{q+1}6$
and, for $|k|=q^vm$ with $q\nmid m$,
\[
 \widehat T_q(k)=-\frac1{2\pi^2k^2}\Bigl(1+(q-1)\sum_{j=0}^{v}q^j\Bigr)
 =-\frac{q^{v+1}}{2\pi^2k^2},
\]
which is \eqref{fcoef} again.

\section{The variation at each level and the Assouad dimension of the graph}
\label{sec:var}

The results of this section hold for every base $q\ge2$.

\begin{definition}
\label{variation}
For $n\ge1$, the variation of $T_q$ at level $n$ is the number
\begin{equation}
 V_n=\sum_{m=0}^{q^n-1}
 \left|T_q\!\left(\frac{m+1}{q^n}\right)-T_q\!\left(\frac m{q^n}\right)\right|
 =q^{-n}\sum_{m=0}^{q^n-1}\bigl|\xi_n(m+1)-\xi_n(m)\bigr|,
 \label{Vn}
\end{equation}
that is, the variation of $T_q$ over the partition of $[0,1]$ into the
intervals of level $n$, or the average of the absolute values of
the slopes of these intervals. We put $V_0=0$. For a function on
$[0,1]$, its total variation is the supremum of the sums of its absolute
increments over all finite partitions of $[0,1]$.
\end{definition}

The second form of \eqref{Vn} follows from the sampling equation
\eqref{samp}, and by \eqref{diff} it gives
\begin{equation}
 V_n=q^{-n}\sum_{m=0}^{q^n-1}\bigl|(q-1)n-2s_q(m)\bigr|,~n\ge0.
 \label{Vslopesum}
\end{equation}
By Theorem~\ref{Fn}, $V_n$ is also the total variation of
$F_n$ on $[0,1]$. Since each interval of level $n$ is the union of $q$
intervals of level $n+1$, the triangle inequality shows that the sequence
$(V_n)$ is nondecreasing. For $n\ge0$ and an integer $k$, let
$C_n(k)$ be the number of integers $m$ with $0\le m<q^n$ for which the
interval $[m/q^n,(m+1)/q^n]$ of level $n$ has slope $k$, that is,
$(q-1)n-2s_q(m)=k$. By \eqref{dnsq}, $C_n(k)$ is also the number of words
of length $n$ over $\{0,1,\ldots,q-1\}$ with digit slope $d_n=k$, so that
$C_n(0)$ is the number $C_n$ of Lemma~\ref{trinom}.

\begin{theorem}
\label{varthm}
For every $n\ge0$,
\begin{equation}
 V_{n+1}-V_n=\frac2{q^{n+1}}\sum_{k=-(q-2)}^{q-2}
 \Bigl\lfloor\frac{(q-|k|)^2}4\Bigr\rfloor C_n(k).
 \label{Vrec}
\end{equation}
Moreover,
\begin{equation}
 \lim_{n\to\infty}\frac{V_n}{\sqrt n}=\sqrt{\frac{2(q^2-1)}{3\pi}},
 \label{Vlim}
\end{equation}
and $T_q$ has infinite total variation on every interval of
positive length. For $q=3$, equations \eqref{Vrec} and \eqref{Vlim} give
\begin{equation}
 V_{n+1}-V_n=\frac{4C_n}{3^{n+1}},~
 V_n=\frac43\sum_{j=0}^{n-1}\frac{C_j}{3^j},
 ~\text{and}~
 \lim_{n\to\infty}\frac{V_n}{\sqrt n}=\frac4{\sqrt{3\pi}},
 \label{Vthree}
\end{equation}
with the central trinomial coefficients $C_n$. For $q=2$ they read
\begin{equation}
 V_{n+1}-V_n=\frac{C_n}{2^n}
 ~\text{and}~
 \lim_{n\to\infty}\frac{V_n}{\sqrt n}=\sqrt{\frac2\pi},
 \label{Vtwo}
\end{equation}
where $C_n=\binom n{n/2}$ if $n$ is even and $C_n=0$ if $n$ is odd.
\end{theorem}

\begin{proof}
We claim that, for every integer $d$,
\begin{equation}
 \sum_{\xi=0}^{q-1}|d+q-1-2\xi|=
 \begin{cases}
  q|d|+2\lfloor(q-|d|)^2/4\rfloor,&|d|\le q-2,\\
  q|d|,&|d|\ge q-1.
 \end{cases}
 \label{Vstepsum}
\end{equation}
Indeed, as $\xi$ runs through $0,1,\ldots,q-1$, so does $q-1-\xi$, and
$q-1-2(q-1-\xi)=-(q-1-2\xi)$. Hence the left side of \eqref{Vstepsum}
does not change when $d$ is replaced by $-d$, and
$\sum_{\xi=0}^{q-1}(q-1-2\xi)=0$. We may therefore assume that $d\ge0$.
Since
$|y|-y=2\max\{-y,0\}$ for real $y$, it follows that
\[
 \sum_{\xi=0}^{q-1}|d+q-1-2\xi|-qd
 =2\sum_{\xi=0}^{q-1}\max\{2\xi-(q-1)-d,0\}
 =2\sum_{\xi=0}^{q-1}\max\{q-1-2\xi-d,0\},
\]
where the last step replaces $\xi$ by $q-1-\xi$. If $d\ge q-1$, then
every term of the last sum is zero. Let $0\le d\le q-2$, and put
$p=q-1-d\ge1$. The positive terms of the last sum are $p,p-2,p-4,\ldots$,
down to $2$ if $p$ is even and to $1$ if $p$ is odd. Their sum is
$p(p+2)/4$ in the first case and $(p+1)^2/4$ in the second. Since
$(p+1)^2=p(p+2)+1$, in both cases it equals
$\lfloor(p+1)^2/4\rfloor=\lfloor(q-d)^2/4\rfloor$, which proves
\eqref{Vstepsum}.

Let $n\ge0$ and $0\le m<q^n$, and let $d=(q-1)n-2s_q(m)$ be the slope of
the interval $[m/q^n,(m+1)/q^n]$. The intervals of level $n+1$ in this
interval are
$[(qm+\xi)/q^{n+1},(qm+\xi+1)/q^{n+1}]$ with $0\le\xi\le q-1$, and since
$s_q(qm+\xi)=s_q(m)+\xi$, their slopes are the numbers
$(q-1)(n+1)-2s_q(qm+\xi)=d+q-1-2\xi$. By \eqref{Vslopesum} for $n+1$,
the number $q^{n+1}V_{n+1}$ is therefore the sum of the left sides of
\eqref{Vstepsum} over $0\le m<q^n$. The terms $q|d|$ on the right sides
add up to $q\cdot q^nV_n$, again by \eqref{Vslopesum}, and for each $k$
with $|k|\le q-2$ there are $C_n(k)$ integers $m$ with $d=k$. Hence
$q^{n+1}(V_{n+1}-V_n)$ is the sum of the numbers
$2\lfloor(q-|k|)^2/4\rfloor C_n(k)$ with $|k|\le q-2$, which is
\eqref{Vrec}.

To prove \eqref{Vlim}, we write $V_n$ as an integral. For every real $x$,
\begin{equation}
 |x|=\frac2\pi\int_0^{\infty}\frac{1-\cos(x\theta)}{\theta^2}\,d\theta.
 \label{Vabsint}
\end{equation}
Indeed, for $x\ne0$ the substitution $s=|x|\theta$ shows that the
integral in \eqref{Vabsint} is $|x|\int_0^{\infty}(1-\cos s)s^{-2}\,ds$.
For $s>0$ we have $s^{-2}=\int_0^{\infty}ye^{-sy}\,dy$, and the function
$(1-\cos s)ye^{-sy}$ is nonnegative, so by Tonelli's theorem we may
integrate over $s$ first. Since
$\int_0^{\infty}(1-\cos s)e^{-sy}\,ds=\frac1y-\frac y{1+y^2}
=\frac1{y(1+y^2)}$ for $y>0$, we obtain
\[
 \int_0^{\infty}\frac{1-\cos s}{s^2}\,ds
 =\int_0^{\infty}\frac{dy}{1+y^2}=\frac\pi2 .
\]
Let $\Upsilon$ be the function of the proof of Lemma~\ref{trinom}. As we
saw there,
$q\Upsilon(\theta)=\sum_{\xi=0}^{q-1}e^{i(q-1-2\xi)\theta}
=\sum_{\xi=0}^{q-1}\cos((q-1-2\xi)\theta)$ for real $\theta$. Let
$x_0,\ldots,x_{n-1}$ be the base $q$ digits of an integer $m$ with
$0\le m<q^n$, as in Section~\ref{sec:graph}. Then
$(q-1)n-2s_q(m)=\sum_{l=0}^{n-1}(q-1-2x_l)$, and as $m$ runs through
$0,1,\ldots,q^n-1$, the word $x_{n-1}\ldots x_0$ takes each value in
$\{0,1,\ldots,q-1\}^n$ exactly once. Therefore
\[
 \sum_{m=0}^{q^n-1}e^{i((q-1)n-2s_q(m))\theta}
 =\Bigl(\sum_{\xi=0}^{q-1}e^{i(q-1-2\xi)\theta}\Bigr)^n
 =q^n\Upsilon(\theta)^n,
\]
and the real part of this equation is
$\sum_{m=0}^{q^n-1}\cos\bigl(((q-1)n-2s_q(m))\theta\bigr)=q^n\Upsilon(\theta)^n$.
We apply \eqref{Vabsint} to each term of \eqref{Vslopesum} and obtain,
for $n\ge1$,
\begin{equation}
 V_n=\frac2\pi\int_0^{\infty}\frac{1-\Upsilon(\theta)^n}{\theta^2}\,d\theta
 ~\text{and hence}~
 \frac{V_n}{\sqrt n}
 =\frac2\pi\int_0^{\infty}\frac{1-\Upsilon(t/\sqrt n)^n}{t^2}\,dt,
 \label{Vint}
\end{equation}
by the substitution $\theta=t/\sqrt n$. We let $n\to\infty$ in the second
integral. In the proof of Lemma~\ref{trinom} we saw that
$\sum_{\xi=0}^{q-1}(q-1-2\xi)^2=q(q^2-1)/3$ and that
$\Upsilon(\theta)=1-(q^2-1)\theta^2/6+O(\theta^4)$ as $\theta\to0$. The
expansion gives $\Upsilon(t/\sqrt n)^n\to e^{-(q^2-1)t^2/6}$ for every
$t>0$. Since $0\le1-\cos y\le y^2/2$ for real $y$, the first fact gives
$0\le1-\Upsilon(\theta)\le(q^2-1)\theta^2/6$. Moreover, $|\Upsilon(\theta)|\le1$,
so that
\[
 \bigl|1-\Upsilon(\theta)^n\bigr|
 =\bigl(1-\Upsilon(\theta)\bigr)\Bigl|\sum_{j=0}^{n-1}\Upsilon(\theta)^j\Bigr|
 \le n\bigl(1-\Upsilon(\theta)\bigr)\le\frac{(q^2-1)n\theta^2}6
\]
and $|1-\Upsilon(\theta)^n|\le2$. For $\theta=t/\sqrt n$ these two bounds
show that the absolute value of the integrand of the second integral in
\eqref{Vint} is at most $\min\{(q^2-1)/6,2/t^2\}$. This function of $t$
is integrable over $(0,\infty)$ and does not depend on $n$, so by
dominated convergence
\[
 \lim_{n\to\infty}\frac{V_n}{\sqrt n}
 =\frac2\pi\int_0^{\infty}\frac{1-e^{-(q^2-1)t^2/6}}{t^2}\,dt .
\]
For every $a>0$ we have $1-e^{-at^2}=t^2\int_0^ae^{-st^2}\,ds$. The
function $e^{-st^2}$ is positive, so by Tonelli's theorem we may
integrate over $t$ first, and
\[
 \int_0^{\infty}\frac{1-e^{-at^2}}{t^2}\,dt
 =\int_0^a\int_0^{\infty}e^{-st^2}\,dt\,ds
 =\int_0^a\frac{\sqrt\pi}{2\sqrt s}\,ds=\sqrt{\pi a}.
\]
With $a=(q^2-1)/6$ the limit is
$\frac2\pi\sqrt{\pi(q^2-1)/6}=\sqrt{2(q^2-1)/(3\pi)}$, which is
\eqref{Vlim}.

Let $q=3$. Then every slope is even, so $C_n(1)=C_n(-1)=0$, and the term
with $k=0$ in \eqref{Vrec} is $\lfloor9/4\rfloor C_n=2C_n$. This gives
the first formula in \eqref{Vthree}, summation over the levels below $n$
gives the second one because $V_0=0$, and
$\sqrt{2\cdot8/(3\pi)}=4/\sqrt{3\pi}$. Let $q=2$. Then the sum in
\eqref{Vrec} has only the term with $k=0$, which is
$\lfloor4/4\rfloor C_n=C_n$, and Lemma~\ref{trinom} gives the values of
$C_n$. This proves \eqref{Vtwo}.

Finally, $V_n\to\infty$ by \eqref{Vlim}, so the total variation of $T_q$ on
$[0,1]$ is infinite. Every interval of positive length contains an
interval of some level $n$ shifted by an integer, and $T_q$ has period $1$.
It therefore suffices to consider an interval $[m/q^n,(m+1)/q^n]$ of
level $n$ with $0\le m<q^n$. Let $d$ be its slope. For $j\ge1$ we divide
it into the $q^j$ intervals of level $n+j$ with the endpoints
$(m+l/q^j)/q^n$, $0\le l\le q^j$. By \eqref{cell},
\[
 T_q\!\left(\frac{m+(l+1)/q^j}{q^n}\right)-T_q\!\left(\frac{m+l/q^j}{q^n}\right)
 =q^{-n}\Bigl(\frac d{q^j}+T_q\!\left(\frac{l+1}{q^j}\right)
 -T_q\!\left(\frac l{q^j}\right)\Bigr)
\]
for $0\le l<q^j$. The sum of the absolute values of these differences is
at least $q^{-n}(V_j-|d|)$, which tends to infinity with $j$. Hence the
total variation of $T_q$ on $[m/q^n,(m+1)/q^n]$ is infinite.
\end{proof}

\begin{example}
\label{ex:var}
Let $q=3$. By \eqref{Vthree}, the first values are $V_1=4/3$, $V_2=16/9$,
$V_3=20/9$, and $V_4=208/81$. For $n=2$ this agrees with the values
$T_q(m/9)=\xi_2(m)/9$ from Example~\ref{ex:F2}: the absolute values of
their differences are $4/9$, $2/9$, $0$, $2/9$, $0$, $2/9$, $0$, $2/9$,
and $4/9$, and their sum is $16/9$.

Let $q=2$. By \eqref{Vtwo}, the sequence $V_1,V_2,\ldots$ begins with
$1$, $1$, $3/2$, $3/2$, $15/8$, and $15/8$, and
$V_{2i}=2i\binom{2i}i/4^i$ for every $i\ge0$. Indeed, this holds for
$i=0$, and if it holds for $i$, then \eqref{Vtwo} gives
$V_{2i+2}=V_{2i+1}=V_{2i}+\binom{2i}i/4^i=(2i+1)\binom{2i}i/4^i$, which
equals $(2i+2)\binom{2i+2}{i+1}/4^{i+1}$ because
$\binom{2i+2}{i+1}=\binom{2i}i(2i+2)(2i+1)/(i+1)^2$.

Let $q=4$. The four intervals of level $1$ have the slopes $3$, $1$,
$-1$, and $-3$, so that $C_1(1)=C_1(-1)=1$ and $C_1(0)=C_1(2)=C_1(-2)=0$.
By \eqref{Vrec}, $V_1=\frac24\lfloor16/4\rfloor C_0(0)=2$ and
$V_2-V_1=\frac2{16}\bigl(\lfloor9/4\rfloor+\lfloor9/4\rfloor\bigr)=1/2$.
Thus $V_2>V_1$, although no interval of level $1$ has slope $0$.

Figure~\ref{varpic} shows the numbers $V_n/\sqrt n$ for $q=2$, $3$, $4$,
and $5$ and $1\le n\le60$, divided by their limit \eqref{Vlim}, with
$V_n$ computed exactly from \eqref{Vslopesum}. For $q=2$ the quotients
decrease from $n=2i+1$ to $n=2i+2$, because $V_{2i+2}=V_{2i+1}$.
\end{example}

\begin{figure}[H]
\centering
\includegraphics{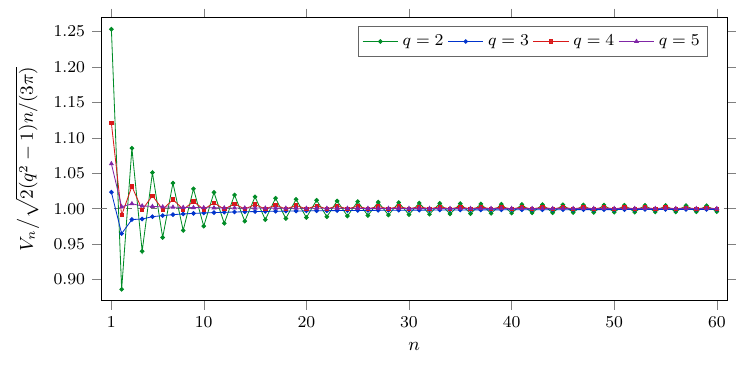}
\caption{The numbers $V_n/\sqrt n$ for $q=2$, $3$, $4$, and $5$ and
$1\le n\le60$, divided by their limit $\sqrt{2(q^2-1)/(3\pi)}$ (see
Theorem~\ref{varthm}).}
\label{varpic}
\end{figure}

The polygonal path through the points $(m/q^n,T_q(m/q^n))$,
$0\le m\le q^n$, has length at least $V_n$. Since $V_n\to\infty$ by
Theorem~\ref{varthm}, the graph of $T_q$ over $[0,1]$ has infinite length.

In the proof of his theorem that the graph of Takagi's function has
infinite length over every interval,
Lagarias~\cite[Theorem~11.4]{Lagarias2012} used the first formula in
\eqref{Vtwo}, with the central binomial coefficients. Variation sums of
series of the form $\sum_{j=0}^{\infty}\alpha^jg(b^jx)$, with a real number
$\alpha$, an integer $b\ge2$, and a periodic function $g$, over the partitions of $[0,1]$ into the intervals
$[m/b^n,(m+1)/b^n]$ have been
studied by Schied and Zhang~\cite{SchiedZhang2020} and by Han, Schied,
and Zhang~\cite{HanSchiedZhang2021}.

Jiang~\cite[Theorem~1.1]{Jiang2025}
showed that the graph of $\sum_{j=0}^{\infty}b^{-j}\gamma_j\dist(b^jx,\Z)$
over $[0,1]$ has Assouad dimension $1$ (see Definition~\ref{assouad})
whenever $b\ge2$ is an integer and the numbers $\gamma_j$ are bounded.
By Remark~\ref{pow2}, this contains the next
proposition if $q$ is a power of $2$. For the other bases his argument
also applies to $T_q$, and with \eqref{cell} it takes the following short
form (see Figure~\ref{assouadpic} for $q=3$, $4$, and $5$).

\begin{proposition}
\label{assouadgraph}
The graph $G=\{(x,T_q(x)):0\le x\le1\}$ has Assouad dimension $1$.
\end{proposition}

\begin{proof}
The projection $(x,y)\mapsto x$ maps $G$ onto $[0,1]$ and does not
increase distances. Hence, if sets $U_1,U_2,\ldots$ cover $G$, then
their images lie in intervals of the lengths $\operatorname{diam}U_i$
that cover $[0,1]$, so $\sum_i\operatorname{diam}U_i\ge1$. If all $U_i$
have diameter at most
$r\le1$ and $0\le s<1$, it follows that
$\sum_i(\operatorname{diam}U_i)^s\ge r^{s-1}\sum_i\operatorname{diam}U_i\ge r^{s-1}$.
By Definition~\ref{hausdorff}, the Hausdorff measure of dimension $s$ of
$G$ is therefore infinite for every $s<1$, so $G$ has Hausdorff dimension
at least $1$. By Remark~\ref{rem:assouad}, the Assouad dimension of $G$
is at least $1$. Since $G$ is compact, it is covered by finitely many
balls $B(z_1,1),\ldots,B(z_M,1)$ with $z_i\in G$. If $R>1$ and
$r'=\min\{r,1/2\}$, then
$N_r(G\cap B(z,R))\le\sum_{i=1}^MN_{r'}(G\cap B(z_i,1))$ and
$1/r'\le2R/r$. Hence it suffices to bound $N_r(G\cap B(z,R))$ for
$z=(x_0,y_0)\in G$ and $0<r<R\le1$. Choose
integers $n\ge0$ and $m\ge1$ with $q^{-n-1}<R\le q^{-n}$ and
$q^{-n-m}\le r<q^{1-n-m}$, so that $q^m<q^2R/r$, and put
$\delta=q^{-n-m}$. We cover $G\cap B(z,R)$ by squares of side $\delta$
of the grid $\delta\Z^2$. Each of them is the union of four squares of
diameter less than $r$, so $N_r(G\cap B(z,R))$ is at most four times the
number of grid squares in such a cover. By Theorem~\ref{limitthm} and
Theorem~\ref{gmax}, $0\le T_q(u)\le(q+1)/4$ for every real $u$, where
for even $q$ we use the inequality $q(q+2)/(4(q+1))<(q+1)/4$. We put
$\beta=(q+9)/4$, so that $T_q(u)+2\le\beta$.

Since $R\le q^{-n}$, the set $[x_0-R,x_0+R]\cap[0,1]$ is contained in the
union of at most three intervals of level $n$. Let
$J=[k/q^n,(k+1)/q^n]$ be one of them, and let $d=(q-1)n-2s_q(k)$ be its
slope. Every interval $J'=[a,a+\delta]$ of
level $n+m$ in $J$ has a slope $d'$ with $|d'-d|\le(q-1)m$,
because each of the $m$ further digits changes the slope by at most
$q-1$. By \eqref{cell},
$T_q(a+\delta u)=T_q(a)+\delta(d'u+T_q(u))$ for $0\le u\le1$, so $T_q$
varies on $J'$ by at most
$(|d'|+\beta-2)\delta\le(|d|+(q-1)m+\beta-2)\delta$,
and the points of $G$ above $J'$ are covered by at most $|d|+(q-1)m+\beta$
squares of the grid in
the column above $J'$. The points of $G\cap B(z,R)$ above $J'$ have second
coordinates in an interval of length $2R$, so they are also covered by
at most $2R/\delta+2\le2q^m+2$ of these squares. If
$d=0$, the $q^m$ intervals $J'$ in $J$ give at most $((q-1)m+\beta)q^m$
squares. Let $d\ne0$, and write a point of $J$ as
$x=k/q^n+q^{-n}u$ with $0\le u\le1$. If $|T_q(x)-y_0|\le R$, then
\eqref{cell}, $R\le q^{-n}$, and $0\le T_q(u)\le\beta-2$ show that $du$
lies between $q^n(y_0-T_q(k/q^n))+1-\beta$ and $q^n(y_0-T_q(k/q^n))+1$,
two numbers that do not depend on $x$. These points
$x$ therefore lie in an interval of length $\beta q^{-n}/|d|$,
which meets
at most $\beta q^m/|d|+2$ of the intervals $J'$. For all but at most two of these columns, we use the bound
$|d|+(q-1)m+\beta$. For the remaining columns, we use $2q^m+2$, which does
not depend on $d$. This keeps the total independent of $|d|$, although
$|d|$ can be as large as $(q-1)n$. We obtain at most
\[
 \frac{\beta q^m}{|d|}\,\bigl(|d|+(q-1)m+\beta\bigr)+2(2q^m+2)
 \le\bigl(\beta((q-1)m+\beta+1)+8\bigr)q^m
\]
squares, since $|d|\ge1$ gives
$(|d|+(q-1)m+\beta)/|d|\le(q-1)m+\beta+1$. This bound is larger than the
bound for $d=0$. Hence, in all cases,
\[
 N_r(G\cap B(z,R))\le4\cdot3\bigl(\beta((q-1)m+\beta+1)+8\bigr)q^m
 \le12q^2\Bigl(\beta(q-1)\log_q\frac Rr+\beta(2q+\beta-1)+8\Bigr)\frac Rr,
\]
because $m<2+\log_q(R/r)$. For every $\varepsilon>0$ the right side is
at most a constant times $(R/r)^{1+\varepsilon}$, where the constant
depends only on $q$ and $\varepsilon$, so the Assouad dimension of $G$ is
at most $1$.
\end{proof}

By the first part of this proof and Remark~\ref{rem:assouad}, the graph
$G$ also has Hausdorff dimension $1$ and lower and upper box
dimension~$1$.

Figure~\ref{assouadpic} illustrates Definition~\ref{assouad} and the
proof of Proposition~\ref{assouadgraph} for $q=3$, $4$, and $5$, from
left to right. Each panel shows a disc $B(z,R)$ of radius $R=1/(2q^2)$,
centered at the point $z$ of the graph above the midpoint of an interval
$J$ of level $2$, so that the disc lies above $J$. In the upper row
$J=[(q-1)/q^2,1/q]$, which has slope $0$, and the centers are
$(5/18,7/9)$, $(7/32,13/16)$, and $(9/50,43/50)$. In the lower row
$J=[0,1/q^2]$, which has slope $2(q-1)$, and the centers are
$(1/18,1/3)$, $(1/32,1/4)$, and $(1/50,11/50)$. Thus the two rows show
the cases $d=0$ and $d\ne0$ of the proof. In the notation of the proof,
$R=1/(2q^2)$ gives $n=2$. Every $r$ with $3^{-5}\le r<3^{-4}$ gives
$m=3$ for $q=3$, and every $r$ with $q^{-4}\le r<q^{-3}$ gives $m=2$ for
$q=4$ and $q=5$. Therefore $\delta=3^{-5}$, $4^{-4}$, and $5^{-4}$, and $J$
consists of $q^m=27$, $16$, and $25$ columns of the grid $\delta\Z^2$.
The shaded squares are the squares of this grid whose interiors contain
points of the graph in the disc. They cover the part of the graph in the
disc, and the rest of the graph is drawn lighter. The graph is drawn
from the exact values of $T_q$ at the multiples of $3^{-9}$, $4^{-8}$,
and $5^{-7}$.

By \eqref{cell}, on $J$ the difference between $T_q$ and a linear
function with slope $d$ is a copy of $T_q$ scaled by $q^{-2}$, with
values between $0$ and $q^{-2}\max T_q$, which is $1/9$, $3/40$, and
$3/50$ by Theorem~\ref{gmax}. Hence the graph above $J$ lies between two
parallel lines at this vertical distance, which are drawn dashed. In the
upper row only the upper line is visible, because the lower one passes
below the disc. For $q=3$ and $q=5$ the midpoint of $J$ is the only
point at which $T_q$ attains its maximum on $J$ (see
Theorem~\ref{cellmax}), so the upper line passes through the center. If
$q=4$, then the maximum $33/40$ of $T_4$ on $J$ is attained at
uncountably many points, none of them a $q$-adic rational (see
Theorem~\ref{cellmax}), and the upper line touches the graph at these
points. The discs of the upper row contain the upper part of the graph
above $J$, and the $45$, $36$, and $44$ shaded squares lie in $19$,
$12$, and $13$ columns of the grid, at most $5$ in each. In the lower
row the graph passes steeply through the disc, and the $35$, $18$, and
$28$ shaded squares lie in only $9$, $3$, and $4$ columns, at most $7$,
$8$, and $9$ in each.

At each vertical line $x=i\delta$ of the grid with $0\le i\delta\le1$,
the graph passes through a corner of four squares, because $T_q(i\delta)$
is an integer multiple of $\delta$ by \eqref{samp}, and $T_q$ has a strict
local minimum at $i\delta$ by Theorem~\ref{thm1}. Hence the graph enters
both squares directly above such a corner in the disc, although the
piece of the graph in one of them can be too short to be seen.

\begin{figure}[H]
\centering
\includegraphics{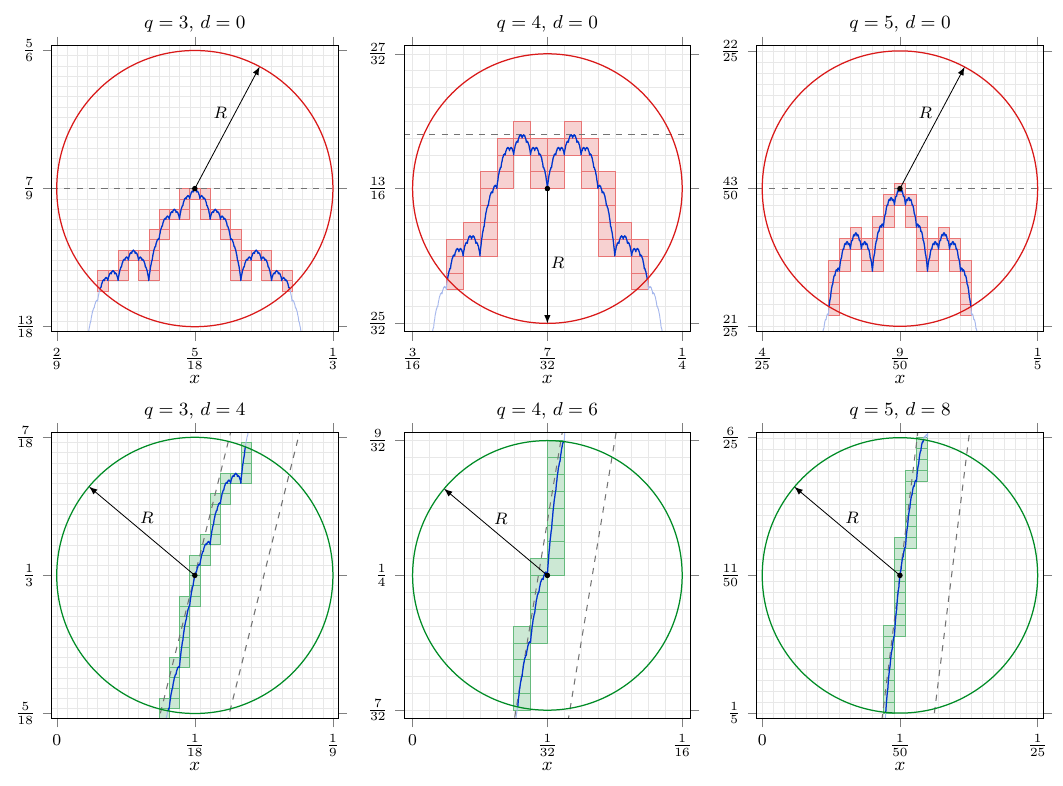}
\caption{Discs $B(z,R)$ centered on the graph of $T_q$ for $q=3$, $4$,
and $5$ (see Definition~\ref{assouad}) and the grid squares that cover
the graph inside them, as in the proof of
Proposition~\ref{assouadgraph}.}
\label{assouadpic}
\end{figure}

\section{Concluding remarks}
\label{sec:questions}

The main theorems answer their questions completely and in every base
$q\ge2$. Theorem~\ref{thm2} and Proposition~\ref{cusp} determine all
points at which $T_q$ has an infinite derivative from one side or from
both sides, Theorem~\ref{thm1} determines all local extrema of $T_q$, and
Theorem~\ref{cellmax} gives the maximum of $T_q$ on every interval of
level $n$ together with all points at which it is attained. Through the
sampling equation, Theorem~\ref{rightmin} gives Theorem~\ref{lamdig},
which determines $\lambda_h(Q_n^q)$ for every $n\ge1$ and every integer
$h$ with $1\le h\le q^n/2$.
In every base, $L(y)$ is finite for almost every $y$, and $L(T_q(x))$ is
uncountable for almost every $x$ (see Theorem~\ref{levelthm}). Several
finer results hold only in small bases, and the following questions
concern them.

For odd $q\ge5$ that are not squares, we do not know whether the set $X$
of Section~\ref{sec:X} has isolated points (see
Proposition~\ref{Xperfect}). For $q\ne3$ we do not know the largest
number $c$ such that $D^+T_q(x)-D_+T_q(x)\ge c$ at every point $x$ at
which both right Dini derivatives are finite. For odd $q$,
Proposition~\ref{midpt} shows that $c\le(q-1)/2$, and for $q=3$
Theorem~\ref{dinigap} gives $c=1$. Kr\"uppel~\cite[Lemma~6.2]{Kruppel2007}
showed that the right Dini derivatives of Takagi's function $T_2$ at $1/3$
are $0$ and $-1$, so that $c\le1$ for $q=2$. For $q\ge4$ we do not know a formula of
the kind of Propositions~\ref{optcount} and~\ref{optcount2} for the
number of the sizes $1\le h\le q^n/2$ with $\lambda_h(Q_n^q)=\xi_n(h)$.

In the rest of this section we let $q=3$.
By Theorem~\ref{threequarters}, the level set $L(3/4)$ has Hausdorff
dimension $\log2/\log9$. By the bound stated without proof in
Remark~\ref{rem:levelbox}, this is the largest Hausdorff dimension of a
level set of $T_q$. The heights $y$ for which $L(y)$ has Hausdorff
dimension at least $\log2/\log9$ are dense in $[0,1]$. Indeed, if
$[m/3^n,(m+1)/3^n]$ has slope $0$, then by \eqref{cell} the level set at
height $T_q(m/3^n)+\frac34\,3^{-n}$ contains the set $(m+L(3/4))/3^n$, a
copy of $L(3/4)$ scaled by $3^{-n}$. By
Lemma~\ref{dzero}, for almost every $x$ we have $d_n(x)=0$ for
infinitely many $n$. For each such $n$ the ternary interval $[a_n,c_n]$ of level $n$ that
contains $x$ has slope $0$, and $T_q(x)$ lies within $3^{-n}$ of the height
$T_q(a_n)+\frac34\,3^{-n}$, because $T_q([a_n,c_n])=[T_q(a_n),T_q(a_n)+3^{-n}]$
by Remark~\ref{rem:signs}. Since $T_q$ is continuous and
maps $[0,1]$ onto $[0,1]$, the values $T_q(x)$ at these points $x$ are
dense in $[0,1]$, and hence so are the heights just described. On the
other hand, by Theorem~\ref{heightdim}(b), the heights $y$ for which
$L(y)$ has positive Hausdorff dimension form a set of Lebesgue measure
zero and Hausdorff dimension $1$. For $0<\alpha\le\log2/\log9$, it would
be natural to determine the Hausdorff dimension of the set of heights $y$
for which $L(y)$ has Hausdorff dimension $\alpha$, and of the set of
heights $y$ for which $L(y)$ has Hausdorff dimension at least $\alpha$.
For the second set and $\alpha=\log2/(\ell\log3)$ with an integer
$\ell\ge4$, Theorem~\ref{heightdim}(b) gives a lower bound. Another question is whether
every level set of $T_q$ has Assouad dimension at most $\log2/\log9$, as
$L(3/4)$ does by Proposition~\ref{assouadlevel}. It would also be
interesting to prove the bound of Remark~\ref{rem:levelbox} by hand.

By Corollary~\ref{levelcount}, the Lebesgue measure of the set of heights
$y$ for which $L(y)$ has at least $M$ points is at most of order
$1/\log M$. On the other hand, the sum of these measures over $M\ge1$ is
$\int_0^1\#L(y)\,dy$, and this integral is infinite. Indeed, by the
intermediate value theorem, every $y$ strictly between $T_q(m/3^n)$ and
$T_q((m+1)/3^n)$ is a value of $T_q$ on the open interval
$(m/3^n,(m+1)/3^n)$. These open intervals are disjoint for
$0\le m<3^n$, so $\int_0^1\#L(y)\,dy\ge V_n$ for every $n$, and
$V_n\to\infty$ by Theorem~\ref{varthm}. It would be interesting
to know how fast these measures tend to $0$, and to describe the heights
with infinite level sets.

By Proposition~\ref{convexeq}(b),
$2T_q(u)-T_q(u-h)-T_q(u+h)\le4h$ for all real $u$ and
$h\ge0$, with equality when $h=0$ and when $[u-h,u+h]$ is an interval
$[m/3^n,(m+1)/3^n]$ with integers $n\ge0$ and $m$, and Allaart's result
quoted in Section~\ref{sec:convex} shows that the inequality is strict
when $u$ and $h>0$ are ternary rationals with $0\le u-h$ and $u+h\le1$.
We do not know whether there are other cases of equality.

By Proposition~\ref{ratval}, $T_q$ takes rational values at rational
points. We do not know which rational numbers in $[0,1]$ are values of
$T_q$ at rational points; Lagarias~\cite[Section~3]{Lagarias2012} noted
that the same question is open for Takagi's function.

For Takagi's function, Kr\"uppel~\cite[Proposition~6.3]{Kruppel2007}
also stated that the largest and the smallest of the four Dini derivatives
differ by at least $3$ wherever all four are finite, and they differ by
$3$ at $1/3$~\cite[Lemma~6.2]{Kruppel2007}. His proof shows that the
difference is at least $2$ at every such point and is $3$ at the points
$k/2^n+1/(3\cdot2^n)$ with integers $n\ge0$ and $k$. It does not exclude
values between $2$ and $3$ at other points, and we do not know whether
the bound $3$ holds. For $T_q$, the four Dini
derivatives at the midpoints of the ternary intervals and at $1/4$ give
the difference $4$ (see Propositions~\ref{midpt} and~\ref{quarter}). We do not
know whether
$\max\{D^+T_q(x),D^-T_q(x)\}-\min\{D_+T_q(x),D_-T_q(x)\}\ge4$ at every point $x$
at which the four Dini derivatives are finite.

By Theorem~\ref{cardmod}, every number $2,6,10,\ldots$ is the number of
points of a finite level set, and by Corollary~\ref{twopoints} the number
$2$ occurs on a set of heights of positive Lebesgue measure. We do not
know whether, for every $i\ge1$, the heights $y$ for which $L(y)$ has
exactly $4i+2$ points form a set of positive Lebesgue measure. The level
set at the height $2/3=T_q(1/3)$ is countably infinite
(see Example~\ref{ex:twothirds}). We do not know whether $L(y)$ can be
countably infinite for a height $y$ that is not a value of $T_q$ at a
ternary rational or at the midpoint of a ternary interval.

\bibliographystyle{amsplain0}
\bibliography{references}
\end{document}